\documentclass[11pt]{amsart}
\DeclareRobustCommand{\gobblefour}[4]{}

\usepackage{amsmath, amssymb, amsthm, latexsym}
\usepackage{amsmath}
\usepackage{amsthm}
\usepackage{bbm}
\usepackage{epsfig}
\usepackage{amsfonts}
\usepackage{amssymb}
\usepackage{amsmath}   
\usepackage{amsthm}
\usepackage{latexsym}
\usepackage{fullpage}
\usepackage{amssymb}
\usepackage{latexsym}
\usepackage[center]{caption}
\usepackage{tikz}
\usetikzlibrary{matrix}
\usepackage{mathtools}
\usepackage{enumerate}
\usepackage{nameref}
\usepackage{varioref}
\usepackage{hyperref}
\usepackage{cleveref}
\usepackage{stmaryrd}
\usepackage{tikz-cd}

\usepackage{listings}

\usepackage{graphicx}

\usepackage[T1]{fontenc}

\usepackage{euflag}

\newcounter{braid}
\newcounter{strands}
\pgfkeyssetvalue{/tikz/braid height}{1cm}
\pgfkeyssetvalue{/tikz/braid width}{1cm}
\pgfkeyssetvalue{/tikz/braid start}{(0,0)}
\pgfkeyssetvalue{/tikz/braid colour}{black}
\pgfkeys{/tikz/strands/.code={\setcounter{strands}{#1}}}
\makeatletter
\def\cross{%
  \@ifnextchar^{\message{Got sup}\cross@sup}{\cross@sub}}

\def\cross@sup^#1_#2{\render@cross{#2}{#1}}

\def\cross@sub_#1{\@ifnextchar^{\cross@@sub{#1}}{\render@cross{#1}{1}}}

\def\cross@@sub#1^#2{\render@cross{#1}{#2}}

\def\render@cross#1#2{
  \def\strand{#1}
  \def\crossing{#2}
  \pgfmathsetmacro{\cross@y}{-\value{braid}*\braid@h}
  \pgfmathtruncatemacro{\nextstrand}{#1+1}
  \foreach \thread in {1,...,\value{strands}}
  {
    \pgfmathsetmacro{\strand@x}{\thread * \braid@w}
    \ifnum\thread=\strand
    \pgfmathsetmacro{\over@x}{\strand * \braid@w + .5*(1 - \crossing) * \braid@w}
    \pgfmathsetmacro{\under@x}{\strand * \braid@w + .5*(1 + \crossing) * \braid@w}
    \draw[braid] \pgfkeysvalueof{/tikz/braid start} +(\under@x pt,\cross@y pt) to[out=-90,in=90] +(\over@x pt,\cross@y pt -\braid@h);
    \draw[braid] \pgfkeysvalueof{/tikz/braid start} +(\over@x pt,\cross@y pt) to[out=-90,in=90] +(\under@x pt,\cross@y pt -\braid@h);
    \else
    \ifnum\thread=\nextstrand
    \else
     \draw[braid] \pgfkeysvalueof{/tikz/braid start} ++(\strand@x pt,\cross@y pt) -- ++(0,-\braid@h);
    \fi
   \fi
  }
  \stepcounter{braid}
}

\tikzset{braid/.style={double=\pgfkeysvalueof{/tikz/braid colour},double distance=1pt,line width=2pt,white}}

\newcommand{\braid}[2][]{%
  \begingroup
  \pgfkeys{/tikz/strands=2}
  \tikzset{#1}
  \pgfkeysgetvalue{/tikz/braid width}{\braid@w}
  \pgfkeysgetvalue{/tikz/braid height}{\braid@h}
  \setcounter{braid}{0}
  \let\sigma=\cross
  #2
  \endgroup
}
\makeatother

\input xypic
\newtheorem{theorem}{Theorem}[section]
\numberwithin{theorem}{section}
\newtheorem{proposition}[theorem]{Proposition}
\newtheorem{corollary}[theorem]{Corollary}
\newtheorem{convention}[theorem]{Convention}

\newtheorem{lemma}[theorem]{Lemma}
\newtheorem{conjecture}[theorem]{Conjecture}
\theoremstyle{plain}

\newtheorem{definition}[theorem]{Definition}

\newtheorem{assumption}[theorem]{Assumption}
\newtheorem{remark}[theorem]{Remark}
\newtheorem{example}[theorem]{Example}
\newtheorem{claim}[theorem]{Claim}
\newtheorem{choice}[theorem]{Choice}
\def\Z{\mathbb{Z}}

\def\Pi{\mathbb{P}^{\infty}}

\def\Zpk{\mathbb{Z}/p^{k}}
\def\Zpk1{\mathbb{Z}/p^{k-1}}

\def\sl2{\widetilde{SL_{2}(\Z)}}

\DeclareMathOperator{\Nm}{Nm}

\DeclareFontFamily{U}{wncy}{}
    \DeclareFontShape{U}{wncy}{m}{n}{<->wncyr10}{}
    \DeclareSymbolFont{mcy}{U}{wncy}{m}{n}
    \DeclareMathSymbol{\Sh}{\mathord}{mcy}{"58}

\title{Supersingular main conjectures, Sylvester's conjecture and Goldfeld's conjecture}

\author{Daniel Kriz}

\address{Institut de Math\'{e}matiques de Jussieu - Paris Rive Gauche\\
Sorbonne Universit\'{e} - Campus Pierre et Marie Curie\\
4, place Jussieu - Boite Courrier 247\\
75252 Paris Cedex 05}

\email{kriz@imj-prg.fr}

\thanks{The author was supported by an NSF Mathematical Sciences Postdoctoral Research Fellowship, award number 1803388. This work was supported by a grant from the Simons Foundation (Grant Number 814268, MSRI). \\
{\euflag}This project has received funding from the European Union's Horizon 2020 research and innovation programme under the Marie Sk\l odowska-Curie grant agreement No 101034255}

\begin{document}

\begin{abstract}We prove a $p$-converse theorem for elliptic curves $E/\mathbb{Q}$ with complex multiplication by the ring of integers $\mathcal{O}_K$ of an imaginary quadratic field $K$ in which $p$ is ramified. Namely, letting $r_p = \mathrm{corank}_{\mathbb{Z}_p}\mathrm{Sel}_{p^{\infty}}(E/\mathbb{Q})$, we show that $r_p \le 1 \implies \mathrm{rank}_{\mathbb{Z}}E(\mathbb{Q}) = \mathrm{ord}_{s = 1}L(E/\mathbb{Q},s) = r_p$ and $\#\Sh(E/\mathbb{Q}) < \infty$. In particular, this has applications to two classical Diophantine problems. First, it resolves Sylvester's conjecture on rational sums of cubes, showing that for all primes $\ell \equiv 4,7,8 \pmod{9}$, there exists $(x,y) \in \mathbb{Q}^{\oplus 2}$ such that $x^3 + y^3 = \ell$. Second, combined with work of Smith, it resolves the congruent number problem in 100\% of cases and establishes Goldfeld's conjecture on ranks of quadratic twists for the congruent number family. The method for showing the above $p$-converse theorem relies on new interplays between Iwasawa theory for imaginary quadratic fields at nonsplit primes and relative $p$-adic Hodge theory. In particular, we show that a certain de Rham period $q_{\mathrm{dR}}$ can be used to construct anticyclotomic $p$-adic $L$-functions for Hecke characters and newforms, interpolating anticyclotomic twists of positive Hodge-Tate weight in the central critical range. Moreover, one can relate the Iwasawa module of elliptic units to these anticyclotomic $p$-adic $L$-functions via a new ``Coleman map'', which is, roughly speaking, the $q_{\mathrm{dR}}$-expansion of the Coleman power series map. Using this, we formulate and prove a new Rubin-type main conjecture for elliptic units, which is eventually related to Heegner points in order to prove the $p$-converse theorem.
\end{abstract}

%\date{June 7, 2022}

\maketitle

\tableofcontents

\pagebreak

\section{Introduction}

\subsection{$p$-converse theorems}\label{Introduction}One of the main motivational problems in modern number theory has been the Birch and Swinnerton-Dyer (BSD) conjecture. Given an elliptic curve $E/\mathbb{Q}$, the conjecture states
$$r_{\mathrm{alg}}(E/\mathbb{Q}) := \mathrm{rank}_{\mathbb{Z}}E(\mathbb{Q}) = \mathrm{ord}_{s = 1}L(E/\mathbb{Q},s) =: r_{\mathrm{an}}(E/\mathbb{Q}),$$
where $L(E/\mathbb{Q},s)$ is the $L$-function of $E/\mathbb{Q}$. To date, the strongest general result toward this conjecture is the following 
implication which follows from the formula of Gross-Zagier \cite{GrossZagier} and Kolyvagin's method of Euler systems \cite{Kolyvagin}: for $r = 0$ or 1,
\begin{equation}\label{GZKimplication}r_{\mathrm{an}}(E/\mathbb{Q}) = r \implies r_{\mathrm{alg}}(E/\mathbb{Q}) = r \hspace{.25cm} \text{and} \hspace{.25cm} \#\Sh(E/\mathbb{Q}) < \infty 
\end{equation}
where $\Sh(E/\mathbb{Q})$ is the Shafarevich-Tate group. One can view this result as ``one direction of BSD'' in the rank $r = 0, 1$ case.  A good approximation for $r_{\mathrm{alg}}(E/\mathbb{Q})$ is the corank of the $p^{\infty}$-Selmer group for any prime $p$: we have
$$r_{\mathrm{alg}}(E/\mathbb{Q}) \le r_p(E/\mathbb{Q}) := \mathrm{corank}_{\mathbb{Z}_p}\mathrm{Sel}_{p^{\infty}}(E/\mathbb{Q}) = \mathrm{rank}_{\mathbb{Z}_p}\mathrm{Hom}_{\mathbb{Z}_p}(\mathrm{Sel}_{p^{\infty}}(E/\mathbb{Q}),\mathbb{Q}_p/\mathbb{Z}_p)$$
with equality if and only if $\#\Sh(E/\mathbb{Q})[p^{\infty}] < \infty$ (which is conjectured to be true). Thus  (\ref{GZKimplication}) implies, for $r = 0, 1 $, 
\begin{equation}\label{GZKimplication2}r_{\mathrm{an}}(E/\mathbb{Q}) = r \implies r_{\mathrm{alg}}(E/\mathbb{Q}) = r\hspace{.25cm} \text{and} \hspace{.25cm} \#\Sh(E/\mathbb{Q}) < \infty \implies r_p(E/\mathbb{Q}) = r.
\end{equation}
However, the analytic rank $r_{\mathrm{an}}(E/\mathbb{Q})$ is often hard to compute in practice, and so a natural question is whether one can prove anything in the ``converse'' direction to (\ref{GZKimplication2}): for $r = 0,1$,
\begin{equation}\label{pconverseimplication}r_p(E/\mathbb{Q}) = r \overset{?}{\implies} r_{\mathrm{an}}(E/\mathbb{Q}) = r.
\end{equation}
To emphasize the dependence on the choice of auxiliary prime $p$, such implications, once proven, have come to be known as ``$p$-converse theorems''. The natural approach for proving $p$-converse theorems is via Iwasawa theory. Rubin \cite{RubinMC} used Iwasawa theory for imaginary quadratic fields to prove the first $p$-converse theorems in the $r = 0$  complex multiplication (CM) case.\footnote{We note that in the $r = 0$ CM case, (\ref{GZKimplication}) was first proven in the landmark work of Coates-Wiles \cite{CoatesWiles}. In the $r = 0$ non-CM case, (\ref{GZKimplication}) also follows from the celebrated work of Kato \cite{Kato}.}  Recently, there has been great progress on $p$-converse theorems. The work of Skinner-Urban \cite{SkinnerUrban} and Wan \cite{Wan} on $GL_2$-main conjectures give $p$-converse theorems in the $r = 0$ non-CM case. Skinner \cite{Skinner} used ``vertical (i.e. $p$-aspect)'' Iwasawa theory and Zhang \cite{WZhang} ``horizontal (i.e. prime-to-$p$-aspect)'' Iwasawa theory in order to prove the first $p$-converse theorems in the $r = 1$ ordinary non-CM case. See also \cite{SkinnerZhang} for the case of multiplicative reduction. The assumptions of these works, however, rule out cases such as supersingular primes $p$ and CM elliptic curves. More recently, substantial progress has been made in the supersingular non-CM case by Castella-Wan \cite{CastellaWan}, and in the ordinary CM case by Burungale-Tian \cite{BurungaleTian} (see also \cite{BurungaleCastella}). See also the work of Castella-Grossi-Lee-Skinner \cite{CastellaGrossiLeeSkinner} for recent considerable progress in the $r = 0, 1$ case when $E[p]$ is reducible as a $\mathrm{Gal}(\overline{\mathbb{Q}}/\mathbb{Q})$-module. 

However, the case of CM at (potentially) supersingular primes $p$, i.e. those primes inert or ramified in the CM field $K$, entails certain technical difficulties in the Iwasawa-theoretic approach, and so significantly less progress was hitherto made in this case. The main result of this paper is the following $p$-converse theorem in the case where $p$ is ramified in $K$. 

\begin{theorem}[Theorem \ref{BSDrank0theorem} and Theorem \ref{BSDrank1theorem}]\label{MainTheorem}Let $E/\mathbb{Q}$ be an elliptic curve and $K/\mathbb{Q}$ an imaginary quadratic field such that the base change $E/K$ has CM by $\mathcal{O}_K$ (i.e. $\mathrm{End}(E/K) \cong \mathcal{O}_K$). Let $p$ be the unique finite prime ramified in $K/\mathbb{Q}$. Then for $r = 0, 1$ the $p$-converse implication (\ref{pconverseimplication}) is true for $E/\mathbb{Q}$. 
\end{theorem}

Throughout the rest of the paper, we will often say ``an elliptic curve $E/\mathbb{Q}$ has CM by $\mathcal{O}_K$'' if $E/\mathbb{Q}$ and $K$ satisfy the conditions of Theorem \ref{MainTheorem}. Note that $K$ necessarily has class number 1 in this setting, and thus there is a unique finite prime $p$ ramified in $K/\mathbb{Q}$.

Broadly speaking, our approach involves relating objects and constructions from relative $p$-adic Hodge theory to classical Iwasawa-theoretic objects. We use a suitable supersingular substitute $q_{\mathrm{dR}}$ for the ordinary Serre-Tate coordinate, coming from the relative $p$-adic de Rham comparison theorem applied to the universal (false) elliptic curve at infinite level. We then consider expansions in the coordinate $q_{\mathrm{dR}}-1$, called $q_{\mathrm{dR}}$-expansions, which give analogues of Serre-Tate expansions on the supersingular locus. These power series can naturally be viewed as sections of a certain completion of the period sheaf $\mathcal{O}\mathbb{B}_{\mathrm{dR}}$ from \cite{Scholze} over the infinite-level Shimura curve. Using $q_{\mathrm{dR}}$-expansions, we are able to ``geometrically'' realize norm-compatible systems of local units infinite towers of modular curves. When applied to the Coleman power series of the norm-compatible systems of elliptic units, these $q_{\mathrm{dR}}$-expansions give rise to $p$-adic $L$-functions which are supersingular analogues of those of Katz (\cite{KatzImQuad}, \cite{KatzCM}), interpolating $L$-values of anticyclotomic twists of positive Hodge-Tate weight of a fixed Hecke character.%\footnote{The author wishes to emphasize that our $p$-adic $L$-functions satisfy a new type of interpolation property, relating twisted Fourier transforms of finite order anticyclotomic characters $\chi$ to $\chi$-twisted central $L$-values (see (\ref{interp1}), (\ref{interp1twist}) and (\ref{interpRankinSelberg})). Thus, our measures should not be viewed as anticyclotomic measures in the traditional sense, even though they interpolate anticyclotomic twists of $L$-values. In fact, this type of interpolation for (untwisted) Fourier transforms of anticyclotomic characters is also satisfied by Katz's ordinary measures.} 
%The coordinate $q_{\mathrm{dR}}$ was also used in \cite{Kriz} and \cite{KrizThesis} in order to construct new Katz-type $p$-adic $L$-functions interpolating anticyclotomic characters of \emph{varying} Hodge-Tate weight and \emph{fixed} finite type, which give functions continuous in the Hodge-Tate weight and not necessarily measures.

In the next section, we give a detailed outline of the main steps and novelties of our approach. 

\subsection{Outline of the argument}Henceforth fix an imaginary quadratic field $K$ and a prime $p$ \emph{ramified} in $K$. Our proof of Theorem \ref{MainTheorem} is broadly divided into three steps:
\begin{enumerate}
\item Establishing a 2-variable Rubin-type main conjecture ``without $p$-adic $L$-function'' using the Euler system of elliptic units. This relates the structure of $\mathrm{Sel}_{p^{\infty}}(E/\mathbb{Q})$ to elliptic units. 
\item Constructing a Coleman map from (a submodule of) norm-compatible systems of local units to a power series ring, satisfying a suitable ``explicit reciprocity law''. This construction involves taking the $q_{\mathrm{dR}}$-expansion of a universal (or ``thickened'') version of the Coleman power series map, using a certain canonical coordinate $q_{\mathrm{dR}}$ coming from relative $p$-adic Hodge theory on the infinite-level Shimura curve. The explicit reciprocity law states that the Coleman map sends elliptic units to a Katz-type $p$-adic $L$-function, which for the remainder of the Introduction we will denote by $\mathcal{L}_E$. We divide this step into three substeps 2.1-2.3 in our description below.
\item Using (1) and (2), the $r = 0$ $p$-converse theorem immediately follows by following similar methods to those of \cite[Section 11]{RubinMC}. For $r = 1$, we use an additional argument relating $\mathcal{L}_E$ to a $p$-adic $L$-function $\mathcal{L}_{g,\chi_0}$ attached to a Rankin-Selberg pair $(g,\chi_0)$ that is a ``good'' twist of the modular form attached to $E$. This $\mathcal{L}_{g,\chi_0}$ has value at the trivial character equal to the logarithm of a Heegner point at the trivial character. Using the interpolation properties of $\mathcal{L}_E$ and $\mathcal{L}_{g,\chi_0}$, one shows that the Heegner point is non-torsion if and only if a value of $\mathcal{L}_E$ outside the range of interpolation is nonzero. Using (1) and (2), this latter value is shown to be nonzero under the assumption $r_p(E/\mathbb{Q}) = 1$. 
\end{enumerate}

\subsubsection{\textbf{Step (1)}}Step (1) involves classical methods of Iwasawa theory, such as those developed by Rubin \cite{RubinMC}, with new inputs from work on equivariant main conjectures of Johnson-Leung-Kings \cite{JohnsonLeungKings} and a new argument involving Gross-Kersey's method for ``factorizing''\footnote{It is interesting that there are two factorization identities of $p$-adic $L$-functions that show up in the proof of Theorem \ref{MainTheorem}; a ``cyclotomic factorization identity'' in Step (1), and an ``anticyclotomic factorization identity'' in Step (3).} logarithms of elliptic units (cf. \cite{Gross}, \cite{Kersey} and \cite{Rubincongruence}) to show the vanishing of the algebraic cyclotomic $\mu$-invariant.\footnote{This step, proving both divisibilities for the Euler system of elliptic units, is the only place where cyclotomic Iwasawa theory intervenes, as the rest of our methods primarily involve Iwasawa theory of the $\mathbb{Z}_p^{\oplus 2}$-extension and anticyclotomic Iwasawa theory. This argument - using the 1-variable cyclotomic main conjecture, the vanishing of the cyclotomic $\mu$-invariant, and the Euler system divisibility of the 2-variable elliptic units main conjecture to prove the full 2-variable elliptic units main conjecture - is reminiscent of the ``easy lemma'' of Skinner-Urban, see \cite[Lemma 3.1.7]{SkinnerUrban}.} From class field theory, one gets an exact sequence
$$0 \rightarrow \mathcal{E} \rightarrow \mathbb{U} \xrightarrow{\mathrm{rec}} \mathcal{X} \rightarrow \mathcal{Y} \rightarrow 0$$
attached to the tower $\mathcal{K}_{\infty}/K = K(E[p^{\infty}])/K$, where, for this Introduction, $\mathbb{U}$ the module of norm-compatible systems of principal semi-local units, $\mathcal{E}$ is the $p$-adic closure in $\mathbb{U}$ of the module of norm-compatible systems of principal global units, $\mathcal{X}$ is the Galois group of the maximal pro-$p$ abelian extension of $\mathcal{K}_{\infty}$ unramified outside places above $p$, $\mathcal{Y}$ is the Galois group of the maximal pro-$p$ abelian extension of $\mathcal{K}_{\infty}$ unramified everywhere and $\mathrm{rec}$ is the Artin reciprocity map. Denote the $p$-adic closure in $\mathbb{U}$ of the module of norm-compatible systems of principal elliptic units by $\mathcal{C} \subset \mathcal{E}$. (See Section \ref{RMCsection} for precise definitions; outside of this Introduction, $\mathcal{E}, \mathbb{U}$ and $\mathcal{C}$ will denote modules of norm-compatible systems of units (not necessarily principal), and $\mathcal{E}^1, \mathbb{U}^1$ and $\mathcal{C}^1$ will denote their principal counterparts.) Dividing out by $\mathcal{C}$, we get an exact sequence of $\Lambda_E = \mathbb{Z}_p\llbracket \mathrm{Gal}(\mathcal{K}_{\infty}/K) \rrbracket$-modules.
$$0 \rightarrow \mathcal{E}/\mathcal{C} \rightarrow \mathbb{U}/\mathcal{C} \xrightarrow{\mathrm{rec}} \mathcal{X} \rightarrow \mathcal{Y} \rightarrow 0,$$
where the $\Lambda_E$-ranks are 0, 1, 1 and 0 respectively. By the aforementioned $\mu$-invariant result and equivariant main conjecture, one gets 
\begin{equation}\label{EMCintro}\mathrm{char}_{\Lambda_E}(\mathcal{E}/\mathcal{C}) = \mathrm{char}_{\Lambda_E}(\mathcal{Y})
\end{equation}
(see Theorem \ref{EMC} for a precise statement; more accurately, we will only consider a certain isotypic component of (\ref{EMCintro})). Now one \emph{chooses} a $\Lambda_E$-rank 1 submodule $U' \subset \mathbb{U} \otimes_{\mathbb{Z}_p}F$, $F$ a finite extension of $\mathbb{Q}_p$ such that $U' \cap \mathcal{E} = 0$. The choice of $U'$ will depend on whether we are in the $r_p(E/\mathbb{Q}) = 0$ or 1 case. In the $r_p(E/\mathbb{Q}) = 0$ case, $U'$ is the kernel of the reciprocity map (see Section \ref{rank0section}). In the $r_p(E/\mathbb{Q}) = 1$ case, $U'$ is any subspace with $U' \cap \mathcal{E} = 0$ (and thus the intersection of $U'$ with the kernel of the Atrin reciprocity map $\mathrm{rec}$, the latter which contains $\mathcal{E}$, is 0). Let $\Lambda_{E,F} = \Lambda_E \otimes_{\mathbb{Z}_p}F$. Then (\ref{EMCintro}) implies a Rubin-type main conjecture 
\begin{equation}\label{RMCintro}\mathrm{char}_{\Lambda_{E,F}}((\mathbb{U}\otimes_{\mathbb{Z}_p}F)/(\mathcal{C},U')) = \mathrm{char}_{\Lambda_{E,F}}((\mathcal{X}\otimes_{\mathbb{Z}_p}F)/\mathrm{rec}(U'))
\end{equation}
(see Theorem \ref{RMC} for a precise statement; again, we will only consider a certain isotypic component of (\ref{RMCintro}), and in our setting $F$ will be the $p$-adic completion of $K$), which does \emph{not} explicitly involve $p$-adic $L$-functions.\footnote{The distinction between main conjectures ``without $p$-adic $L$-functions'' (involving only ``generalized Kato classes'' coming from Euler systems) and classical main conjectures ``with $p$-adic $L$-functions'' was first made by Kato \cite[Chapter III and IV]{Kato}. In particular, this distinction is fruitful for the purposes of relating different main conjectures (where the Kato class more transparently relates to different Selmer groups), and has led to much progress on the BSD formula in recent years. (See, for example, \cite{SkinnerAWS}, \cite{JetchevSkinnerWan}.)}

\subsubsection{\textbf{Step (2).1: Thickened Coleman power series}} Step (2) is the most novel part of the argument. Coleman power series (see \cite[Chapter I.2]{deShalit}) have long been a starting point in studying the Iwasawa theory of imaginary quadratic fields $K$ and relating elliptic units to $p$-adic analytic objects. In the height 1 setting ($p$ split in $K$), it is well-known how to obtain measures from Coleman power series (see Chapter I.3 of op. cit.), and reconstruct the Katz $p$-adic $L$-function from Coleman power series of elliptic units (see Chapter II.4 of op. cit.). Essentially, one uses the fact that all height 1 formal groups are isomorphic (after base change) to the formal multiplicative group $\hat{\mathbb{G}}_m$, whose coordinate ring is naturally the Iwasawa algebra on $\mathbb{Z}_p$ via the Amice transform (see \cite[I.3.1 (1)]{deShalit}). Expanding the Coleman power series into a power series in the natural coordinate on $\hat{\mathbb{G}}_m$, one thus obtains measures; this gives the ``Coleman map''
\begin{equation}\label{Colemanmapintro}\mathbb{U} \xrightarrow{\text{Coleman power series}} \mathbf{\Gamma}(\text{ht. 1 formal group over $W(\overline{\mathbb{F}}_p)$}) \cong \mathbf{\Gamma}(\hat{\mathbb{G}}_{m,W(\overline{\mathbb{F}}_p)}) = W(\overline{\mathbb{F}}_p)\llbracket \mathbb{Z}_p\rrbracket
\end{equation}
where $\mathbf{\Gamma}$ denotes the global section functor, $W(\overline{\mathbb{F}}_p)$ denotes the Witt vectors of $\overline{\mathbb{F}}_p$, and the last equality above is given by the Amice transform. 

However, in the case where the underlying formal group of the Coleman power series has height 2 ($p$ nonsplit in $K$), relatively little was known on how to ``convert'' Coleman power series into (bounded) $p$-adic $L$-functions. This is because the completed group ring of $\mathbb{Z}_p^{\oplus 2}$ can be viewed, via the Amice transform, as the coordinate ring of the $2$-dimensional (component-wise height 1) formal group $\hat{\mathbb{G}}_m^{\oplus 2}$, and is not in any obvious way related to the height 2 formal group. Thus there does not seem to be an obvious map like (\ref{Colemanmapintro}) into bounded measures. However, there is such a map into distributions, as was constructed by Schneider-Teitelbaum in \cite{SchneiderTeitelbaum} by naturally interpreting Coleman power series as Fourier transforms of distributions. However, for our purposes, we wish to at least relate the module of elliptic units $\mathcal{C}$ to a $p$-adic $L$-function with $p$-adically bounded values, and hence some modification is needed to get a Coleman map with this property.

The starting point of our strategy for constructing an appropriate Coleman map is to follow an idea of Tsuji \cite[Section I.4]{Tsuji} and ``thicken'' Coleman power series to an element in the coordinate ring of the universal deformation of a height 2 formal group over $\overline{\mathbb{F}}_p$. Pulling back via a universal torsion point section, one gets a function on the underlying deformation space (viewed as a residue disc in an appropriate Shimura curve). We then study the $p$-adic variation of this function using the geometry of the deformation space. In the ordinary setting, one has a natural Serre-Tate coordinate on the base, and one would then consider the Serre-Tate expansion of this function around ordinary CM points in order to construct measures (see \cite{Brakocevic}, for example). In the supersingular setting, such a natural coordinate in the structure sheaf of the moduli space does not exist; this is due to the more complicated nature of the deformation space of supersingular $p$-divisible groups, particularly the lack of a distinguished \'{e}tale line (which completely parametrizes deformations in the ordinary case, see \cite{KatzST}). 

A natural strategy then is to introduce relative $p$-adic Hodge theory, and enlarge the structure sheaf to a de Rham period sheaf. One has a canonical coordinate $q_{\mathrm{dR}} \in \mathbb{B}_{\mathrm{dR}}\llbracket X\rrbracket$ on the $p$-adic universal cover of the moduli space, where $\mathbb{B}_{\mathrm{dR}}\llbracket X\rrbracket$ is a certain completion (Definition \ref{BdR+Xtdefinition}) of the period sheaf $\mathcal{O}\mathbb{B}_{\mathrm{dR}}$ from \cite[Section 6]{Scholze}. Let $\mathcal{O}\mathbb{B}_{\mathrm{dR}}^+, \mathcal{O}\mathbb{B}_{\mathrm{dR}} = \mathcal{O}\mathbb{B}_{\mathrm{dR}}^+[1/t]$ be as in loc. cit., where $t \in \mathrm{Fil}^1\mathbb{B}_{\mathrm{dR}}^+$ is some ``relative Fontaine $2\pi i$'' (see (\ref{'tdefinition})) that exists on the $p$-adic universal cover. The aforementioned coordinate $q_{\mathrm{dR}}$ is obtained by exponentiating (see (\ref{qdR}) and (\ref{zqw2})) the \emph{fundamental de Rham period} $z_{\mathrm{dR}} \in \mathbb{P}^1(\mathcal{O}\mathbb{B}_{\mathrm{dR}})$ (see Definition \ref{'zqwdefinition}, (\ref{gluezdRsection}) and (\ref{zqw2})), which measures the position of the Hodge filtration in universal de Rham cohomology
$$H^{1,0} \subset H_{\mathrm{dR}}^1 \subset H_{\text{\'{e}t}}^1 \otimes_{\mathbb{Z}_p} \mathcal{O}\mathbb{B}_{\mathrm{dR}} \cong (\mathcal{O}\mathbb{B}_{\mathrm{dR}})^{\oplus 2}.$$
Here, the last isomorphism in the previous displayed map is defined only over the $p$-adic universal cover (i.e. the $\Gamma(p^{\infty})$-level Shimura curve viewed as an adic space). The period $z_{\mathrm{dR}}$ is the natural counterpart to Scholze's Hodge-Tate period $z_{\mathrm{HT}} \in \mathbb{P}^1(\hat{\mathcal{O}})$ (see \cite{ScholzeTorsion}), where $\hat{\mathcal{O}}$ is the $p$-adically completed structure sheaf, measuring the position of the Hodge-Tate filtration 
$$H^{0,1} \subset  H_{\text{\'{e}t}}^1 \otimes_{\mathbb{Z}_p} \hat{\mathcal{O}} \cong \hat{\mathcal{O}}^{\oplus 2}$$
in universal \'{e}tale cohomology. Thus, working on the pro\'{e}tale site and in the period sheaf $\mathcal{O}\mathbb{B}_{\mathrm{dR}}$, one can expand functions on the moduli space as a power series in the coordinate $q_{\mathrm{dR}}-1$, whose coefficients are themselves \emph{functions} on the $p$-adic universal cover. In all, one gets a diagram
\begin{equation}\label{qdRexpansionintro}\begin{split}\mathbb{U} &\xrightarrow{\text{Coleman power series}} \mathcal{O}(\text{ht. 2 formal group}) \\
&\xrightarrow{\text{thicken}} \mathcal{O}(\text{universal ht. 2 formal group}) \\
&\xrightarrow{\text{torsion section}} \mathcal{O}(\text{moduli space of ht. 2 formal groups}) \\&\xrightarrow{\text{$q_{\mathrm{dR}}$-expansion}} \hat{\mathcal{O}}(\text{$p$-adic universal cover of moduli space})\llbracket q_{\mathrm{dR}}-1\rrbracket
\end{split}
\end{equation}
where $\mathcal{O}$ denotes the structure sheaf on the pro\'{e}tale site of \cite[Definition 3.9]{Scholze}, and $\hat{\mathcal{O}}$ the $p$-adic completion thereof.

\subsubsection{\textbf{Step (2).2: $p$-adic boundedness of generalized $p$-adic modular forms}}%One key aspect of working with power series in the coordinate $q_{\mathrm{dR}}$ can be summarized as follows. 
 Let $Y$ be a finite-level quaternionic Shimura curve over $\mathbb{Q}$ of prime-to-$p$ level $\Gamma$, with universal object $\pi : \mathcal{E} \rightarrow Y$, and let $\omega := \pi_*\Omega_{\mathcal{E}/Y}$ be the Hodge bundle. Let $Y(p^n)$ be the Shimura curve of level $\Gamma \cap \Gamma(p^n)$ where $\Gamma(p^n)$ is the principal congruence subgroup of level $p^n$ (i.e. the subgroup of elements of $GL_2(\hat{\mathbb{Z}})$ congruent to the identity modulo $p^n$), and let $Y_{\infty} = \varprojlim_n Y(p^n)$. Then the $p$-adic completion $\hat{Y}_{\infty}$ of $Y_{\infty}$ is the infinite-level Shimura curve (\cite{ScholzeTorsion}), viewed as an adic space. Recall Scholze's Hodge-Tate period $z_{\mathrm{HT}}$. For the remainder of the introduction, a \emph{generalized $p$-adic modular form of weight $k$} (see Definition \ref{generalizedpadicmodularformdefinition}) will be a function defined on the open subset 
\begin{equation}\label{Vxintro}\mathcal{V}_x  = \{z_{\mathrm{HT}} \neq 0\} \subset Y_{\infty}
\end{equation}
arising from trivializing a global section $w \in \omega^{\otimes k}(Y(p^n))$ for some $n \in \mathbb{Z}_{\ge 0}$ via the \emph{canonical differential} 
$$w_{\mathrm{can}} \in \omega^{\otimes k} \otimes_{\mathcal{O}_Y}\mathcal{B}(\mathcal{V}_x)$$
(see Definition \ref{'zqwdefinition}, (\ref{gluezdRsection}) and (\ref{zqw2})). The section $w_{\mathrm{can}}$ generator of this locally free rank-1 $\mathbb{B}(\mathcal{V}_x)$-module $\omega^{\otimes k} \otimes_{\mathcal{O}_Y}\mathcal{B}(\mathcal{V}_x)$. Here $\mathcal{B}$ is a certain period sheaf that in particular contains the structure sheaf $\mathcal{O}_{\mathcal{V}_x}$; see Section \ref{padicmodularformsection}. The generalized $p$-adic modular form obtained from trivializing $w$ is defined as
$$F := \frac{w}{w_{\mathrm{can}}^{\otimes k}} \in \mathcal{B}(\mathcal{V}_x).$$
Moreover, $F$ satisfies a weight $k$ transformation property analogous to that of complex analytic modular forms:
\begin{equation}\label{introweightidentity}\left(\begin{array}{ccc} a & b\\
c & d\\
\end{array}\right)^*F = \left(\frac{cz_{\mathrm{dR}} + a}{ad -bc}\right)^kF
\end{equation}
for all $\left(\begin{array}{ccc} a & b\\
c & d\\
\end{array}\right)$ in $\Gamma_p(p^n) \subset GL_2(\mathbb{Z}_p)$, where $\Gamma_p(p^n)$ denotes the $p$-component of $\Gamma(p^n)$; see Theorem \ref{weighttheorem}, in particular (\ref{Ftransformationidentity}).
%Here $\theta_t(z_{\mathrm{dR}})$ is a certain section of $\mathcal{B}(\mathcal{V}_x)$ which is the image under a certain map $\theta_t$ (\ref{thetat}) of the 
 Here $z_{\mathrm{dR}}$ is the de Rham period from (\ref{zqw2}).

There is a distinguished subset $\mathcal{Y}^{\mathrm{Ig}} = \{z_{\mathrm{HT}} = \infty\} \subset \mathcal{V}_x$, called the \emph{big Igusa tower}, such that the restriction $F|_{\mathcal{Y}^{\mathrm{Ig}}}$ is a $p$-adic modular form of weight $k$ in Katz's sense (\cite{Katzpamf}). See Theorem \ref{Katzpadicmodularformtheorem}. The $q_{\mathrm{dR}}$-expansion of $F$ is a power series
$$F(q_{\mathrm{dR}}) \in \hat{\mathcal{O}}(\mathcal{V}_x)\llbracket q_{\mathrm{dR}}-1\rrbracket.$$
(More precisely, this is the image of the $q_{\mathrm{dR}}$-expansion $F(q_{\mathrm{dR}})$ from Definition \ref{'zqexpansions} under a certain map $\theta_t : \mathcal{B} \rightarrow \hat{\mathcal{O}}(\mathcal{V}_x)\llbracket q_{\mathrm{dR}} -1\rrbracket$, see (\ref{thetat}) and Definition \ref{thetatpowerseriesdefinition}.) Specializing to any point $y \in \mathcal{Y}^{\mathrm{Ig}} \cap \{z_{\mathrm{HT}}  = \infty\}$, the image of $F(q_{\mathrm{dR}})$ under 
\begin{equation}\label{restrictionintro2}\hat{\mathcal{O}}(\mathcal{V}_x)\llbracket q_{\mathrm{dR}}-1\rrbracket \rightarrow \hat{\mathcal{O}}(\mathcal{Y}^{\mathrm{Ig}})\llbracket q_{\mathrm{dR}}-1\rrbracket \xrightarrow{\hat{\mathcal{O}}(\mathcal{Y}^{\mathrm{Ig}}) \rightarrow \hat{\mathcal{O}}(y) \subset \mathbb{C}_p} \mathbb{C}_p\llbracket q_{\mathrm{dR}}-1\rrbracket
\end{equation}
is the Serre-Tate expansion of $F$ centered at the point $y$. (Note that $y$ does not have to be a CM point, and so the Serre-Tate expansion at $y$ is not necessarily centered at the Serre-Tate canonical lifting.) %Thus, one can view $q_{\mathrm{dR}}$-expansions as analytically continuing families of Serre-Tate expansions on $\mathcal{Y}^{\mathrm{Ig}}$ to $\mathcal{V}_x$. %Hence one can use this ``rigidity'' to do computations with $q_{\mathrm{dR}}$-expansions defined on $U^{\mathrm{can}}$, and in particular the coefficients of these $q_{\mathrm{dR}}$-expansions inherit $p$-adic boundedness properties from $p$-adic boundedness of Serre-Tate expansions.  

In general, the coefficients to $q_{\mathrm{dR}}$-expansions in (\ref{qdRexpansionintro}) of thickened Coleman power series of norm-compatible systems of semi-local units can be $p$-adically unbounded. This stems from the fact that thickened Coleman power series are local objects (i.e. only defined over a formal neighborhood of a CM point and not over the entire Shimura curve), which means that one has little control over their $p$-adic behavior outside the ordinary locus. However, in the modular curve case, the thickened Coleman power series arising from a norm-compatible system of elliptic units is a Siegel modular function, which is a global object, i.e. a function on the universal elliptic curve. Its ``vertical derivative'', i.e. derivative along a coordinate on the formal group of the universal object, is a weight 1 Eisenstein series. Pulling back this Eisenstein series via a universal torsion section, one gets a global section on the infinite-level modular curve. The coordinate $q_{\mathrm{dR}}$, on the other hand, is itself defined on the large open subset $\mathcal{V}_x$ of the infinite-level modular curve. Hence, taking the $q_{\mathrm{dR}}$-expansion, one gets a power series with coefficients that are functions defined on all of $\mathcal{V}_x$, and not just on a small neighborhood of a CM point. In all, one gets a ``$q_{\mathrm{dR}}$-expansion map''
$$\mathcal{C} \rightarrow \hat{\mathcal{O}}(\mathcal{V}_x)\llbracket q_{\mathrm{dR}}-1\rrbracket.$$
As mentioned above, given an elliptic unit $\xi_E \in \mathcal{C}$ the vertical derivative of the thickened Coleman power series of $\xi_E$ is a weight 1 Eisenstein series $F$, %\footnote{More accurately, the Coleman power series of an elliptic unit is a Kato-Siegel $\Theta$-function, whose logarithmic derivative is a weight 1 Eisenstein series.} 
and so the image of $\xi_E$ under the above map is the $q_{\mathrm{dR}}$-expansion of a weight 1 Eisenstein series $F$. The $q_{\mathrm{dR}}$-expansion of the Eisenstein series $F$ is an element
$$F(q_{\mathrm{dR}}) \in \hat{\mathcal{O}}(\mathcal{V}_x)\llbracket q_{\mathrm{dR}}-1\rrbracket \hookrightarrow \hat{\mathcal{O}}(\mathcal{Y}^{\mathrm{Ig}})\llbracket q_{\mathrm{dR}}-1\rrbracket.$$
The fact that the coefficients of $F(q_{\mathrm{dR}})$ are defined on $\mathcal{V}_x \supset \mathcal{Y}^{\mathrm{Ig}}$ allows them to inherit some $p$-adic boundedness properties from the $p$-adic boundedness of the coefficients of the Serre-Tate expansions on $\mathcal{Y}^{\mathrm{Ig}}$. For the purposes of constructing the desired $p$-adic $L$-function, we will only need to $p$-adically bound the constant term of $d_1^jF(q_{\mathrm{dR}})$ where $d_1^jF$ is the $p$-adic Maass-Shimura derivative (see Section \ref{operatorsection}, in particular (\ref{dkjdefinition})) of the weight 1 Eisenstein series $F$. We explain how we do this in the remainder of this section.

Let $\hat{\mathcal{Y}}^{\mathrm{Ig}} \subset \hat{Y}_{\infty}$ denote the $p$-adic completion of $\mathcal{Y}^{\mathrm{Ig}}$ from above; $\hat{\mathcal{Y}}^{\mathrm{Ig}}$ can also be defined as the fiber $\pi_{\mathrm{HT}}^{-1}(\{\infty\})$ of the Hodge-Tate period map $\pi_{\mathrm{HT}} : \hat{Y}_{\infty} \rightarrow \mathbb{P}^1$ of \cite{ScholzeTorsion}, see \cite[Section 4]{CaraianiScholze}. Then $\hat{\mathcal{Y}}^{\mathrm{Ig}}$ is a pro-finite \'{e}tale cover of the usual Igusa tower $Y^{\mathrm{Ig}}$ of \cite{KatzIgusa} and \cite{KatzCM}. (See Section \ref{furtherIgusasection} where we recall the definition of $Y^{\mathrm{Ig}}$, and see Definition \ref{mathcalYIgDefinition} and (\ref{hatmathcalYIg}) for the definition of $\hat{\mathcal{Y}}^{\mathrm{Ig}}$.) After multiplying by a constant if needed, we may assume without loss of generality that $d_1^jF$ extends to a section on the natural formal model (see Definition \ref{formalmodeldefinition} for our notion of formal model) of $Y^{\mathrm{Ig}}$, which parametrizes deformations of $p$-divisible groups of false elliptic curves together with trivialization of their connected parts. In this case we call $d_1^jF$ ``normalized''. By (\ref{restrictionintro2}), the restriction $d_1^jF|_{\hat{\mathcal{Y}}^{\mathrm{Ig}}}(q_{\mathrm{dR}})$ interpolates all ordinary Serre-Tate expansions. As Serre-Tate expansions of normalized $p$-adic modular forms are well-known to have $p$-integral coefficients (\cite{KatzST}), the restricted $q_{\mathrm{dR}}$-expansion $d_1^jF|_{\hat{\mathcal{Y}}^{\mathrm{Ig}}}(q_{\mathrm{dR}})$ has $p$-integral coefficients, i.e. coefficients in $\mathcal{O}^+(\hat{\mathcal{Y}}^{\mathrm{Ig}})$. Let $\hat{\mathcal{V}}_x \subset \hat{Y}_{\infty}$ be the adic space obtained by $p$-adically competing $\mathcal{V}_x$ from (\ref{Vxintro}), and let 
$$U^{\mathrm{can}} \subset \hat{\mathcal{V}}_x$$
be the canonical locus; for our purposes, this is the locus where the canonical subgroup exists and is parametrized by the first basis element of the $\Gamma(p^{\infty})$-level structure on $\hat{Y}_{\infty}$. (See (\ref{Ucan}) for the definition of $U^{\mathrm{can}}$.) Then $\hat{\mathcal{Y}}^{\mathrm{Ig}} \subset U^{\mathrm{can}}$ and $U^{\mathrm{can}}$ is perfectoid in the sense of \cite{Scholzeperf}. And one can hope to extend the $p$-integrality of $d_1^jF|_{\hat{\mathcal{Y}}^{\mathrm{Ig}}}(q_{\mathrm{dR}})$ to the $p$-integrality of $d_1^jF|_{U^{\mathrm{can}}}(q_{\mathrm{dR}})$ by using the perfectoidness of $U^{\mathrm{can}}$.
%In fact, one can show such an ``overconvergence'' of $p$-integrality on a neighborhood in an \emph{integral perfectoid} model $\hat{Y}_{\infty}^+$ of $\hat{Y}_{\infty}$. This integral model is obtained by taking an inverse limit of the Katz-Mazur models of finite-level modular curves. 

We will show that this is indeed true for the constant term of $d_1^jF|_{U^{\mathrm{can}}}(q_{\mathrm{dR}})$. To this end, one first defines certain affinoid perfectoid neighborhoods $\hat{\mathcal{Y}}^{\mathrm{Ig}}(\epsilon) \subset U^{\mathrm{can}}$ indexed by $\epsilon > 0$ (see (\ref{U}) and (\ref{hatmathcalYIg})) such that $\hat{\mathcal{Y}}^{\mathrm{Ig}}(0) = \hat{\mathcal{Y}}^{\mathrm{Ig}}$ and 
\begin{equation}\label{Ucanintro}U^{\mathrm{can}} = \bigcup_{0 \le \epsilon < p/(p+1)}\hat{\mathcal{Y}}^{\mathrm{Ig}}(\epsilon).
\end{equation}
These affinoid perfectoids have natural integral models $\hat{\mathcal{Y}}^{\mathrm{Ig}}(\epsilon)^+$ (see (\ref{Uplus}) and Definition \ref{hatmathcalYIgepsilonDefinition}) defined using the Katz-Mazur model $Y_{\infty}^+$ of $Y_{\infty}$ and the theory of the canonical subgroup (\cite[Chapter 3]{Katzpamf}, \cite{Kassaei}). Using the universal property of the universal objects lying over the $\hat{\mathcal{Y}}^{\mathrm{Ig}}(\epsilon)$, one can show that the $\hat{\mathcal{Y}}^{\mathrm{Ig}}(\epsilon)$ satisfy the relation $\varprojlim_{0 < \epsilon < p/(p+1)}\hat{\mathcal{Y}}^{\mathrm{Ig}}(\epsilon) \sim \mathcal{Y}^{\mathrm{Ig}}$ in the notation of \cite[Definition 2.4.1]{ScholzeWeinstein} (see (\ref{thirdpresentation}) and (\ref{daggerequalities})). This in turn implies via an argument using the constructible topology that the constant term $f \in \hat{\mathcal{O}}(\mathcal{V}_x)$ of $d_1^jF(q_{\mathrm{dR}}) \in \hat{\mathcal{O}}(\mathcal{V}_x)\llbracket q_{\mathrm{dR}}-1\rrbracket$ has $|f| \le 1$ on $\hat{\mathcal{Y}}^{\mathrm{Ig}}(\epsilon/p^m)$ for some $m \gg 0$, where $|\cdot |$ denotes the supremum norm. We prove a ``spreading out $p$-integrality theorem'' which shows that this $p$-integrality $|f|\le 1$ on $\hat{\mathcal{Y}}^{\mathrm{Ig}}(\epsilon/p^m)$ extends to $|f| \le 1$ on all of $\hat{\mathcal{Y}}^{\mathrm{Ig}}(\epsilon)$, and then via (\ref{Ucanintro}) to all of $U^{\mathrm{can}}$. The crux of the proof of this spreading out theorem arises from the affinoid perfectoidness of $\hat{\mathcal{Y}}^{\mathrm{Ig}}(\epsilon)^+$ (Theorem \ref{perfectoidpropertytheorem}): we have a distinguished element (see (\ref{gdefinition}))
$$g = \left(\begin{array}{ccc} 1 & 0\\
0 & p 
\end{array}\right) \in GL_2(\mathbb{Q}_p)$$
which induces a map
$$g : \hat{\mathcal{Y}}^{\mathrm{Ig}}(\epsilon/p^m)^+ \rightarrow \hat{\mathcal{Y}}^{\mathrm{Ig}}(\epsilon/p^{m-1})^+$$
via the $GL_2(\mathbb{Q}_p)$-action on $Y_{\infty}$. This gives a sequence of isomorphisms
$$\cdots \rightarrow \hat{\mathcal{Y}}^{\mathrm{Ig}}(\epsilon/p^m)^+ \xrightarrow{g} \hat{\mathcal{Y}}^{\mathrm{Ig}}(\epsilon/p^{m-1})^+ \rightarrow \cdots$$
for $|\epsilon/p^{m-1}| < p^{-p/(p+1)}$, whose pullbacks $g^*$ each act as a lift of relative Frobenius on $f$.\footnote{For this, we two crucial properties of $f$ to show that the pullback $g^*$ acts as a lift of Frobenius on $f$: (1) $f$ is defined on $\hat{\mathcal{V}}_x$ and (2) the weight $k$ transformation property (\ref{introweightidentity}). See Theorems \ref{crystalline2}, \ref{pintegraltheorem} and \ref{pintegraltheorem2}.} Using the action of $g^*$ and the corresponding congruence it satisfies (see (\ref{Frobeniuslift})), one can transport $p$-integrality of $f$ on $\hat{\mathcal{Y}}^{\mathrm{Ig}}(\epsilon/p^m)$ to $p$-integrality on $\hat{\mathcal{Y}}^{\mathrm{Ig}}(\epsilon/p^{m-1})$, and then by induction to $p$-integrality on all of $\hat{\mathcal{Y}}^{\mathrm{Ig}}(\epsilon)$. See Section \ref{spreadingoutsection} for details on this argument, and Theorem \ref{crystalline2} for a precise statement of the spreading out $p$-integrality theorem. See Theorem \ref{pintegraltheorem} for the precise statement on the $p$-integrality of the constant term of $d_1^jF(q_{\mathrm{dR}})$ on $U^{\mathrm{can}}$.  

Let $F^{\flat}$ denote the $p$-depletion of the weight 1 Eisenstein series $F$ (see Section \ref{depletionsection} and \cite[Equation (35)]{BCDDPR}), and let $d_1^jF^{\flat}$ denote the $j$-fold $p$-adic Maass-Shimura derivative of $F$. It can be shown that $d_1^jF^{\flat}|_{\hat{\mathcal{Y}}^{\mathrm{Ig}}}$ is equal to the $j$-fold Atkin-Serre derivative $\theta_{\mathrm{AS}}^jF^{\flat}|_{\hat{\mathcal{Y}}^{\mathrm{Ig}}}$ (see Proposition \ref{prop1}). Thus, by known continuity properties of $\theta_{\mathrm{AS}}^jF^{\flat}|_{\hat{\mathcal{Y}}^{\mathrm{Ig}}}$, $d_1^jF^{\flat}|_{\hat{\mathcal{Y}}^{\mathrm{Ig}}}$ varies continuously in $j \in \mathbb{Z}/(p-1) \times \mathbb{Z}_p$ (see Proposition \ref{prop4}). This continuity extends, again by an argument using the constructible topology (see proof of Theorem \ref{pintegraltheorem2} (3)), to continuity in $j$ of $d_1^jF^{\flat}|_{\hat{\mathcal{Y}}^{\mathrm{Ig}}(\epsilon/p^m)}$ for some $m \gg 0$. By the density of the formal models $\hat{\mathcal{Y}}^{\mathrm{Ig}}(\epsilon/p^m)^+ \subset \hat{\mathcal{Y}}^{\mathrm{Ig}}(\epsilon)^+$ (see Proposition \ref{prop2}), this extends to continuity in $j$ of $d_1^jF^{\flat}$ (see Theorem \ref{pintegraltheorem2} (3)).

Specializing the coefficients of $d_1^jF^{\flat}|_{U^{\mathrm{can}}}(q_{\mathrm{dR}})$ to a CM point $y \in U^{\mathrm{can}}$, we thus obtain a power series
$$d_1^jF^{\flat}(y)(q_{\mathrm{dR}}) \in \mathcal{O}_{\mathbb{C}_p} + (q_{\mathrm{dR}}-1)\mathbb{C}_p\llbracket q_{\mathrm{dR}}-1\rrbracket.$$
Recall that $F$ was assumed to be normalized above; in general, the above argument shows that 
$$d_1^jF^{\flat}(y)(q_{\mathrm{dR}}) \in \left(\mathcal{O}_{\mathbb{C}_p} + (q_{\mathrm{dR}}-1)\mathbb{C}_p\llbracket q_{\mathrm{dR}}-1\rrbracket\right) [1/p].$$
By taking a character-twisted sum of $d_1^jF^{\flat}(y^{\sigma})(q_{\mathrm{dR}})$ over a finite set $\{y^{\sigma}\}$ of $\mathrm{Gal}(\overline{K}/K)$-conjugates of $y$, one gets a power series
$$\mathcal{L}_E(j)(q_{\mathrm{dR}}) \in \mathbb{C}_p\llbracket q_{\mathrm{dR}}-1\rrbracket.$$
The constant term $\mathcal{L}_E(j)$ of this power series is our $p$-adic $L$-function $\mathcal{L}_E$. By the continuity in $j$ of $d_1^jF^{\flat}$, we get a $p$-adically continuous function 
$$\mathcal{L}_E : \mathbb{Z}/(p-1) \times \mathbb{Z}_p \rightarrow \mathbb{C}_p, \hspace{1cm} j \mapsto \mathcal{L}_E(j).$$
See Theorem \ref{interpolation2} for the construction of $\mathcal{L}_E$ (where it is denoted by $\mathcal{L}_{\lambda_E}$, for $\lambda_E$ the infinity type $(1,0)$ Hecke character attached to $E$).

\subsubsection{\textbf{Step (2).3: finishing the construction of the Coleman map}}By composing (\ref{qdRexpansionintro}) with a specialization of coefficients to an appropriate CM point $y \in U^{\mathrm{can}}$
$$\hat{\mathcal{O}}(\mathcal{V}_x) \rightarrow \hat{\mathcal{O}}(\mathcal{V}_x)(y) \subset \mathbb{C}_p,$$
we get a map
\begin{equation}\label{desiredColemanmapintro}\delta : \mathbb{U} \rightarrow \mathbb{C}_p\llbracket q_{\mathrm{dR}}-1\rrbracket
\end{equation}
whose restriction to some rank 1 subspace $\mathbb{U}_0 \subset \mathbb{U}$ gives a map
\begin{equation}\label{desiredColemanmapintro2}\delta : \mathbb{U}_0 \rightarrow \mathbb{C}_p\llbracket q_{\mathrm{dR}}-1\rrbracket^{\circ}
\end{equation}
where the target denotes the $\mathbb{C}_p$-submodule of power series of $\mathbb{C}_p\llbracket q_{\mathrm{dR}}-1\rrbracket$ whose $j$-fold $p$-adic Maass-Shimura derivatives (see Section \ref{operatorsection}) evaluated at $q_{\mathrm{dR}} = 1$ vary $p$-adically continuously in $j \in \mathbb{Z}/(p-1) \times \mathbb{Z}_p$. (For a precise statement, see Theorem \ref{Colemanmapthm} below.)
As mentioned above, one can compute the Coleman power series of elliptic units in terms of Kato-Siegel $\Theta$-functions. Pulling back the logarithmic derivatives of these $\Theta$-functions to $Y_{\infty}$ via universal torsion sections, one obtains weight 1 Eisenstein series. Letting $F$ be an appropriate weight 1 Eisenstein series, and taking this to be the global $p$-adic modular form in the previous section, we see that (\ref{desiredColemanmapintro}) sends the elliptic unit $\xi_E$ from (\ref{xiE}) to 
$$C_E \cdot \mathcal{L}_E \in \mathbb{C}_p\llbracket q_{\mathrm{dR}}-1\rrbracket$$
for some constant $C_E \in \mathbb{C}_p^{\times}$:
\begin{equation}\label{introreciprocity}\delta(\xi_E) = C_E \cdot \mathcal{L}_E.
\end{equation}
This identity is the ``explicit reciprocity law'' to which we alluded in the outline of Step (2) above.

\subsubsection{\textbf{Step (3)}}The $r = 0$ case now follows almost immediately from Step (1) and arguments similar to those of \cite[Section 11]{RubinMC}. We will focus on describing the $r = 1$ case in this section. Step (3) is analogous to a step in previous results on $p$-converse theorems. Namely, we wish to relate elliptic units over the CM field $K$ to Heegner points over $K$, and this is done through relating $GL_1/K$ and $GL_2/\mathbb{Q}$ anticyclotomic $p$-adic $L$-functions through a factorization of $p$-adic $L$-functions. The idea relating elliptic units and Heegner points via factorizations of Rankin-Selberg $p$-adic $L$-functions into Katz-type $p$-adic $L$-functions was introduced in \cite{BDP2}, from which our method draws inspiration. See also \cite{BurungaleTian} and \cite{BurungaleCastella} in the $r = 1$ ordinary case. 

Recall $E/\mathbb{Q}$ is our elliptic curve with CM by $\mathcal{O}_K$, and let $\lambda_E$ be the Hecke character over $K$ of infinity type $(1,0)$ attached to $E/K$ by the theory of complex multiplication. In particular, $L(E,s) = L(\lambda_E,s)$. First, we choose a Rankin-Selberg pair $(g,\chi_0)$ which is a suitable twist of $E$ (in the sense of Section \ref{goodtwistsection}); in particular, $(g,\chi_0)$ is chosen in a way (see Choice \ref{twistchoice}) so that
\begin{equation}\label{intrononvanishing}L(\lambda_E^c(\chi_0/\chi_0^c),1) \neq 0
\end{equation}
where for a Hecke character $\chi : \mathbb{A}_K^{\times}/K^{\times} \rightarrow \mathbb{C}^{\times}$, $\chi^c(x) = \chi(x^c)$ where $c : \mathbb{A}_K^{\times} \xrightarrow{\sim} \mathbb{A}_K^{\times}$ is complex conjugation. The existence of such a $\chi_0$ is highly nontrivial, using the result of Rohrlich \cite{Rohrlich} on the non-vanishing of twists of $L$-values in anticyclotomic towers. In particular, the results of op. cit. guaranteeing the existence of an appropriate $\chi_0$ are computationally ineffective in our situation. By the Artin formalism, we have the factorization
\begin{equation}\label{introfactorization}L(\pi_g \times \pi_{\chi_0},s-\frac{1}{2}) = L(\lambda_E,s)L(\lambda_E^c(\chi_0/\chi_0^c),s)
\end{equation}
where $\pi_g$ and $\pi_{\chi_0}$ are the automorphic representations of $GL_2(\mathbb{A}_{\mathbb{Q}})$ attached to $g$ and the theta series $\theta_{\chi_0}$ of $\chi_0$, respectively, and $L(\pi_g\times \pi_{\chi_0},s)$ is the (unitarily normalized, i.e. centered at $s = 1/2$) Rankin-Selberg $L$-function. (This slightly differs from the notation we will use in the main body of the paper: in Section \ref{RSsection} we will let $\pi_g$ denote the Jacquet-Langlands correspondence, which is a $D$-automorphic representation where $D$ is a quaternion algebra, of the $GL_2$-automorphic representation $\mathrm{JL}(\pi_g)$ attached to $g$. We use this notation in order to conform with the notation of \cite{YZZ} and \cite{Brooks}.) Then (\ref{intrononvanishing}) along with the above factorization 
\begin{equation}\label{analyticrankintro}\begin{split}\mathrm{ord}_{s = 1}L(E/\mathbb{Q},s) = \mathrm{ord}_{s = 1}L(\lambda_E,s) &\overset{(\ref{intrononvanishing})}{=} \mathrm{ord}_{s = 1}L(\lambda_E,s) + \mathrm{ord}_{s = 1}L(\lambda_E^c(\chi_0/\chi_0^c),s) \\
&\overset{(\ref{introfactorization})}{=} \mathrm{ord}_{s = 1/2}L(\pi_g \times \pi_{\chi_0},s).
\end{split}
\end{equation}
 
In all, we will have three $p$-adic $L$-functions: $\mathcal{L}_E$ (from Theorem \ref{interpolation2}, where it is denoted by $\mathcal{L}_{\lambda_E}$), $\mathcal{L}_E'$ (from Theorem \ref{interpolation2}, where it is denoted by $\mathcal{L}_{\psi^c\chi_0}$, where $\psi = \lambda_E/\chi_0$), and $\mathcal{L}_{g,\chi_0}$ (from Theorem \ref{GL2maintheorem}, where it is denoted by $\mathcal{L}_{g \times \chi_0}$). By comparing interpolation properties of these $p$-adic $L$-functions (see Theorems \ref{interpolation2} and \ref{GL2maintheorem}), one can prove a factorization of $p$-adic $L$-functions: for all $j \in \mathbb{Z}/(p-1) \times \mathbb{Z}_p$,
\begin{equation}\label{factorizationintro}\mathcal{L}_E(j-1)\cdot \mathcal{L}_E'(j) = C\cdot \mathcal{L}_{g,\chi_0}(j)^2
\end{equation}
(see (\ref{factorization})) where $C \in \mathbb{C}_p^{\times}$ is some constant. We have a $p$-adic Waldspurger formula (see (\ref{padicWaldspurger}))
$$\mathcal{L}_{g,\chi_0}(0) = \log_g(P_{g,\chi_0}) \cdot C'$$
for some Heegner point $P_{g,\chi_0}$, where $C' \in \mathbb{C}_p^{\times}$. Hence specializing (\ref{factorizationintro}) to $j = 0$, we get
$$\mathcal{L}_E(-1)\cdot \mathcal{L}_E'(0) = CC'^2\cdot\log_g^2(P_{g,\chi_0}).$$
Using the interpolation property of $\mathcal{L}_E'$ (Proposition \ref{nonzerotheorem}), one can show that there exists $C'' \in \mathbb{C}_p^{\times}$ such that 
$$\mathcal{L}_E'(0) = C''\cdot L(\lambda_E^c(\chi_0/\chi_0^c),1) \overset{(\ref{intrononvanishing})}{\neq} 0.$$
Therefore
$$\mathcal{L}_E(-1) \neq 0 \iff P_{g,\chi_0} \; \text{is non-torsion},$$
which by Yuan-Zhang-Zhang's Gross-Zagier formula on Shimura curves (\cite{YZZ}) happens if and only if 
$$\mathrm{ord}_{s = 1}L(E/\mathbb{Q},s) \overset{(\ref{analyticrankintro})}{=} \mathrm{ord}_{s = 1/2}L(\pi_g \times \pi_{\chi_0},s) = 1.$$ 
We also construct a $p$-adic primitive of the map (\ref{desiredColemanmapintro2}) satisfying an the explicit reciprocity law (\ref{introreciprocity}) relating the value on elliptic units to $\mathcal{L}_E(-1)$:
\begin{equation}\label{desiredColemanmapintro3}\delta' := D_1^{-1}\delta|_{q_{\mathrm{dR}}  = 1} : \mathbb{U}_0 \hat{\otimes}_{\mathcal{O}_{K_p}} \mathcal{O}_{\mathbb{C}_p} \rightarrow \mathbb{C}_p, \hspace{1cm} \delta'(\xi_E) = C_E \cdot \mathcal{L}_E(-1)
\end{equation}
(see Corollary \ref{Colemanmapthm}, extended to $j = -1$ by (\ref{delta'}), for a precise statement). Using Steps (1) and (2),\footnote{In the $r = 0$ case, our Rubin-type main conjecture is the main conjecture of \cite[Corollary 11.13(i)]{RubinMC} (under the assumptions of loc. cit.) after tensoring with $\otimes_{\mathbb{Z}_p}\mathbb{Q}_p$. When $r = 1$, the left-hand side of (\ref{RMCintro}) in our Rubin-type main conjecture is isomorphic to $\mathbb{U}_0/\mathcal{C}$, where $\mathbb{U}_0$ is the saturation of the module of elliptic units $\mathcal{C}$ (see (\ref{saturation}), where $\mathbb{U}_0$ is denoted $U_0$). This $\mathbb{U}_0$ is also equal to the domain of our map (\ref{desiredColemanmapintro3}). Thus the map (\ref{desiredColemanmapintro3}), which is intrinsic to the module $\mathcal{C}$, ``cuts out'' the local condition of the $\Lambda_E$-adic Selmer group in our Rubin-type main conjecture. See Theorem \ref{RMC} for a precise statement of the Rubin-type main conjectures. In the $r = 0$ case, we take $\mathbb{U}_0$ as in (\ref{kernelequalities}), and in the $r = 1$ we take $U_2$ as in Choice \ref{U2finallyfixed}. See also Remark \ref{U2remark} for a discussion on the different roles $\delta$ from (\ref{introreciprocity}) plays in the rank 0 and 1 cases; note that $\delta$ is equal to a nonzero multiple of $\delta_E$ by (\ref{deltascoincide}).}  one can prove that 
$$\mathrm{corank}_{\mathbb{Z}_p}\mathrm{Sel}_{p^{\infty}}(E/\mathbb{Q}) = 1 \implies \mathcal{L}_E(-1) \neq 0.$$ From this, the rank 1 $p$-converse theorem follows. See Section \ref{rank1section} for details.

\subsection{Applications to classical Diophantine problems}Theorem \ref{MainTheorem} has applications to certain classical Diophantine equations. First, let $E_d : x^3 + y^3 = d$ denote the cubic twist family, which has CM by $\mathbb{Z}[\sqrt{-3}]$. It is known that there exists $(x,y) \in \mathbb{Q}^{\oplus 2}$ such that $x^3 + y^3 = d$ if and only if $\mathrm{rank}_{\mathbb{Z}}E_d(\mathbb{Q}) > 0$. Sylvester (\cite{Sylvester}) and later Selmer (\cite{Selmer}) and Satg\'{e} (\cite{Satge}) computed the 3-descent of $E_d$ for primes $d$ and proved
$$\mathrm{corank}_{\mathbb{Z}_3}\mathrm{Sel}_{3^{\infty}}(E_d/\mathbb{Q}) = \begin{cases} 0 & d \equiv 2, 5 \pmod{9}\\
1 & d \equiv 4,7,8 \pmod{9}
\end{cases}.$$
A conjecture that has come to be known as ``Sylvester's conjecture'' (an attribution stemming from the work of Sylvester \cite{Sylvester}, and Selmer \cite{Selmer} on the 3-descent of the family $E_d$) states:
\begin{conjecture}[Sylvester's Conjecture]For primes $d$,
$$d \equiv 4,7,8 \pmod{9} \implies \mathrm{rank}_{\mathbb{Z}}E_d(\mathbb{Q}) = 1,$$
and in particular there exist $x,y \in \mathbb{Q}$ such that $x^3 + y^3 = d$.
\end{conjecture}
Applying Theorem \ref{MainTheorem} to $K = \mathbb{Q}(\sqrt{-3})$ and $p = 3$, we resolve Sylvester's conjecture.

\begin{theorem}[Theorem \ref{Sylvesterthm}]Sylvester's conjecture is true.
\end{theorem}

See Section \ref{Sylvestersection} for more details.

Another family of elliptic curves of arithmetic interest is the congruent number family $E^d : y^2 = x^3 - d^2x$. Recall that $d \in \mathbb{Z}_{> 0}$ is called a congruent number if it is the area of a right triangle with rational side lengths. It is known that for a positive integer $d$, $\mathrm{rank}_{\mathbb{Z}}E^d(\mathbb{Q}) > 0$ if and only if $d$ is a congruent number. (See \cite[Introduction]{Tunnell}.) Moreover, it is clear that $dn^2$ is congruent if and only if $d$ is congruent, and thus the study of congruent numbers reduces to the study of squarefree congruent numbers. The ``congruent number problem'' states:
\begin{conjecture}[Congruent Number Problem]\label{congruentnumberproblem}Under the natural ordering, 100\% of squarefree positive integers $d \equiv 1,2,3 \pmod{8}$ are not congruent numbers, and all squarefree positive integers $d \equiv 5,6,7 \pmod{8}$ are congruent numbers.
\end{conjecture}
Smith \cite{Smith} was able to vastly improve and extend previously known 2-descent methods of Heath-Brown \cite{HeathBrown}, Monsky \cite{Monsky} and Kane \cite{Kane} in order to get precise distribution results on $\mathrm{corank}_{\mathbb{Z}_2}\mathrm{Sel}_{2^{\infty}}(E^d/\mathbb{Q})$, showing for 100\% of squarefree $d$: 
$$\mathrm{corank}_{\mathbb{Z}_2}\mathrm{Sel}_{2^{\infty}}(E^d/\mathbb{Q}) = \begin{cases} 0 & d \equiv 1,2,3 \pmod{8}\\
1 & d \equiv 5,6,7 \pmod{8}
\end{cases}.$$
In particular, the $d\equiv 1,2,3 \pmod{8}$ part of the congruent number problem immediately follows. %Further work of Smith \cite{Smith2} also verified that a positive proportion of squarefree positive $d \equiv 5,6,7 \pmod{8}$ are congruent numbers.
 By combining Theorem \ref{MainTheorem} (with $K = \mathbb{Q}(i)$ and $p = 2$) with the results in \cite{Smith}, we obtain the following.

\begin{theorem}[\cite{Smith} and Theorem \ref{congruentnumberthm}]Under the natural ordering, 100\% of positive squarefree integers $d \equiv 1,2,3 \pmod{8}$ are not congruent numbers, and 100\% of positive squarefree integers $d \equiv 5,6,7 \pmod{8}$ are congruent numbers.
\end{theorem}

Moreover, this shows that Goldfeld's conjecture (\cite{Goldfeld}, or see Conjecture \ref{Goldfeldconjecture} below) on ranks of elliptic curves in quadratic twist families is true for the congruent number family $E^d$, and gives the first case for which the full Goldfeld conjecture is established. See Section \ref{Goldfeldsection} for more details. 

\subsection{Structure of the paper}In Section \ref{Ysection}, we recall background on Shimura curves when viewed in various categories, such as schemes, adic spaces, and formal schemes, as well as in pro-categories. We also study the perfectoid structure of infinite-level Shimura curves in the latter half of this section, in particular the action of lifts of Frobenius and how it spreads out $p$-integrality on the infinite-level canonical locus. In Section \ref{ShimuraCurveSection}, we recall Scholze's period sheaves and study their structure on the infinite-level Shimura curve. In Section \ref{zdRqdRqdRexpsection}, we introduce the coordinates $z_{\mathrm{dR}}$ and $q_{\mathrm{dR}}$, as well as the notion of $q_{\mathrm{dR}}$-expansions. Later on in this section, we introduce a notion of ``generalized $p$-adic modular forms'', which generalizes Katz's classical notion of $p$-adic modular forms as functions on the Igusa tower. We also define weight-raising (or ``$p$-adic Maass-Shimura operators'') acting on these forms. In Sections \ref{padicLfunctionsection} and \ref{padicLfunctionsection2}, we use $q_{\mathrm{dR}}$-expansions of Eisenstein series and cusp forms in order to construct the relevant $p$-adic $L$-functions that appear in Steps (2) and (3), and establish the factorization identity of Step (3). In Section \ref{backgroundsection}, we recall some notions from classical Iwasawa theory such as Coleman power series, and also Tsuji's method for constructing relative versions of Coleman power series. This gives the basis for constructing the maps (\ref{desiredColemanmapintro2}) and (\ref{desiredColemanmapintro3}) above. In Section \ref{RMCsection}, we carry out Step (1) of our argument, assuming the elliptic units main conjecture (the proof of which is relegated to the Appendix, Section \ref{EMCsection}). We also prove an explicit reciprocity law, relating the images of elliptic units under the maps (\ref{desiredColemanmapintro2}) and (\ref{desiredColemanmapintro3}) to our $p$-adic $L$-functions. In Section \ref{rank0section}, we establish the $r = 0$ part of Theorem \ref{MainTheorem}, and in Section \ref{rank1section}, we establish the $r = 1$ part. In the Appendix (Section \ref{EMCsection}), we prove the elliptic units main conjecture used in Section \ref{RMCsection} by invoking Johnson-Leung-Kings's work on the equivariant main conjecture. As input for applying the results of op. cit., we prove a vanishing result for the total $\mu$-invariant of the cyclotomic tower over the CM field. 
\\

\noindent\textbf{Acknowledgements} The author would like to thank Shilin Lai, Congling Qiu, Peter Scholze, Christopher Skinner, Ye Tian, Shou-Wu Zhang and Wei Zhang for helpful discussions.

\section{Adic Shimura Curves and Their Formal Models}\label{Ysection}

Fix an algebraic closure $\overline{\mathbb{Q}}$ of $\mathbb{Q}$, and view all number fields as embedded in $\overline{\mathbb{Q}}$. Let $p$ be a prime and fix an algebraic closure $\overline{\mathbb{Q}}_p$ of $\mathbb{Q}_p$. Let $\mathbb{C}_p$ denote the $p$-adic completion of $\overline{\mathbb{Q}}_p$, with valuation ring $\mathcal{O}_{\mathbb{C}_p}$. Fix an embedding 
\begin{equation}\label{fixembeddings}i_p : \overline{\mathbb{Q}} \hookrightarrow \overline{\mathbb{Q}}_p,
\end{equation}
and denote the $p$-adic completion of any field $M \subset \overline{\mathbb{Q}} \overset{i_p}{\hookrightarrow} \overline{\mathbb{Q}}_p \subset \mathbb{C}_p$ by $M_p$. Let $K$ be an imaginary quadratic field, and let $\frak{p}$ be fixed by this embedding so that $K_p$ is the completion of $K$ at $\frak{p}$. We will let $\mathrm{ord}_p : \overline{\mathbb{Q}}_p \rightarrow \mathbb{Q} \cup \{\infty\}$ denote the $p$-adic valuation with $\mathrm{ord}_p(p) = 1$. Given any element $\pi$ in the maximal ideal of $\mathcal{O}_{\mathbb{C}_p}$, we will let $\mathrm{ord}_{\pi} : \mathbb{C}_p \rightarrow \mathbb{Q} \cup \{\infty\}$ denote the $p$-adic valuation $\mathrm{ord}_{\pi}(\pi) = 1$. Letting $(\pi)$ denote the $\mathcal{O}_{\mathbb{C}_p}$-ideal generated by $\pi$, we also let $\mathrm{ord}_{(\pi)} := \mathrm{ord}_{\pi}$. 

Given a number field $L$, let $\mathbb{A}_L$ denote the ad\`{e}les and $\mathbb{A}_L^{\times}$ denote the id\`{e}les over $L$. Given a formal product of (finite or infinite) primes $\mathcal{S}$ of $\mathcal{O}_L$, we let 
$$\mathbb{A}_L^{(\mathcal{S})} := \left(\prod_{v\nmid \mathcal{S}}L_v \times \prod_{v|\mathcal{S}}\{0\}\right) \cap \mathbb{A}_L, \hspace{1cm} \mathbb{A}_L^{\times,(\mathcal{S})} = \left(\prod_{v\nmid \mathcal{S}}L_v^{\times} \times \prod_{v|\mathcal{S}}\{1\}\right) \cap \mathbb{A}_L^{\times}$$
with the subspace topologies induced by the standard topologies on $\mathbb{A}_L$ and $\mathbb{A}_L^{\times}$. In particular, $\mathbb{A}_L^{(\infty)}$ is the finite ad\`{e}les over $L$ and $\mathbb{A}_L^{\times,(\infty)}$ is the finite id\`{e}les over $L$.

\subsection{Algebraic Shimura curves}\label{algebraicYsection}
\begin{definition}Given
\begin{itemize}
\item an indefinite quaternion algebra $D$ over $\mathbb{Q}$ (i.e. $D \otimes_{\mathbb{Q}} \mathbb{R} \cong M_2(\mathbb{R})$) with discriminant $\mathrm{disc}(D)$ (which is in this case a product of an even number of rational primes) and with $p \nmid \mathrm{disc}(D)$, and
\item a compact open subgroup $\Gamma \subset D^{\times}(\mathbb{A}_{\mathbb{Q}}^{(\infty)})$,
\end{itemize}
we will let 
$$\mathbb{Y}(\Gamma) \rightarrow \mathrm{Spec}(\mathbb{Q})$$
denote the associated Shimura curve (see \cite[Chapter 3.1]{YZZ}, or \cite[Section 2]{Buzzard} and \cite[Example 3.4]{Milne} for more detailed expositions). It is a scheme over $\mathrm{Spec}(\mathbb{Q})$. 
\end{definition}
We will sometimes denote $\mathbb{Y}(\Gamma)$ by $\mathbb{Y}$ when the above data is obvious. When $D \neq M_2(\mathbb{Q})$, $\mathbb{Y}(\Gamma)$ is compact. When $D = M_2(\mathbb{Q})$, $\mathbb{Y}(\Gamma)$ is not compact. In any case, fix an isomorphism 
$$\iota_{\infty} : D \otimes_{\mathbb{Q}} \mathbb{R} \xrightarrow{\sim} M_2(\mathbb{R}),$$
and use this to identify $D \otimes_{\mathbb{Q}}\mathbb{R}$ and $M_2(\mathbb{R})$. We have the complex analytic uniformization
\begin{equation}\label{doublequotient}\mathbb{Y}(\Gamma)(\mathbb{C}) = D^{\times} \backslash \mathcal{H}^{\pm} \times D^{\times}(\mathbb{A}_{\mathbb{Q}}^{(\infty)}) /\Gamma,
\end{equation}
where $\mathcal{H}^{\pm} = \mathbb{C} \setminus \mathbb{R}$. Here, $D^{\times}$ acts on the direct product by the left diagonal action. Let $\tau \in \mathcal{H}^{\pm}$ denote the standard coordinate. Then $D^{\times}$ acts by the left modular action of $D^{\times} \subset (D\otimes_{\mathbb{Q}} \mathbb{R})^{\times} \overset{\iota_{\infty}}{=} GL_2(\mathbb{R})$ on $\mathcal{H}^{\pm}$,
$$\left(\begin{array}{ccc} a & b\\
c & d\\
\end{array}\right) \cdot \tau = \frac{a\tau + b}{c\tau + d}, \hspace{1cm} \left(\begin{array}{ccc} a & b\\
c & d\\
\end{array}\right) \in (D\otimes_{\mathbb{Q}}\mathbb{R})^{\times},
$$
and by left multiplication on $D^{\times}(\mathbb{A}_{\mathbb{Q}}^{(\infty)})$. Also, $\Gamma$ acts on the second factor of the direct product in (\ref{doublequotient}) by right multiplication in $D^{\times}(\mathbb{A}_{\mathbb{Q}}^{(\infty)})$. 

\begin{definition}\label{abelianvarietydefinition}Let $S$ be a scheme. Recall that a \emph{$g$-dimensional abelian variety} $A \rightarrow S$ is a smooth, proper group scheme of relative dimension $g$ and with geometrically connected fibers of dimension $g$. An \emph{elliptic curve} is a 1-dimensional abelian variety, and an abelian surface is a 2-dimensional abelian variety. Let $\mathcal{O} \subset D$ be an order, i.e. a subring that is finitely generated as a $\mathbb{Z}$-module. Suppose that $\mathrm{disc}(D) \in \mathcal{O}_S(S)^{\times}$. A \emph{false elliptic curve} is an abelian surface $A \rightarrow S$ together with an embedding $i : \mathcal{O} \hookrightarrow \mathrm{End}_S(A)$. We call this $i$ an \emph{$\mathcal{O}$-endomorphism structure}. 
\end{definition}

Recall that an order is called maximal if it is not strictly contained in another order. 
\begin{choice}\label{maximalorderchoice}Henceforth, pick and fix a maximal order 
$$\mathcal{O}_D \subset D.$$
We will henceforth take $\mathcal{O} = \mathcal{O}_D$ in the context of Definition \ref{abelianvarietydefinition}.
\end{choice}

Now for each $\ell|N$, fix a trivialization $\iota_{\ell} : D\otimes_{\mathbb{Q}}\mathbb{Q}_{\ell} \xrightarrow{\sim}M_2(\mathbb{Q}_{\ell})$ such that $\iota_{\ell}(\mathcal{O}_D \otimes_{\mathbb{Z}} \mathbb{Z}_{\ell}) = M_2(\mathbb{Z}_{\ell})$. This induces an isomorphism
$$\iota_N  : D \otimes_{\mathbb{Q}}\prod_{\ell|N}\mathbb{Q}_{\ell} \xrightarrow{\sim} M_2(\prod_{\ell|N}\mathbb{Q}_{\ell}), \hspace{1cm} \iota_N(\mathcal{O}_D \otimes_{\mathbb{Z}} \prod_{\ell|N}\mathbb{Z}_{\ell}) = M_2(\prod_{\ell|N}\mathbb{Z}_{\ell}),$$
which in turn gives an isomorphism
\begin{equation}\label{iotan}\iota_N : \mathcal{O}_D\otimes_{\mathbb{Z}}\mathbb{Z}/N \xrightarrow{\sim} M_2(\mathbb{Z}/N).
\end{equation}
Throughout, let $\hat{\mathbb{Z}} = \prod_{\ell < \infty}\mathbb{Z}_{\ell}$, the profinite completion of $\mathbb{Z}$. Then (\ref{iotan}) induces a natural map
$$\mathcal{O}_D \otimes_{\mathbb{Z}}\hat{\mathbb{Z}}\twoheadrightarrow \mathcal{O}_D\otimes_{\mathbb{Z}}\mathbb{Z}/N  \xrightarrow{\iota_N} M_2(\mathbb{Z}/N)$$
which we also denote by $\iota_N$. 

\begin{definition}\label{neatdefinition}We say that 
$$\Gamma \subset D^{\times}(\mathbb{A}_{\mathbb{Q}}^{(\infty)}),$$
is \emph{neat} if, or some integer $N \ge 3$,
$$\Gamma \subset \{\gamma \in \left(\mathcal{O}_D \otimes_{\mathbb{Z}} \hat{\mathbb{Z}}\right)^{\times} : \gamma \equiv 1 \pmod{N\hat{\mathbb{Z}}}\}.$$

%When we wish to emphasize some choice of $N$, we say that $\Gamma$ is \emph{neat with respect to $N$}. 
\end{definition}

\begin{definition}\label{congruencesubgroups}Suppose $(N,\mathrm{disc}(D)) = 1$. Let
\begin{align*}\Gamma_0(N) := \left\{\gamma \in \left(\mathcal{O}_D \otimes_{\mathbb{Z}}\hat{\mathbb{Z}}\right)^{\times} : \iota_N(\gamma) \equiv \left(\begin{array}{ccc} * & *\\
0 & *\\
\end{array}\right) \pmod{N\hat{\mathbb{Z}}}\right\},\\
\Gamma_1(N) := \left\{\gamma \in \left(\mathcal{O}_D \otimes_{\mathbb{Z}}\hat{\mathbb{Z}}\right)^{\times} : \iota_N(\gamma) \equiv \left(\begin{array}{ccc} * & *\\
0 & 1\\
\end{array}\right) \pmod{N\hat{\mathbb{Z}}}\right\},\\
\Gamma(N) := \left\{\gamma \in \left(\mathcal{O}_D \otimes_{\mathbb{Z}}\hat{\mathbb{Z}}\right)^{\times} : \iota_N(\gamma) \equiv \left(\begin{array}{ccc} 1 & 0\\
0 & 1\\
\end{array}\right) \pmod{N\hat{\mathbb{Z}}}\right\}.
\end{align*}
These are compact open subgroups of $D^{\times}(\mathbb{A}_{\mathbb{Q}}^{(\infty)})$. 
\end{definition}

In particular, $\Gamma$ is neat in the sense of Definition \ref{neatdefinition} if and only if $\Gamma \subset \Gamma(N)$ for some $N \ge 3$. 

Assume that $\Gamma$ is one of the groups in Definition \ref{congruencesubgroups}. We define the notion of $\Gamma$-level structure (see \cite[Definition 1.1 and discussion below]{Buzzard}). 

\begin{definition}\begin{enumerate}
\item Given an elliptic curve $A/S$,  a $\Gamma_0(N)$-level structure is a choice of rank $N$ sub-group scheme $C \subset A[N]$ over $S$ such that $C$ is f.p.p.f. locally cyclic, a $\Gamma_1(N)$-level structure is a choice of embedding of group schemes $\mathbb{Z}/N \hookrightarrow A[N]$ over $S$, and a $\Gamma(N)$-level structure is a choice of isomorphism of group schemes $(\mathbb{Z}/N)^{\oplus 2} \cong A[N]$ over $S$. 

\item Given a false elliptic curve $A/S$, a $\Gamma_0(N)$-level structure is a choice of rank $N^2$ sub-group scheme $C \subset A[N]$ over $S$ commuting with the action of $i : \mathcal{O}_D \hookrightarrow \mathrm{End}_S(A)$ and such that $C$ is f.p.p.f. locally cyclic (i.e. admits a generator) as an $\mathcal{O}_D$-module, a $\Gamma_1(N)$-level structure is an embedding of group schemes $\mathbb{Z}/N \times \mathbb{Z}/N \hookrightarrow A[N]$ commuting with the action of $i : \mathcal{O}_D \hookrightarrow \mathrm{End}_S(A)$ (where $\mathcal{O}_D$ acts on the source via (\ref{iotan})), and a $\Gamma(N)$-level structure is a choice of isomorphism of group schemes $\mathcal{O}_D \otimes_{\mathbb{Z}}\mathbb{Z}/N \xrightarrow{\sim} A[N]$ over $S$ commuting with the action of $i : \mathcal{O}_D \hookrightarrow \mathrm{End}_S(A)$. 
\end{enumerate}
\end{definition}

Continue supposing $\Gamma$ is as in Definition \ref{congruencesubgroups}. Then by \cite[Theorem 11.16]{Milne2} applied to \cite[Example 3.4]{Milne}, $\mathbb{Y}(\Gamma)$ is the solution to the following moduli problem: If $D = M_2(\mathbb{Q})$, it represents the functor sending a (commutative) $\mathbb{Q}$-algebra $R$ to the set of isomorphism classes of triples $(A,P)$ where $A$ is an elliptic curve over $R$ and $P$ is a $\Gamma$-level structure on $A \rightarrow \mathrm{Spec}(R)$. If $D \neq M_2(\mathbb{Q})$, it represents the functor sending a $\mathbb{Q}$-algebra $R$ to the set of isomorphism classes of triples $(A,i,P)$ where $A$ is a false elliptic curve over $R$, $i : \mathcal{O}_D \hookrightarrow \mathrm{End}_R(A)$ is an embedding, and $P$ is a $\Gamma$-level structure on $A/R$ compatible with $i$. 

Now assume further that $\Gamma$ is neat in the sense of Definition \ref{neatdefinition}. By \cite[Sections 2-2 and 2-3, Theorem 1]{Morita}, the curve $\mathbb{Y}(\Gamma)$ is a fine moduli space. Thus it has a universal object, i.e. universal elliptic curve when $D = M_2(\mathbb{Q})$ (resp. universal false elliptic curve when $D \neq M_2(\mathbb{Q})$) together with $\Gamma$-level structure (resp. $\mathcal{O}_D$-endomorphism structure and $\Gamma$-level structure) 
$$\pi : \mathcal{E}(\Gamma) \rightarrow \mathbb{Y}(\Gamma).$$
When $\Gamma$ is obvious from context, we will denote $\mathcal{E}(\Gamma)$ by $\mathcal{E}$ for simplicity. We will also sometimes abuse notation and let $\mathcal{E}(\Gamma)$ denote the abelian variety underlying $\mathcal{E}(\Gamma)$ in situations where no confusion should arise. 

As in \cite[Chapter 1]{YZZ}, we let (for a given $\Gamma$ as above)
\begin{equation}\label{algebraicYinftyGamma}\mathbb{Y}_{\infty}(\Gamma) := \varprojlim_n \mathbb{Y}(\Gamma \cap \Gamma(p^n)),
\end{equation}
which is an inverse limit of schemes over $\mathrm{Spec}(\mathbb{Q})$, and thus $\mathbb{Y}_{\infty}(\Gamma)$ itself is a scheme over $\mathrm{Spec}(\mathbb{Q})$. (Note that we only take an inverse limit over $p$-level structures, unlike in loc. cit. which takes an inverse limit over all level structures.) 
%Let $\mathbb{Y}(\Gamma)_{\text{pro\'{e}t}}$ denote the pro\'{e}tale site of $\mathbb{Y}(\Gamma)$ as in \cite{BhattScholze}.

\subsection{Adic spaces and the pro\'{e}tale site}In this section and throughout the rest of the paper, we will use the notions and notations of adic spaces from \cite[Section 2]{ScholzeWeinstein}; see also \cite[Lectures 2-3]{ScholzeWeinstein2} for more details and generalizations. In particular, given an affinoid ring $(R,R^+)$ as in \cite[Definition 2.1.1]{ScholzeWeinstein}, we will let $\mathrm{Spa}(R,R^+)$ denote the adic spectrum of $(R,R^+)$ as in Definition 2.1.5 of op. cit. The underlying topological space $|\mathrm{Spa}(R,R^+)|$ (denoted by $\mathrm{Spa}(R,R^+)_{\mathrm{top}}$ in loc. cit.) of $\mathrm{Spa}(R,R^+)$ consists of equivalence classes of continuous valuations $|\cdot|$ (see Definition 2.1.2 of op. cit.) on $R$ such that $|R^+| \le 1$, and with open sets generated by rational subsets (see Definition 2.1.3 of op. cit.). The precise definition of adic spaces we will use is contained in Definition 2.1.5 of op. cit. Given an adic space $X$, we will let $|X|$ denote the underlying topological space.

\begin{remark}In particular, Scholze-Weinstein's notion of adic space in \cite[Section 2.1]{ScholzeWeinstein} is more general than Huber's original definition (\cite{Huber}), i.e. locally ringed spaces $X$ which are locally isomorphic to $\mathrm{Spa}(R,R^+)_{\mathrm{top}}$ together with the sheafification of its natural presheaf and a choice of valuation $v_x$ for every $x \in X$. Huber's notion is called an \emph{honest adic space} in \cite[discussion above Definition 2.1.5]{ScholzeWeinstein}. We will use this terminology when we wish to emphasize that a particular adic space is an honest adic space. 
\end{remark}

Suppose $X$ is an honest adic space or a scheme that is locally noetherian, i.e. is locally of the form $\mathrm{Spa}(R,R^+)$ where $R$ is strongly noetherian or admits a noetherian ring of definition (see \cite{Huber} for definitions of these notions). Let $X_{\text{\'{e}t}}$ denote the (small) \'{e}tale site of $X$ as in op. cit., whose objects are \'{e}tale maps $U \rightarrow X$ and whose coverings are \'{e}tale coverings. Since $X$ is locally noetherian, then such a $U\rightarrow X$ is locally connected. Let $X_{\text{pro\'{e}t}}$ denote the pro\'{e}tale site attached to $X$ as in \cite[Definition 3.9]{Scholze} (see also \cite{Scholzecorrigendum} for a slight modification). Thus, the category of objects of $X_{\text{pro\'{e}t}}$ is the full subcategory of $\mathrm{pro}-X_{\text{\'{e}t}}$ (see Definition 3.1 of op. cit. for a definition of $\mathrm{pro}-X_{\text{\'{e}t}}$) consisting of objects $U \rightarrow X$ that are pro-\'{e}tale. Here, the definition of pro-\'{e}tale map $U \rightarrow X$ we use is that given in the first paragraph of Definition 3.9 of op. cit. The covers of $X_{\text{pro\'{e}t}}$ are given by families of pro-\'{e}tale morphisms $\{f_i : U_i \rightarrow U\}$ such that $|U| = \bigcup_if(|U_i|)$.

\begin{remark}\label{whyremark}The definitions of the pro\'{e}tale site and pro-categories that we recall here are rather technical, and the reader may ask if these concepts are really necessary for our purposes, particularly in light of traditional approaches to $p$-adic $L$-functions which largely use classical rigid analytic geometry. We will require these notions in order to work with infinite-level Shimura curves (Definition \ref{YinftyGammadefinition}); these live naturally in the pro\'{e}tale site $Y(\Gamma)_{\text{pro\'{e}t}}$ of the finite-level Shimura curves $Y(\Gamma)$ (see Definition \ref{YGammadefinition}). We will need to consider these infinite-level Shimura curves in order to define the periods $z_{\mathrm{dR}}$ and $q_{\mathrm{dR}}$ of Section \ref{zqwsection}; such objects require a natural and functorial trivialization of the \'{e}tale cohomology of the universal (false) elliptic curve (see (\ref{etaletrivialization})), which does not exist on any finite level Shimura curve. Moreover, the period sheaves that we work with (Section \ref{periodsheavessection}) are defined only on the pro\'{e}tale site and are usually not sheaves, for example, on the \'{e}tale site. Thus, our need to involve the pro\'{e}tale site, pro-adic spaces and pro-formal schemes stems from the fact that our methods essentially use infinite level, and we cannot ``get away'' with finite level. 

However, from Section \ref{padicmodularformsection} and on, we will predominantly restrict our attention to the fixed infinite-level Shimura curve $Y_{\infty}$ (see Convention \ref{Yconvention}) and open subsets of $Y_{\infty}$. Thus our discussion from Section \ref{padicmodularformsection} onward requires considerably less technical background on the pro\'{e}tale site. The main output we will need from discussion before Section \ref{padicLfunctionsection} will be the notion of generalized $p$-adic modular forms (Definition \ref{generalizedpadicmodularformdefinition}), their $q_{\mathrm{dR}}$-expansions (Definition \ref{'zqexpansions}) and the weight-raising operators (or ``$p$-adic Maass-Shimura operators'') that act on them (Section \ref{operatorsection}). As generalized $p$-adic modular forms admit a natural notion of weight (see Definition \ref{generalizedpadicmodularformdefinition}, Theorem \ref{weighttheorem} and Corollary \ref{weightcorollary}), the study of these objects amounts to the study of $p$-adic representations of the \'{e}tale Galois group $\mathrm{Gal}(Y_{\infty}/Y) \cong GL_2(\mathbb{Z}_p)$ and the action of explicit elements of $GL_2(\mathbb{Q}_p)$ on generalized $p$-adic modular forms, where $GL_2(\mathbb{Q}_p)$ is naturally viewed in the local Hecke algebra of $Y_{\infty}$ at $p$.  
\end{remark}

\begin{convention}\label{objectconvention}Following \cite{Scholze}, given a site $\mathcal{S}$, we will write $U \in \mathcal{S}$ to denote that $U$ is an object of $\mathcal{S}$, i.e. an element of the underlying category of $\mathcal{S}$. Let $\mathcal{S}/U$ denote the localization of the site $\mathcal{S}$ to $U$. That is $\mathcal{S}/U$ is the site consisting of the subcategory of objects $V \in \mathcal{S}$ with a morphism $V \rightarrow U$, and covers of $V$ are covers coming from $\mathcal{S}$. In particular, ``$U \in Y_{\text{pro\'{e}t}}$'' denotes $U$ being an object of $Y_{\text{pro\'{e}t}}$. 
\end{convention}

%$Y(\Gamma)_{\text{pro\'{e}t}}$ denote the pro\'{e}tale site of $Y(\Gamma)$ defined as in \cite[Section 3]{Scholze}. 

\subsection{Integral and formal Shimura curves}\label{formalShimurasection}
Throughout the paper, we will use the definition of formal schemes from \cite{Stacks0AHY}. In particular, no noetherianness or finite-generatedness of ideals of definition assumptions are made \emph{a priori}. 

We will henceforth adopt the following Convention allowing us, under certain assumptions, to view formal schemes as adic spaces. 

\begin{convention}\label{formaladicfunctorconvention}\begin{enumerate}
\item Henceforth, we will ubiquitously use the fully faithful functor of \cite[Proposition 2.2.1]{ScholzeWeinstein} from formal schemes (that are defined over the ring of integers of a complete nonarchimedean field and admit a finitely generated ideal of definition) to adic spaces, which sends $\mathrm{Spf}(R) \mapsto \mathrm{Spa}(R,R)$, in order to view formal schemes as adic spaces. However, we will sometimes still use the notation $\mathrm{Spf}(R)$ for an affine formal scheme when we wish to emphasize that a particular adic space comes from a formal scheme. We will also sometimes abuse notation and identify write $\mathrm{Spf}(R)$ with its image $\mathrm{Spa}(R,R)$ under the above functor.

\item Similarly, consider any pro-formal scheme that can be written as an inverse limit $\varprojlim U_i$ where each $U_i$ is a formal scheme admitting a finitely generated ideal of definition. Applying the aforementioned functor to each $U_i$, we obtain a pro-adic space. We view such pro-formal schemes as pro-adic spaces throughout the rest of our discussion. 
\end{enumerate}
\end{convention}

\begin{definition}\label{formalmodeldefinition}Here and throughout the paper, a \emph{formal model of a (pro-)adic space $X$ over $\mathrm{Spa}(F,\mathcal{O}_F)$}, $F$ a complete nonarchimedean field, will mean a (pro-)adic scheme $\frak{X}$ over $\mathrm{Spa}(\mathcal{O}_F,\mathcal{O}_F)$ whose fiber over the generic point $\mathrm{Spa}(F,\mathcal{O}_F) \in \mathrm{Spa}(\mathcal{O}_F,\mathcal{O}_F)$ (i.e. adic generic fiber) is equal to $X$. Often, $\frak{X}$ will be induced by a (pro-)formal scheme via the functor from formal schemes to adic spaces as in Convention \ref{formaladicfunctorconvention}, or an inverse limit of adic spaces arising from formal schemes as in \cite[Proposition 2.4.2]{ScholzeWeinstein}.
\end{definition}

Recall the quaternion algebra $D/\mathbb{Q}$ underlying our Shimura curves.

\begin{assumption}\label{pdiscassumption}Let $p$ be any prime number, and henceforth assume that $p\nmid \mathrm{disc}(D)$. Thus 
$$D \otimes_{\mathbb{Q}}\mathbb{Q}_p \cong M_2(\mathbb{Q}_p).$$
\end{assumption}

\begin{convention}\label{idempotentconvention}Recall that $p \nmid \mathrm{disc}(D)$ (see Assumption \ref{pdiscassumption}). 
\begin{enumerate}
\item Throughout the rest of our discussion, fix an isomorphism 
$$\iota_p : D \otimes_{\mathbb{Q}}\mathbb{Q}_p \cong M_2(\mathbb{Q}_p)$$
such that $\iota_p(\mathcal{O}_D \otimes_{\mathbb{Z}}\mathbb{Z}_p) = M_2(\mathbb{Z}_p)$. Define
the idempotent
$$e^1 := \left(\begin{array}{ccc} 1 & 0\\
0 & 0 \\
\end{array}\right) \in M_2(\mathbb{Z}_p) \overset{\iota_p}{\subset} \mathcal{O}_D\otimes_{\mathbb{Z}} \mathbb{Z}_p.$$
\item For any false elliptic curve $A$ over a scheme over $\mathrm{Spec}(\mathbb{Z}_p)$, a formal scheme over $\mathrm{Spf}(\mathbb{Z}_p)$ or an adic space over $\mathrm{Spa}(\mathbb{Q}_p,\mathbb{Z}_p)$, we will let $A[p^{\infty}]$ denote \emph{$e^1$ applied to the $p$-divisible group of $A$} and \emph{not} the $p$-divisible group itself. Similarly, we will let $A[p^n]$ denote $e^1$ applied to the $p^n$-torsion sub-group scheme of $A$ and let
$$T_pA := \varprojlim_nA[p^n]$$
denote $e^1$ applied to the $p$-adic Tate module of $A$ (i.e. the inverse limit of all the $p^n$-torsion group schemes). For $A$ an elliptic curve, $A[p^{\infty}]$ will denote the usual $p$-divisible group of $A$, $A[p^n]$ the usual $p^n$-torsion sub-group scheme of $A$, and $T_pA$ the usual $p$-adic Tate module (defined by the same previous displayed equation).

\end{enumerate}
\end{convention}

\begin{definition}\label{YGamma+definition}Let $\Gamma$ be as in Definition \ref{congruencesubgroups}, and let
$$Y(\Gamma)^+ \rightarrow \mathrm{Spec}(\mathbb{Z}_p)$$
denote the integral model of $\mathbb{Y}(\Gamma) \times_{\mathrm{Spec}(\mathbb{Q})}\mathrm{Spec}(\mathbb{Q}_p)$ defined using Drinfeld level structures (\cite[Chapter 3.1-3.4]{KatzMazur}) and following the methods of Chapter 4 of op. cit., see \cite{Buzzard}. For the reader's convenience, we recall the definitions of Drinfeld level structures here.

Let $S$ be a scheme, formal scheme or an adic space, and let $N \in \mathbb{Z}_{> 0}$.

\begin{enumerate}
\item A $\Gamma_0(N)$-Drinfeld level structure on an elliptic curve $A/S$ is a finite flat sub-group scheme $H \subset A[N]$ of rank $N$ that is f.p.p.f. locally cyclic (i.e. admits a generator), a $\Gamma_1(N)$-Drinfeld level structure is a homomorphism $\phi : \mathbb{Z}/N \rightarrow A[N](S)$ such that the effective Cartier divisor
$$\sum_{a \pmod{N}}[\phi(a)]$$
is a sub-group scheme of $A[N]$, and a $\Gamma(N)$-Drinfeld level structure is a homomorphism $\phi : (\mathbb{Z}/N)^{\oplus 2} \rightarrow A[N](S)$ such that we have the following equality of effective Cartier divisors:
$$\sum_{(a,b) \pmod{N}}[\phi(a,b)] = A[N].$$
\item A $\Gamma_0(N)$-Drinfeld level structure on a false elliptic curve $A/S$ is a finite flat sub-group scheme $H \subset A[N]$ of rank $N^2$ that is f.p.p.f. locally cyclic (i.e. admits a generator) as an $\mathcal{O}_D$-module, a $\Gamma_1(N)$-Drinfeld level structure is a homomorphism $\phi : (\mathbb{Z}/N)^{\oplus 2} \rightarrow A[N](S)$ commuting with the action of $\mathcal{O}_D$ (where $\mathcal{O}_D$ acts on the source via (\ref{iotan})) such that the effective Cartier divisor
$$\sum_{(a,b)\pmod{N}}[\phi(a,b)]$$
is a sub-group scheme of $A[N]$, and a $\Gamma(N)$-Drinfeld level structure is a homomophism $\phi : \mathcal{O}_D\otimes_{\mathbb{Z}}\mathbb{Z}/N \rightarrow A[N]$ commuting with the action of $\mathcal{O}_D$ such that we have the following equality of effective Cartier divisors:
$$\sum_{a \in \mathcal{O}_D \pmod{N}}[\phi(a)] = A[N].$$

\item A $\Gamma_0(p^{\infty}), \Gamma_1(p^{\infty})$ or $\Gamma(p^{\infty})$-Drinfeld level structure on a (false) elliptic curve $A/S$ is given by the same definitions as above after replacing $\mathbb{Z}/N$ above with $\mathbb{Z}_p$ and $A[N]$ with the $p$-adic Tate module of $A$ (see Convention \ref{idempotentconvention} (2); when $A$ is a false elliptic curve, recall $T_pA$ is equal to $e^1$ applied to the $p$-adic Tate module of $A$). 

\item For $N \in \mathbb{Z}_{> 0}$ with $(N,p) = 1$, a $\Gamma_0(Np^{\infty})$ (resp. $\Gamma_1(Np^{\infty})$, resp. $\Gamma(Np^{\infty})$)-Drinfeld level structure is a pair consisting of a $\Gamma_0(N)$ (resp. $\Gamma_1(N)$, resp. $\Gamma(N)$)-Drinfeld level structure and a $\Gamma_0(p^{\infty})$ (resp. $\Gamma_1(p^{\infty})$, resp. $\Gamma(p^{\infty})$)-Drinfeld level structure.

\item We will also make the following slight abuse of notation, when $\Gamma(p^{\infty})$-level structures (either usual level structures or Drinfeld level structures, see Definition \ref{YGamma+definition} below) are considered. Let $e^1$ be the idempotent of Convention \ref{idempotentconvention} (1). Suppose we are given $(A,P,\phi)$ where $A$ is a (false) elliptic curve, $P$ is a $\Gamma$-level structure prime to $p$ and $\phi$ is a $\Gamma(p^{\infty})$-level structure (either Drinfeld or usual level structure). 

\begin{itemize}
\item If $A$ is an elliptic curve over $S$, then $\phi : \mathbb{Z}_p^{\oplus 2} \rightarrow T_pA(S) = \varprojlim_n A[p^n](S)$, and so we write $(A,P,\phi) = (A,P,e_1,e_2)$ where $e_1 = \phi|_{\mathbb{Z}_p\oplus \{0\}}$ and $e_2 = \phi|_{\{0\} \oplus \mathbb{Z}_p}$. 

\item If $A$ is a false elliptic curve, then applying the idempotent $e^1$ to the Drinfeld level structure $\phi$, we obtain a homomorphism $e^1\phi : \mathbb{Z}_p^{\oplus 2} \rightarrow T_pA(S) = \varprojlim_nA[p^n](S)$ (using the notation of Convention \ref{idempotentconvention} (2) for $T_pA$). Let $e_1 = e^1\phi|_{\mathbb{Z}_p\oplus \{0\}} : \mathbb{Z}_p \rightarrow T_pA$ and $e_2 = e^1\phi|_{\{0\} \oplus \mathbb{Z}_p} : \mathbb{Z}_p \rightarrow T_pA$, and write $(A,P,\phi) = (A,P,e_1,e_2)$. 

\item Let 
$$e_{i,n} := e_i \pmod{p^n} : \mathbb{Z}/p^n  \rightarrow A[p^n](S).$$
When $D = M_2(\mathbb{Q})$, given $e \in A[p^n]$, let $\langle e\rangle \subset A[p^n]$ denote the subgroup generated by $e$. Thus $A/\langle e\rangle$ is a well-defined elliptic curve and there is a natural isogeny of elliptic curves $A \rightarrow A/\langle e\rangle$. When $D \neq M_2(\mathbb{Q})$, given $e$ in the $p^n$-torsion group scheme of $A$, let $\langle e\rangle$ denote the $\mathcal{O}_D$-submodule of the $p^n$-torsion group scheme of $A$ generated by $e$. Thus $A/\langle e\rangle$ is a well-defined false elliptic curve and there is a natural isogeny of false elliptic curves $A \rightarrow A/\langle e\rangle$. 
\end{itemize}
\end{enumerate}

In particular $Y(\Gamma)^+$ is a scheme over $\mathrm{Spec}(\mathbb{Z}_p)$ % which satisfies
%$$Y(\Gamma)^+ \times_{\mathrm{Spec}(\mathbb{Z}_p)} \mathrm{Spec}(\mathbb{Q}_p) = Y(\Gamma),$$
and represents the moduli problem classifying isomorphism classes of tuples consisting of (false) elliptic curves with $\mathcal{O}_D$-endomorphism structure and $\Gamma$-Drinfeld level structure.
\end{definition}

If $\Gamma$ is neat (see Definition \ref{neatdefinition}), then the results of \cite[Chapter 5]{KatzMazur} imply that $Y(\Gamma)^+$ represents a fine moduli space, which implies the existence of a universal object
\begin{equation}\label{integraluniversalobject}\mathcal{E}(\Gamma)^+ \rightarrow Y(\Gamma)^+,
\end{equation}
which is the universal (false) elliptic curve with $\mathcal{O}_D$-endomorphism structure and $\Gamma$-Drinfeld level structure.

\begin{definition}\label{YinftyGamma+definition}Define a scheme over $\mathrm{Spec}(\mathbb{Z}_p)$
$$Y_{\infty}(\Gamma)^+ = \varprojlim_n Y(\Gamma \cap \Gamma(p^n))^+,$$
where the inverse limit is taken in the category of schemes. 
\end{definition}

\begin{definition}\label{YGammahat+definition}The ind-scheme given by ``$p$-adically completing'' $Y(\Gamma)^+$
$$\widehat{Y(\Gamma)}^+ := \varinjlim Y(\Gamma)^+ \times_{\mathrm{Spec}(\mathbb{Z}_p)} \mathrm{Spec}(\mathbb{Z}/p^n)$$
defines a formal scheme over $\mathrm{Spf}(\mathbb{Z}_p)$. 
\end{definition}

If $\Gamma$ is neat, then let
\begin{equation}\label{formaluniversalobject}\widehat{\mathcal{E}(\Gamma)}^+ = \varinjlim_n \mathcal{E}(\Gamma)^+ \times_{\mathrm{Spec}(\mathbb{Z}_p)}\mathrm{Spec}(\mathbb{Z}/p^n)
\end{equation}
denote the $p$-adic completion of $\mathcal{E}(\Gamma)^+$ from (\ref{integraluniversalobject}). It inherits an $\mathcal{O}_D$-endomorphism structure and $\Gamma$-Drinfeld level structure from $\mathcal{E}(\Gamma)^+$. If $\Gamma$ is obvious from context we will sometimes write $\widehat{\mathcal{E}}^+$ instead of $\widehat{\mathcal{E}(\Gamma)}^+$.

Recall the definition of a pro-category given in \cite[Definition 3.1]{Scholze}, and consider the category of pro-formal schemes. 
\begin{definition}\label{YinftyGammahat+definition}Define
$$\widehat{Y_{\infty}(\Gamma)}^+ := \varprojlim_n \widehat{Y(\Gamma \cap \Gamma(p^n))}^+,$$
which is a pro-formal scheme over $\mathrm{Spf}(\mathbb{Z}_p)$. 
\end{definition}

\subsection{Adic Shimura curves}\label{adicShimurasection} 
In this section, we will define certain honest adic spaces obtained by taking the adic generic fibers (Definition \ref{formalmodeldefinition}) of the formal Shimura curves $\widehat{Y(\Gamma)}^+ \rightarrow \mathrm{Spf}(\mathbb{Z}_p) = \mathrm{Spa}(\mathbb{Z}_p,\mathbb{Z}_p)$ defined in Section \ref{formalShimurasection}. Here, recall that we use Convention \ref{formaladicfunctorconvention} when identifying $\mathrm{Spf}(\mathbb{Z}_p) = \mathrm{Spa}(\mathbb{Z}_p,\mathbb{Z}_p)$.

\begin{definition}\label{YGammadefinition}Given a $\mathbb{Y}(\Gamma)^+$ and $\widehat{Y(\Gamma)}^+$ as in Section \ref{formalShimurasection}, where $\Gamma$ is not necessarily neat, let 
$$Y(\Gamma) \rightarrow \mathrm{Spa}(\mathbb{Q}_p,\mathbb{Z}_p)$$
denote the adic generic fiber of $\widehat{Y(\Gamma)}^+ \rightarrow \mathrm{Spa}(\mathbb{Z}_p,\mathbb{Z}_p)$. This is a locally noetherian honest adic space. %This can also be defined as the honest adic space attached to the rigid generic fiber of $\widehat{Y(\Gamma)}^+ \rightarrow \mathrm{Spec}(\mathbb{Z}_p)$. See \cite[Chapter 5]{Bosch} for an overview of the rigid generic fiber functor. %Here, rigid analytification of a $\mathrm{Spec}(\mathbb{Q}_p)$-scheme is defined in \cite[Chapter 5.4]{Bosch}, and 
%The honest adic space associated to a rigid space is given by \cite[(1.1.11)]{Huber}.
\end{definition}

%Recall the fully faithful functor from formal schemes (that are defined over the ring of integers of a complete nonarchimedean field and admit a finitely generated ideal of definition) to adic spaces (\cite[Proposition 2.1.1]{ScholzeWeinstein}), sending 
%$$\mathrm{Spf}(R) \mapsto \mathrm{Spa}(R,R). $$
Using Convention \ref{formaladicfunctorconvention}, we can view $\widehat{Y(\Gamma)}^+$ as an adic space over $\mathrm{Spa}(\mathbb{Z}_p,\mathbb{Z}_p)$. We have a natural diagram of adic spaces
$$
\begin{tikzcd}[column sep =large]
Y(\Gamma) \arrow{r}{} \arrow{d}{}& \widehat{Y(\Gamma)}^+ \arrow{d} \\
\mathrm{Spa}(\mathbb{Q}_p,\mathbb{Z}_p) \arrow{r}{} & \mathrm{Spa}(\mathbb{Z}_p,\mathbb{Z}_p)
\end{tikzcd}.$$

Given $Y(\Gamma)$ as in Definition \ref{YGammadefinition}, $Y(\Gamma)$ is a locally noetherian honest adic space, as it is the adic generic fiber of the $p$-adic completion $\widehat{Y(\Gamma)}^+ \rightarrow \mathrm{Spf}(\mathbb{Z}_p)$ of the locally noetherian scheme $Y(\Gamma)^+ \rightarrow \mathrm{Spec}(\mathbb{Z}_p)$. Thus $Y(\Gamma)_{\text{pro\'{e}t}}$ is well-defined.

\begin{convention}We will often refer to a morphism $V \rightarrow U$ in $Y(\Gamma)_{\text{pro\'{e}t}}$ as a \emph{pro\'{e}tale open}. This terminology should result in no confusion, as any pro-\'{e}tale map in $\mathrm{pro}-Y(\Gamma)_{\text{\'{e}t}}$ is open (\cite[Lemma 3.10 (iv)]{Scholze}). 
\end{convention}

\begin{definition}[Definition 4.1 of \cite{Scholze}]\label{OYdefinition}Suppose $X$ is a locally noetherian adic space over $\mathrm{Spa}(\mathbb{Q}_p,\mathbb{Z}_p)$. 
\begin{enumerate}
\item Let $\mathcal{O}_{X}$ denote the structure sheaf on $X_{\text{pro\'{e}t}}$, let $\mathcal{O}_{X}^+ \subset \mathcal{O}_{X}$ denote the integral subsheaf, let $\hat{\mathcal{O}}_{X}^+ = \varprojlim_n \mathcal{O}_{X}^+/p^n$, and let $\hat{\mathcal{O}}_{X}= \hat{\mathcal{O}}_{X}^+[1/p]$. 
\item Thus there is a natural map of sheaves $\mathcal{O}_{X}^{(+)} \rightarrow \hat{\mathcal{O}}_{X}^{(+)}$, where ``$(+)$'' denotes the optional presence of a ``$+$''. We call this the \emph{$p$-adic completion map}.
\item For any object $X' \in X_{\text{pro\'{e}t}}$, we let $\mathcal{O}_{X'}$ denote the pullback of $\mathcal{O}_{X}$ along $X' \rightarrow X$, and similarly for $\mathcal{O}_{X'}^+$, $\hat{\mathcal{O}}_{X'}^+$ and $\hat{\mathcal{O}}_{X'}$. In particular, these are sheaves on the localized site $X_{\text{pro\'{e}t}}/X'$. 
\item Given a formal scheme $\frak{X}$, we will let $\mathcal{O}_{\frak{X}}$ denote the structure sheaf on $\frak{X}$ on the formal site. 
\end{enumerate}
\end{definition}

\begin{definition}\label{YinftyGammadefinition}Let 
$$Y_{\infty}(\Gamma) = \varprojlim_nY(\Gamma \cap \Gamma(p^n)) \in Y(\Gamma)_{\text{pro\'{e}t}}.$$
As each of the transition maps in the above inverse limit is finite \'{e}tale, $Y_{\infty}(\Gamma)$ is indeed an object of $Y_{\text{pro\'{e}t}}$ (see also Convention \ref{objectconvention}). When $\Gamma$ is obvious from context, we will often write $Y_{\infty}$ instead.
\end{definition}

We have a commutative diagram of locally ringed spaces
$$\begin{tikzcd}[column sep =large]
Y_{\infty}(\Gamma) \arrow{r}{} \arrow{d}{}& \mathbb{Y}_{\infty}(\Gamma) \times_{\mathrm{Spec}(\mathbb{Q})}\mathrm{Spec}(\mathbb{Q}_p) \arrow{d} \\
\mathrm{Spa}(\mathbb{Q}_p,\mathbb{Z}_p) \arrow{r}{} & \mathrm{Spec}(\mathbb{Q}_p)
\end{tikzcd}$$
%$$Y_{\infty}(\Gamma) \rightarrow \mathbb{Y}_{\infty}(\Gamma)\times_{\mathrm{Spec}(\mathbb{Q})}\mathrm{Spec}(\mathbb{Q}_p)$$
induced by putting together the maps of locally ringed spaces
$$Y(\Gamma \cap \Gamma(p^n)) \rightarrow \mathbb{Y}(\Gamma \cap \Gamma(p^n))^{\mathrm{ad}} \rightarrow \mathbb{Y}(\Gamma \cap \Gamma(p^n))^{\mathrm{rig}} \rightarrow \mathbb{Y}(\Gamma \cap \Gamma(p^n)) \times_{\mathrm{Spec}(\mathbb{Q})}\mathrm{Spec}(\mathbb{Q}_p)$$
for every $n \in \mathbb{Z}_{\ge 0}$. Here $\mathbb{Y}(\Gamma \cap \Gamma(p^n))^{\mathrm{ad}} \rightarrow \mathbb{Y}(\Gamma \cap \Gamma(p^n))^{\mathrm{rig}}$ is the adic space associated with the rigid analytification 
$$\mathbb{Y}(\Gamma \cap \Gamma(p^n))^{\mathrm{rig}} \rightarrow \mathbb{Y}(\Gamma \cap \Gamma(p^n)) \times_{\mathrm{Spec}(\mathbb{Q})}\mathrm{Spec}(\mathbb{Q}_p).$$
See \cite[Chapter 5.4]{Bosch} for the definition of the rigic analytification of a $\mathbb{Q}_p$-scheme, and \cite[(1.1.11)]{Huber} for the adic space associated with a rigid analytic space.

When $\Gamma$ is neat so that there is a universal object $\mathcal{E}(\Gamma) \rightarrow \mathbb{Y}(\Gamma)$, we will make a slight abuse of notation and let 
\begin{equation}\label{adicuniversalobject}\mathcal{E}(\Gamma) \rightarrow Y(\Gamma)
\end{equation}
denote the adic generic fiber (i.e. fiber along $\mathrm{Spa}(\mathbb{Q}_p,\mathbb{Z}_p) \rightarrow \mathrm{Spa}(\mathbb{Z}_p,\mathbb{Z}_p)$) of $\widehat{\mathcal{E}(\Gamma)}^+ \rightarrow \widehat{Y(\Gamma)}^+$. If $\Gamma$ is obvious from context we will simply denote $\mathcal{E}(\Gamma)$ by $\mathcal{E}$.

Now work in the category of pro-adic spaces. In this category, we have a natural diagram
$$
\begin{tikzcd}[column sep =large]
Y_{\infty}(\Gamma) \arrow{r}{} \arrow{d}{}& \widehat{Y_{\infty}(\Gamma)}^+ \arrow{d} \\
\mathrm{Spa}(\mathbb{Q}_p,\mathbb{Z}_p) \arrow{r}{} & \mathrm{Spa}(\mathbb{Z}_p,\mathbb{Z}_p)
\end{tikzcd}.$$

\begin{definition}\label{kdefinition}\begin{enumerate}
\item Henceforth, let
$$\mathcal{O}_k = \widehat{\mathbb{Z}_p[\mu_{p^{\infty}}]} := \varprojlim_n \mathbb{Z}[\mu_{p^{\infty}}]/(p^n), \hspace{1cm} k = \mathcal{O}_k[1/p].$$
\item Let
$$Y(\Gamma)_k = Y(\Gamma) \times_{\mathrm{Spa}(\mathbb{Q}_p,\mathbb{Z}_p)}\mathrm{Spa}(k,\mathcal{O}_k), \hspace{1cm} \widehat{Y(\Gamma)}_{\mathcal{O}_k}^+ = \widehat{Y(\Gamma)}^+ \times_{\mathrm{Spf}(\mathbb{Z}_p)} \mathrm{Spf}(\mathcal{O}_k)$$
denote the base changes, where the fiber product is taken in the category of adic spaces (note that on coordinate rings, this involves a $p$-adically completed tensor product $\hat{\otimes}$). We note that here and throughout the text, $\mathrm{Spf}(\mathcal{O}_k)$ denotes the $p$-adic completion of $\mathrm{Spec}(\mathcal{O}_k)$. We may view this base-change $Y(\Gamma)_k$ as an object $Y(\Gamma)_k \in Y(\Gamma)_{\text{pro\'{e}t}}$, see \cite[Proposition 3.15]{Scholze}. Since $k$ is a perfectoid field we may consider affinoid perfectoid objects in the localized site $Y(\Gamma)_{\text{pro\'{e}t}}/Y(\Gamma)_k$ as in Definition 4.3 of op. cit.
\end{enumerate}
\end{definition}

\begin{convention}\label{basechangeconvention}As we will mainly be working over $k$ and $\mathcal{O}_k$ in our discussion, we will henceforth base change all adic Shimura curves $Y(\Gamma)$ from $\mathrm{Spa}(\mathbb{Q}_p,\mathbb{Z}_p)$ to $\mathrm{Spa}(k,\mathcal{O}_k)$, base change all integral models $Y(\Gamma)^+$ from $\mathrm{Spec}(\mathbb{Z}_p)$ to $\mathrm{Spec}(\mathcal{O}_k)$, and base change all formal models $\widehat{Y(\Gamma)}^+$ from $\mathrm{Spf}(\mathbb{Z}_p)$ to $\mathrm{Spf}(\mathcal{O}_k)$. For ease of notation, we use the same symbols $Y(\Gamma)$, $Y(\Gamma)^+$ and $\widehat{Y(\Gamma)}^+$ to denote these base changes unless otherwise noted. Occasionally we will need to use the fact that these objects are originally defined over $\mathbb{Q}_p$ or $\mathbb{Z}_p$.  
\end{convention}

\subsection{Neighborhoods cut out by the Hasse invariant}\label{Hassenbhdsection}

\begin{definition}\label{reductionformalschemedefinition}For a formal scheme $\frak{U}$ over $\mathcal{O}_k$, and $a \in \mathcal{O}_k$ or $a \subset \mathcal{O}_k$ an ideal, let 
$$\frak{U}/a := \frak{U} \times_{\mathrm{Spf}(\mathcal{O}_k)}\mathrm{Spf}(\mathcal{O}_k/a)$$
as in \cite[line above Lemma III.2.14]{ScholzeTorsion}. 
\end{definition}

Consider any $0 \le \epsilon \le \infty$ in the valuation group of $\mathcal{O}_k$ (see Definition \ref{kdefinition}), and fix any element $p^{\epsilon} \in \mathcal{O}_k$ of valuation $\epsilon$. Assume that $\Gamma$ is neat in the sense of Definition \ref{neatdefinition}. We will briefly recall the notations of $\frak{X}$ and $\frak{X}^+$ in \cite[Section III.2]{ScholzeTorsion}; in our notation, $\frak{X}$ denotes our $\widehat{Y(\Gamma)}^+$, and $\frak{X}^*$ denotes the $p$-adic completion of the compactification of $Y(\Gamma)$ (which only differs from $Y(\Gamma)$ in the case $D = M_2(\mathbb{Q})$). Let $Y(\Gamma)^+(\epsilon)$ be the pullback along $\frak{X} \hookrightarrow \frak{X}^*$ of the formal scheme constructed in \cite[Definition III.2.12, Lemma III.2.13]{ScholzeTorsion}. For the reader's convenience, we recall the Definition from loc. cit. in our setting.

\begin{definition}[Definition III.2.12 of \cite{ScholzeTorsion}]\label{YGamma+epsilondefinition}
Suppose $\Gamma \subset D^{\times}(\mathbb{A}_{\mathbb{Q}}^{(\infty)})$ is neat (see Definition \ref{neatdefinition}). Recall the universal object $\pi : \widehat{\mathcal{E}(\Gamma)}^+ \rightarrow \widehat{Y(\Gamma)}^+$ and let 
\begin{equation}\label{omegaYGamma+}\widehat{\omega}:= \pi_*\Omega_{\widehat{\mathcal{E}(\Gamma)}^+/\widehat{Y(\Gamma)}^+}
\end{equation}
denote the usual Hodge bundle. 
\begin{enumerate}
\item Define
$$Y(\Gamma)^+(\epsilon) \subset \widehat{Y(\Gamma)}^+$$
as the open sub-formal scheme over $\mathrm{Spf}(\mathcal{O}_k)$ which represents the following functor: Send any $p$-adically complete, flat $\mathcal{O}_k$-algebra $S$ to the set of equivalence classes of pairs $(\rho,u)$, where $\rho : S \rightarrow \widehat{Y(\Gamma)}^+$ is a map of formal schemes over $\mathrm{Spf}(\mathcal{O}_k)$, $u \in H^0(\mathrm{Spf}(\mathcal{O}_K),\rho^*\widehat{\omega}^{\otimes (1-p)})$ is a section such that 
$$u \cdot \mathrm{Ha}(\bar{\rho}) = p^{\epsilon} \in S/p,$$
where $\bar{\rho} = \rho \otimes_{\mathbb{Z}_p}\mathbb{F}_p$, $\mathrm{Ha}$ is a choice of local lift of the Hasse invariant, and two pairs $(\rho,u)$, $(\rho',u')$ are considered equivalent if $\rho = \rho'$ and there exists some $s \in S$ with $u' = u(1+p^{1-\epsilon}s)$. As proven in \cite[Lemma III.2.13]{ScholzeTorsion}, the above definition of equivalence makes the definition of $Y(\Gamma)^+(\epsilon)$ independent of choice of local lift $\mathrm{Ha}$ of the Hasse invariant.

\item Let
$$Y(\Gamma(1))^+(\epsilon) \subset \widehat{Y(\Gamma(1))}^+$$
denote the image of the formal subscheme $Y(\Gamma)^+(\epsilon) \subset \widehat{Y(\Gamma)}^+$ under the map of formal schemes $\widehat{Y(\Gamma)}^+ \rightarrow \widehat{Y(\Gamma(1))}^+$. Note that this image is defined over $\mathrm{Spf}(\mathcal{O}_k)$ and is independent of the choice of neat $\Gamma \subset D^{\times}(\mathbb{A}_{\mathbb{Q}}^{(\infty)})$.
\end{enumerate}
\end{definition}

In our setting, $Y(\Gamma)^+(\epsilon)$ is equal to the integral model constructed in \cite[Section 2]{Katzpamf} ($D = M_2(\mathbb{Q})$) and \cite{Kassaei} ($D \neq M_2(\mathbb{Q})$), which classifies isomorphism classes of triples $(A,P,u)$, where $A$ is an elliptic curve, $P$ is a $\Gamma$-level structure and $u \in \omega^{\otimes (1-p)}$ is a section with $u\cdot \mathrm{Ha} = p^{\epsilon}$ for a local lift of the Hasse invariant $\mathrm{Ha}$.

\begin{definition}\label{YinftyGamma+epsilondefinition}Define the formal scheme over $\mathrm{Spf}(\mathcal{O}_k)$ 
$$Y_{\infty}(\Gamma)^+(\epsilon) := Y(\Gamma)^+(\epsilon) \times_{\widehat{Y(\Gamma)}^+}\widehat{Y_{\infty}(\Gamma)}^+,$$
the product being taken in the category of pro-formal schemes. 
\end{definition}

\begin{definition}\label{YGammaepsilondefinition}Let 
$$Y(\Gamma)(\epsilon) := Y(\Gamma)^+(\epsilon)\times_{\mathrm{Spa}(\mathcal{O}_k,\mathcal{O}_k)} \mathrm{Spa}(k,\mathcal{O}_k)$$
be the adic generic fiber of $Y(\Gamma)^+(\epsilon)$. In particular it is an adic space over $\mathrm{Spa}(k,\mathcal{O}_k)$ and $Y(\Gamma)(\epsilon) \subset Y(\Gamma)$ is an open subset.
\end{definition}

\begin{definition}\label{YinftyGammaepsilondefinition}Let
$$Y_{\infty}(\Gamma)(\epsilon) := Y(\Gamma)(\epsilon) \times_{Y(\Gamma)}Y_{\infty}(\Gamma).$$
\end{definition}

%Note that $Y(\Gamma)$ and $Y(\Gamma)^+$ are already compact when $D \neq M_2(\mathbb{Q})$. When $D = M_2(\mathbb{Q})$, we will not consider the compactification throughout this article, and only work with the open modular curve. 

\subsection{Summary of conventions on Shimura curves}We collect our various notations and conventions regarding Shimura curves and summarize them below. 

\begin{convention}\label{Yconvention}For the remainder of the paper, we will fix 
$$\Gamma = \Gamma(N)$$ 
for some $N \ge 4$ that is prime to $p\cdot\mathrm{disc}(D)$.\footnote{Note that Definition \ref{neatdefinition} only requires $N \ge 3$. However, as we will be citing works such as \cite{Buzzard} which assume $N \ge 4$, we will also adopt this assumption.} As we continue to assume that $p \nmid \mathrm{disc}(D)$ (Assumption \ref{pdiscassumption}), we thus have that $p, \mathrm{disc}(D)$ and $N$ are pairwise coprime. Moreover, 
\begin{itemize}
\item let $Y$ denote $Y(\Gamma)$ from Definition \ref{YGammadefinition}, which is an adic space over $\mathrm{Spa}(\mathbb{Q}_p,\mathbb{Z}_p)$ (which we base change to $\mathrm{Spa}(k,\mathcal{O}_k)$ unless otherwise noted, see Convention \ref{basechangeconvention}),
\item let $Y_{\infty}$ denote $Y_{\infty}(\Gamma)$ from Definition \ref{YinftyGammadefinition}, which is an object of $Y_{\text{pro\'{e}t}}$ defied over $\mathrm{Spa}(\mathbb{Q}_p,\mathbb{Z}_p)$ (which base change to $\mathrm{Spa}(k,\mathcal{O}_k)$ unless otherwise noted, see Convention \ref{basechangeconvention}), 
\item let $Y(\epsilon)$ denote $Y(\Gamma)(\epsilon)$ from Definition \ref{YGammaepsilondefinition}, which is an adic space over $\mathrm{Spa}(k,\mathcal{O}_k)$, 
\item let $Y_{\infty}(\epsilon)$ denote $Y_{\infty}(\Gamma)(\epsilon)$ from Definition \ref{YinftyGammaepsilondefinition}, which is an object of $Y_{\text{pro\'{e}t}}$ over $\mathrm{Spa}(k,\mathcal{O}_k)$,
\item let $Y^+$ denote $Y(\Gamma)^+$ from Definition \ref{YGamma+definition}, which is scheme over $\mathrm{Spec}(\mathbb{Z}_p)$ (which we base change to $\mathrm{Spec}(\mathcal{O}_k)$ unless otherwise noted, see Convention \ref{basechangeconvention}),
\item let $\hat{Y}^+$ denote $\widehat{Y(\Gamma)}^+$ from Definition \ref{YGammahat+definition}, which is a formal scheme over $\mathrm{Spf}(\mathbb{Z}_p)$ (which we base change to $\mathrm{Spf}(\mathcal{O}_k)$ unless otherwise noted, see Convention \ref{basechangeconvention}),
\item let $Y_{\infty}^+$ denote $Y_{\infty}(\Gamma)^+$ from Definition \ref{YinftyGamma+definition}, which is a scheme over $\mathrm{Spec}(\mathbb{Z}_p)$ (which we view over $\mathrm{Spec}(\mathcal{O}_k)$ unless otherwise noted, see Convention \ref{basechangeconvention}),
\item let $\hat{Y}_{\infty}^+$ denote $\widehat{Y_{\infty}(\Gamma)}^+$ from Definition \ref{YinftyGammahat+definition}, which is a pro-formal scheme over $\mathrm{Spf}(\mathbb{Z}_p)$ (which we view over $\mathrm{Spf}(\mathcal{O}_k)$ unless otherwise noted, see Convention \ref{basechangeconvention}),
\item let $Y^+(\epsilon)$ denote $Y(\Gamma)^+(\epsilon)$ from Definition \ref{YGamma+epsilondefinition}, which is a formal scheme over $\mathrm{Spf}(\mathcal{O}_k)$,
\item let $Y_{\infty}^+(\epsilon)$ and denote $Y_{\infty}(\Gamma)^+(\epsilon)$ Definition \ref{YinftyGamma+epsilondefinition}, which is a pro-formal scheme over $\mathrm{Spf}(\mathcal{O}_k)$.
\item Let $\pi : \mathcal{E} \rightarrow Y$ denote the universal object, let $\pi_+ : \mathcal{E}^+ \rightarrow Y^+$ denote the universal object, and henceforth let
\begin{equation}\label{omegaY}\omega := \pi_*\Omega_{\mathcal{E}/Y}, \hspace{1cm} \omega_+ := \pi_{+,*}\Omega_{\mathcal{E}^+/Y^+}
\end{equation}
denote the usual Hodge bundles. For any $W \in Y_{\text{pro\'{e}t}}$, we will let $\omega_W = \omega|_W$, and similarly given any map of formal schemes $f : W \rightarrow Y^+$ we will let $\omega_{+,W} = f^*\omega_+$. 
\end{itemize}
%See Sections \ref{adicShimurasection} and \ref{formalShimurasection} for definitions of these objects. 
\end{convention}

\subsection{Formal tori and $p$-adic constant sheaves over adic spaces}

\begin{definition}Given $S$ that is a scheme, (pro-)formal scheme, or a (pro-)adic space with structure sheaf $\mathcal{O}_S$, we say $S$ is \emph{$p$-adically complete} if 
$$\mathcal{O}_S \cong \varprojlim_n \mathcal{O}_S/p^n$$
as sheaves. 
\end{definition}

\begin{definition}\label{fixT}\begin{enumerate}
\item Given a $p$-adically complete $S$, let $\hat{\mathbb{G}}_{m,S}$ denote the formal multiplicative group over $S$, which is a height 1 Lubin-Tate formal $\mathbb{Z}_p$-module. 
\item Given a local affine/affinoid chart $U$ with coordinate ring $R$, the ring of sections of $\mathcal{O}_{\hat{\mathbb{G}}_{m,S}}(U)$ is abstractly isomorphic to $R\llbracket T\rrbracket$; this $T$ is unique up to multiplication by $[a]_{\hat{\mathbb{G}}_{m,S}}$, $a \in \mathbb{Z}_p^{\times}$, where $[\cdot ]_{\hat{\mathbb{G}}_{m,S}}$ denotes the formal $\mathbb{Z}_p$-module action. 
\item Taking an open cover $\{U_i\}$ of $\hat{\mathbb{G}}_{m,S}$, we can choose coordinates $T_i$ on each $U_i$ such that $T_i|_{U_i \cap U_j} = T_j|_{U_i \cap U_j}$. By gluing this gives a global coordinate 
$$T \in \mathcal{O}_{\hat{\mathbb{G}}_{m,S}}(\hat{\mathbb{G}}_{m,S}).$$
\item We henceforth fix such a coordinate $T$. By \cite[proof of Proposition 1]{Tate}, this amounts to picking a basis of the Tate module $\varprojlim_n \mu_{p^n,S}$, where $\mu_{p^n,S}$ is the $p^n$-torsion group scheme\footnote{Throughout the paper, ``group scheme'' over a general (not necessarily scheme) base $S$ will mean a group functor on $S$.} of $\hat{\mathbb{G}}_{m,S}$. \item When $S$ is an affine scheme, affine formal scheme, or affinoid adic space with coordinate ring $R$, we sometimes write $\hat{\mathbb{G}}_{m,S} = \hat{\mathbb{G}}_{m,R}$, and similarly $\mu_{p^n,S} = \mu_{p^n,R}$. 
\end{enumerate}
\end{definition}

\begin{definition}[cf. Definition 8.1 of \cite{Scholze}]\label{Zphatdefinition}\begin{enumerate}
\item Given any locally noetherian adic space or scheme $S$, we have an associated pro\'{e}tale site $S_{\text{pro\'{e}t}}$ from \cite[Definition 3.9]{Scholze}. Let $\nu : S_{\text{pro\'{e}t}} \rightarrow S_{\text{\'{e}t}}$ denote the natural projection of sites (see the discussion after Proposition 3.15 of op. cit.). Then the constant sheaf attached to $\mathbb{Z}/p^n$ on $S_{\text{\'{e}t}}$ pulls back under $\nu$ to a sheaf $(\mathbb{Z}/p^n)_S$. 
\item Let
$$\hat{\mathbb{Z}}_{p,S} = \varprojlim_n (\mathbb{Z}/p^n)_S,$$
which is a sheaf on $S_{\text{pro\'{e}t}}$. Given any object $S' \in S_{\text{pro\'{e}t}}$, we let $\hat{\mathbb{Z}}_{p,S'} = \hat{\mathbb{Z}}_{p,S}|_{S'}$. 

\item The sections of $\hat{\mathbb{Z}}_{p,S}$ can be viewed as a $p$-adically continuous constant sheaf attached to $\mathbb{Z}_p$ in the following sense: for any pro\'{e}tale open $U \rightarrow S$, 
$$\hat{\mathbb{Z}}_{p,S}(U) = \mathrm{Cont}(\pi_0(U),\mathbb{Z}_p),$$
where $\mathrm{Cont}$ denotes the set of continuous maps. Here, $\pi_0(U)$ denotes the set of connected components of $U$ which is endowed with the following natural topology: writing $U = \varprojlim_i U_i$ in some pro\'{e}tale presentation in $S_{\text{pro\'{e}t}}$ (see Definition 3.9 of op. cit. and \cite[p. 1-2]{Scholzecorrigendum}), we have
$$\pi_0(U) = \varprojlim_i\pi_0(U_i),$$
where each $U_i \rightarrow S$ is \'{e}tale (in the sense of Definition 3.9 of op. cit.) and each $U_{i'} \rightarrow U_i$ is finite \'{e}tale (in the sense of loc. cit.) for all $i' \ge i \gg 0$. Each $\pi_0(U_i)$ is thus endowed with a natural profinite topology, so that the above inverse limit has a natural topology (i.e. the coarsest topology such that all the projections $\pi_0(U) \rightarrow \pi_0(U_i)$ are continuous). Also, $\mathbb{Z}_p$ above is endowed with the \emph{$p$-adic topology} (as opposed to the discrete topology usually taken when defining constant sheaves).
\end{enumerate}
\end{definition}

\subsection{Review of the Igusa tower and Serre-Tate coordinates} \label{STreviewsection}

\begin{convention}\label{formalYconvention}For notational convenience, in this section (Section \ref{STreviewsection}), we will let $Y$ denote $\widehat{Y(\Gamma)}^+$ (see Section \ref{formalShimurasection}), viewed over $\mathrm{Spf}(\mathbb{Z}_p)$ (unlike what we will usually do per Convention \ref{basechangeconvention}). Thus $Y$ is a formal scheme over $\mathrm{Spf}(\mathbb{Z}_p)$ for the remainder of this section. We will also let $\mathcal{E} \rightarrow Y$ denote the universal (false) elliptic curve with $\Gamma$-level structure over $Y$, where $\Gamma = \Gamma(N)$ for some $N \ge 4$ with $(N,p) = 1$. (Thus $\Gamma$ is neat in the sense of Definition \ref{neatdefinition}.)
\end{convention}

\begin{definition}\label{Yorddefinition}Define the ordinary locus 
$$Y^{\mathrm{ord}} \subset Y$$
to be the open formal subscheme of $Y$ over $\mathrm{Spec}(\mathbb{Z}_p)$ obtained by deleting the finitely many supersingular points:
$$Y^{\mathrm{ord}} := Y \setminus \bigcup_{y = (A,P) \in Y(\overline{\mathbb{F}}_p),\; \text{$A$ supersingular}}\{y\}.$$
The fact that there are finitely many supersingular points follows from \cite[Chapter V.4]{Silverman} in the case $D = M_2(\mathbb{Q})$, and from \cite[Theorem 5.4 and 6.4]{Milnemodp} in general. Note that Theorem 5.4 of op. cit. implies there is one supersingular isogeny class, and the double coset of Theorem 6.4 corresponding to this isogeny class can be identified with the class group of a particular order of $D$, which has finite cardinality.
\end{definition}

%(To see that these are finitely many points: In the case $D = M_2(\mathbb{Q})$, it is a classical theorem that there are finitely many isomorphism classes of elliptic curves with supersingular good reduction at $p$, see \cite[Chapter V.4]{Silverman}. This implies that there are only finitely many points in $Y(\overline{\mathbb{F}}_p)$ with underlying elliptic curve supersingular. When $D \neq M_2(\mathbb{Q})$, then by \cite[Proposition 5.2]{Milnemodp}, we see that any supersingular false abelian variety over $\overline{\mathbb{F}}_p$ is isogenous to the square $A'^{\oplus 2}$ of a supersingular elliptic curve $A'$. By the above discussion, there are finitely many such squares up to isomorphism, and the isogeny class (modulo isomorphism) of $A'^{\oplus 2}$ over $\overline{\mathbb{F}}_p$ is finite, so we are done.)

Given a point $x = (A_0,P_0) \in Y^{\mathrm{ord}}(\overline{\mathbb{F}}_p)$, $P_0$ a $\Gamma$-level structure, let $M^{\mathrm{ord}}(x)$ denote the formal completion of $Y^{\mathrm{ord}}$ at $x$, which is itself a formal scheme defined over $\mathrm{Spf}(W(k(x)))$, where $k(x)$ denotes the residue field of $x$ (a finite extension of $\mathbb{F}_p$) and $W(k)$ denotes the Witt vectors of a perfect field $k$ of characteristic $p$. 

\begin{theorem}[Main Theorem of \cite{KatzST}]\label{KatzSTtheorem}Let $x = (A_0,P_0) \in Y^{\mathrm{ord}}(\overline{\mathbb{F}}_p)$. There is a canonical isomorphism of formal schemes over $\mathrm{Spf}(W(k(A_0,P_0)))$. 
\begin{equation}\label{KatzSTisomorphism}\mathrm{Hom}_{\mathbb{Z}_p}(T_pA_0(\overline{\mathbb{F}}_p) \otimes_{\mathbb{Z}_p} T_pA_0(\overline{\mathbb{F}}_p),\hat{\mathbb{G}}_{m,\mathbb{Z}_p}) \cong M^{\mathrm{ord}}(A_0,P_0),
\end{equation}
where $\mathrm{Hom}_{\mathbb{Z}_p}$ denotes the set of homomorphisms in the category of $\mathbb{Z}_p$-modules. This induces an isomorphism on the adic generic fibers of these formal schemes.
\end{theorem}

We recall the definition of the Igusa tower. Recall we are using the notation of Convention \ref{formalYconvention} throughout this section, as well as our Convention \ref{idempotentconvention} regarding the notation $A[p^n]$ for a (false) elliptic curve $A$; in particular that $A[p^n]$ is half of the full $p^n$-torsion subgroup of $A$ in the false elliptic curve case. 

\begin{definition}[Igusa tower]\label{Igusadefinition}\begin{enumerate}
\item Let $\mu_{p^n,Y}$ denote the $p^n$-torsion group scheme of the multiplicative group scheme $\hat{\mathbb{G}}_{m,Y}$. \item Let $\mathcal{E}[p^n]^0$ denote the connected component of $\mathcal{E}[p^n]$. 
\item Define the \emph{Igusa tower} as
$$Y^{\mathrm{Ig}} := \varprojlim_n \mathrm{Isom}_{\mathrm{grp-sch}}(\mu_{p^n,Y},\mathcal{E}[p^n]^0).$$
Each term of the above inverse system is the group of isomorphisms in the category of group schemes, and is a formal scheme over $\mathrm{Spf}(\mathbb{Z}_p)$; the transition maps are given by the natural inclusions $\mu_{p^n,Y} \subset \mu_{p^{n+1},Y}$. Hence $Y^{\mathrm{Ig}}$ is a pro-formal scheme over $\mathrm{Spf}(\mathbb{Z}_p)$. (Recall that we use the notion of pro-category from \cite[Definition 3.1]{Scholze}.) In fact, $Y^{\mathrm{Ig}}$ is a formal scheme since it admits an pro-finite \'{e}tale map $Y^{\mathrm{Ig}} \rightarrow Y^{\mathrm{ord}}$ over $\mathrm{Spf}(\mathbb{Z}_p)$ (see (\ref{naturalIgprojection}) below) and thus is equal to its strong completion. 
\item $Y^{\mathrm{Ig}}$ represents the moduli problem classifying deformations of $p$-divisible groups $A[p^{\infty}]$ attached to ordinary (false) elliptic curves $A$ together with an isomorphism of $p$-divisible groups
$$\mu_{p^{\infty}} \cong \hat{A}[p^{\infty}]$$ 
where $\hat{A}[p^{\infty}] \subset A[p^{\infty}]$ denotes the connected component. (Recall we continue to work under Convention \ref{idempotentconvention}.) Note that this latter isomorphism is equivalent to an isomorphism of $\mathbb{Z}_p$-modules
$$\mathbb{Z}_p \xrightarrow{\sim} T_pA^{\text{\'{e}t}}$$
where $T_pA \twoheadrightarrow T_pA^{\text{\'{e}t}}$ denotes the Tate module of the \'{e}tale quotient of $A[p^{\infty}]$. (See \cite[(2.2) (4)]{Tate} for definitions of connected component and \'{e}tale quotients of finite flat group schemes and $p$-divisible groups.)
\end{enumerate}
\end{definition}

\begin{remark}\label{genericremark}Outside of this section, we will let $Y^{\mathrm{ord}}$ and $Y^{\mathrm{Ig}}$ instead denote the adic generic fibers $Y^{\mathrm{ord}} \times_{\mathrm{Spa}(\mathbb{Z}_p,\mathbb{Z}_p)}\mathrm{Spa}(\mathbb{Q}_p,\mathbb{Z}_p)$ and $Y^{\mathrm{Ig}} \times_{\mathrm{Spa}(\mathbb{Z}_p,\mathbb{Z}_p)}\mathrm{Spa}(\mathbb{Q}_p,\mathbb{Z}_p)$ of $Y^{\mathrm{ord}}$ in Definition \ref{Yorddefinition} and $Y^{\mathrm{Ig}}$ from Definition \ref{Igusadefinition}. 
\end{remark}

For $(A_0,P_0) \in Y^{\mathrm{ord}}(\overline{\mathbb{F}}_p)$, let
$$M^{\mathrm{Ig}}(A_0,P_0) = M(A_0,P_0) \times_{Y^{\mathrm{ord}}}Y^{\mathrm{Ig}},$$
where the product is taken in the category of formal schemes. This is a formal scheme over $\mathrm{Spf}(k(A_0,P_0))$. 

There natural projection $Y^{\mathrm{Ig}} \rightarrow Y^{\mathrm{ord}}$ given by 
\begin{equation}\label{naturalIgprojection}(A,P,\mathbb{Z}_p \xrightarrow{\sim} T_pA^{\text{\'{e}t}}) \mapsto (A,P)
\end{equation}
(where $P$ is a $\Gamma$-level structure). This is evidently a pro-finite \'{e}tale map of formal schemes over $\mathrm{Spec}(\mathbb{Z}_p)$ with \'{e}tale Galois group 
$$\mathrm{Aut}(T_pA^{\text{\'{e}t}}) \cong \mathrm{Aut}(\mathbb{Z}_p) = \mathbb{Z}_p^{\times}.$$

Note that $Y^{\mathrm{ord}} \rightarrow \mathrm{Spf}(\mathbb{Z}_p)$ is an open formal subscheme of the formal scheme obtained by $p$-adically completing the locally noetherian scheme $Y^+ \rightarrow \mathrm{Spec}(\mathbb{Z}_p)$, and thus is locally noetherian. Therefore $Y^{\mathrm{ord}} \rightarrow \mathrm{Spa}(\mathbb{Z}_p,\mathbb{Z}_p)$ is locally noetherian when viewed as an adic space and Definition \ref{Zphatdefinition} can be applied to it. Since $Y^{\mathrm{Ig}} \rightarrow Y^{\mathrm{ord}}$ is pro-finite \'{e}tale, then $Y^{\mathrm{Ig}} \in Y_{\text{\'{e}t}}^{\mathrm{ord}}$ and thus the pullback 
$$\hat{\mathbb{Z}}_{p,Y^{\mathrm{Ig}}} := \hat{\mathbb{Z}}_{p,Y^{\mathrm{ord}}}|_{Y^{\mathrm{Ig}}}$$
is well-defined. Recall $\mathcal{E} \rightarrow Y$ is the universal object, and let
$$\mathcal{E}|_{Y^{\mathrm{Ig}}} := \mathcal{E} \times_YY^{\mathrm{Ig}}$$
is the universal object on $Y^{\mathrm{Ig}}$; we let $\mathcal{E}|_{M^{\mathrm{Ig}}(A_0,P_0)}$ denote the restriction to $M^{\mathrm{Ig}}(A_0,P_0) \subset Y^{\mathrm{Ig}}$. Thus we have a canonical section 
$$e : \hat{\mathbb{Z}}_{p,Y^{\mathrm{Ig}}} \xrightarrow{\sim} T_p\mathcal{E}|_{Y^{\mathrm{Ig}}}^{\text{\'{e}t}} = T_p\mathcal{E}|_{Y^{\mathrm{Ig}}}(\overline{\mathbb{F}}_p).$$
% = T_pA_0|_{Y^{\mathrm{Ig}}}(\overline{\mathbb{F}}_p),$$
%where $A_0|_{Y^{\mathrm{Ig}}} = A_0 \times 
 Note that this induces an isomorphism of sheaves of $\mathbb{Z}_p$-modules on $M^{\mathrm{Ig}}(A_0,P_0)$:
\begin{equation}\label{eisomorphism}e : \hat{\mathbb{Z}}_{p,M^{\mathrm{Ig}}(A_0,P_0)} \xrightarrow{\sim} T_p\mathcal{E}|_{M^{\mathrm{Ig}}(A_0,P_0)}(\overline{\mathbb{F}}_p) = T_pA_0(\overline{\mathbb{F}}_p)|_{M^{\mathrm{Ig}}(A_0,P_0)}, 
\end{equation}
the last equality again being given by Hensel's lemma. In particular, the cover $Y^{\mathrm{Ig}} \rightarrow Y^{\mathrm{ord}}$ trivializes on $M(A_0,P_0) \subset Y^{\mathrm{ord}}$:
$$Y^{\mathrm{Ig}} \times_{Y^{\mathrm{ord}}}M(A_0,P_0) = \bigsqcup_{e_0 : \mathbb{Z}_p \xrightarrow{\sim} T_pA_0(\overline{\mathbb{F}}_p)}M(A_0,P_0,e_0) \cong \bigsqcup_{e_0 : \mathbb{Z}_p \xrightarrow{\sim} T_pA_0(\overline{\mathbb{F}}_p)}M(A_0,P_0)|_{M^{\mathrm{Ig}}(A_0,P_0)},$$
where $M(A_0,P_0,e_0)$ denotes the formal completion of $M^{\mathrm{Ig}}(A_0,P_0)$ at $(A_0,P_0,e_0) \in M^{\mathrm{Ig}}(A_0,P_0)(\overline{\mathbb{F}}_p)$ and the last isomorphism is given by compoenent-wise isomorphisms $$M(A_0,P_0,e_0) \cong M(A_0,P_0).$$
Here, the isomorphism $M(A_0,P_0,e_0) \cong M(A_0,P_0)$ is given by $(A,P,e) \mapsto (A,P)$ with inverse $(A,P) \mapsto (A,P,e)$, where 
$$e : \mathbb{Z}_p \underset{\sim}{\xrightarrow{e_0}} T_pA_0(\overline{\mathbb{F}}_p) = T_pA^{\text{\'{e}t}}$$
the last equality coming from the fact that $A$ is a smooth lifting of $A_0$ and Hensel's lemma. 

 %Note that by this \'{e}taleness, $Y^{\mathrm{Ig}}$ is $p$-adically complete and so in fact $Y^{\mathrm{Ig}}$ is itself a formal scheme over $\mathrm{Spf}(\mathbb{Z}_p)$. 
As the automorphism group of the \'{e}tale quotient $\hat{\mathbb{Z}}_{p,Y^{\mathrm{Ig}}} \overset{e}{\cong} T_p\mathcal{E}|_{Y^{\mathrm{Ig}}}^{\text{\'{e}t}}$ of $T_p\mathcal{E}|_{Y^{\mathrm{ord}}}$ is $\mathbb{Z}_p^{\times}$, we have
\begin{equation}\label{IgusaGaloisgroup}\mathrm{Gal}(Y^{\mathrm{Ig}}/Y^{\mathrm{ord}}) = \mathbb{Z}_p^{\times}.
\end{equation}

\begin{definition}
\begin{enumerate}
\item The isomorphism $e$ from (\ref{eisomorphism}) induces a canonical isomorphism of formal schemes over $\mathrm{Spf}(W(k(A_0,P_0)))$:
\begin{equation}\label{Igusatriv}\begin{split}\bigsqcup_{y \in M^{\mathrm{Ig}}(A_0,P_0)(\overline{\mathbb{F}}_p)}\hat{\mathbb{G}}_{m,W(k(y))} &= \mathrm{Hom}_{\mathbb{Z}_p}(\hat{\mathbb{Z}}_{p,M^{\mathrm{Ig}}(A_0,P_0)},\hat{\mathbb{G}}_m) \\
&\overset{e}{\cong} \mathrm{Hom}_{\mathbb{Z}_p}(T_pA_0(\overline{\mathbb{F}}_p) \otimes_{\mathbb{Z}_p}T_pA_0(\overline{\mathbb{F}}_p)|_{M^{\mathrm{Ig}}(A_0,P_0)},\hat{\mathbb{G}}_m) \overset{(\ref{KatzSTisomorphism})}{\cong} M^{\mathrm{Ig}}(A_0,P_0).
\end{split}
\end{equation}
Here the disjoint union on the left runs over all points in $y \in M^{\mathrm{Ig}}(A_0,P_0)(\overline{\mathbb{F}}_p)$, each of which corresponds to an isomorphism class of a triple $(A_0,P_0,\mathbb{Z}_p \xrightarrow{\sim} T_pA_0(\overline{\mathbb{F}}_p))$. From this we also see that the residue field $k(y) = \overline{\mathbb{F}}_p$ and its ring of Witt vectors $W(k(y)) = W(\overline{\mathbb{F}}_p)$. %This is seen using the fact that $A_0[p^{\infty}](\overline{\mathbb{F}}_p) = A[p^{\infty}]^{\text{\'{e}t}}$ for any deformation $A$ of $A_0$. 

\item Denote the inverse of the isomorphism (\ref{Igusatriv}) by
$$e^{-1} : M^{\mathrm{Ig}}(A_0,P_0) \xrightarrow{\sim}  \bigsqcup_{y \in M^{\mathrm{Ig}}(A_0,P_0)(\overline{\mathbb{F}}_p)}\hat{\mathbb{G}}_{m,W(\overline{\mathbb{F}}_p)}.$$

\item This gives mutually inverse isomorphisms of sheaves
$$e^*:  \mathcal{O}_{M^{\mathrm{Ig}}(A_0,P_0)} \xrightarrow{\sim}  \prod_{y \in M^{\mathrm{Ig}}(A_0,P_0)(\overline{\mathbb{F}}_p)}\mathcal{O}_{\hat{\mathbb{G}}_{m,W(\overline{\mathbb{F}}_p)}} , \hspace{.5cm} e^{-1,*} : \prod_{y \in M^{\mathrm{Ig}}(A_0,P_0)(\overline{\mathbb{F}}_p)}\mathcal{O}_{\hat{\mathbb{G}}_{m,W(\overline{\mathbb{F}}_p)}} \xrightarrow{\sim} \mathcal{O}_{M^{\mathrm{Ig}}(A_0,P_0)}.$$
\end{enumerate}
\end{definition}

\begin{definition}\label{STexpansiondefinition}Recall our coordinate $T \in \mathcal{O}_{\hat{\mathbb{G}}_{m,W(\overline{\mathbb{F}}_p)}}(\hat{\mathbb{G}}_{m,W(\overline{\mathbb{F}}_p)})$ fixed in Definition \ref{fixT}. 
\begin{enumerate}
\item The \emph{Serre-Tate coordinate} on $M^{\mathrm{Ig}}(A_0,P_0)$ is defined as 
$$q_{\mathrm{ST}} := 1 + (e^{-1})^*T \in \mathcal{O}_{M^{\mathrm{Ig}}(A_0,P_0)}(M^{\mathrm{Ig}}(A_0,P_0)).$$
\item Given $F \in \mathcal{O}_{M^{\mathrm{Ig}}(A_0,P_0)}(M^{\mathrm{Ig}}(A_0,P_0))$, we have 
$$e^*F \in\prod_{y \in M^{\mathrm{Ig}}(A_0,P_0)(\overline{\mathbb{F}}_p)}\hat{\mathcal{O}}_{\hat{\mathbb{G}}_{m,W(\overline{\mathbb{F}}_p)}}(\hat{\mathbb{G}}_{m,W(\overline{\mathbb{F}}_p)}) =  \prod_{y \in M^{\mathrm{Ig}}(A_0,P_0)(\overline{\mathbb{F}}_p)}W(\overline{\mathbb{F}}_p)\llbracket T\rrbracket.$$
\item When viewing 
$$e^*F \in \prod_{y \in M^{\mathrm{Ig}}(A_0,P_0)(\overline{\mathbb{F}}_p)}W(\overline{\mathbb{F}}_p)\llbracket T\rrbracket$$
we write $e^*F = e^*F(T)$. Pulling this back by $(e^{-1})^*$ we get a power series in $q_{\mathrm{ST}}-1$
$$F(q_{\mathrm{ST}}) := (e^{-1})^*(e^*F(T)) \in \mathcal{O}_{M^{\mathrm{Ig}}(A_0,P_0)}(M^{\mathrm{Ig}}(A_0,P_0)) = \prod_{y \in M^{\mathrm{Ig}}(A_0,P_0)(\overline{\mathbb{F}}_p)}W(\overline{\mathbb{F}}_p)\llbracket q_{\mathrm{ST}}-1\rrbracket.$$
We call this the \emph{Serre-Tate expansion of $F$ at $(A_0,P_0)$}. 

%\item For each $y \in M^{\mathrm{Ig}}(A_0,P_0)(\overline{\mathbb{F}}_p)$, there is a unique geometric connected component $M^{\mathrm{Ig}}(A_0,P_0)_y$ of $M^{\mathrm{Ig}}(A_0,P_0)$ containing $y$. We let $q_{\mathrm{ST},y}$ denote the restriction of $q_{\mathrm{ST}}$ to $M^{\mathrm{Ig}}(A_0,P_0)_y$. We also let $q_{\mathrm{ST},y}$ denote image of $q_{\mathrm{ST}}$ in the stalk $\mathcal{O}_{M^{\mathrm{Ig}}(A_0,P_0),y}$; the two notations are compatible as $M^{\mathrm{Ig}}(A_0,P_0)_y$ is an open formal neighborhood of $y$. We call $q_{\mathrm{ST},y}$ the \emph{Serre-Tate coordinate centered at $y$}.
\end{enumerate}
\end{definition}

\begin{definition}\label{STrecenter}\begin{enumerate}
\item Taking the adic generic fiber of Definition \ref{STexpansiondefinition}, we get induced objects on the adic generic fiber $$M^{\mathrm{Ig}}(A_0,P_0)_{\eta} \rightarrow \mathrm{Spa}(W(k(A_0,P_0))[1/p],W(k(A_0,P_0)))$$
of $M^{\mathrm{Ig}}(A_0,P_0)$. We continue to denote the section induced by $q_{\mathrm{ST}}$ as $q_{\mathrm{ST}}$. 
\item Given $y \in M^{\mathrm{Ig}}(A_0,P_0)(\overline{\mathbb{F}}_p)$, let 
$$M^{\mathrm{Ig}}(A_0,P_0)_y$$ denote the unique geometric connected component of $M^{\mathrm{Ig}}(A_0,P_0)$ containing $y$, which we view as defined over $\mathrm{Spf}(W(\overline{\mathbb{F}}_p))$. Let 
$$M^{\mathrm{Ig}}(A_0,P_0)_{y,\eta}$$
denote its adic generic fiber, which we view over $\mathrm{Spa}(W(\overline{\mathbb{F}}_p)[1/p],W(\overline{\mathbb{F}}_p))$. 
\item Given an affinoid $(W(\overline{\mathbb{F}}_p)[1/p],W(\overline{\mathbb{F}}_p))$-algebra $(R,R^+)$ and any point 
$$y' \in M^{\mathrm{Ig}}(A_0,P_0)_{y,\eta}(R,R^+),$$
define
$$q_{\mathrm{ST},y'} := \frac{q_{\mathrm{ST}}|_{M^{\mathrm{Ig}}(A_0,P_0)_{y,\eta}}}{q_{\mathrm{ST}}(y')} \in \mathcal{O}_{M^{\mathrm{Ig}}(A_0,P_0)_{y,\eta}}^+(M^{\mathrm{Ig}}(A_0,P_0)_{y,\eta}) \hat{\otimes}_{W(\overline{\mathbb{F}}_p)}R^+$$
which is well-defined since $q_{\mathrm{ST}}(y') \in (R^+)^{\times}$. Here, $\hat{\otimes}$ denotes $p$-adically completed tensor product.  We call $q_{\mathrm{ST},y'}$ the \emph{Serre-Tate coordinate centered at $y'$}. 
\item We can thus expand functions on $M^{\mathrm{Ig}}(A_0,P_0)_{y,\eta}$ as a power series in the coordinate $q_{\mathrm{ST},y'}-1$ by performing the change of variables on the usual Serre-Tate expansion
$$q_{\mathrm{ST}} -1 \mapsto q_{\mathrm{ST},y'} -1$$
on the usual Serre-Tate expansion (which by Definition \ref{STexpansiondefinition} is a power series in the coordinate $q_{\mathrm{ST}} -1$). We call this power series in $q_{\mathrm{ST},y'} -1$ the \emph{Serre-Tate expansion centered at $y'$}. In all we have an isomorphism 
$$\mathcal{O}_{M^{\mathrm{Ig}}(A_0,P_0)_{y,\eta}}(M^{\mathrm{Ig}}(A_0,P_0)_{y,\eta}) \cong W(\overline{\mathbb{F}}_p)\llbracket q_{\mathrm{ST},y'}-1\rrbracket.$$ 
\end{enumerate}
\end{definition}

\subsection{Galois groups of Igusa towers}
\label{furtherIgusasection}

We now return to the category of adic spaces; all results of Section \ref{STreviewsection} carry over to the adic generic fiber. Per Remark \ref{genericremark}, we will let $Y^{\mathrm{ord}}$ and $Y^{\mathrm{Ig}}$ denote the adic generic fibers of the formal schemes $Y^{\mathrm{ord}}$ and $Y^{\mathrm{Ig}}$ from Section \ref{STreviewsection}. No ambiguity will arise as we work entirely on the adic generic fiber in this section. %Let 
%$$\pi_{\mathrm{HT}} : \hat{Y}_{\infty} \rightarrow \mathbb{P}^1$$
%be the Hodge-Tate period map of \cite{ScholzeTorsion} (see also \cite[Theorem 2.8]{ChojeckiHansenJohansson}), which is a map of adic spaces over $\mathrm{Spa}(\mathbb{Q}_p,\mathbb{Z}_p)$. Let $x,y \in \mathcal{O}_{\mathbb{P}^1}(1)(\mathbb{P}^1)$ be the standard global sections, let 
%$$z_{\mathrm{HT}} = \pi_{\mathrm{HT}}^*z,$$
%where $z = -x/y$ be the coordinate on the affine $\{y \neq 0\} \subset \mathbb{P}^1$. %and let 
%$$\frak{s} = \pi_{\mathrm{HT}}^*x \in \pi_{\mathrm{HT}}^*\mathcal{O}_{\mathbb{P}^1}(1) \cong \omega \otimes_{\mathcal{O}_Y}\hat{\mathcal{O}}_{Y_{\infty}}(Y_{\infty}).$$ 

Retain the notation of Convention \ref{Yconvention}. Briefly view $Y$ and $Y_{\infty}$ as defined over $\mathrm{Spa}(\mathbb{Q}_p,\mathbb{Z}_p)$ again (see Convention \ref{basechangeconvention}). Let 
\begin{equation}\label{hatYinftyYinfty}\hat{Y}_{\infty} \sim Y_{\infty}
\end{equation}
denote the strong completion of $Y_{\infty}$ (see \cite[Definition 2.4.1, Proposition 2.4.4]{ScholzeWeinstein}); in particular, we have a canonical identification of underlying topological spaces $|\hat{Y}_{\infty}| = |Y_{\infty}|$, and the restriction of $\hat{\mathcal{O}}_{Y_{\infty}}$ to the adic site of $|Y_{\infty}|$ is equal to $\mathcal{O}_{\hat{Y}_{\infty}}$ (\cite[discussion before and after Definition 4.3]{Scholze}). By \cite{ScholzeTorsion}, $\hat{Y}_{\infty}$ is a perfectoid space after base changing to $\mathrm{Spa}(k,\mathcal{O}_k)$ as in Convention \ref{basechangeconvention}. Let
$$\pi_{\mathrm{HT}} : \hat{Y}_{\infty} \rightarrow \mathbb{P}^1$$
be the Hodge-Tate period map of \cite{ScholzeTorsion} (see also \cite[Theorem 2.8]{ChojeckiHansenJohansson}), which is a map of adic spaces over $\mathrm{Spa}(\mathbb{Q}_p,\mathbb{Z}_p)$. This gives an identification of sheaves on the adic site of $\hat{Y}_{\infty}$
$$\pi_{\mathrm{HT}}^*\mathcal{O}_{\mathbb{P}^1}(1) \cong \omega \otimes_{\mathcal{O}_Y}\mathcal{O}_{\hat{Y}_{\infty}}$$
where $\mathcal{O}_{\hat{Y}_{\infty}}$ denotes the adic structure sheaf on $\hat{Y}_{\infty}$ and $\omega$ is as in (\ref{omegaY}). %By \cite[Lemma 4.10]{Scholze}, we have that the restriction of $\hat{\mathcal{O}}_{Y_{\infty}}$ to the adic site on $|Y_{\infty}| = |\hat{Y}_{\infty}|$ is equal to $\mathcal{O}_{\hat{Y}_{\infty}}$. In particular, 
%$$\mathcal{O}_{\hat{Y}_{\infty}}(\hat{Y}_{\infty}) = \hat{\mathcal{O}}_{Y_{\infty}}(Y_{\infty}).$$

Let $x,y \in \mathcal{O}_{\mathbb{P}^1}(1)(\mathbb{P}^1)$ be the usual projective coordinates, let 
$$V_x = \{x \neq 0\} \subset \mathbb{P}^1, \hspace{1cm} V_y = \{y \neq 0\} \subset \mathbb{P}^1,$$
let 
\begin{equation}\label{hatVxVy}\hat{\mathcal{V}}_x = \pi_{\mathrm{HT}}^{-1}(V_x) \subset |\hat{Y}_{\infty}| = |Y_{\infty}|, \hspace{1cm} \hat{\mathcal{V}}_y = \pi_{\mathrm{HT}}^{-1}(V_y) \subset |\hat{Y}_{\infty}| = |Y_{\infty}|
\end{equation}
and let
\begin{equation}\label{zHT}z_{\mathrm{HT}} := \pi_{\mathrm{HT}}^*\left(-\frac{x}{y}\right) \in \hat{\mathcal{O}}_{Y_{\infty}}(\hat{\mathcal{V}}_y).
\end{equation}
This is the \emph{Hodge-Tate period} (\cite{ScholzeTorsion}, \cite[Section 2]{ChojeckiHansenJohansson}). Also let
\begin{equation}\label{fraks}\frak{s} := \pi_{\mathrm{HT}}^*x \in \pi_{\mathrm{HT}}^*\mathcal{O}_{\mathbb{P}^1}(1)(\mathbb{P}^1) \cong \omega \otimes_{\mathcal{O}_Y}\mathcal{O}_{\hat{Y}_{\infty}}(\hat{Y}_{\infty}) .%= \omega_{\otimes_{\mathcal{O}_Y}}\hat{\mathcal{O}}_{Y_{\infty}}(Y_{\infty}).
\end{equation}
This is the \emph{fake Hasse invariant} (\cite{ScholzeTorsion}, \cite[Section 2.4]{ChojeckiHansenJohansson}). 

Let
\begin{equation}\label{VxVy} \hat{\mathcal{V}}_x \sim \mathcal{V}_x \subset Y_{\infty}, \hspace{1cm} \hat{\mathcal{V}}_y \sim \mathcal{V}_y \subset Y_{\infty}
\end{equation}
be the open subsets corresponding to $\hat{\mathcal{V}}_x \subset \hat{Y}_{\infty}$ and $\hat{\mathcal{V}}_y \subset \hat{Y}_{\infty}$ under $\hat{Y}_{\infty} \sim Y_{\infty}$. Since $Y_{\infty} \in Y_{\text{pro\'{e}t}}$ is perfectoid (over $\mathrm{Spa}(k,\mathcal{O}_k)$) and $\mathcal{V}_x \subset Y_{\infty}$ and $\mathcal{V}_y \subset Y_{\infty}$ are open subsets.
%, then $\mathcal{V}_x \in Y_{\text{pro\'{e}t}}$ and $\mathcal{V}_y \in Y_{\text{pro\'{e}t}}$ are affinoid perfectoid (over $\mathrm{Spa}(k,\mathcal{O}_k)$) by \cite[Lemma 4.6]{Scholze}. Thus, by \cite[discussion after Definition 4.3]{Scholze}, $\hat{\mathcal{V}}_x$ and $\hat{\mathcal{V}}_y$ are affinoid perfectoid spaces (over $\mathrm{Spa}(k,\mathcal{O}_k)$). By the discussion after Lemma 4.2 of op. cit., we have 
%\begin{equation}\label{hatOVxVy}\mathcal{O}_{\hat{\mathcal{V}}_x}(\hat{\mathcal{V}}_x) = \hat{\mathcal{O}}_{\mathcal{V}_x}(\mathcal{V}_x), \hspace{1cm}\mathcal{O}_{\hat{\mathcal{V}}_y}(\hat{\mathcal{V}}_y) = \hat{\mathcal{O}}_{\mathcal{V}_y}(\mathcal{V}_y).
%\end{equation}

Since $\mathcal{V}_x = \{\frak{s} = \pi_{\mathrm{HT}}^*x \neq 0\}$, we have that 
\begin{equation}\label{fraksgenerator}\frak{s}|_{\hat{\mathcal{V}}_x} \in \omega\otimes_{\mathcal{O}_Y}\mathcal{O}_{\hat{Y}_{\infty}}(\hat{\mathcal{V}}_x)% = \omega\otimes_{\mathcal{O}_Y}\hat{\mathcal{O}}_{Y_{\infty}}(\mathcal{V}_x)
\end{equation}
is a generator (i.e. nowhere vanishing).

\begin{convention}\label{zdefineconvention}\begin{enumerate}
\item We have a natural supply of open subsets of $\hat{Y}_{\infty}$ defined by $\{|z_{\mathrm{HT}}| > p^c\}$ for $c \in \mathbb{Q}$. Using $\hat{Y}_{\infty} \sim Y_{\infty}$ and the underlying identification $|\hat{Y}_{\infty}| = |Y_{\infty}|$, each of these open subsets defines an open subset of $Y_{\infty} \in Y_{\text{pro\'{e}t}}$. We will henceforth make a slight abuse of notation and let $\{|z_{\mathrm{HT}}| > p^c\}$ also denote the corresponding open subset of $Y_{\infty}$. 
\item Note also that 
\begin{equation}\label{hatVz}\hat{\mathcal{V}}_x = \{z_{\mathrm{HT}} \neq 0\}, \hspace{1cm} \hat{\mathcal{V}}_y = \{z_{\mathrm{HT}} \neq \infty\}.
\end{equation}
As $\hat{\mathcal{V}}_x \sim \mathcal{V}_x$ and $\hat{\mathcal{V}}_y \sim \mathcal{V}_y$, we will also similarly make an abuse of notation and write
\begin{equation}\label{Vz}\mathcal{V}_x = \{z_{\mathrm{HT}} \neq 0\}, \hspace{1cm} \mathcal{V}_y = \{z_{\mathrm{HT}} \neq \infty\}.
\end{equation}
\end{enumerate}
\end{convention}

\begin{definition}\label{mathcalYIgDefinition}\begin{enumerate}
\item Define the \emph{big Igusa tower} as
$$\mathcal{Y}^{\mathrm{Ig}} := \pi_{\mathrm{HT}}^{-1}(\{\infty\}) = \{z_{\mathrm{HT}} = \infty\} \subset Y_{\infty}.$$
Then $\mathcal{Y}^{\mathrm{Ig}}$ classifies isomorphism classes of quintuples $(A,P,e_1,e_2,i)$ where $A$ is a (false) elliptic curve, $P$ is a $\Gamma$-level structure, and $(e_1,e_2) : \mathbb{Z}_p^{\oplus 2} \xrightarrow{\sim} T_pA$ is a $\Gamma(p^{\infty})$-level structure such that $e_1 : \mathbb{Z}_p \xrightarrow{\sim} T_pA^0$ where $T_pA^0 \subset T_pA$ is the (rank 1) connected component, and $i$ is an $\mathcal{O}_D$-level structure.
\item The element $g \in GL_2(\mathbb{Q}_p)$ from (\ref{gdefinition}) acts on $\mathcal{Y}^{\mathrm{Ig}}$, giving an automorphism $g : \mathcal{Y}^{\mathrm{Ig}} \xrightarrow{\sim} \mathcal{Y}^{\mathrm{Ig}}$ which is a lift of relative Frobenius. Hence $\mathcal{Y}^{\mathrm{Ig}}$ is preperfectoid in the terminology of \cite[Definition 2.3.4]{ScholzeWeinstein}. Applying Proposition 2.3.6 of op. cit. to $\mathcal{Y}^{\mathrm{Ig}}$, we get an associated strong completion which by Proposition 2.4.4 of op. cit. satisfies
$$\hat{\mathcal{Y}}^{\mathrm{Ig}} \sim \mathcal{Y}^{\mathrm{Ig}},$$
where ``$\sim$'' is defined in Definition 2.4.1 of op. cit. The above relation between $\hat{\mathcal{Y}}^{\mathrm{Ig}}$ and $\mathcal{Y}^{\mathrm{Ig}}$ follows from Proposition 2.4.4 of op. cit., which also shows that $\hat{\mathcal{Y}}^{\mathrm{Ig}}$ is a perfectoid space over $\mathrm{Spa}(k,\mathcal{O}_k)$, recalling that $k$ from Definition \ref{kdefinition} is a perfectoid field. In particular, on underlying topological spaces we have a natural identification $|\hat{\mathcal{Y}}^{\mathrm{Ig}}| = |\mathcal{Y}^{\mathrm{Ig}}|$.

\item We note that $\hat{\mathcal{Y}}^{\mathrm{Ig}}$ also has an intrinsic definition as one of Mantovan's Igusa varieties, see \cite[p. 716]{CaraianiScholze} and \cite{Mantovan}. In particular, the definition of op. cit. is a formal scheme over $\mathrm{Spf}(W(\overline{\mathbb{F}}_p))$ where $W(\overline{\mathbb{F}}_p)$ is the Witt vectors of $\overline{\mathbb{F}}_p$, but the definition also works over $\mathrm{Spf}(\mathbb{Z}_p)$. The adic generic fiber of this formal scheme is our $\hat{\mathcal{Y}}^{\mathrm{Ig}}$. We will also consider the formal scheme version $\hat{\mathcal{Y}}^{\mathrm{Ig}}(0)^+$ later, see Definition \ref{hatmathcalYIgepsilonDefinition}.
\end{enumerate}
\end{definition}

%In terms of the adic fiber of the Hodge-Tate period map, we have (see \cite[Section 4]{CaraianiScholze})
%$$\overline{\hat{\mathcal{Y}}^{\mathrm{Ig}}} := \pi_{\mathrm{HT}}^{-1}(\{\infty\}) \subset \hat{Y}_{\infty},$$
%where the left-hand side denotes the adic closure of $\hat{\mathcal{Y}}^{\mathrm{Ig}} \subset \hat{Y}_{\infty}$. Here the above inclusion is an inclusion of adic spaces over $\mathrm{Spa}(\mathbb{Q}_p,\mathbb{Z}_p)$. (However, working under Convention \ref{basechangeconvention}, we will soon view these adic spaces as over $\mathrm{Spa}(k,\mathcal{O}_k)$.)

We have a natural sequence of surjective pro\'{e}tale maps 
\begin{equation}\label{Igusacovers}\mathcal{Y}^{\mathrm{Ig}} \rightarrow Y^{\mathrm{Ig}} \rightarrow Y^{\mathrm{ord}}, \hspace{1cm} (A,P,e_1,e_2) \mapsto (A,P,e_2 \pmod{T_pA^0}) \mapsto (A,P),
\end{equation}
where $Y^{\mathrm{Ig}}$ is the adic generic fiber of $Y^{\mathrm{Ig}}$ from Definition \ref{Igusadefinition}, $P$ is a $\Gamma$-level structure, $e_2 \pmod{T_pA^0} : \mathbb{Z}_p \xrightarrow{\sim} T_pA^{\text{\'e{t}}}$, and $Y^{\mathrm{ord}}$ is the ordinary locus of $Y$ (see Definition \ref{Yorddefinition} and Remark \ref{genericremark}). 

Recall that $Y^{\mathrm{ord}} \subset Y$ is an open sub-adic space (see Definition \ref{Yorddefinition} and Remark \ref{genericremark}). Thus $Y^{\mathrm{ord}} \in Y_{\text{pro\'{e}t}}$ (see Convention \ref{objectconvention}). Thus (\ref{Igusacovers}) shows
\begin{equation}\label{YIgopen}Y^{\mathrm{Ig}}, \mathcal{Y}^{\mathrm{Ig}} \in Y_{\text{pro\'{e}t}}.
\end{equation} 

\begin{definition}Given $\gamma \in GL_2(\mathbb{Q}_p)$, let $\gamma^*$ denote pullback by $\gamma$.
\end{definition}

As $\gamma = \left(\begin{array}{ccc} a & b\\
c & d\\
\end{array}\right) \in GL_2(\mathbb{Q}_p)$ acts on $z_{\mathrm{HT}}$ by 
\begin{equation}\label{zHTtransformationproperty}\gamma^*z_{\mathrm{HT}} = \frac{dz_{\mathrm{HT}} + b}{cz_{\mathrm{HT}} + a}
\end{equation}
(see Proposition \ref{modulartransformationproposition} below or \cite[Section 2.4]{ChojeckiHansenJohansson}; note that our object $z_{\mathrm{HT}}$ is equal to $1/\frak{z}$ in the notation of loc. cit.), we have
\begin{equation}\label{BQp}\mathrm{Stab}_{GL_2(\mathbb{Q}_p)}(\{z_{\mathrm{HT}} = \infty\}) = \left\{\left(\begin{array}{ccc} * & * \\
0 & *\\
\end{array}\right) \in GL_2(\mathbb{Q}_p)\right\} =:  B(\mathbb{Q}_p),
\end{equation}
\begin{equation}\label{BGalois}\mathrm{Gal}(\mathcal{Y}^{\mathrm{Ig}}/Y^{\mathrm{ord}}) = \left\{\left(\begin{array}{ccc} * & * \\
0 & *\\
\end{array}\right) \in GL_2(\mathbb{Z}_p)\right\} =: B.
\end{equation}
Here given a group $G$ acting on a set $S$ and subset $S'\subset S$, $\mathrm{Stab}_G(S') \subset G$ denotes the stabilizer subgroup of $S'$. From the map $\mathcal{Y}^{\mathrm{Ig}} \rightarrow Y^{\mathrm{Ig}}$ of (\ref{Igusacovers}), we thus see 
$$\mathrm{Gal}(\mathcal{Y}^{\mathrm{Ig}}/Y^{\mathrm{Ig}}) = \left\{\left(\begin{array}{ccc} * & * \\
0 & 1\\
\end{array}\right) \in GL_2(\mathbb{Z}_p)\right\} =: B^1.$$
Thus, 
\begin{equation}\label{YIgGaloisgroup}\mathrm{Gal}(Y^{\mathrm{Ig}}/Y^{\mathrm{ord}}) = B/B^1 \cong \mathbb{Z}_p^{\times}.
\end{equation}
Hence from (\ref{IgusaGaloisgroup}), we see that the exact sequence of Galois groups corresponding to (\ref{Igusacovers}) is 
$$1 \rightarrow B^1 \rightarrow B \rightarrow \mathbb{Z}_p^{\times} \rightarrow 1.$$

Since $\hat{\mathcal{Y}}^{\mathrm{Ig}} \sim \mathcal{Y}^{\mathrm{Ig}}$, the conclusions of the previous paragraph all hold with $\hat{\mathcal{Y}}^{\mathrm{Ig}}$ in place of $\mathcal{Y}^{\mathrm{Ig}}$. 

\subsection{Formal models of Igusa towers}

Return to the situation of Convention \ref{Yconvention}, so that $Y = Y(\Gamma)$ where $\Gamma = \Gamma(N)$ for some $N \ge 4$ that is prime to $p$.  In this section we will define generalized big Igusa towers $\mathcal{Y}^{\mathrm{Ig}}(\epsilon)$ and their formal models $\mathcal{Y}^{\mathrm{Ig}}(\epsilon)^+$. These towers are roughly defined as the loci in $Y_{\infty}$ on which the maximal level of canonical subgroup that exists on $|\mathrm{Ha}| \ge p^{-\epsilon}$ is trivialized and satisfy $\mathcal{Y}^{\mathrm{Ig}}(\epsilon)^{(+)} \subset \mathcal{Y}^{\mathrm{Ig}}(\epsilon')^{(+)}$ if $\epsilon \le \epsilon'$ (where ``$(+)$'' denotes the optional presence of a ``$+$''), and $\mathcal{Y}^{\mathrm{Ig}}(0) = \mathcal{Y}^{\mathrm{Ig}}$ where $\mathcal{Y}^{\mathrm{Ig}}$ is the big Igusa tower of Definition \ref{mathcalYIgDefinition}. In particular, the $\mathcal{Y}^{\mathrm{Ig}}(\epsilon)$ give a system of overconvergent neighborhoods of $\mathcal{Y}^{\mathrm{Ig}}$. Throughout the paper, the term ``canonical subgroup'' will refer to the notion first introduced by Lubin and Katz (see \cite[Chapter 3]{Katzpamf} for the modular curve case, and Kassaei \cite{Kassaei} and Goren-Kassaei \cite{GorenKassaei} for more general bases such as Shimura curves). In particular, when the canonical subgroup exists, it is the unique characteristic 0 lift of the kernel of relative Frobenius on the special fiber. We will give precise references to the definitions we will need in the discussion below. Although we work in the infinite-level Shimura curve setting later, all canonical subgroups at infinite level that we consider will be pullbacks of canonical subgroups on finite level Shimura curves over complete DVRs (as constructed in op. cit.) along the maps $Y_{\infty} \rightarrow Y$, $Y_{\infty}^+ \rightarrow Y^+$, or $\hat{Y}_{\infty} \rightarrow Y$ or $\hat{Y}_{\infty}^+ \rightarrow Y^+$ (see Convention \ref{Yconvention}). 

\begin{convention}We follow the convention of \cite[Section 2.2]{ChojeckiHansenJohansson}, and view canonical subgroups of false elliptic curves as subgroups of $A[p^{\infty}]$ where $A[p^{\infty}]$ is the $p$-divisible group defined in Convention \ref{idempotentconvention} (obtained by applying the idempotent $e^1$ to the usual $p$-divisible group of $A$). This will be compatible with the notation $(A,P,e_1,e_2)$ defined in Definition \ref{YGamma+definition}.
\end{convention}

\begin{convention}\label{tameconvention}Henceforth, we will often suppress the $\Gamma$-level structure when considering points in $Y$ or $Y_{\infty}$ when no confusion should arise. 
\end{convention}

Recall the notation of Convention \ref{Yconvention}. 

\begin{definition}\label{universaldefinition}Let $\mathcal{E} \rightarrow Y_{\infty}$, $\mathcal{E}^+ \rightarrow Y_{\infty}^+$ and $\hat{\mathcal{E}}^+ \rightarrow \hat{Y}_{\infty}^+$ denote the pullbacks of the universal object (\ref{adicuniversalobject}), (\ref{integraluniversalobject}) and (\ref{formaluniversalobject}) along $Y_{\infty} \rightarrow Y$, $Y_{\infty}^+ \rightarrow Y^+$ and $\hat{Y}_{\infty}^+ \rightarrow Y^+$, respectively. Let $\mathcal{E}(\epsilon) \rightarrow Y_{\infty}(\epsilon)$ and $\mathcal{E}^+(\epsilon) \rightarrow Y_{\infty}^+(\epsilon)$ denote the restrictions of these universal objects to $Y_{\infty}(\epsilon) \subset Y_{\infty}$ and $Y_{\infty}^+(\epsilon) \subset Y_{\infty}^+$. 
\end{definition}

We define a natural function which describes the maximal level of canonical subgroup that exists on the formal scheme $Y^+(\epsilon)$. 

\begin{definition}Given $\epsilon \ge 0$, let $n(\epsilon) \in \mathbb{Z}_{\ge 0} \cup \{\infty\}$ be the maximum level of canonical subgroup that exists in all false elliptic curves $A$ with $|\mathrm{Ha}(A)| \ge p^{-\epsilon}$. By the maximality statement of \cite[Theorem 3.9]{GorenKassaei}, we have the equivalent definition 
\begin{equation}\label{nepsilondefinition}n(\epsilon) := \mathrm{max}(\lceil 1 - \log_p(\epsilon) - \log_p(p+1)\rceil, 0),
\end{equation}
where $\log_p(x) = \log_e(x)/\log_e(p)$ and $\log_e$ is the usual archimedean natural logarithm. In other words, $n(\epsilon) \in \mathbb{Z}_{\ge 0}$ is equal to the maximum of the following two integers:
\begin{enumerate}
\item the maximum $n \in \mathbb{Z}$ such that $\epsilon < 1/(p^{n-2}(p+1))$, 
\item 0.
\end{enumerate}
\end{definition}

\begin{definition}\label{canonicaldefinition}\begin{enumerate}
\item Let 
$$C_{n(\epsilon)}^+\subset \mathcal{E}^+(\epsilon)[p^{n(\epsilon)}]$$
be the level $n(\epsilon)$ canonical subgroup (i.e. of order $p^{n(\epsilon)}$) over $Y_{\infty}^+(\epsilon)$. Let $$C_{n(\epsilon)} \subset \mathcal{E}(\epsilon)[p^{n(\epsilon)}]$$
denote its adic generic fiber. 
\item Let 
$$(e_1,e_2) : \hat{\mathbb{Z}}_{p,Y_{\infty}}^{\oplus 2} \xrightarrow{\sim} T_p\mathcal{E}$$
denote the $\Gamma(p^{\infty})$-level structure on $Y_{\infty}$. 

\item Let
\begin{equation}\label{Drinfeld1}(e_1^+,e_2^+) : \mathbb{Z}_p^{\oplus 2} \rightarrow T_p\mathcal{E}^+(Y_{\infty})
\end{equation}
be the $\Gamma(p^{\infty})$-Drinfeld level on $\mathcal{E}^+ \rightarrow Y_{\infty}^+$. %This induces a map of $\hat{\mathbb{Z}}_{p,Y_{\infty}}$-local systems 
%\begin{equation}\label{Drinfeld2}(e_1^+,e_2^+) : \hat{\mathbb{Z}}_{p,Y_{\infty}^+}^{\oplus 2} \rightarrow T_p\mathcal{E}^+.
%\end{equation}
%Conversely, (\ref{Drinfeld2}) induces (\ref{Drinfeld1}) by taking $Y_{\infty}$-points, so the two structures are equivalent. We will often prefer to view $(e_1^+,e_2^+)$ as a map (\ref{Drinfeld2}). Here, recall that $Y_{\infty}^+$ is a scheme over $\mathrm{Spec}(\mathbb{Z}_p)$ (see Convention \ref{Yconvention}), and its $p$-adic completion $\hat{Y}_{\infty}^+$ is a formal scheme over $\mathrm{Spf}(\mathbb{Z}_p)$, and can thus be viewed as a pro-adic space over $\mathrm{Spa}(\mathbb{Z}_p,\mathbb{Z}_p)$ by applying Convention \ref{formaladicfunctorconvention}. Thus Definition \ref{Zphatdefinition} can be applied to $\hat{Y}_{\infty}^+$. 
\item Given $\Gamma_1(p^{\infty})$ (Drinfeld)-level structures, 
$$e_i : \hat{\mathbb{Z}}_{p,Y_{\infty}} \rightarrow T_p\mathcal{E}, \hspace{.5cm} e_i^+ : \hat{\mathbb{Z}}_{p,Y_{\infty}} \rightarrow T_p\mathcal{E}^+, \hspace{.5cm} e_i^+ : \hat{\mathbb{Z}}_{p,Y_{\infty}} \rightarrow T_p\hat{\mathcal{E}}^+,$$
let 
\begin{equation}\label{eindefinition}\begin{split}e_{i,n} &= e_i \pmod{p^n} : (\mathbb{Z}/p^n)_{Y_{\infty}} \rightarrow \mathcal{E}[p^n],\\
e_{i,n}^+ &= e_i^+ \pmod{p^n} : (\mathbb{Z}/p^n)_{Y_{\infty}^+} \rightarrow \mathcal{E}^+[p^n],\\
e_{i,n}^+ &= e_i^+ \pmod{p^n} : (\mathbb{Z}/p^n)_{Y_{\infty}^+} \rightarrow \hat{\mathcal{E}}^+[p^n].
\end{split}
\end{equation}
\end{enumerate}
\end{definition}

As per Definition \ref{YGamma+definition}, we will let $\langle e_{i,n}\rangle$ (resp. $\langle e_{i,n}^+\rangle$) denote the $\mathcal{O}_D$-module generated by $e_{i,n}$ (resp. $e_{i,n}^+$). Thus $\langle e_{i,n}\rangle$ (resp. $\langle e_{i,n}^+\rangle$) is an $\mathcal{O}_D$-submodule of the $p^n$-torsion group scheme of $\mathcal{E}$ (resp. $\mathcal{E}^+$). 

We introduce the following generalized Igusa tower $\mathcal{Y}^{\mathrm{Ig}}(\epsilon)^+$ indexed by $\epsilon > 0$ in the valuation group of $\mathcal{O}_k$ (see Definition \ref{kdefinition}) and satisfying $\mathcal{Y}^{\mathrm{Ig}}(\epsilon)^+ \subset \mathcal{Y}^{\mathrm{Ig}}(\epsilon')$ if $\epsilon \le \epsilon'$. In particular, $\mathcal{Y}^{\mathrm{Ig}}(0) = \mathcal{Y}^{\mathrm{Ig}}$ from Definition \ref{mathcalYIgDefinition}. 

\begin{definition}\label{plusdefinitions}\begin{enumerate}
\item Define the pro-formal scheme
\begin{equation}\label{Uplus}\mathcal{Y}^{\mathrm{Ig}}(\epsilon)^+ := (e_{1,n(\epsilon)}^+)^{-1}(C_{n(\epsilon)}^+) \cap Y_{\infty}^+(\epsilon),
\end{equation}
which is defined over $\mathrm{Spf}(\mathcal{O}_k)$ (where $k$ is as in Definition \ref{kdefinition}). 

\item Note that $Y_{\infty}^+(\epsilon)$ is the preimage of the open formal subscheme $Y^+(\epsilon) \subset Y$ under the map $Y_{\infty}^+ \rightarrow Y^+$. Also, $e_{1,n(\epsilon)}^{-1}(C_{n(\epsilon)}^+)$ is the inverse image under $e_{1,n}^+$ of the sub-group scheme $C_{n(\epsilon)}^+ \subset \mathcal{E}^+(\epsilon)[p^{n(\epsilon)}]$. Note that this inverse image is open and closed when $\epsilon > 0$, since $C_{n(\epsilon)}^+$ is the canonical subgroup and so in particular a connected component over $Y_{\infty}^+(\epsilon)$ of the group scheme $\mathcal{E}^+[p^{n(\epsilon)}]$. Hence $\mathcal{Y}^{\mathrm{Ig}}(\epsilon)^+ \subset Y_{\infty}^+(\epsilon)$ is an open sub-pro formal scheme when $\epsilon > 0$. 

\item Define the pro-adic space
\begin{equation}\label{U}\mathcal{Y}^{\mathrm{Ig}}(\epsilon) := e_{1,n(\epsilon)}^{-1}(C_{n(\epsilon)}) \cap Y_{\infty}(\epsilon)
\end{equation}
which is defined over $\mathrm{Spa}(k,\mathcal{O}_k)$. Since $\mathcal{Y}^{\mathrm{Ig}}(\epsilon)$ is the adic generic fiber of $\mathcal{Y}^{\mathrm{Ig}}(\epsilon)^+$, it is an open sub-pro adic space $\mathcal{Y}^{\mathrm{Ig}}(\epsilon) \subset Y_{\infty}$ when $\epsilon > 0$. Since $Y_{\infty}(\epsilon) \in Y_{\text{pro\'{e}t}}$ and $e_{1,n(\epsilon)}^{-1}(C_{n(\epsilon)})$ is open when $\epsilon > 0$, then $\mathcal{Y}^{\mathrm{Ig}}(\epsilon) \in Y_{\text{pro\'{e}t}}$ in this case. 

\item For $\epsilon \ge 0$, the above definitions show
$$\mathcal{Y}^{\mathrm{Ig}}(\epsilon) \in Y_{\text{pro\'{e}t}}.$$

\item Note also that 
$$\mathcal{Y}^{\mathrm{Ig}}(0) = \mathcal{Y}^{\mathrm{Ig}} \in Y_{\text{pro\'{e}t}}$$
by Definition \ref{mathcalYIgDefinition} and (\ref{YIgopen}). Note that this pro-adic space is defined over $\mathrm{Spa}(\mathbb{Q}_p,\mathbb{Z}_p)$, but per Convention \ref{basechangeconvention} we will often consider it as defined over $\mathrm{Spa}(k,\mathcal{O}_k)$. %By (\ref{YIgopen}) we have $\mathcal{Y}^{\mathrm{Ig}}(0) = \mathcal{Y}^{\mathrm{Ig}} \in Y_{\text{pro\'{e}t}}$. 

\end{enumerate}
\end{definition}

We will later show that $\mathcal{Y}^{\mathrm{Ig}}(\epsilon) \in Y_{\text{pro\'{e}t}}$ is affinoid perfectoid over $\mathrm{Spa}(k,\mathcal{O}_k)$ in the sense of \cite[Definition 4.3]{Scholze}, see Lemma \ref{Uaffinoidperfectoidlemma}.

\subsection{Strong completions of formal models of Igusa towers}\label{strongformalsection}

Now base change back to $\mathrm{Spa}(k,\mathcal{O}_k)$ (per Convention \ref{basechangeconvention}). In the next few sections (Sections \ref{strongformalsection} through  \ref{spreadingoutsection}) we will study formal schemes which arise as strong completions of the pro-formal schemes introduced in (\ref{Uplus}).

%\begin{definition}\label{mathcalYIgepsilon+}Recall the pro-formal scheme $\mathcal{Y}^{\mathrm{Ig}}(\epsilon)^+$ (see (\ref{Uplus})). That is, choose a cover of $\mathcal{Y}^{\mathrm{Ig}}(\epsilon)^+$ by $\{U_i = \varprojlim_n U_{i,n}\}_i$ where each $U_{i,n}$ is an affine formal scheme and transition maps are finite. define the formal scheme 
%$$\hat{U}_i = \mathrm{Spf}(\varprojlim_m\varinjlim_n \mathbf{\Gamma}(\mathcal{O}_{U_{i,n}})/(p^m))$$
%and glue over all $i$. Call the formal scheme obtained $\hat{\mathcal{Y}}^{\mathrm{Ig}}(\epsilon)^+$. This is defined over $\mathrm{Spf}(\mathcal{O}_k)$ by construction. Following the notation of Definition \ref{OYdefinition}, we let $\mathcal{O}_{\hat{\mathcal{Y}}^{\mathrm{Ig}}(\epsilon)^+}$ be its structure sheaf. Also by construction, we have 
%$$\hat{\mathcal{Y}}^{\mathrm{Ig}}(\epsilon)^+ \sim \mathcal{Y}^{\mathrm{Ig}}(\epsilon)^+$$
%in the notation of \cite[Definition 2.4.1]{ScholzeWeinstein}. We could also use the presentation 
%\end{definition}

\begin{definition}\label{hatmathcalYIgepsilonDefinition}
\begin{enumerate}
\item For any $n \in \mathbb{Z}_{\ge 0}$, in the notation of Section \ref{formalShimurasection} and letting $\Gamma = \Gamma(N)$ as in Convention \ref{Yconvention},  in the notation of Section \ref{formalShimurasection} let 
$$Y_n^+ := Y(\Gamma \cap \Gamma(p^n))^+$$
(which is a scheme over $\mathrm{Spec}(\mathbb{Z}_p)$) and 
$$\hat{Y}_n^+ := \widehat{Y(\Gamma \cap \Gamma(p^n))}^+$$
(which is a formal scheme over $\mathrm{Spf}(\mathbb{Z}_p)$). 
%Let $Y_n$ be the adic generic fiber (an adic space over $\mathrm{Spa}(\mathbb{Q}_p,\mathbb{Z}_p)$) of the formal scheme $\hat{Y}_n^+$. 
\item Let $\hat{\mathcal{E}}^+ \rightarrow \hat{Y}^+$ be the universal object of (\ref{formaluniversalobject}). Recall $C_{n(\epsilon)}^+ \subset \hat{\mathcal{E}}^+[p^n]$ denotes the canonical subgroup of order $p^{n(\epsilon)}$ (Definition \ref{canonicaldefinition} (1)), and recall $e_{i,n}$ from (\ref{eindefinition}) and $n(\epsilon)$ from (\ref{nepsilondefinition}). Write 
$$Y_n^{\mathrm{Ig}}(\epsilon)^+ := \left(Y^+(\epsilon) \times_{\hat{Y}^+}\hat{Y}_n^+\right) \cap e_{1,n}^{-1}(C_{n(\epsilon)}^+),$$
the fiber product being taken in the category of formal schemes. Thus, $Y_n^{\mathrm{Ig}}(\epsilon)^+$ is formal scheme over $\mathrm{Spf}(k,\mathcal{O}_k)$. As  $Y^+(\epsilon)$, $\hat{Y}^+$, $\hat{Y}_n^+$ and $C_{n(\epsilon)}^+$ are formal completions of locally noetherian schemes along locally noetherian closed subschemes, each $Y_n^{\mathrm{Ig}}(\epsilon)^+$ is a locally noetherian formal scheme. Recall the pro-formal scheme $\mathcal{Y}^{\mathrm{Ig}}(\epsilon)^+$ from (\ref{Uplus}). Then 
\begin{equation}\label{firstpresentation}\mathcal{Y}^{\mathrm{Ig}}(\epsilon)^+ = \varprojlim_n Y_n^{\mathrm{Ig}}(\epsilon)^+.\end{equation}
\item Taking an affine cover of $\mathcal{Y}^{\mathrm{Ig}}(\epsilon)^+$ (which amounts to inverse systems of affine covers of each locally noetherian formal scheme $Y_n^{\mathrm{Ig}}(\epsilon)^+$; each member of such a cover is of the form $\mathrm{Spa}(A,A)$ where $A$ admits a ring of definition with finitely generated ideal of definition in the sense of \cite[Definition 2.1.1]{ScholzeWeinstein}), applying Proposition 2.4.2 of op. cit. to each member of this cover and then gluing the resulting affine formal schemes, we obtain a formal scheme over $\mathrm{Spf}(\mathcal{O}_k)$
$$\hat{\mathcal{Y}}^{\mathrm{Ig}}(\epsilon)^+ \sim \mathcal{Y}^{\mathrm{Ig}}(\epsilon)^+$$
where the relation ``$\sim$'' is as in Definition 2.4.1 of op. cit. Since $\mathcal{Y}^{\mathrm{Ig}}(\epsilon)^+ \subset Y_{\infty}^+$ is an open sub-pro-formal scheme if $\epsilon > 0$ (see Definition \ref{plusdefinitions}), $\hat{\mathcal{Y}}^{\mathrm{Ig}}(\epsilon)^+ \subset \hat{Y}_{\infty}^+$ is an open sub-formal scheme if $\epsilon > 0$. 

\item Let $Y_n^{\mathrm{Ig}}(\epsilon)$ denote the adic generic fiber (an adic space over $\mathrm{Spa}(k,\mathcal{O}_k)$) of $Y_n^{\mathrm{Ig}}(\epsilon)^+$. Recall the pro-adic space $\mathcal{Y}^{\mathrm{Ig}}(\epsilon)$ from (\ref{U}). Then
\begin{equation}\label{secondpresentation}\mathcal{Y}^{\mathrm{Ig}}(\epsilon) = \varprojlim_n Y_n^{\mathrm{Ig}}(\epsilon).
\end{equation}

\item Note that in the above description (\ref{secondpresentation}), each $Y_n^{\mathrm{Ig}}(\epsilon)$ is locally of topologically finite type, and so locally is of the form $\mathrm{Spf}(A,A^+)$ where $A$ admits a ring of definition with finitely generated ideal of definition (see \cite[Definition 2.1.1]{ScholzeWeinstein}). Thus, applying Proposition 2.4.2 of op. cit. to an affinoid open cover of $\varprojlim_n Y_n^{\mathrm{Ig}}(\epsilon)$ and then gluing, we have associated adic spaces $\hat{\mathcal{Y}}^{\mathrm{Ig}}(\epsilon)$ over $\mathrm{Spa}(k,\mathcal{O}_k)$ satisfying
\begin{equation}\label{hatmathcalYIg}\hat{\mathcal{Y}}^{\mathrm{Ig}}(\epsilon) \sim \mathcal{Y}^{\mathrm{Ig}}(\epsilon).
\end{equation}
By Proposition 2.4.4 of op. cit. this is the strong completion (in the sense of the discussion after Proposition 2.3.6 of op. cit.) of the affinoid perfectoid $\mathcal{Y}^{\mathrm{Ig}}(\epsilon)$. Here, ``$\sim$'' is the relation of Section 2.3 of op. cit.

\item We note that, by construction, $\hat{\mathcal{Y}}^{\mathrm{Ig}}(\epsilon)^+$ over $\mathrm{Spa}(\mathcal{O}_k,\mathcal{O}_k)$ is a formal model of the adic space $\hat{\mathcal{Y}}^{\mathrm{Ig}}(\epsilon)$ over $\mathrm{Spa}(k,\mathcal{O}_k)$, in the terminology of Definition \ref{formalmodeldefinition}.
\end{enumerate}
\end{definition}

By the discussion of Definition \ref{hatmathcalYIgepsilonDefinition}, each $Y_n^{\mathrm{Ig}}(\epsilon)$ is a locally noetherian adic space. Thus, by  (\ref{secondpresentation}) and (\ref{hatmathcalYIg}) and \cite[Theorem 2.4.7]{ScholzeWeinstein} (which is a restatement of \cite[Theorem 7.17]{Scholzeperf}), we have an equivalence of the \'{e}tale topos $\hat{\mathcal{Y}}^{\mathrm{Ig}}(\epsilon)_{\text{\'{e}t}}$ with the projective limit of the \'{e}tale topoi $Y_n^{\mathrm{Ig}}(\epsilon)_{\text{\'{e}t}}$.

\begin{definition}\begin{enumerate}
\item Let $\hat{\mathbb{Z}}_{p,\hat{\mathcal{Y}}^{\mathrm{Ig}}(\epsilon)}$ denote the sheaf on $\hat{\mathcal{Y}}^{\mathrm{Ig}}(\epsilon)_{\text{\'{e}t}}$ corresponding under this equivalence of topoi to the pullback of any $\hat{\mathbb{Z}}_{p,\hat{Y}_n^{\mathrm{Ig}}(\epsilon)}$ (defined in Definition \ref{Zphatdefinition} on the site $Y_n^{\mathrm{Ig}}(\epsilon)_{\text{\'{e}t}}$) to the projective limit of the \'{e}tale topoi $Y_n^{\mathrm{Ig}}(\epsilon)_{\text{\'{e}t}}$. 

\item Similarly, using $\hat{Y}_{\infty} \sim \varprojlim_n Y_n$ and remarking that each $Y_n$ is a locally noetherian adic space, we have an equivalence between the \'{e}tale topos $\hat{Y}_{\infty,\text{\'{e}t}}$ and the projective limit of \'{e}tale topoi $Y_{n,\text{\'{e}t}}$. We define $\hat{\mathbb{Z}}_{p,\hat{Y}_{\infty}}$ to be the sheaf corresponding under this equivalence to the pullback of any $\hat{\mathbb{Z}}_{p,\hat{Y}_n}$ (defined in Definition \ref{Zphatdefinition} on the site $Y_{n,\text{\'{e}t}}$) to the projective limit of the \'{e}tale topoi $Y_{n,\text{\'{e}t}}$. 
\end{enumerate}
\end{definition}

We will often make use of the following Convention from now on. 

\begin{convention}\label{Gammaconvention}\begin{enumerate}
\item Given a sheaf $\mathcal{F}$ on any site $\mathcal{S}$, we will sometimes let $\mathbf{\Gamma}(\mathcal{F})$ denote the global sections of $\mathcal{F}$. This $\mathbf{\Gamma}$ has nothing to do with the level $\Gamma$ appearing in Convention \ref{Yconvention}, or with the Galois group $\Gamma_K$ in (\ref{GammaK}) below. On the other hand, when we work with a sheaf $\mathcal{F}$ on $\mathcal{S}$ and consider various objects $U \in \mathrm{Obj}(\mathcal{S})$, we will use the usual notation $\mathcal{F}(U)$ to denote the sections on $U$. 
\item Given a point $\mathbf{y}$ of $\mathcal{S}$, we will let $\mathcal{F}(\mathbf{y})$ denote the evaluation of $\mathcal{F}$ to $\mathbf{y}$, provided this evaluation is well-defined (as is often the case when $\mathcal{F}$ is a structure sheaf or a period sheaf $\mathbf{O}$ on the pro\'{e}tale site as in Section \ref{periodsheavessection} (e.g. $\mathbf{O} = \mathcal{O}^{(+)}, \mathbb{B}_{\mathrm{dR}}^{(+)}, \mathcal{O}\mathbb{B}_{\mathrm{dR}}^{(+)}$), or $\mathcal{F}$ is a sheaf of $\mathbf{O}$-modules for such an $\mathbf{O}$).
\end{enumerate}
\end{convention}

\subsection{The overconvergent neighborhood}\label{overconvergentsection}

We continue to use the notation of Convention \ref{Yconvention}, Definition \ref{canonicaldefinition}, and Definition \ref{hatmathcalYIgepsilonDefinition}. From the inclusion $\mathcal{Y}^{\mathrm{Ig}}(0) \subset \mathcal{Y}^{\mathrm{Ig}}(\epsilon)$ (resp. $\hat{\mathcal{Y}}^{\mathrm{Ig}}(0) \subset \hat{\mathcal{Y}}^{\mathrm{Ig}}(\epsilon)$) for every $\epsilon \ge 0$, we get a natural maps of pro-adic spaces over $\mathrm{Spa}(k,\mathcal{O}_k)$
$$\mathcal{Y}^{\mathrm{Ig}}(0) \rightarrow \varprojlim_{0 < \epsilon < p/(p+1)}\mathcal{Y}^{\mathrm{Ig}}(\epsilon), \hspace{1cm} \hat{\mathcal{Y}}^{\mathrm{Ig}}(0) \rightarrow \varprojlim_{0 < \epsilon < p/(p+1)}\hat{\mathcal{Y}}^{\mathrm{Ig}}(\epsilon)$$
where the transition maps are given by the inclusions $\mathcal{Y}^{\mathrm{Ig}}(\epsilon) \subset \mathcal{Y}^{\mathrm{Ig}}(\epsilon')$ and $\hat{\mathcal{Y}}^{\mathrm{Ig}}(\epsilon) \subset \hat{\mathcal{Y}}^{\mathrm{Ig}}(\epsilon')$ for $\epsilon \le \epsilon'$. 

The $p$-adic completion of the inverse limit on the right-hand sides of these maps has the natural structure of an adic space. 

\begin{proposition}There is an adic space $\hat{\mathcal{Y}}^{\mathrm{Ig}}(0)^{\dagger}$ over $\mathrm{Spa}(k,\mathcal{O}_k)$ with
\begin{equation}\label{thirdpresentation}\hat{\mathcal{Y}}^{\mathrm{Ig}}(0)^{\dagger} \sim \varprojlim_{0 < \epsilon < p/(p+1)} \mathcal{Y}^{\mathrm{Ig}}(\epsilon) \sim \varprojlim_{0 < \epsilon < p/(p+1)} \hat{\mathcal{Y}}^{\mathrm{Ig}}(\epsilon),
\end{equation}
where ``$\sim$'' denotes the relation of \cite[Definition 2.4.1]{ScholzeWeinstein}. This induces natural maps of adic spaces 
\begin{equation}\label{daggermaps}i^{\dagger} : \hat{\mathcal{Y}}^{\mathrm{Ig}}(0) \rightarrow \hat{\mathcal{Y}}^{\mathrm{Ig}}(0)^{\dagger}, \hspace{1cm} \hat{\mathcal{Y}}^{\mathrm{Ig}}(0)^{\dagger} \rightarrow \hat{\mathcal{Y}}^{\mathrm{Ig}}(\epsilon), \hspace{1cm}\hat{\mathcal{Y}}^{\mathrm{Ig}}(0)^{\dagger} \rightarrow \hat{Y}_{\infty}
\end{equation}
over $\mathrm{Spa}(k,\mathcal{O}_k)$ for all $0 < \epsilon < p/(p+1)$, such that the composition of the second map of (\ref{daggermaps}) with the natural inclusion $\hat{\mathcal{Y}}^{\mathrm{Ig}}(\epsilon) \subset \hat{Y}_{\infty}$ is equal to the third map. 
\end{proposition}

\begin{proof}From (\ref{hatmathcalYIg}) we have $\hat{\mathcal{Y}}^{\mathrm{Ig}}(\epsilon) \sim \mathcal{Y}^{\mathrm{Ig}}(\epsilon)$ for each $\epsilon$. Thus, the second ``$\sim$'' in (\ref{thirdpresentation}) follows, and it suffices to show the existence of $\hat{\mathcal{Y}}^{\mathrm{Ig}}(0)^{\dagger}$ satisfying the first ``$\sim$''. For this, recall from the discussion in Definition \ref{hatmathcalYIgepsilonDefinition} that each $\mathcal{Y}^{\mathrm{Ig}}(\epsilon) = \varprojlim_n Y_n^{\mathrm{Ig}}(\epsilon)$ where each $Y_n^{\mathrm{Ig}}(\epsilon)$ is an adic space over $\mathrm{Spa}(k,\mathcal{O}_k)$ that is locally topologically of finite type, and so locally is of the form $\mathrm{Spa}(A,A^+)$ where $A$ admits a ring of definition with a finitely generated ideal of definition (in the sense of \cite[Definition 2.11]{ScholzeWeinstein}). Hence, applying Proposition 2.4.2 of op. cit. to an affinoid cover of the double inverse limit
\begin{equation}\label{doubleinversepresentation}\varprojlim_{0 < \epsilon < p/(p+1)}\mathcal{Y}^{\mathrm{Ig}}(\epsilon) = \varprojlim_{0 < \epsilon < p/(p+1)}\varprojlim_n Y_n^{\mathrm{Ig}}(\epsilon)
\end{equation}
and then gluing, we get an adic space $\hat{\mathcal{Y}}^{\mathrm{Ig}}(0)^{\dagger}$ over $\mathrm{Spa}(k,\mathcal{O}_k)$ satisfying the first ``$\sim$'' of (\ref{thirdpresentation}). This finishes the proof of (\ref{thirdpresentation}). 

Now we construct the maps (\ref{daggermaps}). Note that we have inclusions $\hat{\mathcal{Y}}^{\mathrm{Ig}}(0) \subset \hat{\mathcal{Y}}^{\mathrm{Ig}}(\epsilon)$ for every $\epsilon > 0$ (see Definition \ref{hatmathcalYIgepsilonDefinition}), which are compatible with the inclusions $\hat{\mathcal{Y}}^{\mathrm{Ig}}(\epsilon) \subset \hat{\mathcal{Y}}^{\mathrm{Ig}}(\epsilon')$ for every $\epsilon \le \epsilon'$. Hence we have a map $\hat{\mathcal{Y}}^{\mathrm{Ig}}(0) \rightarrow \varprojlim_{0 < \epsilon < p/(p+1)} \hat{\mathcal{Y}}^{\mathrm{Ig}}(\epsilon)$.  This extends to the $p$-adic completion, which by (\ref{thirdpresentation}) gives the first map of (\ref{daggermaps}). For the second map of (\ref{daggermaps}), first note that there is a natural map of pro-adic spaces $\varprojlim_{0 < \epsilon < p/(p+1)} \hat{\mathcal{Y}}^{\mathrm{Ig}}(\epsilon) \rightarrow \hat{\mathcal{Y}}^{\mathrm{Ig}}(\epsilon)$ for every $0 < \epsilon < p/(p+1)$, given by projection to the $\epsilon$-term of the inverse limit. These extend to the $p$-adic completion, and so from (\ref{thirdpresentation}) we get the map of adic spaces $\hat{\mathcal{Y}}^{\mathrm{Ig}}(0)^{\dagger} \rightarrow \hat{\mathcal{Y}}^{\mathrm{Ig}}(\epsilon)$ over $\mathrm{Spa}(k,\mathcal{O}_k)$. Composing this with the natural inclusion $\hat{\mathcal{Y}}^{\mathrm{Ig}}(\epsilon) \subset \hat{Y}_{\infty}$, we get the map of adic spaces $\hat{\mathcal{Y}}^{\mathrm{Ig}}(0)^{\dagger} \rightarrow \hat{Y}_{\infty}$ over $\mathrm{Spa}(k,\mathcal{O}_k)$. Note that since the composition $\hat{\mathcal{Y}}^{\mathrm{Ig}}(\epsilon) \subset \hat{\mathcal{Y}}^{\mathrm{Ig}}(\epsilon') \subset \hat{Y}_{\infty}$ is equal to $\hat{\mathcal{Y}}^{\mathrm{Ig}}(\epsilon) \subset \hat{Y}_{\infty}$ for every $\epsilon \le \epsilon'$, the map $\hat{\mathcal{Y}}^{\mathrm{Ig}}(0)^{\dagger} \rightarrow \hat{Y}_{\infty}$ does not depend on the $\epsilon$ used in its construction. This gives the third map of (\ref{daggermaps}). 

\end{proof}

One may think of $\hat{\mathcal{Y}}^{\mathrm{Ig}}(0)^{\dagger}$ as a kind of ``overconvergent neighborhood'' of $\hat{\mathcal{Y}}^{\mathrm{Ig}}(0)$ inside $\hat{Y}_{\infty}$. However, we will soon show that $\hat{\mathcal{Y}}^{\mathrm{Ig}}(0)^{\dagger}$ is equal to $\hat{\mathcal{Y}}^{\mathrm{Ig}}(0)$ (see (\ref{daggerequalities})). First, we define a natural contraction $\varPhi : \hat{\mathcal{Y}}^{\mathrm{Ig}}(0)^{\dagger} \rightarrow \hat{\mathcal{Y}}^{\mathrm{Ig}}(0)$ via the universal property of the universal object $\mathcal{E}(0)^{\dagger} \rightarrow \hat{\mathcal{Y}}^{\mathrm{Ig}}(0)$ (see (\ref{universalcanonicaldagger}) below). This map will be the inverse of $i^{\dagger}$ from (\ref{daggermaps}). 

Recall the notation of Definition \ref{universaldefinition}, (\ref{nepsilondefinition}) and Definition \ref{canonicaldefinition}. From (\ref{U}), the level structure $e_{1,n(\epsilon)} : \mathcal{Y}^{\mathrm{Ig}}(\epsilon) \rightarrow \mathcal{E}(\epsilon)[p^{n(\epsilon)}]$ trivializes the order-$p^{n(\epsilon)}$ canonical subgroup $C_{n(\epsilon)}$. It is clear that if $\epsilon \le \epsilon'$, then in the notation of (\ref{eindefinition}),
\begin{equation}\label{e1ncompatible}e_{1,n(\epsilon')}|_{\mathcal{Y}^{\mathrm{Ig}}(\epsilon)} = e_{1,n(\epsilon)} : \mathcal{Y}^{\mathrm{Ig}}(\epsilon) \rightarrow \mathcal{E}(\epsilon)[p^{n(\epsilon)}].
\end{equation}

\begin{definition}\begin{enumerate}
\item Define the fiber product adic spaces 
\begin{equation}\label{universalcanonicaldagger}\mathcal{E}(0)^{\dagger} := \hat{\mathcal{E}} \times_{\hat{Y}_{\infty}} \hat{\mathcal{Y}}^{\mathrm{Ig}}(0)^{\dagger}, \hspace{1cm} C_{n(\epsilon)}(0)^{\dagger} := C_{n(\epsilon)}(\epsilon) \times_{\hat{\mathcal{Y}}^{\mathrm{Ig}}(\epsilon)}\hat{\mathcal{Y}}^{\mathrm{Ig}}(0)^{\dagger} \subset \mathcal{E}(0)^{\dagger}[p^{n(\epsilon)}].
\end{equation}
These are well-defined by (\ref{thirdpresentation}) and \cite[Proposition 2.4.3]{ScholzeWeinstein}.  Note also that as $\epsilon \rightarrow 0$ in the valuation group of $\mathcal{O}_k$ (see Definition \ref{kdefinition}), $n(\epsilon)$ ranges over all non-negative integers, i.e. for every $n \in \mathbb{Z}_{\ge 0}$ there exists $\epsilon > 0$ with $n(\epsilon) = n$. Hence $C_n(0)^{\dagger}$ is well-defined for every $n \in \mathbb{Z}_{\ge 0}$. 
\item Thus, from the second map of (\ref{thirdpresentation}) and the compatibility (\ref{e1ncompatible}), we get a trivialization
\begin{equation}\label{e1trivialization}e_1 : \hat{\mathbb{Z}}_{p,\hat{\mathcal{Y}}^{\mathrm{Ig}}(0)^{\dagger}} \xrightarrow{\sim} \varprojlim_n C_n(0)^{\dagger} \subset T_p\mathcal{E}(0)^{\dagger}.
\end{equation}
Here, $\hat{\mathbb{Z}}_{p,\hat{\mathcal{Y}}^{\mathrm{Ig}}(0)^{\dagger}}$ is defined using Definition \ref{Zphatdefinition}, (\ref{thirdpresentation}), (\ref{doubleinversepresentation}) and Theorem 2.4.7 of op. cit. (or see \cite[Theorem 7.17]{Scholzeperf}). That is, since each $Y_n^{\mathrm{Ig}}(\epsilon)$ is a strongly noetherian adic space, by op. cit. we have an equivalence of the \'{e}tale topos $\hat{Y}^{\mathrm{Ig}}(0)_{\text{\'{e}t}}^{\mathrm{Ig}}$ with the projective limit of the \'{e}tale topoi $Y_n^{\mathrm{Ig}}(\epsilon)_{\text{\'{e}t}}$. Then $\hat{\mathbb{Z}}_{p,\hat{\mathcal{Y}}^{\mathrm{Ig}}(0)^{\dagger}}$ is the sheaf on $\hat{Y}^{\mathrm{Ig}}(0)_{\text{\'{e}t}}^{\dagger}$ corresponding under this equivalence of topoi to the pullback of any $\hat{\mathbb{Z}}_{p,\hat{Y}_n^{\mathrm{Ig}}(\epsilon)}$ (defined in Definition \ref{Zphatdefinition} on the site $Y_n^{\mathrm{Ig}}(\epsilon)_{\text{\'{e}t}}$) to the projective limit of the \'{e}tale topoi $Y_n^{\mathrm{Ig}}(\epsilon)_{\text{\'{e}t}}$. 

\item Note in (\ref{e1trivialization}) that $\mathcal{E}(0)^{\dagger}$ also inherits a natural (prime to $p$) $\Gamma$-level structure from the $\Gamma$-level structure on $\hat{E} \rightarrow \hat{Y}_{\infty}$ (see Convention \ref{Yconvention}). In keeping with Convention \ref{tameconvention}, we continue to suppress the prime to $p$ level structure $\Gamma$ as it will play a negligible role in the our arguments.
\end{enumerate}
\end{definition}

Recall the universal object $\mathcal{E}(0) \rightarrow \hat{\mathcal{Y}}^{\mathrm{Ig}}(0)$, i.e. the universal ordinary (false) elliptic curve with $\mathcal{O}_D$-endomorphism structure, $\Gamma$-level structure and trivialization 
$$e_1 : \hat{\mathbb{Z}}_{p,\hat{\mathcal{Y}}^{\mathrm{Ig}}(0)} \xrightarrow{\sim} \varprojlim_nC_n(0) \subset T_p\mathcal{E}(0).$$
From (\ref{e1trivialization}) and the universal property, there exist unique maps defined over $\mathrm{Spa}(k,\mathcal{O}_k)$
\begin{equation}\label{universalmaps}\varPhi : \hat{\mathcal{Y}}^{\mathrm{Ig}}(0)^{\dagger} \rightarrow \hat{\mathcal{Y}}^{\mathrm{Ig}}(0), \hspace{1cm}\varSigma : \mathcal{E}(0)^{\dagger} \rightarrow \mathcal{E}(0)
\end{equation}
which fit into a pullback diagram
$$\begin{tikzcd}[column sep =large]
\mathcal{E}(0)^{\dagger} \arrow{r}{\varSigma} \arrow{d}{}& \mathcal{E}(0) \arrow{d} \\
\hat{\mathcal{Y}}^{\mathrm{Ig}}(0)^{\dagger} \arrow{r}{\varPhi} & \hat{\mathcal{Y}}^{\mathrm{Ig}}(0)
\end{tikzcd}.$$

\begin{proposition}The maps $i^{\dagger} : \hat{\mathcal{Y}}^{\mathrm{Ig}}(0) \rightarrow \hat{\mathcal{Y}}^{\mathrm{Ig}}(0)^{\dagger}$ from (\ref{daggermaps}) and $\varPhi : \hat{\mathcal{Y}}^{\mathrm{Ig}}(0)^{\dagger}\rightarrow \hat{\mathcal{Y}}^{\mathrm{Ig}}(0)$ from (\ref{universalmaps}) are mutual inverses. Hence we have an equality
\begin{equation}\label{daggerequalities}\mathcal{E}(0)^{\dagger} = \mathcal{E}(0), \hspace{1cm} \hat{\mathcal{Y}}^{\mathrm{Ig}}(0)^{\dagger} = \hat{\mathcal{Y}}^{\mathrm{Ig}}(0).
\end{equation}
In particular,
\begin{equation}\label{presentation4}\varprojlim_{0 < \epsilon < p/(p+1)}\hat{\mathcal{Y}}^{\mathrm{Ig}}(\epsilon) \sim \hat{\mathcal{Y}}^{\mathrm{Ig}}.
\end{equation}
\end{proposition}

\begin{proof}We will first show that $\mathcal{E}(0)^{\dagger} \rightarrow \hat{\mathcal{Y}}^{\mathrm{Ig}}(0)^{\dagger}$ is itself a universal ordinary (false) elliptic curve with $\mathcal{O}_D$-endomorphism structure, $\Gamma$-level structure and trivialization $e_1 : \hat{\mathbb{Z}}_{p,\hat{\mathcal{Y}}^{\mathrm{Ig}}(0)^{\dagger}} \xrightarrow{\sim} \varprojlim_nC_n(0)^{\dagger}$. Given any ordinary (false) elliptic curve $A$ over an adic space $S$ over $\mathrm{Spa}(k,\mathcal{O}_k)$ with $\mathcal{O}_D$-endomorphism structure, $\Gamma$-level structure and trivialization $e_1 : \hat{\mathbb{Z}}_{p,S} \xrightarrow{\sim} \varprojlim_n C_n(A)$, where $C_n(A)$ denotes the order-$p^n$ canonical subgroup, by the universal property for $\mathcal{E}(\epsilon) \rightarrow \hat{\mathcal{Y}}^{\mathrm{Ig}}(\epsilon)$ for any $0 < \epsilon < p/(p+1)$, we get unique maps $S \rightarrow \hat{\mathcal{Y}}^{\mathrm{Ig}}(\epsilon)$ and $A \rightarrow \mathcal{E}(\epsilon)$ fitting into a pullback diagram
$$\begin{tikzcd}[column sep =large]
A\arrow{r}{\varSigma} \arrow{d}{}& \mathcal{E}(\epsilon) \arrow{d} \\
S \arrow{r}{\varPhi} & \hat{\mathcal{Y}}^{\mathrm{Ig}}(\epsilon)
\end{tikzcd}.$$
These diagrams are compatible along the natural maps $\mathcal{E}(\epsilon) \rightarrow \mathcal{E}(\epsilon')$ and $\hat{\mathcal{Y}}^{\mathrm{Ig}}(\epsilon) \rightarrow \hat{\mathcal{Y}}^{\mathrm{Ig}}(\epsilon')$ for any $\epsilon \le \epsilon'$. Hence we get unique maps $A \rightarrow \varprojlim_{0 < \epsilon < p/(p+1)}\mathcal{E}(\epsilon)$ and $S \rightarrow \varprojlim_{0 < \epsilon < p/(p+1)}\hat{\mathcal{Y}}^{\mathrm{Ig}}(\epsilon)$ fitting into a pullback diagram
$$\begin{tikzcd}[column sep =large]
A\arrow{r}{\varSigma} \arrow{d}{}& \varprojlim_{0 < \epsilon < p/(p+1)}\mathcal{E}(\epsilon) \arrow{d} \\
S \arrow{r}{\varPhi} & \varprojlim_{0 < \epsilon < p/(p+1)}\hat{\mathcal{Y}}^{\mathrm{Ig}}(\epsilon)
\end{tikzcd}.$$
Since $A \rightarrow S$ is an adic space and $\mathcal{O}_A^+$ is thus $p$-adically complete (\cite[Definition 2.1.4]{ScholzeWeinstein}), the horizontal maps extend to $p$-adic completions to give unique maps $A \rightarrow \mathcal{E}(0)^{\dagger}$ and $S \rightarrow \hat{\mathcal{Y}}^{\mathrm{Ig}}(0)^{\dagger}$ fitting into a pullback diagram
$$\begin{tikzcd}[column sep =large]
A\arrow{r}{\varSigma} \arrow{d}{}& \mathcal{E}(0)^{\dagger} \arrow{d} \\
S \arrow{r}{\varPhi} & \hat{\mathcal{Y}}^{\mathrm{Ig}}(0)^{\dagger}
\end{tikzcd}.$$
Thus we have shown that $\mathcal{E}(0)^{\dagger} \rightarrow \hat{\mathcal{Y}}^{\mathrm{Ig}}(0)^{\dagger}$ has the desired universal property. 

Now the universal property of $\mathcal{E}(0)^{\dagger} \rightarrow \hat{\mathcal{Y}}^{\mathrm{Ig}}(0)^{\dagger}$ and $\mathcal{E}(0) \rightarrow \mathcal{Y}^{\mathrm{Ig}}$ shows that $i^{\dagger}$ and $\varPhi$ are mutual inverses, and thus define isomorphisms $\hat{\mathcal{E}}(0)^{\dagger} \cong \mathcal{E}(0)$ and $\hat{\mathcal{Y}}^{\mathrm{Ig}}(0)^{\dagger}\cong \hat{\mathcal{Y}}^{\mathrm{Ig}}(0)$. However, since $i^{\dagger}$ is induced by the natural inclusions $\hat{\mathcal{Y}}^{\mathrm{Ig}}(0) \subset \hat{\mathcal{Y}}^{\mathrm{Ig}}(\epsilon)$, these isomorphisms are in fact identity maps. This gives (\ref{daggerequalities}). Finally, (\ref{presentation4}) follows immediately from (\ref{thirdpresentation}) and (\ref{daggerequalities}). 

\end{proof}

\begin{corollary}\label{completioncorollary0}$\mathcal{O}_{\hat{\mathcal{Y}}^{\mathrm{Ig}}(0)^+}(\hat{\mathcal{Y}}^{\mathrm{Ig}}(0)^+)$ is the $p$-adic completion of $\varinjlim_{\epsilon > 0}\mathcal{O}_{\hat{\mathcal{Y}}^{\mathrm{Ig}}(\epsilon)^+}(\hat{\mathcal{Y}}^{\mathrm{Ig}}(\epsilon)^+)$.
\end{corollary}

\begin{proof}From (\ref{thirdpresentation}) we have that $\mathcal{O}_{\hat{\mathcal{Y}}^{\mathrm{Ig}}(0)^{\dagger}}^+(\hat{\mathcal{Y}}^{\mathrm{Ig}}(0)^{\dagger})$ is the $p$-adic completion of 
$$\varinjlim_{\epsilon > 0}\mathcal{O}_{\hat{\mathcal{Y}}^{\mathrm{Ig}}(\epsilon)}^+(\hat{\mathcal{Y}}^{\mathrm{Ig}}(\epsilon)).$$
By the second equality of (\ref{daggerequalities}), $\mathcal{O}_{\hat{\mathcal{Y}}^{\mathrm{Ig}}(0)^{\dagger}}^+(\hat{\mathcal{Y}}^{\mathrm{Ig}}(0)^{\dagger})$ is also equal to 
$$\mathcal{O}_{\hat{\mathcal{Y}}^{\mathrm{Ig}}(0)}^+(\hat{\mathcal{Y}}^{\mathrm{Ig}}(0)).$$

\end{proof}

\subsection{The action of $GL_2(\mathbb{Q}_p)$ on $\mathcal{Y}^{\mathrm{Ig}}(\epsilon)$}

\begin{definition}
\begin{enumerate}
\item Let 
\begin{equation}\label{gdefinition}g := \left(\begin{array}{ccc} 1 & 0\\
0 & p\\
\end{array}\right) \in GL_2(\mathbb{Q}_p).
\end{equation}
\item Given any $\gamma \in GL_2(\mathbb{Q}_p)$, let 
$$\gamma^{\vee} := \mathrm{det}(\gamma)\gamma^{-1}$$
denote its contragredient.
\end{enumerate}
\end{definition}

\begin{theorem}\label{perfectoidpropertytheorem}
\begin{enumerate}
\item Suppose $0 \le \epsilon < p/(p+1)$.  We have a morphism of formal schemes over $\mathrm{Spf}(\mathcal{O}_k)$ induced by the $GL_2(\mathbb{Q}_p)$-action of $g$ from (\ref{gdefinition})
$$g : \hat{\mathcal{Y}}^{\mathrm{Ig}}(\epsilon)^+ \rightarrow \hat{\mathcal{Y}}^{\mathrm{Ig}}(p\epsilon)^+.$$
If $0 \le \epsilon < 1/(p+1)$, this is an isomorphism. 

\item For any $0 \le \epsilon < p/(p+1)$ we thus get an induced map
$$g^* : \mathbf{\Gamma}(\mathcal{O}_{\hat{\mathcal{Y}}^{\mathrm{Ig}}(p\epsilon)^+}) \rightarrow \mathbf{\Gamma}(\mathcal{O}_{\hat{\mathcal{Y}}^{\mathrm{Ig}}(\epsilon)^+})$$
which is an isomorphism of rings if $0 \le \epsilon < 1/(p+1)$. 
\end{enumerate}
\end{theorem}

\begin{proof}\textbf{(1)}: Recall $\omega_+$ from (\ref{omegaY}). From the moduli interpretation of $\hat{Y}_{\infty}^+$ in terms of Drinfeld level structures (see Section \ref{formalShimurasection}) and following the definition in \cite[End of Section 2.2]{ChojeckiHansenJohansson}, we see that there is a right $GL_2(\mathbb{Q}_p)$-action acting on $\hat{Y}_{\infty}^+$ acting invertibly on the universal $\Gamma(p^{\infty})$-Drinfeld level structures and tame $\Gamma$-level structures. 
Note that the kernel of $g^{\vee} \pmod{p} \in M_2(\mathbb{Z}/p)$ acting on the left of $\mathcal{E}[p^n]$ (where $\mathcal{E}[p^n]$ is as in Convention \ref{idempotentconvention}) is $\mathbb{Z}_p\cdot e_{1,1}$, the subgroup generated by $e_{1,1}$ (see (\ref{eindefinition})), which on $\mathcal{Y}^{\mathrm{Ig}}(\epsilon)$ is the order-$p$ canonical subgroup $C_1^+(\epsilon) \subset \mathcal{E}^+(\epsilon)$. Recall that for any $(A,e_1,e_2) \in \hat{\mathcal{Y}}^{\mathrm{Ig}}(\epsilon)^+(S)$, so that for any local lift $\mathrm{Ha} \in \omega_+^{\otimes (p-1)}$ of the Hasse invariant there is a section $u\in \omega_+^{\otimes (1-p)}(\mathrm{Spf}(S))$ such that $\mathrm{Ha}(A) \cdot u = p^{\epsilon} \in S/p$. (See Definitions \ref{YGamma+epsilondefinition} and \ref{hatmathcalYIgepsilonDefinition}.) Letting $(A',e_1',e_2') = (A,e_1,e_2)\cdot g$, then $A'$ is isomorphic to $A^{(p)}$ modulo $p$ since $C_1^+(\epsilon)$ is the canonical subgroup, and so $\mathrm{Ha}(A') = \mathrm{Ha}(A)^p \in S/p$, which implies $\mathrm{Ha}(A') \cdot u^{(p)} = p^{p\epsilon} S/p$. (Cf. with \cite[proof of Lemma III.2.14]{Scholze}.) Recall $n(\epsilon)$ from (\ref{nepsilondefinition}). By standard properties of order-$p^n$-canonical subgroups under division by the canonical subgroup (\cite[Proposition III.2.8 (iii)]{ScholzeTorsion}), since $\mathbb{Z}_p \cdot e_{1,n(\epsilon)}$ is the order-$p^{n(\epsilon)}$-canonical subgroup of $A$, $\mathbb{Z}_p\cdot e_{1,n(\epsilon)-1}'$ is the order-$p^{n(\epsilon)-1} = p^{n(p\epsilon)}$-canonical subgroup of $A'$. Hence, in all, $(A',e_1',e_2') \in \hat{\mathcal{Y}}^{\mathrm{Ig}}(p\epsilon)^+(S)$. Thus the action of $g$ on $\hat{Y}_{\infty}^+$ induces a map
$$g : \hat{\mathcal{Y}}^{\mathrm{Ig}}(\epsilon)^+ \rightarrow \hat{\mathcal{Y}}^{\mathrm{Ig}}(p\epsilon)^+.$$

Now assume $0 \le \epsilon < 1/(p+1)$. We will produce an inverse of the above map. On the other hand, given $(A',e_1',e_2') \in \hat{\mathcal{Y}}^{\mathrm{Ig}}(p\epsilon)^+(S)$, defining $(A,e_1,e_2) = (A',e_1',e_2') \cdot g^{-1}$, we see that $A/\langle e_{1,1}\rangle$ is isomorphic to $A'$. Moreover, since $\mathbb{Z}_p \cdot e_{1,1}'$ is the order-$p$ canonical subgroup of $A'$ and is disjoint with the kernel of the isogeny $A' \rightarrow A$, then $e_{1,1}$ is the order-$p$ canonical subgroup of $A$. Reversing the $\mathrm{Ha}$ calculation of the previous paragraph shows that there exists $u \in \omega_+^{\otimes (1-p)}$ with $\mathrm{Ha}(A) \cdot u = p^{\epsilon} \in S/p$. Moreover, $\mathbb{Z}_p \cdot e_{1,n(\epsilon)} = \mathbb{Z}_p \cdot e_{1,n(p\epsilon) + 1}$ maps under the division by the canonical subgroup isogeny $A \rightarrow A/\langle e_{1,1}\rangle \cong A'$ to $\mathbb{Z}_p \cdot e_{1,n(p\epsilon)}$ which is the order-$p^{n(p\epsilon)}$ canonical subgroup of $A'$. By standard properties of the canonical subgroup (\cite[Proposition III.2.8 (iii)]{ScholzeTorsion}), this implies $\mathbb{Z}_p \cdot e_{1,n(\epsilon)}$ is the order-$p^{n(\epsilon)}$ canonical subgroup of $A$. Hence in all, $(A,e_1,e_2) \in \hat{\mathcal{Y}}^{\mathrm{Ig}}(\epsilon)^+(S)$. Thus, the action of $g^{-1}$ on $\hat{Y}_{\infty}$ induces a map
$$g^{-1} : \hat{\mathcal{Y}}^{\mathrm{Ig}}(p\epsilon)^+ \rightarrow \hat{\mathcal{Y}}^{\mathrm{Ig}}(\epsilon)^+.$$
Going through the definitions, we see that this map and the map $g$ from the end of the previous paragraph are mutual inverses. Now the first isomorphism of the statement of the Theorem follows. \\

%Let $\pi_g : \mathcal{E}^+(\epsilon) \rightarrow \mathcal{E}^+(\epsilon)/C_1^+(\epsilon)$ be the natural projection. The (false) elliptic curve 
%$$\mathcal{E}^+(\epsilon)/C_1^+(\epsilon) \rightarrow \hat{\mathcal{Y}}^{\mathrm{Ig}}(\epsilon)^+ \xrightarrow{g} \hat{\mathcal{Y}}^{\mathrm{Ig}}(p\epsilon)^+$ is a deformation of the special fiber of $\mathcal{E}^+(p\epsilon)$.
%Hence we have a diagram
%\begin{equation}\label{E+diagram}
%\begin{tikzcd}[column sep = large]
%\mathcal{E}^+(\epsilon) \arrow{dr} \arrow{r}{\pi_g} & \mathcal{E}^+(\epsilon)/C_1^+(\epsilon) \arrow{d} \arrow{r} & \mathcal{E}^+(p\epsilon) \arrow{d} \\
%& \hat{\mathcal{Y}}^{\mathrm{Ig}}(\epsilon) \arrow{r}{g} & \hat{\mathcal{Y}}^{\mathrm{Ig}}(p\epsilon)^+
%\end{tikzcd}
%\end{equation}
%where the right square is given by the universal property of $\mathcal{E}^+(p\epsilon) \rightarrow \hat{\mathcal{Y}}^{\mathrm{Ig}}(p\epsilon)^+$. The bottom arrow (given by the classifying map of the universal property) is a morphism $g : \hat{\mathcal{Y}}^{\mathrm{Ig}}(\epsilon)^+ \arrow{r} \hat{\mathcal{Y}}^{\mathrm{Ig}}(p\epsilon)^+$ which coincides with the $GL_2(\mathbb{Q}_p)$-action of $g$. 
%\begin{equation}\label{E+diagram2}
%\begin{tikzcd}[column sep = large]
%\mathcal{E}^+(p\epsilon) \arrow{dr} \arrow{r}{\pi_g} & \mathcal{E}^+(p\epsilon)/\langle e_{2,1}\rangle \arrow{d} \arrow{r} & \mathcal{E}^+(\epsilon) \arrow{d} \\
%& \hat{\mathcal{Y}}^{\mathrm{Ig}}(p\epsilon) \arrow{r}{g^{-1}} & \hat{\mathcal{Y}}^{\mathrm{Ig}}(\epsilon)^+
%\end{tikzcd}
%\end{equation}

\textbf{(2)}: This follows immediately from (1).
\end{proof}

Note that on $\mathcal{Y}^{\mathrm{Ig}}(\epsilon)^+$ and $\mathcal{Y}^{\mathrm{Ig}}(\epsilon)$, if $\epsilon < p/(p+1)$ then under the moduli-theoretic interpretation of the $GL_2(\mathbb{Q}_p)$-action, the isogeny underlying $g \in GL_2(\mathbb{Q}_p)$ corresponds to division by the canonical subgroup (cf. \cite[Lemma 2.11]{ChojeckiHansenJohansson}). As a corollary of Theorem \ref{perfectoidpropertytheorem}, we have the following.

\begin{theorem}\begin{enumerate}
\item For any $0 \le \epsilon < p/(p+1)$, we have maps induced by the $GL_2(\mathbb{Q}_p)$-action of $g$ from (\ref{gdefinition})
\begin{equation}\label{gisomorphism}g : \mathcal{Y}^{\mathrm{Ig}}(\epsilon)^+ \rightarrow \mathcal{Y}^{\mathrm{Ig}}(p\epsilon)^+, \hspace{1cm} g : \mathcal{Y}^{\mathrm{Ig}}(\epsilon) \rightarrow \mathcal{Y}^{\mathrm{Ig}}(p\epsilon), 
\end{equation}
where the first map is a morphism of pro-formal schemes over $\mathrm{Spf}(\mathcal{O}_k)$ and the second is a morphism of pro-adic spaces over $\mathrm{Spa}(k,\mathcal{O}_k)$. If $0 \le \epsilon < 1/(p+1)$, these are isomorphisms with inverses given by 
\begin{equation}\label{g-1isomorphism}g^{-1} : \mathcal{Y}^{\mathrm{Ig}}(p\epsilon)^+ \rightarrow \mathcal{Y}^{\mathrm{Ig}}(\epsilon)^+, \hspace{1cm} g^{-1} : \mathcal{Y}^{\mathrm{Ig}}(p\epsilon) \rightarrow \mathcal{Y}^{\mathrm{Ig}}(\epsilon).
\end{equation}

\item For any $0 \le \epsilon < p/(p+1)$, this above maps induce maps of rings
$$g^* : \mathbf{\Gamma}(\mathcal{O}_{\mathcal{Y}^{\mathrm{Ig}}(p\epsilon)^+})\rightarrow \mathbf{\Gamma}(\mathcal{O}_{\mathcal{Y}^{\mathrm{Ig}}(\epsilon)^+}), \hspace{1cm} g^* : \mathbf{\Gamma}(\mathcal{O}_{\mathcal{Y}^{\mathrm{Ig}}(p\epsilon)})\rightarrow \mathbf{\Gamma}(\mathcal{O}_{\mathcal{Y}^{\mathrm{Ig}}(\epsilon)}).$$
If $0 \le \epsilon < 1/(p+1)$, these are isomorphisms with inverses given by
$$(g^{-1})^* : \mathbf{\Gamma}(\mathcal{O}_{\mathcal{Y}^{\mathrm{Ig}}(\epsilon)^+}) \rightarrow \mathbf{\Gamma}(\mathcal{O}_{\mathcal{Y}^{\mathrm{Ig}}(p\epsilon)^+}), \hspace{1cm} (g^{-1})^* : \mathbf{\Gamma}(\mathcal{O}_{\mathcal{Y}^{\mathrm{Ig}}(\epsilon)}) \rightarrow \mathbf{\Gamma}(\mathcal{O}_{\mathcal{Y}^{\mathrm{Ig}}(p\epsilon)}).$$
\end{enumerate}
\end{theorem}

\begin{proof}\textbf{(1)}: The first map of (\ref{gisomorphism}) is induced by the ring map in Theorem \ref{perfectoidpropertytheorem} and the relation $\hat{\mathcal{Y}}^{\mathrm{Ig}}(\epsilon)^+ \sim \mathcal{Y}^{\mathrm{Ig}}(\epsilon)^+$ (see Definition \ref{hatmathcalYIgepsilonDefinition}). The second map of (\ref{gisomorphism}) follows from taking the adic generic fiber of the ring map of Theorem \ref{perfectoidpropertytheorem}.\\

\textbf{(2)}: The statement for the maps of rings follows immediately from (\ref{gisomorphism}).
\end{proof}

\subsection{The canonical locus in terms of the Hodge-Tate period}
We will need a lemma describing the \emph{canonical locus} 
\begin{equation}\label{Ucan}U^{\mathrm{can}} := \bigcup_{0 \le \epsilon < p/(p+1)}\mathcal{Y}^{\mathrm{Ig}}(\epsilon),
\end{equation}
also referred to as $U^{\mathrm{can}}$ in the Introduction (Section \ref{Introduction}), in terms of the Hodge-Tate period $z_{\mathrm{HT}}$. We will follow Convention \ref{zdefineconvention} extensively throughout this section. 

\begin{lemma}\label{blemma}\begin{enumerate}
\item We have 
\begin{equation}\label{zUcan}\{|z_{\mathrm{HT}}| > p^{p/(p^2-1)}\} \subset U^{\mathrm{can}} \subset \{|1/z_{\mathrm{HT}}| > p^{p/(p^2-1)}|\} ^c 
\end{equation}
where the superscript ``$c$'' denotes set complement in $Y_{\infty}$. Moreover, there exists $b \in \mathbb{Z}_{> 0}$ such that 
\begin{equation}\label{zUb}\{|z_{\mathrm{HT}}| > p^b\} \subset U = \mathcal{Y}^{\mathrm{Ig}}(\epsilon_0).
\end{equation}
\item For any $0 \le \epsilon < p^2/(p+1)$ in the valuation group of $\mathcal{O}_k$, we have 
\begin{equation}\label{YIgVx}\mathcal{Y}^{\mathrm{Ig}}(\epsilon) \subset \{z_{\mathrm{HT}} \neq 0\} \overset{(\ref{Vz})}{=} \mathcal{V}_x, \hspace{1cm} \hat{\mathcal{Y}}^{\mathrm{Ig}}(\epsilon) \subset \{z_{\mathrm{HT}} \neq 0\} \overset{(\ref{hatVz})}{=} \hat{\mathcal{V}}_x.
\end{equation}
\end{enumerate}
\end{lemma}

\begin{proof}
\textbf{(1)}: By \cite[Lemma 2.14]{ChojeckiHansenJohansson}, $e_{1,1}$ (see (\ref{eindefinition})) parametrizes the canonical subgroup if $|z_{\mathrm{HT}}| > p^{p/(p^2-1)}$. Thus, since $U^{\mathrm{can}}$ is also exactly the locus in $Y_{\infty}$ on which $e_{1,1}$ parametrizes the canonical subgroup, we have the first inclusion of (\ref{zUcan}). For the second inclusion, note that since $e_{1,1}$ trivializes the canonical subgroup on $U_{\mathrm{can}}$, then $e_{2,1}$ does not trivialize the canonical subgroup on $U_{\mathrm{can}}$. Thus, 
$$U_{\mathrm{can}} \subset \{e_{2,1} : \mathbb{Z}/p \xrightarrow{\sim} C_1\}^c$$
where $C_1 \subset \mathcal{E}|_{U_{\mathrm{can}}}[p]$ is the canonical subgroup, $\{e_{2,1} : \mathbb{Z}/p \xrightarrow{\sim} C_1\}$ denotes the locus where $e_{2,1}$ (in the notation of (\ref{eindefinition})) trivializes $C_1$.  By Lemma 2.14 of op. cit. again, we have 
$$\{|1/z_{\mathrm{HT}}| > p^{p/(p^2-1)}\} \subset \{e_{2,1} : \mathbb{Z}/p \xrightarrow{\sim} C_1\},$$
and so
$$U_{\mathrm{can}} \subset \{e_{2,1} : \mathbb{Z}/p \xrightarrow{\sim} C_1\}^c \subset \{|1/z_{\mathrm{HT}}| > p^{p/(p^2-1)}|\} ^c.$$
This gives (\ref{zUcan}).

For any $m \in \mathbb{Z}_{\ge -1}$, applying the action of $g^{-m}$ to both sides of the previous equality (recall $g$ from (\ref{gdefinition})), and using $(g^{-m})^*z_{\mathrm{HT}} = z_{\mathrm{HT}}/p^m$ from (\ref{zHTtransformationproperty}) and the isomorphism $g^{-m} : \mathcal{Y}^{\mathrm{Ig}}(\epsilon) \rightarrow \mathcal{Y}^{\mathrm{Ig}}(\epsilon/p^m)$ from (\ref{g-1isomorphism}), we get 
\begin{equation}\label{zUcanm}\{|z_{\mathrm{HT}}| > p^{p/(p^2-1)+m}\} \subset \bigcup_{0 \le \epsilon < p^{1-m}/(p+1)}\mathcal{Y}^{\mathrm{Ig}}(\epsilon).
\end{equation}
The right-hand side is contained in $\mathcal{Y}^{\mathrm{Ig}}(p^{1-m}/(p+1))$, which is contained in $\mathcal{Y}^{\mathrm{Ig}}(\epsilon)$ for any $m \in \mathbb{Z}_{\ge 0}$ with $p^{1-m}/(p+1) \le \epsilon$. Taking $b = p/(p^2-1) + m$ for any such $m$ gives (\ref{zUb}).\\

\textbf{(2)}: Now (\ref{YIgVx}) follows from (\ref{zUcanm}), noting that $\mathcal{Y}^{\mathrm{Ig}}(\epsilon)$ for any $0 \le \epsilon < p^2/(p+1)$ is contained in the right-hand side of that equation specialized to $m = -1$, and then using (\ref{hatmathcalYIg}).

\end{proof}

\subsection{The affinoid perfectoidness of $\mathcal{Y}^{\mathrm{Ig}}(\epsilon)$}

In this section, we prove that $\mathcal{Y}^{\mathrm{Ig}}(\epsilon)$ is affinoid perfectoid (Lemma \ref{Uaffinoidperfectoidlemma}). We will also show that the $\mathcal{Y}^{\mathrm{Ig}}(\epsilon)$ form a collection of affinoid perfectoid neighborhoods of $\mathcal{Y}^{\mathrm{Ig}}$ such that for any intersection $W$ of rational subsets of $\mathcal{Y}^{\mathrm{Ig}}(\epsilon)$ such that $\mathcal{Y}^{\mathrm{Ig}} \subset W$, there exists some $0 < \epsilon' < \epsilon$ such that $\mathcal{Y}^{\mathrm{Ig}}(\epsilon') \subset  W \subset \mathcal{Y}^{\mathrm{Ig}}(\epsilon)$. See Corollary \ref{rationalintersectioncorollary}. First, let us isolate a certain ``negligible'' CM point on $\mathbb{Y}(\Gamma(1))$, which we will often delete from our Shimura curve in order to work on an affine curve.

\begin{choice}\label{badCMpointchoice}
%Retain the notation of Convention \ref{Yconvention}. Let $Y^+$ be the natural integral model of $Y$ over $\mathrm{Spec}(\mathbb{Z}_p)$ (and recall that the level $\Gamma$ of $Y$ is prime to $p$). Then $Y^+$ is an affine scheme over $\mathrm{Spec}(\mathbb{Z}_p)$ when $D = M_2(\mathbb{Q})$.
Fix an imaginary quadratic field $K'$ such that $p$ is \emph{inert} in $K'$. 
\begin{enumerate}
\item When $D = M_2(\mathbb{Q})$, let 
$y_0 \in \mathbb{Y}(\Gamma(1))(\overline{\mathbb{Z}}_p)$
be the point corresponding to the isomorphism class of an elliptic curve $E_0'$ with CM by $\mathcal{O}_{K'}$. 
\item When $D \neq M_2(\mathbb{Q})$, let $y_0 \in \mathbb{Y}(\Gamma(1))(\overline{\mathbb{Z}}_p)$ be a point corresponding to the isomorphism class of a false elliptic curve $E_0'$ with CM by $\mathcal{O}_{K'}$. 
\item In both cases, let $\mathbf{S}_0 \subset \mathbb{Y}(\Gamma(1))(\overline{\mathbb{Z}}_p)$ be the $\mathrm{Gal}(\overline{\mathbb{Q}}_p/\mathbb{Q}_p)$-orbit of $y_0$. Thus $\mathbf{S}_0 \subset \mathbb{Y}(\Gamma(1))$ is a finite closed set defined over $\mathbb{Z}_p$. 
% When $D \neq M_2(\mathbb{Q})$, choose a point 
%$$y_0 \in Y^+(\overline{\mathbb{F}}_p)$$
%whose underlying false elliptic curve is the reduction of the square of an elliptic curve with CM by the ring of integers $\mathcal{O}_{K'}$ of some imaginary quadratic field $K'$ in which $p$ is inert. 
%Let
%$$Y_0^+ = \begin{cases} Y^+ & D = M_2(\mathbb{Q})\\
%Y^+ \setminus \{y_0\} & D \neq M_2(\mathbb{Q})
%\end{cases},$$
%which is an affine scheme, since the complement of finitely many points on a projective curve is affine. 
%Then 
%$$H^1(Y_0^+,\omega^{\otimes p-1}) = 0,$$
%by Serre's affineness criterion, and so by the argument of \cite[Section 7]{Kassaei} there is a lift of the Hasse invariant to $Y_0^+$, which we henceforth fix and denote by 
%$$\mathrm{Ha} \in \omega^{\otimes p-1}(Y_0^+).$$
\end{enumerate}
\end{choice}

\begin{proposition}\label{y0proposition}Let $E_0'$ be as in Choice \ref{badCMpointchoice}. For any $\sigma \in \mathrm{Gal}(\overline{\mathbb{Q}}_p/\mathbb{Q}_p)$, $(E_0')^{\sigma}$ has no canonical subgroup. Thus the underlying (false) elliptic curve of any point of $\mathbf{S}_0$ has no canonical subgroup. 
\end{proposition}

\begin{proof}We argue by contradiction. Recall that $E_0'$ had CM by $\mathcal{O}_{K'}$, and hence $(E_0')^{\sigma}$ has CM by $\mathcal{O}_{K'}$. Suppose $(E_0')^{\sigma}$ had a canonical subgroup $C$. Then since $C$ is unique, it is fixed by the CM action of $\mathcal{O}_{K'}$. Since $C \subset (E_0')^{\sigma}[p]$, the kernel of this action contains $p\mathcal{O}_{K'}$, and thus $\mathcal{O}_{K'}/p \cong \mathbb{F}_{p^2}$ acts nontrivially on $C$, and since $\mathbb{Z}/p \subset \mathcal{O}_{K'}/p$ acts by multiplication on $C \cong \mathbb{Z}/p$, we see that $(\mathcal{O}_{K'}/p) \cdot C = C$. Thus $C$ is an $\mathbb{F}_{p^2}$-vector space, which is impossible since it has order $p$. 

\end{proof}

\begin{definition}\label{YGamma1definition}\begin{enumerate}
\item Let 
$$\mathbb{Y}(\Gamma(1)) \rightarrow \mathrm{Spec}(\mathbb{Q})$$
be the algebraic Shimura curve associated to $\Gamma = \Gamma(1)$ (see Section \ref{algebraicYsection}). Let
$$\mathbb{Y}(\Gamma(1))^+ \rightarrow  \mathrm{Spec}(\mathbb{Z}[1/\mathrm{disc}(D)])$$
denote the algebraic Shimura curve from \cite[Theorem 2.1]{Buzzard}; it is smooth and satisfies 
$$\mathbb{Y}(\Gamma(1))^+ \times_{\mathrm{Spec}(\mathbb{Z}[1/\mathrm{disc}(D)])}\mathrm{Spec}(\mathbb{Q}) = \mathbb{Y}(\Gamma(1)).$$
When $D = M_2(\mathbb{Q})$, these are simply the algebraic $j$-lines.
\item Recalling $p \nmid \mathrm{disc}(D)$, let
$$Y(\Gamma(1))^+ := \mathbb{Y}(\Gamma(1))^+ \times_{\mathrm{Spec}(\mathbb{Z}[1/\mathrm{disc}(D))]} \mathrm{Spec}(\mathbb{Z}_p) \rightarrow \mathrm{Spec}(\mathbb{Z}_p)$$
be the $p$-adic integral model of $\mathbb{Y}(\Gamma(1))$ (see Section \ref{formalShimurasection}). Let
$$\widehat{Y(\Gamma(1))}^+ \rightarrow \mathrm{Spf}(\mathbb{Z}_p)$$
be the $p$-adic completion of $Y(\Gamma(1))^+$. Let 
$$Y(\Gamma(1)) \rightarrow \mathrm{Spa}(\mathbb{Q}_p,\mathbb{Z}_p)$$
be the adic generic fiber of $\widehat{Y(\Gamma(1))}^+$ (see Section \ref{adicShimurasection}).
When $D = M_2(\mathbb{Q})$, these are simply the formal and adic  j-lines. 
\item Now we consider base changes to $k$ and $\mathcal{O}_k$. Let
\begin{equation}\label{mathbbYbasechange}Y(\Gamma(1))_{\mathcal{O}_k}^+ := \mathbb{Y}(\Gamma(1))^+ \times_{\mathrm{Spec}(\mathbb{Z}[1/\mathrm{disc}(D)])}\mathrm{Spec}(\mathcal{O}_k) = Y(\Gamma(1))^+ \times_{\mathrm{Spec}(\mathbb{Z}_p)}\mathrm{Spec}(\mathcal{O}_k). 
\end{equation}
Let
$$\widehat{Y(\Gamma(1))}_{\mathcal{O}_k}^+ \rightarrow \mathrm{Spf}(\mathcal{O}_k)$$
be the $p$-adic completion of $Y(\Gamma(1))_{\mathcal{O}_k}^+$. Then 
$$\widehat{Y(\Gamma(1))}_{\mathcal{O}_k}^+ = \widehat{Y(\Gamma(1))}^+ \times_{\mathrm{Spf}(\mathbb{Z}_p)}\mathrm{Spf}(\mathcal{O}_k).$$
Let 
$$Y(\Gamma(1))_k \rightarrow \mathrm{Spa}(k,\mathcal{O}_k)$$
be the adic generic fiber of $\widehat{Y(\Gamma(1))}_{\mathcal{O}_k}^+$. Then 
$$Y(\Gamma(1))_k = Y(\Gamma(1)) \times_{\mathrm{Spa}(\mathbb{Q}_p,\mathbb{Z}_p)} \mathrm{Spa}(k,\mathcal{O}_k).$$
\end{enumerate}
\end{definition}
%Let $Y(1)^+ = \mathrm{Spec}(\mathcal{O}_k[j])$ and let $Y(1) = \mathbb{A}_k^1$ denote the $j$-line considered as an adic space over $\mathrm{Spa}(k,\mathcal{O}_k)$, where here $j$ is the j-invariant. 

%Let $\mathbb{Y}(1)_{\mathbb{Z}_{(p)}}$ be the canonical integral model over $\mathbb{Z}_{(p)}$ of $\mathbb{Y}(1)$.

We now define a certain affine open subscheme $V \subset \mathbb{Y}(\Gamma(1))_{\mathcal{O}_k}^+$ which will be needed for certain constructions in Section \ref{jsection}.

\begin{definition}\label{V0definition}\begin{enumerate}
\item Define the open set 
$$V_0 = \mathbb{Y}(\Gamma(1))^+ \setminus \mathbf{S}_0 \subset \mathbb{Y}(\Gamma(1))^+,$$
which has a natural structure as an open subscheme of the smooth scheme $\mathbb{Y}(\Gamma(1))^+$ over $\mathrm{Spec}(\mathbb{Z}_p)$. In fact, $V_0$ is smooth affine since it is the complement of a finite nonzero number of points in a projective curve.
\item Let
$$V:= V_0 \times_{\mathrm{Spec}(\mathbb{Z}_p)}\mathrm{Spec}(\mathcal{O}_k)\subset \mathbb{Y}(\Gamma(1))_{\mathcal{O}_k}^+$$
which by the above is an affine open defined over $\mathrm{Spec}(\mathcal{O}_k)$. 
\item Let $\widehat{V}_0$ denote the $p$-adic completion of $V_0$ (a formal scheme over $\mathrm{Spf}(\mathbb{Z}_p)$) and let $\widehat{V}$ denote the $p$-adic completion of $V$ (a formal scheme over $\mathrm{Spf}(\mathcal{O}_k)$). Then $\widehat{V}_0 \subset \widehat{Y(\Gamma(1))}^+$ is an open formal subscheme over $\mathrm{Spf}(\mathbb{Z}_p)$ and $\widehat{V} \subset \widehat{Y(\Gamma(1))}_{\mathcal{O}_k}^+$ is an open formal subscheme over $\mathrm{Spf}(\mathcal{O}_k)$. 
\end{enumerate}
\end{definition}

Now suppose that $0 \le\epsilon < p/(p+1)$ is in the valuation group of $\mathcal{O}_k$. Then since $n(\epsilon) \ge 1$ (see \ref{nepsilondefinition}) (or by \cite[Theorem 3.10.7]{Katzpamf} and \cite{Kassaei}) the canonical subgroup exists on $Y^+(\epsilon)$. Thus, by Proposition \ref{y0proposition}, the image $Y(\Gamma(1))^+(\epsilon)$ of $Y^+(\epsilon)$ under $\hat{Y}^+ \rightarrow \widehat{Y(\Gamma(1))}_{\mathcal{O}_k}^+$ is contained in $\widehat{V}$. Thus we have a sequence maps of formal schemes over $\mathrm{Spf}(\mathcal{O}_k)$
\begin{equation}\label{avoidy0inclusion}Y^+(\epsilon) \twoheadrightarrow Y(\Gamma(1))^+(\epsilon)_{\mathcal{O}_k} \subset \widehat{V}
\end{equation}
where the first map is finite \'{e}tale and the second map is an open immersion. 
% we have 
%\begin{equation}\label{YepsilonY0}Y^+(\epsilon) \subset Y_0^+.
%\end{equation}

\begin{lemma}\label{Uaffinoidperfectoidlemma}For any $0 \le \epsilon < p/(p+1)$ in the valuation group of $\mathcal{O}_k$ (see Definition \ref{kdefinition}), $\mathcal{Y}^{\mathrm{Ig}}(\epsilon) \in Y_{\text{pro\'{e}t}}$ is affinoid perfectoid over $\mathrm{Spa}(k,\mathcal{O}_k)$ (see \cite[Definition 4.3]{Scholze}).
\end{lemma}
\begin{proof} If $\epsilon = 0$, then $\mathcal{Y}^{\mathrm{Ig}}(0)$ is the perfection of $Y^{\mathrm{Ig}}$ from Definition \ref{Igusadefinition}, and so is affinoid perfectoid in $Y_{\text{pro\'{e}t}}$ over $\mathrm{Spa}(k,\mathcal{O}_k)$.  If $\epsilon > 0$, we will show that $\mathcal{Y}^{\mathrm{Ig}}(\epsilon)$ is a rational subset of an affinoid perfectoid, and hence is affinoid perfectoid by \cite[Theorem 6.3 (ii)]{Scholzeperf}. 

Now assume that $1/(p+1) \le \epsilon < p/(p+1)$, so that $n(\epsilon) = 1$ by (\ref{nepsilondefinition}). Let $Y(\Gamma_0(p) \cap \Gamma(N))^+$ be the scheme over $\mathrm{Spec}(\mathbb{Z}_p)$ defined as in Definition \ref{YGamma+definition} (where $N \ge 4$ with $(N,p) = 1$, as in Convention \ref{Yconvention}, so that in particular $N$ is neat in the sense of Definition \ref{neatdefinition}). Let $\mathcal{E}_0^+ \rightarrow Y(\Gamma_0(p)\cap \Gamma(N))^+$ be the universal object (which exists since $N$ is neat), and let $H \subset \mathcal{E}_0^+[p]$ be the order-$p$- group scheme determined by the $\Gamma_0(p)$-level structure on $Y_0^+$. Recall that $D$ is the quaternion algebra over $\mathbb{Q}$ used in the definition of $Y$. Let $\mathbf{S}_0 \subset Y(\Gamma(1))^+(\overline{\mathbb{Z}}_p)$ be as in Choice \ref{badCMpointchoice}. Then $Y(\Gamma_0(p) \cap \Gamma(N))^+ \rightarrow Y(\Gamma(1))^+$ is finite, so the preimage $S_0$ of $\mathbf{S}_0$ under this map is a union of finitely many points. Let
$$Y_0^+ = \begin{cases} Y(\Gamma_0(p) \cap \Gamma(N))^+ & D = M_2(\mathbb{Q})\\
Y(\Gamma_0(p) \cap \Gamma(N))^+ \setminus S_0 & D \neq M_2(\mathbb{Q})\\
\end{cases}.$$
In either case, the curve $Y_0^+$ is the complement of a nonempty set of finitely many points in the projective closure of $Y(\Gamma_0(p) \cap \Gamma_0(N))^+$ over $\mathrm{Spec}(\mathbb{Z}_p)$, and so $Y_0^+$ is an affine scheme over $\mathrm{Spec}(\mathbb{Z}_p)$. Hence $H$ is an order-$p$ group scheme over an affine scheme $Y_0^+ = \mathrm{Spec}(R)$ over $\mathrm{Spec}(\mathbb{Z}_p)$, and so applying the Oort-Tate classification (\cite[Theorem 2]{OortTate}), we can identify $H = \mathrm{Spec}(R[T]/(T^p-\mathcal{A}T))$ for some formal variable $T$ and $\mathcal{A} \in R$. Now let $Y_0^+(\epsilon) = Y(\Gamma_0(p) \cap \Gamma_1(N))^+(\epsilon)$ as in Definition \ref{YGamma+epsilondefinition}, which is a formal scheme over $\mathrm{Spf}(\mathcal{O}_k)$, and let $Y_0(\epsilon)$ be its adic generic fiber over $\mathrm{Spa}(k,\mathcal{O}_k)$. Let $\mathcal{E}_0^+(\epsilon) \rightarrow Y_0^+(\epsilon)$ and $\mathcal{E}_0(\epsilon) \rightarrow Y_0(\epsilon)$ be the corresponding universal objects. By our assumption that $\epsilon < p/(p+1)$ and (\ref{avoidy0inclusion}), we have 
$$Y_0^+(\epsilon) \subset Y_0^+, \hspace{1cm} Y_0(\epsilon) \subset Y_0.$$
Thus, the global section $\mathcal{A} \in R = \mathcal{O}_{Y_0^+}(Y_0^+)$ induces global sections on $Y_0^+(\epsilon)$ and $Y_0(\epsilon)$, which we continue to denote by $\mathcal{A}$. 

Consider the rational subset 
$$\mathcal{V} := \{|\mathcal{A}| \le p^{\epsilon -1}\} \subset Y_0^+(\epsilon).$$
By \cite[Theorem 1.2.8]{Kassaei2}, $\mathcal{V}$ is exactly the locus in $Y_0(\epsilon)$ such that $H$ is the canonical subgroup of $\mathcal{E}_0[p]$. Therefore, by (\ref{U}), $\mathcal{Y}^{\mathrm{Ig}}(\epsilon)$ is the inverse image of $\mathcal{V}$ under the projection 
$$Y_{\infty}\rightarrow Y(\Gamma_0(p) \cap \Gamma_1(N)), \hspace{1cm} (A,i,P,e_1,e_2) \mapsto (A,i,P,\mathbb{Z}_p \cdot e_{1,1}).$$
Since the $p$-adic completion $\hat{Y}_{\infty}(\epsilon) \rightarrow Y_0(\epsilon)$ is an adic morphism, $\mathcal{Y}^{\mathrm{Ig}}(\epsilon) \subset Y_{\infty}$ is a rational subset. By (\ref{zUcan}) we have 
$$\mathcal{Y}^{\mathrm{Ig}}(\epsilon) \subset \{|1/z_{\mathrm{HT}}| > p^{p/(p^2-1)}|\}^c \subset \{|z_{\mathrm{HT}}| \ge p^{1/(p-1)}\}.$$
Here, $p^{1/(p-1)}$ can be taken to be any element of $\mathcal{O}_k$ (see Definition \ref{kdefinition}) with $p$-adic valuation $1/(p-1)$. Now $\{|z_{\mathrm{HT}}| \ge 1\} \subset Y_{\infty}$ is an affinoid perfectoid subset of $Y_{\infty}$ over $\mathrm{Spa}(k,\mathcal{O}_k)$ by \cite[Theorem IV.1.1 (i)]{ScholzeTorsion}. Hence $\{|z_{\mathrm{HT}}| \ge p^{1/(p-1)}\} \subset \{|z_{\mathrm{HT}}| \ge 1\}$ is a rational subset of an affinoid perfectoid subset, and so is affinoid perfectoid over $\mathrm{Spa}(k,\mathcal{O}_k)$ by \cite[Theorem 6.3 (ii)]{Scholzeperf}. Thus $\mathcal{Y}^{\mathrm{Ig}}(\epsilon) \subset Y_{\infty}$ is a rational subset of an affinoid perfectoid, and so by the discussion at the end of the first paragraph we have that $\mathcal{Y}^{\mathrm{Ig}}(\epsilon)$ is affinoid perfectoid over $\mathrm{Spa}(k,\mathcal{O}_k)$. So we are done in the case $1/(p+1) \le \epsilon < p/(p+1)$. 

%\begin{lemma}For $1/(p+1) \le \epsilon < p/(p+1)$, we have 
%$$\mathcal{Y}^{\mathrm{Ig}}(\epsilon) \subset \{|z_{\mathrm{HT}}| \ge p^{p/(p^2-1)}\}.$$
%\end{lemma}
%t, and so since $Y_{\infty}(\epsilon) \subset Y_{\infty}$ is a rational subset (its characteristic $p$ tilt is defined by $|\mathrm{Ha}| \le $)

Now the case of general $0 < \epsilon < p/(p+1)$ now follows from (\ref{g-1isomorphism}), since the condition of being affinoid perfectoid is stable under the action of $g^{-1} \in GL_2(\mathbb{Q}_p)$ (see \cite[discussion after Definition 3.5]{ScholzeTorsion}).

% then $U \subset Y_{\infty}$ is open, and since $Y_{\infty} \in Y_{\text{pro\'{e}t}}$ is perfectoid, then \cite[Lemma 4.6]{Scholze} shows that $U$ is perfectoid, and thus affinoid perfectoid since $U$ is affinoid.

\end{proof}

\begin{lemma}\label{Uhataffinoidperfectoidlemma}For any $0 \le \epsilon < p/(p+1)$, we have that $\hat{\mathcal{Y}}^{\mathrm{Ig}}(\epsilon)$ is an affinoid perfectoid space over $\mathrm{Spa}(k,\mathcal{O}_k)$. Moreover,
\begin{equation}\label{recallintegralsame}\mathbf{\Gamma}(\hat{\mathcal{O}}_{\mathcal{Y}^{\mathrm{Ig}}(\epsilon)}^{(+)}) = \mathbf{\Gamma}(\mathcal{O}_{\hat{\mathcal{Y}}^{\mathrm{Ig}}(\epsilon)}^{(+)})
\end{equation}
where the superscript ``$(+)$'' denotes the optional presence of a ``$+$''.
\end{lemma}

\begin{proof}
Recall that $\mathcal{Y}^{\mathrm{Ig}}(\epsilon)$, for any $0 \le \epsilon  < p/(p+1)$, was shown to be affinoid perfectoid (Lemma \ref{Uaffinoidperfectoidlemma}). Thus by \cite[discussion after Definition 4.3]{Scholze}, $\hat{\mathcal{Y}}^{\mathrm{Ig}}(\epsilon)$ is an affinoid perfectoid space over $\mathrm{Spa}(k,\mathcal{O}_k)$. Now (\ref{recallintegralsame}) follows from this and the discussion after Lemma 4.2 of op. cit.

\end{proof}

Next, we have the following result which says that the for any $0 < \epsilon < p/(p+1)$ and any intersection $W$ of rational subsets of the affinoid $\hat{\mathcal{Y}}^{\mathrm{Ig}}(\epsilon)$ (see Lemma \ref{Uhataffinoidperfectoidlemma} for a proof of the affinoid perfectoidness of $\hat{\mathcal{Y}}^{\mathrm{Ig}}(\epsilon)$) such that $\hat{\mathcal{Y}}^{\mathrm{Ig}} \subset W$, there exists some $\epsilon'$ such that for any $0 < \epsilon'' \le \epsilon'$, $\hat{\mathcal{Y}}^{\mathrm{Ig}}(\epsilon') \subset W$. The argument, similar to \cite[proof of Lemma III.3.8]{ScholzeTorsion}, uses the constructible topology on the affinoid adic $\hat{\mathcal{Y}}^{\mathrm{Ig}}(\epsilon)$. Thus the underlying topological space $|\hat{\mathcal{Y}}^{\mathrm{Ig}}(\epsilon)|$ is spectral by \cite[Theorem 7.30 (1)]{Wedhorn}. Moreover, recall that for any $0 < \epsilon'' \le \epsilon'$, $\hat{\mathcal{Y}}^{\mathrm{Ig}}(\epsilon'') \subset \hat{\mathcal{Y}}^{\mathrm{Ig}}(\epsilon')$ is affinoid open by Lemma \ref{Uhataffinoidperfectoidlemma}. 

\begin{corollary}\label{rationalintersectioncorollary}\begin{enumerate}
\item Fix any $0 < \epsilon < p/(p+1)$ in the valuation group of $\mathcal{O}_k$. For any intersection $W$ of rational subsets of $\hat{\mathcal{Y}}^{\mathrm{Ig}}(\epsilon)$ such that $\hat{\mathcal{Y}}^{\mathrm{Ig}} \subset W$, there exists $0 < \epsilon' \le \epsilon$ in the valuation group of $\mathcal{O}_k$ such that for all $0 < \epsilon'' \le \epsilon'$ in the valuation group of $\mathcal{O}_k$,
$$\hat{\mathcal{Y}}^{\mathrm{Ig}} \subset \hat{\mathcal{Y}}^{\mathrm{Ig}}(\epsilon'') \subset W.$$
\item The same statement holds for $\mathcal{Y}^{\mathrm{Ig}}(\epsilon)$, $\mathcal{Y}^{\mathrm{Ig}}(\epsilon')$, $\mathcal{Y}^{\mathrm{Ig}}(\epsilon'')$ and $\mathcal{Y}^{\mathrm{Ig}}$ in place of $\hat{\mathcal{Y}}^{\mathrm{Ig}}(\epsilon)$, $\hat{\mathcal{Y}}^{\mathrm{Ig}}(\epsilon')$, $\hat{\mathcal{Y}}^{\mathrm{Ig}}(\epsilon'')$ and $\hat{\mathcal{Y}}^{\mathrm{Ig}}$, respectively. 
\end{enumerate}
\end{corollary}

\begin{proof} We prove (1); in light of (\ref{hatmathcalYIg}), (2) then immediately follows. All cited references in this proof will be to \cite{Wedhorn}. We have
$$\bigcap_{0 < \epsilon' < p/(p+1)}|\hat{\mathcal{Y}}^{\mathrm{Ig}}(\epsilon')| = \left|\varprojlim_{0 < \epsilon' < p/(p+1)}\hat{\mathcal{Y}}^{\mathrm{Ig}}(\epsilon')\right| \overset{(\ref{presentation4})}{=} |\hat{\mathcal{Y}}^{\mathrm{Ig}}|,$$
recalling that $|X|$ denotes the underlying topological space of an adic space $X$. Now let $|X|_{\mathrm{const}}$ denote the underlying set of $X$ endowed with the constructible topology; recall that since an adic space $X$ is locally spectral by Theorem 7.35 (1), then $|X|_{\mathrm{const}}$ is finer than $|X|$ by Proposition 3.23. We
$$\bigcap_{0 < \epsilon' < \epsilon}|\hat{\mathcal{Y}}^{\mathrm{Ig}}(\epsilon')| = |\hat{\mathcal{Y}}^{\mathrm{Ig}}| \subset |W| \subset |\hat{\mathcal{Y}}^{\mathrm{Ig}}(\epsilon)|.$$
Taking complements in $|\hat{\mathcal{Y}}^{\mathrm{Ig}}(\epsilon)|$, this implies
\begin{equation}\label{topologicalcomplement}|\hat{\mathcal{Y}}^{\mathrm{Ig}}(\epsilon)| \setminus |W| \subset \bigcup_{0 < \epsilon' < \epsilon}\left(|\hat{\mathcal{Y}}^{\mathrm{Ig}}(\epsilon)| \setminus |\hat{\mathcal{Y}}^{\mathrm{Ig}}(\epsilon')|\right).
\end{equation}
Recall that since $\hat{\mathcal{Y}}^{\mathrm{Ig}}(\epsilon') \subset \hat{\mathcal{Y}}^{\mathrm{Ig}}(\epsilon)$ is an affinoid open subset, it is spectral (Theorem 7.35 (1)) and thus quasicompact open in $|\hat{\mathcal{Y}}^{\mathrm{Ig}}(\epsilon)|$. This implies that $|\hat{\mathcal{Y}}^{\mathrm{Ig}}(\epsilon')|$ is closed in $|\hat{\mathcal{Y}}^{\mathrm{Ig}}(\epsilon)|_{\mathrm{const}}$ (Definition 3.17 (1)). Thus (\ref{topologicalcomplement}) implies that $\bigcup_{0 < \epsilon' < \epsilon}\left(|\hat{\mathcal{Y}}^{\mathrm{Ig}}(\epsilon)| \setminus |\hat{\mathcal{Y}}^{\mathrm{Ig}}(\epsilon')|\right)$ is an open cover of $|\hat{\mathcal{Y}}^{\mathrm{Ig}}(\epsilon)| \setminus |W|$ in $|\hat{\mathcal{Y}}^{\mathrm{Ig}}(\epsilon)|_{\mathrm{const}}$. We will show that $|\hat{\mathcal{Y}}^{\mathrm{Ig}}(\epsilon)| \setminus |W|$ is quasicompact in $|\hat{\mathcal{Y}}^{\mathrm{Ig}}(\epsilon)|_{\mathrm{const}}$, which implies that there exists a finite subcover of $\bigcup_{0 < \epsilon' < \epsilon}\left(|\hat{\mathcal{Y}}^{\mathrm{Ig}}(\epsilon)| \setminus |\hat{\mathcal{Y}}^{\mathrm{Ig}}(\epsilon')|\right)$ that contains $|\hat{\mathcal{Y}}^{\mathrm{Ig}}(\epsilon)| \setminus |W|$. Since $\epsilon'' \le \epsilon'$ implies $\hat{\mathcal{Y}}^{\mathrm{Ig}}(\epsilon'') \subset \hat{\mathcal{Y}}^{\mathrm{Ig}}(\epsilon')$, this will give the assertion. 

Recall that any rational subset of the affinoid $\hat{\mathcal{Y}}^{\mathrm{Ig}}(\epsilon)$ is quasicompact open (Theorem 7.35 (2)), and hence is a closed set in $|\hat{\mathcal{Y}}^{\mathrm{Ig}}(\epsilon)|_{\mathrm{const}}$ (Definition 3.17 (1)). Since $|W|$ is an intersection of constructible closed sets, it is pro-constructible (Definition 3.17 (2)) and closed $|\hat{\mathcal{Y}}^{\mathrm{Ig}}(\epsilon)|_{\mathrm{const}}$. Since $|W|$ is pro-constructible, it is retrocompact in $|\hat{\mathcal{Y}}^{\mathrm{Ig}}(\epsilon)|_{\mathrm{const}}$ (Remark 3.21 (3)), but since $|\hat{\mathcal{Y}}^{\mathrm{Ig}}(\epsilon)|_{\mathrm{const}}$ is itself quasicompact, this implies $|W|$ is quasicompact in $|\hat{\mathcal{Y}}^{\mathrm{Ig}}(\epsilon)|_{\mathrm{const}}$ (Definition 3.11 (2)-(3)). Thus by Corollary 3.20 (2), $|\hat{\mathcal{Y}}^{\mathrm{Ig}}(\epsilon)| \setminus |W|$ is quasicompact in $|\hat{\mathcal{Y}}^{\mathrm{Ig}}(\epsilon)|_{\mathrm{const}}$. By the previous paragraph, we are done. 

\end{proof}

%\begin{definition}\label{integralmodulidefinition}
%Suppose $0\le \epsilon < p/(p+1)$. Let 
%$$\mathrm{Ha}(\epsilon) \in \mathbf{\Gamma}(\omega_{E^+(\epsilon)/Y^+(\epsilon)}^{\otimes p-1})$$
%denote the restriction to $Y^+(\epsilon) \subset Y_0^+$ (see (\ref{YepsilonY0})) of the fixed choice of lift of Hasse invariant $\mathrm{Ha} \in \mathbf{\Gamma}(\omega_{E^+/Y^+}^{\otimes p-1})$ (see Choice \ref{globalHasseassumption}). 
%\end{definition}

\subsection{Comparing integral structures at infinite level}\label{compareintegralsection}Recall the notation of Convention \ref{Yconvention} and Definition \ref{OYdefinition}. The main purpose of this section is to show that the ring of functions $\mathbf{\Gamma}(\mathcal{O}_{\hat{\mathcal{Y}}^{\mathrm{Ig}}(\epsilon)^+})$ on the formal scheme $\hat{\mathcal{Y}}^{\mathrm{Ig}}(\epsilon)^+$ is equal to the ring of power-bounded functions $\mathbf{\Gamma}(\hat{\mathcal{O}}_{\mathcal{Y}^{\mathrm{Ig}}(\epsilon)}^+) \overset{(\ref{recallintegralsame})}{=} \mathbf{\Gamma}(\mathcal{O}_{\hat{\mathcal{Y}}^{\mathrm{Ig}}(\epsilon)}^+)$ on the adic space $\hat{\mathcal{Y}}^{\mathrm{Ig}}(\epsilon)$. This will follow from the normality and flatness of Katz-Mazur models (\cite[First Main Theorem 5.1.1]{KatzMazur}). We will often use the notation of Convention \ref{Gammaconvention}, i.e. denote global sections of sheaves $\mathcal{F}$ by $\mathbf{\Gamma}(\mathcal{F})$ throughout Sections  \ref{compareintegralsection} and \ref{spreadingoutsection}.%Similarly, we have a map
%$$\mathcal{O}_{Y'(\epsilon)^+} \subset \mathcal{O}_{Y'(\epsilon)}^+.$$

%We wish to study the difference between the integral models induced by $\mathcal{O}_{Y'(\epsilon)^+}$ and $\mathcal{O}_{Y'(\epsilon)}^+$ at infinite level. Recall that by \cite[Proposition 2.1.1]{ScholzeWeinstein}, there is a natural fully faithful functor from affine formal schemes to affinoid adic spaces $\mathrm{Spf}(R) \rightarrow \mathrm{Spa}(R,R)$. Note that we have an equality of underlying topological spaces $|\mathrm{Spa}(R,R)| = |\mathrm{Spa}(R[1/p],R)|$. 

%Note that $\hat{\mathcal{Y}}^{\mathrm{Ig}}(\epsilon)$ is an affinoid subdomain of $\hat{Y}_{\infty}$, which is perfectoid. Hence by \cite[Lemma 4.6]{Scholze}, it is also perfectoid, and so by Lemma 4.10(iv) of op. cit., we have 
%\begin{equation}\label{completesheaves}\mathcal{O}_{\hat{\mathcal{Y}}^{\mathrm{Ig}}(\epsilon)}^+ = \hat{\mathcal{O}}_{\mathcal{Y}^{\mathrm{Ig}}(\epsilon)}^+, \hspace{1cm} \mathcal{O}_{\hat{\mathcal{Y}}^{\mathrm{Ig}}(\epsilon)} = \mathcal{O}_{\hat{\mathcal{Y}}^{\mathrm{Ig}}(\epsilon)}^+[1/p] = \hat{\mathcal{O}}_{\mathcal{Y}^{\mathrm{Ig}}(\epsilon)}^+[1/p] = \hat{\mathcal{O}}_{\mathcal{Y}^{\mathrm{Ig}}(\epsilon)}.
%\end{equation}
In the notation of Convention \ref{Gammaconvention}, from the definitions (\ref{U}) and (\ref{hatmathcalYIg}) we have a map (see \cite[(7.1.8.2)]{deJong})
\begin{equation}\label{integralpowerboundedmap}\mathbf{\Gamma}(\mathcal{O}_{\hat{\mathcal{Y}}^{\mathrm{Ig}}(\epsilon)^+}) \rightarrow \mathbf{\Gamma}(\mathcal{O}_{\hat{\mathcal{Y}}^{\mathrm{Ig}}(\epsilon)}^+) \overset{(\ref{recallintegralsame})}{=} \mathbf{\Gamma}(\hat{\mathcal{O}}_{\mathcal{Y}^{\mathrm{Ig}}(\epsilon)}^+).
\end{equation}
This is an inclusion by the affinoidness of $\hat{\mathcal{Y}}^{\mathrm{Ig}}(\epsilon)$ (Lemma \ref{Uhataffinoidperfectoidlemma}). In fact, we have:

\begin{proposition}For any $0 \le \epsilon < p/(p+1)$ in the valuation group of $k$ (see Definition \ref{kdefinition}), the inclusion (\ref{integralpowerboundedmap}) is the identity, which induces the identification
\begin{equation}\label{integralpowerboundedsame}\mathbf{\Gamma}(\mathcal{O}_{\hat{\mathcal{Y}}^{\mathrm{Ig}}(\epsilon)^+}) = \mathbf{\Gamma}(\hat{\mathcal{O}}_{\mathcal{Y}^{\mathrm{Ig}}(\epsilon)}^+) = \mathbf{\Gamma}(\mathcal{O}_{\hat{\mathcal{Y}}^{\mathrm{Ig}}(\epsilon)}^+).
\end{equation}
\end{proposition}

\begin{proof}We will work in the notation and setting of Definition \ref{hatmathcalYIgepsilonDefinition} throughout this proof. By \cite[First Main Theorem 5.1.1]{KatzMazur}, $Y_n^+$ is regular (and hence normal) and flat as a scheme over $\mathrm{Spec}(\mathbb{Z}_p)$. Thus $\hat{Y}_n^+$ is regular (and hence normal) and flat as a formal scheme over $\mathrm{Spf}(\mathbb{Z}_p)$. Thus each $Y_n^{\mathrm{Ig}}(\epsilon)^+$ is regular (and hence normal) and flat as a formal scheme over $\mathrm{Spf}(\mathcal{O}_k)$. Recall that $Y_n^{\mathrm{Ig}}(\epsilon)$ denotes the adic generic fiber (an adic space over $\mathrm{Spa}(k,\mathcal{O}_k)$) of $Y_n^{\mathrm{Ig}}(\epsilon)^+$. Let $\mathcal{O}_{Y_n^{\mathrm{Ig}}(\epsilon)}^+$ be the integral sheaf of $Y_n^{\mathrm{Ig}}(\epsilon)$ as in Definition \ref{OYdefinition}. Therefore by \cite[Theorem 7.4.1]{deJong}, we have
$$\mathbf{\Gamma}(\mathcal{O}_{Y_n^{\mathrm{Ig}}(\epsilon)^+}) = \mathbf{\Gamma}(\mathcal{O}_{Y_n^{\mathrm{Ig}}(\epsilon)}^+).$$
Thus 
\begin{equation}\label{presentationssame}\varinjlim_n \mathbf{\Gamma}(\mathcal{O}_{Y_n^{\mathrm{Ig}}(\epsilon)^+}) = \varinjlim_n \mathbf{\Gamma}(\mathcal{O}_{Y_n^{\mathrm{Ig}}(\epsilon)}^+).
\end{equation}
By Lemma \ref{Uaffinoidperfectoidlemma} and \cite[Lemma 4.10 (iv)]{Scholze}, we have
\begin{equation}\label{completionscoincide}\widehat{\mathbf{\Gamma}(\mathcal{O}_{\mathcal{Y}^{\mathrm{Ig}}(\epsilon)}^+)} = \mathbf{\Gamma}(\hat{\mathcal{O}}_{\mathcal{Y}^{\mathrm{Ig}}(\epsilon)}^+)
\end{equation}
where the left-hand side is the $p$-adic completion of $\mathbf{\Gamma}(\mathcal{O}_{\mathcal{Y}^{\mathrm{Ig}}(\epsilon)}^+)$. Now taking $p$-adic completions of (\ref{presentationssame}), we get
\begin{align*}\mathbf{\Gamma}(\mathcal{O}_{\hat{\mathcal{Y}}^{\mathrm{Ig}}(\epsilon)^+}) \overset{(\ref{firstpresentation})}{=} \widehat{\left(\varinjlim_n\mathbf{\Gamma}(\mathcal{O}_{Y_n^{\mathrm{Ig}}(\epsilon)^+})\right)} \overset{(\ref{presentationssame})}{=} \widehat{\left(\varinjlim_n\mathbf{\Gamma}(\mathcal{O}_{Y_n^{\mathrm{Ig}}(\epsilon)}^+)\right)} \overset{(\ref{hatmathcalYIg})}{=} \widehat{\mathbf{\Gamma}(\mathcal{O}_{\mathcal{Y}^{\mathrm{Ig}}(\epsilon)}^+)} &\overset{(\ref{completionscoincide})}{=} \mathbf{\Gamma}(\hat{\mathcal{O}}_{\mathcal{Y}^{\mathrm{Ig}}(\epsilon)}^+) \\
&\overset{(\ref{recallintegralsame})}{=} \mathbf{\Gamma}(\mathcal{O}_{\hat{\mathcal{Y}}^{\mathrm{Ig}}(\epsilon)}^+),
\end{align*}
where all hats above rings in the above displayed equation denote $p$-adic completion. 

\end{proof}

%\begin{corollary}For any $0 \le \epsilon < p/(p+1)$, we have
%\begin{equation}\label{integralpowerboundedsame2}\mathbf{\Gamma}(\mathcal{O}_{\hat{\mathcal{Y}}^{\mathrm{Ig}}(\epsilon)^+}) = \mathbf{\Gamma}(\mathcal{O}_{\hat{\mathcal{Y}}^{\mathrm{Ig}}(\epsilon)}^+) = \mathbf{\Gamma}(\hat{\mathcal{O}}_{\mathcal{Y}^{\mathrm{Ig}}(\epsilon)}^+).
%\end{equation}
%\end{corollary}

%\begin{proof}We have
%$$\mathbf{\Gamma}(\mathcal{O}_{\hat{\mathcal{Y}}^{\mathrm{Ig}}(\epsilon)^+})\overset{(\ref{integralpowerboundedsame})}{=} \mathbf{\Gamma}(\mathcal{O}_{\hat{\mathcal{Y}}^{\mathrm{Ig}}(\epsilon)}^+) \overset{(\ref{recallintegralsame})}{=} \mathbf{\Gamma}(\hat{\mathcal{O}}_{\mathcal{Y}^{\mathrm{Ig}}(\epsilon)}^+).$$
%\end{proof}

We have the following immediate Corollary.

\begin{corollary}\label{completioncorollary}$\hat{\mathcal{O}}_{\mathcal{Y}^{\mathrm{Ig}}(0)}^+(\mathcal{Y}^{\mathrm{Ig}}(0))$ is the $p$-adic completion of $\varinjlim_{\epsilon > 0}\hat{\mathcal{O}}_{\mathcal{Y}^{\mathrm{Ig}}(\epsilon)}^+(\mathcal{Y}^{\mathrm{Ig}}(\epsilon))$.
\end{corollary}

\begin{proof}This follows from Corollary \ref{completioncorollary0} and (\ref{integralpowerboundedsame}). 

\end{proof}

\subsection{Lifts of Frobenius via the canonical subgroup}In this section, we study how the action of $g$ from (\ref{gdefinition}) gives rise to a characteristic 0 lift of a certain Frobenius operator on certain functions on $\hat{\mathcal{Y}}^{\mathrm{Ig}}(\epsilon)^+$. The precise statement is contained in Theorem \ref{Frobeniuslifttheorem}, with the ``lifting Frobenius'' property being described by (\ref{Frobeniuslift}). We note that since $g$ has inverse $g^{-1}$, we can view $g^{-1}$ as lifting the $p^{\mathrm{th}}$ root map (see (\ref{Frobeniuslift2})). This existence of a characteristic 0 lifting of a root of Frobenius is a salient feature of infinite level.

%\begin{definition}\label{Vdefinition}Suppose we have an open subset $\mathcal{V} \subset Y_{\infty}$ with $p$-adic completion $\hat{\mathcal{V}} \subset \hat{Y}_{\infty}$ such that 
%$$\hat{\mathcal{V}} = \{f_1 \neq 0, f_2 \neq 0, \ldots, f_k \neq 0\}$$
%for some $f_i \in \mathcal{O}_{\hat{\mathcal{V}}_x}(\hat{\mathcal{V}}_x)$, $1 \le i \le k$ with $k \in \mathbb{Z}_{\ge 1}$. Then $\mathcal{V} \subset Y_{\infty}$ is perfectoid by \cite[Lemma 4.6]{Scholze} and $\hat{\mathcal{V}}$ is a perfectoid space with $\hat{\mathcal{V}} \sim \mathcal{V}$ and
%\begin{equation}\label{mathcalOV}\mathcal{O}_{\hat{\mathcal{V}}}(\hat{\mathcal{V}}) = \hat{\mathcal{O}}_{\mathcal{V}}(\mathcal{V})
%\end{equation}
%by Lemma 4.5 of op. cit. The restriction 
%$$\mathcal{E}_{\mathcal{V}} := \mathcal{E} \times_{Y_{\infty}}\mathcal{V}$$
%is the universal object for the moduli problem classifying tuples $(A,P,e_1,e_2,i)$ where $A$ is a (false) elliptic curve, $P$ a $\Gamma$-level structure, $(e_1,e_2)$ a $\Gamma(p^{\infty})$-level structure, $i$ an $\mathcal{O}_D$-endomorphism structure, and such that $f_i(A,P,e_1,e_2,i) \neq 0$ for all $1 \le i \le k$. 
%\end{definition}

Recall the set $\mathcal{V}_x \overset{(\ref{Vz})}{=} \{z_{\mathrm{HT}} \neq 0\} \subset Y_{\infty}$. Recall $g \in GL_2(\mathbb{Q}_p)$ from (\ref{gdefinition}). By (\ref{zHTtransformationproperty}), we have $(g^m)^*z_{\mathrm{HT}} = p^m\cdot z_{\mathrm{HT}}$ for every $m \in \mathbb{Z}_{\ge 0}$. Hence $\mathcal{V}_x \cdot g^m = \mathcal{V}_x$ for every $m \in \mathbb{Z}$.

\begin{remark}\label{universalremark}Note that the restriction $\mathcal{E}_{\mathcal{V}_x} \rightarrow \mathcal{V}_x$ of the universal object $\mathcal{E} \rightarrow Y_{\infty}$ to $\mathcal{V}_x \subset Y_{\infty}$ is the universal object classifying quintuples $(A,i,P,e_1,e_2)$, $A$ a (false) elliptic curve, $i$ an $\mathcal{O}_D$-endomorphism structure, $P$ a $\Gamma$-level structure (recall Convention \ref{Yconvention}) and $(e_1,e_2)$ a $\Gamma(p^{\infty})$-level structure with $e_2$ not trivializing the Hodge-Tate filtration (by the definition of $z_{\mathrm{HT}}$, this last condition is equivalent to $z_{\mathrm{HT}} \neq 0$, see for example \cite[Section 2.4]{ChojeckiHansenJohansson} recalling that our $z_{\mathrm{HT}}$ is $1/\frak{z}$ in loc. cit.).
\end{remark} 

We will study properties of certain special functions on $\mathcal{V}_x$ in Theorems \ref{Frobeniuslifttheorem}, \ref{crystalline2}, \ref{pintegraltheorem} and \ref{pintegraltheorem2}.

\begin{convention}\label{Y1convention}Recall the group $\Gamma_1(p^n) \subset GL_2(\hat{\mathbb{Z}})$ from Definition \ref{congruencesubgroups}. Let $\Gamma$ continue to be the tame level structure from Convention \ref{Yconvention}. 
\begin{enumerate}
\item We have associated Shimura curves $\mathbb{Y}(\Gamma \cap \Gamma_1(p^n))$, $Y(\Gamma \cap \Gamma_1(p^n))$, $Y(\Gamma \cap \Gamma_1(p^n))^+$ and $\widehat{Y(\Gamma \cap \Gamma_1(p^n))}^+$ from Sections \ref{algebraicYsection},  \ref{formalShimurasection} and \ref{adicShimurasection}. For brevity, we will denote these objects by $\mathbb{Y}_1(p^n)$, $Y_1(p^n)$ and $Y_1^+(p^n)$, respectively.
\item For any $\epsilon$ in the valuation group of $\mathcal{O}_k$ (see Definition \ref{kdefinition}), we have the formal scheme $Y(\Gamma \cap \Gamma_1(p^n))^+(\epsilon)$ from Definition \ref{YGamma+epsilondefinition}. We will denote this by $Y_1^+(n,\epsilon)$ for brevity. 
\end{enumerate}
\end{convention}

\begin{remark}\begin{enumerate}
\item Note that for any $0 \le \epsilon < p/(p+1)$ in the valuation group of $\mathcal{O}_k$ and $n(\epsilon)$ as in (\ref{nepsilondefinition}), the natural projection 
\begin{equation}\label{IgusaY1map0}\hat{\mathcal{Y}}^{\mathrm{Ig}}(\epsilon')^+ \rightarrow Y_1^+(n(\epsilon'),\epsilon'), \hspace{1cm} (A,i,P,e_1,e_2) \mapsto (A,i,P,e_{1,n(\epsilon')})
\end{equation}
(where $A$ is a (false) elliptic curve, $i$ an $\mathcal{O}_D$-endomorphism structure, $P$ a $\Gamma = \Gamma(N)$-level structure, and $(e_1,e_2)$ a $\Gamma(p^{\infty})$-level structure) is a map of adic spaces over $\mathrm{Spa}(\mathcal{O}_k,\mathcal{O}_k)$. For any $0 \le n \le n(\epsilon')$, we have an inclusion $Y_1^+(n,\epsilon') \subset Y_1^+(n(\epsilon'),\epsilon')$ from the definition, and so composing (\ref{IgusaY1map0}) with this we get
\begin{equation}\label{IgusaY1map}\mathcal{Y}^{\mathrm{Ig}}(\epsilon')^+ \rightarrow Y_1^+(n,n(\epsilon')).
\end{equation}

\item For any $m \in \mathbb{Z}_{> 0}$, $n(\epsilon/p^m) = n(\epsilon) + m$,
$$\left(\mathcal{Y}^{\mathrm{Ig}}(\epsilon) \cap Y_{\infty}(\epsilon/p^m)\right) \rightarrow Y_1(n(\epsilon),\epsilon/p^m), \hspace{1cm} (A,i,P,e_1,e_2) \mapsto (A,i,P,e_{1,n(\epsilon)})$$
is a pro-finite \'{e}tale Galois cover of adic spaces with (pro-finite \'{e}tale) Galois group $\Gamma_{1,p}(p^{n(\epsilon)})$ as defined in (\ref{Gamma1pdefinition}). 

\item Recall that 
$$\mathcal{Y}^{\mathrm{Ig}}(\epsilon/p^m) \subset \left(\mathcal{Y}^{\mathrm{Ig}}(\epsilon) \cap Y_{\infty}(\epsilon/p^m)\right)$$
is a strict inclusion; $e_{1,n(\epsilon)+m}$ (see (\ref{eindefinition})) trivializes the order-$p^{n(\epsilon)+m}$-canonical subgroup on $\mathcal{Y}^{\mathrm{Ig}}(\epsilon/p^m)$, and $e_{1,n(\epsilon)}$ trivializes the order-$p^{n(\epsilon)}$ canonical subgroup on $\mathcal{Y}^{\mathrm{Ig}}(\epsilon)\cap Y_{\infty}(\epsilon/p^m)$ but $e_{1,n(\epsilon) + m}$ does not necessarily trivialize the order-$p^{n(\epsilon) + m}$-canonical subgroup. 
\end{enumerate}
\end{remark}

\begin{definition}For any $\epsilon' \le \epsilon$, let
$$\mathrm{res}_{\epsilon,\epsilon'} : \mathbf{\Gamma}(\mathcal{O}_{\hat{\mathcal{Y}}^{\mathrm{Ig}}(\epsilon)^+})\rightarrow \mathbf{\Gamma}(\mathcal{O}_{\hat{\mathcal{Y}}^{\mathrm{Ig}}(\epsilon')^+})$$
denote restriction induced by the inclusion $\hat{\mathcal{Y}}^{\mathrm{Ig}}(\epsilon')^+ \subset \hat{\mathcal{Y}}^{\mathrm{Ig}}(\epsilon)^+$.
\end{definition}

\begin{theorem}\label{Frobeniuslifttheorem}For any $0\le \epsilon < p/(p+1)$ in the valuation group of $\mathcal{O}_k$ and $\mathcal{V}_x$ as in (\ref{Vz}), suppose we are given
$$F \in \hat{\mathcal{O}}_{Y_{\infty}}(\mathcal{V}_x).$$
Let 
$$f := F|_{\mathcal{Y}^{\mathrm{Ig}}(\epsilon)} \in \mathbf{\Gamma}(\hat{\mathcal{O}}_{\mathcal{Y}^{\mathrm{Ig}}(\epsilon)})$$
denote the restriction to $\mathcal{Y}^{\mathrm{Ig}}(\epsilon) \subset \mathcal{V}_x$. 

Suppose that
\begin{enumerate}
\item $$\mathrm{res}_{\epsilon,\epsilon/p^m}(f) \in \mathbf{\Gamma}(\hat{\mathcal{O}}_{\mathcal{Y}^{\mathrm{Ig}}(\epsilon/p^m)}^+)$$
for some $m \in \mathbb{Z}_{\ge 0}$, and 
\item moreover assume that there exists
\begin{equation}\label{lowerlevelassumption}f_0 \in \mathbf{\Gamma}(\mathcal{O}_{Y_1^+(\alpha+1,\epsilon/p^m)/\frak{m}}) = \mathbf{\Gamma}(\mathcal{O}_{Y_1^+(\alpha+1,\epsilon/p^m)}) \otimes_{\mathcal{O}_k}\mathcal{O}_k/\frak{m}
\end{equation}
for some integer $0 \le \alpha \le m$ such that
\begin{equation}\label{congruenceassumption}\mathrm{res}_{\epsilon,\epsilon/p^m}(f) \equiv f_0|_{\mathcal{Y}^{\mathrm{Ig}}(\epsilon/p^m)^+/\frak{m}} \pmod{\frak{m}\mathbf{\Gamma}(\hat{\mathcal{O}}_{\mathcal{Y}^{\mathrm{Ig}}(\epsilon/p^m)}^+)}, 
\end{equation}
where $f_0|_{\mathcal{Y}^{\mathrm{Ig}}(\epsilon/p^m)^+/\frak{m}}$ denotes the pullback of $f_0$ along the map 
$$\mathcal{Y}^{\mathrm{Ig}}(\epsilon/p^m)^+/\frak{m} \rightarrow Y_1^+(\alpha+1,\epsilon/p^m)/\frak{m}$$
obtained by reducing (\ref{IgusaY1map}) (with $\epsilon' = \epsilon/p^m$ and $n = \alpha+1$) modulo $\frak{m}$. 
\end{enumerate}
%In other words, assume $f_0$ descends to $Y_1^+(\alpha+1,\epsilon/p^m)/\frak{m}$. 
(Here, recall $\frak{m} \subset \mathcal{O}_k$ is the maximal ideal, see Definition \ref{kdefinition}. See Convention \ref{Y1convention} for the definition of the formal scheme $Y_1^+(\alpha+1,\epsilon/p^m)$ and Definition \ref{reductionformalschemedefinition} for the notation of $Y_1^+(\alpha+1,\epsilon/p^m)$.)

Then we have the following congruence of elements of $\mathbf{\Gamma}(\hat{\mathcal{O}}_{\mathcal{Y}^{\mathrm{Ig}}(\epsilon/p^m)}^+)$: 
\begin{equation}\label{Frobeniuslift}(g^{m-\alpha})^*\mathrm{res}_{\epsilon,\epsilon/p^{\alpha}}(f) \equiv \mathrm{res}_{\epsilon,\epsilon/p^m}(f)^{p^{m-\alpha}} \pmod{\frak{m}\mathbf{\Gamma}(\hat{\mathcal{O}}_{\mathcal{Y}^{\mathrm{Ig}}(\epsilon/p^m)}^+)}.
\end{equation}
Moreover, 
$$\mathrm{res}_{\epsilon,\epsilon/p^{\alpha}}(f) \in \mathbf{\Gamma}(\hat{\mathcal{O}}_{\mathcal{Y}^{\mathrm{Ig}}(\epsilon/p^{\alpha})}^+)$$
and we have the following congruence of elements of $\mathbf{\Gamma}(\hat{\mathcal{O}}_{\mathcal{Y}^{\mathrm{Ig}}(\epsilon/p^{\alpha})}^+)$:
\begin{equation}\label{Frobeniuslift2}\mathrm{res}_{\epsilon,\epsilon/p^{\alpha}}(f) \equiv ((g^{-(m-\alpha)})^*\mathrm{res}_{\epsilon,\epsilon/p^m}(f))^{p^{m-\alpha}} \pmod{\frak{m}\mathbf{\Gamma}(\hat{\mathcal{O}}_{\mathcal{Y}^{\mathrm{Ig}}(\epsilon/p^{\alpha})}^+)}.
\end{equation}
\end{theorem}

\begin{proof}

Recall $\mathbf{\Gamma}(\hat{\mathcal{O}}_{\mathcal{Y}^{\mathrm{Ig}}(\epsilon/p^m)}^+) \overset{(\ref{integralpowerboundedsame})}{=} \mathbf{\Gamma}(\mathcal{O}_{\hat{\mathcal{Y}}^{\mathrm{Ig}}(\epsilon/p^m)^+})$. We will show the formula (\ref{Frobeniuslift}) holds on any test object on $\hat{\mathcal{Y}}^{\mathrm{Ig}}(\epsilon/p^m)^+$, which then gives the assertion. First, let us make a few observations. Suppose 
$$(A,i,P,e_1^+,e_2^+,u_{\epsilon/p^m}) \in \hat{\mathcal{Y}}^{\mathrm{Ig}}(\epsilon/p^m)^+(R^+,R^+)$$
for any affinoid $(\mathcal{O}_k,\mathcal{O}_k)$-algebra $(R^+,R^+)$; here, $A$ is a (false) elliptic curve over $\mathrm{Spa}(R^+,R^+)$, $i$ is an $\mathcal{O}_D$-endomorphism structure, $P$ is a $\Gamma$-level structure (where $\Gamma = \Gamma(N)$ is as in Convention \ref{Yconvention}), $(e_1^+,e_2^+)$ is a $\Gamma(p^{\infty})$-Drinfeld level structure over $\mathrm{Spa}(R^+,R^+)$, and $u_{\epsilon/p^m}$ is an equivalence class of section of $\omega^{\otimes (1-p)}$ such that $u_{\epsilon/p^m}\cdot \mathrm{Ha} = p^{\epsilon/p^m}$ in $R^+/p$ for a local lift of the Hasse invariant $\mathrm{Ha}$ (cf. Definition \ref{YGamma+epsilondefinition}). Then since for any $\epsilon' \ge \epsilon/p^m$ in the valuation group of $\mathcal{O}_k$ we have $\hat{\mathcal{Y}}^{\mathrm{Ig}}(\epsilon/p^m)^+ \subset \hat{\mathcal{Y}}^{\mathrm{Ig}}(\epsilon')^+$, there exists a local section $u_{\epsilon'}$ of $\omega_+^{\otimes (1-p)}$ ($\omega_+$ as in (\ref{omegaY})) such that 
$$(A,i,P,e_1^+,e_2^+,u_{\epsilon'}) \in \hat{\mathcal{Y}}^{\mathrm{Ig}}(\epsilon')^+(R^+,R^+).$$

%Note that by (\ref{hatmathcalYIg}) we have an equality of underlying topological spaces $|\hat{\mathcal{Y}}^{\mathrm{Ig}}(\epsilon)^+| = \mathcal{Y}^{\mathrm{Ig}}(\epsilon)^+|$, and so we map view $(A,P,e_1,e_2,i,u_{\epsilon}) \in \mathcal{Y}^{\mathrm{Ig}}$

Recall the universal object $\mathcal{E}_{\mathcal{V}_x} \rightarrow \mathcal{V}_x$ (Remark \ref{universalremark}) which is the universal (false) elliptic curve with $\mathcal{O}_D$-endomorphism structure, $\Gamma$-level structure and $\Gamma(p^{\infty})$-level structure $(e_1,e_2)$ such that $e_2$ does not trivialize the Hodge-Tate filtration. In particular, there is no data of a local section of $\omega_+^{\otimes (1-p)}$ in this universal object. Because of this, the points $(A,i,P,e_1^+,e_2^+,u_{\epsilon/p^m}) \in \hat{\mathcal{Y}}^{\mathrm{Ig}}(\epsilon/p^m)^+(R^+,R^+) $ and $(A,i,P,e_1^+,e_2^+,u_{\epsilon'}) \in \hat{\mathcal{Y}}^{\mathrm{Ig}}(\epsilon')^+(R^+,R^+)$
map under the inclusions
$$\hat{\mathcal{Y}}^{\mathrm{Ig}}(\epsilon/p^m)^+(R^+,R^+) \subset \hat{\mathcal{V}}_x(R,R^+), \hspace{1cm} \hat{\mathcal{Y}}^{\mathrm{Ig}}(\epsilon')^+(R^+,R^+) \subset \hat{\mathcal{V}}_x(R,R^+)$$
to the \emph{same point} 
$$y = (A,i,P,e_1,e_2)  \in \hat{\mathcal{V}}_x(R,R^+),$$
where here we make a slight abuse of notation and let $(A,P)$ denote $(A,P) \otimes_{R^+}R$. Thus, since $F \in \hat{\mathcal{O}}_{Y_{\infty}}(\mathcal{V}_x)$ by assumption,
$$F(A,i,P,e_1^+,e_2^+,u_{\epsilon/p^m}) = F(A,i,P,e_1,e_2) = F(A,i,P,e_1^+,e_2^+,u_{\epsilon'}).$$
That is, the value of $F$  on any $(A,i,P,e_1^+,e_2^+,u_{\epsilon/p^m}) \in \hat{\mathcal{Y}}^{\mathrm{Ig}}(\epsilon/p^m)^+(R^+,R^+)$ does not depend on the section $u_{\epsilon/p^m}$. Because of these values do not depend on $u_{\epsilon/p^m}$, we will often omit $u_{\epsilon/p^m}$ from the notation whenever we evaluate $F$ (or its restriction $f$) at test objects $(A,i,P,e_1^+,e_2^+,u_{\epsilon/p^m})$ for the remainder of the proof. We will also often omit the $\mathcal{O}_D$-endomorphism structure $i$ from notation below, since it will play no role in our arguments. By (\ref{congruenceassumption}), $f_0$ is also independent of the section $u_{\epsilon/p^m}$, so similarly we will often omit $u_{\epsilon/p^m}$ and $i$ from notation when we evaluate $f_0$ at test objects.

If $0 \le \epsilon < p/(p+1)$, then the first map of Theorem \ref{perfectoidpropertytheorem} as well as the first map of (\ref{gisomorphism}) are equal to relative Frobenius after reducing modulo $p^{1-p\epsilon}$, by \cite[Theorem 3.1]{Katzpamf}.\footnote{Theorem 3.1 of op. cit. is stated for the case $D = M_2(\mathbb{Q})$. In the $D \neq M_2(\mathbb{Q})$ case, for a false elliptic curve $A$, the formal group of $e^1A[p^{\infty}]$ is also 1-dimensional (see Convention \ref{idempotentconvention}), and so the arguments of Chapter 3 of op. cit. apply to this case as well. The assumptions on the level $n$ in op. cit. are to ensure the existence of a global lift of the Hasse invariant on all of $Y^+$. As we will work with Definition \ref{reductionformalschemedefinition}, which only uses local lifts of the Hasse invariant (which always exist by Serre's affineness criterion and the argument of \cite[Section 7]{Kassaei}), we do not need these extra assumptions on the tame level $\Gamma$.} %We will restrict to $Y_0^+$ from Choice \ref{globalHasseassumption}, which is affine and so admits a lift of the Hasse lift. The locus $\mathrm{ord}_p(r) < p/(p+1)$ in the notation of op. cit. is contained in $Y_0^+$, and one can check that the proof of Theorem 3.1 of op. cit. works with the lift of Hasse invariant $\mathrm{Ha}$ in Choice \ref{globalHasseassumption}, and we do not need any assumptions on the level $n$.

Let $0 \le \epsilon < p/(p+1)$, and let $(A,P,e_1^+,e_2^+,u_{\epsilon/p^m}) \in \hat{\mathcal{Y}}^{\mathrm{Ig}}(\epsilon/p^m)^+(R^+,R^+)$ for any $p$-adically complete $(\mathcal{O}_k,\mathcal{O}_k)$-algebra $(R^+,R^+)$ as in the first paragraph of this proof. As mentioned at the end of the second paragraph, we will often omit $u_{\epsilon/p^m}$ from notation when evaluating $F$ (and thus $f$) at objects in $\hat{\mathcal{Y}}^{\mathrm{Ig}}(\epsilon/p^m)^+$, as this value does not depend on $u_{\epsilon/p^m}$, and we will often omit the $\mathcal{O}_D$-endomorphism structure $i$ from notation for brevity. Recall the notation of Convention \ref{idempotentconvention} and Definition \ref{YGamma+definition} (5). Then $\mathbb{Z}_p\cdot e_{1,m+1}^+\subset A[p^{m+1}]$ is the order-$p^{m+1}$ canonical subgroup, and its reduction modulo $\frak{m}R^+$ is the kernel of relative $p^{m+1}$-power Frobenius. Thus, 
$$e_{1,m+1}^+/(\mathbb{Z}_p \cdot e_{1,m-\alpha}^+)\subset A[p^{m+1}]/(\mathbb{Z}_p\cdot e_{1,m-\alpha}^+) \subset A^{(p^{m-\alpha})}[p^{\alpha+1}] \pmod{\frak{m}R^+}$$
is the kernel of relative $p$-power Frobenius, and so
\begin{equation}\label{Frobeniusquotient}e_{1,m+1}^+/(\mathbb{Z}_p \cdot e_{1,m-\alpha}^+) \cong e_{1,\alpha+1}^+\otimes_{R^+}R^+/\frak{m}R^+ \otimes_{R^+/\frak{m}R^+,\mathrm{Frob}^{m-\alpha}}R^+/\frak{m}R^+
\end{equation}
where $\mathrm{Frob}^{m-\alpha} : R^+/\frak{m}R^+ \rightarrow R^+/\frak{m}R^+$ is the $p^{m-\alpha}$-power absolute Frobenius. 

Let $f_0 \in \mathbf{\Gamma}(\mathcal{O}_{Y_1^+(\alpha+1,\epsilon/p^m)/\frak{m}})$ be as in (\ref{congruenceassumption}). Via pullback along the surjective map 
$$Y_1^+(m+1,\epsilon/p^m) \rightarrow Y_1^+(\alpha+1,\epsilon/p^m)$$
(recall $\alpha \le m$), we can view $f_0 \in \mathbf{\Gamma}(\mathcal{O}_{Y_1^+(m+1,\epsilon/p^m)/\frak{m}})$. We will use the same notation for viewing $f_0$ in $\mathbf{\Gamma}(\mathcal{O}_{Y_1^+(\alpha+1,\epsilon/p^m)/\frak{m}})$ or in $\mathbf{\Gamma}(\mathcal{O}_{Y_1^+(m+1,\epsilon/p^m)/\frak{m}})$, as the ambient ring will be obvious from context. 

Note
$$(A,P, e_1^+,e_2^+) \mapsto (A,P,e_{1,m+1}^+)$$
under the natural projection (given by $p$-adically completing (\ref{IgusaY1map0}))
$$\hat{\mathcal{Y}}^{\mathrm{Ig}}(\epsilon/p^m)^+ \rightarrow Y_1^+(m+1,\epsilon/p^m).$$
Moreover, since $\mathbb{Z}_p\cdot e_{1,m-\alpha}$ is the kernel of the $p^{m-\alpha}$-power relative Frobenius,
\begin{equation}\label{Frobeniuscalculation}\begin{split}&(A/\langle e_{1,m-\alpha}^+\rangle,P/\langle e_{1,m-\alpha}^+\rangle, e_{m+1}^+/(\mathbb{Z}_p\cdot e_{1,m-\alpha}^+)) \overset{(\ref{Frobeniusquotient})}{\equiv}  (A^{(p^{m-\alpha})},P^{(p^{m-\alpha})},(e_{1,\alpha+1}^+)^{(p^{m-\alpha})}) \\
&\equiv (A,P,e_{1,\alpha+1}^+) \otimes_{R^+}R^+/\frak{m}R^+ \otimes_{R^+/\frak{m}R^+,\mathrm{Frob}^{m-\alpha}}R^+/\frak{m}R^+\pmod{\frak{m}R^+},
\end{split}
\end{equation}
where 
$$A^{(p^{m-\alpha})} := A \otimes_{R^+}R^+/\frak{m}R^+\otimes_{R^+/\frak{m}R^+,\mathrm{Frob}^{m-\alpha}}R^+/\frak{m}R^+,$$
$$P^{(p^{m-\alpha})} := P \otimes_{R^+}R^+/\frak{m}R^+ \otimes_{R^+/\frak{m}R^+,\mathrm{Frob}^{m-\alpha}}R^+/\frak{m}R^+$$
and 
$$(e_{1,\alpha+1}^+)^{(p^{m-\alpha})} := e_{1,\alpha+1}^+ \otimes_{R^+}R^+/\frak{m}R^+ \otimes_{R^+/\frak{m}R^+,\mathrm{Frob}^{m-\alpha}}R^+/\frak{m}R^+$$
are the base changes along $\mathrm{Frob}^{m-\alpha}: R^+/\frak{m}R^+ \rightarrow R^+/\frak{m}R^+$. 

%Since $(A,P,e_{1,\alpha+1}^+) \pmod{\frak{m}R^+}$ is in the special fiber of $Y_1^+(\alpha+1,\epsilon/p^m)$ and the action of $p$-power absolute Frobenius on $(A,P,e_{1,\alpha+1}^+) \pmod{\frak{m}R^+}$ is periodic with finite orbit (by finite typeness), then $(A/\langle e_{1,m-\alpha}^+\rangle,P/\langle e_{1,m-\alpha}^+\rangle, e_{m+1}^+/(\mathbb{Z}_p\cdot e_{1,m-\alpha}^+)) \pmod{\frak{m}R^+}$ is in the special fiber of $Y_1^+(\alpha+1,\epsilon/p^m)$ by (\ref{Frobeniuscalculation}). 
Recall that $\frak{m}$ is the maximal ideal of $\mathcal{O}_k$. Thus the $Y_1^+(\alpha+1,\epsilon')/\frak{m}$ (in the notation of Definition \ref{reductionformalschemedefinition}) are all isomorphic for $0 < \epsilon' < 1$ in the valuation group of $\mathcal{O}_k$ (see \cite[Lemma III.2.13]{ScholzeTorsion}). Therefore $f_0$
%$$f_0 \pmod{\frak{m}\mathbf{\Gamma}(\mathcal{O}_{Y_1^+(\alpha+1,\epsilon/p^m)})} \in \mathbf{\Gamma}(\mathcal{O}_{Y_1^+(\alpha+1,\epsilon/p^m)})/\frak{m}\mathbf{\Gamma}(\mathcal{O}_{Y_1^+(\alpha+1,\epsilon/p^m)})$$
 can be evaluated at $(A/\langle e_{1,m-\alpha}^+\rangle,P/\langle e_{1,m-\alpha}^+\rangle, e_{m+1}^+/(\mathbb{Z}_p\cdot e_{1,m-\alpha}^+)) \pmod{\frak{m}R^+}$. We will use this fact in the second congruence of the chain of congruences in the next paragraph. 

By definition of the $GL_2(\mathbb{Q}_p)$-action (see \cite[End of Section 2.2]{ChojeckiHansenJohansson}) we have 
\begin{align*}(g^{m-\alpha})^*f(A,P,e_1^+,e_2^+) &\overset{(\ref{congruenceassumption})}{\equiv} (g^{m-\alpha})^*f_0(A,P,e_1^+,e_2^+)\\
&\equiv f_0(A/\langle e_{1,m-\alpha}^+\rangle,P/\langle e_{1,m-\alpha}^+\rangle,e_1^+/(\mathbb{Z}_p\cdot e_{1,m-\alpha}^+),e_2^+/(\mathbb{Z}_p\cdot e_{1,m-\alpha}^+)) \\
&\overset{(\ref{lowerlevelassumption})}{\equiv} f_0(A/\langle e_{1,m-\alpha}^+\rangle ,P/\langle e_{1,m-\alpha}^+\rangle,e_{1,m+1}^+/(\mathbb{Z}_p \cdot e_{1,m-\alpha}^+))\\
&\overset{(\ref{Frobeniuscalculation})}{\equiv} f_0((A,P,e_{1,\alpha+1}^+)\otimes_{R^+/\frak{m}R^+,\mathrm{Frob}^{m-\alpha}}R^+/\frak{m}R^+) \\
&\equiv f_0(A,P,e_{1,\alpha+1}^+) \otimes_{R^+/\frak{m}R^+,\mathrm{Frob}^{m-\alpha}}R^+/\frak{m}R^+\\
&\equiv f_0(A,P,e_{1,\alpha+1}^+)^{p^{m-\alpha}}\\
&\overset{(\ref{congruenceassumption})}{\equiv} f(A,P,e_1^+,e_2^+)^{p^{m-\alpha}}  \pmod{\frak{m}R^+}.
\end{align*}
Since the test object $(A,i,P,e_1^+,e_2^+,u_{\epsilon/p^m}) \in \hat{\mathcal{Y}}^{\mathrm{Ig}}(\epsilon/p^m)^+(R^+,R^+)$ was arbitrary, this proves (\ref{Frobeniuslift}).

Applying 
$$(g^{-(m-\alpha)})^* : \mathbf{\Gamma}(\hat{\mathcal{O}}_{\mathcal{Y}^{\mathrm{Ig}}(\epsilon/p^m)}^+) \rightarrow \mathbf{\Gamma}(\hat{\mathcal{O}}_{\mathcal{Y}^{\mathrm{Ig}}(\epsilon/p^{\alpha})}^+)$$
(the inverse of $(g^{m-\alpha})^*$ from Theorem \ref{perfectoidpropertytheorem}, using $\mathbf{\Gamma}(\mathcal{O}_{\hat{\mathcal{Y}}^{\mathrm{Ig}}(\epsilon/p^m)^+}) \overset{(\ref{integralpowerboundedsame})}{=} \mathbf{\Gamma}(\hat{\mathcal{O}}_{\mathcal{Y}^{\mathrm{Ig}}(\epsilon/p^m)}^+)$) to both sides of (\ref{Frobeniuslift}), we get (\ref{Frobeniuslift2}).

\end{proof}

\subsection{Spreading out $p$-integrality}\label{spreadingoutsection}Throughout this section, we continue to use Convention \ref{Gammaconvention}.

\begin{definition}\label{pintegralitydefinition}\begin{enumerate}
\item Given a formal scheme $\frak{X}$ with adic generic fiber $X$ (i.e. $\frak{X}$ is a formal model of $X$ in the sense of Definition \ref{formalmodeldefinition}), we say that $f \in \mathbf{\Gamma}(\mathcal{O}_X)$ is \emph{$p$-integral (with respect to $\frak{X}$)} if it is in the image of the map $\mathbf{\Gamma}(\mathcal{O}_{\frak{X}}) \rightarrow \mathbf{\Gamma}(\mathcal{O}_X^+)$ from (\ref{integralpowerboundedmap}). 
\item As $\frak{X}$ will be obvious from context, and in fact usually $\frak{X} = \hat{\mathcal{Y}}^{\mathrm{Ig}}(\epsilon)^+$ and $X = \hat{\mathcal{Y}}^{\mathrm{Ig}}(\epsilon)$, then we will often just use the term ``$p$-integral'' and suppress the dependence on $\frak{X}$. This should result in no confusion, particularly given (\ref{integralpowerboundedsame}).
\end{enumerate}
\end{definition}

The following Theorem says that (for appropriate $\epsilon$) given a function $f$ on $\hat{\mathcal{V}}_x$ as in (\ref{hatVz}) which is $p$-integral on the neighborhood $\hat{\mathcal{Y}}^{\mathrm{Ig}}(\epsilon/p^m)$ for some $m \in \mathbb{Z}_{\ge 0}$, and which descends modulo $\frak{m}$ to a function $f_0$ on $Y_1^+(\alpha+1,\epsilon/p^m)$ for some integer $0 \le \alpha \le m$ (see Convention \ref{Y1convention} for the definition of $Y_1^+(\alpha+1,\epsilon/p^m)$), then $f$ is $p$-integral on $\hat{\mathcal{Y}}^{\mathrm{Ig}}(\epsilon/p^{\alpha})$. 

\begin{theorem}\label{crystalline2}Suppose we are in the situation of Theorem \ref{Frobeniuslifttheorem}. That is, $\mathcal{V}_x \subset Y_{\infty}$ is as in (\ref{Vz}), $0 \le \epsilon < p/(p+1)$ is in the valuation group of $\mathcal{O}_k$, and suppose we are given
$$F \in \mathbf{\Gamma}(\hat{\mathcal{O}}_{\mathcal{V}_x}).$$
Let
$$f := F|_{\mathcal{Y}^{\mathrm{Ig}}(\epsilon)} \in \mathbf{\Gamma}(\hat{\mathcal{O}}_{\mathcal{Y}^{\mathrm{Ig}}(\epsilon)})$$
denote the restriction. Suppose further that
$$\mathrm{res}_{\epsilon,\epsilon/p^m}(f) \in \mathbf{\Gamma}(\hat{\mathcal{O}}_{\mathcal{Y}^{\mathrm{Ig}}(\epsilon/p^m)}^+),$$
for some $m \in \mathbb{Z}_{\ge 0}$ and that there is some 
$$f_0 \in \mathbf{\Gamma}(\mathcal{O}_{Y_1^+(\alpha+1,\epsilon/p^m)/\frak{m}})$$
for some integer $0 \le \alpha \le m$ (see Convention \ref{Y1convention} for the definition of $Y_1^+(\alpha+1,\epsilon/p^m)$) such that 
\begin{equation}\label{keyassumption}\mathrm{res}_{\epsilon,\epsilon/p^m}(f) \equiv f_0 \pmod{\frak{m}\mathbf{\Gamma}(\hat{\mathcal{O}}_{\mathcal{Y}^{\mathrm{Ig}}(\epsilon/p^m)}^+)}.
\end{equation}
Then
$$\mathrm{res}_{\epsilon,\epsilon/p^{\alpha}}(f) \in \mathbf{\Gamma}(\hat{\mathcal{O}}_{\mathcal{Y}^{\mathrm{Ig}}(\epsilon/p^{\alpha})}^+).$$
\end{theorem}

%\begin{remark}\label{globallydefinedremark}The assumption that $f$ is the restriction of a function $F$ defined on $\mathcal{V}_x \subset Y_{\infty}$ is needed. (In the proof of Theorem \ref{crystalline2}, we use this assumption in order to apply (\ref{Frobeniuslift2}) from Theorem \ref{Frobeniuslifttheorem}.) If we drop this assumption, we have the following counterexample: consider $F := p^{\epsilon}/\mathrm{Ha}$ where $\mathrm{Ha}$ is a local lift of the Hasse invariant defined on $Y^+(p\epsilon)$ for some sufficiently small $0 < \epsilon < p/(p+1)$ (see Convention \ref{Yconvention} for notation). Then $f := F|_{\hat{\mathcal{Y}}^{\mathrm{Ig}}(p\epsilon)} \in \mathbf{\Gamma}(\mathcal{O}_{\hat{\mathcal{Y}}^{\mathrm{Ig}}(p\epsilon)})$ and $\mathrm{res}_{p\epsilon,\epsilon}(f) \in \mathbf{\Gamma}(\mathcal{O}_{\hat{\mathcal{Y}}^{\mathrm{Ig}}(\epsilon)^+})$, but $f \not\in \mathbf{\Gamma}(\mathcal{O}_{\hat{\mathcal{Y}}^{\mathrm{Ig}}(p\epsilon)^+})$. 

%Note that $p^{\epsilon}/\mathrm{Ha}$ will never be defined on all of $\mathcal{V}_x$, as any lift of the Hasse invariant $\mathrm{Ha}$ defined on all of $Y^+$ will have zeros, and thus $p^{\epsilon}/\mathrm{Ha}$, when pulled back to $\mathcal{V}_x$, will have poles.

%\end{remark}

\begin{proof}%If $m = 0$ there is nothing to prove, so assume $m \ge 1$. We will proceed by induction and show that $\mathrm{res}_{\epsilon,\epsilon/p^{m-1}}(f) \in \mathbf{\Gamma}(\mathcal{O}_{\hat{\mathcal{Y}}^{\mathrm{Ig}}(\epsilon/p^{m-1})^+})$. %If $\mathrm{res}_{\epsilon,\epsilon/p^{m-1}}(f) = 0$, then this inclusion is trivial so assume not.
 By (\ref{gisomorphism}) we have 
\begin{equation}\label{isintegral}(g^{-(m-\alpha)})^*\mathrm{res}_{\epsilon,\epsilon/p^{\alpha}}(f) \in \mathbf{\Gamma}(\mathcal{O}_{\hat{\mathcal{Y}}^{\mathrm{Ig}}(\epsilon/p^{\alpha})^+}).
\end{equation}
%Let $p^{\delta} \in k$ be as in Lemma \ref{pdeltalemma} (1) for $\mathrm{res}_{\epsilon,\epsilon/p^{m-1}}(f) \in \mathbf{\Gamma}(\mathcal{O}_{\hat{\mathcal{Y}}^{\mathrm{Ig}}(\epsilon/p^{m-1})})$. Thus
%\begin{equation}\label{isintegral2}p^{\delta}\cdot\mathrm{res}_{\epsilon,\epsilon/p^{m-1}}(f) \in \mathbf{\Gamma}(\mathcal{O}_{\hat{\mathcal{Y}}^{\mathrm{Ig}}(\epsilon/p^{m-1})^+}).
%\end{equation}
%By minimality of $\delta$, we have 
%$$p^{\delta}\cdot\mathrm{res}_{\epsilon,\epsilon/p^{m-1}}(f) \not\equiv 0 \pmod{\frak{m} \mathbf{\Gamma}(\mathcal{O}_{\hat{\mathcal{Y}}^{\mathrm{Ig}}(\epsilon/p^{m-1})^+})}.$$
Since $0\le \epsilon < p/(p+1)$, we get a congruence 
\begin{equation}\label{Frobeniuslift3}\begin{split} \mathrm{res}_{\epsilon,\epsilon/p^{\alpha}}(f) &\overset{(\ref{Frobeniuslift2})}{\equiv} ((g^{-(m-\alpha)})^*(\mathrm{res}_{\epsilon,\epsilon/p^m}(f)))^{p^{m-\alpha}} \\
&\equiv ((g^{-(m-\alpha)})^*\mathrm{res}_{\epsilon,\epsilon/p^m}(f))^{p^{m-\alpha}} \pmod{\frak{m}\mathbf{\Gamma}(\hat{\mathcal{O}}_{\mathcal{Y}^{\mathrm{Ig}}(\epsilon/p^{\alpha})}^+)}.
\end{split}
\end{equation}
By (\ref{isintegral}) and (\ref{Frobeniuslift3}) we get $\mathrm{res}_{\epsilon,\epsilon/p^{\alpha}}(f) \in \mathbf{\Gamma}(\hat{\mathcal{O}}_{\mathcal{Y}^{\mathrm{Ig}}(\epsilon/p^{\alpha})}^+)$. %Since $\mathcal{V} \subset Y_{\infty}$ is open and $\mathcal{Y}^{\mathrm{Ig}}(\epsilon)$ is affinoid perfectoid (Lemma \ref{Uaffinoidperfectoidlemma}), then $\mathcal{Y}^{\mathrm{Ig}}(\epsilon) \cap \mathcal{V}$ is affinoid perfectoid by \cite[Lemma 4.6]{Scholze}. Thus by Lemma 4.10 (iv) of op. cit., we have $\hat{\mathcal{O}}_{\mathcal{Y}^{\mathrm{Ig}}(\epsilon)}^+(\mathcal{Y}^{\mathrm{Ig}}(\epsilon)\cap \mathcal{V}) = \mathbf{\Gamma}(\hat{\mathcal{O}}_{\mathcal{Y}^{\mathrm{Ig}}(\epsilon) \cap \mathcal{V}}^+)$. 

%Moreover,
%$$\mathbf{\Gamma}(\hat{\mathcal{O}}_{\mathcal{Y}^{\mathrm{Ig}}(\epsilon/p^{m-1})^+}) \overset{(\ref{recallintegralsame})}{=} \mathbf{\Gamma}(\mathcal{O}_{\hat{\mathcal{Y}}^{\mathrm{Ig}}(\epsilon)^+}).$$
%Thus have the following congruence of elements of $\mathbf{\Gamma}(\mathcal{O}_{\hat{\mathcal{Y}}^{\mathrm{Ig}}(\epsilon/p^{m-1})^+})$:
%\begin{align*}0 \not\equiv p^{\delta}\cdot\mathrm{res}_{\epsilon,\epsilon/p^{m-1}}(f) = \mathrm{res}_{\epsilon,\epsilon/p^{m-1}}(p^{\delta}f) &\overset{(\ref{Frobeniuslift3})}{\equiv} ((g^{-1})^*\mathrm{res}_{\epsilon,\epsilon/p^m}(p^{\delta}f))^p \\
%&= p^{p\delta}\cdot ((g^{-1})^*\mathrm{res}_{\epsilon,\epsilon/p^m}(f))^p \pmod{\frak{m}\mathbf{\Gamma}(\mathcal{O}_{\hat{\mathcal{Y}}^{\mathrm{Ig}}(\epsilon/p^{m-1})^+})}.
%\end{align*}
%By (\ref{isintegral}) and the fact that above value is nonzero, we have $\delta \le 0$, which by (\ref{isintegral2}) implies 
%$$\mathrm{res}_{\epsilon,\epsilon/p^{m-1}}(f) \in \mathbf{\Gamma}(\mathcal{O}_{\hat{\mathcal{Y}}^{\mathrm{Ig}}(\epsilon/p^{m-1})^+}).$$

%and so 
%$$g^*(\mathrm{res}_{\epsilon,\epsilon/p^{m-1}}(f)) \in \mathbf{\Gamma}(\mathcal{O}_{\hat{\mathcal{Y}}^{\mathrm{Ig}}(\epsilon/p^m)^+}),$$ which by Theorem \ref{perfectoidpropertytheorem} gives 
%$$\mathrm{res}_{\epsilon,\epsilon/p^{m-1}}(f) \in \mathbf{\Gamma}(\mathcal{O}_{\mathcal{Y}^{\mathrm{Ig}}(\epsilon/p^{m-1})^+}).$$
%This completes the induction.
\end{proof}

\begin{remark}The assumption (\ref{keyassumption}) (which is (\ref{congruenceassumption}) in Theorem \ref{Frobeniuslifttheorem}) will be satisfied in our main application (Theorem \ref{pintegraltheorem2}) because $f$ will satisfy a certain transformation property under the Galois group 
$$\mathrm{Gal}(\left(\mathcal{Y}^{\mathrm{Ig}}(\epsilon/p^{\alpha})\cap Y_{\infty}(\epsilon/p^m)\right)/Y_1(\alpha+1,\epsilon/p^m)) \cong \Gamma_{1,p}(p^{\alpha+1})$$
which allows $f \pmod{\frak{m}}$ to descend to $Y_1(\alpha+1,\epsilon/p^m)$. See (\ref{Gamma1pdefinition}) and (\ref{weightidentity}). 
\end{remark}

We will later use the following Lemma in order to verify that the assumption $\mathrm{res}_{\epsilon,\epsilon/p^m}(f) \in \mathbf{\Gamma}(\mathcal{O}_{\hat{\mathcal{Y}}^{\mathrm{Ig}}(\epsilon/p^m)^+})$ for some $m \in \mathbb{Z}_{\ge 0}$. 

\begin{lemma}\label{findepsilonmlemma}Let $0 \le \epsilon < p/(p+1)$ and that we are given $f \in \mathbf{\Gamma}(\mathcal{O}_{\hat{\mathcal{Y}}^{\mathrm{Ig}}(\epsilon)}) \overset{(\ref{integralpowerboundedsame})}{=} \mathbf{\Gamma}(\hat{\mathcal{O}}_{\mathcal{Y}^{\mathrm{Ig}}(\epsilon)})$ such that
$$f|_{\hat{\mathcal{Y}}^{\mathrm{Ig}}(0)} \in \mathbf{\Gamma}(\mathcal{O}_{\hat{\mathcal{Y}}^{\mathrm{Ig}}(0)^+}) \overset{(\ref{integralpowerboundedsame})}{=} \mathbf{\Gamma}(\hat{\mathcal{O}}_{\mathcal{Y}^{\mathrm{Ig}}(0)}^+).$$
Then there exists $m \in \mathbb{Z}_{\ge 0}$ such that
$$f|_{\hat{\mathcal{Y}}^{\mathrm{Ig}}(\epsilon/p^m)} \in \mathbf{\Gamma}(\mathcal{O}_{\hat{\mathcal{Y}}^{\mathrm{Ig}}(\epsilon/p^m)^+}) \overset{(\ref{integralpowerboundedsame})}{=} \mathbf{\Gamma}(\hat{\mathcal{O}}_{\mathcal{Y}^{\mathrm{Ig}}(\epsilon/p^m)}^+).$$

\end{lemma}

\begin{proof}By assumption, we have $f|_{\hat{\mathcal{Y}}^{\mathrm{Ig}}(0)} \in \mathbf{\Gamma}(\mathcal{O}_{\hat{\mathcal{Y}}^{\mathrm{Ig}}(0)^+})$. Consider the rational subset $\{|f| \le 1\} \subset \hat{\mathcal{Y}}^{\mathrm{Ig}}(\epsilon)$. By Corollary \ref{rationalintersectioncorollary}, there exists $m \in \mathbb{Z}_{\ge 0}$ with $\mathrm{res}_{\epsilon,\epsilon/p^m}(f) \in \mathbf{\Gamma}(\mathcal{O}_{\hat{\mathcal{Y}}^{\mathrm{Ig}}(\epsilon/p^m)^+})$. 

\end{proof}

\subsection{The coordinates $j_i$}\label{jsection}Recall $\mathbb{Y}(\Gamma(1))$, $\mathbb{Y}(\Gamma(1))^+$ and $Y(\Gamma(1))_{\mathcal{O}_k}^+$ from Definition \ref{YGamma1definition} and the affines $V_0$ (defined over $\mathrm{Spec}(\mathbb{Z}_p)$) and $V$ (defined over $\mathrm{Spec}(\mathcal{O}_k)$) from Definition \ref{V0definition}. 
For later arguments, we will need to consider ``j-invariant-like coordinates'' which arise as pullbacks to $Y_{\infty}$ of affine coordinates on the Shimura curve $\mathbb{Y}(\Gamma(1))$. In particular, these coordinates will be used to trivialize the period sheaf $\mathcal{O}\mathbb{B}_{\mathrm{dR}}^+$ of \cite[Section 6]{Scholze} (recalled in Section \ref{periodsheavessection}) over the open neighborhood 
$$U = \mathcal{Y}^{\mathrm{Ig}}(\epsilon_0) \subset Y_{\infty} \in Y_{\text{pro\'{e}t}}$$
from Definition \ref{Udefinition} below; see Section \ref{trivializeOBdRsection}, and Theorem \ref{Utheorem} for the aforementioned trivialization of $\mathcal{O}\mathbb{B}_{\mathrm{dR},U}^+$. When $D = M_2(\mathbb{Q})$ and $\mathbb{Y}(\Gamma(1))$ is the modular j-line, then a single coordinate $j$ suffices which we choose to be the usual j-invariant. 

For any ring $R$, let
$$\mathbb{A}_R^n := \mathrm{Spec}(R[x_1,x_2,\ldots,x_n]).$$
Recall that $V_0$ is a smooth affine curve over $\mathrm{Spec}(\mathbb{Z}_p)$, and thus there is a closed embedding $V_0 \hookrightarrow \mathbb{A}_{\mathbb{Z}_p}^n$ for some $n \ge 1$. 
\begin{definition}\label{ndefinition}
Let $\mathbf{n} \in \mathbb{Z}_{\ge 1}$ be minimal such that there exists such a closed embedding $V_0 \hookrightarrow \mathbb{A}_{\mathbb{Z}_p}^{\mathbf{n}}$. 
\end{definition}
Base changing along $\times_{\mathrm{Spec}(\mathbb{Z}_p)}\mathrm{Spec}(\mathcal{O}_k)$, we get a closed embedding $V \hookrightarrow \mathbb{A}_{\mathcal{O}_k}^{\mathbf{n}}$. Then for all $1 \le i \le \mathbf{n}$, the projection onto the $i^{\mathrm{th}}$ coordinate
\begin{equation}\label{iso0}p_i : V_0 \rightarrow \mathbb{A}_{\mathbb{Z}_p}^1 = \mathrm{Spec}(\mathbb{Z}_p[x_i])
\end{equation}
is not constant (by minimality of $\mathbf{n}$). Let $V_{0,i}$ denote the \'{e}tale locus of $p_i$. 

\begin{proposition}We have
\begin{equation}\label{etalelocusunion}V_0 = \bigcup_{i = 1}^{\mathbf{n}}V_{0,i}.
\end{equation}
\end{proposition}

\begin{proof}Recall that $V_0 \subset \mathbb{A}_{\mathbb{Z}_p}^{\mathbf{n}}$ is an affine smooth curve over $\mathbb{Z}_p$. Then the immersion $V_0 \subset \mathbb{A}_{\mathbb{Z}_p}^{\mathbf{n}}$ is necessarily regular since both $V_0$ and $\mathbb{A}_{\mathbb{Z}_p}^{\mathbf{n}}$ are smooth over $\mathbb{Z}_p$, and thus is defined by some ideal $(f_1,\ldots f_{\mathbf{n}-1}) \subset \mathbb{Z}_p[x_1,\ldots,x_{\mathbf{n}}]$. By the Jacobian criterion of smoothness, the $(\mathbf{n}-1) \times \mathbf{n}$ Jacobian matrix $J = (\frac{\partial f_i}{\partial x_j})_{1 \le i \le \mathbf{n}-1, 1 \le j \le \mathbf{n}}$ has rank $\mathbf{n}-1$ at every point of $V_0$. At any point in $V_0 \setminus \bigcup_{i = 1}^{\mathbf{n}}V_{0,i}$, every $(\mathbf{n}-1) \times (\mathbf{n}-1)$ minor of $J$ vanishes, which means $J$ has rank less than $\mathbf{n}-1$. Thus $V_0 \setminus \bigcup_{i = 1}^{\mathbf{n}}V_{0,i} = \emptyset$, giving the assertion.

\end{proof}

When $D = M_2(\mathbb{Q})$ then $\mathbf{n} = 1$, and in fact we have a canonical inclusion
$$V_0  \subset \mathbb{Y}(\Gamma(1))^+ \times_{\mathrm{Spec}(\mathbb{Z}[1/\mathrm{disc}(D)])}\mathrm{Spec}(\mathbb{Z}_p) \xrightarrow{\sim} \mathbb{A}_{\mathbb{Z}_p}^1 = \mathrm{Spec}(\mathbb{Z}_p[x_1])$$
where $x_1 = j$ is the usual j-invariant. %Recall the point $y_0 \in \mathbb{Y}(\Gamma(1))(\mathbb{Z}_p)$ from Choice \ref{badCMpointchoice}. 
%Let $x_i(y_0) \in \mathbb{Z}_p$ denote the evaluation of $x_i$ at the image of $y_0$ under (\ref{iso0}). 

\begin{definition}Henceforth, for all $1 \le i \le \mathbf{n}$ let 
\begin{equation}\label{zerothj}j_i := x_i|_{V_0} \in \mathcal{O}_{\mathbb{A}_{\mathbb{Z}_p}^{\mathbf{n}}}(V_0).
\end{equation}
In particular, when $D = M_2(\mathbb{Q})$, then $j_1 = j$ is the usual \emph{j}-invariant.
\end{definition}

Base-changing $p_i$ to $\mathrm{Spec}(\mathcal{O}_k)$, we get a non-constant map
\begin{equation}\label{iso1}p_i : V \rightarrow \mathbb{A}_{\mathcal{O}_k}^1 = \mathrm{Spec}(\mathcal{O}_k[j_i])
\end{equation}
%In fact, this map is \'{e}tale and $V \subset \{j \neq 0\} \subset \mathbb{A}_{\mathcal{O}_k}^n$. 
which we continue to denote by $p_i$. Now recall $\widehat{V}$ (a formal scheme over $\mathrm{Spf}(\mathcal{O}_k)$) is the $p$-adic completion of $V$ and let $\widehat{\mathbb{A}}_{\mathcal{O}_k}^1 = \mathrm{Spf}(\mathcal{O}_k\langle j_i\rangle)$ be the $p$-adic completion of $\mathbb{A}_{\mathcal{O}_k}^1 = \mathrm{Spec}(\mathcal{O}_k[j_i])$. Precomposing (\ref{iso1}) with (\ref{avoidy0inclusion}), we get a sequence of maps of formal schemes over $\mathrm{Spf}(\mathcal{O}_k)$
\begin{equation}\label{iso1'}Y^+(\epsilon) \subset \widehat{V} \xrightarrow{p_i} \widehat{\mathbb{A}}_{\mathcal{O}_k}^1 = \mathrm{Spec}(\mathcal{O}_k[j_i]).
\end{equation}

Letting 
$$V^{\mathrm{ad}} \subset Y(\Gamma(1))$$
denote the adic generic fiber of $\widehat{V} \rightarrow \mathrm{Spf}(\mathcal{O}_k)$ and letting $\mathbb{A}_k^{1,\mathrm{ad}} = \mathrm{Spa}(k\langle j_i\rangle, \mathcal{O}_k\langle j_i\rangle)$ denote the adic generic fiber of $\widehat{\mathbb{A}}_k^1$ (both adic spaces over $\mathrm{Spa}(k,\mathcal{O}_k)$), we get an smooth map of adic spaces over $\mathrm{Spa}(k,\mathcal{O}_k)$
\begin{equation}\label{iso2}p_i : V^{\mathrm{ad}} \rightarrow \mathbb{A}_k^{1,\mathrm{ad}}.
\end{equation}
Taking the adic generic fiber of (\ref{iso1'}) we get a sequence of maps in $Y_{\text{pro\'{e}t}}$ defined over $\mathrm{Spa}(k,\mathcal{O}_k)$ which, in a slight abuse of notation, we also denote by $\pi_i$:
\begin{equation}\label{iso3}%j_0 :  Y(\epsilon) \rightarrow V^{\mathrm{ad}} \rightarrow \mathbb{A}_k^{1,\mathrm{ad}}, \hspace{1cm} 
\pi_i : \mathcal{Y}^{\mathrm{Ig}}(\epsilon) \twoheadrightarrow Y(\epsilon) \subset V^{\mathrm{ad}} \xrightarrow{p_i} \mathbb{A}_k^{1,\mathrm{ad}}.
\end{equation}
Here, the first map is pro-finite \'{e}tale, the second map is an open immersion and the third map is smooth.

We also have sequences of morphisms of locally ringed spaces (see Convention \ref{Yconvention} and (\ref{mathbbYbasechange}) for definitions of the objects below)
$$Y_{\infty} \rightarrow Y \rightarrow Y(\Gamma(1)), \hspace{1cm} Y_{\infty}^+\rightarrow Y^+ \rightarrow Y(\Gamma(1))^+ \rightarrow \mathbb{Y}(\Gamma(1))_{\mathcal{O}_k}^+
$$
where in the first sequence, the first morphism is pro-finite \'{e}tale and the second morphism is a finite \'{e}tale map of adic spaces. Pulling back $j_i$ along (\ref{iso1'}), we get a section (which we also denote by $j_i$)
\begin{equation}\label{firstj}j_i \in \mathcal{O}_{Y^+(\epsilon)}(Y^+(\epsilon)).
\end{equation}
%Recalling $V \subset \mathbb{Y}(\Gamma(1))_{\mathcal{O}_k}^+$ and pulling (\ref{firstj}) back along the second map of (\ref{isos}), we get a section (for which we use the same notation)
%$$j \in \mathcal{O}_{Y_{\infty}^+}(V_{\infty})$$
%where is the fiber product of pro-adic spaces
%$$V_{\infty} := V^{\mathrm{ad}} \times_{Y(\Gamma(1))^+}Y_{\infty}^+.$$
%Note that we have inclusions 
%$$\mathcal{Y}^{\mathrm{Ig}}(\epsilon)^+ \subset Y_{\infty}^+(\epsilon) \subset V_{\infty}$$
%of pro-formal schemes over $\mathrm{Spf}(\mathcal{O}_k)$. 
Pulling back along $\mathcal{Y}^{\mathrm{Ig}}(\epsilon)^+ \rightarrow Y^+(\epsilon)$, we get
\begin{equation}\label{finalj}j_i \in \mathbf{\Gamma}(\mathcal{O}_{\mathcal{Y}^{\mathrm{Ig}}(\epsilon)^+}),
\end{equation}
recalling $\mathbf{\Gamma}$ from Convention \ref{Gammaconvention}.

%by pulling back \cite[Theorem III.2.15]{ScholzeTorsion} along the map $\mathcal{Y}^{\mathrm{Ig}}(\epsilon) \rightarrow \mathcal{X}(\epsilon)$, where $\mathcal{X}(\epsilon)$ is as in op. cit.\footnote{Under further assumptions on $\Gamma$ and $D$, this statement can be extended to $\epsilon < 1/(p+1)$. See, for example, \cite[Theorem 3.1.II]{Katzpamf}} 

If $0 \le \epsilon < p/(p+1)$ is in the valuation group of $\mathcal{O}_k$, then by \cite[Proposition 3.10]{GorenKassaei} we have the following congruence of elements in $\mathcal{O}_{\mathcal{Y}^{\mathrm{Ig}}(\epsilon)^+}(\mathcal{Y}^{\mathrm{Ig}}(\epsilon)^+)$ (cf. (\ref{Frobeniuslift})):
\begin{equation}\label{jcongruence}g^*j_i \equiv j_i^p \pmod{p^{1-\epsilon}\mathbf{\Gamma}(\mathcal{O}_{\mathcal{Y}^{\mathrm{Ig}}(\epsilon)^+})}.
\end{equation}
The congruence in (\ref{jcongruence}) follows from the fact that $j_i$ is a pullback from a section on $V \subset Y(\Gamma(1))_{\mathcal{O}_k}^+$, and the action of $g$ on $\mathcal{Y}^{\mathrm{Ig}}(\epsilon)^+$ corresponds to division by the canonical subgroup. We explain the equality in (\ref{jcongruence}) in detail for the reader's convenience: %Recall that 
%$$j_i \in \mathcal{O}_{Y(\Gamma(1))^+}(V_0).$$
%This induces the section 
%$$j \in \mathcal{O}_{Y(\Gamma(1))^+}(V) \rightarrow \mathcal{O}_{Y(\Gamma(1))_{\mathcal{O}_k}^+}(V) \overset{(\ref{avoidy0inclusion})}{\rightarrow} \mathcal{O}_{Y(\Gamma(1))_{\mathcal{O}_k}^+}(Y(\Gamma(1))^+(\epsilon)) \rightarrow \mathcal{O}_{\mathcal{Y}^{\mathrm{Ig}}(\epsilon)^+}(\mathcal{Y}^{\mathrm{Ig}}(\epsilon)^+)$$
%in (\ref{jcongruence}), where the first arrow denotes base change to $\mathrm{Spf}(\mathcal{O}_k)$. 
Recall by (\ref{zerothj}) that $j_i$ is a $p$-adic modular form on $V_0$ over $\mathbb{Z}_p$ of weight 0 in the sense of \cite{Katzpamf} and \cite{Kassaei}. Let $\widehat{V}/p^{1-\epsilon}$ be defined as in Definition \ref{reductionformalschemedefinition} for the formal scheme $\widehat{V} \rightarrow \mathrm{Spf}(\mathcal{O}_k)$. Let 
$$\mathbf{F} : \widehat{V}/p^{1-\epsilon} \rightarrow \widehat{V}/p^{1-\epsilon}\times_{\mathrm{Spec}(\mathcal{O}_k/p^{1-\epsilon}),\mathrm{Frob}}\mathrm{Spec}(\mathcal{O}_k/p^{1-\epsilon})$$
denote relative Frobenius (where $\mathrm{Frob} : \mathrm{Spec}(\mathcal{O}_k/p^{1-\epsilon}) \rightarrow \mathrm{Spec}(\mathcal{O}_k/p^{1-\epsilon})$ denotes absolute $p$-power Frobenius). Then since $\mathbf{S}_0$ from (\ref{badCMpointchoice}) is defined over $\mathbb{Z}_p$ and $V_0 = Y(\Gamma(1))^+ \setminus \mathbf{S}_0$ (see Definition \ref{V0definition}) and $V = V_0 \times_{\mathrm{Spec}(\mathbb{Z}_p)} \mathrm{Spec}(\mathcal{O}_k)$, $\widehat{V}/p^{1-\epsilon}$ is preserved by $\mathbf{F}$. Thus, for any $p$-adically complete $\mathcal{O}_k$-algebra $S$ and test object $(A) \in \widehat{V}(S)$ (i.e. $(A)$ is the isomorphism class of a (false) elliptic curve), we have
\begin{align*}(\mathbf{F}^*j_i)(A \otimes_S S/p^{1-\epsilon}) &= j_i(A \otimes_SS/p^{1-\epsilon}\otimes_{S/p^{1-\epsilon},\mathrm{Frob}}S/p^{1-\epsilon}) \\
&= j_i(A)\otimes_S1 \otimes_{S/p^{1-\epsilon},\mathrm{Frob}}1\\
&\equiv 1 \otimes_S 1 \otimes_{S/p^{1-\epsilon},\mathrm{Frob}} j_i(A)^p = j_i(A)^p \pmod{p^{1-\epsilon}S}
\end{align*}
where the second equality follows from the base-change compatibility of $p$-adic modular forms in \cite{Katzpamf} and \cite{Kassaei}. On $\mathcal{Y}^{\mathrm{Ig}}(\epsilon)^+$, the kernel of the isogeny underlying $g$ in the moduli interpretation is the canonical subgroup, and so on $\mathcal{Y}^{\mathrm{Ig}}(\epsilon)^+/p^{1-\epsilon}$ (using the notation of Definition \ref{reductionformalschemedefinition} for the formal scheme $\mathcal{Y}^{\mathrm{Ig}}(\epsilon)^+ \rightarrow \mathrm{Spf}(\mathcal{O}_k)$), the isogeny underlying $g$ is relative Frobenius. Hence we have 
$$g^*j_i \equiv  \mathbf{F}^*j_i \pmod{p^{1-\epsilon}\mathbf{\Gamma}(\mathcal{O}_{\mathcal{Y}^{\mathrm{Ig}}(\epsilon)^+})}$$
when viewing $j_i \in \mathbf{\Gamma}(\mathcal{O}_{\mathcal{Y}^{\mathrm{Ig}}(\epsilon)^+})$. Now (\ref{jcongruence}) follows from the previous two displayed congruences.

% \cite[Theorem 3.1.I]{Katzpamf} with the relative Frobenius on $Y(1)^+$ (denoted by $\mathrm{Frob}$):
%\begin{equation}\label{Frobenii}\begin{tikzcd}[column sep =small]
%\mathcal{Y}^{\mathrm{Ig}}(\epsilon)^+/p^{1-\epsilon} \arrow{d} \arrow{r}{g} & \mathcal{Y}^{\mathrm{Ig}}(p\epsilon)^+/p^{1-\epsilon}\arrow{d} \\
 %    Y(1)^+/p^{1-\epsilon} \arrow{r}{\mathrm{Frob}} &Y(1)^+/p^{1-\epsilon} 
  %   \end{tikzcd}.
   %  \end{equation}
%Recall that the bottom arrow is defined by $\mathrm{Frob}^*(j) = j^p$.  In particular, considering the element
%$$j \in \hat{\mathcal{O}}_{\mathcal{Y}^{\mathrm{Ig}}(p\epsilon)^+}(\mathcal{Y}^{\mathrm{Ig}}(p\epsilon)^+) \subset \hat{\mathcal{O}}_{\mathcal{Y}^{\mathrm{Ig}}(\epsilon)^+}(\mathcal{Y}^{\mathrm{Ig}}(\epsilon)^+),$$
%we have, by (\ref{Frobenii}), 

\subsection{The \'{e}tale loci $\mathcal{Y}^{\mathrm{Ig}}(\epsilon)_i$}Finally, for later purposes, we will need to consider the \'{e}tale locus of (\ref{iso3}) for each $1 \le i \le \mathbf{n}$ ($\mathbf{n}$ as in Definition \ref{ndefinition}).  
\begin{definition}\label{etalelocusdefinition}Let 
$$\mathcal{Y}^{\mathrm{Ig}}(\epsilon)_i \subset \mathcal{Y}^{\mathrm{Ig}}(\epsilon)$$
denote the \'{e}tale locus of $\pi_i$ from (\ref{iso3}). In particular, $\mathcal{Y}^{\mathrm{Ig}}(\epsilon)_i \subset \mathcal{Y}^{\mathrm{Ig}}(\epsilon)$ is open. Note that when $D = M_2(\mathbb{Q})$ then $\mathbf{n} = 1$ and (\ref{iso3}) is \'{e}tale, so we simply have $\mathcal{Y}^{\mathrm{Ig}}(\epsilon)_1 = \mathcal{Y}^{\mathrm{Ig}}(\epsilon)$. 

%Let $\mathcal{Y}^{\mathrm{Ig}}(\epsilon)_j$ denote the pullback of $Y(\epsilon)_j \subset Y(\epsilon)$ under $\mathcal{Y}^{\mathrm{Ig}}(\epsilon) \twoheadrightarrow Y(\epsilon)$. 
\end{definition}

\begin{proposition}We have 
\begin{equation}\label{etalelocusunion2}\mathcal{Y}^{\mathrm{Ig}}(\epsilon) = \bigcup_{i = 1}^{\mathbf{n}}\mathcal{Y}^{\mathrm{Ig}}(\epsilon)_i.
\end{equation}
\end{proposition}

\begin{proof}This follows immediately from (\ref{etalelocusunion}) and (\ref{iso3}).

\end{proof}

\begin{remark}Note that in the modular curve case $D = M_2(\mathbb{Q})$, $\mathbf{n} = 1$ in the context of Definition \ref{ndefinition}. Hence $i = 1$ and we can ignore this subscript in our discussion in this case (see also Remark \ref{Uiremark}). 
\end{remark}

\section{Period Sheaves on Infinite-Level Shimura Curves}\label{ShimuraCurveSection}

\subsection{Period sheaves}\label{periodsheavessection} For the background concepts and details of the constructions recalled in this section, we refer to \cite[Section 4, Section 6]{Scholze} and \cite[(3)]{Scholzecorrigendum}.
%Let $Y$ be a local or global modular curve. Let $Y_{\text{pro\'{e}t}}$ denote the pro\'{e}tale site as in op. cit. 
%Let $\mathcal{O}_Y$ denote the pro\'{e}tale structure sheaf on $Y_{\text{pro\'{e}t}}$, let $\mathcal{O}_Y^+ \subset \mathcal{O}_Y$ denote the integral subsheaf, let
%$$\hat{\mathcal{O}}_Y^+ := \varprojlim_n \mathcal{O}_Y^+/p^n$$
%denote the $p$-adic completion, and $\hat{\mathcal{O}}_Y = \hat{\mathcal{O}}_Y^+[1/p]$. In particular, we have natural $p$-adic completion maps
%$$\mathcal{O}_Y^+ \rightarrow \hat{\mathcal{O}}_Y, \hspace{1cm} \mathcal{O}_Y \rightarrow \hat{\mathcal{O}}_Y.$$ 
Retain the notation of Convention \ref{Yconvention}. Let $\mathbb{B}_{\mathrm{dR},Y}^+$ denote the relative de Rham period sheaf on $Y_{\text{pro\'{e}t}}$, and let $\mathcal{O}\mathbb{B}_{\mathrm{dR},Y}^+$ denote the de Rham structure sheaf, which is equipped with inclusions
$$\mathcal{O}_Y \subset \mathcal{O}\mathbb{B}_{\mathrm{dR},Y}^+, \hspace{1cm} \mathbb{B}_{\mathrm{dR},Y}^+ \subset \mathcal{O}\mathbb{B}_{\mathrm{dR},Y}^+.$$
Let
\begin{equation}\label{'thetas}\theta : \mathbb{B}_{\mathrm{dR},Y}^+ \rightarrow \hat{\mathcal{O}}_Y, \hspace{1cm} \theta_{\mathrm{dR}} : \mathcal{O}\mathbb{B}_{\mathrm{dR},Y}^+ \rightarrow \hat{\mathcal{O}}_Y
\end{equation}
be the natural projections defined in \cite[Section 6]{Scholze}. Then we have natural filtrations 
$$\mathrm{Fil}^i\mathbb{B}_{\mathrm{dR},Y}^+ = (\ker\theta)^i\mathbb{B}_{\mathrm{dR},Y}^+, \hspace{1cm} \mathrm{Fil}^i\mathcal{O}\mathbb{B}_{\mathrm{dR},Y}^+ = (\ker\theta_{\mathrm{dR}})^i\mathcal{O}\mathbb{B}_{\mathrm{dR},Y}^+,$$
where by convention $(\ker\theta)^i = (1)$ and $(\ker\theta_{\mathrm{dR}})^i = (1)$ if $i \le 0$. Also let $\mathcal{O}\mathbb{B}_{\mathrm{dR},Y} = \mathcal{O}\mathbb{B}_{\mathrm{dR},Y}^+[1/t]$ as in loc. cit., where $t$ denotes a local generator of $\mathrm{Fil}^1\mathbb{B}_{\mathrm{dR},Y}^+$. (Later, we will define a distinguished generator $t \in \mathrm{Fil}^1\mathbb{B}_{\mathrm{dR},Y}^+(Y_{\infty})$, see (\ref{'tdefinition}).) Then we have a natural filtration
$$\mathrm{Fil}^i\mathcal{O}\mathbb{B}_{\mathrm{dR},Y} = \sum_{j \in \mathbb{Z}}\left(\mathrm{Fil}^{i+j}\mathcal{O}\mathbb{B}_{\mathrm{dR},Y}^+\right)\cdot t^{-j}.$$

Moreover, the composition
$$\mathbb{B}_{\mathrm{dR},Y}^+ \subset \mathcal{O}\mathbb{B}_{\mathrm{dR},Y}^+ \xrightarrow{\theta_{\mathrm{dR}}} \hat{\mathcal{O}}_Y$$
is equal to $\theta: \mathbb{B}_{\mathrm{dR},Y}^+ \twoheadrightarrow \hat{\mathcal{O}}_Y$, and the composition
$$\mathcal{O}_Y \subset \mathcal{O}\mathbb{B}_{\mathrm{dR},Y}^+ \xrightarrow{\theta_{\mathrm{dR}}} \hat{\mathcal{O}}_Y$$
is the natural $p$-adic completion map.

Let $\Omega_Y$ be the sheaf of K\"{a}hler differentials on $Y_{\text{pro\'{e}t}}$. We also have a connection 
$$\nabla : \mathcal{O}\mathbb{B}_{\mathrm{dR},Y}^+ \rightarrow \mathcal{O}\mathbb{B}_{\mathrm{dR},Y}^+ \otimes_{\mathcal{O}_Y} \Omega_Y.$$
By construction of $\nabla$, the inclusion
$$\mathcal{O}_Y \subset \mathcal{O}\mathbb{B}_{\mathrm{dR},Y}^+ \xrightarrow{\nabla} \mathcal{O}\mathbb{B}_{\mathrm{dR},Y}^+ \otimes_{\mathcal{O}_Y} \Omega_Y$$
factors through the natural derivation $d : \mathcal{O}_Y \rightarrow \Omega_Y$. We have that $\mathbb{B}_{\mathrm{dR},Y}^+$ is the sheaf of horizontal sections of $\nabla$:
$$\mathcal{O}\mathbb{B}_{\mathrm{dR},Y}^{+,\nabla = 0} = \mathbb{B}_{\mathrm{dR},Y}^+.$$

\subsection{Trivializing $\mathcal{O}\mathbb{B}_{\mathrm{dR}}^+$ on $\mathcal{Y}^{\mathrm{Ig}}(\epsilon)$} \label{trivializeOBdRsection}

Recall we view $Y$ over $\mathrm{Spa}(k,\mathcal{O}_k)$ per Conventions \ref{basechangeconvention} and \ref{Yconvention}. Define the tilted structure sheaf on $Y_{\text{pro\'{e}t}}$
$$\hat{\mathcal{O}}_{Y}^{+,\flat} := \varprojlim_{x \mapsto x^p}\hat{\mathcal{O}}_{Y}^+/p^{1-\epsilon}.$$
%We have a map $\hat{\mathcal{O}}_{Y^+}^{\flat} \rightarrow \hat{\mathcal{O}}_{Y}^{+,\flat}$. 
 It is easily checked that $\hat{\mathcal{O}}_{Y}^{+,\flat}$ coincides with the sheaf $\mathcal{O}_{Y^{\flat}}^+$ from \cite[Definition 5.9]{Scholze}. Given $W \in Y_{\text{pro\'{e}t}}$ we will let $\hat{\mathcal{O}}_W^{+,\flat} := \hat{\mathcal{O}}_Y^{+,\flat}|_W$. Now suppose $0 \le \epsilon < p/(p+1)$. Since $\mathcal{Y}^{\mathrm{Ig}}(\epsilon/p^n) \subset \mathcal{Y}^{\mathrm{Ig}}(\epsilon)$, applying $g^n$ we have 
$$\mathcal{Y}^{\mathrm{Ig}}(\epsilon) \overset{(\ref{gisomorphism})}{=} \mathcal{Y}^{\mathrm{Ig}}(\epsilon/p^n) \cdot g^n \subset \mathcal{Y}^{\mathrm{Ig}}(\epsilon) \cdot g^n$$
which gives a map $\mathbf{\Gamma}(\hat{\mathcal{O}}_{\mathcal{Y}^{\mathrm{Ig}}(\epsilon) \cdot g^n}^+) \rightarrow \mathbf{\Gamma}(\hat{\mathcal{O}}_{\mathcal{Y}^{\mathrm{Ig}}(\epsilon)}^+)$. Thus via pullback by $g^{-n}$ we get a map
$$(g^{-n})^* : \mathbf{\Gamma}(\hat{\mathcal{O}}_{\mathcal{Y}^{\mathrm{Ig}}(\epsilon)}^+) \xrightarrow{\sim} \mathbf{\Gamma}(\hat{\mathcal{O}}_{\mathcal{Y}^{\mathrm{Ig}}(\epsilon) \cdot g^n}^+) \rightarrow \mathbf{\Gamma}(\hat{\mathcal{O}}_{\mathcal{Y}^{\mathrm{Ig}}(\epsilon)}^+).$$
Recall $j_i \in \mathbf{\Gamma}(\hat{\mathcal{O}}_{\mathcal{Y}^{\mathrm{Ig}}(\epsilon)^+})$ from (\ref{finalj}). Therefore, for every $n \ge 0$, we have 
$$g^{-n,*}j_i = g^{-(n+1),*}(g^*j_i) \overset{(\ref{jcongruence})}{\equiv} g^{-(n+1),*}j_i^p = (g^{-(n+1),*}j_i)^p \pmod{p^{1-\epsilon}\mathbf{\Gamma}(\hat{\mathcal{O}}_{\mathcal{Y}^{\mathrm{Ig}}(\epsilon)}^+)}.$$
Hence we have a well-defined section
\begin{equation}\label{jflat}j_i^{\flat} := (j_i,(g^{-1})^*j_i,(g^{-2})^*j_i,\ldots) \in \mathbf{\Gamma}(\hat{\mathcal{O}}_{\mathcal{Y}^{\mathrm{Ig}}(\epsilon)\cdot}^{+,\flat}),
\end{equation}
whose Teichm\"{u}ller lift gives a section 
$$[j_i^{\flat}] \in \mathbf{\Gamma}(\mathbb{B}_{\mathrm{dR},\mathcal{Y}^{\mathrm{Ig}}(\epsilon)}^+).$$

We now introduce a special open subset $U \subset Y_{\infty}$ which will recur throughout the rest of the paper.

\begin{definition}\label{Udefinition}Fix 
$$0 < \epsilon_0 < p/(p+1)$$
in the valuation group of $\mathcal{O}_k$ (see Definition \ref{kdefinition}). For every $1 \le i \le \mathbf{n}$ (where $\mathbf{n}$ is as in Definition \ref{ndefinition}), define
$$U := \mathcal{Y}^{\mathrm{Ig}}(\epsilon_0), \hspace{1cm} U_i := \mathcal{Y}^{\mathrm{Ig}}(\epsilon_0)_i$$
where $\mathcal{Y}^{\mathrm{Ig}}(\epsilon_0)_i$ is as in Definition \ref{etalelocusdefinition} with $\epsilon = \epsilon_0$. In particular, (\ref{etalelocusunion2}) implies 
\begin{equation}\label{etalelocusunion3}U = \bigcup_{i = 1}^{\mathbf{n}}U_i.
\end{equation} 
\end{definition}

\begin{remark}Throughout the rest of the text, $\epsilon_0$ will denote the choice fixed in Definition \ref{Udefinition}, whereas $\epsilon, \epsilon'$ will denote general elements in the valuation group of $\mathcal{O}_k$.
\end{remark}

\begin{remark}\label{Uiremark}The sets $U_i$ and subsets thereof (such as those defined in (\ref{'U'definition})) will appear ubiquitously in the ensuing discussion until Lemma \ref{thetaXGamma0commutelemma}. In order to avoid becoming bogged down in notation, on a first reading the reader might specialize to the modular curve case $D = M_2(\mathbb{Q})$ in which we have $\mathbf{n} = 1$ and $U = U_1$. Thus one may eliminate the subscripts ``$i$'' entirely from the discussion in this case. The arguments in the general $D$ case are nearly the same except for the occasional need to glue certain objects initially defined on each $U_i$ to get objects on $U$. (See Section \ref{gluingsection}.)

\end{remark}

\begin{theorem}\label{Utheorem}
For every $1 \le i \le \mathbf{n}$ (where $\mathbf{n}$ is as in Definition \ref{ndefinition}), we have an isomorphism of sheaves on $Y_{\text{pro\'{e}t}}/U_i$
\begin{equation}\label{Utriv}\mathcal{O}\mathbb{B}_{\mathrm{dR},U_i}^+ \cong \mathbb{B}_{\mathrm{dR},U_i}^+\llbracket X_i\rrbracket
\end{equation}
sending $j_i \mapsto [j_i^{\flat}] + X_i$. This isomorphism is compatible with connections and filtrations (giving $X_i$ filtration degree 1).
\end{theorem}

\begin{proof}We will follow the strategy of \cite[proof of Proposition 6.10]{Scholze}. We have an obvious map
\begin{equation}\label{easymap}\mathbb{B}_{\mathrm{dR},U_i}^+\llbracket X_i\rrbracket \rightarrow \mathcal{O}\mathbb{B}_{\mathrm{dR},U_i}^+
\end{equation}
given by $X_i \mapsto j_i - [j_i^{\flat}]$ and extending $\mathbb{B}_{\mathrm{dR},U_i}^+$-linearly; note that this is well-defined since $X_i \in \ker\theta_{\mathrm{dR}}$ and $\mathcal{O}\mathbb{B}_{\mathrm{dR},U_i}^+$ is $\ker\theta_{\mathrm{dR}}$-adically complete. 

Now we construct the inverse map. Let 
$$\theta_{X_i} : \mathbb{B}_{\mathrm{dR},U_i}^+\llbracket X_i\rrbracket \rightarrow \mathbb{B}_{\mathrm{dR},U_i}^+$$
be reduction modulo $(X_i)$. As in \cite[proof of Proposition 6.10]{Scholze}, by the definition of $\mathcal{O}\mathbb{B}_{\mathrm{dR}}^+$ (Section 6 of op. cit. and \cite{Scholzecorrigendum}) it suffices to show that there is a unique map
\begin{equation}\label{desiredOmap}\mathcal{O}_{U_i} \rightarrow \mathbb{B}_{\mathrm{dR},U_i}^+\llbracket X_i\rrbracket
\end{equation}
whose composition with $\theta \circ \theta_{X_i} : \mathbb{B}_{\mathrm{dR},U_i}^+ \rightarrow \hat{\mathcal{O}}_{U_i}$ is the natural $p$-adic completion map $\mathcal{O}_{U_i} \rightarrow \hat{\mathcal{O}}_{U_i}$. 

Let $\mathbb{A}_{\mathbb{Q}_p}^{1,\mathrm{ad}} = \mathrm{Spa}(\mathbb{Q}_p\langle j_i\rangle, \mathbb{Z}_p\langle j_i\rangle)$, and note that we have a map
$$\mathcal{O}_{\mathbb{A}_{\mathbb{Q}_p}^{1,\mathrm{ad}}}(\mathbb{A}_{\mathbb{Q}_p}^{1,\mathrm{ad}}) \rightarrow \mathbb{B}_{\mathrm{dR},U_i}^+\llbracket X_i\rrbracket$$
given by $j_i \mapsto [j_i^{\flat}] + X_i$, where $j_i^{\flat}$ is defined in (\ref{jflat}). 
%%Let us briefly view $Y$ and $Y(\Gamma(1))$ as adic spaces over $\mathrm{Spa}(\mathbb{Q}_p,\mathbb{Z}_p)$ again, and let $Y_k$, $Y(\Gamma(1))_k$, etc. denote the base changes to $k$. 
% Let $V_0, V = V_0 \times_{\mathrm{Spec}(\mathbb{Z}_{(p)})} \mathrm{Spec}(\mathcal{O}_k)$ be as Definition \ref{V0definition}. 
%%$$V \subset \mathbb{Y}(\Gamma(1))_S^+ = \mathbb{Y}(\Gamma(1))^+ \otimes_{\mathbb{Z}[1/(N\cdot\mathrm{disc}(D))]}S.$$
%Let $V_0^{\mathrm{ad}}$ be the adic space associated to the rigid analytification of $V_0 \times_{\mathrm{Spec}(\mathbb{Z}_{(p)})}\mathrm{Spec}(\mathbb{Q}_p)$. We get a map %(which is an \'{e}tale map of adic spaces by (\ref{iso2}))
%$$V_0^{\mathrm{ad}} \rightarrow \mathbb{A}_{\mathbb{Q}_p}^{1,\mathrm{ad}} := \bigcup_{n \in \mathbb{Z}_{\ge 0}}\mathrm{Spa}(\mathbb{Q}_p\langle p^nj\rangle, \mathbb{Z}_p\langle p^nj\rangle).$$
%Let $W_0^{\mathrm{ad}}$ denote the \'{e}tale locus of this map; in particular, $W_0^{\mathrm{ad}} \subset V_0^{\mathrm{ad}}$ is open. Let $W^{\mathrm{ad}} = W_0^{\mathrm{ad}} \times_{\mathrm{Spa}(\mathbb{Q}_p,\mathbb{Z}_p)}\mathrm{Spa}(k,\mathcal{O}_k)$. Then $W^{\mathrm{ad}} \subset V^{\mathrm{ad}}$ is open where $V^{\mathrm{ad}}$ is the adic space associated to the rigid analytification of $V \times_{\mathrm{Spec}(\mathcal{O}_k)}\mathrm{Spec}(k)$. Moreover, $W^{\mathrm{ad}}$ is the \'{e}tale locus of 
%$$V^{\mathrm{ad}} \rightarrow V_0^{\mathrm{ad}} \rightarrow \mathbb{A}_{\mathbb{Q}_p}^{1,\mathrm{ad}}.$$
%Let $U_0 = U\cap W^{\mathrm{ad}}$, so that $U_0 \subset U$ is an open subset. 

Suppose $U' = \varprojlim_m U_m' \in Y_{\text{pro\'{e}t}}/U_i$ is any affinoid perfectoid. We will construct the evaluation of the map (\ref{desiredOmap}) at any such affinoid perfectoid $U'$, and that such maps naturally glue over all affinoid perfectoid $U' \in Y_{\text{pro\'{e}t}}/U_i$. Then since affinoid perfectoids form a basis of $Y_{\text{pro\'{e}t}}/U_i$ by \cite[Proposition 4.8]{Scholze}, we get the sheaf map (\ref{desiredOmap}). 

%Let $U_0' = U_0 \times_U U'$, so that $U_0' \subset U'$ is an open subset. %Since $U'$ is affinoid perfectoid, then $U_0'$ is perfectoid by \cite[Lemma 4.6]{Scholze}. 
%Write $U_0'  = \varprojlim_i U_{0i}'$ where $U_{0i}' = U_0 \times_U U_i'$, so that $U_{0i}' \subset U_i'$ is an open subset. Then via the composition 
%$$U_0' \subset U_0 \subset W^{\mathrm{ad}} \rightarrow W_0^{\mathrm{ad}} \rightarrow \mathbb{A}_{\mathbb{Q}_p}^{1,\mathrm{ad}},$$
%where the first two maps are open immersions, the third map is pro-finite \'{e}tale (as it is given by the projection $W^{\mathrm{ad}} = W_0^{\mathrm{ad}} \times_{\mathrm{Spa}(\mathbb{Q}_p,\mathbb{Z}_p)}\mathrm{Spa}(k,\mathcal{O}_k) \rightarrow W_0^{\mathrm{ad}}$ which is a base change along $W_0^{\mathrm{ad}} \rightarrow \mathrm{Spa}(\mathbb{Q}_p,\mathbb{Z}_p)$ of the pro-finite \'{e}tale map $\mathrm{Spa}(k,\mathcal{O}_k) \rightarrow \mathrm{Spa}(\mathbb{Q}_p,\mathbb{Z}_p)$) and the fourth map is \'{e}tale. Hence we have $U_0' \in (\mathbb{A}_{\mathbb{Q}_p}^{1,\mathrm{ad}})_{\text{pro\'{e}t}}$. Moreover $\mathcal{O}_Y(U') = \varinjlim_n \mathcal{O}_Y(U_i')$ and $\mathcal{O}_Y(U_0') = \varinjlim_n \mathcal{O}_Y(U_{0i}')$. 

We will need the following analogue of \cite[Lemma 6.13]{Scholze}. %Let $\mathbf{k}$ denote the residue field of $S$. 

\begin{lemma}\label{intermediatelemma}Suppose $\mathrm{Spa}(R,R^+)$ is an affinoid adic space of finite type which is defined over $\mathrm{Spa}(\mathbb{Q}_p,\mathbb{Z}_p)$ and equipped with an \'{e}tale map $\mathrm{Spa}(R,R^+) \rightarrow \mathrm{Spa}(\mathbb{Q}_p\langle j_i\rangle,\mathbb{Z}_p\langle j_i \rangle)$ of adic spaces over $\mathrm{Spa}(\mathbb{Q}_p,\mathbb{Z}_p)$. Then there exists a finitely generated $\mathbb{Z}_p[j_i]$-algebra $R_0^+$ such that $R_0 = R_0^+[1/p]$ is \'{e}tale over $\mathbb{Q}_p[j_i]$
and $R^+$ is the $p$-adic completion of $R_0^+$.  
\end{lemma}

\begin{proof}[Proof of Lemma \ref{intermediatelemma}]We follow the proof of Lemma 6.13 of op. cit. almost verbatim. One uses \cite[Corollary 1.7.3 (iii)]{Huber} to construct $(R_0,R_0^+)$. It remains to see that $R_0^+$ is a finitely generated $\mathbb{Z}_p[j_i]$-algebra. Remark 1.2.6 (iii) of op. cit. shows that it is the integral closure of a finitely generated $\mathbb{Z}_p[j_i]$-algebra $S_0^+ \subset R_0^+$, such that $S_0^+ \subset R_0$ and $S_0^+[1/p] = R_0$. But $\mathbb{Z}_p$ is excellent, and so for any reduced, flat, finitely generated $\mathbb{Z}_p$-algebra $S^+$, the normalization of $S^+$ in $S^+[1/p]$ is finite over $S^+$, which gives the desired finite-generatedness. 
\end{proof}

%Recall $U_{0i}' \subset U_i' \rightarrow \mathbb{A}_{\mathbb{Q}_p}^{1,\mathrm{ad}} = \bigcup_{n \in \mathbb{Z}_{\ge 0}}\mathrm{Spa}(\mathbb{Q}_p\langle j\rangle,\mathbb{Z}_p\langle j\rangle)$, and 
%Write
%$$U_{m,n}' = U_m' \times_{\mathbb{A}_{\mathbb{Q}_p}^{1,\mathrm{ad}}}\mathrm{Spa}(\mathbb{Q}_p\langle p^nj_i\rangle, \mathbb{Z}_p\langle p^nj_i\rangle).$$
Applying Lemma \ref{intermediatelemma} to the \'{e}tale map $U_m' \rightarrow \mathrm{Spa}(\mathbb{Q}_p\langle j_i\rangle, \mathbb{Z}_p\langle j_i\rangle)$, we get algebras $R_{m,0}^+$ such that $R_{m,0} = R_{m,0}^+[1/p]$ is \'{e}tale over $\mathbb{Q}_p[j_i]$, and $\mathcal{O}_U^+(U_{m}')$ is the $p$-adic completion of $R_{m,0}^+$. Applying Hensel's lemma, we get a $\mathbb{Q}_p[j_i]$-linear map 
$$R_{m,0} \rightarrow \mathbb{B}_{\mathrm{dR},U_i}^+(U')\llbracket X_i\rrbracket$$
sending $j_i \mapsto [j_i^{\flat}] + X_i$. %, and such that the composition
%$$R_{0i,n,0} \rightarrow \mathbb{B}_{\mathrm{dR},U}^+(U')\llbracket X\rrbracket \xrightarrow{\theta \circ \theta_X} \hat{\mathcal{O}}_U(U')$$
%is the natural map induced by $p$-adic completion $R_{0i,n,0} \rightarrow \mathcal{O}_U(U_i') \rightarrow \mathcal{O}_U(U')$. 
This restricts to a $\mathbb{Z}_p[j_i]$-linear map $R_{m,0}^+ \rightarrow \mathbb{B}_{\mathrm{dR},U_i}^+(U')\llbracket X_i\rrbracket$. This extends by $p$-adic completion (cf. \cite[Lemma 6.11]{Scholze}) to a $\mathbb{Z}_p\langle j_i\rangle$-linear map 
\begin{equation}\label{previousdisplayedmap}\mathcal{O}_{U_i}^+(U_{m}') \rightarrow \mathbb{B}_{\mathrm{dR},U_i}^+(U')\llbracket X_i\rrbracket
\end{equation}
sending $j_i \mapsto [j_i^{\flat}] + X_i$. %Recall the map $U_{m,n}' \rightarrow \mathbb{A}_{\mathbb{Q}_p}^{1,\mathrm{ad}}$ where the first map is an open immersion of adic spaces over $\mathbb{A}_{\mathbb{Q}_p}^{1,\mathrm{ad}}$, and the map $U_{m,n}' \rightarrow \mathbb{A}_{\mathbb{Q}_p}^1$ induces the $\mathbb{Z}_p\langle p^nj_i\rangle$-algebra structure $\mathbb{Z}_p\langle p^nj_i \rangle \rightarrow \mathcal{O}_U^+(U_{0i,n}')$. Thus precomposing (\ref{previousdisplayedmap}) with the natural $\mathbb{Z}_p\langle p^nj\rangle$-linear restriction $\mathcal{O}_U^+(U_{i,n}') \rightarrow \mathcal{O}_U^+(U_{0i,n}')$, we get a $\mathbb{Z}_p\langle p^nj\rangle$-linear map
%$$\mathcal{O}_U^+(U_{i,n}') \rightarrow \mathbb{B}_{\mathrm{dR},U}^+(U_0')\llbracket X\rrbracket.$$
This gives a $\mathbb{Q}_p\langle j_i\rangle$-linear map 
$$\mathcal{O}_{U_i}(U_{m}') \rightarrow \mathbb{B}_{\mathrm{dR},U_i}^+(U')\llbracket X_i\rrbracket$$
sending $j_i \mapsto [j_i^{\flat}] + X_i$. 
Taking the direct limit over $m$ and $n$, we get a map 
\begin{equation}\label{affinoiddesiredOmap}\mathcal{O}_{U_i}(U') \rightarrow \mathbb{B}_{\mathrm{dR},U_i}^+(U')\llbracket X_i\rrbracket
\end{equation}
sending $j_i \mapsto [j_i^{\flat}] + X_i$ and whose composition with $\theta \circ \theta_{X_i}$ is the natural map $\mathcal{O}_{U_i}(U') \rightarrow \hat{\mathcal{O}}_{U_i}(U')$. By uniqueness of the maps (\ref{affinoiddesiredOmap}) satisfying the above properties, these maps glue over all affinoid perfectoids $U' \in Y_{\text{pro\'{e}t}}/U_i$. This finishes the proof.

\end{proof}

%Recall $\mathcal{V}_x = \{z_{\mathrm{HT}} \neq 0\} \subset Y_{\infty}$ be the affinoid open from (\ref{VxVy}). The following Corollary of (\ref{Utriv}) will be used in order to define coefficients of $q_{\mathrm{dR}}$-expansions of normalized algebraic modular forms (see Definition \ref{'zqexpansions}) as functions on $\mathcal{V}_x$. Thus, such coefficients satisfy the assumption of Theorem \ref{keyintegraltheorem}, which will be used in order to prove their $p$-integrality in Theorem \ref{pintegraltheorem}.

\begin{definition}\label{Xextenddefinition}
Define the collection of opens 
$$V_m := \{|z_{\mathrm{HT}}| > p^{-m}\} \subset \mathcal{V}_x$$
for $m \in \mathbb{Z}$. Then $V_m \subset V_{m+1}$ and for any $M \in \mathbb{Z}$,
\begin{equation}\label{Vxunion}\mathcal{V}_x = \bigcup_{m \in\mathbb{Z}_{\ge M}}V_m,
\end{equation}
and for $g$ as in (\ref{gdefinition}), since $g^*z_{\mathrm{HT}} = p\cdot z_{\mathrm{HT}}$ and $(g^{-1})^*z_{\mathrm{HT}} = z_{\mathrm{HT}}/p$ by (\ref{zHTtransformationproperty}), we have mutually inverse isomorphisms given by that $GL_2(\mathbb{Q}_p)$-action
\begin{equation}\label{Vmcompatibilities}g : V_m \xrightarrow{\sim} V_{m+1}, \hspace{1cm} g^{-1} : V_{m+1} \xrightarrow{\sim} V_m
\end{equation}
for all $m \in \mathbb{Z}$. 
\end{definition}

\begin{definition}For any $1 \le i \le \mathbf{n}$ (where $\mathbf{n}$ is as in Definition \ref{ndefinition}), recall the open subset $U_i \subset U$ from Definition \ref{Udefinition}. For any $m \in \mathbb{Z}$, let
\begin{equation}\label{Um}U_m := U \cdot g^m \overset{(\ref{YIgVx}), (\ref{zHTtransformationproperty})}{\subset} \{p^m\cdot z_{\mathrm{HT}} \neq 0\} \overset{(\ref{Vz})}{=} \mathcal{V}_x, \hspace{1cm} U_{i,m} := U_i \cdot g^m \subset U_m \subset \mathcal{V}_x.
\end{equation}
Then applying $g^m$ to (\ref{etalelocusunion3}), we see that
\begin{equation}\label{etalelocusunion4}U_m = \bigcup_{i = 1}^{\mathbf{n}}U_{i,m}.
\end{equation}
\end{definition}

Pullback by $g^{-m} : U_m \rightarrow U$ induces an isomorphism $(g^{-m})^* : \mathbf{\Gamma}(\hat{\mathcal{O}}_U^+) \xrightarrow{\sim} \mathbf{\Gamma}(\hat{\mathcal{O}}_{U_m}^+)$. For any $1 \le i \le \mathbf{n}$ (where $\mathbf{n}$ is as in Definition \ref{ndefinition}) and any $m \in \mathbb{Z}_{\ge 0}$, we have
\begin{equation}\label{gjflat}\begin{split}(g^{-m})^*j_i^{\flat} &\overset{(\ref{jflat})}{=} (g^{-m})^*(j_i,(g^{-1})^*j_i,(g^{-2})^*j_i,\ldots,(g^{-n})^*j_i,\ldots) \\
&= ((g^{-m})^*j_i,(g^{-m-1})^*j_i,(g^{-m-2})^*j_i,\ldots,(g^{-m-n})^*j_i,\ldots) \\
&\overset{(\ref{jcongruence})}{\equiv} ((g^{-m})^*j_i,((g^{-m})^*j_i)^{1/p},((g^{-m})^*j_i)^{1/p^2},\ldots,((g^{-m})^*j_i)^{1/p^n},\ldots) \pmod{p^{1-\epsilon_0}\mathbf{\Gamma}(\hat{\mathcal{O}}_{U_m}^+)}
\end{split}
\end{equation}
where $((g^{-m})^*j_i)^{1/p^n}$ denotes some section with $(((g^{-m})^*j_i)^{1/p^n})^{p^n} = (g^{-m})^*j_i$. Thus 
$$(g^{-m})^*j_i^{\flat}\in \mathbf{\Gamma}(\hat{\mathcal{O}}_{U_m}^{+,\flat}).$$

\begin{definition}\label{Xextenddefinition2}
If moreover $m \in \mathbb{Z}_{\ge 0}$, let
\begin{equation}\label{Xm}X_{i,m} := (g^{-m})^*X_i = (g^{-m})^*(j_i - [j_i^{\flat}]) \in \mathcal{O}\mathbb{B}_{\mathrm{dR},U_m}^+(U_m).
\end{equation}
\end{definition}

\begin{proposition}Let $U = \mathcal{Y}^{\mathrm{Ig}}(\epsilon_0)$ be as in Definition \ref{Udefinition}. %For any $m \in \mathbb{Z}_{\ge 0}$ we have 
%\begin{equation}\label{Umequality}U_m = \mathcal{Y}^{\mathrm{Ig}}(\epsilon_0/p^m).
%\end{equation}
 For any $m \in \mathbb{Z}$, we have 
\begin{equation}\label{Uminclusion}U_m \subset U_{m+1}.
\end{equation}
\end{proposition}

\begin{proof}First note that from the definitions (\ref{U}), we have 
$$\mathcal{Y}^{\mathrm{Ig}}(\epsilon_0/p) \subset \mathcal{Y}^{\mathrm{Ig}}(\epsilon_0) = U.$$
Applying $g^{m+1}$, we get
$$\mathcal{Y}^{\mathrm{Ig}}(\epsilon_0/p)\cdot g^{m+1} \subset U \cdot g^{m+1}.$$
Thus
$$U_m = U \cdot g^m = \mathcal{Y}^{\mathrm{Ig}}(\epsilon_0) \cdot g^m \overset{(\ref{gisomorphism})}{=} \mathcal{Y}^{\mathrm{Ig}}(\epsilon_0/p) \cdot g^{m+1} \subset U \cdot g^{m+1} = U_{m+1}.$$

%$$U_m \overset{(\ref{Um})}{=} U \cdot g^m.$$
%Now $U_m = \mathcal{Y}^{\mathrm{Ig}}(\epsilon_0/p^m) \subset \mathcal{Y}^{\mathrm{Ig}}(\epsilon_0/p^{m+1}) = U_{m+1}$, giving (\ref{Uminclusion}). 

\end{proof}

Letting $b$ be as in Lemma \ref{blemma}, we have $V_{-b} = \{|z_{\mathrm{HT}}| > p^b\} \overset{(\ref{zUb})}{\subset} U$. Applying $g^m$ to both sides, we have 
\begin{equation}\label{V-binclusion}V_{-b} \cdot g^m \subset U \cdot g^m = U_m.
\end{equation}
Thus for any $M \in \mathbb{Z}$,
$$\mathcal{V}_x \overset{(\ref{Vxunion})}{=} \bigcup_{m \in \mathbb{Z}_{\ge M-b}}V_m = \bigcup_{m \in \mathbb{Z}_{\ge M}}V_{-b+m} \overset{(\ref{Vmcompatibilities})}{=} \bigcup_{m \in \mathbb{Z}_{\ge M}}V_{-b}\cdot g^m \overset{(\ref{V-binclusion})}{\subset} \bigcup_{m \in \mathbb{Z}_{\ge M}}U_m.$$
Since $U_m \overset{(\ref{Um})}{\subset} \mathcal{V}_x$ for every $m \in \mathbb{Z}$, we thus have 
\begin{equation}\label{Vxunion'}\mathcal{V}_x = \bigcup_{m \in \mathbb{Z}_{\ge M}}U_m
\end{equation}
for any $M \in \mathbb{Z}$.

\begin{corollary}Let $X_{i,m}$ be as in (\ref{Xm}). We have an isomorphism of rings
\begin{equation}\label{Vxtriv}\mathcal{O}\mathbb{B}_{\mathrm{dR},U_{i,m}}^+(U_{i,m}) \cong \mathbb{B}_{\mathrm{dR},U_{i,m}}^+(U_{i,m})\llbracket X_{i,m}\rrbracket
\end{equation}
compatible with connections and filtrations (giving $X_{i,m}$ filtration degree 1). 
\end{corollary}

\begin{proof}%Let $b$ be as in Lemma \ref{blemma}. Recall the affinoid opens $V_m \subset \mathcal{V}_x$ from Definition \ref{Xextenddefinition}. By Lemma \ref{blemma}, we have $V_{-b} \subset U$.
%From (\ref{Utriv}), we have an isomorphism of rings
%$$\mathcal{O}\mathbb{B}_{\mathrm{dR},\mathcal{V}_x}^+(U) \cong \mathbb{B}_{\mathrm{dR},\mathcal{V}_x}^+(U)\llbracket X\rrbracket
%$$
%compatible with connections and filtrations (giving $X$ filtration degree 1). This restricts along $V_{-b} \subset U$ to an isomorphism of rings
%\begin{equation}\label{Utrivsection}\mathcal{O}\mathbb{B}_{\mathrm{dR},\mathcal{V}_x}^+(V_{-b}) \cong \mathbb{B}_{\mathrm{dR},\mathcal{V}_x}^+(V_{-b})\llbracket X\rrbracket
%\end{equation}
%compatible with connections and filtrations. 

For any $m \in \mathbb{Z}_{\ge 0}$, define an isomorphism of rings 
$$\mathcal{O}\mathbb{B}_{\mathrm{dR},U_{i,m}}^+(U_{i,m}) \xrightarrow{\sim}  \mathbb{B}_{\mathrm{dR},U_{i,m}}^+(U_{i,m})\llbracket X_{i,m}\rrbracket$$
as the composition of isomorphisms of rings
$$\mathcal{O}\mathbb{B}_{\mathrm{dR},U_{i,m}}^+(U_{i,m}) \xrightarrow{(g^m)^*} \mathcal{O}\mathbb{B}_{\mathrm{dR},U_i}^+(U_i) \overset{(\ref{Utriv})}{\cong} \mathbb{B}_{\mathrm{dR},U_i}^+(U_i)\llbracket X_i\rrbracket \xrightarrow{(g^{-m})^*} \mathbb{B}_{\mathrm{dR},U_{i,m}}^+(U_{i,m})\llbracket X_{i,m}\rrbracket.$$
Since the $GL_2(\mathbb{Q}_p)$-action preserves filtrations and commutes with connections, the isomorphisms are compatible with connections and filtrations. %By (\ref{Vmcompatibilities}), the isomorphisms $\gamma_m$ glue together over all $m \ge -b$ to give an isomorphism $\mathcal{O}\mathbb{B}_{\mathrm{dR},\mathcal{V}_x}^+(\mathcal{V}_x) \xrightarrow{\sim}\mathbb{B}_{\mathrm{dR},\mathcal{V}_x}^+(\mathcal{V}_x)\llbracket X\rrbracket$ compatible with connections and filtrations. 
 This gives (\ref{Vxtriv}). 

\end{proof}

\section{$z_{\mathrm{dR}}, q_{\mathrm{dR}}$ and $q_{\mathrm{dR}}$-Expansions}\label{zdRqdRqdRexpsection}Continuing the setting of Section \ref{Ysection} and Convention \ref{Yconvention}, recall that $Y$ is an adic space over $\mathrm{Spa}(k,\mathcal{O}_k)$ where $k$ is as in Definition \ref{kdefinition}, and $Y_{\infty} \in Y_{\text{pro\'{e}t}}$ is the uncompleted infinite-level modular curve, with strong completion $\hat{Y}_{\infty}$ (see \cite[Proposition 2.4.4]{ScholzeWeinstein}). In this section, we develop the theory of the de Rham period $z_{\mathrm{dR}}$ (Definition \ref{'zqwdefinition}, (\ref{gluezdRsection}) and (\ref{zqw2})) which measures the position of the Hodge filtration $H^{1,0} \subset H_{\mathrm{dR}}^1$ in de Rham cohomology, in analogy with the Hodge-Tate period $z_{\mathrm{HT}}$ of \cite{ScholzeTorsion} which measures the position of the Hodge-Tate filtration $H^{0,1} \subset H_{\text{\'{e}t}}^1$ in \'{e}tale cohomology. We will then introduce the notion of $q_{\mathrm{dR}}$-expansions (Definition \ref{'zqexpansions}), which generalizes the notion of Serre-Tate expansions on the ordinary locus and provides an analogous theory of power series expansions of $p$-adic modular forms on the supersingular locus. These latter expansions will provide a key step in our constructions of $p$-adic $L$-functions in Sections \ref{padicLfunctionsection} and \ref{padicLfunctionsection2}.

\subsection{Relative de Rham and \'{e}tale cohomologies}

Recall that $\pi : \mathcal{E} \rightarrow Y$ is the universal object (i.e. the universal (false) elliptic curve with $\Gamma$-level structure, $\Gamma(p^{\infty})$-level structure, and $\mathcal{O}_D$-endomorphism structure) and $\omega = \pi_*\Omega_{\mathcal{E}/Y}$ as in (\ref{omegaY})

Let us briefly view $Y$ and $Y_{\infty}$ as defined over $\mathrm{Spa}(\mathbb{Q}_p,\mathbb{Z}_p)$ again. We will base change back to $\mathrm{Spa}(k,\mathcal{O}_k)$ again (as in Convention \ref{basechangeconvention}) at the beginning of Section \ref{XthetaXsection}. Let
$$H_{\text{\'{e}t}}^1(\mathcal{E}) = R^1\pi_*\hat{\mathbb{Z}}_{p,\mathcal{E}}$$
be universal \'{e}tale cohomology. This is a $\hat{\mathbb{Z}}_{p,Y}$-local system of rank 2 on $Y_{\text{pro\'{e}t}}$, in the sense of \cite{Scholze}. Let 
$$T_p\mathcal{E} = \varprojlim_n \mathcal{E}[p^n]$$
be as in Convention \ref{idempotentconvention}. This is again a $\hat{\mathbb{Z}}_{p,Y}$-local system of rank 2 on $Y_{\text{pro\'{e}t}}$. Let
$$\langle \cdot, \cdot \rangle : T_p\mathcal{E} \times T_p\mathcal{E} \rightarrow \hat{\mathbb{Z}}_{p,Y}(1)$$
be the universal Weil pairing, where $``(1)''$ denotes the Tate twist. By its nondegeneracy, we get a natural identification 
$$T_p\mathcal{E} = H_{\text{\'{e}}t}^1(\mathcal{E})(1).$$
Let
\begin{equation}\label{etaletrivialization}(e_1,e_2) : \hat{\mathbb{Z}}_{p,Y_{\infty}}^{\oplus 2} \xrightarrow{\sim} T_p\mathcal{E}|_{Y_{\infty}} = H_{\text{\'{e}t}}^1(\mathcal{E})(1)|_{Y_{\infty}}
\end{equation}
be the $\Gamma(p^{\infty})$-level structure. In particular, we have a section
\begin{equation}\label{universalWeil}\langle e_1, e_2\rangle \in \hat{\mathbb{Z}}_{p,Y}(1)(Y_{\infty}).
\end{equation}
Let 
$$\hat{\mathcal{O}}_Y^{+,\flat} := \varprojlim_{x \mapsto x^p}\hat{\mathcal{O}}_Y^+/p$$
(cf. \cite[Definition 5.9]{Scholze}; $\hat{\mathcal{O}}_Y^{+,\flat}$ is equal to $\hat{\mathcal{O}}_{Y^{\flat}}^+$ in loc. cit.), and
$$[\cdot ] : \hat{\mathcal{O}}_Y^{+,\flat} \rightarrow W(\hat{\mathcal{O}}_Y^{+,\flat})$$
denote the Teichm\"{u}ller lift which, composed with the natural map $W(\hat{\mathcal{O}}_Y^{+,\flat}) \rightarrow \mathbb{B}_{\mathrm{dR},Y}^+$ from \cite[Section 6]{Scholze}, gives a map
$$[\cdot ] : \hat{\mathcal{O}}_Y^{+,\flat} \rightarrow \mathbb{B}_{\mathrm{dR},Y}^+.$$
Note that there is a natural inclusion $\hat{\mathbb{Z}}_{p,Y}(1) \subset \hat{\mathcal{O}}_Y^{+,\flat}$. Hence, from the Weil pairing (\ref{universalWeil}), we get a section
\begin{equation}\label{'tdefinition}t  := \log([\langle e_1,e_2\rangle]) \in \mathbb{B}_{\mathrm{dR},Y}^+(Y_{\infty}),
\end{equation}
which is a generator of $\mathrm{Fil}^1\mathbb{B}_{\mathrm{dR},Y}^+(Y_{\infty}) = \ker\theta$.

Let
$$H_{\mathrm{dR}}^1(\mathcal{E}) = R^1\pi_{\mathrm{dR},*}(0 \rightarrow \mathcal{O}_{\mathcal{E}} \rightarrow \Omega_{\mathcal{E}/Y} \rightarrow 0)$$
be the universal de Rham cohomology, a rank 2, locally free $\mathcal{O}_Y$-module on $Y_{\text{pro\'{e}t}}$. It is equipped with the Hodge filtration
$$\mathrm{Fil}^iH_{\mathrm{dR}}^1(\mathcal{E}) = \begin{cases} H_{\mathrm{dR}}^1(\mathcal{E}) & i \le 0\\
\omega & i = 1\\
0 & i \ge 2\\
\end{cases}.$$
The $\mathcal{O}_Y$-module $H_{\mathrm{dR}}^1(\mathcal{E})$ is also equipped with the Gauss-Manin connection
$$\nabla : H_{\mathrm{dR}}^1(\mathcal{E}) \rightarrow H_{\mathrm{dR}}^1(\mathcal{E}) \otimes_{\mathcal{O}_Y}\Omega_Y.$$
Using the Leibniz rule and the natural connection $\nabla : \mathcal{O}\mathbb{B}_{\mathrm{dR},Y}^{(+)} \rightarrow \mathcal{O}\mathbb{B}_{\mathrm{dR},Y}^{(+)} \otimes_{\mathcal{O}_Y}\Omega_Y$, we can extend $\nabla$ to a connection
$$\nabla : H_{\mathrm{dR}}^1(\mathcal{E}) \otimes_{\mathcal{O}_Y} \mathcal{O}\mathbb{B}_{\mathrm{dR},Y}^{(+)} \rightarrow H_{\mathrm{dR}}^1(\mathcal{E}) \otimes_{\mathcal{O}_Y}\mathcal{O}\mathbb{B}_{\mathrm{dR},Y}^{(+)} \otimes_{\mathcal{O}_Y}\Omega_Y.$$
Here, ``$(+)$'' denotes the optional presence of a ``$+$''. We give $H_{\mathrm{dR}}^1(\mathcal{E}) \otimes_{\mathcal{O}_Y}\mathcal{O}\mathbb{B}_{\mathrm{dR},Y}^{(+)}$ the filtration given by the convolution of the Hodge filtration with the filtration on $\mathcal{O}\mathbb{B}_{\mathrm{dR},Y}^{(+)}$. That is,
\begin{equation}\label{Hodgefiltration}\mathrm{Fil}^i\left(H_{\mathrm{dR}}^1(\mathcal{E}) \otimes_{\mathcal{O}_Y}\mathcal{O}\mathbb{B}_{\mathrm{dR},Y}^{(+)}\right) = \sum_{j \in \mathbb{Z}}\mathrm{Fil}^{i+j}H_{\mathrm{dR}}^1(\mathcal{E}) \otimes_{\mathcal{O}_Y}\mathrm{Fil}^{-j}\mathcal{O}\mathbb{B}_{\mathrm{dR},Y}^{(+)}.
\end{equation}
We sometimes refer to the inclusion
\begin{equation}\label{Hodgefiltrationinclusion}\omega = \mathrm{Fil}^1H_{\mathrm{dR}}^1(\mathcal{E}) \subset H_{\mathrm{dR}}^1(\mathcal{E})
\end{equation}
as the ``Hodge filtration''.

We give the $\hat{\mathbb{Z}}_{p,Y}$-local systems $H_{\text{\'{e}t}}^1(\mathcal{E})$ and $T_p\mathcal{E}$ the trivial filtrations
$$\mathrm{Fil}^iH_{\text{\'{e}t}}^1(\mathcal{E}) = \begin{cases} H_{\text{\'{e}t}}^1(\mathcal{E}) & i \ge 0\\
0 & i < 0\\
\end{cases}, \hspace{1cm} \mathrm{Fil}^iT_p\mathcal{E} = \begin{cases} T_p\mathcal{E} & i \ge 0\\
0 & i < 0\\
\end{cases}.$$
We get $\mathcal{O}\mathbb{B}_{\mathrm{dR},Y}^{(+)}$-modules of rank 2 
$$H_{\text{\'{e}t}}^1(\mathcal{E}) \otimes_{\hat{\mathbb{Z}}_{p,Y}} \mathcal{O}\mathbb{B}_{\mathrm{dR},Y}^{(+)}, \hspace{1cm} T_p\mathcal{E} \otimes_{\hat{\mathbb{Z}}_{p,Y}}\mathcal{O}\mathbb{B}_{\mathrm{dR},Y}^{(+)}$$
which are equipped with natural filtrations given by the convolutions of the trivial filtrations on $H_{\text{\'{e}t}}^1(\mathcal{E})$ and $T_p\mathcal{E}$ with the natural filtration on $\mathcal{O}\mathbb{B}_{\mathrm{dR},Y}^{(+)}$ (see Section \ref{periodsheavessection}). That is,
$$\mathrm{Fil}^i\left(H_{\text{\'{e}t}}^1(\mathcal{E}) \otimes_{\hat{\mathbb{Z}}_{p,Y}} \mathcal{O}\mathbb{B}_{\mathrm{dR},Y}^{(+)}\right) = \sum_{j \in \mathbb{Z}}\mathrm{Fil}^{i+j}H_{\text{\'{e}t}}^1(\mathcal{E}) \otimes_{\hat{\mathbb{Z}}_{p,Y}} \mathrm{Fil}^{-j}\mathcal{O}\mathbb{B}_{\mathrm{dR},Y}^{(+)},$$
and similarly with $T_p\mathcal{E}$. Moreover, the natural connection $\nabla : \mathcal{O}\mathbb{B}_{\mathrm{dR},Y}^{(+)} \rightarrow \mathcal{O}\mathbb{B}_{\mathrm{dR},Y}^{(+)} \otimes_{\mathcal{O}_Y}\Omega_Y$ induces natural connections
$$\nabla : H_{\text{\'{e}t}}^1(\mathcal{E}) \otimes_{\hat{\mathbb{Z}}_{p,Y}} \mathcal{O}\mathbb{B}_{\mathrm{dR},Y}^{(+)} \rightarrow H_{\text{\'{e}t}}^1(\mathcal{E}) \otimes_{\hat{\mathbb{Z}}_{p,Y}}\mathcal{O}\mathbb{B}_{\mathrm{dR},Y}^{(+)} \otimes_{\mathcal{O}_Y}\Omega_Y,$$
$$\nabla : T_p\mathcal{E} \otimes_{\hat{\mathbb{Z}}_{p,Y}} \mathcal{O}\mathbb{B}_{\mathrm{dR},Y}^{(+)} \rightarrow T_p\mathcal{E} \otimes_{\hat{\mathbb{Z}}_{p,Y}}\mathcal{O}\mathbb{B}_{\mathrm{dR},Y}^{(+)} \otimes_{\mathcal{O}_Y}\Omega_Y,$$
such that $H_{\text{\'{e}t}}^1(\mathcal{E}) \otimes_{\hat{\mathbb{Z}}_{p,Y}}\mathbb{B}_{\mathrm{dR},Y}^{(+)}$ (resp. $T_p\mathcal{E}\otimes_{\hat{\mathbb{Z}}_{p,Y}}\mathbb{B}_{\mathrm{dR},Y}^{(+)}$) is the sheaf of horizontal sections of $\nabla$. We have canonical identifications 
\begin{equation}\label{tidentifications}\hat{\mathbb{Z}}_{p,Y_{\infty}}(1) = \hat{\mathbb{Z}}_{p,Y_{\infty}}\cdot t, \hspace{1cm} H_{\text{\'{e}t}}^1(\mathcal{E})|_{Y_{\infty}} = T_p\mathcal{E}|_{Y_{\infty}} \cdot t^{-1}.
\end{equation}

\subsection{Review of Scholze's de Rham comparison theorem}\label{zdRqdRqdRexpansionssection}

We have two $\mathbb{B}_{\mathrm{dR},Y}^+$-local systems
$$\mathbb{M}_0 := (H_{\mathrm{dR}}^1(\mathcal{E}) \otimes_{\mathcal{O}_Y}\mathcal{O}\mathbb{B}_{\mathrm{dR},Y}^+)^{\nabla = 0}, \hspace{.5cm} \mathbb{M} := (H_{\text{\'{e}t}}^1(\mathcal{E}) \otimes_{\hat{\mathbb{Z}}_{p,Y}} \mathcal{O}\mathbb{B}_{\mathrm{dR},Y}^+)^{\nabla = 0} = H_{\text{\'{e}t}}^1(\mathcal{E}) \otimes_{\hat{\mathbb{Z}}_{p,Y}}\mathbb{B}_{\mathrm{dR},Y}^+.$$
From \cite[Proposition 7.9]{Scholze} (and \cite[Proposition 2.2.3]{CaraianiScholze}), we have a natural inclusion
\begin{equation}\label{MM0inclusion}\mathbb{M}_0 \subset \mathbb{M}.
\end{equation}
By \cite[Theorem 7.2]{Scholze}., we have 
$$\mathbb{M}_0 \otimes_{\mathbb{B}_{\mathrm{dR},Y}^+}\mathcal{O}\mathbb{B}_{\mathrm{dR},Y}^+ = H_{\mathrm{dR}}^1(\mathcal{E}) \otimes_{\mathcal{O}_Y}\mathcal{O}\mathbb{B}_{\mathrm{dR},Y}^+, \hspace{.5cm} \mathbb{M} \otimes_{\mathbb{B}_{\mathrm{dR},Y}^+}\mathcal{O}\mathbb{B}_{\mathrm{dR},Y}^+ = H_{\text{\'{e}t}}^1(\mathcal{E}) \otimes_{\hat{\mathbb{Z}}_{p,Y}}\mathcal{O}\mathbb{B}_{\mathrm{dR},Y}^+.$$
Hence, tensoring (\ref{MM0inclusion}) with $\otimes_{\mathbb{B}_{\mathrm{dR},Y}^+}\mathcal{O}\mathbb{B}_{\mathrm{dR},Y}^+$, we get a map 
\begin{equation}\label{comparisonmap}i_{\mathrm{dR}} : H_{\mathrm{dR}}^1(\mathcal{E}) \rightarrow H_{\text{\'{e}t}}^1(\mathcal{E}) \otimes_{\hat{\mathbb{Z}}_{p,Y}} \mathcal{O}\mathbb{B}_{\mathrm{dR},Y}^+.
\end{equation}
Tensoring with $\otimes_{\mathcal{O}\mathbb{B}_{\mathrm{dR},Y}^+}\mathcal{O}\mathbb{B}_{\mathrm{dR},Y}$, we get an isomorphism compatible with connections and filtrations
\begin{equation}\label{comparisonmap2}i_{\mathrm{dR}} : H_{\mathrm{dR}}^1(\mathcal{E}) \otimes_{\mathcal{O}\mathbb{B}_{\mathrm{dR},Y}^+}\mathcal{O}\mathbb{B}_{\mathrm{dR},Y} \xrightarrow{\sim} H_{\text{\'{e}t}}^1(\mathcal{E}) \otimes_{\hat{\mathbb{Z}}_{p,Y}} \mathcal{O}\mathbb{B}_{\mathrm{dR},Y}
\end{equation}
which is a special case of the isomorphism from Theorem 8.8 of op. cit. In particular, (\ref{comparisonmap}) respects filtrations and connections; here the filtration on the source is given by the (descending) Hodge filtration (\ref{Hodgefiltration}), and the filtration on the target is given by the natural filtration on $\mathcal{O}\mathbb{B}_{\mathrm{dR},Y}^+$ (see Section \ref{periodsheavessection}). Pulling back along $Y_{\infty} \rightarrow Y$, we get
$$i_{\mathrm{dR}} : H_{\mathrm{dR}}^1(\mathcal{E})|_{Y_{\infty}} \rightarrow H_{\text{\'{e}t}}^1(\mathcal{E})\otimes_{\hat{\mathbb{Z}}_{p,Y}}\mathcal{O}\mathbb{B}_{\mathrm{dR},Y_{\infty}}^+.$$

Let 
$$\langle \cdot,\cdot \rangle_{\mathrm{dR}} : H_{\mathrm{dR}}^1(\mathcal{E}) \times H_{\mathrm{dR}}^1(\mathcal{E}) \rightarrow \mathcal{O}_Y$$
be the $\mathcal{O}_Y$-linear Poincar\'{e} pairing. Tensoring with $\otimes_{\mathcal{O}_Y}\mathcal{O}\mathbb{B}_{\mathrm{dR},Y}^+$, we get an $\mathcal{O}\mathbb{B}_{\mathrm{dR},Y}^+$-linear pairing
$$\langle \cdot,\cdot \rangle_{\mathrm{dR}} : (H_{\mathrm{dR}}^1(\mathcal{E}) \otimes_{\mathcal{O}_Y}\mathcal{O}\mathbb{B}_{\mathrm{dR},Y}^+) \times (H_{\mathrm{dR}}^1(\mathcal{E}) \otimes_{\mathcal{O}_Y}\mathcal{O}\mathbb{B}_{\mathrm{dR},Y}^+) \rightarrow \mathcal{O}\mathbb{B}_{\mathrm{dR},Y}^+.$$

From $i_{\mathrm{dR}}$, get an induced pairing
$$\langle i_{\mathrm{dR}}(\cdot), i_{\mathrm{dR}}(\cdot)\rangle : (H_{\mathrm{dR}}^1(\mathcal{E}) \otimes_{\hat{\mathbb{Z}}_{p,Y}} \mathcal{O}\mathbb{B}_{\mathrm{dR},Y}^+) \times (H_{\mathrm{dR}}^1(\mathcal{E}) \otimes_{\hat{\mathbb{Z}}_{p,Y}} \mathcal{O}\mathbb{B}_{\mathrm{dR},Y}^+) \rightarrow \mathcal{O}\mathbb{B}_{\mathrm{dR},Y}^+.$$
It is known that the absolute de Rham comparison theorem identifies the Weil and Poincar\'{e} pairings (see, for example, \cite{Colmez}), and so on every classical point of $Y$ we have that the evaluations of $\langle i_{\mathrm{dR}}(\cdot),i_{\mathrm{dR}}(\cdot)\rangle$ and $\langle \cdot, \cdot \rangle_{\mathrm{dR}}$ are equal. Since $Y$ is induced by a rigid analytic space, and the sheaves $H_{\mathrm{dR}}^1(\mathcal{E})$ and $H_{\text{\'{e}t}}^1(\mathcal{E})$ are induced from sheaves on this rigid analytic space, we thus have
\begin{equation}\label{comparepairings}\langle i_{\mathrm{dR}}(\cdot),i_{\mathrm{dR}}(\cdot)\rangle = \langle \cdot,\cdot\rangle_{\mathrm{dR}}.
\end{equation}

\subsection{$X_i$, $\theta$ and $\theta_{X_i}$}\label{XthetaXsection}

Base change $Y$ back to $\mathrm{Spa}(k,\mathcal{O}_k)$ (per Convention \ref{basechangeconvention}). Let $0 < \epsilon_0 < p/(p+1)$, $U = \mathcal{Y}^{\mathrm{Ig}}(\epsilon_0) \subset Y_{\infty}$, $U_i = \mathcal{Y}^{\mathrm{Ig}}(\epsilon_0)_i \subset Y_{\infty}$ be as in Definition \ref{Udefinition}. By (\ref{Utriv}) and (\ref{Vxtriv}), we have isomorphisms for every $1 \le i \le \mathbf{n}$ ($\mathbf{n}$ as in Definition \ref{ndefinition})
\begin{equation}\label{'triv}\mathcal{O}\mathbb{B}_{\mathrm{dR},U_i}^+ \cong \mathbb{B}_{\mathrm{dR},U_i}^+\llbracket X_i\rrbracket, \hspace{1cm} \mathcal{O}\mathbb{B}_{\mathrm{dR},U_{i,m}}^+(U_{i,m}) \cong \mathbb{B}_{\mathrm{dR},U_{i,m}}^+(U_{i,m})\llbracket X_{i,m}\rrbracket
\end{equation}
which are compatible with connections and filtrations (giving $X_i$ and $X_{i,m}$ filtration degree 1).

\begin{choice}\label{'choice1}For the remainder of this section, we will identify 
\begin{equation}\label{Uidentify}\mathcal{O}\mathbb{B}_{\mathrm{dR},U_i}^+ = \mathbb{B}_{\mathrm{dR},U_i}^+\llbracket X_i\rrbracket, \hspace{1cm} \mathcal{O}\mathbb{B}_{\mathrm{dR},U_{i,m}}^+(U_{i,m}) = \mathbb{B}_{\mathrm{dR},U_{i,m}}^+(U_{i,m})\llbracket X_{i,m}\rrbracket
\end{equation}
using (\ref{'triv}).  
\end{choice}

\begin{definition}\label{BdR+Xtdefinition}\begin{enumerate}
\item The $X_i$-adic completion of 
$$\mathcal{O}\mathbb{B}_{\mathrm{dR},U_i} = \mathbb{B}_{\mathrm{dR},U_i}^+\llbracket X_i\rrbracket [1/t]$$ 
is the sheaf of formal power series in $X_i$ of $\mathbb{B}_{\mathrm{dR},U_i}$
$$\mathbb{B}_{\mathrm{dR},U_i}\llbracket X_i\rrbracket,$$
which still has an $\mathcal{O}\mathbb{B}_{\mathrm{dR},U_i}^+$-module structure and thus an $\mathcal{O}_{U_i}$-module structure by restriction of scalars $\mathcal{O}_{U_i} \subset \mathcal{O}\mathbb{B}_{\mathrm{dR},U_i}^+$. 
\item Similarly, the $X_{i,m}$-adic completion of 
$$\mathcal{O}\mathbb{B}_{\mathrm{dR},U_{i,m}}(U_{i,m}) = \mathbb{B}_{\mathrm{dR},U_{i,m}}^+(U_{i,m})\llbracket X_{i,m}\rrbracket [1/t]$$
is 
$$\mathbb{B}_{\mathrm{dR},U_{i,m}}(U_{i,m})\llbracket X_{i,m}\rrbracket.$$
\item We also have the $\mathcal{O}\mathbb{B}_{\mathrm{dR},U_i}^+$-submodule which is the subsheaf of $\mathbb{B}_{\mathrm{dR},U_i}\llbracket X_i\rrbracket$ given by power series in $X_i/t$ over $\mathbb{B}_{\mathrm{dR},U_i}^+$:
$$\mathbb{B}_{\mathrm{dR},U_i}^+\llbracket X_i/t\rrbracket \subset \mathbb{B}_{\mathrm{dR},U_i}\llbracket X_i\rrbracket.$$
This also an $\mathcal{O}_{U_i}$-submodule by restriction of scalars $\mathcal{O}_{U_i} \subset \mathcal{O}\mathbb{B}_{\mathrm{dR},U_i}^+$. Moreover, 
\begin{equation}\label{OBdRinclusion}\mathcal{O}\mathbb{B}_{\mathrm{dR},U_i} = \mathbb{B}_{\mathrm{dR},U_i}^+\llbracket X_i\rrbracket [1/t] \subset \mathbb{B}_{\mathrm{dR},U_i}^+\llbracket X_i/t\rrbracket.
\end{equation}
\item Similarly, we have the $\mathcal{O}\mathbb{B}_{\mathrm{dR},U_{i,m}}^+(U_{i,m})$-submodule which is given by adjoining $X_{i,m}/t$ to $\mathbb{B}_{\mathrm{dR},U_{i,m}}^+(U_{i,m})$ inside the ring $\mathbb{B}_{\mathrm{dR},U_{i,m}}(U_{i,m})\llbracket X_{i,m}\rrbracket$:
$$\mathbb{B}_{\mathrm{dR},U_{i,m}}^+(U_{i,m})\llbracket X_{i,m}/t\rrbracket \subset \mathbb{B}_{\mathrm{dR},U_{i,m}}(U_{i,m})\llbracket X_{i,m}\rrbracket.$$
This is also an $\mathcal{O}_{U_{i,m}}(U_{i,m})$-submodule by restriction of scalars $\mathcal{O}_{U_{i,m}}(U_{i,m}) \subset \mathcal{O}\mathbb{B}_{\mathrm{dR},U_{i,m}}^+(U_{i,m})$. Moreover, 
$$\mathcal{O}\mathbb{B}_{\mathrm{dR},U_{i,m}}^+(U_{i,m}) = \mathbb{B}_{\mathrm{dR},U_{i,m}}^+(U_{i,m})\llbracket X_{i,m}\rrbracket \subset \mathbb{B}_{\mathrm{dR},U_{i,m}}^+(U_{i,m})\llbracket X_{i,m}/t\rrbracket.$$
\end{enumerate}
\end{definition}

\begin{definition}\label{'thetadefinitions}\begin{enumerate}
\item Let 
$$\theta_{X_i} : \mathcal{O}\mathbb{B}_{\mathrm{dR},U_i} = \mathbb{B}_{\mathrm{dR},U_i}^+\llbracket X_i\rrbracket [1/t] \rightarrow \mathbb{B}_{\mathrm{dR},U_i},$$
$$\theta_{X_{i,m}} : \mathcal{O}\mathbb{B}_{\mathrm{dR},U_{i,m}}(U_{i,m}) = \mathbb{B}_{\mathrm{dR},U_{i,m}}^+(U_{i,m})\llbracket X_{i,m}\rrbracket [1/t] \rightarrow \mathbb{B}_{\mathrm{dR},U_{i,m}}(U_{i,m})$$
be given by reduction modulo $X_i\cdot\mathcal{O}\mathbb{B}_{\mathrm{dR},U_i}$ and $X_{i,m}\cdot\mathcal{O}\mathbb{B}_{\mathrm{dR},U_{i,m}}$, respectively. These extend to maps on the $X$-adic and $X_{i,m}$-adic completions 
$$\theta_{X_i} : \mathbb{B}_{\mathrm{dR},U_i}\llbracket X_i\rrbracket \rightarrow \mathbb{B}_{\mathrm{dR},U_i}, \hspace{1cm} \theta_{X_{i,m}} : \mathbb{B}_{\mathrm{dR},U_{i,m}}(U_{i,m})\llbracket X_{i,m}\rrbracket \rightarrow \mathbb{B}_{\mathrm{dR},U_{i,m}}(U_{i,m}).$$
\item Let 
$$\theta_{X_i}^+ : \mathcal{O}\mathbb{B}_{\mathrm{dR},U_i}^+ \rightarrow \mathbb{B}_{\mathrm{dR},U_i}^+, \hspace{1cm} \theta_{X_{i,m}}^+ : \mathcal{O}\mathbb{B}_{\mathrm{dR},U_{i,m}}^+(U_{i,m}) \rightarrow \mathbb{B}_{\mathrm{dR},U_{i,m}}^+(U_{i,m})$$
be the restriction of $\theta_{X_i}$ to the subsheaf 
$$\mathcal{O}\mathbb{B}_{\mathrm{dR},U_i}^+ \subset \mathbb{B}_{\mathrm{dR},U_i}\llbracket X_i\rrbracket$$
and the restriction of $\theta_{X_{i,m}}$ to the subring 
$$\mathcal{O}\mathbb{B}_{\mathrm{dR},U_{i,m}}^+(U_{i,m})\subset \mathbb{B}_{\mathrm{dR},U_{i,m}}(U_{i,m})\llbracket X_{i,m}\rrbracket,$$
respectively.
\end{enumerate}
\end{definition}

\begin{proposition}\label{'thetaproposition}$\theta \circ \theta_{X_i}^+ = \theta_{\mathrm{dR}}$ as maps of sheaves $\mathcal{O}\mathbb{B}_{\mathrm{dR},U_i}^+ \rightarrow \hat{\mathcal{O}}_{U_i}$, and $\theta \circ \theta_{X_{i,m}}^+ = \theta_{\mathrm{dR}}$ as maps of rings $\mathcal{O}\mathbb{B}_{\mathrm{dR},U_{i,m}}^+(U_{i,m}) \rightarrow \hat{\mathcal{O}}_{U_{i,m}}(U_{i,m})$. 
\end{proposition}

\begin{proof}This follows immediately from (\ref{'triv}).
\end{proof}

\subsection{The periods $x_{\mathrm{dR}}$ and $y_{\mathrm{dR}}$}\label{xysection}

Since $\mathrm{Fil}^1H_{\mathrm{dR}}^1(\mathcal{E}) = \omega$, by the compatibility of (\ref{comparisonmap}) with filtrations we have
\begin{equation}\label{Hodgecomparisoninclusion}i_{\mathrm{dR}}(\omega) \subset H_{\text{\'{e}t}}^1 \otimes_{\hat{\mathbb{Z}}_{p,Y}} \mathrm{Fil}^1\mathcal{O}\mathbb{B}_{\mathrm{dR},Y}^+ \rightarrow T_p\mathcal{E} \otimes_{\hat{\mathbb{Z}}_{p,Y}} (\mathrm{Fil}^1\mathcal{O}\mathbb{B}_{\mathrm{dR},Y_{\infty}}^+)\cdot t^{-1} \overset{(e_1,e_2)}{=} (\mathrm{Fil}^1\mathcal{O}\mathbb{B}_{\mathrm{dR},Y_{\infty}}^+\cdot t^{-1})^2.
\end{equation}

\begin{definition}\label{'xydefinition}\begin{enumerate}
\item For any $W \in Y_{\text{pro\'{e}t}}/Y_{\infty}$ and any $w' \in \omega(W)$, define $$x_{\mathrm{dR}}(w'),y_{\mathrm{dR}}(w') \in \mathrm{Fil}^1\mathcal{O}\mathbb{B}_{\mathrm{dR},Y_{\infty}}^+(W) \cdot t^{-1}$$ by
$$i_{\mathrm{dR}}(w') = (x_{\mathrm{dR}}(w'),y_{\mathrm{dR}}(w')).$$
\item Moreover, if $W \in Y_{\text{pro\'{e}t}}/U_{i,m}$ and $w' \in \omega \otimes_{\mathcal{O}_Y}\mathbb{B}_{\mathrm{dR},U_{i,m}}^+(W)\llbracket X_{i,m}/t\rrbracket$, define
$$x_{\mathrm{dR}}(w'), y_{\mathrm{dR}}(w') \in \mathbb{B}_{\mathrm{dR},U_{i,m}}^+(W)\llbracket X_{i,m}/t\rrbracket$$
by tensoring the left-hand side of (\ref{Hodgecomparisoninclusion}) with $\otimes_{\mathcal{O}_Y}\mathbb{B}_{\mathrm{dR},U_{i,m}}^+(W)\llbracket X_{i,m}/t\rrbracket$ and composing the right-hand side with the map 
$$\mathrm{Fil}^1\mathcal{O}\mathbb{B}_{\mathrm{dR},Y_{\infty}}^+\cdot t^{-1} \rightarrow \mathbb{B}_{\mathrm{dR},U_{i,m}}^+(W)\llbracket X_{i,m}/t\rrbracket.$$
\end{enumerate}
\end{definition}

Recall the open subsets
$$U_m \overset{(\ref{Um})}= U \cdot g^m \subset \mathcal{V}_x \overset{(\ref{VxVy})}{=} \{x \neq 0\} \overset{(\ref{Vz})}{=} \{z_{\mathrm{HT}} \neq 0\} \subset Y_{\infty}.$$
Since $U$ is affinoid perfectoid by Lemma \ref{Uaffinoidperfectoidlemma} and the action of $g \in GL_2(\mathbb{Q}_p)$ preserves the property of being affinoid perfectoid (\cite[discussion after Definition III.3.5]{ScholzeTorsion}), we have that $U_m$ is affinoid perfectoid. 
\begin{proposition} $\omega(U_m) \neq 0$ for any $m \in \mathbb{Z}_{\ge 0}$. 
\end{proposition}
\begin{proof} Recall the fake Hasse invariant $0 \neq \frak{s} \in \omega\otimes_{\mathcal{O}_Y}\hat{\mathcal{O}}_Y(Y_{\infty})$, which thus induces $\frak{s}|_{U_m} \in \omega\otimes_{\mathcal{O}_Y}\hat{\mathcal{O}}_Y(U_m)$ and shows $\omega \otimes_{\mathcal{O}_Y}\hat{\mathcal{O}}_Y(U_m) \neq 0$. Since $U_m \overset{(\ref{Um})}{=} U \cdot g^m$ is affinoid perfectoid (Lemma \ref{Uaffinoidperfectoidlemma} implies $U$ is affinoid perfectoid, and the action of $g^m \in GL_2(\mathbb{Q}_p)$ preserves the property of being affinoid perfectoid by \cite[discussion after Definition 3.5]{ScholzeTorsion}), by \cite[Lemma 4.10(iv)]{Scholze} we have that $\omega\otimes_{\mathcal{O}_Y}\hat{\mathcal{O}}_Y(U_m)$ is the $p$-adic completion of $\omega(U_m)$. Hence $\omega(U_m) \neq 0$. %Then the rational sets $V_1 = \{|z_{\mathrm{HT}}| \le 1\}$ and $V_2 = \{|z_{\mathrm{HT}}|\ge 1\}$ are affinoid perfectoid by \cite[Theorem IV.1.1(i)]{ScholzeTorsion} and $Y_{\infty} = V_1 \cup V_2$. By \cite[Lemma 4.10(iv)]{Scholze}, we have $\omega\otimes_{\mathcal{O}_Y}\hat{\mathcal{O}}_Y(V_i)$ is the $p$-adic completion 
% by the Riemann-Roch theorem and the Kodaira-Spencer isomorphism $\omega^{\otimes 2} \cong \Omega_Y$ (see \cite[Appendix A1.3.17]{Katzpamf}; we recall the statement in (\ref{'KS}) below), and so $\omega(\mathcal{V}_x) \neq 0$.  
 \end{proof}

Many constructions in Section \ref{zqwsection} and Section \ref{wcansection} will depend (at least \emph{a priori}) on a choice as follows. 

\begin{choice}\label{'choicew}Given $m \in \mathbb{Z}_{\ge 0}$, fix $w \in \omega(U_m)$. 
\end{choice}

%\begin{convention}\label{xyconvention} When the choice of $w \in \omega(U_m)$ in Choice \ref{'choicew} is obvious from context, we will often write $x_{\mathrm{dR}}$ and $y_{\mathrm{dR}}$ instead of $x_{\mathrm{dR}}(w)$ and $y_{\mathrm{dR}}(w)$ for brevity. 
%\end{convention}

In what follows, we will often consider the restrictions of the sections $x_{\mathrm{dR}}, y_{\mathrm{dR}}$ to the open subset $U_{i,m} \subset U_m$. By a slight abuse of notation, we will use the same symbols to denote these restrictions, as the object of $Y_{\text{pro\'{e}t}}$ over which we work will be clear from context.

\begin{proposition}\label{'thetaxproposition}$\theta_{X_{i,m}}(y_{\mathrm{dR}}(w)) \in \mathbb{B}_{\mathrm{dR},U_{i,m}}^+(U_{i,m})$.
\end{proposition}

\begin{proof}We have $t \cdot y_{\mathrm{dR}}(w) \in \mathrm{Fil}^1\mathcal{O}\mathbb{B}_{\mathrm{dR},U_{i,m}}^+(U_{i,m})$, and we have
\begin{align*}t \cdot y_{\mathrm{dR}}(w) \in \mathrm{Fil}^1\mathcal{O}\mathbb{B}_{\mathrm{dR},U_{i,m}}^+(U_{i,m}) &= (\ker\theta_{\mathrm{dR}})\mathcal{O}\mathbb{B}_{\mathrm{dR},U_{i,m}}^+ \\
&= t\cdot \mathbb{B}_{\mathrm{dR},U_{i,m}}^+(U_{i,m}) + X_{i,m}\cdot\mathbb{B}_{\mathrm{dR},U_{i,m}}^+(U_{i,m})\llbracket X_{i,m}\rrbracket.
\end{align*}
So applying $\theta_{X_{i,m}}$, we see that $t\cdot \theta_{X_{i,m}}(y_{\mathrm{dR}}(w)) \in t\cdot\mathbb{B}_{\mathrm{dR},U_{i,m}}^+(U_{i,m})$, giving the statement of the Proposition.

\end{proof}

\subsection{Gluing over the $U_i$}\label{gluingsection}

\begin{definition}For $1 \le i \le \mathbf{n}$, $\mathbf{n}$ as in Definition \ref{ndefinition}, recall the sets $U_{i,m}$ from (\ref{Um}), and let 
$$U_{ij,m} := U_{i,m} \cap U_{j,m}.$$
When $m = 0$, so that $U_{i,m} = U_i$, we write $U_{ij,0} = U_{ij}$. Then from (\ref{'triv}) we have 
\begin{equation}\label{ijtriv}\mathbb{B}_{\mathrm{dR},U_{ij,m}}^+\llbracket X_{i,m}\rrbracket  = \mathcal{O}\mathbb{B}_{\mathrm{dR},U_{ij,m}}^+ = \mathbb{B}_{\mathrm{dR},U_{ij,m}}^+\llbracket X_{j,m}\rrbracket.
\end{equation}
\end{definition}

%\begin{proposition}We have the equality of ideals 
%$$(t,X_i)\cdot \mathcal{O}\mathbb{B}_{\mathrm{dR},U_{ij}}^+ = (t,X_j)\cdot\mathcal{O}\mathbb{B}_{\mathrm{dR},U_{ij}}^+$$
%and 
%\begin{equation}\label{Xij}X_i\cdot \mathcal{O}\mathbb{B}_{\mathrm{dR},U_{ij}} = X_j\cdot \mathcal{O}\mathbb{B}_{\mathrm{dR},U_{ij}}.
%\end{equation}
%\end{proposition}

%\begin{proof}For the first equality, both ideals are equal to $\ker\theta_{\mathrm{dR}}$ by (\ref{ijtriv}). The second equality follows immediately from the first since $\mathcal{O}\mathbb{B}_{\mathrm{dR},U_{ij,m}} = \mathcal{O}\mathbb{B}_{\mathrm{dR},U_{ij,m}}^+[1/t]$. 
%\end{proof}

The next Proposition says that the $\mathcal{O}\mathbb{B}_{\mathrm{dR},U_{ij}}^+$-ideals generated by $X_i$ and $X_j$ for any $1 \le i,j\le \mathbf{n}$ are the same, and the $\mathcal{O}\mathbb{B}_{\mathrm{dR},U_{ij,m}}^+(U_{ij,m})$-ideals generated by $X_{i,m}$ and $X_{j,m}$ for any $1 \le i,j \le \mathbf{n}$ are the same. 

\begin{proposition}\label{Xiidealsameproposition}We have the equality of sheaves of ideals
\begin{equation}\label{Xiidealsame}X_i \cdot \mathcal{O}\mathbb{B}_{\mathrm{dR},U_{ij}}^+ = X_j \cdot \mathcal{O}\mathbb{B}_{\mathrm{dR},U_{ij}}^+
\end{equation}
and for every $m \in \mathbb{Z}_{\ge 0}$, the equality of ideals
\begin{equation}\label{Ximidealsame}X_{i,m} \cdot \mathcal{O}\mathbb{B}_{\mathrm{dR},U_{ij,m}}^+(U_{ij,m}) = X_{j,m} \cdot \mathcal{O}\mathbb{B}_{\mathrm{dR},U_{ij,m}}^+(U_{ij,m}).
\end{equation}
\end{proposition}

\begin{proof}From (\ref{Uidentify}) we get 
$$\mathbb{B}_{\mathrm{dR},U_{ij}}^+\llbracket X_j\rrbracket = \mathcal{O}\mathbb{B}_{\mathrm{dR},U_{ij}}^+ = \mathbb{B}_{\mathrm{dR},U_{ij}}^+\llbracket X_i\rrbracket.$$
Recall that in (\ref{Uidentify}) we have that the natural inclusion $\mathbb{B}_{\mathrm{dR},U_i}^+ \subset \mathcal{O}\mathbb{B}_{\mathrm{dR},U_i}^+$ composed with $\theta_{X_i} : \mathcal{O}\mathbb{B}_{\mathrm{dR},U_i}^+\rightarrow \mathbb{B}_{\mathrm{dR},U_i}^+$ is the identity. Similarly for $\mathbb{B}_{\mathrm{dR},U_j}^+$ and $\theta_{X_j}$. Thus the composition
$$\mathbb{B}_{\mathrm{dR},U_{ij}}^+ \subset \mathcal{O}\mathbb{B}_{\mathrm{dR},U_{ij}}^+ \xrightarrow{\theta_{X_i}} \mathbb{B}_{\mathrm{dR},U_{ij}}^+ \subset \mathcal{O}\mathbb{B}_{\mathrm{dR},U_{ij}}^+ \xrightarrow{\theta_{X_j}} \mathbb{B}_{\mathrm{dR},U_{ij}}^+$$
is the identity, which implies that the composition
$$\mathbb{B}_{\mathrm{dR},U_{ij}}^+ \subset \mathcal{O}\mathbb{B}_{\mathrm{dR},U_{ij}}^+ \twoheadrightarrow \mathcal{O}\mathbb{B}_{\mathrm{dR},U_{ij}}^+/(X_i,X_j),$$
where $(X_i,X_j) := (X_i,X_j) \cdot \mathcal{O}\mathbb{B}_{\mathrm{dR},U_{ij}}^+$, is an injection. On the other hand, this previous map factors through
$$\mathbb{B}_{\mathrm{dR},U_{ij}}^+ = \mathcal{O}\mathbb{B}_{\mathrm{dR},U_{ij}}^+/(X_i) \twoheadrightarrow \mathcal{O}\mathbb{B}_{\mathrm{dR},U_{ij}}^+/(X_i,X_j)$$
where $(X_i) := X_i \cdot \mathcal{O}\mathbb{B}_{\mathrm{dR},U_{ij}}^+$; thus the map is surjective. Therefore the map is both injective and surjective, which implies that it is an isomorphism and $(X_i) = (X_i,X_j)$, which gives (\ref{Xiidealsame}). 

The equality (\ref{Ximidealsame}) follows from essentially the same argument \emph{mutatis mutandis}, replacing $\mathcal{O}\mathbb{B}_{\mathrm{dR},U_{ij}}^+$ by $\mathcal{O}\mathbb{B}_{\mathrm{dR},U_{ij,m}}^+(U_{ij,m})$ and $X_i, X_j$ by $X_{i,m}, X_{j,m}$. 

\end{proof}

Thus, we are able to glue together the $\mathcal{O}\mathbb{B}_{\mathrm{dR},U_i}^+$-ideals generated by the $X_i$ as well as the maps $\theta_{X_i}$ from Definition \ref{'thetadefinitions}. 

\begin{definition}\label{gluingdefinition}\begin{enumerate}
\item Henceforth, let
$$(X) := (X_1,\ldots,X_{\mathbf{n}}) \cdot \mathcal{O}\mathbb{B}_{\mathrm{dR},U}^+ \subset \mathcal{O}\mathbb{B}_{\mathrm{dR},U}^+$$
to be the $\mathcal{O}\mathbb{B}_{\mathrm{dR},U}^+$-ideal generated by all of the $X_i$, $1 \le i \le \mathbf{n}$. By (\ref{Xiidealsame}) we have 
$$(X)|_{U_i} = X_i \cdot \mathcal{O}\mathbb{B}_{\mathrm{dR},U_i}^+.$$
Note that $\left(\mathrm{Fil}^1\mathcal{O}\mathbb{B}_{\mathrm{dR},U}^+\right)\cdot t^{-1}$ is an $\mathcal{O}\mathbb{B}_{\mathrm{dR},U}^+$-module. Let
\begin{equation}\label{thetaX0}\theta_X : \left(\mathrm{Fil}^1\mathcal{O}\mathbb{B}_{\mathrm{dR},U}^+\right)\cdot t^{-1} \rightarrow \mathbb{B}_{\mathrm{dR},U}^+
\end{equation}
be reduction modulo $(X)$; this is a map of $\mathcal{O}\mathbb{B}_{\mathrm{dR},U}^+$-modules. %We have an inclusion
%$$\mathcal{O}\mathbb{B}_{\mathrm{dR},U}^+ \subset \left(\mathrm{Fil}^1\mathcal{O}\mathbb{B}_{\mathrm{dR},U}^+\right)\cdot t^{-1}$$
%of $\mathcal{O}\mathbb{B}_{\mathrm{dR},U}^+$-modules, by restriction we get a map of $\mathcal{O}\mathbb{B}_{\mathrm{dR},U}^+$-modules
%$$\theta_X : \mathcal{O}\mathbb{B}_{\mathrm{dR},U}^+ \rightarrow \mathbb{B}_{\mathrm{dR},U}^+.$$

\item By construction, $\theta_X$ and $\theta_{X_i}$ restrict to the same map on the $\mathcal{O}\mathbb{B}_{\mathrm{dR},U_i}^+$-module $\left(\mathrm{Fil}^1\mathcal{O}\mathbb{B}_{\mathrm{dR},U_i}^+\right)\cdot t^{-1}$:
\begin{equation}\label{thetaX0specialize}\theta_X|_{\left(\mathrm{Fil}^1\mathcal{O}\mathbb{B}_{\mathrm{dR},U_i}^+\right)\cdot t^{-1}} = \theta_{X_i}|_{\left(\mathrm{Fil}^1\mathcal{O}\mathbb{B}_{\mathrm{dR},U_i}^+\right)\cdot t^{-1}}
\end{equation}
for all $1 \le i \le \mathbf{n}$. 

\item For every $m \in \mathbb{Z}_{\ge 0}$, let
$$(X_m) := (X_{1,m},\ldots,X_{\mathbf{n},m})\cdot \mathcal{O}\mathbb{B}_{\mathrm{dR},U_m}^+(U_m) \subset \mathcal{O}\mathbb{B}_{\mathrm{dR},U_m}^+(U_m)$$
be the $\mathcal{O}\mathbb{B}_{\mathrm{dR},U_m}^+(U_m)$-ideal generated by all of the $X_{i,m}$, $1 \le i \le \mathbf{n}$. By (\ref{Ximidealsame}) we have 
$$(X_m)|_{U_{i,m}} = X_{i,m} \cdot \mathcal{O}\mathbb{B}_{\mathrm{dR},U_{i,m}}^+(U_{i,m}).$$
Note that $\left(\mathrm{Fil}^1\mathcal{O}\mathbb{B}_{\mathrm{dR},U_m}^+\right)\cdot t^{-1}$ is an $\mathcal{O}\mathbb{B}_{\mathrm{dR},U_m}^+(U_m)$-module. Let 
\begin{equation}\label{thetaXm}\theta_{X_m} : \left(\mathrm{Fil}^1\mathcal{O}\mathbb{B}_{\mathrm{dR},U_m}^+(U_m)\right)\cdot t^{-1} \rightarrow \mathbb{B}_{\mathrm{dR},U_m}^+
\end{equation}
to be reduction modulo $(X_m)$; this is a map of $\mathcal{O}\mathbb{B}_{\mathrm{dR},U_m}^+(U_m)$-modules. 
%For every $m \in \mathbb{Z}_{\ge 0}$, define a map
%$$\theta_{X_m} : \mathrm{Fil}^1\mathcal{O}\mathbb{B}_{\mathrm{dR},U_m}^+\cdot t^{-1} \rightarrow \mathbb{B}_{\mathrm{dR},U_m}^+$$
%as follows. For any $1 \le i \le \mathbf{n}$, we have an additive map
%\begin{align*}\mathrm{Fil}^1\mathcal{O}\mathbb{B}_{\mathrm{dR},U_m}^+\cdot t^{-1} \rightarrow \mathrm{Fil}^1\mathcal{O}\mathbb{B}_{\mathrm{dR},U_{i,m}}^+\cdot t^{-1} &\overset{(\ref{'triv})}{=} \mathbb{B}_{\mathrm{dR},U_{i,m}}^+ + \frac{X_{i,m}}{t}\cdot\mathbb{B}_{\mathrm{dR},U_{i,m}}^+(U_{i,m})\llbracket X_{i,m} \rrbracket\\
%&\hspace{.19cm}\subset \mathbb{B}_{\mathrm{dR},U_{i,m}}^+ + X_{i,m}\cdot\mathbb{B}_{\mathrm{dR},U_{i,m}}\llbracket X_{i,m}\rrbracket.
%\end{align*}
%Reducing this map modulo $X_{i,m}\cdot\mathbb{B}_{\mathrm{dR},U_{i,m}}\llbracket X_{i,m}\rrbracket$, we get a map
%$$\mathrm{Fil}^1\mathcal{O}\mathbb{B}_{\mathrm{dR},U_m}^+\cdot t^{-1} \rightarrow \mathbb{B}_{\mathrm{dR},U_{i,m}}^+$$
%By (\ref{Xij}), we have 
%$$X_i\cdot\mathbb{B}_{\mathrm{dR},U_{ij,m}}\llbracket X_{i,m}\cdot\rrbracket = X_j\cdot\mathbb{B}_{\mathrm{dR},U_{ij,m}}\llbracket X_{j,m}\rrbracket.$$
%Hence the previous displayed maps glue together for all $1 \le i ,j \le n$ to give a map 
%\begin{equation}\label{thetaXm}\theta_{X_m} : \mathrm{Fil}^1\mathcal{O}\mathbb{B}_{\mathrm{dR},U_m}^+\cdot t^{-1} \rightarrow \mathbb{B}_{\mathrm{dR},U_m}^+,
%\end{equation}
%which is the desired map. When $m = 0$, we let $\theta_X := \theta_{X_m}$ for brevity. 

\item By construction, $\theta_{X_m}$ and $\theta_{X_{i,m}}$ restrict to the same map on the $\mathcal{O}\mathbb{B}_{\mathrm{dR},U_{i,m}}^+(U_{i,m})$-module $\left(\mathrm{Fil}^1\mathcal{O}\mathbb{B}_{\mathrm{dR},U_{i,m}}^+(U_{i,m})\right)\cdot t^{-1}$:
\begin{equation}\label{thetaXmspecialize}\theta_{X_m}|_{\left(\mathrm{Fil}^1\mathcal{O}\mathbb{B}_{\mathrm{dR},U_{i,m}}^+(U_{i,m})\right)\cdot t^{-1}} = \theta_{X_{i,m}}|_{\left(\mathrm{Fil}^1\mathcal{O}\mathbb{B}_{\mathrm{dR},U_{i,m}}^+(U_{i,m})\right)\cdot t^{-1}}
\end{equation}
for all $1\le i \le \mathbf{n}$. 
\end{enumerate}
\end{definition}

We can similarly glue together all of the $\mathbb{B}_{\mathrm{dR},U_i}\llbracket X\rrbracket$ and $\mathbb{B}_{\mathrm{dR},U_i}^+\llbracket X_i/t\rrbracket$ from Definition \ref{BdR+Xtdefinition}. 

\begin{definition}\label{newcompletionsdefinition}Retain the notation of Definition \ref{gluingdefinition}.
\begin{enumerate}
\item Define the sheaf on $Y_{\text{pro\'{e}t}}/U$
$$\mathbb{B}_{\mathrm{dR},U}\llbracket X\rrbracket$$
to be the $(X)$-adic completion of $\mathcal{O}\mathbb{B}_{\mathrm{dR},U}$. 
\item Similarly, for any pro\'{e}tale open $W \rightarrow U_m$, define the ring
$$\mathbb{B}_{\mathrm{dR},U_m}(W)\llbracket X_m\rrbracket$$
to be the $(X_m)$-adic completion of $\mathcal{O}\mathbb{B}_{\mathrm{dR},U_m}(W)$. 
\item Define the sheaf on $Y_{\text{pro\'{e}t}}/U$ 
$$\mathbb{B}_{\mathrm{dR},U}^+\llbracket X/t\rrbracket$$
to be the subsheaf of $\mathbb{B}_{\mathrm{dR},U}\llbracket X\rrbracket$ obtained by gluing together the sheaves $\mathbb{B}_{\mathrm{dR},U_i}^+\llbracket X_i/t\rrbracket$ on $Y_{\text{pro\'{e}t}}/U_i$ over all $1 \le i \le \mathbf{n}$, using (\ref{Xiidealsame}) and the fact that $U \overset{(\ref{etalelocusunion3})}{=} \bigcup_{i = 1}^{\mathbf{n}}U_i$. In particular, since $\mathcal{O}\mathbb{B}_{\mathrm{dR},U_i}\overset{(\ref{OBdRinclusion})}{\subset} \mathbb{B}_{\mathrm{dR},U_i}\llbracket X_i/t\rrbracket$ for every $1 \le i \le \mathbf{n}$, we have 
\begin{equation}\label{newOBdRinclusion}\mathcal{O}\mathbb{B}_{\mathrm{dR},U} \subset \mathbb{B}_{\mathrm{dR},U}^+\llbracket X/t\rrbracket.
\end{equation}
\item Finally, for any pro\'{e}tale open $W \rightarrow U_m$, we define the ring
$$\mathbb{B}_{\mathrm{dR},U_m}^+(W)\llbracket X_m/t\rrbracket,$$
to be the $(X_m/t) := (X_{1,m}/t,\ldots,X_{\mathbf{n},m}/t)\cdot \mathcal{O}\mathbb{B}_{\mathrm{dR},U_m}^+(W)$-adic completion of 
$$\mathbb{B}_{\mathrm{dR},U_m}^+(W)[X_{1,m}/t,\ldots,X_{\mathbf{n},m}/t] \subset \mathcal{O}\mathbb{B}_{\mathrm{dR},U_m}(W).$$
From construction, we have an inclusion 
\begin{equation}\label{newOBdRinclusionm}\mathcal{O}\mathbb{B}_{\mathrm{dR},U}(W) \subset \mathbb{B}_{\mathrm{dR},U}^+(W)\llbracket X/t\rrbracket.
\end{equation}
\end{enumerate}

%but its definition is slightly more subtle as we only define the ring $\mathbb{B}_{\mathrm{dR},U_{i,m}}(U_{i,m})\llbracket X_{i,m}\rrbracket$ (Definition \ref{BdR+Xtdefinition}) and not the sheaf $\mathbb{B}_{\mathrm{dR},U_{i,m}}\llbracket X_{i,m}\rrbracket$. Moreover, unlike $\mathbb{B}_{\mathrm{dR},U_m}(V)\llbracket X_m\rrbracket$ above, we cannot define $\mathbb{B}_{\mathrm{dR},U_m}

%Given a pro\'{e}tale open $V \rightarrow U_{i,m}$, let $\mathbb{B}_{\mathrm{dR},U_{i,m}}^+(V)\llbracket X_{i,m}/t\rrbracket$ be the formal power series ring in $X_{i,m}/t$ over $\mathbb{B}_{\mathrm{dR},U_{i,m}}^+(V)$. 

%to be the subring of the product ring
%$$\prod_{i = 1}^{\mathbf{n}}\mathbb{B}_{\mathrm{dR},U_{i,m}}^+(U_{i,m})\llbracket X_{i,m}/t\rrbracket$$
%given by
%$$\left\{(F_1,\ldots,F_n) \in \prod_{i = 1}^{\mathbf{n}}\mathbb{B}_{\mathrm{dR},U_{im}}^+(U_{i,m})\llbracket X_{i,m}/t\rrbracket : F_i|_{U_{ij,m}} = F_j|_{U_{ij,m}} \; \forall \; 1\le i,j \le n\right\}.$$

\end{definition}

\begin{remark}Note that despite the notation in Definition \ref{newcompletionsdefinition}, $X$ and $X_m$ do not denote formal variables, but are simply part of the notation for the sheaves and rings defined. On the other hand, the $X_{i,m}$ are the formal variables defined in (\ref{Xm}). When $D = M_2(\mathbb{Q})$, so that $\mathbf{n} = 1$ in Definition \ref{ndefinition} and $i = 1$, then we may take $X = X_1$, $X_m = X_{1,m}$ and thus treat $X$ and $X_m$ as actual formal variables in this case. This is compatible with Remark \ref{Uiremark}.
\end{remark}

\subsection{Inverting $y_{\mathrm{dR}}$}

In this section, we show that $y_{\mathrm{dR}}(w) \in \left(\mathrm{Fil}^1\mathcal{O}\mathbb{B}_{\mathrm{dR},U_m}^+(U_m) \right)\cdot t^{-1}$ is multiplicatively invertible in certain period rings defined over a certain open subset $U_{m,w}' \subset U_m$. This will be needed in order to define the de Rham period $z_{\mathrm{dR}}$ later (Definition \ref{'zqwdefinition} and (\ref{gluezdRsection})). 

\begin{definition}Let $w \in \omega(U_m)$ be as in Choice \ref{'choicew}, with associated $y_{\mathrm{dR}}(w)$ from Definition \ref{'xydefinition}. Define 
\begin{equation}\label{'U'definition}U_{i,m,w}' := \{\theta(\theta_{X_{i,m}}(y_{\mathrm{dR}}(w))) \neq 0\} \subset U_{i,m}, \hspace{1cm} U_{m,w}' = \bigcup_{i = 1}^{\mathbf{n}}U_{i,m,w}' \overset{(\ref{thetaXmspecialize})}{=} \{\theta(\theta_{X_m}(y_{\mathrm{dR}}(w))) \neq 0\}.
\end{equation}
\end{definition}

%\begin{convention}\label{Uimconvention}Note that $U_{i,m}'$ depends $y_{\mathrm{dR}}$ and thus on a choice of section of $\omega$ as in Choice \ref{'choicew}. When we wish to emphasize this dependence on $w' \in \omega(U_m)$, we write $U_{i,m,w'}'$ instead of $U_{i,m}'$. Similarly, we will sometimes write $U_m' = \bigcup_{i = 1}^{\mathbf{n}}U_{i,m}'$ as $U_{m,w'}'$ when we wish to emphasize the dependence on $w'$. In particular,
%$$U_{m,w'}' = \bigcup_{i =1}^nU_{i,m,w'}'.$$ 
%\end{convention}

\begin{remark}Later we will consider the union of all $U_{i,m,w}'$,
$$\mathbf{U}' := \bigcup_{m \in \mathbb{Z}_{\ge 0}}\bigcup_{w \in \omega(U_m)}U_{m,w}'$$
and after some further setup we will show
$$\mathbf{U}' = \mathcal{V}_x.$$
See (\ref{fullU'}), where we formally introduce $\mathbf{U}'$, and Theorem \ref{U'Utheorem} where we prove it is in fact equal to $\mathcal{V}_x$. However, we will need results concerning the behavior of $y_{\mathrm{dR}}$ on the ordinary locus (see Corollary \ref{U'YIg}) before we can prove this Theorem, and so we will state many results in this section over $\mathbf{U}'$, and then generalize them to $\mathcal{V}_x$ later.
\end{remark}

 Recall that $U_m \subset U_{m+1}$ for all $m\in \mathbb{Z}_{\ge 0}$ by (\ref{Uminclusion}). The next result shows that this is true for the open subsets $U_{m,w}' \subset U_{m+1,w}'$ as well. 

\begin{proposition}\label{Um'proposition}For all $m \in \mathbb{Z}_{\ge 0}$ and any $w \in \omega(U_{m+1})$ as in Choice \ref{'choicew}, we have $U_{m,w}' \subset U_{m+1,w}'$. 
\end{proposition}

\begin{proof}%Recall 
%$$y_{\mathrm{dR}}(w') \in \ker(\theta_{\mathrm{dR}})\mathcal{O}\mathbb{B}_{\mathrm{dR},U_{m+1}}^+(U_{m+1}) \cdot t^{-1} \subset \mathbb{B}_{\mathrm{dR},U_{m+1}}^+(U_{m+1}) + \frac{X_{m+1}}{t}\mathbb{B}_{\mathrm{dR},U_{m+1}}^+(U_{m+1})\llbracket X_{m+1}/t\rrbracket.$$
We need to show that $\theta(\theta_{X_m}(y_{\mathrm{dR}}(w)|_{U_m})) \neq 0$ implies 
$$\theta(\theta_{X_{m+1}}(y_{\mathrm{dR}}(w)))|_{U_m} \neq 0.$$
We will show that $X_{i,m+1}|_{U_m} \in \ker(\theta\circ \theta_{X_m})$ for all $1 \le i \le \frak{n}$. Admitting this, we then get 
$$\ker(\theta \circ \theta_{X_{m+1}}|_{U_m}) \subset \ker(\theta \circ \theta_{X_m})$$
from Definition \ref{'thetadefinitions}, which gives the assertion.

%Note that 
%$$0 \neq \theta(\theta_{X_m}(y_{\mathrm{dR}}(w)|_{U_m})) \equiv y_{\mathrm{dR}}(w)|_{U_m} \pmod{(t,X_m)}$$
%where for this proof let $(t,X_m) := t\cdot \mathbb{B}_{\mathrm{dR},U_m}^+(U_m) + X_m \cdot \mathcal{O}\mathbb{B}_{\mathrm{dR},U_m}^+(U_m) \cdot t^{-1} = \ker(\theta_{X_m})$. To conclude, we first show that 
%$$X_{m+1}|_{U_m} \in (t,X_m).$$
%Admitting this, we then have $(t,X_{m+1}|_{U_m}) \subset (t,X_m)$, where $(t,X_{m+1}|_{U_m}) := t\cdot \mathbb{B}_{\mathrm{dR},U_m}^+(U_m) + X_{m+1}|_{U_m}\cdot \mathcal{O}\mathbb{B}_{\mathrm{dR},U_m}^+(U_m) \cdot t^{-1} = \ker(\theta_{X_{m+1}}|_{U_m})$. This combined with the above congruence then implies
%$$\theta(\theta_{X_{m+1}}(y_{\mathrm{dR}}(w)))|_{U_m} \not\equiv 0 \pmod{(t,X_{m+1}|_{U_m})}$$
%and so $\theta(\theta_{X_{m+1}}(y_{\mathrm{dR}}(w)))|_{U_m} \neq 0$.  

Recall $X_{i,m+1}|_{U_m} = (g^{-1})^*X_{i,m}$ by (\ref{Xm}). Since $(g^{-1})^*X_{i,m} \in \mathcal{O}\mathbb{B}_{\mathrm{dR},U_m}^+(U_m)$, we have
\begin{align*}\theta(\theta_{X_m}(X_{i,m+1}|_{U_m})) = \theta(\theta_{X_m}((g^{-1})^*X_{i,m})) &\overset{\text{Proposition \ref{'thetaproposition} }}{=} \theta_{\mathrm{dR}}((g^{-1})^*X_{i,m}) \\
&\overset{(\ref{Xm})}{=} \theta_{\mathrm{dR}}((g^{-(m+1)})^*j - (g^{-(m+1)})^*[j_i^{\flat}]) \\
&\overset{(\ref{jflat})}{=} (g^{-(m+1)})^*j_i - (g^{-(m+1)})^*j_i  = 0
\end{align*}
which implies $X_{i,m+1}|_{U_m} \in \ker(\theta \circ \theta_{X_m})$. By the previous paragraph, we are done.

\end{proof}

\begin{proposition}\label{'invertibleproposition}$$\theta_{X_m}(y_{\mathrm{dR}}(w)) \in \mathbb{B}_{\mathrm{dR},U_m}^+(U_{m,w}')^{\times}.$$
\end{proposition}

\begin{proof}Since $U_{m,w}' = \bigcup_{i = 1}^{\mathbf{n}}U_{i,m,w}'$ and $\mathbb{B}_{\mathrm{dR},U_m}^+$ is a sheaf of rings on $Y_{\text{pro\'{e}t}}/U_m$, by (\ref{thetaXmspecialize}) it suffices to show that
$$\theta_{X_{i,m}}(y_{\mathrm{dR}}(w)) \in \mathbb{B}_{\mathrm{dR},U_{i,m}}^+(U_{i,m,w}')^{\times}$$
for all $1 \le i \le \mathbf{n}$. 

Given any $s \in U_{i,m,w}'$, the stalk $\mathbb{B}_{\mathrm{dR},U_{i,m,w}',s}^+$ is a local ring with maximal ideal $\theta^{-1}(\frak{m}_s)$, where $\frak{m}_s \subset \hat{\mathcal{O}}_{U_{i,m,w}',s}$ is the maximal ideal. Since $s \in U_{i,m,w}'$, $\theta(\theta_{X_{i,m}}(y_{\mathrm{dR}}(w)))(s) \neq 0$, and so $\theta_{X_{i,m}}(y_{\mathrm{dR}}(w))_s \in (\mathbb{B}_{\mathrm{dR},U_{i,m,w}',s}^+)^{\times}$. Since $\mathbb{B}_{\mathrm{dR},U_{i,m,w}'}^+$ is a sheaf, we are done.

%Recall $U' \rightarrow U \subset Y_{\infty}$ is pro\'{e}tale, and thus $U'$ is a perfectoid object in $Y_{\text{pro\'{e}t}}/U$. Let $\hat{U}' \sim U'$ be its associated perfectoid space, whose adic structure sheaf we denote by $\mathcal{O}_{\hat{U}'}$. Then $\mathcal{O}_{\hat{U}'}$ corresponds to $\mathcal{O}_{U'}$ under $U' \sim \hat{U}'$. Let $\mathbb{B}_{\mathrm{dR},\hat{U}'}^+$ denote the sheaf on the adic site $\hat{U}_{\mathrm{ad}}'$ associated to $\mathbb{B}_{\mathrm{dR},U'}^+$ under $U' \sim \hat{U}'$. Then $\theta : \mathbb{B}_{\mathrm{dR},\hat{U}'}^+ \twoheadrightarrow \mathcal{O}_{\hat{U}'}$. For any $y \in \hat{U}'$, let $\frak{m}_y \subset \mathcal{O}_{\hat{U}',y}$ be the maximal ideal and note that $\mathbb{B}_{\mathrm{dR},\hat{U}',y}^+$ is a local ring with maximal ideal $\theta^{-1}(\frak{m}_y)$. Now $\theta_X(y_{\mathrm{dR}})(y) \neq 0$ for all $y \in \hat{U}'$, and so $\theta_X(y_{\mathrm{dR}})_y \in (\mathbb{B}_{\mathrm{dR},\hat{U}',y}^+)^{\times}$ for all $y \in \hat{U}'$. Thus $\theta_X(y_{\mathrm{dR}})|_{V_y} \in (\mathbb{B}_{\mathrm{dR},\hat{U}'}^+(V_y))^{\times}$ for some open neighborhood $V_y \subset \hat{U}'$ of $y$ for all $y \in \hat{U}'$, and so gluing on $\hat{U}_{\mathrm{ad}}'$ we get $\theta_X(y_{\mathrm{dR}}) \in \mathbb{B}_{\mathrm{dR},\hat{U}'}^+(\hat{U}')^{\times} = \mathbb{B}_{\mathrm{dR},U'}^+(U')^{\times}$. 
\end{proof}

\begin{corollary}\label{'unitcorollary}
$$y_{\mathrm{dR}}(w) \in \mathbb{B}_{\mathrm{dR},U_m}^+(U_{m,w}')\llbracket X_m/t\rrbracket^{\times}.$$
\end{corollary}
\begin{proof}Since $\mathbb{B}_{\mathrm{dR},U_m}^+\llbracket X_m/t\rrbracket$ is a sheaf on $Y_{\text{pro\'{e}t}}/U_m$, it suffices to show that 
$$y_{\mathrm{dR}}(w) \in \mathbb{B}_{\mathrm{dR},U_m}^+(U_{i,m,w}')\llbracket X_m/t\rrbracket^{\times}$$
for all $1 \le i \le \mathbf{n}$. This follows from Proposition \ref{'invertibleproposition} and the fact that 
$$y_{\mathrm{dR}}(w) \in \mathcal{O}\mathbb{B}_{\mathrm{dR},U_{i,m}}^+(U_{i,m}) \cdot t^{-1} = \mathbb{B}_{\mathrm{dR},U_{i,m}}^+(U_{i,m})\llbracket X_{i,m}\rrbracket \cdot t^{-1} \subset \mathbb{B}_{\mathrm{dR},U_{i,m}}^+(U_{i,m})\llbracket X_{i,m}/t\rrbracket.$$
\end{proof}

%Similarly, observe
%$$x_{\mathrm{dR}} \in \mathcal{O}\mathbb{B}_{\mathrm{dR},U_{i,m}}^+(U_{i,m}) \cdot t^{-1} = \mathbb{B}_{\mathrm{dR},U_{i,m}}^+(U_{i,m})\llbracket X_{i,m}\rrbracket \cdot t^{-1} \subset \mathbb{B}_{\mathrm{dR},U_{i,m}}^+(U_{i,m})\llbracket X_{i,m}/t\rrbracket.$$

\begin{lemma}The inclusion 
$$\mathcal{O}_{\mathcal{V}_x}(U_{i,m,w}')^{\times} \subset \mathcal{O}\mathbb{B}_{\mathrm{dR},U_{i,m}}^+(U_{i,m,w}') = \mathbb{B}_{\mathrm{dR},\mathcal{V}_x}^+(U_{i,m,w}')\llbracket X_{i,m}\rrbracket$$
factors through
\begin{equation}\label{'laterinclusion}\mathcal{O}_{\mathcal{V}_x}(U_{i,m,w}')^{\times} \subset \mathbb{B}_{\mathrm{dR},U_{i,m}}^+(U_{i,m,w}')^{\times} + X_{i,m}\cdot\mathbb{B}_{\mathrm{dR},U_{i,m}}^+(U_{i,m,w}')\llbracket X_{i,m}\rrbracket \subset \mathbb{B}_{\mathrm{dR},U_{i,m}}^+(U_{i,m,w}')\llbracket X_{i,m}\rrbracket.
\end{equation}
\end{lemma}

\begin{proof}Recall $\mathcal{O}_{U_{i,m}} \subset \mathcal{O}\mathbb{B}_{\mathrm{dR},U_{i,m}}^+ \xrightarrow{\theta_{\mathrm{dR}}} \hat{\mathcal{O}}_{U_{i,m}}$ is the natural $p$-adic completion map. Suppose $f \in \mathcal{O}_{U_{i,m}}(U_{i,m,w}')^{\times} \subset \mathcal{O}\mathbb{B}_{\mathrm{dR},U_{i,m}}^+(U_{i,m,w}')$. Then $\theta(\theta_{X_{i,m}}(f)) = \theta_{\mathrm{dR}}(f) \in \hat{\mathcal{O}}_{U_{i,m}}(U_{i,m,w}')^{\times}$, and thus $\theta_{X_{i,m}}(f) \in \mathbb{B}_{\mathrm{dR},U_{i,m}}^+(U_{i,m,w}')^{\times}$ by the same argument as in the proof of Proposition \ref{'invertibleproposition}.

%we have 
%$$\mathcal{O}_U(U) \subset \mathbb{B}_{\mathrm{dR},U}^+(U)^{\times} + X\mathbb{B}_{\mathrm{dR},U}^+(U)\llbracket X\rrbracket \subset \mathbb{B}_{\mathrm{dR},U}^+(U)\llbracket X\rrbracket.$$
\end{proof}

\subsection{Trivializing $\mathbb{B}_{\mathrm{dR}}^+(U_m)$}The goal of this section is to show how the period ring $\mathbb{B}_{\mathrm{dR},U_m}^+(U_m)$ is isomorphic to $\hat{\mathcal{O}}_{U_m}(U_m)\llbracket t\rrbracket$. This isomorphism will be needed later to define a splitting (\ref{split}) of the relative Hodge filtration (\ref{Hodgefiltrationinclusion}), which we in turn use to define our $p$-adic Maass-Shimura operators (Definition \ref{padicMSdefinition}). The construction of this isomorphism essentially uses the affinoid perfectoidness of $U_m$.

\begin{corollary}\label{OBmapcorollary}\begin{enumerate}
\item There is a natural map of sheaves on $Y_{\text{pro\'{e}t}}/U$
\begin{equation}\label{OBmap}\mathcal{O}_U \rightarrow \mathbb{B}_{\mathrm{dR},U}^+
\end{equation}
whose composition with $\theta : \mathbb{B}_{\mathrm{dR},U}^+ \rightarrow \hat{\mathcal{O}}_U$ is the natural $p$-adic completion map $\mathcal{O}_U \rightarrow \hat{\mathcal{O}}_U$. This induces a map of $\hat{\mathcal{O}}_U(U)$-aglebras 
\begin{equation}\label{completedOBmap}\iota_U : \hat{\mathcal{O}}_U(U) \rightarrow \mathbb{B}_{\mathrm{dR},U}^+(U)
\end{equation}
which is a section of $\theta : \mathbb{B}_{\mathrm{dR},U}^+(U) \rightarrow \hat{\mathcal{O}}_U(U)$. This in turn induces an isomorphism of $\hat{\mathcal{O}}_U(U)$-algebras
\begin{equation}\label{Bdecomposition}\mathbb{B}_{\mathrm{dR},U}^+(U) \cong \hat{\mathcal{O}}_U(U)\llbracket t\rrbracket
\end{equation}
where $t \in \mathrm{Fil}^1\mathbb{B}_{\mathrm{dR},Y}^+(Y_{\infty})$ is as in (\ref{'tdefinition}).

\item Similarly, for any $m \in \mathbb{Z}_{\ge 0}$, letting $U_m$ be as in (\ref{Um}) there is a natural map of rings
\begin{equation}\label{OBmap'}\mathcal{O}_{U_m}(U_m) \rightarrow \mathbb{B}_{\mathrm{dR},U_m}^+(U_m)
\end{equation}
whose composition with $\theta : \mathbb{B}_{\mathrm{dR},U_m}^+(U_m) \rightarrow \hat{\mathcal{O}}_{U_m}(U_m)$ is the natural $p$-adic completion map $\mathcal{O}_{U_m}(U_m) \rightarrow \hat{\mathcal{O}}_{U_m}(U_m)$. This induces a map of $\hat{\mathcal{O}}_{U_m}(U_m)$-algebras
\begin{equation}\label{completedOBmap'}\iota_{U_m} : \hat{\mathcal{O}}_{U_m}(U_m) \rightarrow \mathbb{B}_{\mathrm{dR},U_m}^+(U_m)
\end{equation}
which is a section of $\theta : \mathbb{B}_{\mathrm{dR},U_m}^+(U_m) \rightarrow \hat{\mathcal{O}}_{U_m}(U_m)$. This in turn induces an isomorphism of $\hat{\mathcal{O}}_{U_m}(U_m)$-algebras
\begin{equation}\label{Bdecomposition'}\mathbb{B}_{\mathrm{dR},U_m}^+(U_m) \cong \hat{\mathcal{O}}_{U_m}(U_m)\llbracket t\rrbracket.
\end{equation}
\end{enumerate}
\end{corollary}

\begin{proof}\textbf{(1)}: Recall $\theta_X$ from (\ref{thetaX0}). Observe that there is a natural map $\mathcal{O}\mathbb{B}_{\mathrm{dR},U}^+ \rightarrow \left(\mathrm{Fil}^1\mathcal{O}\mathbb{B}_{\mathrm{dR},U}^+\right)\cdot t^{-1}$ provided by gluing together the maps
$$\mathcal{O}\mathbb{B}_{\mathrm{dR},U}^+ \rightarrow \mathcal{O}\mathbb{B}_{\mathrm{dR},U_i}^+ \overset{(\ref{Uidentify})}{=} \mathbb{B}_{\mathrm{dR},U_i}^+\llbracket X_i\rrbracket =\left(t\cdot\mathbb{B}_{\mathrm{dR},U_i}^+ + tX_i\cdot\mathbb{B}_{\mathrm{dR},U_i}^+\llbracket X \rrbracket\right)\cdot t^{-1}\subset \mathrm{Fil}^1\mathcal{O}\mathbb{B}_{\mathrm{dR},U_i}^+\cdot t^{-1}$$
for $1 \le i \le \mathbf{n}$. Precomposing $\theta_X : \left(\mathrm{Fil}^1\mathcal{O}\mathbb{B}_{\mathrm{dR},U}^+\right)\cdot t^{-1} \rightarrow \mathbb{B}_{\mathrm{dR},U}^+$ with this map, we get a map 
$$\theta_X : \mathcal{O}\mathbb{B}_{\mathrm{dR},U}^+ \rightarrow \mathbb{B}_{\mathrm{dR},U}^+.$$
The map (\ref{OBmap}) is provided composition
\begin{equation}\label{compositionU}\mathcal{O}_U \rightarrow \mathcal{O}\mathbb{B}_{\mathrm{dR},U}^+ \xrightarrow{\theta_X}\mathbb{B}_{\mathrm{dR},U}^+.
\end{equation}

By (\ref{completionscoincide}), recalling that $U = \mathcal{Y}^{\mathrm{Ig}}(\epsilon_0)$ from Definition \ref{Udefinition}, we have 
\begin{equation}\label{completionU}\hat{\mathcal{O}}_U^+(U) = \widehat{\mathcal{O}_U^+(U)} := \varprojlim_n \mathcal{O}_U^+(U)/p^n.
\end{equation}
This implies that $W(\hat{\mathcal{O}}_U^{+,\flat})(U) = W(\hat{\mathcal{O}}_U^{+,\flat}(U))$ is $p$-adically complete, where the equality follows from \cite[discussion above Lemma 6.3]{Scholze}. We have
$$\mathbb{B}_{\mathrm{dR},U}^+(U) = \varprojlim_n W(\hat{\mathcal{O}}_U^{+,\flat})(U)[1/p]/(\ker\theta)^n= \varprojlim_n W(\hat{\mathcal{O}}_U^{+,\flat}(U))[1/p]/(\ker\theta)^n$$
also from loc. cit. Thus for every $n \in \mathbb{Z}_{\ge 0}$, the evaluation of (\ref{compositionU}) at $U \in Y_{\text{pro\'{e}t}}$ induces
$$\mathcal{O}_U(U) \rightarrow \mathbb{B}_{\mathrm{dR},U}^+(U) = \varprojlim_n W(\hat{\mathcal{O}}_U^{+,\flat}(U))[1/p]/(\ker\theta)^n \rightarrow W(\hat{\mathcal{O}}_U^{+,\flat}(U))[1/p]/(\ker\theta)^n$$
where $\theta : W(\hat{\mathcal{O}}_U^{+,\flat}(U))[1/p] \twoheadrightarrow \hat{\mathcal{O}}_U(U)$ is the map from Definition 6.1 of op. cit. The image of the previous displayed map is contained in the ring $\left(W(\hat{\mathcal{O}}_U^{+,\flat}(U))/(\ker\theta)^n\right)[\frac{t}{p^k}]$ for some $k \in \mathbb{Z}_{\ge 0}$, where $t \in \ker\theta(U)$ is as in (\ref{'tdefinition}),
%\cite[Lemma 6.3]{Scholze} (recall $U$ is affinoid perfectoid by Lemma \ref{Uaffinoidperfectoidlemma})
and this latter ring is $p$-adically complete (cf. Proof of Lemma 6.11 of op. cit.). 
Thus, the previous displayed map extends to the $p$-adic completion 
$$\hat{\mathcal{O}}_U(U) \overset{(\ref{completionU})}{=} \widehat{\mathcal{O}_U^+(U)}[1/p] \rightarrow W(\hat{\mathcal{O}}_U^{+,\flat}(U))[1/p]/(\ker\theta)^n.$$
These maps glue over all $n$ to give (\ref{completedOBmap}). By construction, it is a section of $\theta : \mathbb{B}_{\mathrm{dR},U}^+(U)\rightarrow \hat{\mathcal{O}}_U(U)$.

%is $p$-adically continuous, since $\mathcal{O}_U(U) \rightarrow \mathcal{O}\mathbb{B}_{\mathrm{dR},U}^+(U)$ is $p$-adically continuous by \cite[(3)]{Scholzecorrigendum} and the reduction modulo $(X)$ map $\mathbb{B}_{\mathrm{dR},U}^+(U)\llbracket X\rrbracket \rightarrow \mathbb{B}_{\mathrm{dR},U}^+(U)$ is clearly $p$-adically continuous. Hence (\ref{completedOBmap}) is naturally defined by extending the evaluation of (\ref{compositionU}) at $U$ via $p$-adic continuity. 

Finally, we prove (\ref{Bdecomposition}). We have a natural map induced by (\ref{completedOBmap}):
$$i_U : \hat{\mathcal{O}}_U(U)\llbracket t\rrbracket \rightarrow \mathbb{B}_{\mathrm{dR},U}^+(U), \hspace{1cm} i_U\left(\sum_{n = 0}^{\infty}\alpha_n\cdot t^n\right) := \sum_{n = 0}^{\infty} \iota_U(\alpha_n) \cdot t^n,$$
which is well-defined since by \cite[discussion above Lemma 6.3]{Scholze} (using the fact that $U$ is affinoid perfectoid)
$$\mathbb{B}_{\mathrm{dR},U}^+(U) = \varprojlim_n \mathbb{B}_{\mathrm{inf},U}(U)/t^n,$$
which is complete with respect to the $t$-adic topology. In the other direction, we define a map 
$$\vartheta_U: \mathbb{B}_{\mathrm{dR},U}^+(U) \rightarrow \hat{\mathcal{O}}_U(U) \llbracket t\rrbracket$$
by
$$f \mapsto \sum_{n = 0}^{\infty}\alpha_n \cdot t^n, \hspace{1cm} \alpha_0 = \theta(f), \hspace{.25cm} \alpha_N = \theta\left(\frac{f - \sum_{n = 0}^{N-1}\iota_U(\alpha_n) \cdot t^n}{t^N}\right), \; N\ge 1.$$
Since $\iota_U$ is a section of $\theta$, $i_U$ and $\vartheta_U$ are inverses of each other. \\

\textbf{(2)}: The assertions (\ref{OBmap'}), (\ref{completedOBmap'}) and (\ref{Bdecomposition'}) are proven completely analogously: Recall $\theta_{X_m}$ from (\ref{thetaXm}). Gluing together the maps 
\begin{align*}\mathcal{O}\mathbb{B}_{\mathrm{dR},U_m}^+(U_m) \rightarrow \mathcal{O}\mathbb{B}_{\mathrm{dR},U_m}^+(U_{i,m}) &\overset{(\ref{Uidentify})}{=} \mathbb{B}_{\mathrm{dR},U_{i,m}}^+(U_{i,m})\llbracket X_{i,m}\rrbracket \\
&=\left(t\cdot\mathbb{B}_{\mathrm{dR},U_{i,m}}^+(U_{i,m}) + tX_{i,m}\cdot\mathbb{B}_{\mathrm{dR},U_{i,m}}^+(U_{i,m})\llbracket X \rrbracket\right)\cdot t^{-1}\\
&\subset \mathrm{Fil}^1\mathcal{O}\mathbb{B}_{\mathrm{dR},U_{i,m}}^+(U_{i,m})\cdot t^{-1}
\end{align*}
for $1 \le i \le \mathbf{n}$, we get a natural map $\mathcal{O}\mathbb{B}_{\mathrm{dR},U_m}^+(U_m) \rightarrow \left(\mathrm{Fil}^1\mathcal{O}\mathbb{B}_{\mathrm{dR},U_m}^+(U_m)\right)\cdot t^{-1}$. Precomposing 
$$\theta_{X_m} : \mathrm{Fil}^1\mathcal{O}\mathbb{B}_{\mathrm{dR},U_m}^+(U_m) \cdot t^{-1} \rightarrow \mathbb{B}_{\mathrm{dR},U_m}^+(U_m)$$
with this map, we get a map 
$$\theta_{X_m} : \mathcal{O}\mathbb{B}_{\mathrm{dR},U_m}^+(U_m) \rightarrow \mathbb{B}_{\mathrm{dR},U_m}^+(U_m).$$ Now use the composition
$$\mathcal{O}_{U_m}(U_m) \rightarrow \mathcal{O}\mathbb{B}_{\mathrm{dR},U_m}^+(U_m) \xrightarrow{\theta_{X_m}}\mathbb{B}_{\mathrm{dR},U_m}^+(U_m)$$
to define (\ref{OBmap'}). One can $p$-adically complete this to get (\ref{completedOBmap'}), and then (\ref{Bdecomposition'}) is proven using the same argument as in the previous paragraph substituting $U$ with $U_m$.

\end{proof}

\subsection{$z_{\mathrm{dR}}, q_{\mathrm{dR}}, w_{\mathrm{can}}$}\label{zqwsection}

Recall $U_{m,w}'$ from (\ref{'U'definition}) and $x_{\mathrm{dR}}(w), y_{\mathrm{dR}}(w)$ from Definition \ref{'xydefinition}, all depending on the choice of $w \in \omega(U_m)$ made in Choice \ref{'choicew}. %(To emphasize this dependence on $w$, we will sometimes replace $U_m', x_{\mathrm{dR}}$ and $y_{\mathrm{dR}}$ by $U_{m,w}', x_{\mathrm{dR}}(w)$ and $y_{\mathrm{dR}}(w)$, respectively, when invoking Definition \ref{'zqwdefinition}, as per Conventions \ref{xyconvention} and \ref{Uimconvention}.)

\begin{definition}\label{'zqwdefinition} For any $m \in \mathbb{Z}_{\ge 0}$ let
$$z_{\mathrm{dR}}(w) := -\frac{x_{\mathrm{dR}}(w)}{y_{\mathrm{dR}}(w)} \in \mathbb{B}_{\mathrm{dR},U_m
}^+(U_{m,w}')\llbracket X_m/t\rrbracket,$$
$$w_{\mathrm{can}}(w) := \frac{w}{y_{\mathrm{dR}}(w)} \in \omega \otimes_{\mathcal{O}_Y}\mathbb{B}_{\mathrm{dR},U_m}^+(U_{m,w}')\llbracket X_m/t\rrbracket.$$
Here we use Corollary \ref{'unitcorollary} in order to invert $y_{\mathrm{dR}}(w)$ on $U_{m,w}'$. Recall the comparison map $i_{\mathrm{dR}}$ from (\ref{comparisonmap2}).
Note that, by Definition \ref{'xydefinition} (1), 
\begin{equation}\label{'iwcan}i_{\mathrm{dR}}(w_{\mathrm{can}}(w)) = -z_{\mathrm{dR}}(w)e_1 + e_2.
\end{equation}
\end{definition}

\begin{remark}We will later show that $z_{\mathrm{dR}}(w)$ and $w_{\mathrm{can}}(w)$ are independent of $w \in \omega(U_m)$ and extend to sections defined on $\mathcal{V}_x$. See (\ref{zqw2}).
\end{remark}

%Applying $\theta \circ \theta_{X_m} : \mathbb{B}_{\mathrm{dR},U_m}^+(U_m')\llbracket X_m/t\rrbracket \rightarrow \hat{\mathcal{O}}_{U_m}(U_m')$ to $w_{\mathrm{can}}$, we get
%$$\theta(\theta_{X_m}(w_{\mathrm{can}})) \in \omega \otimes_{\mathcal{O}_Y}\hat{\mathcal{O}}_{U_m}(U_m').$$

%\begin{proposition}We have 
%$$U_m' = \{\theta(\theta_{X_m}(w_{\mathrm{can}})) \neq 0\}.$$
%\end{proposition}

%\begin{proof}
%\end{proof}

\begin{definition}\label{expalphadefinition}\begin{enumerate}
\item In the notation of Definition \ref{newcompletionsdefinition}, let
$$(t,(X_m/t))\cdot \mathbb{B}_{\mathrm{dR},U_m}^+(U_{m,w}')\llbracket X_m/t\rrbracket := t\cdot\mathbb{B}_{\mathrm{dR},U_m}^+(U_{m,w}')\llbracket X_m/t\rrbracket + (X_m/t)\cdot \mathbb{B}_{\mathrm{dR},U_m}^+(U_{m,w}')\llbracket X_m/t\rrbracket.$$
\item Define
\begin{equation}\label{zbfdefinition}\mathbf{z}(w) := t\cdot z_{\mathrm{dR}}(w) \in (t,(X_m/t))\cdot \mathbb{B}_{\mathrm{dR},U_m}^+(U_{m,w}')\llbracket X_m/t\rrbracket .
\end{equation}
\item Also define
\begin{equation}\label{qdR}q_{\mathrm{dR}}(w) := \mathrm{exp}(\mathbf{z}(w)) \in 1 + (t,(X_m/t))\cdot \mathbb{B}_{\mathrm{dR},U_m}^+(U_{m,w}')\llbracket X_m/t\rrbracket ,
\end{equation}
where 
$$\mathrm{exp}(Y) = \sum_{n = 0}^{\infty}\frac{Y^n}{n!}$$
is formal exponentiation. Note that $q_{\mathrm{dR}}$ is a well-defined element of $\mathbb{B}_{\mathrm{dR},U_m}^+(U_{m,w}')\llbracket X_m/t\rrbracket$ because that ring is $(t,(X_m/t))$-adically complete. 
\end{enumerate}
\end{definition}

\begin{remark}We will later show that $\mathbf{z}(w)$ and $q_{\mathrm{dR}}(w)$ are independent of $w \in \omega(U_m)$ and extend to sections on $\mathcal{V}_x$, see (\ref{zqw2}). Note also that 
$$\nabla(q_{\mathrm{dR}}(w))/q_{\mathrm{dR}}(w) = \nabla(\mathbf{z}(w)) \hspace{.5cm} \text{and} \hspace{.5cm}\nabla(\mathbf{z}(w)) = t\cdot \nabla(z_{\mathrm{dR}}(w)).$$
\end{remark}

\subsection{Properties of $w_{\mathrm{can}}$}\label{wcansection}

Recall $z_{\mathrm{dR}}(w)$ and $w_{\mathrm{can}}(w)$ from Definition \ref{'zqwdefinition}. Note that both section \emph{a priori} depend on the choice of section $w$ as in Choice \ref{'choicew}. The next Proposition shows that $z_{\mathrm{dR}}(w)$ and $w_{\mathrm{can}}(w)$ are in fact independent of $w$. %Occasionally in this section, we will let $w_{\mathrm{can}}(w')$ denote $w_{\mathrm{can}}$ to emphasize the ostensible dependence on a choice $w' \in \omega(U_m)$. Similarly, recall we let $U_m'  = U_{m',w'}$ when we wish to emphasize the dependence of $U_m'$ on the choice of $w' \in \omega(U_m)$. 

\begin{proposition}\label{wcangeneratorproposition}For any $m \in \mathbb{Z}_{\ge 0}$ and any  $w \in \omega(U_m)$, let $U_{m,w}'$ be the corresponding open sets defined by (\ref{'U'definition}). For any $w, w' \in \omega(U_m)$, we have
\begin{equation}\label{zdRequality}z_{\mathrm{dR}}(w)|_{U_{m,w}'\cap U_{m,w'}'} = z_{\mathrm{dR}}(w')|_{U_{m,w}' \cap U_{m,w'}'}
\end{equation}
and
\begin{equation}\label{wcanequality}w_{\mathrm{can}}(w)|_{U_{m,w}' \cap U_{m,w'}'} = w_{\mathrm{can}}(w')|_{U_{m,w}' \cap U_{m,w'}'} .
\end{equation}
In particular, defining
\begin{equation}\label{fullU'}\mathbf{U}' := \bigcup_{m \in \mathbb{Z}_{\ge 0}}\bigcup_{w \in \omega(U_m)}U_{m,w}',
\end{equation}
the $z_{\mathrm{dR}}(w), \mathbf{z}(w)$ and $q_{\mathrm{dR}}(w)$ glue to give an elements
\begin{equation}\label{gluezdRsection}z_{\mathrm{dR}}, \mathbf{z}, q_{\mathrm{dR}} \in \varprojlim_{m \in \mathbb{Z}_{\ge 0}}\varprojlim_{w \in \omega(U_m)}\mathbb{B}_{\mathrm{dR},U_m}^+(U_{m,w}')\llbracket X_m/t\rrbracket,
\end{equation}
and the $w_{\mathrm{can}}(w)$ glue to give a generator
\begin{equation}\label{wcangenerator0}w_{\mathrm{can}} \in \omega \otimes_{\mathcal{O}_Y}\varprojlim_{m \in \mathbb{Z}_{\ge 0}}\varprojlim_{w \in \omega(U_m)}\mathbb{B}_{\mathrm{dR},U_m}^+(U_{m,w}')\llbracket X_m/t\rrbracket.
\end{equation}
Moreover we have
\begin{equation}\label{iwcan}i_{\mathrm{dR}}(w_{\mathrm{can}}) = -z_{\mathrm{dR}}e_1 + e_2.
\end{equation}
\end{proposition}

\begin{proof}Let $w, w' \in \omega(U_m)$ be as in Choice \ref{'choicew}. Since the Hodge filtration (\ref{Hodgefiltrationinclusion}) on $U_{m,w}' \cap U_{m,w'}'$ is uniquely defined and thus has unique image under the comparison map (\ref{Hodgecomparisoninclusion}) measured by an element 
$$z_{\mathrm{dR}} \in \mathbb{P}^1(\mathbb{B}_{\mathrm{dR},U_m
}^+(U_{m,w}' \cap U_{m,w'}')\llbracket X_m/t\rrbracket),$$
we have (\ref{zdRequality}). This, (\ref{zbfdefinition}) and (\ref{qdR}), immediately imply the existence of the elements (\ref{gluezdRsection}). 
Now we have 
\begin{align*}i_{\mathrm{dR}}(w_{\mathrm{can}}(w)|_{U_{m,w}' \cap U_{m,w'}'}) \overset{(\ref{'iwcan})}{=} -z_{\mathrm{dR}}(w)|_{U_{m,w}' \cap U_{m,w'}'}e_1 + e_2 &= -z_{\mathrm{dR}}(w')|_{U_{m,w}' \cap U_{m,w'}'} + e_2 \\
&\overset{(\ref{'iwcan})}{=} i_{\mathrm{dR}}(w_{\mathrm{can}}(w')|_{U_{m,w}' \cap U_{m,w'}'}),
\end{align*}
which immediately implies (\ref{wcanequality}). Given $m \in \mathbb{Z}_{\ge 0}$, choose a set $\mathbb{S}\subset \omega(U_m)$ such that for every $w \in \mathbb{S}$, there is some open $V_{m,w} \subset U_m$ such that $w|_{V_{m,w}} \in \omega(V_{m,w})$ is a generator, and 
$$\bigcup_{w \in \mathbb{S}}V_{m,w} = U_m.$$
Note that such $\mathbb{S}$ exists because $\omega|_{U_m}$ is basepoint-free, which in turn follows from the fact that $\omega\otimes_{\mathcal{O}_Y}\hat{\mathcal{O}}_{\mathcal{V}_x}$ is basepoint-free since the section (\ref{fraksgenerator}) on $\mathcal{V}_x \supset U_m$ is nowhere vanishing. From Corollary \ref{'unitcorollary}, we have that 
$$w_{\mathrm{can}}(w)|_{V_{m,w}'} \in \omega \otimes_{\mathcal{O}_Y}\mathbb{B}_{\mathrm{dR},U_{m,w}'}^+(V_{m,w}')\llbracket X_m/t\rrbracket$$
is a generator. Let $w \in \omega(U_m)$ be any fixed choice as in Choice \ref{'choicew}. By (\ref{wcanequality}), we have 
$$w_{\mathrm{can}}(w)|_{V_{m,w}'  \cap V_{m,w}'} = w_{\mathrm{can}}(w')|_{V_{m,w}' \cap V_{m,w}'},$$
and so these glue together to give a global generator 
$$w_{\mathrm{can}} \in \omega \otimes_{\mathcal{O}_Y}\mathbb{B}_{\mathrm{dR},U_m}^+(\bigcup_{w \in \omega(U_m)}U_{m,w}')\llbracket X_m/t\rrbracket.$$
Again by (\ref{'iwcan}), these glue together along the inclusions $U_{m,w}' \subset U_{m+1,w}'$ from Proposition \ref{Um'proposition} to give the generator (\ref{wcangenerator0}). Now (\ref{iwcan}) immediately follows from (\ref{gluezdRsection}), (\ref{wcangenerator0}) and (\ref{'iwcan}). 

\end{proof}

Recall $\nabla : \mathcal{O}\mathbb{B}_{\mathrm{dR}} \rightarrow \mathcal{O}\mathbb{B}_{\mathrm{dR}} \otimes \Omega$ is the natural connection. Recall the Kodaira-Spencer isomorphism of sheaves (\cite[Appendix A1.3.17]{Katzpamf}, \cite{Yuan}; recall in our situation $p \nmid \mathrm{disc}(D)$)
\begin{equation}\label{'KS}\mathrm{KS} : \omega^{\otimes 2} \xrightarrow{\sim} \Omega_Y, \hspace{1cm} \mathrm{KS}(w^{\otimes 2}) = \langle w,\nabla(w)\rangle_{\mathrm{dR}}
\end{equation}
where $\Omega_Y$ is the sheaf of K\"{a}hler differentials on $Y$, $\langle \cdot, \cdot \rangle_{\mathrm{dR}}$ is the Poincar\'{e} pairing on $H_{\mathrm{dR}}^1(\mathcal{E})$, and $\nabla : H_{\mathrm{dR}}^1(\mathcal{E}) \rightarrow H_{\mathrm{dR}}^1(\mathcal{E}) \otimes_{\mathcal{O}_Y}\Omega_Y$ is the Gauss-Manin connection. 

\begin{proposition}\label{'KSproposition}$\mathrm{KS}(w_{\mathrm{can}}^{\otimes 2}) = \nabla\mathbf{z}$. 
\end{proposition}

\begin{proof}We have 
\begin{align*}\mathrm{KS}(w_{\mathrm{can}}^{\otimes 2}) = \langle w_{\mathrm{can}},\nabla(w_{\mathrm{can}})\rangle_{\mathrm{dR}} \overset{(\ref{comparepairings})}{=} \langle i_{\mathrm{dR}}(w_{\mathrm{can}}),\nabla(i_{\mathrm{dR}}(w_{\mathrm{can}}))\rangle &\overset{(\ref{iwcan})}{=} \langle -z_{\mathrm{dR}}e_1 + e_2, \nabla(-z_{\mathrm{dR}}e_1 + e_2)\rangle\\
& = \langle -z_{\mathrm{dR}}e_1+e_2,-e_1\nabla z_{\mathrm{dR}}\rangle \\
&= \langle e_2, -e_1\rangle \cdot \nabla z_{\mathrm{dR}} \\
&= \langle e_1,e_2\rangle \cdot \nabla z_{\mathrm{dR}}\\
&\overset{(\ref{'tdefinition}), (\ref{tidentifications})}{=} t\cdot \nabla z_{\mathrm{dR}} \overset{(\ref{zbfdefinition})}{=} \nabla\mathbf{z}.
\end{align*}
\end{proof}

\begin{corollary}\label{nablazgeneratorcorollary}For any $m \in \mathbb{Z}_{\ge 0}$, the element
$$\nabla\mathbf{z} \in \Omega_Y \otimes_{\mathcal{O}_Y}\mathbb{B}_{\mathrm{dR},\mathcal{V}_x}^+(U_{m,w}')\llbracket X_m/t\rrbracket$$
is a $\mathbb{B}_{\mathrm{dR},\mathcal{V}_x}^+(U_{m,w}')\llbracket X_m/t\rrbracket$-generator. Thus we have identifications
\begin{equation}\label{changeofvars}\mathbb{B}_{\mathrm{dR},U_m}^+(U_{m,w}')\llbracket X_m/t\rrbracket = \mathbb{B}_{\mathrm{dR},U_m}^+(U_{m,w}')\llbracket \mathbf{z}\rrbracket = \mathbb{B}_{\mathrm{dR},U_m}^+(U_{m,w}')\llbracket q_{\mathrm{dR}}-1\rrbracket.
\end{equation}
\end{corollary}

\begin{proof}
Since $U_{m,w}' = \bigcup_{i = 1}^{\mathbf{n}}U_{i,m,w}'$ and $\mathbb{B}_{\mathrm{dR},U_m}^+\llbracket X_m/t\rrbracket$ is a sheaf on $Y_{\text{pro\'{e}t}}/U_m$, it suffices to show that 
$$\mathbb{B}_{\mathrm{dR},U_m}^+(U_{i,m,w}')\llbracket X_{i,m}/t\rrbracket = \mathbb{B}_{\mathrm{dR},U_m}^+(U_{i,m,w}')\llbracket \mathbf{z}\rrbracket = \mathbb{B}_{\mathrm{dR},U_m}^+(U_{i,m,w}')\llbracket q_{\mathrm{dR}}-1\rrbracket$$
for all $1 \le i \le \mathbf{n}$. 

Note that $\theta_{X_{i,m}}(\mathbf{z}) = \theta_{X_{i,m}}(q_{\mathrm{dR}}-1) = 0$ by Definition \ref{'zqwdefinition}, (\ref{gluezdRsection}) and (\ref{thetaXmspecialize}). The equality now follows from making changes of variables $X_{i,m}/t \rightarrow \mathbf{z}$ and $X_{i,m}/t\rightarrow q_{\mathrm{dR}}-1$, using (\ref{wcangenerator0}) and Proposition \ref{'KSproposition} which imply
$$\frac{\nabla\mathbf{z}}{\nabla (X_{i,m}/t)}|_{U_{i,m,w}'} \in \mathbb{B}_{\mathrm{dR},U_{m,w}'}^+(U_{i,m,w}') \llbracket X_m/t\rrbracket^{\times}, \hspace{.9cm} \frac{\nabla (q_{\mathrm{dR}}-1)}{\nabla (X_{i,m}/t)}|_{U_{i,m,w}'} \in \mathbb{B}_{\mathrm{dR},U_m}^+(U_{i,m,w}') \llbracket X_{i,m}/t\rrbracket^{\times}.$$
 
\end{proof}

\begin{corollary}Consider the ring $\varprojlim_{m \in \mathbb{Z}_{\ge 0}}\varprojlim_{w \in \omega(U_m)}\mathbb{B}_{\mathrm{dR},U_m}^+(U_{m,w}')\llbracket X_m/t\rrbracket$ appearing in (\ref{wcangenerator0}), where the transition maps in the inverse limit are given by the inclusions $U_{m,w}' \subset U_{m+1,w}'$ from Proposition \ref{Um'proposition} for $w \in \omega(U_m)$ varying. Then we have 
\begin{equation}\label{fullchangeofvars}\varprojlim_{m \in \mathbb{Z}_{\ge 0}}\varprojlim_{w \in \omega(U_m)}\mathbb{B}_{\mathrm{dR},U_m}^+(U_{m,w}')\llbracket X_m/t\rrbracket = \mathbb{B}_{\mathrm{dR},\mathbf{U}'}^+(\mathbf{U}')\llbracket q_{\mathrm{dR}}-1\rrbracket.
\end{equation}

\end{corollary}

\begin{proof}This follows from the equalities (\ref{changeofvars}), which are clearly compatible with respect to the inclusions $U_{m,w}' \subset U_{m+1,w}'$ from Proposition \ref{Um'proposition}, and compatible for $w \in \omega(U_m)$ varying by (\ref{wcanequality}).

\end{proof}

\begin{proposition}\label{wcangeneratorproposition2}\begin{enumerate}
\item We have
\begin{equation}\label{wcangenerator}w_{\mathrm{can}} \in \omega \otimes_{\mathcal{O}_Y}\in \mathbb{B}_{\mathrm{dR},\mathbf{U}'}^+(\mathbf{U}')\llbracket q_{\mathrm{dR}}-1\rrbracket.
\end{equation}
Moreover, $z_{\mathrm{dR}}, \mathbf{z}, q_{\mathrm{dR}} \in \mathbb{B}_{\mathrm{dR},U_m}^+(U_{m,w}')\llbracket X/t\rrbracket$ extend to elements
\begin{equation}\label{zzqU'}z_{\mathrm{dR}}, \mathbf{z},q_{\mathrm{dR}} \in \mathbb{B}_{\mathrm{dR},\mathbf{U}'}^+(\mathbf{U}')\llbracket q_{\mathrm{dR}}-1\rrbracket.
\end{equation}
%Moreover, the element $y_{\mathrm{dR}}(w) \in \mathcal{O}\mathbb{B}_{\mathrm{dR},U}^+(U_w') \cdot t^{-1}$ i the restriction of an element 
%\begin{equation}\label{yglue}y_{\mathrm{dR}} \in \mathcal{O}\mathbb{B}_{\mathrm{dR},U}^+(U)\cdot t^{-1}.
%\end{equation}
%(which still depends on $w$ from Choice \ref{'choicew}). 
\item For all $m \in \mathbb{Z}_{\ge 0}$, we have
\begin{equation}\label{UmU'}U_m \cap \mathbf{U}' = \bigcup_{w \in \omega(U_m)}U_{m,w}'.
\end{equation}
\end{enumerate}
\item Moreover, for any $M \in \mathbb{Z}_{\ge 0}$, we have
\begin{equation}\label{fullU'2}\mathbf{U}' = \bigcup_{m \in \mathbb{Z}_{\ge M}}\bigcup_{w \in \omega(U_m)}U_{m,w}'.
\end{equation}
\end{proposition}

\begin{proof}
\textbf{(1)}: The inclusion (\ref{wcangenerator}) follows immediately from (\ref{wcangenerator0}) and (\ref{fullchangeofvars}). The $z_{\mathrm{dR}}|_{U_{m,w}'}$ glue together to give a section $z_{\mathrm{dR}} \in \mathbb{B}_{\mathrm{dR},U_{m,w}'}^+(U_{m,w}')\llbracket X_m/t\rrbracket$, and by the same reasoning $\mathbf{z}$ and $q_{\mathrm{dR}}$ extend to sections on $\mathbf{U}'$ to give (\ref{zzqU'}).\\

\textbf{(2)}: Both sides of (\ref{UmU'}) are equal to the locus in $U_m$ on which the section $w_{\mathrm{can}}$ from (\ref{wcangenerator}) is defined.\\

\textbf{(3)}: Applying $\bigcup_{m \in \mathbb{Z}_{\ge M}}$ to (\ref{UmU'}), we get
$$\mathbf{U}' = \mathcal{V}_x \cap \mathbf{U}' \overset{(\ref{Vxunion'})}{=} \left(\bigcup_{m \in M}U_m\right) \cap \mathbf{U}' = \bigcup_{m \in \mathbb{Z}_{\ge M}}\left(U_m \cap \mathbf{U}'\right) \overset{(\ref{UmU'})}{=} \bigcup_{m \in \mathbb{Z}_{\ge M}}\bigcup_{w \in \omega(U_m)}U_{m,w}'.$$

\end{proof}

The right $GL_2(\mathbb{Q}_p)$-action acts on $z_{\mathrm{dR}}$ and $z_{\mathrm{HT}}$ (from (\ref{zHT})) via linear fractional transformations.

\begin{proposition}\label{modulartransformationproposition}For $\gamma = \left(\begin{array}{ccc} a & b\\
c & d\\
\end{array}\right) \in GL_2(\mathbb{Q}_p)$ with $\mathbf{U}' \cdot \gamma \subset \mathbf{U}'$, we have 
\begin{equation}\label{modulartransformationidentity}\gamma^*z_{\mathrm{dR}} = \frac{dz_{\mathrm{dR}} + b}{cz_{\mathrm{dR}} + a}, \hspace{1cm} \gamma^*z_{\mathrm{HT}} = \frac{dz_{\mathrm{HT}} + b}{cz_{\mathrm{HT}} + a}.
\end{equation}
\end{proposition}

\begin{proof}This follows from the same argument as in \cite[Section 2.4]{ChojeckiHansenJohansson}, replacing $\frak{z}$ there by $1/z_{\mathrm{dR}}$ and $1/z_{\mathrm{HT}}$ accordingly. We give details for the reader's convenience. Since $\gamma^*$ acts on $H_{\mathrm{dR}}^1(\mathcal{E})$ by an isogeny, it preserves the Hodge filtration (\ref{Hodgefiltrationinclusion}). %Thus, given any section $w'$ of $\omega(\mathbf{U})$, $\gamma^*w' \in \omega(U_m')$ and in particular 
%\begin{equation}\label{zsequal}z_{\mathrm{dR}}(w) = z_{\mathrm{dR}}(\gamma^*w').
%\end{equation}
From (\ref{iwcan}) and \cite[diagram in proof of Lemma 2.11]{ChojeckiHansenJohansson}, we see that for some $f \in \mathcal{O}\mathbb{B}_{\mathrm{dR},\mathbf{U}'}(\mathbf{U}')^{\times}$,
\begin{equation}\label{intermediatetransformation}\begin{split}-fz_{\mathrm{dR}}e_1 + fe_2 &= f \cdot i_{\mathrm{dR}}(w_{\mathrm{can}}) = \gamma^*(i_{\mathrm{dR}}(w_{\mathrm{can}})) = \gamma^*(-z_{\mathrm{dR}}e_1 + e_2) \\
&= -(\gamma^*z_{\mathrm{dR}})(ae_1 + ce_2) + (be_1 + de_2) = (-a\gamma^*z_{\mathrm{dR}} + b)e_1 + (-c\gamma^*z_{\mathrm{dR}} + d)e_2.
\end{split}
\end{equation}
From this we get 
$$f\left(\begin{array}{ccc} -z_{\mathrm{dR}}\\
1\\
\end{array}\right) = \left(\begin{array}{ccc} a & b\\
c & d\\
\end{array}\right)\left(\begin{array}{ccc}-\gamma^*z_{\mathrm{dR}}\\
1\\
\end{array}\right).$$
Solving for $\gamma^*z_{\mathrm{dR}}$, we get the first equality of (\ref{modulartransformationidentity}). The proof for $z_{\mathrm{HT}}$ is entirely analogous.

\end{proof}

\begin{remark}As previously mentioned, we will soon show that 
$$\mathbf{U}' = \mathcal{V}_x,$$
see Theorem \ref{U'Utheorem}, but we will need some intermediate results (Corollary \ref{U'YIg}) before we can do this. Hence the results of Propositions \ref{wcangeneratorproposition} and \ref{wcangeneratorproposition2} hold for $\mathcal{V}_x$ in place of $\mathbf{U}'$, and similarly with all other results involving $\mathbf{U}'$. 
\end{remark}

\subsection{$q_{\mathrm{dR}}$-expansions}\label{qdRfirstsection}

\begin{remark}All the notions and results in this section (Section \ref{qdRfirstsection}) and in Section \ref{padicmodularformsection} are defined with respect to the base $\mathbf{U}'$ from (\ref{fullU'}). In light of Theorem \ref{U'Utheorem} below, we have $\mathbf{U}' = \mathcal{V}_x$, and so all the statements below hold for $\mathcal{V}_x \overset{(\ref{Vz})}{=} \{z_{\mathrm{HT}} \neq 0\}$ in place of $\mathbf{U}'$. 
\end{remark}

\begin{definition}\label{'zqfunctionexpansions}Given a pro\'{e}tale open $W \rightarrow \mathbf{U}'$ and $f \in \mathbb{B}_{\mathrm{dR},\mathbf{U}'}(W)\llbracket q_{\mathrm{dR}}-1\rrbracket$, let
$$f(q_{\mathrm{dR}}) = \sum_{n = 0}^{\infty} a_n \cdot (q_{\mathrm{dR}}-1)^n\in \mathbb{B}_{\mathrm{dR},\mathbf{U}'}(W)\llbracket q_{\mathrm{dR}}-1\rrbracket$$
where
\begin{equation}\label{Taylorcoefficient}a_n := \frac{1}{n!}\theta_X\left(\left(\frac{\nabla}{\nabla q_{\mathrm{dR}}}\right)^n(f)\right) \in \mathbb{B}_{\mathrm{dR},\mathbf{U}'}(W).
\end{equation}
We call this the \emph{$q_{\mathrm{dR}}$-expansion of $f$}. 
\end{definition}

We will show that $\mathbf{U}' = \mathcal{V}_x$ in Theorem \ref{U'Utheorem} below. Assume this equality briefly and let $W = \mathbf{U}' = \mathcal{V}_x$. From (\ref{Bdecomposition'}) we get the isomorphism of $\hat{\mathcal{O}}_{U_m}(U_m)$-algebras
\begin{equation}\label{BdecompositionUm}\mathbb{B}_{\mathrm{dR},U_m}(U_m) \cong \hat{\mathcal{O}}_{U_m}(U_m)(\!(t)\!).
\end{equation}
Henceforth identify $\mathbb{B}_{\mathrm{dR},U_m}(U_m) = \hat{\mathcal{O}}_{U_m}(U_m)(\!(t)\!)$.  
By gluing together (\ref{BdecompositionUm}) for all $m$, using $\mathcal{V}_x = \bigcup_{m \in \mathbb{Z}_{\ge 0}}U_m$ from (\ref{Vxunion'}) and the fact that $t \in \mathrm{Fil}^1\mathbb{B}_{\mathrm{dR},Y}^+(\mathcal{V}_x)$ is a generator, we have a natural isomorphism of $\hat{\mathcal{O}}_{\mathcal{V}_x}(\mathcal{V}_x)$-algebras
\begin{equation}\label{Bdecomposition''}\mathbb{B}_{\mathrm{dR},\mathcal{V}_x}(\mathcal{V}_x) \cong \hat{\mathcal{O}}_{\mathcal{V}_x}(\mathcal{V}_x)(\!(t)\!).
\end{equation}
Henceforth identify $\mathbb{B}_{\mathrm{dR},\mathcal{V}_x}(\mathcal{V}_x) = \hat{\mathcal{O}}_{\mathcal{V}_x}(\mathcal{V}_x)(\!(t)\!)$. Then in the notation of (\ref{Taylorcoefficient}) we can we can view
$$a_n = a_n(t) \in \hat{\mathcal{O}}_{\mathcal{V}_x}(\mathcal{V}_x)(\!(t)\!)$$
and
$$f(q_{\mathrm{dR}}) = \sum_{n = 0}^{\infty}a_n(t) \cdot (q_{\mathrm{dR}}-1)^n \in \hat{\mathcal{O}}_{\mathcal{V}_x}(\mathcal{V}_x)(\!(t)\!)\llbracket q_{\mathrm{dR}}-1\rrbracket.$$

\subsection{Generalized $p$-adic modular forms and their $q_{\mathrm{dR}}$-expansions}\label{padicmodularformsection}For certain $p$-adic analytic applications (e.g. constructing $\mathbb{C}_p$-valued $p$-adic $L$-functions), we want to consider power series expansions of functions with $\hat{\mathcal{O}}$-coefficients, which provide an analogue of Serre-Tate expansions on the supersingular locus. These objects will be provided by the notion of ``$q_{\mathrm{dR}}$-expansions of generalized $p$-adic modular forms'' defined in Definition \ref{'zqexpansions}. The comparison between these expansions and Serre-Tate expansions will be analyzed in Section \ref{'ordinarysection}, culminating in Theorem \ref{STanalyticcontinuationtheorem}.

We view $p$-adic modular forms as functions on covers of Shimura curves attached to sections of $\omega^{\otimes k}$ obtained by trivializing via a generator of $\omega^{\otimes k}$ which exists on the cover of the Shimura curve. This role of this latter generator is played by $w_{\mathrm{can}}^{\otimes k}$ where $w_{\mathrm{can}}$ is from (\ref{wcangenerator0}); note that this generator only exists on the $p$-adic universal cover $Y_{\infty} \rightarrow Y$, and is only defined after extending coefficients from the structure sheaf to a large period sheaf. For any pro\'{e}tale open $W \rightarrow \mathbf{U}'$, and any $k \ge 0$ and any 
$$w' \in \omega^{\otimes k} \otimes_{\mathcal{O}_Y}\mathbb{B}_{\mathrm{dR},\mathbf{U}'}(W)\llbracket q_{\mathrm{dR}}-1\rrbracket,$$ we have an associated ``function''
$$G := \frac{w'}{w_{\mathrm{can}}^{\otimes k}} \in \mathbb{B}_{\mathrm{dR},\mathbf{U}'}(W)\llbracket q_{\mathrm{dR}}-1\rrbracket,$$
which gives our notion of $p$-adic modular form. 

\begin{definition}\label{generalizedpadicmodularformdefinition}We call $G$ as above the \emph{generalized $p$-adic modular form attached to $w'$}. We call the set of all $G$ arising from some $w' \in \omega^{\otimes k} \otimes_{\mathcal{O}_Y}\mathbb{B}_{\mathrm{dR},U}(W)\llbracket q_{\mathrm{dR}}-1\rrbracket$ the \emph{space of generalized $p$-adic modular forms of weight $k$}. Note that we will later (Theorem \ref{U'Utheorem}) show that $\mathbf{U}' = \mathcal{V}_x$, and so generalized $p$-adic modular forms are defined for any pro\'{e}tale open $W \rightarrow \mathcal{V}_x$. 

\end{definition}
When $w'$ is pulled back from a modular form of weight $k$ on some finite level Shimura curve, the associated generalized $p$-adic modular form $G$ satisfies a weight $k$ modular transformation property analogous to that satisfied by classical modular forms, see Theorem \ref{weighttheorem}. As we will see in Section \ref{'ordinarysection}, $G|_{\mathcal{Y}^{\mathrm{Ig}}}$ is a $p$-adic modular form in Katz's sense (\cite{Katzpamf}); see Theorem \ref{Katzpadicmodularformtheorem}.

\begin{definition}Let 
\begin{equation}\label{Gamma0pdefinition}\Gamma_{0,p}(p^n) := \left\{\left(\begin{array}{ccc} a & b\\
c & d\\
\end{array}\right) \in GL_2(\mathbb{Z}_p) : c \equiv 0 \pmod{p^n}\right\} \subset GL_2(\mathbb{Z}_p).
\end{equation}
and 
\begin{equation}\label{Gamma1pdefinition}\Gamma_{1,p}(p^n) := \left\{\left(\begin{array}{ccc} a & b\\
c & d\\
\end{array}\right) \in GL_2(\mathbb{Z}_p) : c \equiv 0 \pmod{p^n}, \hspace{.25cm} d \equiv 1 \pmod{p^n}\right\} \subset \Gamma_{0,p}(p^n).
\end{equation}
\end{definition}

From the definition of $U = \mathcal{Y}^{\mathrm{Ig}}(\epsilon_0)$ in Definition \ref{Udefinition}, one can easily check that
$$\mathrm{Gal}(U/Y(\epsilon_0)) = \mathrm{Gal}(\mathcal{Y}^{\mathrm{Ig}}(\epsilon_0)/Y(\epsilon_0)) = \Gamma_{0,p}(p^{n(\epsilon_0)}),$$
and more generally from (\ref{U}) we have for any $0\le \epsilon < p/(p+1)$ in the valuation group of $\mathcal{O}_k$
\begin{equation}\label{generalIgusaGalois}\mathrm{Gal}(\mathcal{Y}^{\mathrm{Ig}}(\epsilon)/Y(\epsilon)) = \Gamma_{0,p}(p^{n(\epsilon)}).
\end{equation}
This implies that for any $\alpha \in \mathbb{Z}_{\ge 0}$,
\begin{equation}\label{Galoisepsilonalpha}\mathrm{Gal}(\mathcal{Y}^{\mathrm{Ig}}(\epsilon_0/p^{\alpha})/Y(\epsilon_0/p^{\alpha})) = \Gamma_{0,p}(p^{n(\epsilon_0/p^{\alpha})}) = \Gamma_{0,p}(p^{n(\epsilon_0) + \alpha}).
\end{equation}

\begin{proposition}By (\ref{wcangenerator0}), we can view 
$$w_{\mathrm{can}} \in \omega \otimes_{\mathcal{O}_Y}\mathbb{B}_{\mathrm{dR},\mathcal{V}_x}^+(\mathcal{Y}^{\mathrm{Ig}}(\epsilon))\llbracket q_{\mathrm{dR}}-1\rrbracket$$
for any $0 \le \epsilon < p/(p+1)$ in the valuation group of $\mathcal{O}_k$. Let $\gamma = \left(\begin{array}{ccc} a & b\\
c & d\\
\end{array}\right) \in \Gamma_{0,p}(p^{n(\epsilon)})$. Then 
\begin{equation}\label{wcantransformationidentity}\gamma^*w_{\mathrm{can}} = \frac{ad-bc}{cz_{\mathrm{dR}} + a}\cdot w_{\mathrm{can}}.
\end{equation}
\end{proposition}

\begin{proof}Recall that $i_{\mathrm{dR}}$ from (\ref{comparisonmap}) is $GL_2(\mathbb{Q}_p)$-equivariant. We have, by \cite[diagram in proof of Lemma 2.11]{ChojeckiHansenJohansson},
\begin{align*}i_{\mathrm{dR}}(\gamma^*w_{\mathrm{can}}) &= \gamma^*i_{\mathrm{dR}}(w_{\mathrm{can}}) = \gamma^*(-z_{\mathrm{dR}}e_1 + e_2) \\
&\overset{(\ref{modulartransformationidentity})}{=} - \frac{dz_{\mathrm{dR}} + b}{cz_{\mathrm{dR}} + a}(ae_1 + ce_2) + be_1 + de_2 \\
&= \frac{1}{cz_{\mathrm{dR}}+a}\left(-(dz_{\mathrm{dR}}+b)(ae_1 + ce_2) + (cz_{\mathrm{dR}} +a)(be_1 + de_2)\right)\\
&= \frac{ad-bc}{cz_{\mathrm{dR}}+a}(-z_{\mathrm{dR}}e_1 + e_2) = \frac{ad-bc}{cz_{\mathrm{dR}}+a}\cdot i_{\mathrm{dR}}(w_{\mathrm{can}}) = i_{\mathrm{dR}}\left(\frac{ad-bc}{cz_{\mathrm{dR}} + a}\cdot w_{\mathrm{can}}\right).
\end{align*}
This immediately implies (\ref{wcantransformationidentity}).
\end{proof}

\begin{theorem}\label{weighttheorem}Suppose 
\begin{enumerate}
\item $0 \le \epsilon < p/(p+1)$ is any element in the valuation group of $\mathcal{O}_k$ (see Definition \ref{kdefinition}),
\item $n \ge n(\epsilon)$ (see (\ref{nepsilondefinition})), 
\item $$\gamma = \left(\begin{array}{ccc} a & b\\
c & d\\
\end{array}\right) \in \Gamma_{0,p}(p^n),$$
\item $W$ is the inverse image under the $\Gamma_{0,p}(p^{n(\epsilon)})$-cover $\mathcal{Y}^{\mathrm{Ig}}(\epsilon) \rightarrow Y(\epsilon)$ of an open subset $W_0 \subset Y(\epsilon)$, 
\item $W \subset \mathbf{U}'$, and
\item $W \cdot \gamma \subset W$.
\end{enumerate}
In particular, this gives a map
$$\gamma^* : \mathbb{B}_{\mathrm{dR},\mathbf{U}'}(W)\llbracket q_{\mathrm{dR}}-1\rrbracket \rightarrow \mathbb{B}_{\mathrm{dR},\mathbf{U}'}(W)\llbracket q_{\mathrm{dR}}-1\rrbracket.$$

Then for any generalized $p$-adic modular form $G \in \mathbb{B}_{\mathrm{dR},\mathbf{U}'}(W)\llbracket q_{\mathrm{dR}}-1\rrbracket$ of weight $k$, we have
\begin{equation}\label{Ftransformationidentity}\gamma^*G = \left(\frac{cz_{\mathrm{dR}} + a}{ad-bc}\right)^kG.
\end{equation}
\end{theorem}

\begin{proof}
Recall that 
$$\mathrm{Gal}(\mathcal{Y}^{\mathrm{Ig}}(\epsilon)/Y(\epsilon)) \overset{(\ref{generalIgusaGalois})}{=} \Gamma_{0,p}(p^{n(\epsilon)}),$$
so that 
$$\mathrm{Gal}(W/W_0) = \Gamma_{0,p}(p^{n(\epsilon)}).$$
Let 
$$w' = G \cdot w_{\mathrm{can}}^{\otimes k} \in \omega^{\otimes k} \otimes_{\mathcal{O}_Y}\mathbb{B}_{\mathrm{dR},U}(W)\llbracket q_{\mathrm{dR}}-1\rrbracket.$$
Then $\gamma^*w' = w'$. Thus,
$$G \cdot w_{\mathrm{can}}^{\otimes k} = w' = \gamma^*w' = \gamma^*(G \cdot w_{\mathrm{can}}^{\otimes k}) = \gamma^*G\cdot \gamma^*w_{\mathrm{can}}^{\otimes k} \overset{(\ref{wcantransformationidentity})}{=} \gamma^*G \cdot \left(\frac{ad-bc}{cz_{\mathrm{dR}} + a}\right)^kw_{\mathrm{can}}^{\otimes k}.$$
This gives (\ref{Ftransformationidentity}).
\end{proof}

For example, one may take $\epsilon = \epsilon_0$, $W = \mathbf{U}' \overset{(\ref{U'U})}{=} U$, $W_0 = Y(\epsilon_0)$ and $n = n(\epsilon_0)$ in Theorem \ref{weighttheorem}. We now define the $q_{\mathrm{dR}}$-expansion of the $p$-adic modular form $G$ by applying Definition \ref{'zqfunctionexpansions} to $G$. This will be also called the $q_{\mathrm{dR}}$-expansion of $w'$. %As we will see in Section \ref{'ordinarysection}, one can view $G(q_{\mathrm{dR}})$ as an analytic continuation of a family of Serre-Tate expansions. 

%Since
%$$w = y_{\mathrm{dR}} \cdot w_{\mathrm{can}},$$
%we view $y_{\mathrm{dR}}$ as the $p$-adic modular form attached to $w$. This motivates the following definition, analogous to the Serre-Tate expansion of a modular form on the ordinary locus. (See Section \ref{'ordinarysection} for a discussion of how to recover Serre-Tate expansions from our theory.) 

%By (\ref{'yinclusion}), we can apply the definitions of Definition \ref{'zqfunctionexpansions} to $y_{\mathrm{dR}} \in \mathbb{B}_{\mathrm{dR},U'}^+(U')\llbracket X/t\rrbracket$. 

\begin{definition}\label{'zqexpansions}
%Define the \emph{$\mathbb{B}_{\mathrm{dR}}^+$-valued $z_{\mathrm{dR}}$-expansion of $w'$} as
%$$w'(z_{\mathrm{dR}}) := F(z_{\mathrm{dR}})\in \mathbb{B}_{\mathrm{dR},U'}^+(V)\llbracket z_{\mathrm{dR}} - \theta_X(z_{\mathrm{dR}})\rrbracket.$$

For any pro\'{e}tale open $W \rightarrow \mathbf{U}'$, and any 
$$w' \in \omega^{\otimes k}\otimes_{\mathcal{O}_Y}\mathbb{B}_{\mathrm{dR},\mathbf{U}'}(W)\llbracket q_{\mathrm{dR}}-1\rrbracket,$$
define the \emph{$q_{\mathrm{dR}}$-expansion of $w'$} as
$$w'(q_{\mathrm{dR}}) := G(q_{\mathrm{dR}}) \in \mathbb{B}_{\mathrm{dR},\mathbf{U}'}(W)\llbracket q_{\mathrm{dR}}-1\rrbracket,$$
where 
$$G = \frac{w'}{w_{\mathrm{can}}^{\otimes k}} \in \mathbb{B}_{\mathrm{dR},\mathbf{U}'}(W)\llbracket q_{\mathrm{dR}}-1\rrbracket.$$

%Define the \emph{$\hat{\mathcal{O}}$-valued $z_{\mathrm{dR}}$-expansion of $w'$} as
%$$\tilde{w}'(z_{\mathrm{dR}}) :=  \tilde{F}(z_{\mathrm{dR}}) \in \hat{\mathcal{O}}_{U'}(V)\llbracket z_{\mathrm{dR}} - \theta_X(z_{\mathrm{dR}})\rrbracket.$$

%Define the \emph{$\hat{\mathcal{O}}$-valued $q_{\mathrm{dR}}$-expansion of $w'$} as 
%$$\tilde{w}'(q_{\mathrm{dR}}) := \tilde{F}(q_{\mathrm{dR}}) \in \hat{\mathcal{O}}_{U'}(V)\llbracket %q_{\mathrm{dR}}-1\rrbracket.$$

By Theorem \ref{U'Utheorem} below, we can replace ``$\mathbf{U}'$'' with $\mathcal{V}_x$ in the above definition. However, we will need the above definition to prove that Theorem. 

\end{definition}

\subsection{The action of $B(\mathbb{Q}_p)$ on $U$}
In this section, we obtain some intermediate results on how certain upper-triangular elements $B(\mathbb{Q}_p) \subset GL_2(\mathbb{Q}_p)$ (see (\ref{BQp})) act on $U = \mathcal{Y}^{\mathrm{Ig}}(\epsilon_0)$ (see Definition \ref{Udefinition}).

\begin{definition}
For any $m, n \in \mathbb{Z}_{\ge 0}$ and any $j \in \mathbb{Z}_p$, define the element
\begin{equation}\label{gammajpnm}\gamma_{j/p^n,m} = \left(\begin{array}{ccc} 1 & j/p^n\\
0 & 1/p^m\\
\end{array}\right) \in GL_2(\mathbb{Q}_p).
\end{equation}
\end{definition}

\begin{proposition}\label{shiftproposition}Recall $U = \mathcal{Y}^{\mathrm{Ig}}(\epsilon_0)$ from Definition \ref{Udefinition}. 
\begin{enumerate}
\item If $n \le n(\epsilon_0) \le m$ then $\gamma_{j/p^n,m}$ satisfies 
\begin{equation}\label{UU'gamma}U \cdot \gamma_{j/p^n,m} \subset U
%, \hspace{1cm} U' \cdot \gamma_{j/p^n,m} \subset U',
\end{equation}
so that we get an induced map
$$\gamma_{j/p^n,m}^* : \mathcal{O}\mathbb{B}_{\mathrm{dR},U}(U) \rightarrow \mathcal{O}\mathbb{B}_{\mathrm{dR},U}(U)$$
via pullback. 
\item Letting $U^+ = \mathcal{Y}^{\mathrm{Ig}}(\epsilon_0)^+$, under the same assumptions we also have
$$U^+ \cdot \gamma_{j/p^n,m} \subset U^+.$$
\item If moreover $m = 0$, then the above inclusions are equalities. 
%\item For general $m \le n \le m + n(\epsilon_0)$,
%\begin{equation}\label{qshift}\mathrm{exp}(\alpha)\left(\frac{j e_1}{p^n} + \frac{e_2}{p^m}\right) = [(\zeta_{p^n}^j,\zeta_{p^{n+1}}^j,\ldots)]q_{\mathrm{dR}}^{1/p^m}.
%\end{equation}
%\end{enumerate}
\end{enumerate}
\end{proposition}

\begin{proof}We show (1) and (3); the proof of (2) is entirely analogous to the proof of (1), replacing level structures by Drinfeld level structures. Recall that $\mathcal{E}(\epsilon_0) \rightarrow U = \mathcal{Y}^{\mathrm{Ig}}(\epsilon_0)$ is the universal object, with $\Gamma(p^{\infty})$-level structure $(e_1,e_2)$. From the definition of the $GL_2(\mathbb{Q}_p)$-action, $\mathcal{E}(\epsilon_0) \cdot \gamma_{j/p^n,m}$ has a canonical subgroup of maximal level $n(\epsilon_0)$, and $\Gamma(p^{\infty})$-level structure $(e_1',e_2')$ with $e_1' \pmod{p^{n(\epsilon_0)}}$ trivializing the order-$p^{n(\epsilon_0)}$ canonical subgroup. 

Moreover, since 
$$\gamma_{j/p^n,m} \cdot \left(\begin{array}{ccc} 1 & 0\\
0 & p^n\\
\end{array}\right) = \left(\begin{array}{ccc} 1 & 0\\
0 & p^{n-m}\\
\end{array}\right)\cdot \left(\begin{array}{ccc} 1 & j\\
0 & 1\\
\end{array}\right),$$
and $\left(\begin{array}{ccc} 1 & j\\
0 & 1\\
\end{array}\right) \in GL_2(\mathbb{Z}_p)$ has trivial underlying isogeny under the moduli interpretation of the right $GL_2(\mathbb{Q}_p)$-action on $Y_{\infty}$, the underlying (false) elliptic curves of $\mathcal{E}(\epsilon_0) \cdot \left(\begin{array}{ccc} 1 & 0\\
0 & p^{n-m}\\
\end{array}\right)$ and $\mathcal{E}(\epsilon_0) \cdot \gamma_{j/p^n,m} \cdot \left(\begin{array}{ccc} 1 & 0\\
0 & p^n\\
\end{array}\right)$ are isomorphic. 

Suppose we are given any $(A,i,P,e_1,e_2) \in U(R,R^+)$ where $A$ is a (false) elliptic curve, $i$ an $\mathcal{O}_D$-endomorphism structure, $P$ a $\Gamma$-level structure as in Convention \ref{Yconvention} and $(e_1,e_2)$ a $\Gamma(p^{\infty})$-level structure. Write 
$$(A',i',P',e_1',e_2') = (A,i,P,e_1,e_2) \cdot \gamma_{j/p^n,m}.$$
The previous paragraph shows that $A = A'/\langle e_{1,m-n}' \rangle$, where $e_{1,m-n}' : \mathbb{Z}/p^{m-n} \xrightarrow{\sim} C_{m-n} \subset A'[p^{m-n}]$ trivializes the order $p^{m-n}$-canonical subgroup $C_{m-n}$ (recall the notation (\ref{eindefinition})). This shows that 
$$\mathrm{Ha}(A')^{p^{m-n}} \equiv \mathrm{Ha}(A) \pmod{\frak{m}R^+},$$
and so $|\mathrm{Ha}(A')| = |\mathrm{Ha}(A)|^{1/p^{m-n}} \le |\mathrm{Ha}(A)|$. Moreover, since $\mathbb{Z}_p \cdot e_{1,m-n}' \subset A'[p^{m-n}]$ is the order-$p^{m-n}$ canonical subgroup and the image of $\mathbb{Z}_p \cdot e_{1,m}' \subset A'[p^m]$ under the isogeny $A' \rightarrow A$ is the order-$p^n$ canonical subgroup $\mathbb{Z}_p \cdot e_{1,n} \subset A[p^n]$ by the assumption $n \le n(\epsilon_0)$, from standard properties of the canonical subgroup (\cite[Proposition III.2.8 (iii)]{ScholzeTorsion}) we have that $\mathbb{Z}_p \cdot e_{1,m} \subset A'[p^m]$ is the order-$p^m$ canonical subgroup. Since $m \ge n(\epsilon_0)$, we thus have $(A',i',P',e_1',e_2') \in U(R,R^+)$. Since $(A,i,P,e_1,e_2) \in U(R,R^+)$ was arbitrary, this proves (\ref{UU'gamma}).

\end{proof}

%\begin{corollary}In the setting of Proposition \ref{shiftproposition}, we have
%\begin{equation}\label{qshift}\gamma_{j/p^n,m}^*q_{\mathrm{dR}} = [(\zeta_{p^n}^j,\zeta_{p^{n+1}}^j,\ldots)]q_{\mathrm{dR}}^{1/p^m}.
%\end{equation}
%\end{corollary}

%\begin{proof}Since $t^{j/p^n} = \log([(\zeta_{p^n}^j,\zeta_{p^{n+1}}^j,\ldots)])$, we have
%$$\gamma^*q_{\mathrm{dR}} = \mathrm{exp}(\gamma^*\mathbf{z}) \overset{(\ref{zshift})}{=} \mathrm{exp}(t^{j/p^n} + z_{\mathrm{dR}}/p^m) = [(\zeta_{p^n}^j,\zeta_{p^{n+1}}^j,\ldots)]q_{\mathrm{dR}}^{1/p^m}.$$
%\end{proof}

\subsection{Comparison between $q_{\mathrm{dR}}$-expansions and Serre-Tate expansions on the ordinary locus}\label{'ordinarysection}

Let us explain how the objects in (\ref{gluezdRsection}) and Definition \ref{'zqexpansions} recover Katz's theory of $p$-adic modular forms (\cite{Katzpamf}) and the theory of Serre-Tate expansions (\cite{KatzST}) on the ordinary locus.

%Let $\mathcal{Y}^{\mathrm{Ig}}$ be as in Definition \ref{mathcalYIgDefinition}. Recall $\mathcal{Y}^{\mathrm{Ig}} \in Y_{\text{pro\'{e}t}}$. Given any sheaf $\mathcal{F}$ on $Y_{\text{\'{e}t}}/U$, we have a restriction to $\overline{\mathcal{Y}^{\mathrm{Ig}}}$ from (\ref{fiberinfty}) $\mathcal{F}|_{\overline{\mathcal{Y}^{\mathrm{Ig}}}}$. Since $\mathcal{Y}^{\mathrm{Ig}} \subset \overline{\mathcal{Y}^{\mathrm{Ig}}}$ is an affinoid open, we can take a further restriction $\mathcal{F}|_{\mathcal{Y}^{\mathrm{Ig}}}$. 
%In particular, we can restrict the results of Section \ref{zdRqdRqdRexpansionssection} to $\mathcal{Y}^{\mathrm{Ig}} \subset U$. %In this section, we make a particularly convenient choice of such a $U$ (See (\ref{'ordinaryU}) below.)

Recall the universal object $\pi : \mathcal{E} \rightarrow Y$. Then the universal trivialization $e_{\text{\'{e}t}} : \hat{\mathbb{Z}}_{p,Y^{\mathrm{Ig}}} \xrightarrow{\sim} T_p\mathcal{E}^{\text{\'{e}t}}|_{Y^{\mathrm{Ig}}}$ gives (via the Weil pairing, see \cite[Proof of Theorem 2.1]{KatzST}) a canonical isomorphism
$$e_{\text{\'{e}t}} : \hat{\mathcal{E}}|_{Y^{\mathrm{Ig}}} \xrightarrow{\sim} \hat{\mathbb{G}}_{m,Y^{\mathrm{Ig}}},$$
where $\hat{\mathcal{E}}$ is the formal group attached by \cite[Proposition 1]{Tate} to $\mathcal{E}[p^{\infty}]$ (where $\mathcal{E}[p^n]$ is as in Convention \ref{idempotentconvention}), and $\hat{\mathbb{G}}_{m,S}$ is the formal multiplicative group on $S$.
This pulls back to a canonical isomorphism
\begin{equation}\label{'torustrivialization}e_2 : \hat{\mathcal{E}}|_{\mathcal{Y}^{\mathrm{Ig}}} \xrightarrow{\sim} \hat{\mathbb{G}}_{m,\mathcal{Y}^{\mathrm{Ig}}}.
\end{equation}
Here, recall that $\mathcal{Y}^{\mathrm{Ig}} \rightarrow Y^{\mathrm{Ig}}$ is defined by $e_2 \mapsto e_2 \pmod{\mathbb{Z}_p \cdot e_1}$ (see (\ref{Igusacovers})).

%Let $\mathcal{D} \subset \mathcal{Y}^{\mathrm{Ig}}$ be the preimage of any ordinary residue disc $D \subset Y^{\mathrm{ord}}$ under $\mathcal{Y}^{\mathrm{Ig}} \rightarrow Y^{\mathrm{ord}}$. Let 
%\begin{equation}\label{DIg}D^{\mathrm{Ig}} = D \times_{Y^{\mathrm{ord}}}Y^{\mathrm{Ig}},
%\end{equation}
%and let $A_0$ be the elliptic curve corresponding to the special fiber of the formal model of $D$. Taking the adic generic fiber of the isomorphism (\ref{Igusatriv}) gives us a natural isomorphism of adic spaces over $\mathrm{Spa}(W(\overline{\mathbb{F}}_p)[1/p],W(\overline{\mathbb{F}}_p))$ 
%$$D^{\mathrm{Ig}} \xrightarrow{\sim} \bigsqcup_{\mathbb{Z}_p \xrightarrow{\sim}T_pA_0(\overline{\mathbb{F}_p})} \hat{\mathbb{G}}_m$$
%given by the Serre-Tate coordinate $q_{\mathrm{ST}} \in \mathcal{O}_Y(D^{\mathrm{Ig}})$. In particular, we have a natural pro-finite \'{e}tale cover $\mathcal{D} \rightarrow D^{\mathrm{Ig}}$ with Galois group $B^1$. Let $\tilde{\hat{\mathbb{G}}}_m = \varprojlim_{q \mapsto q^p}\hat{\mathbb{G}}_m$, where $q-1$ is the coordinate on $\hat{\mathbb{G}}_m$, and let
%\begin{equation}\label{'ordinaryU}U = \mathcal{D} \times_{D^{\mathrm{Ig}}} \bigsqcup \tilde{\hat{\mathbb{G}}}_m.
%\end{equation}
%Then by mimicking the proof of \cite[Proposition 6.10]{Scholze}, we have 
%$$\mathcal{O}\mathbb{B}_{\mathrm{dR},U}^+ \cong \mathbb{B}_{\mathrm{dR},U}^+\llbracket X\rrbracket,\hspace{1cm} q_{\mathrm{ST}} \mapsto [q_{\mathrm{ST}}^{\flat}] + X.$$
%Thus we can and do take this $U$ for Choice \ref{'choice1}.

Let $T$ be the coordinate on the formal group $\hat{\mathbb{G}}_{m,Y^{\mathrm{Ig}}}$ determined by 
$$\zeta := (1,\zeta_p,\zeta_{p^2},\ldots) \in \mathbf{\Gamma}(T_p\hat{\mathbb{G}}_{m,Y^{\mathrm{Ig}}}) = \mathbf{\Gamma}(\hat{\mathbb{Z}}_{p,Y^{\mathrm{Ig}}}(1))$$
via the procedure give in \cite[Proof of Proposition 1]{Tate}, where
$$\zeta_{p^n} := \langle e_1,e_2\rangle \pmod{p^n} = \langle e_1,e_{\text{\'{e}t}}\rangle \pmod{p^n} \in \mathbf{\Gamma}(\hat{\mathbb{Z}}_{p,Y^{\mathrm{Ig}}}(1)/p^n) = \mathbf{\Gamma}(\mu_{p^n,Y^{\mathrm{Ig}}}).$$
Then 
$$\frac{dT}{1+T} \in \mathbf{\Gamma}(\Omega_{\hat{\mathbb{G}}_{m,Y^{\mathrm{Ig}}}/Y^{\mathrm{Ig}}})$$
is a generator. (Here, we use the notation of Convention \ref{Gammaconvention} on the right-hand side.) Let 
$$\nu \in \mathbf{\Gamma}(\mathrm{Lie}(\hat{\mathbb{G}}_{m,Y^{\mathrm{Ig}}}))$$
be its Poincar\'{e} dual. The image of $\nu \cdot t$ under the relative Hodge-Tate filtration 
%$$\mathrm{Lie}(\hat{\mathbb{G}}_{m,Y^{\mathrm{Ig}}})\otimes_{\mathcal{O}_{Y^{\mathrm{Ig}}}}\hat{\mathcal{O}}_{Y^{\mathrm{Ig}}}(\mathcal{Y}^{\mathrm{Ig}})(1) = \mathrm{Lie}(\hat{\mathbb{G}}_{m,,Y^{\mathrm{Ig}}})\otimes_{\mathcal{O}_{Y^{\mathrm{Ig}}}}\hat{\mathcal{O}}_{Y^{\mathrm{Ig}}}(\mathcal{Y}^{\mathrm{Ig}})\cdot t 
$$\mathbf{\Gamma}(\mathrm{Lie}(\hat{\mathbb{G}}_{m,Y^{\mathrm{Ig}}})) \cdot t = \mathbf{\Gamma}(\mathrm{Lie}(\hat{\mathbb{G}}_{m,Y^{\mathrm{Ig}}}))(1) \rightarrow  T_p\hat{\mathbb{G}}_{m,Y^{\mathrm{Ig}}} \otimes_{\hat{\mathbb{Z}}_{p,Y^{\mathrm{Ig}}}} \hat{\mathcal{O}}_{Y^{\mathrm{Ig}}}(Y^{\mathrm{Ig}}) \rightarrow T_p\hat{\mathbb{G}}_{m,\mathcal{Y}^{\mathrm{Ig}}} \otimes_{\hat{\mathbb{Z}}_{p,\mathcal{Y}^{\mathrm{Ig}}}}\hat{\mathcal{O}}_{\mathcal{Y}^{\mathrm{Ig}}}(\mathcal{Y}^{\mathrm{Ig}})$$
is 
\begin{equation}\label{zetae}\zeta = e_2(e_1),
\end{equation}
where the right-hand side is the image of $e_1 \in \mathbf{\Gamma}(T_p\mathcal{E}|_{\mathcal{Y}^{\mathrm{Ig}}})$ under the map $e_2 : T_p\mathcal{E}|_{\mathcal{Y}^{\mathrm{Ig}}} \rightarrow T_p\hat{\mathbb{G}}_{m,\mathcal{Y}^{\mathrm{Ig}}}$ induced by (\ref{'torustrivialization}).

Let 
\begin{equation}\label{'Katzgenerator}w_{\mathrm{can}}^{\mathrm{Katz}} := e_{\text{\'{e}t}}^*\frac{dT}{1+T}\in \omega(Y^{\mathrm{Ig}})
\end{equation}
be the generator constructed in \cite[Section 3.3]{KatzST} (where it is denoted ``$\omega(\alpha)$'' with $\alpha = e_{\text{\'{e}t}}$ in loc. cit.). Since $z_{\mathrm{HT}} = \infty$ on $\mathcal{Y}^{\mathrm{Ig}}$, $e_1$ generates the connected component $T_p\mathcal{E}^0$ of $T_p\mathcal{E}$ and so $e_2 \mapsto e_{\text{\'{e}t}}$ under the projection $\mathcal{Y}^{\mathrm{Ig}} \rightarrow Y^{\mathrm{Ig}}$ from (\ref{Igusacovers}). Thus we have
$$w_{\mathrm{can}}^{\mathrm{Katz}}|_{\mathcal{Y}^{\mathrm{Ig}}} = e_2^*\frac{dT}{1+T}.$$
By slight abuse of notation, we will let $w_{\mathrm{can}}^{\mathrm{Katz}}$ also denote $w_{\mathrm{can}}^{\mathrm{Katz}}|_{\mathcal{Y}^{\mathrm{Ig}}}$. Then the Main Theorem 3.7.1. of op. cit. shows
$$\mathrm{KS}(w_{\mathrm{can}}^{\mathrm{Katz}, \otimes 2}) = d\log q_{\mathrm{ST}}.$$
%Take $w = w_{\mathrm{can}}^{\mathrm{Katz}}$ in Choice \ref{'choicew}. 

\begin{theorem}We have 
\begin{equation}\label{'y1}y_{\mathrm{dR}}(w_{\mathrm{can}}^{\mathrm{Katz}}) = 1.
\end{equation}
\end{theorem}

\begin{proof}For the rest of the proof, let $w = w_{\mathrm{can}}^{\mathrm{Katz}}$. Let $\zeta \in \mathbf{\Gamma}(T_p\hat{\mathbb{G}}_{m,\mathcal{Y}^{\mathrm{Ig}}}) = \mathbf{\Gamma}(\hat{\mathbb{Z}}_{p,\mathcal{Y}^{\mathrm{Ig}}}\cdot t)$ be the image of $e_1$ under (\ref{'torustrivialization}). Let 
$$i : \pi_*\Omega_{\hat{\mathbb{G}}_{m,\mathcal{Y}^{\mathrm{Ig}}}/\mathcal{Y}^{\mathrm{Ig}}} \rightarrow T_p\hat{\mathbb{G}}_{m,\mathcal{Y}^{\mathrm{Ig}}} \otimes_{\hat{\mathbb{Z}}_{p,\mathcal{Y}^{\mathrm{Ig}}}}\mathcal{O}\mathbb{B}_{\mathrm{dR},\mathcal{Y}^{\mathrm{Ig}}}^+$$
be the relative de Rham comparison map from \cite[Theorem 8.8]{Scholze} applied to $\hat{\mathbb{G}}_{m,\mathcal{Y}^{\mathrm{Ig}}} \rightarrow \mathcal{Y}^{\mathrm{Ig}}$. Thus, by the discussion immediately preceding (\ref{zetae}), we have
\begin{equation}\label{zetat}\langle \zeta,i(\frac{dT}{1+T})\rangle = \langle \nu \cdot t,\frac{dT}{1+T}\rangle_{\mathrm{dR}} = t,\end{equation}
where $\langle \cdot,\cdot\rangle_{\mathrm{dR}}$ is the Poincar\'{e} duality pairing as before. 

By functoriality of the Weil pairing along with (\ref{'torustrivialization}), we see that 
\begin{align*}y_{\mathrm{dR}}(w) = \frac{1}{t}\langle e_1,x_{\mathrm{dR}}(w)e_1 + y_{\mathrm{dR}}(w)e_2\rangle = \frac{1}{t}\langle e_1,i_{\mathrm{dR}}(e_2^*\frac{dT}{1+T})\rangle &= \frac{1}{t}\langle e_2(e_1),i(\frac{dT}{1+T})\rangle \\
&= \frac{1}{t}\langle \zeta,i(\frac{dT}{1+T})\rangle \overset{(\ref{zetat})}{=} 1.
\end{align*}
%By \cite{Colmez}, the specialization at each classical point of the right-hand side of the above equation is 1.

% explicitly compute the relative comparison theorem for $\mathcal{E}|_{\mathcal{Y}^{\mathrm{Ig}}}$ and conclude that the $y_{\mathrm{dR}}$ corresponding to $w$ is 1.  
\end{proof}

Fix any $m \in \mathbb{Z}_{\ge 0}$, and let $w \in \omega(U_m)$ be as in Choice \ref{'choicew}. Note that since $\mathcal{Y}^{\mathrm{Ig}} = \mathcal{Y}^{\mathrm{Ig}}(0) \subset \mathcal{Y}^{\mathrm{Ig}}(\epsilon_0) = U$ (Definition \ref{Udefinition}), then $\mathcal{Y}^{\mathrm{Ig}} = \mathcal{Y}^{\mathrm{Ig}} \cdot g^m  \subset U \cdot g^m \overset{(\ref{Um})}{=} U_m$.

%\begin{corollary}\label{'nonempty1}$U'$ for the choice of $U$ (\ref{'ordinaryU}) is nonempty, and in fact $U' = U$. 
%\end{corollary}

%\begin{proof}From (\ref{'y1}), we see that $\theta(\theta_X(y_{\mathrm{dR}})) = 1$. 
%\end{proof}

%Hence,
%\begin{equation}\label{'zcalculation}z_{\mathrm{dR}} = -x_{\mathrm{dR}} \in \mathrm{Fil}^1\mathcal{O}\mathbb{B}_{\mathrm{dR},\mathcal{Y}^{\mathrm{Ig}}}^+(\mathcal{Y}^{\mathrm{Ig}}) \cdot t^{-1}, \hspace{1cm} t\cdot z_{\mathrm{dR}} \in \mathrm{Fil}^1\mathcal{O}\mathbb{B}_{\mathrm{dR},\mathcal{Y}^{\mathrm{Ig}}}^+(\mathcal{Y}^{\mathrm{Ig}}).
%\end{equation}

%Recall that $B \cong \mathrm{Gal}(\mathcal{D}/D) = \mathrm{Gal}(\mathcal{Y}^{\mathrm{Ig}}/Y^{\mathrm{ord}})$. 
%\begin{proposition}\label{Bactionproposition}For $g = \left(\begin{array}{ccc} a & b\\
%0 & d\\
%\end{array}\right) \in B \cap GL_2(\mathbb{Z}_p)$, we have 
%$$g^*(t\cdot z_{\mathrm{dR}}) = (ad-bc)\frac{dz_{\mathrm{dR}} + b}{a}.$$
%\end{proposition}

%\begin{proof}From the definition of the $GL_2(\mathbb{Q}_p)$-action on $Y_{\infty}$, we have $(e_1,e_1) \cdot \left(\begin{array}{ccc} a & b\\
%c & d\\
%\end{array}\right) = (ae_1 + ce_2, be_1 + de_2)$. Now from (\ref{'tdefinition}), we see that $g^*t = (ad-bc)t$. Now the assertion follows once we show that 
%$$g^*z_{\mathrm{dR}} = \frac{dz_{\mathrm{dR}} + b}{a}.$$
%This is proven exactly in the same way as it is for $z_{\mathrm{HT}}$ (see \cite[Section 2.4]{ChojeckiHansenJohansson}; again, our $z_{\mathrm{HT}}$ is $1/\frak{z}$ in loc. cit.).
%\end{proof}

\begin{corollary}\label{U'YIg}$\mathcal{Y}^{\mathrm{Ig}} \subset \mathbf{U}'$.
\end{corollary}

\begin{proof}From (\ref{'y1}), we see that 
\begin{equation}\label{YIgwcan}\mathcal{Y}^{\mathrm{Ig}} \subset \{\theta(\theta_X(y_{\mathrm{dR}}(w_{\mathrm{can}}^{\mathrm{Katz}}))) = 1\}.
\end{equation}
%Since the set on the right-hand side is closed, we have
%$$\overline{\mathcal{Y}^{\mathrm{Ig}}} \subset \{\theta(\theta_X(y_{\mathrm{dR}}(w_{\mathrm{can}}))) = 1\}$$
%which gives the assertion.
\end{proof}

Recall $\mathcal{Y}^{\mathrm{Ig}}(0) = \mathcal{Y}^{\mathrm{Ig}}$, and $\mathcal{Y}^{\mathrm{Ig}}(0)^+$ is the integral model of $\mathcal{Y}^{\mathrm{Ig}}(0)$ (see Definition \ref{mathcalYIgDefinition} and (\ref{Uplus})). Letting $\mathbf{z}$ be from (\ref{zbfdefinition}), Proposition \ref{'KSproposition} shows that
\begin{equation}\label{'STequal}\nabla\mathbf{z}|_{\mathcal{Y}^{\mathrm{Ig}}} = \mathrm{KS}(w_{\mathrm{can}}^{\otimes 2}|_{\mathcal{Y}^{\mathrm{Ig}}}) = \mathrm{KS}(w_{\mathrm{can}}^{\mathrm{Katz},\otimes 2}) = d\log q_{\mathrm{ST}} \in \Omega_{\mathcal{Y}^{\mathrm{Ig}}(0)^+}(\mathcal{Y}^{\mathrm{Ig}}(0)^+),
\end{equation}
where the last inclusion follows from \cite[Main Theorem]{KatzST}.

Recall $U = \mathcal{Y}^{\mathrm{Ig}}(\epsilon_0)$ from Definition \ref{Udefinition}, and let $\pi_+ : \mathcal{E}^+ \rightarrow Y^+$ be the universal object (universal (false) elliptic curve with $\mathcal{O}_D$-endomorphism structure and $\Gamma = \Gamma(N)$-level structure, $N \ge 4$, $(N,p) = 1$ as in Convention \ref{Yconvention}). Let 
$$\omega_+ = \pi_{+,*}\Omega_{\mathcal{E}^+/Y^+}$$
be as in (\ref{omegaY}).

Recall the notation of Convention \ref{Yconvention}, and that $Y(0) = Y^{\mathrm{ord}}$, where $Y^{\mathrm{ord}}$ is the adic generic fiber of the ordinary locus as in Section \ref{furtherIgusasection}.

\begin{theorem}\label{STanalyticcontinuationtheorem}Suppose $k \ge 0$ and $w' \in \omega^{\otimes k} \otimes_{\mathcal{O}_Y}\mathbb{B}_{\mathrm{dR},U}(W)\llbracket q_{\mathrm{dR}}-1\rrbracket$ where $W \subset U$ is the inverse image under the $\Gamma_{0,p}(p^{n(\epsilon_0)})$-cover $U = \mathcal{Y}^{\mathrm{Ig}}(\epsilon_0) \rightarrow Y(\epsilon_0)$ of some open set $W_0 \subset Y(\epsilon_0)$ with $Y(0) \subset W_0$. 
Then
$$\mathcal{Y}^{\mathrm{Ig}} \subset W.$$
Recall that $\mathcal{Y}^{\mathrm{Ig}} = \mathcal{Y}^{\mathrm{Ig}}(0)$ and $\mathcal{Y}^{\mathrm{Ig}}(0)^+$ is the formal model of $\mathcal{Y}^{\mathrm{Ig}}(0)$. Assume moreover that 
$$w'|_{\mathcal{Y}^{\mathrm{Ig}}} \in \omega_+^{\otimes k}(\mathcal{Y}^{\mathrm{Ig}}(0)^+).$$
Then:
\begin{enumerate}
\item The $q_{\mathrm{dR}}$-expansion $w'(q_{\mathrm{dR}})$ (Definition \ref{'zqexpansions}) satisfies
%$$w'(q_{\mathrm{dR}})|_{\mathcal{Y}^{\mathrm{Ig}}} \in \mathbb{B}_{\mathrm{dR},\mathcal{Y}^{\mathrm{Ig}}}^+(\mathcal{Y}^{\mathrm{Ig}})\llbracket q_{\mathrm{dR}}-1\rrbracket,$$
%and moreover
\begin{equation}\label{ordinaryintegral}w'(q_{\mathrm{dR}})|_{\mathcal{Y}^{\mathrm{Ig}}} \in \hat{\mathcal{O}}_{\mathcal{Y}^{\mathrm{Ig}}}^+(\mathcal{Y}^{\mathrm{Ig}})\llbracket q_{\mathrm{dR}}-1\rrbracket, 
\end{equation}
which we view inside $\mathbb{B}_{\mathrm{dR},\mathcal{Y}^{\mathrm{Ig}}}^+(\mathcal{Y}^{\mathrm{Ig}})\llbracket q_{\mathrm{dR}}-1\rrbracket$ via the restriction of (\ref{completedOBmap}) to $\mathcal{Y}^{\mathrm{Ig}} \subset U$
$$\hat{\mathcal{O}}_{\mathcal{Y}^{\mathrm{Ig}}}^+(\mathcal{Y}^{\mathrm{Ig}}) \rightarrow \mathbb{B}_{\mathrm{dR},\mathcal{Y}^{\mathrm{Ig}}}^+(\mathcal{Y}^{\mathrm{Ig}}).$$
\item Moreover, for any $y \in \mathcal{Y}^{\mathrm{Ig}}$, the power series
$$w'(y)(q_{\mathrm{dR}})  \in \hat{\mathcal{O}}_{\mathcal{Y}^{\mathrm{Ig}}}^+(y)\llbracket q_{\mathrm{dR}}-1\rrbracket,$$
obtained by specializing the coefficients of $w'(q_{\mathrm{dR}})|_{\mathcal{Y}^{\mathrm{Ig}}}$ along the map 
$$\hat{\mathcal{O}}_{\mathcal{Y}^{\mathrm{Ig}}}^+(\mathcal{Y}^{\mathrm{Ig}}) \rightarrow \hat{\mathcal{O}}_{\mathcal{Y}^{\mathrm{Ig}}}^+(y),$$
is equal, after formally replacing the variable $q_{\mathrm{dR}}-1$ with $q_{\mathrm{ST},y}-1$ (see Definition \ref{STrecenter} (3)), to the Serre-Tate expansion of $w'$ centered at $y$ (Definition \ref{STrecenter} (4)). 
\end{enumerate}
\end{theorem}

\begin{proof}The fact that $\mathcal{Y}^{\mathrm{Ig}} \subset W$ follows immediately from the assumptions. From the last inclusion of (\ref{'STequal}), we get (\ref{ordinaryintegral}). For the final assertion, note that the Serre-Tate expansion centered at $y$ is computed the Taylor expansion in the variable $q_{\mathrm{ST},y}-1$ centered at $q_{\mathrm{ST},y} - 1 = 0$, which by (\ref{'STequal}), (\ref{Taylorcoefficient}), the fact that $\nabla\mathbf{z} = \nabla(q_{\mathrm{dR}})/q_{\mathrm{dR}}$ and Definition \ref{'zqexpansions} is easily seen to coincide with $w'(y)(q_{\mathrm{dR}}-1)$ after formally replacing the variable $q_{\mathrm{ST},y}-1$ with $q_{\mathrm{dR}}-1$. 
\end{proof}

%\begin{remark}The previous theorem justifies viewing $q_{\mathrm{dR}}$-expansions as analytic continuations of families of Serre-Tate expansions.
%\end{remark}

\subsection{Extending the domain of $q_{\mathrm{dR}}$-expansions}
All of the constructions of the previous sections, including $q_{\mathrm{dR}}$-expansions (Definition \ref{'zqexpansions}), are \emph{a priori} only defined on the open subset 
$$\mathbf{U}' \subset \mathcal{V}_x = \{z_{\mathrm{HT}} \neq 0\}.$$
The purpose of this section is to in fact show that $\mathbf{U}' = \mathcal{V}_x$ (see Theorem \ref{U'Utheorem} below), thus showing that all these constructions are defined over $\mathcal{V}_x$. 

\begin{definition}\label{thetaqdefinition}\begin{enumerate}
\item For any pro\'{e}tale open $W \rightarrow \mathbf{U}'$, define the $\mathbb{B}_{\mathrm{dR},\mathbf{U}'}(W)$-algebra homomorphism 
\begin{equation}\label{thetaq}\theta_q : \mathbb{B}_{\mathrm{dR},\mathbf{U}'}(W)\llbracket q_{\mathrm{dR}}-1\rrbracket \rightarrow \mathbb{B}_{\mathrm{dR},\mathbf{U}'}(W)
\end{equation}
to be reduction modulo the ideal $(q_{\mathrm{dR}}-1) \cdot \mathbb{B}_{\mathrm{dR},\mathbf{U}'}(W)\llbracket q_{\mathrm{dR}}-1\rrbracket$. In other words, 
$$\theta_q(F(q_{\mathrm{dR}})) \in \mathbb{B}_{\mathrm{dR},\mathbf{U}'}(W)$$
is the constant term of 
$$F(q_{\mathrm{dR}}) \in \mathbb{B}_{\mathrm{dR},\mathbf{U}'}(W)\llbracket q_{\mathrm{dR}}-1\rrbracket.$$
In particular, 
$$\theta_q(q_{\mathrm{dR}}-1) = 0.$$
%This extends to a map
%$$\theta_q : \mathbb{B} \rightarrow \mathbb{B}_{\mathrm{dR},{\mathcal{V}_x}}({\mathcal{V}_x}),$$
%given by reducing modulo $(q_{\mathrm{dR}}-1) \cdot \mathbb{B}$.
%, and similarly extends to a map
%$$\theta_q : \mathbb{B} \rightarrow \mathbb{B}_{\mathrm{dR},{\mathcal{V}_x}}(W),$$
%given by reducing modulo $(q_{\mathrm{dR}}-1) \cdot \mathbb{B}$.
\item Recall the identification of $\hat{\mathcal{O}}_{U_m}(U_m)$-algebras
$$\mathbb{B}_{\mathrm{dR},U_m}(U_m) \overset{(\ref{BdecompositionUm})}{=} \hat{\mathcal{O}}_{U_m}(U_m)(\!(t)\!).$$
Define a $\hat{\mathcal{O}}_{U_m}(U_m)$-module homomorphism
$$\theta_t : \mathbb{B}_{\mathrm{dR},U_m}(U_m) \rightarrow \hat{\mathcal{O}}_{U_m}(U_m)$$
to be reduction modulo the $\hat{\mathcal{O}}_{U_m}(U_m)$-ideal
$$\{t^n\}_{n \in \mathbb{Z} \setminus \{0\}} \cdot \hat{\mathcal{O}}_{U_m}(U_m) \subset \hat{\mathcal{O}}_{U_m}(U_m)(\!(t)\!) = \mathbb{B}_{\mathrm{dR},U_m}(U_m).$$
Note that 
$$\theta_t|_{\mathbb{B}_{\mathrm{dR},U_m}^+(U_m)} = \theta$$
as maps $\mathbb{B}_{\mathrm{dR},U_m}^+(U_m) \rightarrow \hat{\mathcal{O}}_{U_m}(U_m)$. 
\item Recall the identification of $\hat{\mathcal{O}}_{\mathcal{V}_x}(\mathcal{V}_x)$-algebras
$$\mathbb{B}_{\mathrm{dR},\mathcal{V}_x}(\mathcal{V}_x) \overset{(\ref{Bdecomposition''})}{=} \hat{\mathcal{O}}_{\mathcal{V}_x}(\mathcal{V}_x)(\!(t)\!).$$
Define a $\hat{\mathcal{O}}_{\mathcal{V}_x}(\mathcal{V}_x)$-module homomorphism
$$\theta_t : \mathbb{B}_{\mathrm{dR},\mathcal{V}_x}(\mathcal{V}_x) =  \hat{\mathcal{O}}_{\mathcal{V}_x}(\mathcal{V}_x)(\!(t)\!) \rightarrow \hat{\mathcal{O}}_{\mathcal{V}_x}(\mathcal{V}_x)
$$
to be reduction modulo the $\hat{\mathcal{O}}_{\mathcal{V}_x}(\mathcal{V}_x)$-ideal 
$$\{t^n\}_{n \in \mathbb{Z} \setminus \{0\}}\cdot \hat{\mathcal{O}}_{\mathcal{V}_x}(\mathcal{V}_x) \subset \hat{\mathcal{O}}_{\mathcal{V}_x}(\mathcal{V}_x)(\!(t)\!) = \mathbb{B}_{\mathrm{dR},\mathcal{V}_x}(\mathcal{V}_x).$$
Note that 
$$\theta_t|_{\mathbb{B}_{\mathrm{dR},\mathcal{V}_x}^+(\mathcal{V}_x)} = \theta$$
as maps $\mathbb{B}_{\mathrm{dR},\mathcal{V}_x}^+(\mathcal{V}_x) \rightarrow \hat{\mathcal{O}}_{\mathcal{V}_x}(\mathcal{V}_x)$. 
\item In all, we get maps
\begin{equation}\label{thetat}\begin{split}&\theta_t : \mathbb{B}_{\mathrm{dR},U_m}(U_m) = \hat{\mathcal{O}}_{U_m}(U_m)(\!(t)\!) \rightarrow \hat{\mathcal{O}}_{U_m}(U_m) \rightarrow \hat{\mathcal{O}}_{U_m}(U_m), \\
&\theta_t : \mathbb{B}_{\mathrm{dR},\mathcal{V}_x}(\mathcal{V}_x) = \hat{\mathcal{O}}_{\mathcal{V}_x}(\mathcal{V}_x)(\!(t)\!) \rightarrow \hat{\mathcal{O}}_{\mathcal{V}_x}(\mathcal{V}_x)
\end{split}
\end{equation}
where the first map is an $\hat{\mathcal{O}}_{U_m}(U_m)$-module homomorphism and the second map is an $\hat{\mathcal{O}}_{\mathcal{V}_x}(\mathcal{V}_x)$-module homomorphism. The use of the same notation $\theta_t$ for both maps is compatible: It is clear from construction that under the inclusion 
$$U_m \subset \bigcup_{m \in \mathbb{Z}_{\ge 0}}U_m \overset{(\ref{Vxunion'})}{\subset} \mathcal{V}_x$$
we have that the composition
$$\mathbb{B}_{\mathrm{dR},\mathcal{V}_x}(\mathcal{V}_x) \rightarrow \mathbb{B}_{\mathrm{dR},U_m}(U_m) \xrightarrow{\theta_t} \hat{\mathcal{O}}_{U_m}(U_m)$$
is equal to $\theta_t : \mathbb{B}_{\mathrm{dR},U_m}(U_m) \rightarrow \hat{\mathcal{O}}_{U_m}(U_m)$. 
\end{enumerate}
\end{definition}

\begin{proposition}\label{'thetaproposition'}For any $m \in \mathbb{Z}_{\ge 0}$ such that 
$$U_m \subset \mathbf{U}',$$
we have that the restriction of $\theta_t \circ \theta_q$ to $\mathcal{O}\mathbb{B}_{\mathrm{dR},\mathcal{V}_x}^+(U_m)$ is equal to $\theta_{\mathrm{dR}} : \mathcal{O}\mathbb{B}_{\mathrm{dR},\mathcal{V}_x}^+(U_m) \rightarrow \hat{\mathcal{O}}_{\mathcal{V}_x}(U_m)$ from Definition \ref{'thetadefinitions}.
\end{proposition}

\begin{proof}Since $U_m \subset \mathbf{U}'$ by assumption, then 
\begin{equation}\label{Umunion}U_m = U_m \cap \mathbf{U}' \overset{(\ref{UmU'})}{=}  \bigcup_{w \in \omega(U_m)}U_{m,w}'.
\end{equation}
We have compatible natural inclusions
\begin{align*}\mathcal{O}\mathbb{B}_{\mathrm{dR},U_m}^+(U_{m,w}') &\overset{(\ref{newOBdRinclusionm})}{\subset} \mathbb{B}_{\mathrm{dR},U_m}^+(U_{m,w}') \llbracket X_m/t\rrbracket \overset{(\ref{changeofvars})}{=} \mathbb{B}_{\mathrm{dR},U_m}^+(U_{m,w}')\llbracket q_{\mathrm{dR}}-1\rrbracket
\end{align*}
which satisfy $\theta_{\mathrm{dR}} = \theta \circ \theta_{X_m} = \theta \circ \theta_q$. Letting $w \in \omega(U_m)$ vary, and then gluing using (\ref{Umunion}), we get a natural inclusion 
$$\mathcal{O}\mathbb{B}_{\mathrm{dR},U_m}^+(U_m) \subset \mathbb{B}_{\mathrm{dR},U_m}^+(U_m) \llbracket q_{\mathrm{dR}}-1\rrbracket$$
satisfying $\theta_{\mathrm{dR}} = \theta \circ \theta_q$. 

\end{proof}

\begin{lemma}\label{thetaXGamma0commutelemma}%Recall the elements
%$$g, \gamma_{j/p^n,m} \in GL_2(\mathbb{Q}_p)$$
%from (\ref{gdefinition}) and (\ref{gammajpnm}). Recall $0 < \epsilon_0 < p/(p+1)$ is as in Definition \ref{Udefinition}; in particular $n(\epsilon_0) \ge 1$ (see (\ref{nepsilondefinition})). 

Let $U_m \overset{(\ref{Vxunion'})}{\subset} \mathcal{V}_x$ be the open set as in (\ref{Um}) for some $m \in \mathbb{Z}_{\ge 0}$, and suppose 
$$U_m \subset \mathbf{U}'.$$
Let $\gamma \in GL_2(\mathbb{Q}_p)$ be any element with 
$$U_m \cdot \gamma \subset U_m.$$
Then for any 
$$f \in \mathbb{B}_{\mathrm{dR},\mathbf{U}'}(U_m)\llbracket q_{\mathrm{dR}}-1\rrbracket$$
%and any 
%$$\gamma \in GL_2(\mathbb{Z}_p) \cup \{g^{-n}\}_{n \in \mathbb{Z}_{\ge 0}} \cup \{\gamma_{j/p,0}\}_{0 \le j \le p-1},$$
%we have 
%$$\mathcal{V}_x \cdot \gamma \subset \mathcal{V}_x$$
%and
we have 
\begin{equation}\label{gammathetacommute}\theta_t(\theta_q(\gamma^*f)) = \gamma^*\theta_t(\theta_q(f)).
\end{equation}
\end{lemma}

\begin{remark}As we will soon show $\mathbf{U}' = \mathcal{V}_x$ in Theorem \ref{U'Utheorem}, the assumption $U_m \subset \mathbf{U}'$ always holds \emph{a posteriori}. We note however that Lemma \ref{thetaXGamma0commutelemma} will be used in the proof of Theorem \ref{U'Utheorem}. 

\end{remark}

\begin{proof}[Proof of Lemma \ref{thetaXGamma0commutelemma}]%The fact that $\mathcal{V}_x \cdot \gamma \subset \mathcal{V}_x$ for any $\gamma \in GL_2(\mathbb{Q}_p)$ as in the statement follows immediately from $\mathcal{V}_x \overset{(\ref{Vz})}{=} \{z_{\mathrm{HT}} \neq 0\}$ and (\ref{modulartransformationidentity}). 

%Now we show (\ref{gammathetacommute}). 
Recall that $X_i = j_i - [j_i^{\flat}]$ from Theorem \ref{Utheorem}, and 
$$X_{i,m} = (g^{-m})^*X_i = (g^{-m})^*(j_i - [j_i^{\flat}]) = (g^{-m})^*j_i - [(g^{-m})^*j_i^{\flat}]$$
from (\ref{Xm}). 
Recall from (\ref{etalelocusunion4}) that $U_m = \bigcup_{i = 1}^{\mathbf{n}}U_{i,m}$ ($\mathbf{n}$ as in Definition \ref{ndefinition}). Since $U_m \subset \mathbf{U}'$ by assumption, then 
$$U_m = U_m \cap \mathbf{U}' \overset{(\ref{UmU'})}{=}  \bigcup_{w \in \omega(U_m)}U_{m,w}'.$$
%by (\ref{fullU'definition}) we have 
%$$U_m = U_m \cap \mathbf{U}' = \left(\bigcup_{i = 1}^{\mathbf{n}}U_{i,m}\right) \cap \left(\bigcup_{i = 1}^{\mathbf{n}}\bigcup_{m \in \mathbb{Z}_{\ge 0}}U_{i,m}'\right) \overset{\text{Proposition \ref{Um'proposition}}}{=} 
Thus by (\ref{changeofvars}) for $U_{m,w}'$ and $w \in \omega(U_m)$ varying, we have 
$$\mathbb{B}_{\mathrm{dR},U_m}^+(U_m)\llbracket X_m/t\rrbracket = \mathbb{B}_{\mathrm{dR},U_m}^+(U_m)\llbracket q_{\mathrm{dR}}-1\rrbracket$$
which implies
$$\mathbb{B}_{\mathrm{dR},U_m}^+(U_m) + (X_{1,m},\ldots,X_{\mathbf{n},m})\cdot \mathbb{B}_{\mathrm{dR},U_m}(U_m)\llbracket X_m\rrbracket = \mathbb{B}_{\mathrm{dR},U_m}^+(U_m) + (q_{\mathrm{dR}}-1) \cdot \mathbb{B}_{\mathrm{dR},\mathbf{U}'}(U_m)\llbracket q_{\mathrm{dR}}-1\rrbracket.$$
%Thus (\ref{gammathetacommute}) is equivalent to
%$$\gamma^*\left(f \pmod{(\ker(\theta_t \circ\theta_{X_m}))}\right) \equiv \gamma^*f\pmod{(\ker(\theta_t \circ \theta_{X_m}))},$$
%for all $1 \le i \le \mathbf{n}$. This would follow from 
%$$\gamma^*\left(\ker\left(\theta_t\circ \theta_{X_m}\right)\right) \subset \ker\left(\theta_t \circ \theta_{X_m}\right).$$
%By (\ref{thetaXm}), (\ref{thetat}), we have 
%$$\ker\left(\theta_t\circ \theta_{X_m}\right) = \bigoplus_{r \in \mathbb{Z} \setminus \{0\}}t^r \cdot \hat{\mathcal{O}}_{U_m}(U_m) + (X_{1,m},\ldots,X_{\mathbf{n},m}) \cdot \mathbb{B}_{\mathrm{dR},U_m}(U_m)\llbracket X_{i,m}\rrbracket.$$
Thus the previous inclusion would follow from the inclusions
\begin{equation}\label{theassertion}\gamma^*t \in \ker\left(\theta_t \circ \theta_q\right) \hspace{.5cm} \text{and} \hspace{.5cm} \gamma^*X_{i,m} \in \ker\left(\theta_t \circ \theta_q\right)
\end{equation}
for all $1 \le i \le \mathbf{n}$. We will show this last assertion. From (\ref{'tdefinition}) and the moduli interpretation of the $GL_2(\mathbb{Q}_p)$ action, we see that in fact $\gamma^*t \in t \cdot \mathbb{B}_{\mathrm{dR},U_m}^+(U_m)$, which implies $\theta_t(\gamma^*t) = 0$ and gives (\ref{theassertion}) for $t$. 

Now we show (\ref{theassertion}) for $X_{i,m}$. Since $\gamma^*X_{i,m} \in\mathcal{O}\mathbb{B}_{\mathrm{dR},U_m}^+(U_m)$, we have 
$$\theta_t(\theta_q(\gamma^*X_{i,m})) = \theta(\theta_q(\gamma^*X_{i,m})).$$
Thus
\begin{align*}\theta_t(\theta_q(\gamma^*X_{i,m})) &= \theta(\theta_q(\gamma^*(g^{-m})^*j_i - \gamma^*[(g^{-m})^*j_i^{\flat}])) = \theta(\theta_q(\gamma^*(g^{-m})^*j_i - [\gamma^*(g^{-m})^*j_i^{\flat}])) \\
&\overset{\text{Proposition \ref{'thetaproposition'}}}{=} \theta_{\mathrm{dR}}(\gamma^*(g^{-m})^*j_i - [\gamma^*(g^{-m})^*j_i^{\flat}]) \\
&= \theta_{\mathrm{dR}}(\gamma^*(g^{-m})^*j_i - [(\gamma^*(g^{-m})^*j_i,\gamma^*(g^{-(m+1)})^*j_i,\ldots)]) \\
&=  \gamma^*(g^{-m})^*j_i - \gamma^*(g^{-m})^*j_i = 0.
\end{align*}
This gives (\ref{theassertion}) for $X_{i,m}$.

\end{proof}

\begin{proposition}\begin{equation}\label{thetaqzzero}\theta_q(z_{\mathrm{dR}}) = 0.
\end{equation}
\end{proposition}
\begin{proof}Recall that since $q_{\mathrm{dR}} = \mathrm{exp}(t\cdot z_{\mathrm{dR}})$ (see (\ref{finalqzrelation})), we have 
$$z_{\mathrm{dR}} = \frac{1}{t}\left(q_{\mathrm{dR}}-1 - \frac{(q_{\mathrm{dR}}-1)^2}{2} + \frac{(q_{\mathrm{dR}}-1)^3}{3} - \ldots\right).$$
The assertion now follows since $\theta_q$ is reduction modulo $(q_{\mathrm{dR}}-1)\cdot\mathbb{B}_{\mathrm{dR},\mathcal{V}_x}(\mathcal{V}_x)\llbracket q_{\mathrm{dR}}-1\rrbracket$ (see (\ref{thetaq})).

\end{proof}

We can finally prove that the open subset $\mathbf{U}' \subset \mathcal{V}_x$ from (\ref{fullU'}) is equal to all of $\mathcal{V}_x$. The key input to the proof will be (\ref{YIgwcan}) which is used to prove Claim \ref{U'Uclaim} (1) below.  %At the same time, we will prove that $\mathcal{V} \subset \mathcal{V}_x$ from Definition \ref{V'V'0definition} is also equal to all of $\mathcal{V}_x$.

\begin{theorem}\label{U'Utheorem}We have 
\begin{equation}\label{U'U}\mathbf{U}' = \mathcal{V}_x.
% = \mathcal{V}.
\end{equation}
\end{theorem}

\begin{proof} %Since $\theta(\theta_q(z_{\mathrm{dR}})) = -\theta(\theta_q(x_{\mathrm{dR}}))/\theta(\theta_q(y_{\mathrm{dR}}))$, the assertion (\ref{U'U}) is equivalent to showing that $\theta(\theta_q(z_{\mathrm{dR}})) \neq \infty$ on $\mathcal{V}_x$. 

Let $g \in GL_2(\mathbb{Q}_p)$ be as in (\ref{gdefinition}). We will show:
\begin{claim}\label{U'Uclaim}
\begin{enumerate}
\item For some $m \in \mathbb{Z}_{\ge 0}$,
\begin{equation}\label{Uminside}U_{-m} \overset{(\ref{Um})}{=} \mathcal{Y}^{\mathrm{Ig}}(\epsilon_0/p^m) \subset \mathbf{U}'.
\end{equation}

\item $\mathbf{U}' \cdot g \subset \mathbf{U}'$. 
\end{enumerate}
\end{claim}
Then by induction, (2) implies 
\begin{equation}\label{U'g-r}\mathbf{U'} \cdot g^{r} \subset \mathbf{U}'
\end{equation}
for all $r \in \mathbb{Z}_{\ge 0}$. Admitting Claim \ref{U'Uclaim}, we thus have 
$$U_{-m+r} \overset{(\ref{Um})}{=} U_{-m} \cdot g^{r} \overset{(\ref{Uminside})}{\subset} \mathbf{U}' \cdot g^{r} \overset{(\ref{U'g-r})}{\subset} \mathbf{U'}$$
for all $r \in \mathbb{Z}_{\ge 0}$. This implies
$$\mathcal{V}_x \overset{(\ref{Vxunion'})}{=} \bigcup_{k \ge -m}U_k =\bigcup_{r \in \mathbb{Z}_{\ge 0}}U_{-m+r} \subset \mathbf{U}' \overset{(\ref{fullU'})}{\subset} \mathcal{V}_x,$$
which then gives (\ref{U'U}). 

We now prove Claim \ref{U'Uclaim}. 

\begin{proof}[Proof of Claim \ref{U'Uclaim}]\textbf{(1)}: Pick an arbitrary $M \in \mathbb{Z}_{\ge 0}$. Then pick $w \in \omega(U_M)$ as in Choice \ref{'choicew} and let 
\begin{equation}\label{yinclusion'}y_{\mathrm{dR}}(w) \in \left(\mathrm{Fil}^1\mathcal{O}\mathbb{B}_{\mathrm{dR},U_M}(U_M)\right)\cdot t^{-1} \overset{(\ref{changeofvars})}{\rightarrow} \mathbb{B}_{\mathrm{dR},U_M}^+(U_M') + (q_{\mathrm{dR}}-1)\cdot \mathbb{B}_{\mathrm{dR},U_M}(U_M')\llbracket q_{\mathrm{dR}}-1\rrbracket
\end{equation}
be as in Definition \ref{'xydefinition}. Again by (\ref{changeofvars}), we have 
\begin{equation}\label{thetaXq}\theta(\theta_X(y_{\mathrm{dR}}(w))) = \theta(\theta_q(y_{\mathrm{dR}}(w))).
\end{equation}
The assertion follows if we can show the existence of $m \in \mathbb{Z}_{\ge 0}$ such that $U_{-m} \subset U_M'$, i.e. the existence of $m \in \mathbb{Z}_{\ge 0}$ such that $\theta(\theta_q(y_{\mathrm{dR}}(w))) \neq 0$ on $U_{-m}$. 

Recall the generator $w_{\mathrm{can}}^{\mathrm{Katz}} \in \omega(\mathcal{Y}^{\mathrm{Ig}})$ from (\ref{'Katzgenerator}). Then since 
$$\mathcal{Y}^{\mathrm{Ig}} \subset \mathcal{Y}^{\mathrm{Ig}}(\epsilon_0) = U \overset{(\ref{Uinclusion})}{\subset} U_M,$$
we have 
$$y_{\mathrm{dR}}(w)|_{\mathcal{Y}^{\mathrm{Ig}}} = y_{\mathrm{dR}}(w|_{\mathcal{Y}^{\mathrm{Ig}}}) \overset{(\ref{'y1})}{=} \frac{y_{\mathrm{dR}}(w|_{\mathcal{Y}^{\mathrm{Ig}}})}{y_{\mathrm{dR}}(w_{\mathrm{can}}^{\mathrm{Katz}})} = \frac{w|_{\mathcal{Y}^{\mathrm{Ig}}}}{w_{\mathrm{can}}} \in \mathcal{O}_{\mathcal{Y}^{\mathrm{Ig}}}(\mathcal{Y}^{\mathrm{Ig}}).$$
Here in the third equality, we use the fact that the assignment $w \mapsto y_{\mathrm{dR}}(w)$ is $\mathcal{O}$-linear since (\ref{comparisonmap}) is $\mathcal{O}$-linear (see Definition \ref{'xydefinition}) and the fact that $w|_{\mathcal{Y}^{\mathrm{Ig}}}$ is an $\mathcal{O}_{\mathcal{Y}^{\mathrm{Ig}}}(\mathcal{Y}^{\mathrm{Ig}})$-multiple of the generator $w_{\mathrm{can}}^{\mathrm{Katz}} \in \omega(\mathcal{Y}^{\mathrm{Ig}})$. In particular,
$$\theta(\theta_X(y_{\mathrm{dR}}(w|_{\mathcal{Y}^{\mathrm{Ig}}}))) = y_{\mathrm{dR}}(w)|_{\mathcal{Y}^{\mathrm{Ig}}} \in \mathcal{O}_{\mathcal{Y}^{\mathrm{Ig}}}(\mathcal{Y}^{\mathrm{Ig}}).$$
Let $C_1$ be in the valuation group of $k$ such that (recall Definition \ref{kdefinition}) such that the supremum norm $|\cdot |_{\mathrm{sup},\mathcal{Y}^{\mathrm{Ig}}}$ on $\mathcal{Y}^{\mathrm{Ig}}$ satisfies 
\begin{equation}\label{YIggoodbound}\left|\theta(\theta_X(y_{\mathrm{dR}}(w|_{\mathcal{Y}^{\mathrm{Ig}})}))\right|_{\mathrm{sup},\mathcal{Y}^{\mathrm{Ig}}} \le p^{C_1}.
\end{equation}

Consider the rational subset 
$$U_y := \{|\theta(\theta_q(y_{\mathrm{dR}}(w)))| \le p^{C_1}\} \subset U_M.$$
We have 
$$\mathcal{Y}^{\mathrm{Ig}} \overset{(\ref{YIggoodbound})}{\subset} \{\theta(\theta_X(y_{\mathrm{dR}}(w))) = 1\} \overset{(\ref{thetaXq})}{=} \{\theta(\theta_q(y_{\mathrm{dR}}(w))) = 1\} \subset U_y.$$
Thus, by Corollary \ref{rationalintersectioncorollary} applied to the rational subset $U_y$, we have $\mathcal{Y}^{\mathrm{Ig}}(\epsilon_0/p^m) \subset U_y$ for some $m \ge 0$ in the valuation group of $\mathcal{O}_k$. Then $\theta(\theta_q(y_{\mathrm{dR}}(w))) \neq 0$ on $\mathcal{Y}^{\mathrm{Ig}}(\epsilon_0/p^m) \overset{(\ref{gisomorphism})}{=} \mathcal{Y}^{\mathrm{Ig}}(\epsilon_0) \cdot g^{-m} \overset{(\ref{Um})}{=} U_{-m}$, which gives Claim \ref{U'Uclaim} (1).\\ %Choose any $m_0 \in \mathbb{Z}_{\ge 0}$ such that $p^{1-m_0}/(p+1) < \epsilon_0'$. Now there exists some $m \in \mathbb{Z}_{\ge 0}$ such that 
%$$|z_{\mathrm{HT}}(y'\cdot g^{-m})| \overset{(\ref{zHTtransformationproperty})}{=} |z_{\mathrm{HT}}(y')/p^m| = |z_{\mathrm{HT}}(y')|p^m > p^{p/(p^2-1) + m_0}.$$
%By (\ref{zUcanm}), this implies $y'\cdot g^{-n} \in \mathcal{Y}^{\mathrm{Ig}}(\epsilon_0') \subset U_y$ and so $\theta(\theta_q(z_{\mathrm{dR}}))(y'\cdot g^{-n}) \neq \infty$. 

%Thus it suffices to show that $U_m \subset \mathbf{U}'$ for every $m$. Given $m$, let $w \in \omega(U_m)$ be as in Choice \ref{'choicew}, and let $y_{\mathrm{dR}}(w)$ as in Definition \ref{'xydefinition}.

%We need to show that $\theta(\theta_q(y_{\mathrm{dR}})) \neq 0$ on $U_m$. Assume for the sake of contradiction that there exists a point $y' \in U_m(R,R^+)$ such that $\theta(\theta_q(y_{\mathrm{dR}}))(y') = 0$. Recall 
%$$g = \left(\begin{array}{ccc} 1 & 0\\
%0 & p\\
%\end{array}\right) \in GL_2(\mathbb{Q}_p)$$
%from (\ref{gdefinition}). 

\textbf{(2)}: %Recall that for any $M \in \mathbb{Z}_{\ge 0}$, we have 
%$$\mathbf{U'} \overset{(\ref{fullU'2})}{=} \bigcup_{m \in \mathbb{Z}_{\ge M}}\bigcup_{w \in \omega(U_m)}U_{m,w}'.$$
We will show that 
\begin{equation}\label{showthat}U_{m,w}' \cdot g \subset U_{m+1,w}'
\end{equation}
for every $m \in \mathbb{Z}_{\ge 0}$ and every $w \in \omega(U_m)$. Admitting (\ref{showthat}), we then have
$$\mathbf{U}' \cdot g \overset{(\ref{fullU'})}{=} \bigcup_{m \in \mathbb{Z}_{\ge 0}}\bigcup_{w \in \omega(U_m)}U_{m,w}'\cdot g \overset{(\ref{showthat})}{\subset} \bigcup_{m \in \mathbb{Z}_{\ge 0}}\bigcup_{w \in \omega(U_m)}U_{m+1,w}' = \bigcup_{m \in \mathbb{Z}_{\ge 1}}\bigcup_{w \in \omega(U_m)}U_{m,w}' \overset{(\ref{fullU'2})}{=} \mathbf{U}'$$
which gives (2). 

We now show (\ref{showthat}). %Applying $g^{-1}$ to (\ref{showthat}), we get the equivalent statement
%\begin{equation}\label{showthat'}U_{m,w}'  \subset U_{m+1,w}'\cdot g^{-1}.
%\end{equation}
Let $m\in \mathbb{Z}_{\ge 0}$, $w \in \omega(U_{m+1})$ as in Choice \ref{'choicew} and $y_{\mathrm{dR}}(w)$ be as in Definition \ref{'xydefinition}. By (\ref{yinclusion'}), (\ref{thetaq}) and (\ref{thetat}) we have
\begin{equation}\label{needtheta1}\theta(\theta_q(y_{\mathrm{dR}}(w))) = \theta_t(\theta_q(y_{\mathrm{dR}}(w))).
\end{equation}
Moreover, we have $U_{m+1}\cdot g^{-1} \overset{(\ref{Um})}{=} U_m$ and so
$$(g^{-1})^*y_{\mathrm{dR}}(w) \in \left(\mathrm{Fil}^1\mathcal{O}\mathbb{B}_{\mathrm{dR},U_m}(U_m)\right)\cdot t^{-1} \overset{(\ref{changeofvars})}{\rightarrow} \mathbb{B}_{\mathrm{dR},U_m}^+(U_{m,w}') + (q_{\mathrm{dR}}-1)\cdot \mathbb{B}_{\mathrm{dR},U_m}(U_{m,w}')\llbracket q_{\mathrm{dR}}-1\rrbracket.$$
Thus
\begin{equation}\label{needtheta2}\theta(\theta_q((g^{-1})^*y_{\mathrm{dR}}(w))) = \theta_t(\theta_q((g^{-1})^*y_{\mathrm{dR}}(w))).
\end{equation}
%and since $\theta\circ \theta_q$ and $\theta_t \circ \theta_q$ are both $\hat{\mathcal{O}}_{\mathcal{V}_x}(U_m)$-linear, we have
%\begin{equation}\label{needtheta2y'}\theta(\theta_q((g^{-n})^*y_{\mathrm{dR}}(w)(y'))) = \theta(\theta_q((g^{-n})^*y_{\mathrm{dR}}(w)))(y') \overset{(\ref{needtheta2})}{=} \theta_t(\theta_q((g^{-n})^*y_{\mathrm{dR}}(w)))(y') = \theta_t(\theta_q((g^{-n})^*y_{\mathrm{dR}}(w)(y'))).
%\end{equation}
By (\ref{Um}) we have $U_{m+1} \cdot g^{-1} = U_m \overset{(\ref{Uminclusion})}{\subset} U_{m+1}$, and so the assumptions of Lemma \ref{thetaXGamma0commutelemma} are satisfied for $\gamma = g^{-1}$ and $U_{m+1}$. Thus 
\begin{equation}\label{gammathetacommute'}(g^{-1})^*\theta_t(\theta_q(y_{\mathrm{dR}}(w)))\overset{(\ref{gammathetacommute})}{=} \theta_t(\theta_q((g^{-1})^*y_{\mathrm{dR}}(w))).
\end{equation}
Therefore
\begin{equation}\label{g-1calc}\begin{split}(g^{-1})^*\theta(\theta_q(y_{\mathrm{dR}}(w))) \overset{(\ref{needtheta1})}{=} (g^{-1})^*\theta_t(\theta_q(y_{\mathrm{dR}}(w)))&\overset{(\ref{gammathetacommute'})}{=} \theta_t(\theta_q((g^{-1})^*y_{\mathrm{dR}}(w))) 
\overset{(\ref{intermediatetransformation})}{=} \theta_t(\theta_q(f \cdot y_{\mathrm{dR}}(w))) \\
& \overset{(\ref{needtheta1})}{=}  \theta(\theta_q(f \cdot y_{\mathrm{dR}}(w))) = \theta(\theta_q(f)) \cdot \theta(\theta_q(y_{\mathrm{dR}}(w))) 
\end{split}
\end{equation}
for some $f \in \mathcal{O}\mathbb{B}_{\mathrm{dR},U_{m+1}}(U_{m+1})$ (in particular implying that $\theta(\theta_q(f)) \in \hat{\mathcal{O}}_{U_{m+1}}(U_{m+1})$). %By (\ref{Uidentify}), we have $\theta(\theta_{X_{i,m+1}}(f)) \in \hat{\mathcal{O}}_{U_{m+1}}(U_{i,m+1})^{\times}$ for all $1 \le i \le \mathbf{n}$, and hence 
%$$\theta(\theta_q(f)) \overset{(\ref{changeofvars})}{=} \theta(\theta_{X_m}(f)) \in \hat{\mathcal{O}}_{U_{m+1}}(U_{m+1})^{\times}.$$ 
Restricting (\ref{g-1calc}) to $U_m \overset{(\ref{Uminclusion})}{\subset} U_{m+1}$ and applying $g^*$ to the previous equality, we get the following equality of elements of $\hat{\mathcal{O}}_{U_m}(U_m)$
\begin{equation}\label{gcalc}\theta(\theta_q(y_{\mathrm{dR}}(w)))|_{U_m} = g^*\theta(\theta_q(f)) \cdot g^*\theta(\theta_q(y_{\mathrm{dR}}(w))).
\end{equation}
Therefore (\ref{gcalc}) gives the implication
$$\theta(\theta_q(y_{\mathrm{dR}}(w)))|_{U_m} \in \hat{\mathcal{O}}_{U_m}(U_m)^{\times} \implies g^*\theta(\theta_q(y_{\mathrm{dR}}(w))) \in \hat{\mathcal{O}}_{U_m}(U_m)^{\times}.$$
This implies $U_{m,w}'\cdot g \subset U_{m+1,w}'$, which is (\ref{showthat}). 

\end{proof}

This finishes the proof of Theorem \ref{U'Utheorem}.

\end{proof}

\begin{corollary}\label{wcangeneratorproposition3}We have 
\begin{equation}\label{zqw2}z_{\mathrm{dR}},\mathbf{z}, q_{\mathrm{dR}} \in \mathbb{B}_{\mathrm{dR},\mathcal{V}_x}^+(\mathcal{V}_x)\llbracket q_{\mathrm{dR}}-1\rrbracket, \hspace{1cm} w_{\mathrm{can}} \in \omega \otimes_{\mathcal{O}_Y} \mathbb{B}_{\mathrm{dR},\mathcal{V}_x}^+(\mathcal{V}_x)\llbracket q_{\mathrm{dR}}-1\rrbracket
\end{equation}
where $w_{\mathrm{can}}$ is a $\mathbb{B}_{\mathrm{dR},\mathcal{V}_x}^+(\mathcal{V}_x)\llbracket q_{\mathrm{dR}}-1\rrbracket$ generator. Moreover,
\begin{equation}\label{finalqzrelation}q_{\mathrm{dR}} = \mathrm{exp}(\mathbf{z}) = \mathrm{exp}(t\cdot z_{\mathrm{dR}}).
\end{equation}
\end{corollary}

\begin{proof}This follows immediately from Proposition \ref{wcangeneratorproposition2} (particularly (\ref{wcangenerator}) and (\ref{zzqU'})) and Theorem \ref{U'Utheorem}.
\end{proof}

\begin{corollary}
We have a natural embedding of rings
\begin{equation}\label{changeofvarsV}\mathcal{O}\mathbb{B}_{\mathrm{dR},\mathcal{V}_x}(\mathcal{V}_x) \hookrightarrow \mathbb{B}_{\mathrm{dR},\mathcal{V}_x}(\mathcal{V}_x)\llbracket q_{\mathrm{dR}}-1\rrbracket.
\end{equation}
compatible with connections. The restriction of (\ref{changeofvarsV}) to $\mathcal{O}\mathbb{B}_{\mathrm{dR},\mathcal{V}_x}^+(\mathcal{V}_x) \subset \mathcal{O}\mathbb{B}_{\mathrm{dR},\mathcal{V}_x}(\mathcal{V}_x)$ factors through
\begin{equation}\label{changeofvarsV2}\mathcal{O}\mathbb{B}_{\mathrm{dR},\mathcal{V}_x}^+(\mathcal{V}_x) \hookrightarrow \mathbb{B}_{\mathrm{dR},\mathcal{V}_x}^+(\mathcal{V}_x)\llbracket q_{\mathrm{dR}}-1\rrbracket.
\end{equation}
Moreover,
$$\nabla \mathbf{z} \in \Omega_Y \otimes_{\mathcal{O}_Y}\mathbb{B}_{\mathrm{dR},\mathcal{V}_x}(\mathcal{V}_x)\llbracket q_{\mathrm{dR}}-1\rrbracket$$ 
is a $\mathbb{B}_{\mathrm{dR},\mathcal{V}_x}(\mathcal{V}_x)\llbracket q_{\mathrm{dR}}-1\rrbracket$-generator.
\end{corollary}

\begin{proof}The embedding (\ref{changeofvarsV}) is constructed by gluing (\ref{changeofvars}) over all $m \in \mathbb{Z}_{\ge 0}$ and using (\ref{Vxunion'}), (\ref{fullU'}) and (\ref{U'U}). The induced map (\ref{changeofvarsV2}) is constructed from the gluing together the compositions
$$\mathcal{O}\mathbb{B}_{\mathrm{dR},U_m'}^+(U_m') = \mathbb{B}_{\mathrm{dR},U_m'}^+(U_m')\llbracket X\rrbracket \subset \mathbb{B}_{\mathrm{dR},U_m'}^+(U_m')\llbracket X_m/t \rrbracket \overset{(\ref{changeofvars})}{=} \mathbb{B}_{\mathrm{dR},U_m'}^+(U_m')\llbracket q_{\mathrm{dR}}-1\rrbracket.$$

The generator statement follows from gluing together the generator statement of Corollary \ref{nablazgeneratorcorollary} over all $m \in \mathbb{Z}_{\ge 0}$.

\end{proof}

Although we will not need it in the remainder of our discussion, we also note the following. 

\begin{corollary}We have 
$$z_{\mathrm{dR}}|_{\mathcal{Y}^{\mathrm{Ig}}} = -x_{\mathrm{dR}}(w_{\mathrm{can}}^{\mathrm{Katz}}) \in \mathrm{Fil}^1\mathcal{O}\mathbb{B}_{\mathrm{dR},\mathcal{Y}^{\mathrm{Ig}}}^+(\mathcal{Y}^{\mathrm{Ig}}) \cdot t^{-1}.$$
Hence
$$\mathbf{z}|_{\mathcal{Y}^{\mathrm{Ig}}} \overset{(\ref{zbfdefinition})}{=} t\cdot z_{\mathrm{dR}}|_{\mathcal{Y}^{\mathrm{Ig}}} \in \mathrm{ker}(\theta_{\mathrm{dR}})|_{\mathcal{Y}^{\mathrm{Ig}}}(\mathcal{Y}^{\mathrm{Ig}}) = (t,X) \cdot \mathcal{O}\mathbb{B}_{\mathrm{dR},\mathcal{Y}^{\mathrm{Ig}}}^+(\mathcal{Y}^{\mathrm{Ig}}).$$
\end{corollary}

\begin{proof}This follows from 
\begin{align*}z_{\mathrm{dR}}|_{\mathcal{Y}^{\mathrm{Ig}}} \overset{(\ref{zqw2}), \;\text{Definition \ref{'zqwdefinition}}}{=} -\frac{x_{\mathrm{dR}}(w)}{y_{\mathrm{dR}}(w)} \overset{(\ref{zqw2})}{=} -\frac{x_{\mathrm{dR}}(w_{\mathrm{can}}^{\mathrm{Katz}})}{y_{\mathrm{dR}}(w_{\mathrm{can}}^{\mathrm{Katz}})} &\overset{(\ref{'y1})}{=} -x_{\mathrm{dR}}(w_{\mathrm{can}}^{\mathrm{Katz}})|_{\mathcal{Y}^{\mathrm{Ig}}} \\
&\overset{\text{Definition \ref{'xydefinition}}}{\in}  \mathrm{Fil}^1\mathcal{O}\mathbb{B}_{\mathrm{dR},\mathcal{Y}^{\mathrm{Ig}}}^+(\mathcal{Y}^{\mathrm{Ig}}) \cdot t^{-1}.
\end{align*}

\end{proof}

We will henceforth make ubiquitous use of the embeddings (\ref{changeofvarsV}) and (\ref{changeofvarsV2}). We will need the following ``evaluation'' maps for the remainder of our discussion. 

\begin{definition}\label{evaluationdefinition}
Recall the $\mathbb{B}_{\mathrm{dR},\mathcal{V}_x}(\mathcal{V}_x)$-algebra homomorphism 
$$\theta_q : \mathbb{B}_{\mathrm{dR},\mathcal{V}_x}(\mathcal{V}_x)\llbracket q_{\mathrm{dR}}-1\rrbracket \rightarrow \mathbb{B}_{\mathrm{dR},\mathcal{V}_x}(\mathcal{V}_x)$$
from (\ref{thetaq}), and the $\hat{\mathcal{O}}_{\mathcal{V}_x}(\mathcal{V}_x)$-module homomorphism 
$$\theta_t : \mathbb{B}_{\mathrm{dR},\mathcal{V}_x}(\mathcal{V}_x) \rightarrow \hat{\mathcal{O}}_{\mathcal{V}_x}(\mathcal{V}_x)$$
from (\ref{thetat}). Composing $\theta_q$ with $\theta_t$, we get a $\hat{\mathcal{O}}_{\mathcal{V}_x}(\mathcal{V}_x)$-module homomorphism
$$\theta_t \circ \theta_q : \mathbb{B}_{\mathrm{dR},\mathcal{V}_x}(\mathcal{V}_x)\llbracket q_{\mathrm{dR}}-1\rrbracket \rightarrow \hat{\mathcal{O}}_{\mathcal{V}_x}(\mathcal{V}_x).$$
\end{definition}

\subsection{Generalized $p$-adic modular forms recover Katz $p$-adic modular forms}

Letting $w_{\mathrm{can}}$ be from (\ref{zqw2}), by (\ref{'y1}) we have
\begin{equation}\label{wcanspecialize}w_{\mathrm{can}}|_{\mathcal{Y}^{\mathrm{Ig}}} = w_{\mathrm{can}}^{\mathrm{Katz}}.\end{equation}
%Since $w_{\mathrm{can}}^{\mathrm{Katz}}$ is a generator, Assumption \ref{'wgeneratorassumption} is satisfied.

Recall the notations of Convention \ref{Yconvention}, in particular the definition of $Y(\epsilon)$ and the fact that $Y(0) = Y^{\mathrm{ord}}$. The next result says that if $w'$ is the pullback along $W \rightarrow W_0$ of some element $w_0' \in \omega^{\otimes k}(W_0)$ where $W_0 \subset Y(\epsilon_0)$ is an open subset with $Y(0) \subset W_0$, then $\frac{w'}{w_{\mathrm{can}}^{\otimes k}}|_{\mathcal{Y}^{\mathrm{Ig}}}$ is a Katz $p$-adic modular form of weight $k$ (see \cite{KatzIgusa}, \cite{Katzpamf}, and \cite[Section 1.3]{BDP} for an overview of Katz $p$-adic modular forms).

\begin{theorem}\label{Katzpadicmodularformtheorem}
Let $U = \mathcal{Y}^{\mathrm{Ig}}(\epsilon_0)$ be as in Definition \ref{Udefinition}. Let $W \subset U$ be an open subset such that $W$ is the pullback along the $\Gamma_{0,p}(p^{n(\epsilon_0)})$-cover $\mathcal{Y}^{\mathrm{Ig}}(\epsilon_0) \rightarrow Y(\epsilon_0)$ of some open subset $W_0 \subset Y(\epsilon_0)$ satisfying $Y(0) \subset W_0$; this in particular implies $\mathcal{Y}^{\mathrm{Ig}} \subset W$.
Suppose 
$$w' \in \omega^{\otimes k} \otimes_{\mathcal{O}_Y}\mathbb{B}_{\mathrm{dR},U}(W)\llbracket q_{\mathrm{dR}}-1\rrbracket$$
is the pullback of a section $w_0' \in \omega^{\otimes k}(W_0)$, and let
$$G := \frac{w'}{w_{\mathrm{can}}^{\otimes k}} \in \mathbb{B}_{\mathrm{dR},U}(W)\llbracket q_{\mathrm{dR}}-1\rrbracket.$$

%$$G|_{\mathcal{Y}^{\mathrm{Ig}} \times_U W} \in \mathbb{B}_{\mathrm{dR},U}(\mathcal{Y}^{\mathrm{Ig}} \times_U W)\llbracket q_{\mathrm{dR}}-1\rrbracket.$$
Then:
\begin{enumerate}
\item $$G|_{\mathcal{Y}^{\mathrm{Ig}}} \in \mathcal{O}_{\mathcal{Y}^{\mathrm{Ig}}}(\mathcal{Y}^{\mathrm{Ig}}) \subset \mathcal{O}\mathbb{B}_{\mathrm{dR},\mathcal{Y}^{\mathrm{Ig}}}^+(\mathcal{Y}^{\mathrm{Ig}}) \overset{(\ref{changeofvars})}{\subset} \mathbb{B}_{\mathrm{dR},\mathcal{Y}^{\mathrm{Ig}}}^+(\mathcal{Y}^{\mathrm{Ig}})\llbracket q_{\mathrm{dR}}-1\rrbracket.$$
\item For any
$$\gamma = \left(\begin{array}{ccc} a & b\\
0 & d\\
\end{array}\right) \in \Gamma_{0,p}(p^{\infty}) \overset{(\ref{BGalois})}{=} B \overset{(\ref{BGalois})}{=} \mathrm{Gal}(\mathcal{Y}^{\mathrm{Ig}}/Y^{\mathrm{ord}})$$
we have
\begin{equation}\label{dequivariance}\gamma^*G|_{\mathcal{Y}^{\mathrm{Ig}}} = d^{-k}\cdot G|_{\mathcal{Y}^{\mathrm{Ig}}}.
\end{equation}
In other words, $G|_{\mathcal{Y}^{\mathrm{Ig}}}$ is a Katz $p$-adic modular form of weight $k$. 
\end{enumerate}
\end{theorem}

\begin{proof}(1) follows immediately from (\ref{wcanspecialize}) and the fact that $w_{\mathrm{can}}^{\mathrm{Katz}} \in \omega(\mathcal{Y}^{\mathrm{Ig}})$ is a generator (\cite[Main Theorem]{KatzST}). (2) follows from (\ref{Ftransformationidentity}) with $n = \infty$. 
\end{proof}

\begin{remark}Note that in the situation of Theorem \ref{Katzpadicmodularformtheorem} (2), (\ref{dequivariance}) implies that the action of $B = \mathrm{Gal}(\mathcal{Y}^{\mathrm{Ig}}/Y^{\mathrm{ord}})$ on $G|_{\mathcal{Y}^{\mathrm{Ig}}}$ factors through 
$$\mathbb{Z}_p^{\times} \cong B/B^1 \overset{(\ref{YIgGaloisgroup})}{=} \mathrm{Gal}(Y^{\mathrm{Ig}}/Y^{\mathrm{ord}}).$$ Thus $G|_{\mathcal{Y}^{\mathrm{Ig}}}$ descends to an element of $G \in \mathcal{O}_{Y^{\mathrm{Ig}}}(Y^{\mathrm{Ig}})$. This conforms with the viewpoint of \cite{KatzIgusa}, which views Katz $p$-adic modular forms as functions on $Y^{\mathrm{Ig}}$ transforming under the action of $\mathrm{Gal}(Y^{\mathrm{Ig}}/Y^{\mathrm{ord}}) \cong \mathbb{Z}_p^{\times}$ via some character on $\mathbb{Z}_p^{\times}$. 
\end{remark}

\subsection{$p$-integrality of generalized $p$-adic modular forms}

In this section, we will show that under certain assumptions, the constant term of the $q_{\mathrm{dR}}$-expansion (Definition \ref{'zqexpansions}) $w'(q_{\mathrm{dR}})$ of a weight $k \in \mathbb{Z}_{\ge 0}$ modular form 
$$w' \in \omega^{\otimes k}(Y)$$
(where $Y$ continues to be as in Convention \ref{Yconvention}) is $p$-integral (in the terminology of Definition \ref{pintegralitydefinition}), i.e. belongs to $\hat{\mathcal{O}}_{\mathcal{V}_x}^+(\mathcal{Y}^{\mathrm{Ig}}(\epsilon_0/p^{\alpha}))$ for some $\alpha \in \mathbb{Z}_{\ge 0}$. 

Let us summarize our main results here before giving the detailed results below: Assume without loss of generality (i.e. after possibly rescaling) that 
$$w'|_{\mathcal{Y}^{\mathrm{Ig}}} \in \omega^{\otimes k} \otimes_{\mathcal{O}_Y} \mathcal{O}_{\mathcal{Y}^{\mathrm{Ig}}}^+(\mathcal{Y}^{\mathrm{Ig}}),$$
i.e. $w'$ is ``normalized'' in the terminology of Section \ref{Introduction}. Viewing 
$$w' \in \omega^{\otimes k}(Y) \subset \omega^{\otimes k} \otimes_{\mathcal{O}_Y}\mathcal{O}\mathbb{B}_{\mathrm{dR},\mathcal{V}_x}^+(\mathcal{V}_x) \overset{(\ref{changeofvarsV})}{\subset} \omega^{\otimes k} \otimes_{\mathcal{O}_Y} \mathbb{B}_{\mathrm{dR},\mathcal{V}_x}(\mathcal{V}_x)\llbracket q_{\mathrm{dR}}-1\rrbracket,$$
let 
$$G = w'/w_{\mathrm{can}}^{\otimes k} \overset{(\ref{zqw2})}{\in} \mathbb{B}_{\mathrm{dR},\mathcal{V}_x}(\mathcal{V}_x)\llbracket q_{\mathrm{dR}}-1\rrbracket$$
be the generalized $p$-adic modular form attached to $w'$ (Definition \ref{generalizedpadicmodularformdefinition}). Recall the $\hat{\mathcal{O}}_{\mathcal{V}_x}(\mathcal{V}_x)$-module homomorphism
$$\theta_t \circ \theta_q : \mathbb{B}_{\mathrm{dR},\mathcal{V}_x}(\mathcal{V}_x)\llbracket q_{\mathrm{dR}}-1\rrbracket \rightarrow \hat{\mathcal{O}}_{\mathcal{V}_x}(\mathcal{V}_x)$$
from Definition \ref{evaluationdefinition}. We will show that 
$$\theta_t(\theta_q(G)) \in \hat{\mathcal{O}}_{\mathcal{Y}^{\mathrm{Ig}}}^+(\mathcal{Y}^{\mathrm{Ig}}).$$
%Moreover, after $p$-depleting $G$ to obtain $G^{\flat}$ (see (\ref{Gflatgeneraldefinition})), one can show that $\theta_t(\theta_q(d_k^jG^{\flat}))$ is also $p$-integral and varies $p$-adically continuously in the weight variable $j \in \mathbb{Z}/(p-1) \times \mathbb{Z}_p$ (see Theorem \ref{pintegraltheorem2}). This last fact will be the crux of the construction of our $p$-adic $L$-functions (Sections \ref{padicLfunctionsection} and \ref{padicLfunctionsection2}). The precise statement is in Theorem \ref{pintegraltheorem} below. 
 We remark that the proof of this Theorem will crucially use the identity (\ref{Ftransformationidentity}) in order to show that $w'(q_{\mathrm{dR}})$ has constant term satisfying the assumptions of Theorem \ref{crystalline2}, which will ultimately give the $p$-integrality statements. %We first want to find an appropriate affinoid open subset $\mathcal{V} \subset \mathbf{U}' \overset{(\ref{U'U})}{=} \mathcal{V}_x$ as in Definition \ref{Vdefinition} in order to apply Theorem \ref{crystalline2}.

\begin{theorem}\label{pintegraltheorem}Assume that $0 \le \epsilon_0 < p/(p+1)$ as in Definition \ref{Udefinition} satisfies $n(\epsilon_0) = 1$ (see (\ref{nepsilondefinition})), i.e. $1/(p+1) \le \epsilon_0 < p/(p+1)$. Suppose $k \ge 0$ is an integer, that we are given 
$$w' \in \omega^{\otimes k} \otimes_{\mathcal{O}_Y}\mathbb{B}_{\mathrm{dR},\mathcal{V}_x}(\mathcal{V}_x)\llbracket q_{\mathrm{dR}}-1\rrbracket,$$
and suppose 
$$G := w'/w_{\mathrm{can}}^{\otimes k} \in \mathbb{B}_{\mathrm{dR},\mathcal{V}_x}(\mathcal{V}_x)\llbracket q_{\mathrm{dR}}-1\rrbracket$$
satisfies the identity (\ref{Ftransformationidentity}) for all $\gamma \in \Gamma_{1,p}(p^{\epsilon_0/p^{\alpha}})$ and $W = \mathcal{Y}^{\mathrm{Ig}}(\epsilon_0/p^{\alpha})$ for some $\alpha \in \mathbb{Z}_{\ge 0}$. Further suppose that
$$\theta_t(\theta_q(G))|_{\mathcal{Y}^{\mathrm{Ig}}} \in \omega_+^{\otimes k}(\mathcal{Y}^{\mathrm{Ig}}(0)^+)$$
where $\omega_+$ is as in (\ref{omegaY}). Then 
$$\theta_t(\theta_q(G))|_{\mathcal{Y}^{\mathrm{Ig}}(\epsilon_0/p^{\alpha})} \in \hat{\mathcal{O}}_{\mathcal{V}_x}^+(\mathcal{Y}^{\mathrm{Ig}}(\epsilon_0/p^{\alpha})).$$
\end{theorem}

\begin{remark}\label{assumptionsatisfiedremark}For example, by Theorem \ref{weighttheorem} the assumption that $G$ satisfy (\ref{Ftransformationidentity}) in Theorem \ref{pintegraltheorem} holds if $w'$ is the pullback along $\mathcal{Y}^{\mathrm{Ig}}(\epsilon_0/p^{\alpha}) \rightarrow Y_1(\alpha+1,\epsilon_0/p^{\alpha})$ of an element of $\omega^{\otimes k}(Y_1(\alpha+1,\epsilon_0/p^{\alpha}))$; note that under our assumption, $n(\epsilon_0/p^{\alpha}) = \alpha + 1$. The assumption is also satisfied for the $p$-adic Maass-Shimura derivatives $\partial_k^jw'$ of $w'$ as in the previous sentence by Corollary \ref{weightcorollary} below. This will be the case in our main applications of Theorem \ref{pintegraltheorem} in Sections \ref{padicLfunctionsection} and \ref{padicLfunctionsection2}. 
\end{remark}

\begin{proof}[Proof of Theorem \ref{pintegraltheorem}]We want to apply Theorem \ref{crystalline2}, whose hypotheses we must show are satisfied. Adopt the notation of the statement of that Theorem, and recall $\mathcal{V}_x = \{z_{\mathrm{HT}} \neq 0\}$ (see (\ref{Vz})). %Then Corollary \ref{YIgepsilonpmcorollary} shows that there exists $m \in \mathbb{Z}_{\ge 0}$ such that $\mathcal{Y}^{\mathrm{Ig}}(\epsilon_0/p^m) \subset \mathcal{V}_{\alpha}' \cap \mathcal{V}_x \subset \mathcal{V}_{\alpha}'$.
Let 
\begin{equation}\label{FG}F := \theta_t(\theta_q(G)).
\end{equation} Now we need to show the existence of 
$$f_0 \in \mathbf{\Gamma}(Y_1^+(\alpha+1,\epsilon_0/p^m)/\frak{m})$$
satisfying the necessary congruence 
$$\mathrm{res}_{\epsilon_0,\epsilon_0/p^m}(f) \equiv f_0 \pmod{\frak{m}\mathbf{\Gamma}(\hat{\mathcal{O}}_{\mathcal{Y}^{\mathrm{Ig}}(\epsilon_0/p^m)}^+)}$$
as in (\ref{keyassumption}).

For any
$$\gamma = \left(\begin{array}{ccc} a & b\\
c & d\\
\end{array}\right) \in \Gamma_{0,p}(p^{\alpha + 1}) \subset GL_2(\mathbb{Z}_p) \subset GL_2(\mathbb{Q}_p)$$
and $m = \alpha+1$ we have $U_m = \mathcal{Y}^{\mathrm{Ig}}(\epsilon_0/p^{\alpha})$ and (\ref{Galoisepsilonalpha}) implies $U_m \cdot \gamma \subset U_m$. Thus the assumptions of Lemma \ref{thetaXGamma0commutelemma} are satisfied. 
Consider the setting of Lemma \ref{thetaXGamma0commutelemma} with $m = 0$ (so that $U_m = U = \mathcal{Y}^{\mathrm{Ig}}(\epsilon_0)$) and note that 
\begin{enumerate}
\item $G|_U \in \mathbb{B}_{\mathrm{dR},U}(U)\llbracket q_{\mathrm{dR}}-1\rrbracket$ and 
\item any $\gamma \in \Gamma_{0,p}(p^{\alpha+1}) \subset \Gamma_{0,p}(p) \overset{(\ref{Galoisepsilonalpha})}{=} \mathrm{Gal}(U/Y(\epsilon_0))$ satisfies $U\cdot \gamma = U$,
\end{enumerate}
so that the assumptions of Lemma \ref{thetaXGamma0commutelemma} are satisfied for $f = G|_U$ and in particular we can apply (\ref{gammathetacommute}) to $G|_U$. Moreover, since $n(\epsilon_0) = 1$ by our assumptions, the moduli interpretation of $\mathcal{Y}^{\mathrm{Ig}}(\epsilon_0/p^{\alpha})$ (see (\ref{U})) and the definition of $Y_1(\alpha+1,\epsilon_0/p^{\alpha})$ (Convention \ref{Y1convention}) shows
$$\Gamma_{1,p}(p^{\alpha+1}) =  \mathrm{Gal}(\mathcal{Y}^{\mathrm{Ig}}(\epsilon_0/p^{\alpha})/Y_1(\alpha+1,\epsilon_0/p^{\alpha}))$$
we can apply (\ref{Ftransformationidentity}) to $G|_{\mathcal{Y}^{\mathrm{Ig}}(\epsilon_0/p^{\alpha})}$ for and $\gamma \in \Gamma_{1,p}(p^{n(\epsilon_0/p^{\alpha})}) = \Gamma_{1,p}(p^{\alpha+1})$. 
Therefore, for any $\left(\begin{array}{ccc} a & b \\
c  & d \\
\end{array}\right)\in \Gamma_{1,p}(p^{\alpha+1})$ we have
\begin{equation}\label{weightidentity}\begin{split}\left(\begin{array}{ccc} a & b\\
c & d\\
\end{array}\right)^*F|_{\mathcal{Y}^{\mathrm{Ig}}(\epsilon_0/p^{\alpha})} = \left(\begin{array}{ccc} a & b\\
c & d\\
\end{array}\right)^*\theta_t(\theta_q(G|_{\mathcal{Y}^{\mathrm{Ig}}(\epsilon_0/p^{\alpha})})) &= \left(\begin{array}{ccc} a & b\\
c & d\\
\end{array}\right)^*\theta_t(\theta_q(G|_U))|_{\mathcal{Y}^{\mathrm{Ig}}(\epsilon_0/p^{\alpha})} \\
&\overset{(\ref{gammathetacommute})}{=} \theta_t\left(\theta_q\left(\left(\begin{array}{ccc} a & b\\
c & d\\
\end{array}\right)^*G|_U\right)\right)|_{\mathcal{Y}^{\mathrm{Ig}}(\epsilon_0/p^{\alpha})} \\
&= \theta_t\left(\theta_q\left(\left(\begin{array}{ccc} a & b\\
c & d\\
\end{array}\right)^*G|_{\mathcal{Y}^{\mathrm{Ig}}(\epsilon_0/p^{\alpha})}\right)\right) \\
&\overset{(\ref{Ftransformationidentity})}{=} \theta_t\left(\theta_q\left(\frac{cz_{\mathrm{dR}} + a}{ad-bc}\right)^k\cdot G|_{\mathcal{Y}^{\mathrm{Ig}}(\epsilon_0/p^{\alpha})}\right) \\
&\overset{(\ref{FG})}{=} \left(\frac{c\theta_t(\theta_q(z_{\mathrm{dR}})) + a}{ad-bc}\right)^k \cdot F|_{\mathcal{Y}^{\mathrm{Ig}}(\epsilon_0/p^{\alpha})}\\ &\overset{(\ref{thetaqzzero})}{=} \left(\frac{a}{ad-bc}\right)^k \cdot F|_{\mathcal{Y}^{\mathrm{Ig}}(\epsilon_0/p^{\alpha})}.
\end{split}
\end{equation}
%Recall that there exists $m \in \mathbb{Z}_{\ge 0}$ such that $\mathcal{Y}^{\mathrm{Ig}}(\epsilon_0/p^m) \overset{(\ref{mathcalYIgepsilonpmVinclusion})}{\subset} \mathcal{V}_{\alpha}'$. 
 Using Lemma \ref{findepsilonmlemma} and increasing $m$ if necessary, we may assume that 
 $$\mathrm{res}_{\epsilon_0,\epsilon_0/p^m}(F|_{\mathcal{Y}^{\mathrm{Ig}}(\epsilon_0)}) \in \mathbf{\Gamma}(\hat{\mathcal{O}}_{\mathcal{Y}^{\mathrm{Ig}}(\epsilon_0/p^m)}^+).$$
 %Note that by Definition \ref{V'V'0definition} we have $|\theta(\theta_q(z_{\mathrm{dR}}))| \le p^{\alpha+1/2}$ on $\mathcal{V}_{\alpha}'$, and also $c \equiv 0 \pmod{p^{\alpha + 1}}$, so in all we have that $|c\theta(\theta_q(z_{\mathrm{dR}}))| < 1$ on $\mathcal{Y}^{\mathrm{Ig}}(\epsilon_0/p^m) \subset \mathcal{V}_{\alpha}'$ and so $|c\theta(\theta_q(z_{\mathrm{dR}})) + a| \le 1$ on $\mathcal{Y}^{\mathrm{Ig}}(\epsilon_0/p^m)$. 
 Thus (\ref{weightidentity}) implies that 
$$f := F|_{\mathcal{Y}^{\mathrm{Ig}}(\epsilon_0/p^{\alpha}) \cap Y_{\infty}(\epsilon_0/p^m)} \in \mathbf{\Gamma}(\hat{\mathcal{O}}_{\mathcal{Y}^{\mathrm{Ig}}(\epsilon_0/p^{\alpha}) \cap Y_{\infty}(\epsilon_0/p^m)}^+).$$
Moreover, (\ref{weightidentity}) implies that 
$$f_0' := f \pmod{\frak{m}\mathbf{\Gamma}(\hat{\mathcal{O}}_{\mathcal{Y}^{\mathrm{Ig}}(\epsilon_0/p^{\alpha}) \cap Y_{\infty}(\epsilon_0/p^m)}^+)}$$ satisfies $\gamma^*f_0' = f_0'$ for every $\gamma \in \Gamma_{1,p}(p^{\alpha+1})$. Thus by \'{e}tale descent of 
$$\mathbf{\Gamma}(\hat{\mathcal{O}}_{\mathcal{Y}^{\mathrm{Ig}}(\epsilon_0/p^{\alpha}) \cap Y_{\infty}(\epsilon_0/p^m)}^+)\otimes_{\mathcal{O}_k}\mathcal{O}_k/\frak{m}$$
along the $\Gamma_{1,p}(p^{\alpha+1})$-Galois pro-\'{e}tale cover 
$$\left(\mathcal{Y}^{\mathrm{Ig}}(\epsilon_0/p^{\alpha}) \cap Y_{\infty}(\epsilon_0/p^m) \right)\rightarrow Y_1(\alpha+1,\epsilon_0/p^m),$$
we see that $f_0'$ descends to an element
$$f_0 \in \mathbf{\Gamma}(\hat{\mathcal{O}}_{Y_1(\alpha+1,\epsilon_0/p^m)}^+) \otimes_{\mathcal{O}_k}\mathcal{O}_k/\frak{m} = \mathbf{\Gamma}(\mathcal{O}_{Y_1^+(\alpha+1,\epsilon_0/p^m)}/\frak{m})$$
where the equality follows from \cite[Theorem 7.4.1]{deJong}. Now all the assumptions of Theorem \ref{crystalline2} are satisfied. Applying it, we get $F|_{\mathcal{Y}^{\mathrm{Ig}}(\epsilon_0/p^{\alpha})} \in \hat{\mathcal{O}}_{\mathcal{V}_x}^+(\mathcal{Y}^{\mathrm{Ig}}(\epsilon_0/p^{\alpha}))$.

%Write 
%$$w'(q_{\mathrm{dR}}) = \sum_{n = 0}^{\infty}a_n(t) \cdot (q_{\mathrm{dR}}-1)^n, \hspace{.5cm} a_n(t) \in \hat{\mathcal{O}}_{\mathcal{V}_x}(\mathcal{V}_x)(\!(t)\!).$$
%For each $n \in \mathbb{Z}_{\ge 0}$, write 
%$$a_n(t) = \sum_{i = -\infty}^{\infty}a_{n,i} \cdot t^i \in  \hat{\mathcal{O}}_{\mathcal{V}_x}(\mathcal{V}_x)(\!(t)\!), \hspace{.5cm} a_{n,i} = 0, \hspace{.25cm} \forall i \ll 0$$
%By Theorem \ref{STanalyticcontinuationtheorem}, we see that
%$$a_{n,i}|_{\mathcal{Y}^{\mathrm{Ig}}} \in \hat{\mathcal{O}}_{\mathcal{Y}^{\mathrm{Ig}}(0)^+}(\mathcal{Y}^{\mathrm{Ig}}(0)^+).$$
%In fact, if $i \neq 0$, we have the stronger identity $a_{n,i}|_{\mathcal{Y}^{\mathrm{Ig}}} = 0$. Applying Theorem \ref{crystalline2} to each $a_{n,i} \in \hat{\mathcal{O}}_{\mathcal{V}_x}(\mathcal{V}_x)$, we conclude that 
%$$a_{n,i} \overset{(\ref{keyformalinclusion})}{\in} \mathcal{O}_{\hat{\mathcal{Y}}^{\mathrm{Ig}}(\epsilon_0)^+}(\hat{\mathcal{Y}}^{\mathrm{Ig}}(\epsilon_0)^+).$$
%This gives the Theorem. 
\end{proof}

Let us end this section by summarizing how we will use these results later in our discussion: We will apply Theorem \ref{pintegraltheorem} to a family of $p$-adic modular forms $d_k^jw'$ of weights $k+ 2j$ for $j$ varying $p$-adically. Here, the $d_k^j$ are the $p$-adic Maass-Shimura derivatives defined in Section \ref{operatorsection} (Definition \ref{padicMSdefinition}) which specialize to the classical Atkin-Serre operators $\theta_{\mathrm{AS}}^j$ (see \cite[Section 3.4]{Brooks}) on $\mathcal{Y}^{\mathrm{Ig}} \subset \mathcal{Y}^{\mathrm{Ig}}(\epsilon_0/p^{\alpha})$ (see Remark \ref{recoverASremark}). The Theorem shows that 
$$\theta_t(\theta_q(d_k^jw')) \in \hat{\mathcal{O}}_{\mathcal{V}_x}(\mathcal{Y}^{\mathrm{Ig}}(\epsilon_0/p^{\alpha}))$$
are $p$-integral, i.e 
$$\theta_t(\theta_q(d_k^jw')) \in \hat{\mathcal{O}}_{\mathcal{V}_x}^+(\mathcal{Y}^{\mathrm{Ig}}(\epsilon_0/p^{\alpha})),$$
for all $j \in \mathbb{Z}_{\ge 0}$. The classical theory of the ordinary $p$-adic Maass-Shimura derivatives $\theta_{\mathrm{AS}}^jw'$ shows that the restriction to $\mathcal{Y}^{\mathrm{Ig}} = \mathcal{Y}^{\mathrm{Ig}}(0) \subset \mathcal{Y}^{\mathrm{Ig}}(\epsilon_0/p^{\alpha})$ of the above family 
$$\theta_t(\theta_q(d_k^jw'))|_{\mathcal{Y}^{\mathrm{Ig}}} \in \hat{\mathcal{O}}_{\mathcal{V}_x}^+(\mathcal{Y}^{\mathrm{Ig}})$$ 
has good $p$-adic continuity properties in $j$. Using a rigidity argument, we will show that these good continuity properties in $j$ extend to the family $\theta_t(\theta_q(d_k^jw'))$ itself. This will be the basis of the construction of our $p$-adic $L$-functions. See Theorem \ref{pintegraltheorem2} for a precise statement. 

\subsection{$p$-adic Maass-Shimura operators}\label{operatorsection}
For certain purposes, it will be necessary to apply ``weight raising'' operators to generalized $p$-adic modular forms in order to obtain generalized $p$-adic modular forms of higher weights. Varying families of generalized $p$-adic modular forms in weights will ultimately be the cornerstone of our construction of $p$-adic $L$-functions.

Recall the Hodge-Tate filtration (\cite[Proposition 2.2.5]{CaraianiScholze})
$$\omega^{\otimes -1} \otimes_{\mathcal{O}_Y}\hat{\mathcal{O}}_Y \hookrightarrow H_{\text{\'{e}t}}^1(\mathcal{E}) \otimes_{\hat{\mathbb{Z}}_{p,Y}}\hat{\mathcal{O}}_Y.$$
Pulling back to $Y_{\infty} \rightarrow Y$, and recalling the $\Gamma(p^{\infty})$-level structure $(e_1,e_2)$ induces
$$(e_1 \cdot t^{-1},e_2 \cdot t^{-1}) : \left(\hat{\mathbb{Z}}_{p,Y_{\infty}} \cdot t^{-1}\right)^{\oplus 2} \xrightarrow{\sim} H_{\text{\'{e}t}}^1(\mathcal{E}),$$
we get
$$i_{\mathrm{HT}} : \omega^{\otimes -1} \otimes_{\mathcal{O}_Y}\hat{\mathcal{O}}_{Y_{\infty}} \hookrightarrow H_{\text{\'{e}t}}^1(\mathcal{E}) \otimes_{\hat{\mathbb{Z}}_{p,Y}}\hat{\mathcal{O}}_{Y_{\infty}} \overset{(e_1 \cdot t^{-1},e_2 \cdot t^{-1})^{-1}}{\cong} \left(\hat{\mathcal{O}}_{Y_{\infty}}\cdot t^{-1}\right)^{\oplus 2}.$$
Recall $U = \mathcal{Y}^{\mathrm{Ig}}(\epsilon_0)$, and $w' \in \omega(U)$, and write $i_{\mathrm{HT}}(w') = (x_{\mathrm{HT}},y_{\mathrm{HT}})$. Recall that the canonical subgroup is trivialized by $e_{1,1}$ from (\ref{eindefinition}) on $U$ (see Definition \ref{Udefinition}). Thus by \cite[Lemma 2.14]{ChojeckiHansenJohansson} (whose proof shows that $e_{1,1}$ trivializes the canonical subgroup if $|z_{\mathrm{HT}}| > p^{p/(p^2-1)}$), we have $|z_{\mathrm{HT}}| > p^{p/(p^2-1)}$. (Note that our $z_{\mathrm{HT}}$ is equal to the reciprocal of $1/\frak{z}$ in Section 2.4 of op. cit., and the assumption of Lemma 2.14 of op. cit. can be equivalently stated as $|1/\frak{z}| > p^{p/(p^2-1)}$.) Thus $z_{\mathrm{HT}} = -x_{\mathrm{HT}}/y_{\mathrm{HT}} \in \hat{\mathcal{O}}_U(U)$ is well-defined. Thus the image of $i_{\mathrm{HT}}$ is generated by $(-z_{\mathrm{HT}},1)$. 

Consider the map
\begin{equation}\label{OVxmap}\mathcal{O}_{\mathcal{V}_x}(\mathcal{V}_x) \subset \mathcal{O}\mathbb{B}_{\mathrm{dR},\mathcal{V}_x}^+(\mathcal{V}_x) \overset{(\ref{changeofvarsV})}{\hookrightarrow} \mathbb{B}_{\mathrm{dR},\mathcal{V}_x}^+(\mathcal{V}_x)\llbracket q_{\mathrm{dR}}-1\rrbracket \xrightarrow{\theta_q} \mathbb{B}_{\mathrm{dR},\mathcal{V}_x}^+(\mathcal{V}_x).
\end{equation}
Recall that the open subset $\mathcal{V}_x \overset{(\ref{Vz})}{=} \{z_{\mathrm{HT}} \neq 0\} \subset Y_{\infty}$ is open and write
\begin{equation}\label{Vcunion}\mathcal{V}_x = \bigcup_{c > 0,\; c \in |\mathcal{O}_k|}U_c, \hspace{1cm} U_c := \{|z_{\mathrm{HT}}| \ge c\},
\end{equation}
where $|\mathcal{O}_k|$ denotes the set of $p$-adic absolute values of elements of $\mathcal{O}_k$ (see Definition \ref{kdefinition}). We now show that $U_c$ is affinoid perfectoid over $\mathrm{Spa}(k,\mathcal{O}_k)$ for any $c$. Recall that $\{|z_{\mathrm{HT}}| \ge 1\}$ is affinoid perfectoid over $\mathrm{Spa}(k,\mathcal{O}_k)$ by \cite[Theorem IV.1.1 (i)]{ScholzeTorsion}. By (\ref{zHTtransformationproperty}), we thus have that for any $c \in |\mathcal{O}_k|$ and any $n  \in \mathbb{Z}_{\ge 0}$ such that $p^nc \ge 1$, 
$$U_c \cdot g^{-n} = \{|z_{\mathrm{HT}}| \ge p^nc\} \subset \{|z_{\mathrm{HT}}| \ge 1\}.$$
Thus $U_c \cdot g^{-n}$ is a rational subset of the affinoid perfectoid $\{|z_{\mathrm{HT}}| \ge 1\}$ over $\mathrm{Spa}(k,\mathcal{O}_k)$, and thus is itself affinoid perfectoid by \cite[Theorem 6.3 (ii)]{Scholzeperf}. Since the $GL_2(\mathbb{Q}_p)$-action preserves the property of being affinoid perfectoid (see \cite[discussion after Definition III.3.5]{ScholzeTorsion}), we have that $U_c$ is affinoid perfectoid over $\mathrm{Spa}(k,\mathcal{O}_k)$. Thus, by \cite[Lemma 4.10 (iv)]{Scholze} we have
\begin{equation}\label{intermediatepadiccompletion}\hat{\mathcal{O}}_{\mathcal{V}_x}^+(U_c) = \widehat{\mathcal{O}_{\mathcal{V}_x}^+(U_c)}
\end{equation}
where the right-hand side is the $p$-adic completion of $\mathcal{O}_{\mathcal{V}_x}^+(U_c)$. Hence, by Definition 6.1 of op. cit. we have $\mathbb{B}_{\mathrm{inf}}(U_c) = W(\hat{\mathcal{O}}_{\mathcal{V}_x}^{+,\flat}(U_c))[1/p]$. By the definitions in loc. cit., we have 
$$\mathbb{B}_{\mathrm{dR},\mathcal{V}_x}^+(U_c) = \left(\varprojlim_nW(\hat{\mathcal{O}}_{\mathcal{V}_x}^{+,\flat}(U_c)) [1/p]/(\ker\theta)^n\right),$$
where $\theta: W(\hat{\mathcal{O}}_{\mathcal{V}_x}^{+,\flat}(U_c))[1/p] \twoheadrightarrow \hat{\mathcal{O}}_{\mathcal{V}_x}(U_c)$ is as in loc. cit. for every $n \in \mathbb{Z}_{\ge 0}$ we get a map
\begin{align*}\mathcal{O}_{\mathcal{V}_x}(\mathcal{V}_x) \overset{(\ref{OVxmap})}{\rightarrow} \mathbb{B}_{\mathrm{dR},\mathcal{V}_x}^+(\mathcal{V}_x)\rightarrow \mathbb{B}_{\mathrm{dR},\mathcal{V}_x}^+(U_c) &= \left(\varprojlim_nW(\hat{\mathcal{O}}_{\mathcal{V}_x}^{+,\flat}(U_c)) [1/p]/(\ker\theta)^n\right) \\
&\rightarrow W(\hat{\mathcal{O}}_{\mathcal{V}_x}^{+,\flat}(U_c))[1/p]/(\ker\theta)^n
\end{align*}
which factors through
$$\mathcal{O}_{\mathcal{V}_x}(U_c) \rightarrow W(\hat{\mathcal{O}}_{\mathcal{V}_x}^{+,\flat}(U_c))[1/p]/(\ker\theta)^n.$$
The target of this map is $p$-adically complete, and so the above map extends to 
$$\hat{\mathcal{O}}_{\mathcal{V}_x}(U_c) \overset{(\ref{intermediatepadiccompletion})}{=} \widehat{\mathcal{O}_{\mathcal{V}_x}^+(U_c)} [1/p] \rightarrow W(\hat{\mathcal{O}}_{\mathcal{V}_x}^{+,\flat}(U_c))[1/p]/(\ker\theta)^n.$$
These maps are clearly compatible over all $n$, and so we get an induced map
$$\iota_{V_c} : \hat{\mathcal{O}}_{\mathcal{V}_x}(U_c) \rightarrow  \mathbb{B}_{\mathrm{dR},\mathcal{V}_x}(U_c) \overset{(\ref{Bdecomposition''})}{=} \hat{\mathcal{O}}_{\mathcal{V}_x}(U_c)(\!(t)\!).$$
These maps glue over all $c > 0$ varying in $|\mathcal{O}_k|$, and so using (\ref{Vcunion}) we get a map
\begin{equation}\label{iotaVx}\iota_{\mathcal{V}_x} : \hat{\mathcal{O}}_{\mathcal{V}_x}(\mathcal{V}_x) \rightarrow   \hat{\mathcal{O}}_{\mathcal{V}_x}(\mathcal{V}_x)(\!(t)\!)\overset{(\ref{Bdecomposition''})}{=} \mathbb{B}_{\mathrm{dR},\mathcal{V}_x}(\mathcal{V}_x) .
\end{equation}
This gives
\begin{equation}\label{iHT}\omega^{\otimes -1} \otimes_{\mathcal{O}_Y}\hat{\mathcal{O}}_{\mathcal{V}_x}({\mathcal{V}_x}) \overset{i_{\mathrm{HT}}}{\hookrightarrow} H_{\text{\'{e}t}}^1(\mathcal{E}) \otimes_{\hat{\mathbb{Z}}_{p,Y}}\hat{\mathcal{O}}_{\mathcal{V}_x}({\mathcal{V}_x}) \overset{\iota_{\mathcal{V}_x}}{\rightarrow}H_{\text{\'{e}t}}^1(\mathcal{E}) \otimes_{\hat{\mathbb{Z}}_{p,Y}} \mathbb{B}_{\mathrm{dR},{\mathcal{V}_x}}({\mathcal{V}_x}). 
\end{equation}

\begin{definition}Let 
\begin{equation}\label{Bdefinition}\mathbb{B} := \mathbb{B}_{\mathrm{dR},{\mathcal{V}_x}}({\mathcal{V}_x})\llbracket q_{\mathrm{dR}}-1\rrbracket.
\end{equation}
\end{definition}

Then 
$$\mathcal{O}\mathbb{B}_{\mathrm{dR},{\mathcal{V}_x}}({\mathcal{V}_x}) \overset{(\ref{'triv})}{\subset} \mathbb{B}_{\mathrm{dR},{\mathcal{V}_x}}({\mathcal{V}_x})\llbracket X\rrbracket \overset{(\ref{changeofvars})}{=} \mathbb{B}.$$
Thus by \cite[Theorem 8.8]{Scholze},
\begin{equation}\label{comparisonmap2}H_{\text{\'{e}t}}^1(\mathcal{E}) \otimes_{\hat{\mathbb{Z}}_{p,Y}}\mathbb{B} \cong H_{\mathrm{dR}}^1(\mathcal{E}) \otimes_{\mathcal{O}_Y}\mathbb{B}
\end{equation}
induced by $i_{\mathrm{dR}}$ from (\ref{comparisonmap}). We get a map
\begin{equation}\label{lineHT}\omega^{-1} \otimes_{\mathcal{O}_Y}\hat{\mathcal{O}}_{\mathcal{V}_x}({\mathcal{V}_x}) \overset{(\ref{iHT})}{\rightarrow} H_{\text{\'{e}t}}^1(\mathcal{E}) \otimes_{\hat{\mathbb{Z}}_{p,Y}} \mathbb{B} \overset{i_{\mathrm{dR}}}{\cong} H_{\mathrm{dR}}^1(\mathcal{E}) \otimes_{\mathcal{O}_Y}\mathbb{B}.
\end{equation}

\begin{convention}In a slight abuse of notation, we let $z_{\mathrm{HT}}$ denote the element $i_{\mathrm{HT}}(z_{\mathrm{HT}}) \in \mathbb{B}_{\mathrm{dR},{\mathcal{V}_x}}^+({\mathcal{V}_x})$. We will use $z_{\mathrm{HT}}$ to split the Hodge-Tate filtration over ${\mathcal{V}_x}$ over the large ring $\mathbb{B}$ containing $\mathbb{B}_{\mathrm{dR},{\mathcal{V}_x}}^+({\mathcal{V}_x})$. Similarly, we let $z_{\mathrm{HT}}$ denote the image of $i_{\mathrm{HT}}(z_{\mathrm{HT}})$ under $\mathbb{B}_{\mathrm{dR},{\mathcal{V}_x}}({\mathcal{V}_x}) \subset \mathbb{B}$. 
\end{convention}

Under the isomorphism
$$H_{\mathrm{dR}}^1(\mathcal{E}) \otimes_{\mathcal{O}_Y}\mathbb{B} \overset{i_{\mathrm{dR}}}{\cong} H_{\text{\'{e}t}}^1(\mathcal{E})\otimes_{\hat{\mathbb{Z}}_{p,Y}}\mathbb{B}\underset{\sim}{\xrightarrow{(e_1,e_2)}} \mathbb{B}^{\oplus 2}, 
$$
we have
\begin{equation}\label{zdRzHTcoordinates}\omega \otimes_{\mathcal{O}_Y}\mathbb{B} \overset{i_{\mathrm{dR}}}{\cong} \mathbb{B}\cdot (-z_{\mathrm{dR}},1), \hspace{1cm} \omega^{\otimes -1} \otimes_{\mathcal{O}_Y}\mathbb{B} \overset{(\ref{lineHT})}{\cong} \mathbb{B} \cdot (-z_{\mathrm{HT}},1).
\end{equation}
%Recall 
%\begin{equation}\label{Wdefinition}\mathcal{V} = \{1/\theta(\theta_q(z_{\mathrm{dR}})) \neq 0, z_{\mathrm{HT}} \neq 0, 1/\theta(\theta_q(z_{\mathrm{dR}})) - 1/z_{\mathrm{HT}} \neq 0\} \subset {\mathcal{V}_x}.
%\end{equation}
%from Definition \ref{V'V'0definition}. (We recall the definition of $\mathcal{V}$ here for easy reference.) Then letting 
%\begin{equation}\label{mathbbBdefinition}\mathbb{B} = \mathbb{B}_{\mathrm{dR},{\mathcal{V}_x}}(\mathcal{V})\llbracket q_{\mathrm{dR}}-1\rrbracket,
%\end{equation}
\begin{proposition}We have
\begin{equation}\label{zdRzHTinvertible}z_{\mathrm{dR}}-z_{\mathrm{HT}} \in \mathbb{B}^{\times}.
\end{equation}
\end{proposition}

\begin{proof}Recall $\theta_q$ from (\ref{thetaq}) and $\theta_t$ from Definition \ref{evaluationdefinition}. By (\ref{Bdefinition}) and the fact that $\theta_q(q_{\mathrm{dR}}-1) = 0$, (\ref{zdRzHTinvertible}) is equivalent to 
$$\theta_t(\theta_q(z_{\mathrm{dR}}-z_{\mathrm{HT}})) \in \hat{\mathcal{O}}_{\mathcal{V}_x}(\mathcal{V}_x)^{\times}.$$
But
$$\theta_t(\theta_q(z_{\mathrm{dR}}-z_{\mathrm{HT}})) \overset{(\ref{thetaqzzero})}{=} - \theta_t(\theta_q(z_{\mathrm{HT}})) \overset{(\ref{iotaVx})}{=} -z_{\mathrm{HT}} \in \hat{\mathcal{O}}_{\mathcal{V}_x}(\mathcal{V}_x)^{\times}$$ 
where the last inclusion follows since $\mathcal{V}_x \overset{(\ref{Vz})}{=} \{z_{\mathrm{HT}} \neq 0\}$.

%Since $\mathbb{B}_{\mathrm{dR},\mathcal{V}_x}\llbracket q_{\mathrm{dR}}-1\rrbracket$ is a sheaf on $Y_{\text{pro\'{e}t}}/\mathcal{V}_x$, then since $\mathcal{V}_x = \bigcup_{m \in \mathbb{Z}_{\ge 0}}U_m$ by (\ref{Vxunion'}), it suffices to show that $z_{\mathrm{dR}} - z_{\mathrm{HT}} \in \mathbb{B}_{\mathrm{dR},\mathcal{V}_x}(U_m)\llbracket q_{\mathrm{dR}}-1\rrbracket^{\times}$ for all $m \in \mathbb{Z}_{\ge 0}$. This is equivalent to showing that $\theta(\theta_z)

%Recall $\mathcal{Y}^{\mathrm{Ig}} = \{z_{\mathrm{HT}} = \infty\}$. 

%By Definition \ref{'zqwdefinition} we 
\end{proof}

By (\ref{zdRzHTinvertible}), (\ref{zdRzHTcoordinates}) induces an isomorphism
\begin{equation}\label{HTdecomposition}H_{\mathrm{dR}}^1(\mathcal{E}) \otimes_{\mathcal{O}_Y}\mathbb{B} \cong \left(\omega \otimes_{\mathcal{O}_Y}\mathbb{B}\right) \oplus \left(\omega^{\otimes -1} \otimes_{\mathcal{O}_Y}\mathbb{B}\right).
\end{equation}
Note that the second factor of the above decomposition is annihilated by the Gauss-Manin connection. 

\begin{definition}
Let
\begin{equation}\label{split}\mathrm{split} : H_{\mathrm{dR}}^1(\mathcal{E}) \otimes_{\mathcal{O}_Y}\mathbb{B} \overset{i_{\mathrm{dR}}}{\cong} H_{\text{\'{e}t}}^1(\mathcal{E}) \otimes_{\hat{\mathbb{Z}}_{p,Y}}\mathbb{B} \rightarrow \omega \otimes_{\mathcal{O}_Y} \mathbb{B}
\end{equation}
denote the projection onto the first coordinate. Since the second factor of (\ref{HTdecomposition}) is annihilated by the connection (given by the convolution of the Gauss-Manin connection and the natural connection on $\mathbb{B}$)
$$\nabla : H_{\mathrm{dR}}^1(\mathcal{E}) \otimes_{\mathcal{O}_Y}\mathbb{B} \rightarrow H_{\mathrm{dR}}^1(\mathcal{E}) \otimes_{\mathcal{O}_Y}\mathbb{B} \otimes_{\mathcal{O}_Y}\Omega_Y,$$
(\ref{split}) gives a horizontal splitting of the Hodge filtration (\ref{Hodgefiltrationinclusion}), i.e. a left inverse of $i_{\mathrm{dR}} : \omega \otimes_{\mathcal{O}_Y}\mathbb{B} \hookrightarrow H_{\mathrm{dR}}^1(\mathcal{E})\otimes_{\mathcal{O}_Y}\mathbb{B}$ whose kernel is annihilated by the Gauss-Manin connection. 
\end{definition}

The Hodge filtration (\ref{Hodgefiltrationinclusion}) induces a natural inclusion
$$\omega^{\otimes k} \hookrightarrow \mathrm{Sym}^kH_{\mathrm{dR}}^1(\mathcal{E}).$$
The Gauss-Manin connection
$$\nabla : \mathrm{Sym}^kH_{\mathrm{dR}}^1(\mathcal{E}) \rightarrow \mathrm{Sym}^kH_{\mathrm{dR}}^1(\mathcal{E}) \otimes_{\mathcal{O}_Y}\Omega_Y$$
along with the natural connection $\mathbb{B} \rightarrow \mathbb{B} \otimes_{\mathcal{O}_Y}\Omega_Y$ induces a connection
$$\nabla : \mathrm{Sym}^kH_{\mathrm{dR}}^1(\mathcal{E}) \rightarrow \mathrm{Sym}^kH_{\mathrm{dR}}^1(\mathcal{E}) \otimes_{\mathcal{O}_Y}\Omega_Y.$$

\begin{definition}[$p$-adic Maass-Shimura operator]\label{padicMSdefinition}\begin{enumerate}
\item Define 
$$\partial_k : \omega^{\otimes k}\otimes_{\mathcal{O}_Y}\mathbb{B} \rightarrow \omega^{\otimes k+2} \otimes_{\mathcal{O}_Y}\mathbb{B}$$
as the following composition:
\begin{align*}\omega^{\otimes k} \otimes_{\mathcal{O}_Y}\mathbb{B} \hookrightarrow \mathrm{Sym}^kH_{\mathrm{dR}}^1(\mathcal{E}) \otimes_{\mathcal{O}_Y}\mathbb{B} \xrightarrow{\nabla} \mathrm{Sym}^kH_{\mathrm{dR}}^1(\mathcal{E}) \otimes_{\mathcal{O}_Y}\mathbb{B} \otimes_{\mathcal{O}_Y}\Omega_Y &\xrightarrow{\mathrm{split}} \omega^{\otimes k} \otimes_{\mathcal{O}_Y} \mathbb{B} \otimes_{\mathcal{O}_Y} \Omega_Y \\
&\underset{\sim}{\xrightarrow{\mathrm{KS}}} \omega^{\otimes k+2} \otimes_{\mathcal{O}_Y}\mathbb{B}.
\end{align*}
Here $\mathrm{KS}$ is the Kodaira-Spencer isomorphism from (\ref{'KS}). Given $w' \in \omega^{\otimes k}\otimes_{\mathcal{O}_Y}\mathbb{B}$, write $w' = F \cdot w_{\mathrm{can}}^{\otimes k}$ where $F \in \mathbb{B}$. Then define 
$$d_kF \in \mathbb{B}$$
by 
\begin{equation}\label{partialdrelate}\partial_k(F \cdot w_{\mathrm{can}}^{\otimes k}) = d_kF \cdot w_{\mathrm{can}}^{\otimes k+2}.
\end{equation}
This gives a map
$$d_k : \mathbb{B} \rightarrow \mathbb{B}.$$

\item Let 
\begin{equation}\label{partialkjdefinition}\partial_k^j := \partial_{k+2(j-1)} \circ \partial_{k + 2(j-2)}\circ \cdots \circ \partial_{k+2} \circ \partial_k :  \omega^{\otimes k}\otimes_{\mathcal{O}_Y}\mathbb{B} \rightarrow \omega^{\otimes k+2j} \otimes_{\mathcal{O}_Y}\mathbb{B}
\end{equation}
\item Let 
\begin{equation}\label{dkjdefinition}d_k^j := d_{k+2(j-1)} \circ d_{k+2(j-2)} \circ \cdots \circ d_{k+2} \circ d_k : \mathbb{B} \rightarrow \mathbb{B}.
\end{equation}
\item Thus, (\ref{partialdrelate}), (\ref{partialkjdefinition}) and (\ref{dkjdefinition}) imply
\begin{equation}\label{partialdkjrelate}\partial_k^j(F \cdot w_{\mathrm{can}}^{\otimes k}) = d_k^jF \cdot w_{\mathrm{can}}^{\otimes k + 2j}.
\end{equation}
\end{enumerate}
\end{definition}

Recall the notion of a generalized $p$-adic modular form of weight $k$ on $\mathcal{V}_x$ from Definition \ref{generalizedpadicmodularformdefinition}. This is in particular an element 
$$F \in \mathbb{B}_{\mathrm{dR},\mathcal{V}_x}(\mathcal{V}_x)\llbracket q_{\mathrm{dR}}-1\rrbracket = \mathbb{B}.$$ %We will view the space of generalized $p$-adic modular forms on $\mathcal{V}_x$ as living inside $\mathbb{B}$ via the natural map
%$$\mathbb{B}_{\mathrm{dR},\mathcal{V}_x}(\mathcal{V}_x)\llbracket q_{\mathrm{dR}}-1\rrbracket \rightarrow \mathbb{B}_{\mathrm{dR},\mathcal{V}_x}(\mathcal{V}_x)\llbracket q_{\mathrm{dR}}-1\rrbracket \overset{(\ref{fullchangeofvars})}{=} \mathbb{B}.$$

\begin{corollary}\label{weightcorollary}Suppose $F \in \mathbb{B}$ is induced from a generalized $p$-adic modular form of weight $k$. Then $d_k^jF \in \mathbb{B}$ is a generalized $p$-adic modular form of weight $k + 2j$. Moreover, if $F|_{\mathcal{Y}^{\mathrm{Ig}}(\epsilon)}$ satisfies (\ref{Ftransformationidentity}) for a given $n \ge n(\epsilon)$ and given $\gamma \in \Gamma_{0,p}(p^n)$, then $d_k^jF|_{\mathcal{Y}^{\mathrm{Ig}}(\epsilon)}$ satisfies the identity (\ref{Ftransformationidentity}) for $\gamma \in \Gamma \in \Gamma_{0,p}(p^n)$ with $k + 2j$ in place of $k$. 
\end{corollary}

\begin{proof}This is immediate from the defining formula (\ref{dkjdefinition}) and a direct calculation using (\ref{Ftransformationidentity}) and (\ref{modulartransformationidentity}). 

\end{proof}

\begin{proposition}We have
\begin{equation}\label{dkjformula}\begin{split}d_k^j = \sum_{i = 0}^j\frac{\binom{j}{i}\binom{j+k-1}{i}i!}{(t(z_{\mathrm{dR}}-z_{\mathrm{HT}}))^i}\left(\frac{\nabla}{\nabla(t\cdot z_{\mathrm{dR}})}\right)^{j-i} &= \sum_{i = 0}^j\frac{\binom{j}{i}\binom{j+k-1}{i}i!}{(t(z_{\mathrm{dR}}-z_{\mathrm{HT}}))^i}\left(\frac{\nabla}{\nabla(\mathbf{z})}\right)^{j-i} \\
&= \sum_{i = 0}^j\frac{\binom{j}{i}\binom{j+k-1}{i}i!}{(t(z_{\mathrm{dR}}-z_{\mathrm{HT}}))^i}\left(\frac{q_{\mathrm{dR}}\nabla}{\nabla q_{\mathrm{dR}}}\right)^{j-i}.
\end{split}
\end{equation}
\end{proposition}

\begin{proof}This follows from a direct calculation. We give some details for the reader's convenience.
We have 
$$i_{\mathrm{dR}}(F \cdot w_{\mathrm{can}}^{\otimes k}) = F \cdot (-e_1z_{\mathrm{dR}} + e_2)^{\otimes k}.$$
Thus, using the Leibniz rule, we have
\begin{align*}i_{\mathrm{dR}}(\nabla(F\cdot w_{\mathrm{can}}^{\otimes k})) &= \nabla (i_{\mathrm{dR}}(F \cdot w_{\mathrm{can}}^{\otimes k})) \\
&= \frac{\nabla}{\nabla z_{\mathrm{dR}}}F \cdot (-e_1 z_{\mathrm{dR}} + e_2)^{\otimes k} \otimes \nabla z_{\mathrm{dR}} + k \cdot (-e_1) \cdot (-e_1z_{\mathrm{dR}}+e_2)^{\otimes k-1} \otimes \nabla z_{\mathrm{dR}}.
\end{align*}
Applying $\mathrm{split}$, we get
$$\mathrm{split}(i_{\mathrm{dR}}(\nabla(F \cdot w_{\mathrm{can}}^{\otimes k}))) = \left(\frac{\nabla}{\nabla z_{\mathrm{dR}}}F + \frac{k}{z_{\mathrm{dR}}-z_{\mathrm{HT}}}\right) \cdot i_{\mathrm{dR}}(w_{\mathrm{can}})^{\otimes k} \otimes \nabla z_{\mathrm{dR}}.$$
Applying Proposition \ref{'KSproposition}, we get 
$$i_{\mathrm{dR}}(\mathrm{split}(\nabla (F\cdot w_{\mathrm{can}}^{\otimes k}))) = \mathrm{split}(i_{\mathrm{dR}}(\nabla(F \cdot w_{\mathrm{can}}^{\otimes k}))) = \frac{1}{t}\left(\frac{\nabla}{\nabla z_{\mathrm{dR}}}F + \frac{k}{z_{\mathrm{dR}}-z_{\mathrm{HT}}}\right) \cdot i_{\mathrm{dR}}(w_{\mathrm{can}}^{\otimes k +2}).$$
By (\ref{comparisonmap2}), this proves the formula for $j = 1$. The general case is proven by an easy calculation and induction on $j$. 
\end{proof}

\begin{convention}We will henceforth abuse notation and denote $\nabla z_{\mathrm{dR}}$ by $dz_{\mathrm{dR}}$ and $\nabla q_{\mathrm{dR}}$ by $dq_{\mathrm{dR}}$. No confusion should arise as $\nabla : \mathbb{B} \rightarrow \mathbb{B}\otimes_{\mathcal{O}_{\mathcal{V}_x}(\mathcal{V}_x)}\Omega(\mathcal{V}_x)$ extends the exterior derivative $d : \mathcal{O}_{\mathcal{V}_x}(\mathcal{V}_x) \rightarrow \Omega_{\mathcal{V}_x}(\mathcal{V}_x)$. The reason for choosing this notation is to emphasize the appearance of $q_{\mathrm{dR}}$-logarithmic derivatives in the formula (\ref{dkjformula}). Such logarithmic derivatives are known by classical work of Kummer (\cite[p. 493]{Kummer}) to have an interpretation as moments of measures. In classical cases such as the measures of Kubota-Leopoldt (\cite{KubotaLeopoldt}), Deligne-Ribet (\cite{DeligneRibet}), Manin-Vishik (\cite{ManinVishik}) and Katz (\cite{KatzCM}), these moments are related to twists of central critical $L$-values by characters with positive Hodge-Tate weights. As we will later see, the $p$-adic $L$-functions we construct from power series in $q_{\mathrm{dR}}-1$ will also have $p$-adic Maass-Shimura derivatives (which can be viewed as ``generalized logarithmic derivatives'' in view of the formula (\ref{dkjformula2}) below) related to such twisted central critical $L$-values. 

Hence (\ref{dkjformula}) becomes
\begin{equation}\label{dkjformula2}\begin{split}d_k^j = \sum_{i = 0}^j\frac{\binom{j}{i}\binom{j+k-1}{i}i!}{(t(z_{\mathrm{dR}}-z_{\mathrm{HT}}))^i}\left(\frac{d}{d(t\cdot z_{\mathrm{dR}})}\right)^{j-i} &= \sum_{i = 0}^j\frac{\binom{j}{i}\binom{j+k-1}{i}i!}{(t(z_{\mathrm{dR}}-z_{\mathrm{HT}}))^i}\left(\frac{d}{d\mathbf{z}}\right)^{j-i} \\
&= \sum_{i = 0}^j\frac{\binom{j}{i}\binom{j+k-1}{i}i!}{(t(z_{\mathrm{dR}}-z_{\mathrm{HT}}))^i}\left(\frac{q_{\mathrm{dR}}d}{d q_{\mathrm{dR}}}\right)^{j-i}.
\end{split}
\end{equation}
\end{convention}

\begin{remark}Note the analogy between (\ref{dkjformula2}) and the formula for the real analytic Maass-Shimura operator from \cite{KatzCM} that sends weight $k$ nearly holomorphic modular forms to weight $k + 2j$ nearly holomorphic modular forms:
\begin{equation}\label{complexMSoperator}\frak{d}_k^j := \sum_{i = 0}^j\frac{\binom{j}{i}\binom{j+k-1}{i}i!}{(2\pi i(\tau - \bar{\tau}))^i}\left(\frac{d}{2\pi id\tau}\right)^{j-i}.
\end{equation}
Here, $\tau$ is the coordinate on the complex universal cover $\mathcal{H}^+ = \{\mathrm{Im}(\tau) > 0\}$ and $\bar{\tau}$ is its complex conjugate. The coordinate $z_{\mathrm{dR}}$ is the $p$-adic analogue of $\tau$, $z_{\mathrm{HT}}$ is the $p$-adic analogue of $\bar{\tau}$ and  $t$ (see (\ref{'tdefinition})) is the $p$-adic analogue of $2 \pi i$. Note that 
$$\frac{d}{dz_{\mathrm{dR}}}z_{\mathrm{HT}} = 0$$
since $z_{\mathrm{HT}} \in \mathcal{O}_{\hat{Y}_{\infty}}(\hat{\mathcal{V}}_x)$, $\hat{Y}_{\infty}$ is a perfectoid space and perfectoid spaces have no nontrivial differentials (see \cite[discussion below Theorem 2.8]{ScholzeCDM}). This is in analogy with the Cauchy-Riemann equation
$$\frac{d}{d\tau}\bar{\tau} = 0.$$
Similarly, one may view $\mathcal{O}_{\mathcal{V}_x}({\mathcal{V}_x})$ as analogous to the ring of holomorphic functions on $\mathcal{H}^+$, and $\mathbb{B}$ as analogous to the ring of real analytic functions on $\mathcal{H}^+$, over which we have the Hodge decomposition (\ref{HTdecomposition}).
\end{remark}

\begin{remark}\label{recoverASremark}Note that on $\mathcal{Y}^{\mathrm{Ig}} = \{z_{\mathrm{HT}} = \infty\}$, we have 
$$\frac{1}{z_{\mathrm{dR}}-z_{\mathrm{HT}}} = 0$$
and so (\ref{dkjformula2}) specializes to $\left(\frac{q_{\mathrm{dR}}d}{dq_{\mathrm{dR}}}\right)^j$. In view of Theorem \ref{STanalyticcontinuationtheorem} (2), this shows that $d_k^j|_{\mathcal{Y}^{\mathrm{Ig}}}$ acts as the ordinary Atkin-Serre operator $\theta_{\mathrm{AS}}^j = \left(\frac{q_{\mathrm{ST}}d}{dq_{\mathrm{ST}}}\right)^j$. 
\end{remark}

Recall the $\hat{\mathcal{O}}_{\mathcal{V}_x}(\mathcal{V}_x)$-module homomorphism 
$$\theta_t \circ \theta_q : \mathbb{B} = \mathbb{B}_{\mathrm{dR},\mathcal{V}_x}(\mathcal{V}_x)\llbracket q_{\mathrm{dR}}-1\rrbracket \rightarrow \hat{\mathcal{O}}_{\mathcal{V}_x}(\mathcal{V}_x)$$
from Definition \ref{evaluationdefinition}. %From this, we get a $\hat{\mathcal{O}}_{\mathcal{V}_x}({\mathcal{V}_x})$-module homomorphism
%\begin{equation}\label{thetatthetaX}\theta_t \circ \theta_q : B\rightarrow \hat{\mathcal{O}}_{\mathcal{V}_x}({\mathcal{V}_x}).
%\end{equation}
%Note that the pullback of $\theta_t \circ \theta_q$ along $\mathbb{B}_{\mathrm{dR},{\mathcal{V}_x}}({\mathcal{V}_x})\llbracket q_{\mathrm{dR}}-1\rrbracket \rightarrow B$ is equal to $\theta \circ \theta_q$ where $\theta$ is from (\ref{'thetas}) and $\theta_q$ is from Definition \ref{thetaqdefinition}. 

\begin{definition}Define $\tilde{w}_{\mathrm{can}}$ to be the image of $w_{\mathrm{can}}$ under the extension of scalars $\otimes_{\mathbb{B},\theta_t\circ \theta_q}\hat{\mathcal{O}}_{\mathcal{V}_x}(\mathcal{V}_x)$, i.e.
\begin{equation}\label{tildewcan}\tilde{w}_{\mathrm{can}} := w_{\mathrm{can}}\otimes_{\mathbb{B},\theta_t \circ \theta_q}1\in \left(\omega(\mathcal{V}_x) \otimes_{\mathcal{O}_Y}\mathbb{B}\right) \otimes_{\mathbb{B},\theta_t\circ \theta_q}\hat{\mathcal{O}}_{\mathcal{V}_x}(\mathcal{V}_x) = \omega({\mathcal{V}_x}) \otimes_{\mathcal{O}_Y}\hat{\mathcal{O}}_{\mathcal{V}_x}({\mathcal{V}_x}),
\end{equation}
where $\otimes_{\mathbb{B},\theta_t\circ \theta_q}\hat{\mathcal{O}}_{\mathcal{V}_x}(\mathcal{V}_x)$ denotes that $\mathbb{B}$ acts on $\hat{\mathcal{O}}_{\mathcal{V}_x}(\mathcal{V}_x)$ in the tensor product through $\theta_t \circ \theta_q : \mathbb{B} \rightarrow \hat{\mathcal{O}}_{\mathcal{V}_x}(\mathcal{V}_x)$.
\end{definition}

\begin{proposition}$\tilde{w}_{\mathrm{can}} \in \omega \otimes_{\mathcal{O}_Y}\hat{\mathcal{O}}_{\mathcal{V}_x}(\mathcal{V}_x)$ is a generator (i.e. nowhere vanishing). 
\end{proposition}

\begin{proof}From Corollary \ref{wcangeneratorproposition3} we know that $w_{\mathrm{can}} \in \omega \otimes_{\mathcal{O}_Y} \mathbb{B}_{\mathrm{dR},{\mathcal{V}_x}}^+({\mathcal{V}_x})\llbracket q_{\mathrm{dR}}-1\rrbracket$ is a generator. Now the assertion follows because the invertible elements $f \in\mathbb{B}_{\mathrm{dR},{\mathcal{V}_x}}^+({\mathcal{V}_x})\llbracket q_{\mathrm{dR}}-1\rrbracket$ are exactly those elements such that the constant $\theta_q(f) \in \mathbb{B}_{\mathrm{dR},\mathcal{V}_x}^+(\mathcal{V}_x)$ is invertible, which is in turn true if and only if $\theta_t \circ \theta_q(f) = \theta \circ \theta_q(f) \in \hat{\mathcal{O}}_{\mathcal{V}_x}({\mathcal{V}_x})^{\times}$. 
\end{proof}

%\begin{proposition}\label{thetacommuteproposition}Suppose $g \in GL_2(\mathbb{Z}_p)$ with $W \cdot g \subset W$ and $f \in \mathbb{B}$. Then 
%$$\theta_t \circ \theta_X(g^*f) = g^*\theta_t\circ \theta_X(f).$$
%\end{proposition}

%\begin{proof}The element $X = j - [j^{\flat}]$ from Theorem \ref{Utriv} and the identification (\ref{'triv}) is invariant under the $GL_2(\mathbb{Z}_p)$-action, since $j$ is a section on $Y$ and $\mathrm{Gal}(Y_{\infty}/Y) = GL_2(\mathbb{Z}_p)$. Thus the ideal $X \cdot \mathbb{B} \subset \mathbb{B}$ is fixed by the $GL_2(\mathbb{Z}_p)$-action. Similarly, from (\ref{'tdefinition}) we have
%$$g^*t = \mathrm{det}(g)\cdot t$$
%for any $g \in GL_2(\mathbb{Z}_p)$. Thus the $\hat{\mathcal{O}}_{\mathcal{V}_x}$-submodule $t^n \cdot \hat{\mathcal{O}}_{\mathcal{V}_x}(W)$ in 
%$$\mathbb{B}_{\mathrm{dR},{\mathcal{V}_x}}(W) \overset{(\ref{Bdecomposition}), (\ref{Wdefinition})}{=} \hat{\mathcal{O}}_{\mathcal{V}_x}(W)(\!(t)\!)$$
%is fixed by the $GL_2(\mathbb{Z}_p)$-action. Since $\theta_X$ is reduction modulo $X \cdot \mathbb{B}$ and $\theta_t$ is reduction modulo $\{t^n\}_{n \in \mathbb{Z} \setminus 0}\cdot \hat{\mathcal{O}}_{\mathcal{V}_x}(W)$, this gives the assertion.
%\end{proof}

We will often consider values
$$\theta_t \circ \theta_q(d_k^jF)$$
for a weight $k$ generalized $p$-adic modular form $F \in \mathbb{B}$. 

\subsection{Values of $p$-adic Maass-Shimura operators at CM points}

Recall the notation of Section \ref{algebraicYsection}. As in \cite[Chapter 1.3.1]{YZZ}, let
\begin{equation}\label{mathbbYinfty}\mathbb{Y}_{\infty} = \varprojlim_{\Gamma' \subset D^{\times}(\mathbb{A}_{\mathbb{Q}}^{(\infty)})} \mathbb{Y}(\Gamma'),
\end{equation}
where $\Gamma' \subset D^{\times}(\mathbb{A}_{\mathbb{Q}}^{(\infty)})$ ranges over all compact open subgroups. Given an imaginary quadratic field $K/\mathbb{Q}$, fix an embedding $i : K \hookrightarrow D$, which induces an embedding 
\begin{equation}\label{idefinition}i : \mathbb{A}_K \hookrightarrow D(\mathbb{A}_{\mathbb{Q}}).
\end{equation} Note that $D^{\times}(\mathbb{A}_{\mathbb{Q}}^{(\infty)})$ acts on the right of $\mathbb{Y}_{\infty}$, and thus so does $\mathbb{A}_K^{\times,(\infty)}$. 

\begin{definition}\begin{enumerate}
\item Let 
\begin{equation}\label{mathbbYinftyK}\mathbb{Y}_{\infty}^{K^{\times}} \subset \mathbb{Y}_{\infty}
\end{equation}
denote the subscheme of fixed points of $K^{\times} \overset{i}{\hookrightarrow} D^{\times}$ under the right action. It is defined over $\mathbb{Q}$ (see Chapter 1.3.1 of op. cit.). We call any closed point in $\mathbb{Y}_{\infty}^{K^{\times}}$ a \emph{CM point for $K$}. 
\item By the theory of complex multiplication (see loc. cit.), we have 
$$\mathbb{Y}_{\infty}^{K^{\times}}(\overline{K}) = \mathbb{Y}_{\infty}^{K^{\times}}(K^{\mathrm{ab}}),$$
where $K^{\mathrm{ab}}/K$ denotes the maximal abelian extension. 
\item By Shimura's reciprocity law (see Chapter 3.1.2 of op. cit.), the right action of $\mathbb{A}_K^{\times,(\infty)}$ on $\mathbb{Y}_{\infty}^{K^{\times}}$ factors through the reciprocity map $\mathrm{rec}_K : \mathbb{A}_K^{\times} \rightarrow \mathrm{Gal}(K^{\mathrm{ab}}/K)$. 
\end{enumerate}
\end{definition}

Let $\pi : \mathcal{E} \rightarrow \mathbb{Y}$ denote the algebraic universal object (see Section \ref{algebraicYsection}; also recall that by Convention \ref{Yconvention}, we assume that the level $\Gamma$ of $\mathbb{Y}$ is neat). We will slightly the conflate notation of (\ref{omegaY}) and let $\omega = \pi_*\Omega_{\mathcal{E}/\mathbb{Y}}$ below. (No confusion should arise, as (\ref{omegaY}) is the $p$-adic analytification of $\omega$ here.)

Suppose $k \in \mathbb{Z}_{\ge 0}$ and $w' \in \omega^{\otimes k}(\mathbb{Y}_{\infty})$. Write
$$w' = F \cdot w_{\mathrm{can}}^{\otimes k}, \hspace{1cm} F \in \mathbb{B},$$
where $\mathbb{B}$ is as in (\ref{Bdefinition}). Recall that $\bigsqcup_{\pi_0(\mathbb{Y}_{\infty}(\mathbb{C})^{\mathrm{an}})}\mathcal{H}^+ \rightarrow \mathbb{Y}_{\infty}(\mathbb{C})^{\mathrm{an}}$ is the complex analytic universal cover, where $\mathbb{Y}_{\infty}(\mathbb{C})^{\mathrm{an}}$ is the complex analytic space associated to $\mathbb{Y}_{\infty}$ and $\pi_0$ is the component set. More precisely, $\mathcal{H}^+$ is the universal cover of every connected component of $\mathbb{Y}_{\infty}(\mathbb{C})^{\mathrm{an}}$. Then the elliptic curve (resp. false elliptic curve) underlying $\mathcal{E}|_{\mathbb{Y}_{\infty}(\mathbb{C})^{\mathrm{an}}} \times_{\mathbb{Y}_{\infty}(\mathbb{C})^{\mathrm{an}}}\mathcal{H}^+$ is $\mathbb{C}/(\mathbb{Z}\tau + \mathbb{Z})$ (resp. $\mathbb{C}^{\oplus 2}/\left(\iota_{\infty}(\mathcal{O}_D)\cdot \left(\begin{array}{ccc} \tau \\
1\\
\end{array}\right)\right)$ in the notation of Choice \ref{choice} below), and letting $z$ denote the standard coordinate on $\mathbb{C}$ we get a section
\begin{equation}\label{2piidz}2\pi idz \in \omega(\mathcal{H}^+).
\end{equation}
Write
$$w' = G \cdot (2\pi idz)^{\otimes k}, \hspace{1cm} G \in \mathcal{O}_{\mathcal{H}^+}(\mathcal{H}^+),$$
where $\mathcal{O}_{\mathcal{H}^+}$ denotes the sheaf of holomorphic functions on $\mathcal{H}^+$. 

Let $\mathcal{Y}_{\infty}$ denote the adic space over $\mathrm{Spa}(\mathbb{Q}_p,\mathbb{Z}_p)$ attached to $\mathbb{Y}_{\infty}\times_{\mathrm{Spec}(\mathbb{Q})}\mathrm{Spec}(\mathbb{Q}_p)$. Then we have a natural map of pro-adic spaces over $\mathrm{Spa}(\mathbb{Q}_p,\mathbb{Z}_p)$
\begin{equation}\label{analyticprojection}\mathcal{Y}_{\infty} \rightarrow Y_{\infty}.
\end{equation}

Given a neat $\Gamma'$ as in Definitions \ref{congruencesubgroups} and \ref{neatdefinition} with associated universal object $\pi : \mathcal{E}(\Gamma') \rightarrow \mathbb{Y}(\Gamma')$ as in the end of Section \ref{algebraicYsection}, let 
$$\omega_{\mathbb{Y}(\Gamma')} := \pi_*\Omega_{\mathcal{E}(\Gamma')/\mathbb{Y}(\Gamma')}.$$
Below, we will let 
$$\omega^{\otimes k}(\mathbb{Y}_{\infty}) = \omega_{\mathbb{Y}(\Gamma')}^{\otimes k}(\mathbb{Y}_{\infty})$$
for any $\Gamma' \subset \Gamma(Np^{\infty})$ (recall $\mathbb{Y}_{\infty} = \mathbb{Y}(\Gamma(Np^{\infty}))$ from the end of Section \ref{algebraicYsection}, recalling our choice of $\Gamma = \Gamma(N)$ from Convention \ref{Yconvention}). It is evident that $\omega^{\otimes k}(\mathbb{Y}_{\infty})$ is independent of the choice of such $\Gamma'$. 

\begin{theorem}[cf. Key Lemma 5.1.27 of \cite{KatzCM}]\label{CMcoincidetheorem}
Given $k \in \mathbb{Z}_{\ge 0}$ and $w' \in \omega^{\otimes k}(\mathbb{Y}_{\infty})$, let $F$ and $G$ be as above. Let us make the following assumptions and notational conventions. 
\begin{enumerate}
\item Suppose $p$ is inert or ramified in $K$ and suppose $\mathbf{y} \in \mathbb{Y}_{\infty}^{K^{\times}}(\overline{K})$ (i.e. $\mathbf{y}$ is a CM point for $K$). 
\item Suppose $\mathbf{y} \in \mathbb{Y}_{\infty}^{K^{\times}}$ whose $p$-adic analytification $\mathbf{y}^{\mathrm{an}} \in \mathcal{Y}_{\infty}$ has image under the map (\ref{analyticprojection}) 
$$y \in Y_{\infty}$$
which lies in 
$$U_{n(\epsilon_0)-1} \overset{(\ref{Um})}{=} U \cdot g^{n(\epsilon_0)-1} = \mathcal{Y}^{\mathrm{Ig}}(\epsilon_0) \cdot g^{n(\epsilon_0)-1} \overset{(\ref{gisomorphism})}{=} \mathcal{Y}^{\mathrm{Ig}}(p^{n(\epsilon_0)-1}\epsilon_0) \subset Y_{\infty},$$
where $U$ is as in Definition \ref{Udefinition} and $n(\epsilon_0)$ is as in (\ref{nepsilondefinition}). %Then in fact, $y \in \mathcal{V}$ where $\mathcal{V}$ is as in (\ref{Wdefinition}). 
%\item %Furthermore write the complex analytification of $\mathbf{y}$ as $y = [\tau,\gamma] \in \mathbb{Y}(\mathbb{C})$ under (\ref{doublequotient}). 
\item Fix a choice of generator
$$w_0 \in \omega(\mathbf{y}),$$
which is thus a differential defined over $K^{\mathrm{ab}}$. Recall $2\pi idz \in \omega(\mathcal{H}^+)$ from (\ref{2piidz}). Given $\mathbf{y} \in \mathbb{Y}$ with underlying false elliptic curve $A$, let $(2\pi idz)(\mathbf{y}) \in \omega(\mathbf{y}) = \Omega_{A/\mathbb{C}}$ denote the generator induced by the \'{e}taleness of $\bigsqcup_{\pi_0(\mathbb{Y}_{\infty}(\mathbb{C})^{\mathrm{an}})}\mathcal{H}^+ \rightarrow \mathbb{Y}_{\infty}(\mathbb{C})^{\mathrm{an}}$. By abuse of notation, we will write 
$$(2\pi idz)(y) = (2\pi idz)(\mathbf{y}).$$ 
\item Define $\Omega_p(y) \in \mathbb{C}_p^{\times}$ and $\Omega_{\infty}(y) \in \mathbb{C}^{\times}$ by
\begin{equation}\label{defineomegaperiods}\tilde{w}_{\mathrm{can}}(y) = \Omega_p(y) \cdot w_0, \hspace{1cm} (2\pi idz)(y) =  \Omega_{\infty}(y) \cdot w_0
\end{equation}
where $\tilde{w}_{\mathrm{can}}$ is as in (\ref{tildewcan}). Let $\frak{d}_k^j$ be as in (\ref{complexMSoperator}). 
\end{enumerate}

Then we have the following equality of elements of $\overline{\mathbb{Q}}$ for all $j \ge 0$:
\begin{equation}\label{compareMSvalues}\Omega_p(y)^{k+2j}\cdot \theta_t \circ \theta_q \circ d_k^jF(y) = \Omega_{\infty}(y)^{k+2j} \cdot\frak{d}_k^jG(y).
\end{equation}
%Note that since $w' \in \omega^{\otimes k}(\mathbb{Y})$, then the value $G(\tau)$ does not depend on the representative $\tau$ of $\mathbf{y}^{\mathrm{an}}$. 
\end{theorem}

\begin{remark}Note that since $n(\epsilon_0) \ge 1$ (see (\ref{nepsilondefinition}) and Definition \ref{Udefinition}), then 
$$U = \mathcal{Y}^{\mathrm{Ig}}(\epsilon_0) \subset \mathcal{Y}^{\mathrm{Ig}}(p^{n(\epsilon_0)-1}\epsilon_0) \overset{(\ref{gisomorphism})}{=} \mathcal{Y}^{\mathrm{Ig}}(\epsilon_0) \cdot g^{n(\epsilon_0)-1} = U_{n(\epsilon_0)-1}.$$
Thus Theorem \ref{CMcoincidetheorem} (2) applies to any CM point $\mathbf{y}$ with corresponding $y \in \mathcal{Y}^{\mathrm{Ig}}(\epsilon_0) = U$. On a first reading, one may also substitute $U_{n(\epsilon_0) -1}$ with $U$ in both the statement and the proof of Theorem \ref{CMcoincidetheorem} with no further change to either. 

\end{remark}

\begin{proof}[Proof of Theorem \ref{CMcoincidetheorem}]
%Recall the level $\Gamma \subset D^{\times}(\mathbb{A}_{\mathbb{Q}}^{(\infty)})$ from Convention \ref{Yconvention}. Let $\Gamma_p$ denote its component at $p$. 
The CM point $\mathbf{y}$ has an underlying elliptic curve  (resp. false elliptic curve) $A/\overline{\mathbb{Q}}$ which has endomorphism ring containing some order $\mathcal{O} \subset \mathcal{O}_K$ strictly containing $\mathbb{Z}$. Let 
$$[\cdot ]_A : \mathcal{O} \rightarrow \mathrm{End}(A/\overline{\mathbb{Q}})$$
denote the CM action restricted to $\mathcal{O}$.  %Let $f \in \mathbb{Z}_{\ge 1}$ be the conductor of $\mathcal{O}$, and write $f = f_0p^{r_0}$ where $(f_0,p) = 1$. Then $\mathcal{O}$ has a unique suborder $\mathcal{O}(r)$ of conductor $f_0p^r$ for every $r \in \mathbb{Z}_{\ge r_0}$.  
%By the assumption that $p$ is inert or ramified in $K$, the order $\mathcal{O}_p := \mathcal{O}\otimes_{\mathbb{Z}} \mathbb{Z}_p \subset \mathcal{O}_{K_p}$ strictly contains $\mathbb{Z}_p$.
% and the same is true for the order $\mathcal{O}(r)_p := \mathcal{O}(r) \otimes_{\mathbb{Z}_p}\mathbb{Z}_p \subset \mathcal{O}_{K_p}$. 
 Recall that by the theory of complex multiplication, the action $[\alpha]_A^*$ for $\alpha \in \mathcal{O}$ (where $[\alpha]_A^*$ is the pullback by $[\alpha]_A \in \mathrm{End}(A/\overline{\mathbb{Q}})$) on 
$$H^{1,0} = H^0(A,\Omega_A) \subset H_{\mathrm{dR}}^1(A/\overline{\mathbb{Q}})$$ is through multiplication by $\alpha$, and the action of $[\alpha]_A^*$ on 
$$H^{0,1} = H^1(A,\mathcal{O}_A) \subset H_{\mathrm{dR}}^1(A/\overline{\mathbb{Q}})$$
is through multiplication by $\overline{\alpha}$, where $x \mapsto \overline{x}$ is the nontrivial element of $\mathrm{Gal}(K/\mathbb{Q}) \overset{(\ref{fixembeddings})}{=} \mathrm{Gal}(K_p/\mathbb{Q}_p) \cong \mathbb{Z}/2$. (Here, the equality of Galois groups holds because $p$ is inert or ramified in $K/\mathbb{Q}$.) Thus we have an eigendecomposition 
\begin{equation}\label{CMeigendecomposition}H_{\mathrm{dR}}^1(A/\overline{\mathbb{Q}}) = H^{1,0} \oplus H^{0,1}
\end{equation} 
for the $\mathcal{O}$-action, where each $H^{1,0}$ and $H^{0,1}$ is an $\overline{\mathbb{Q}}$-vector space. %This is evidently a splitting of the Hodge filtration. 

Now let us recall our $p$-adic Hodge decomposition and study its specialization to the CM point $\mathbf{y}$. Recall $\mathbb{B} \overset{(\ref{Bdefinition})}{=} \mathbb{B}_{\mathrm{dR},\mathcal{V}_x}(\mathcal{V}_x)\llbracket q_{\mathrm{dR}}-1\rrbracket$. Let 
$$\mathcal{B} = \mathbb{B}_{\mathrm{dR},U_{n(\epsilon_0)-1}}(U_{n(\epsilon_0)-1})\llbracket q_{\mathrm{dR}}-1\rrbracket,$$
so that there is a natural map $\mathbb{B}\rightarrow \mathcal{B}$ induced by restriction to $U_{n(\epsilon_0)-1} \subset \mathcal{V}_x$. Tensoring the Hodge decomposition (\ref{HTdecomposition}) with $\otimes_{\mathbb{B}}\mathcal{B}$, we get
\begin{equation}\label{abovedecomposition1}H_{\mathrm{dR}}^1(\mathcal{E}) \otimes_{\mathcal{O}_Y}\mathcal{B} \cong \left(\omega \otimes_{\mathcal{O}_Y} \mathcal{B}\right) \oplus \left(\omega^{-1}\otimes_{\mathcal{O}_Y}\mathcal{B}\right).
\end{equation}
Recall the map
$$\theta_t \circ \theta_q : \mathbb{B} \twoheadrightarrow \hat{\mathcal{O}}_{\mathcal{V}_x}(\mathcal{V}_x)$$
from Definition \ref{evaluationdefinition}; we thus get an induced map 
$$\theta_t \circ \theta_q : \mathcal{B} \twoheadrightarrow \hat{\mathcal{O}}_{U_{n(\epsilon_0)-1}}(U_{n(\epsilon_0)-1})$$ 
by restriction to $U_{n(\epsilon_0)-1} \subset \mathcal{V}_x$. Tensoring (\ref{abovedecomposition1}) with $\otimes_{\mathbb{B},\theta_t \circ \theta_q} \hat{\mathcal{O}}_{U_{n(\epsilon_0)-1}}(U_{n(\epsilon_0)-1})$, we get
\begin{equation}\label{abovedecomposition2}H_{\mathrm{dR}}^1(\mathcal{E})\otimes_{\mathcal{O}_Y}\hat{\mathcal{O}}_{U_{n(\epsilon_0)-1}}(U_{n(\epsilon_0)-1})  \cong \left(\omega\otimes_{\mathcal{O}_Y} \hat{\mathcal{O}}_{U_{n(\epsilon_0)-1}}(U_{n(\epsilon_0)-1})\right) \oplus \left(\omega^{-1} \otimes_{\mathcal{O}_Y}\hat{\mathcal{O}}_{U_{n(\epsilon_0)-1}}(U_{n(\epsilon_0)-1})\right).
\end{equation}
Then $\theta_t \circ \theta_q : \mathcal{B} \twoheadrightarrow \hat{\mathcal{O}}_{U_{n(\epsilon_0)-1}}(U_{n(\epsilon_0)-1})$ induces a map 
\begin{equation}\label{Hodgethetatthetaq}(\theta_t \circ \theta_q) : H_{\mathrm{dR}}^1(\mathcal{E})|_{U_{n(\epsilon_0)-1}} \otimes_{\mathcal{O}_{U_{n(\epsilon_0)-1}}}\mathcal{B} \twoheadrightarrow H_{\mathrm{dR}}^1(\mathcal{E})\otimes_{\mathcal{O}_Y}\hat{\mathcal{O}}_{U_{n(\epsilon_0)-1}}(U_{n(\epsilon_0)-1})
\end{equation}
that maps the decomposition (\ref{abovedecomposition1}) to the decomposition (\ref{abovedecomposition2}). Since $U_{n(\epsilon_0)-1} = \mathcal{Y}^{\mathrm{Ig}}(p^{n(\epsilon_0)-1}\epsilon_0)$ (for this last equality, see (2) in the statement of this Theorem) and $n(p^{n(\epsilon_0)-1}\epsilon_0) = n(\epsilon_0) - (n(\epsilon_0)-1) = 1$, we have that any $\gamma \in \Gamma_{0,p}(p)$ satisfies 
$$U_{n(\epsilon_0)-1} \cdot \gamma = \mathcal{Y}^{\mathrm{Ig}}(p^{n(\epsilon_0)-1}\epsilon_0) \cdot \gamma \overset{(\ref{IgusaGaloisgroup})}{=} \mathcal{Y}^{\mathrm{Ig}}(p^{n(\epsilon_0)-1}\epsilon_0) = U_{n(\epsilon_0)-1}.$$
Hence the assumptions of Lemma \ref{thetaXGamma0commutelemma} are satisfied, and so by (\ref{gammathetacommute}) we have 
\begin{equation}\label{Hodgethetacommute}(\theta_t \circ \theta_q) \circ \gamma^* = \gamma^* \circ (\theta_t\circ \theta_q)
\end{equation}
for all $\gamma \in \Gamma_{0,p}(p)$, where $(\theta_t \circ \theta_q)$ is as in (\ref{Hodgethetatthetaq}). 

Now recall that $\mathbf{y}$ is our CM point and that $y$ is the image of its $p$-adic analytification $\mathbf{y}^{\mathrm{an}} \in \mathcal{Y}_{\infty}$ under $\mathcal{Y}_{\infty} \rightarrow Y_{\infty}$ from (\ref{analyticprojection}). Let $\mathcal{B}(y)$ denote the specialization of $\mathcal{B}$ to $y \in U_{n(\epsilon_0)-1}(\mathbb{C}_p,\mathcal{O}_{\mathbb{C}_p})$. Specializing (\ref{abovedecomposition1}) to $y$, we get 
\begin{equation}\label{abovedecomposition1'}H_{\mathrm{dR}}^1(\mathcal{E})\otimes_{\mathcal{O}_Y}\mathcal{B}(y) \cong \left(\omega \otimes_{\mathcal{O}_Y} \mathcal{B}(y)\right) \oplus \left(\omega^{-1} \otimes_{\mathcal{O}_Y}\mathcal{B}(y)\right),
\end{equation}
and specializing (\ref{abovedecomposition2}) to $y$ we get
\begin{equation}\label{abovedecomposition2'}H_{\mathrm{dR}}^1(\mathcal{E})\otimes_{\mathcal{O}_Y}\hat{\mathcal{O}}_{U_{n(\epsilon_0)-1}}(y)  \cong \left(\omega\otimes_{\mathcal{O}_Y} \hat{\mathcal{O}}_{U_{n(\epsilon_0)-1}}(y)\right) \oplus \left(\omega^{-1} \otimes_{\mathcal{O}_Y}\hat{\mathcal{O}}_{U_{n(\epsilon_0)-1}}(y)\right). 
\end{equation}
%Specializing (\ref{Hodgethetatthetaq}) to $y$ we get a map 
%\begin{equation}\label{Hodgethetatthetaq2}(\theta_t \circ \theta_q)(y) : H_{\mathrm{dR}}^1(\mathcal{E})|_{U_{n(\epsilon_0)-1}} \otimes_{\mathcal{O}_{U_{n(\epsilon_0)-1}}}\mathcal{B}(y) \twoheadrightarrow H_{\mathrm{dR}}^1(\mathcal{E})\otimes_{\mathcal{O}_Y}\hat{\mathcal{O}}_{U_{n(\epsilon_0)-1}}(y)
%\end{equation}
%that sends the decomposition (\ref{abovedecomposition1'}) to (\ref{abovedecomposition2'}). 
Recall that the first factor of (\ref{HTdecomposition}) is given by the Hodge filtration and the second factor by the Hodge-Tate filtration. Thus the first factor of (\ref{abovedecomposition1'}) is the Hodge filtration 
$$\Omega_{A/\mathbb{C}_p} \subset H_{\mathrm{dR}}^1(A/\mathbb{C}_p)$$
of the CM (false) elliptic curve $A$ underlying the point $y$ and the second factor is the Hodge-Tate filtration 
$$\mathrm{Lie}(A/\mathbb{C}_p) \subset H_{\text{\'{e}t}}^1(A/\mathbb{C}_p,\mathbb{Z}_p) \otimes_{\mathbb{Z}_p}\mathbb{C}_p(-1).$$
Therefore, by the definition of the CM action, the action of $\gamma^*$ for any $\gamma \in 1 + p\mathcal{O} \overset{(\ref{idefinition})}{\subset} \Gamma_{1,p}(p) \subset \Gamma_{0,p}(p)$ on the first factor of (\ref{abovedecomposition1'}) is through multiplication by $\gamma$ and the action on the second factor is through multiplication by $\overline{\gamma}$. Thus for any $w \in \omega\otimes_{\mathcal{O}_Y}\mathcal{B}(y)$ and any $\eta' \in \omega^{-1}\otimes_{\mathcal{O}_Y}\mathcal{B}(y)$ we have 
\begin{equation}\label{CMactions}[\gamma]_A^*w' = \gamma \cdot w', \hspace{1cm} [\gamma]_A^*\eta' = \overline{\gamma} \cdot \eta',
\end{equation}
where here ``$\gamma \cdot$'' denotes usual multiplication by $\gamma \in \mathcal{O}_{K_p}$.

Now let 
$$w_{\mathrm{can}} \in \omega \otimes_{\mathcal{O}_Y}\mathbb{B}_{\mathrm{dR},\mathcal{V}_x}^+(\mathcal{V}_x)\llbracket q_{\mathrm{dR}}-1\rrbracket$$ be the generator from (\ref{zqw2}); then 
$$\tilde{w} := w_{\mathrm{can}}|_{U_{n(\epsilon_0)-1}} \in \omega \otimes_{\mathcal{O}_Y}\mathbb{B}_{\mathrm{dR},U_{n(\epsilon_0)-1}}^+(U_{n(\epsilon_0)-1})\llbracket q_{\mathrm{dR}}-1\rrbracket$$
is a generator. Note that for any open $W \subset \mathcal{V}_x$, an element $f \in \mathbb{B}_{\mathrm{dR},\mathcal{V}_x}^+(W)\llbracket q_{\mathrm{dR}}-1\rrbracket$ is invertible if any only if $\theta_t\circ \theta_q(f) \in \hat{\mathcal{O}}_{\mathcal{V}_x}(W)^{\times}$ (see Definition \ref{evaluationdefinition}). Thus 
$$(\theta_t \circ \theta_q)(w_{\mathrm{can}}) \in \omega \otimes_{\mathcal{O}_Y}\hat{\mathcal{O}}_{\mathcal{V}_x}(\mathcal{V}_x)$$ 
and 
\begin{equation}\label{provisionalw}w := (\theta_t \circ \theta_q)(\tilde{w}) \in \omega \otimes_{\mathcal{O}_Y} \hat{\mathcal{O}}_{U_{n(\epsilon_0)-1}}(U_{n(\epsilon_0)-1})
\end{equation}
are generators. 
%Extending by scalars along the natural inclusions 
%$$\mathbb{B}_{\mathrm{dR},\mathcal{V}_x}^+(\mathcal{V}_x)\llbracket q_{\mathrm{dR}}-1\rrbracket \rightarrow \mathbb{B}_{\mathrm{dR},\mathcal{V}_x}(\mathcal{V}_x)\llbracket q_{\mathrm{dR}}-1\rrbracket =: \mathbb{B},$$
%$$\mathbb{B}_{\mathrm{dR},U_{n(\epsilon_0)-1}}^+(U_{n(\epsilon_0)-1})\llbracket q_{\mathrm{dR}}-1\rrbracket \rightarrow \mathbb{B}_{\mathrm{dR},U_{n(\epsilon_0)-1}}(U_{n(\epsilon_0)-1})\llbracket q_{\mathrm{dR}}-1\rrbracket =: \mathcal{B},$$
%we see that 
%$$w_{\mathrm{can}} \in \omega \otimes_{\mathcal{O}_Y}\mathbb{B}$$
%and 
%$$\tilde{w} \in \omega \otimes_{\mathcal{O}_Y} \mathcal{B}$$
%are also generators. In particular, the specialization 
%$$\tilde{w}(y) \in \omega \otimes_{\mathcal{O}_Y}\mathcal{B}(y)$$
%is a generator. 
%%Moreover by the end of the previous paragraph, we see that 
%%$$(\theta_t \circ \theta_q)(w_{\mathrm{can}}) \in \omega \otimes_{\mathcal{O}_Y} \hat{\mathcal{O}}_{\mathcal{V}_x}(\mathcal{V}_x)$$
%%and 
%%\begin{equation}\label{provisionalw}w := (\theta_t \circ \theta_q)(\tilde{w}) \in \omega \otimes_{\mathcal{O}_Y} \hat{\mathcal{O}}_{U_{n(\epsilon_0)-1}}(U_{n(\epsilon_0)-1})
%%\end{equation}
%%are generators. 
%Moreover, recall (\ref{provisionalw}) is a generator (i.e. nowhere vanishing) and so 
 Moreover, the specializations to $y \in U_{n(\epsilon_0)-1}(\mathbb{C}_p,\mathcal{O}_{\mathbb{C}_p})$
$$\tilde{w}(y)  \in \omega \otimes_{\mathcal{O}_Y}\mathcal{B}(y), \hspace{1cm} w(y) \in \omega \otimes_{\mathcal{O}_Y} \hat{\mathcal{O}}_{U_{n(\epsilon_0)-1}}(y)$$
are generators.

In all, we have 
\begin{align*}[\gamma]_A^*(w(y)) &= (\gamma^*w)(y) \overset{(\ref{provisionalw})}{=} \left(\gamma^*(\theta_t\circ \theta_q)(\tilde{w})\right)(y) \\
&\overset{(\ref{Hodgethetacommute})}{=} \left((\theta_t\circ \theta_q)\left(\gamma^*\tilde{w}\right)\right)(y) = (\theta_t\circ \theta_q)\left(\gamma^*\tilde{w}(y)\right) \\&\overset{(\ref{CMactions}), \; w' = \tilde{w}(y)}{=} (\theta_t\circ \theta_q)\left(\gamma \cdot \tilde{w}(y)\right) = \gamma \cdot (\theta_t\circ \theta_q)(\tilde{w}(y)) \\
&= \gamma \cdot (\theta_t\circ \theta_q)(\tilde{w})(y) \overset{(\ref{provisionalw})}{=} \gamma \cdot w(y).
\end{align*}
Let $\eta \in \omega^{-1} \otimes_{\mathcal{O}_Y}\hat{\mathcal{O}}_{U_{n(\epsilon_0)-1}}(U_{n(\epsilon_0)-1})$ be the Poincar\'{e} dual of $w(y)$. An entirely analogous calculation using (\ref{CMactions}) with $\eta' \in \omega^{-1}\otimes_{\mathcal{O}_Y}\mathcal{B}(y)$ equal to the Poincar\'{e} dual of $\tilde{w}(y)$ shows
$$[\gamma]_A^*\eta' = \overline{\gamma} \cdot \eta'.$$
%for any $w' \in \omega\otimes_{\mathcal{O}_Y}\mathcal{B}(y)$ we have 
Since $w(y)$ is a generator of the first factor of (\ref{abovedecomposition2'}) and $\eta'$ is a generator of the second factor of (\ref{abovedecomposition2'}), we thus see that $[\gamma]_A^*$ acts on the first factor by multiplication by $\gamma$ and on the second factor by multiplication by $\overline{\gamma}$. 
 Pick any $\gamma \in 1 + p\mathcal{O}$ with $\gamma \neq \overline{\gamma}$ (which exists since $\mathcal{O}$ strictly contains $\mathbb{Z}$). By the first and previous paragraphs, we see that $[\gamma]_A^*$ acts on the first factors of both (\ref{CMeigendecomposition}) and (\ref{abovedecomposition2'}) through multiplication by $\gamma$, and acts on the second factors through multiplication by $\overline{\gamma}$. %Hence the decomposition
% \begin{align*}H_{\mathrm{dR}}^1(A) \otimes_{\overline{\mathbb{Q}}} \mathbb{C}_p &= \theta_t\circ \theta_q\left(H_{\mathrm{dR}}^1(A) \otimes_{\overline{\mathbb{Q}}}\mathcal{B}(y)\right) = \theta_t\circ\theta_q(H_{\mathrm{dR}}(\mathcal{E}) \otimes_{\mathcal{O}_Y}\mathcal{B}(y)) \\
% &\cong \theta_t\circ \theta_q\left(\omega \otimes_{\mathcal{O}_Y}\mathcal{B}(y)\right) \oplus \theta_t \circ \theta_q\left(\omega^{-1} \otimes_{\mathcal{O}_Y}\mathcal{B}(y)\right) = H^{1,0} \otimes_{\overline{\mathbb{Q}}}\mathbb{C}_p \oplus H^{0,1} \otimes_{\overline{\mathbb{Q}}}\mathbb{C}_p
% \end{align*}
Hence by uniqueness of the eigendecomposition for the action of $[\gamma]_A^*$, (\ref{abovedecomposition2'}) is equal to (\ref{CMeigendecomposition}) tensored with $\otimes_{\overline{\mathbb{Q}}_p}\hat{\mathcal{O}}_{U_{n(\epsilon_0)-1}}(y) = \otimes_{\overline{\mathbb{Q}}}\mathbb{C}_p$. The specialization to $y$ of the map $(\theta_t \circ \theta_q) \circ \mathrm{split}$ 
$$(\theta_t \circ \theta_q) \circ \mathrm{split}(y) : H_{\mathrm{dR}}^1(A/\mathbb{C}_p) = H_{\mathrm{dR}}^1(\mathcal{E}) \otimes_{\mathcal{O}_Y}\hat{\mathcal{O}}_{U_{n(\epsilon_0)-1}}(y) \rightarrow \omega \otimes_{\mathcal{O}_Y}\hat{\mathcal{O}}_{U_{n(\epsilon_0)-1}}(y)  = H^{1,0} \otimes_{\overline{\mathbb{Q}}}\mathbb{C}_p,$$
where $\mathrm{split}$ is as in (\ref{split}), is thus equal to the projection $H_{\mathrm{dR}}^1(A/\overline{\mathbb{Q}}_p) \rightarrow H^{1,0}$ given by (\ref{CMeigendecomposition}) tensored with $\otimes_{\overline{\mathbb{Q}}}\mathbb{C}_p$. 

Similarly, it is known that the specialization to a CM point of the real analytic Hodge decomposition coincides with the CM eigendecomposition (\ref{CMeigendecomposition}) after tensoring the latter with $\otimes_{\overline{\mathbb{Q}}}\mathbb{C}$, see \cite[5.1.27]{KatzCM} (note loc. cit. also proves the analogous $p$-adic statement in the case when $p$ is split in $K$). Thus the specializations to the CM point $\mathbf{y}$ of the $p$-adic splitting $\theta_t \circ \theta_q \circ \mathrm{split}$ defining $\theta_t \circ \theta_q \circ d_k^j$ and of the real analytic Hodge splitting defining $\frak{d}_k^j$ are equal to the splitting $H_{\mathrm{dR}}^1(A) \rightarrow H^{1,0}$ tensored with $\otimes_{\overline{\mathbb{Q}}}\mathbb{C}_p$ and tensored with $\otimes_{\overline{\mathbb{Q}}}\mathbb{C}$, respectively. Therefore, the same argument as in \cite[Proposition 1.12]{BDP} proves (\ref{compareMSvalues}). 

\end{proof}

\subsection{$p$-depletion of modular forms}\label{depletionsection}

\begin{definition}Define
\begin{equation}\label{UVdefinition}V_p := g \overset{(\ref{gdefinition})}{=} \left(\begin{array}{ccc} 1 & 0\\
0 & p\\
\end{array}\right) \in GL_2(\mathbb{Q}_p), \hspace{1cm} U_p := \frac{1}{p}\sum_{j= 0}^{p-1}\left(\begin{array}{ccc} 1 & j/p\\
0 & 1/p\\
\end{array}\right) \in \mathbb{Q}[GL_2(\mathbb{Q}_p)].
\end{equation}
These act on the right of $Y_{\infty}$, and are essentially the usual Atkin-Lehner $U/V$ Hecke operators, see \cite[Section 3.8]{BDP} for example. We caution the reader that our $U_p$ is $[p^{-1}]U$ in the notation of loc. cit., and $V_p$ is $[p]V$ in the notation of loc. cit.; thus in our setting, $[p] = \left(\begin{array}{ccc} p & 0\\
0 & p \\
\end{array}\right) \in GL_2(\mathbb{Q}_p)$ and 
$$T_p = [p]U_p + \frac{1}{p}V_p = U + \frac{[p]}{p}V,$$
where $T_p$ is the usual Hecke $T_p$-operator (see loc. cit.). Thus $U_p$ and $V_p$ act on sections of sheaves on $Y_{\infty}$ via pullback. 
\end{definition}

\begin{remark}Note that our convention to view $GL_2(\mathbb{Q}_p)$ as acting on the right (as in \cite{ScholzeTorsion} and \cite{ChojeckiHansenJohansson}) implies that given $\gamma_1,\gamma_2 \in GL_2(\mathbb{Q}_p)$ and a section $F \in \mathcal{F}(W)$, where $\mathcal{F}$ is a sheaf on $Y_{\text{pro\'{e}t}}/W$ and $W \rightarrow Y_{\infty}$ is a pro\'{e}tale open, such that $\gamma_i^* : \mathcal{F}(W) \rightarrow \mathcal{F}(W)$ for $i = 1,2$, we have 
$$(\gamma_1\gamma_2)^*f = \gamma_1^*\gamma_2^*f.$$
\end{remark}

Recall that $Y(\epsilon_0)$ is the image of $U = \mathcal{Y}^{\mathrm{Ig}}(\epsilon_0)$ (Definition \ref{Udefinition}) under the map $Y_{\infty} \rightarrow Y$ and that $\mathrm{Gal}(U/Y(\epsilon_0)) = \Gamma_{0,p}(p^{n(\epsilon_0)})$. %Let $V_0 \subset Y$ be any open set containing $Y(\epsilon_0)$. 

\begin{definition}\label{flatdefinition}\begin{enumerate}
\item Given a Hecke eigenform $w' \in \omega^{\otimes k}(Y(\epsilon_0))$, we call
\begin{equation}\label{w'flat}w'^{\flat} := (U_p^*V_p^* - V_p^*U_p^*)w' \in \omega^{\otimes k}\otimes_{\mathcal{O}_Y}\mathbb{B}_{\mathrm{dR},Y}^+(Y_{\infty})\llbracket X/t\rrbracket
\end{equation}
the \emph{$p$-depletion of $w$'} (see \cite[Section 3.8]{BDP}). One calculates
\begin{equation}\label{UV}U_pV_p = \frac{1}{p}\sum_{j = 0}^{p-1}\left(\begin{array}{ccc} 1 & j\\
0 & 1\\
\end{array}\right) \in \mathbb{Q}[\mathrm{Gal}(U/Y(\epsilon_0))].
\end{equation}
Since $w' \in \omega^{\otimes k}(Y(\epsilon_0))$ and $\mathrm{Gal}(U/Y(\epsilon_0))$ acts trivially on $Y(\epsilon_0)$, we have
$$U_p^*V_p^*w' = w'.$$
Hence
$$w'^{\flat} = (1-V_p^*U_p^*)w'.$$
\item Recall (\ref{gammajpnm}),
$$\gamma_{j/p^n,m} := \left(\begin{array}{ccc} 1 & j/p^n\\
  0 & 1/p^m\\
  \end{array}\right).$$
One calculates
\begin{equation}\label{VU}V_pU_p = \frac{1}{p}\sum_{j = 0}^{p-1}\left(\begin{array}{ccc} 1 & j/p\\
0 & 1\\
\end{array}\right) = \frac{1}{p}\sum_{j = 0}^{p-1}\gamma_{j/p,0} \in \mathbb{Q}[GL_2(\mathbb{Q}_p)].
\end{equation}
\item Suppose now that $D = M_2(\mathbb{Q})$, and let $\overline{Y}(p^r)$ denote the compactification of $Y(p^r) = Y(\Gamma \cap \Gamma(p^r))$ (recall $\Gamma = \Gamma(N)$ from Convention \ref{Yconvention}). If moreover $w' \in \omega^{\otimes k}(\overline{Y}(p^r))$ for any $r \in \mathbb{Z}_{\ge 0}$ and $w'$ has $q$-expansion 
$$\sum_{n = 0}^{\infty}a_nq^n,$$
then $w'^{\flat}$ has $q$-expansion
$$\sum_{n = 0, p\nmid n}^{\infty}a_nq^n.$$
See \cite[(3.8.4)]{BDP}.
\end{enumerate}
\end{definition}

%Define
%\begin{equation}\label{w'flat}w'^{\flat}(q_{\mathrm{dR}}) := (1 - V_p^*U_p^*)\left(w'(q_{\mathrm{dR}})\right).
%\end{equation}
%Theorem \ref{gammajpnacttheorem} with $m = 0, n = 1$ (recall that since $\epsilon_0 < p/(p+1)$, $n(\epsilon_0) \ge 1$) has the following immediate corollary. 
%\begin{corollary}We have the following equality in $\mathcal{O}_{\hat{U}^+}(\hat{U}^+)(\!(t)\!)\llbracket q_{\mathrm{dR}}^{1/p^m}-1\rrbracket$:
%\begin{equation}\label{flatformula}w'^{\flat}(q_{\mathrm{dR}}) = w'(q_{\mathrm{dR}}) - \frac{1}{p}\sum_{j = 0}^{p-1}w'([(\zeta_p^j,\zeta_{p^2}^j,\ldots)]q_{\mathrm{dR}}).
%\end{equation}

%\begin{proposition}Recall $U = \mathcal{Y}^{\mathrm{Ig}}(\epsilon_0)$ from Definition \ref{Udefinition} and suppose $\gamma \in GL_2(\mathbb{Q}_p)$ satisfies $U \cdot \gamma \subset U$. Consider $w_{\mathrm{can}} \in \omega(U)\otimes_{\mathcal{O}_Y}\mathbb{B}_{\mathrm{dR},\mathcal{V}_x}^+(U)\llbracket X/t\rrbracket$ from Proposition \ref{wcangeneratorproposition3}. Then we have 
%\begin{equation}\label{wcaninvariant}\gamma^*w_{\mathrm{can}} = w_{\mathrm{can}}.
%\end{equation}
%\end{proposition}

%\begin{proof}Recall from (\ref{iwcan}) that $i_{\mathrm{dR}}(w_{\mathrm{can}}) = -z_{\mathrm{dR}}e_1 + e_2$. Thus (\ref{wcaninvariant}) follows if one can show that $z_{\mathrm{dR}}$
%\end{proof}

 Let $\mathcal{E}\rightarrow Y_{\infty}$ denote the universal object (universal (false) elliptic curve with $\mathcal{O}_D$-endomorphism structure and $\Gamma(Np^{\infty})$-level structure, see Convention \ref{Yconvention}). Let $f : \mathcal{E} \rightarrow \mathcal{E}/C_j$ denote the kernel of the isogeny underlying $\gamma_{j/p^n,m}$, let $(e_1,e_2)$ denote the $\Gamma(p^{\infty})$-level structure on $\mathcal{E}$, and let $(e_1',e_2') = (e_1,e_2) \cdot \gamma_{j/p^n,m}$. %Let $f^{\vee}: \mathcal{E}/C_j \rightarrow \mathcal{E}$ be the dual isogeny of $f$. 
We get an induced morphism
$$f_* : T_p\mathcal{E} \otimes_{\mathbb{Z}_p}\mathbb{Q}_p \xrightarrow{\sim} T_p(\mathcal{E}/C_j).$$ 
Let 
$$i_j : \mathcal{E}/C_j \rightarrow \mathcal{E}$$
denote the classifying map of $\mathcal{E}/C_j$ (which exists by the universal property of $\mathcal{E} \rightarrow Y_{\infty}$), and let 
$$i_{j,*} : T_p(\mathcal{E}/C_j) \rightarrow T_p\mathcal{E}$$
denote the induced map on Tate modules.
%Define 
%$$\mathrm{HT}_{\mathcal{E}} : T_p\mathcal{E} \rightarrow \pi_*\Omega_{\mathcal{E}/Y_{\infty}}, \hspace{1cm} \mathrm{HT}(e) = e^*\frac{dT}{1+T},$$
%which is the Hodge-Tate map of \cite{ChojeckiHansenJohansson}. Similarly define $\mathrm{HT}_{\mathcal{E}/C_j} : T_p(\mathcal{E}/C_j) \rightarrow \pi_*\Omega_{(\mathcal{E}/C_j)/Y_{\infty}}$. 
From the definition of the $GL_2(\mathbb{Q}_p)$-action we have a diagram 
\begin{equation}\label{Udiagram}
\begin{tikzcd}[column sep = large]
\hat{\mathbb{Z}}_{p,Y_{\infty}}^{\oplus 2} \arrow{r}{(e_1,e_2)} \arrow{d}{\gamma_{j/p^n,m}} & T_p\mathcal{E}  \arrow{d}{f_*}  \\
  \hat{\mathbb{Z}}_{p,Y_{\infty}}^{\oplus 2} \arrow{r}{(e_1',e_2')} \arrow{d}{\mathrm{id}}  & T_p(\mathcal{E}/C_j) \arrow{d}{i_{j,*}} \\
  \hat{\mathbb{Z}}_{p,Y_{\infty}}^{\oplus 2} \arrow{r}{(e_1,e_2)}  &T_p\mathcal{E}
  \end{tikzcd}.
  \end{equation}
%  \hat{\mathbb{Z}}_{p,Y_{\infty}}^{\oplus 2} \otimes_{\mathbb{Z}_p}\mathbb{Q}_p\arrow{r}{(e_1,e_2)} & T_p\mathcal{E}\otimes_{\mathbb{Z}_p}\mathbb{Q}_p
Here the left two vertical arrows are given by the action of elements of $GL_2(\mathbb{Q}_p)$ acting on $\hat{\mathbb{Z}}_{p,Y_{\infty}}^{\oplus 2}$ via left multiplication of matrices on column vectors. 

Recall $U = \mathcal{Y}^{\mathrm{Ig}}(\epsilon_0) \subset \mathcal{V}_x$ from Definition \ref{Udefinition}. Now move to the localized site $Y_{\text{pro\'{e}t}}/U$ ($U = \mathcal{Y}^{\mathrm{Ig}}(\epsilon_0)$ as in Definition \ref{Udefinition}) and assume that $n \le n(\epsilon_0) \le m$; then $U \cdot \gamma_{j/p^n,m} \subset U$ by (\ref{UU'gamma}). The commutativity of (\ref{Udiagram}) thus implies for $i = 1,2$, 
\begin{equation}\label{gammae}\gamma_{j/p^n,m}^*(e_i) = \gamma_{j/p^n,m}(e_i) = \begin{cases} e_1 & i = 1\\
\frac{j}{p^n}e_1 + \frac{1}{p^m}e_2 & i = 2\\
\end{cases}.
\end{equation}
%Specifically, the left-hand side of this equality is the image of $(0,1) \in \hat{\mathbb{Z}}_{p,Y_{\infty}}^{\oplus 2}$ under the composition of the top horizontal arrows with the right vertical arrows of (\ref{Udiagram}), the middle term of the equality is the image of $(0,1) \in \hat{\mathbb{Z}}_{p,Y_{\infty}}^{\oplus 2}$ under the composition of the top left vertical arrow with the middle horizontal arrows and then with the bottom right vertical arrow, and the right-hand side of the equality is the image of $(0,1) \in \hat{\mathbb{Z}}_{p,Y_{\infty}}^{\oplus 2}$ under the composition of the left vertical arrows with the bottom horizontal arrows. 
%Since $w_{\mathrm{can}} \in  \omega \otimes_{\mathcal{O}_Y}\mathbb{B}_{\mathrm{dR},\mathcal{V}_x}^+(U)\llbracket X/t\rrbracket$ is a generator by Proposition \ref{wcangeneratorproposition3}. Since $\gamma_{j/p^n,m}$ acts on $\mathcal{E}$ via an isogeny, $\gamma_{j/p^n,m}^*$ preserves the Hodge filtration (\ref{Hodgefiltration}). Hence,
%$$\gamma_{j/p^n,m}^*w_{\mathrm{can}} = \Omega \cdot w_{\mathrm{can}}$$
%for some $\Omega \in \mathbb{B}_{\mathrm{dR},\mathcal{V}_x}^+(U)\llbracket X/t\rrbracket^{\times}$. We will show that in fact $\Omega = 1$. 
We thus have
\begin{align*}i_{\mathrm{dR}}(\gamma_{j/p^n,m}^*w_{\mathrm{can}}) &= \gamma_{j/p^n,m}^*(i_{\mathrm{dR}}(w_{\mathrm{can}})) \overset{(\ref{iwcan})}{=} \gamma_{j/p^n,m}^*(-z_{\mathrm{dR}} e_1 + e_2) \\&\hspace{-1cm}\overset{(\ref{modulartransformationidentity}),(\ref{gammae})}{=} -\left(\frac{z_{\mathrm{dR}}}{p^m} + \frac{j}{p^n}\right)e_1 + \left(\frac{j}{p^n}e_1 + \frac{1}{p^m}e_2\right) = \frac{1}{p^m}\left(-z_{\mathrm{dR}}e_1 + e_2\right)  \overset{(\ref{iwcan})}{=}i_{\mathrm{dR}}\left(\frac{1}{p^m}w_{\mathrm{can}}\right).
\end{align*}
%\cdot i_{\mathrm{dR}}(w_{\mathrm{can}}) = i_{\mathrm{dR}}(\Omega \cdot w_{\mathrm{can}}) = i_{\mathrm{dR}}(\gamma_{j/p^n,m}^*w_{\mathrm{can}}) = \gamma_{j/p^n,m}^*i_{\mathrm{dR}}(w_{\mathrm{can}}) = \gamma_{j/p^n,m}^*(-z_{\mathrm{dR}}e_1 + e_2).$$
%However, from (\ref{gammae}), we see that 
%Note that 
%$$\gamma_{j/p^n,m}^* : \omega \otimes_{\mathcal{O}_Y}\mathbb{B}_{\mathrm{dR},\mathcal{V}_x}^+(U)\llbracket X/t\rrbracket \rightarrow \omega \otimes_{\mathcal{O}_Y}\mathbb{B}_{\mathrm{dR},\mathcal{V}_x}^+(U)\llbracket X/t\rrbracket$$
%is a map of free $\mathbb{B}_{\mathrm{dR}}^+(U)\llbracket X/t\rrbracket$-modules of rank 1 and 
%$$\mathrm{HT}_{\mathcal{E}}(e_2) = \pi_{\mathrm{HT}}^*x \overset{(\ref{fraks})}{=} \frak{s} \in \omega\otimes_{\mathcal{O}_Y}\mathbb{B}_{\mathrm{dR},\mathcal{V}_x}^+(U)\llbracket X/t\rrbracket$$
%is nowhere-vanishing by (\ref{fraksgenerator}); here the first equality in the previous displayed equation follows from \cite[Section 2.4]{ChojeckiHansenJohansson} (recalling that $z,\frak{z}$ in loc. cit. are $1/z,1/z_{\mathrm{HT}}$ in our setting). 
Thus, since $i_{\mathrm{dR}}$ is injective, we have 
\begin{equation}\label{gammawcan}\gamma_{j/p^n,m}^*w_{\mathrm{can}} = \frac{1}{p^m}w_{\mathrm{can}}.
\end{equation}

Recall
$$g = \left(\begin{array}{ccc} 1 & 0\\
0 & p\\
\end{array}\right).$$
Let $f^n : \mathcal{E} \rightarrow \mathcal{E}/C_{g^n}$ denote the isogeny underlying $g^n$. %Let $\mathrm{HT}_{\mathcal{E}/C} : T_p(\mathcal{E}/C) \rightarrow \pi_*\Omega_{\mathcal{E}/C}$ denote the Hodge-Tate map from \cite{ChojeckiHansenJohansson}, and let
Let
$$f_* : T_p\mathcal{E} \xrightarrow{\sim} T_p(\mathcal{E}/C_{g^n})$$
be the induced morphism. Let  $(e_1',e_2') = (e_1,e_2)\cdot g^n$, let 
$$i_{g^n} : \mathcal{E}/C_{g^n} \rightarrow \mathcal{E}$$
be the classifying map and let 
$$i_{g^n,*} : T_p\mathcal{E}/C_{g^n} \rightarrow T_p\mathcal{E}$$
be the induced map on Tate modules. From the definition of the $GL_2(\mathbb{Q}_p)$-action (or the diagram in the proof of Lemma 2.11 of op. cit.), we have a diagram
\begin{equation}\label{Vdiagram}
\begin{tikzcd}[column sep = large]
\hat{\mathbb{Z}}_{p,Y_{\infty}}^{\oplus 2} \arrow{r}{(e_1,e_2)} \arrow{d}{g^n} & T_p\mathcal{E}  \arrow{d}{f^n_*} \\
  \hat{\mathbb{Z}}_{p,Y_{\infty}}^{\oplus 2} \arrow{r}{(e_1',e_2')} \arrow{d}{\mathrm{id}} & T_p(\mathcal{E}/C_{g^n})\arrow{d}{i_{g^n,*}} \\
  \hat{\mathbb{Z}}_{p,Y_{\infty}}^{\oplus 2} \arrow{r}{(e_1,e_2)} & T_p\mathcal{E}     \end{tikzcd}.
\end{equation}
Thus we have, for $i = 1,2$,
\begin{equation}\label{ge}(g^n)^*(e_i) = \begin{cases} e_1 & i = 1\\
p^ne_2 & i = 2
\end{cases}.
\end{equation}
Thus
\begin{align*}i_{\mathrm{dR}}((g^n)^*w_{\mathrm{can}}) &= (g^n)^*i_{\mathrm{dR}}(w_{\mathrm{can}}) \overset{(\ref{iwcan})}{=} (g^n)^*(-z_{\mathrm{dR}}e_1 + e_2) \\
&\overset{(\ref{modulartransformationidentity}), (\ref{ge})}{=} -p^nz_{\mathrm{dR}}e_1 + p^ne_2 = p^n(-z_{\mathrm{dR}}e_1 + e_2) = i_{\mathrm{dR}}(p^nw_{\mathrm{can}}).
\end{align*}
Since $i_{\mathrm{dR}}$ is injective, this implies 
\begin{equation}\label{gwcan}(g^n)^*w_{\mathrm{can}} = p^nw_{\mathrm{can}}.
\end{equation}

For the remainder of Section \ref{depletionsection} and starting in Assumption \ref{startinginassumption} of Section \ref{continuitysection}, we work under the following Assumption.

\begin{assumption}\label{nepsilonassumption}Assume that $1/(p+1) \le \epsilon_0 < p/(p+1)$ in Definition \ref{Udefinition}, so that 
$$n(\epsilon_0) \overset{(\ref{nepsilondefinition})}{=} 1.$$
\end{assumption}

Recall $U = \mathcal{Y}^{\mathrm{Ig}}(\epsilon_0)$ (Definition \ref{Udefinition}) so that 
$$\mathcal{Y}^{\mathrm{Ig}}(\epsilon_0/p) \cdot g \overset{(\ref{gisomorphism})}{=} \mathcal{Y}^{\mathrm{Ig}}(\epsilon_0) = U.$$
Since we assume $n(\epsilon_0) = 1$ (see Assumption \ref{nepsilonassumption}), then $n(\epsilon_0/p) = 2$.  By (\ref{UVdefinition}), (\ref{gammawcan}) with $m = n = n(\epsilon_0) = 1$ and (\ref{gwcan}), on $U$ we have 
\begin{equation}\label{VUwcan}\left(\begin{array}{ccc} 1 & j\\
0 & 1\\
\end{array}\right)^*w_{\mathrm{can}} = w_{\mathrm{can}}, \hspace{1cm} \left(\begin{array}{ccc} 1 & j/p\\
0 & 1 \\
\end{array}\right)^*w_{\mathrm{can}} = w_{\mathrm{can}}.
\end{equation}

\begin{definition}Suppose we are given $w' \in \omega^{\otimes k}(Y_{\infty})$. Write 
$$G = \frac{w'}{w_{\mathrm{can}}^{\otimes k}} \in \mathbb{B}_{\mathrm{dR},\mathcal{V}_x}(\mathcal{V}_x)\llbracket q_{\mathrm{dR}}-1\rrbracket.$$
\begin{enumerate}
\item Letting $w'^{\flat}$ be as in (\ref{w'flat}), define
\begin{equation}\label{Gflatgeneraldefinition}G^{\flat} := \frac{w'^{\flat}}{w_{\mathrm{can}}^{\otimes k}}  \in \mathbb{B}_{\mathrm{dR},\mathcal{V}_x}(\mathcal{V}_x)\llbracket q_{\mathrm{dR}}-1\rrbracket.
\end{equation}
\item From Definition \ref{flatdefinition}, we thus have
\begin{align*}G^{\flat}\cdot w_{\mathrm{can}}^{\otimes k} := w'^{\flat} := (U_p^*V_p^* - V_p^*U_p^*)w' &= (U_p^*V_p^* - V_p^*U_p^*)(G \cdot w_{\mathrm{can}}^{\otimes k}) \\
&\overset{(\ref{UV}), (\ref{VU})}{=} \frac{1}{p}\sum_{j = 0}^{p-1}\left(\left(\begin{array}{ccc} 1 & j \\
0 & 1 \\
\end{array}\right)^* - \left(\begin{array}{ccc} 1 & j/p\\
0 & 1\\
\end{array}\right)^*\right)(G \cdot w_{\mathrm{can}}^{\otimes k})\\
&\overset{(\ref{VUwcan})}{=} \frac{1}{p}\left(\sum_{j = 0}^{p-1}\left(\left(\begin{array}{ccc} 1 & j \\
0 & 1 \\
\end{array}\right)^* - \left(\begin{array}{ccc} 1 & j/p\\
0 & 1\\
\end{array}\right)^*\right)G\right)\cdot w_{\mathrm{can}}^{\otimes k} \\
&\overset{(\ref{UV}), (\ref{VU})}{=} (U_p^*V_p^* - V_p^*U_p^*)G \cdot w_{\mathrm{can}}^{\otimes k}.
\end{align*}
Thus we get the following expression for $G^{\flat}$ on $\mathcal{Y}^{\mathrm{Ig}}(\epsilon_0/p)$:
\begin{equation}\label{Gflatformula}G^{\flat}|_{\mathcal{Y}^{\mathrm{Ig}}(\epsilon_0/p)} := (U_p^*V_p^* - V_p^*U_p^*)G|_{\mathcal{Y}^{\mathrm{Ig}}(\epsilon_0/p)} \in \mathbb{B}_{\mathrm{dR},\mathcal{V}_x}(\mathcal{Y}^{\mathrm{Ig}}(\epsilon_0/p))\llbracket q_{\mathrm{dR}}-1\rrbracket.
\end{equation}
\end{enumerate}
\end{definition}

\subsection{Continuity of $p$-adic Maass-Shimura derivatives}\label{continuitysection}

%\begin{convention}\label{jconvention}Throughout this section, for notational convenience (in particular to avoid conflation with certain other notations) we will let 
%$$\mathbf{j} \in \mathbb{Z}/(p-1) \times \mathbb{Z}_p$$
%denote the weight variable appearing in powers of $p$-adic Maass-Shimura derivatives and powers of the Hodge bundle (e.g. $\partial_k^{\mathbf{j}}, d_k^{\mathbf{j}}$, $\omega^{\otimes k +2\mathbf{j}}$) as opposed to in the rest of the paper, where we use 
%$$j \in \mathbb{Z}/(p-1) \times \mathbb{Z}_p.$$
%\end{convention}

For the purposes of constructing $p$-adic $L$-functions, it will be important to show that $\theta_t$-projections of the constant terms of the $p$-adic Maass-Shimura derivatives $d_k^{j}G$ of a weight $k \ge 0$ generalized $p$-adic modular form $G$ vary nicely in $j \in \mathbb{Z}/(p-1) \times \mathbb{Z}_p$. Namely, we will show that under certain assumptions on $F$, the function
$$j \mapsto \theta_t(\theta_q(d_k^{j}G))$$
varies $p$-adic continuously in the space of generalized $p$-adic modular forms. The $p$-adic $L$-functions we construct will essentially come from sums of specializations to CM points of continuous functions of the above type, thus giving functions which are continuous functions in the weight variable $j$.

Recall the Atkin-Serre operator $\theta_{\mathrm{AS}}$  (\cite[Section 3.4]{Brooks}) can be expressed locally on $\mathcal{Y}^{\mathrm{Ig}}$ as $\frac{q_{\mathrm{ST}}d}{dq_{\mathrm{ST}}}$. Here, $q_{\mathrm{ST}}$ is the Serre-Tate coordinate from Definition \ref{STexpansiondefinition}, which is locally defined on each formal neighborhood $M^{\mathrm{Ig}}(A_0,P_0)$.

For the rest of the paper, we consider
$$\mathbb{Z}_{\ge 0} \subset \mathbb{Z}/(p-1) \times \mathbb{Z}_p$$
to be embedded diagonally. Note that the image of this embedding is dense with respect to the product of the discrete and $p$-adic topologies on the target. 

\begin{proposition}\label{prop1}Suppose $G \in \mathbb{B}_{\mathrm{dR},\mathcal{V}_x}(U)\llbracket q_{\mathrm{dR}}-1\rrbracket$ is a generalized $p$-adic modular form of weight $k \in \mathbb{Z}_{\ge 0}$ (Definition \ref{generalizedpadicmodularformdefinition}). We have 
$$\theta_t(\theta_q(d_k^{j}G))|_{\mathcal{Y}^{\mathrm{Ig}}} = \theta_{\mathrm{AS}}^{j}G|_{\mathcal{Y}^{\mathrm{Ig}}}|_{q_{\mathrm{ST}} = 1}.$$
\end{proposition}

\begin{proof}From (\ref{dkjformula2}) and Remark \ref{recoverASremark} we see that $d_k^{j}G|_{\mathcal{Y}^{\mathrm{Ig}}} = \theta_{\mathrm{AS}}^{j}G|_{\mathcal{Y}^{\mathrm{Ig}}}$. We also have
$$\theta_t(\theta_q(d_k^{j}G))|_{\mathcal{Y}^{\mathrm{Ig}}} = \theta_t(\theta_{\mathrm{AS}}^{j}G|_{\mathcal{Y}^{\mathrm{Ig}}})|_{q_{\mathrm{dR}} = 1} \overset{(\ref{ordinaryintegral})}{=} (\theta_{\mathrm{AS}}^{\mathrm{j}}G|_{\mathcal{Y}^{\mathrm{Ig}}})|_{q_{\mathrm{ST}} = 1}.$$
%By Theorem \ref{STanalyticcontinuationtheorem}, we see that the right-hand side is equal to $\theta_{\mathrm{AS}}^{j}G|_{\mathcal{Y}^{\mathrm{Ig}}}|_{q_{\mathrm{ST}} = 1}$. 

\end{proof}

\begin{proposition}\label{prop2}For any $0 < \epsilon < p/(p+1)$ in the valuation group of $\mathcal{O}_k$, the formal scheme $\hat{\mathcal{Y}}^{\mathrm{Ig}}(\epsilon)^+ \rightarrow Y^+(\epsilon)$ is geometrically irreducible over every geometric irreducible component of $Y^+(\epsilon)$. In particular, for any $0 < \epsilon' \le \epsilon$, the open sub-formal scheme $\hat{\mathcal{Y}}^{\mathrm{Ig}}(\epsilon')^+ \subset \hat{\mathcal{Y}}^{\mathrm{Ig}}(\epsilon)^+$ is dense (in the formal topology). 
\end{proposition}

\begin{proof}The formal scheme $\hat{Y}_{\infty}^+\rightarrow \hat{Y}^+$ (all over $\mathrm{Spf}(\mathbb{Z}_p)$) is geometrically irreducible over every irreducible component of $\hat{Y}^+$. The open sub-formal scheme $\hat{Y}_{\infty}^+(\epsilon) \rightarrow Y^+(\epsilon)$ is geometrically irreducible over every geometric irreducible component of $Y^+(\epsilon)$. Now $\hat{\mathcal{Y}}^{\mathrm{Ig}}(\epsilon) \subset Y_{\infty}^+(\epsilon)$ is an open sub-formal scheme (see Definition \ref{hatmathcalYIgepsilonDefinition}), and so is also irreducible over each irreducible component of $Y^+(\epsilon)$. The density statements follow immediately because any open subset of an irreducible set is irreducible, and by the moduli interpretations, $\hat{\mathcal{Y}}^{\mathrm{Ig}}(\epsilon')^+ \subset \hat{\mathcal{Y}}^{\mathrm{Ig}}(\epsilon)^+$ intersects every geometric irreducible component of $\hat{\mathcal{Y}}^{\mathrm{Ig}}(\epsilon)^+$. In fact, it intersects every irreducible component of $\hat{Y}_{\infty}^+$; by \cite[Theorem 13.7.6]{KatzMazur}, the irreducible components of $\hat{Y}_{\infty}^+$ all intersect at each supersingular $\overline{\mathbb{F}}_p$-point, and so the assertion follows after noting that $\hat{\mathcal{Y}}^{\mathrm{Ig}}(\epsilon')^+$ contains a supersingular $\overline{\mathbb{F}}_p$-point since $\epsilon' > 0$. 

\end{proof}

\begin{proposition}\label{prop3}Suppose $G \in \mathbb{B}_{\mathrm{dR},\mathcal{V}_x}(U)\llbracket q_{\mathrm{dR}}-1\rrbracket$ is a weight $k$ generalized $p$-adic modular form (Definition \ref{generalizedpadicmodularformdefinition}). Further suppose that $G|_{\mathcal{Y}^{\mathrm{Ig}}} \in \hat{\mathcal{O}}_{\mathcal{Y}^{\mathrm{Ig}}}^+(\mathcal{Y}^{\mathrm{Ig}})$. Then if $j, j' \in \mathbb{Z}_{\ge 0}$ and $j \equiv j' \pmod{(p-1)p^n}$ then
$$\theta_{\mathrm{AS}}^{j}G|_{\mathcal{Y}^{\mathrm{Ig}}}|_{q_{\mathrm{ST}} = 1} \equiv \theta_{\mathrm{AS}}^{j'}G|_{\mathcal{Y}^{\mathrm{Ig}}}|_{q_{\mathrm{ST}} = 1} \pmod{p^n\hat{\mathcal{O}}_{\mathcal{Y}^{\mathrm{Ig}}}^+(\mathcal{Y}^{\mathrm{Ig}})}.$$
\end{proposition}

\begin{proof}Let $D \subset Y^{\mathrm{ord}}$ be an ordinary residue disc and let $\tilde{\mathcal{D}} = D\times_{Y^{\mathrm{ord}}}\mathcal{Y}^{\mathrm{Ig}} \rightarrow D$. As the $\tilde{\mathcal{D}}$ cover $\mathcal{Y}^{\mathrm{Ig}}$, it suffices to prove the Proposition on each $\tilde{\mathcal{D}}$. This is a classical result that can be found in our Shimura curve setting in \cite[Corollary 4.19]{Brooks}.
 
\end{proof}

Proposition \ref{prop3} shows that $\theta_{\mathrm{AS}}^{j}G|_{\mathcal{Y}^{\mathrm{Ig}}}|_{q_{\mathrm{ST}} = 1}$ for $j \in \mathbb{Z}/(p-1) \times \mathbb{Z}$ can be defined via taking a limit of $\theta_{\mathrm{AS}}^{j}G|_{\mathcal{Y}^{\mathrm{Ig}}}|_{q_{\mathrm{ST}} = 1}$ for $j \in \mathbb{Z}_{\ge 0}$ (because $\mathbb{Z}_{\ge 0} \subset \mathbb{Z}/(p-1) \times \mathbb{Z}_p$ is dense). However, it does not imply that the function
$$\mathbb{Z}_{\ge 0} \ni j \mapsto \theta_{\mathrm{AS}}^{j}G|_{\mathcal{Y}^{\mathrm{Ig}}}|_{q_{\mathrm{ST}} = 1} \in \hat{\mathcal{O}}_{\mathcal{Y}^{\mathrm{Ig}}}(\mathcal{Y}^{\mathrm{Ig}})$$
is continuous. For this, one must take the \emph{$p$-depletion} of $G|_{\mathcal{Y}^{\mathrm{Ig}}}$.

\begin{assumption}\label{startinginassumption}We now work under Assumption \ref{nepsilonassumption}, i.e. $n(\epsilon_0) = 1$, for the rest of this section; this is predominantly needed in order to use the formula (\ref{Gflatformula}) for $G^{\flat}$ (see (\ref{Gflatgeneraldefinition})) for $G$ as below. 
\end{assumption}

Supposing $G \in \mathbb{B}_{\mathrm{dR},\mathcal{V}_x}(\mathcal{Y}^{\mathrm{Ig}}(\epsilon_0))\llbracket q_{\mathrm{dR}}-1\rrbracket$ is a generalized $p$-adic modular form of weight $k$, let $w' = G \cdot w_{\mathrm{can}}^{\otimes k}$, and recall $G^{\flat} \in \mathbb{B}_{\mathrm{dR},\mathcal{V}_x}(\mathcal{Y}^{\mathrm{Ig}}(\epsilon_0/p))\llbracket q_{\mathrm{dR}}-1\rrbracket$ from (\ref{Gflatgeneraldefinition}). 
%Now define
%\begin{equation}\label{Gflatnoq}G^{\flat} := (1-V_p^*U_p^*)G = (1-V_p^*U_p^*)\left(\frac{w'^{\flat}}{w_{\mathrm{can}}^{\otimes k}}\right) \overset{(\ref{VUwcan})}{=} \frac{w'^{\flat}}{w_{\mathrm{can}}^{\otimes k}}.
%\end{equation}
 If $G$ satisfies the identity (\ref{Ftransformationidentity}) for all $\gamma \in \Gamma_{1,p}(p^{n(\epsilon)})$, $0\le \epsilon < p/(p+1)$, a direct calculation using (\ref{Gflatformula}) shows that $G^{\flat}$ satisfies the identity (\ref{Ftransformationidentity}) for all $\gamma \in \Gamma_{1,p}(p^{n(\epsilon/p)})$ (note that $n(\epsilon/p) = n(\epsilon)  +1$, see (\ref{nepsilondefinition})). 

\begin{proposition}\label{prop4}In the situation of Proposition \ref{prop3} and further working under Assumption \ref{nepsilonassumption}, the function 
$$\mathbb{Z}_{\ge 0} \ni j \mapsto \theta_{\mathrm{AS}}^{j}G^{\flat}|_{\mathcal{Y}^{\mathrm{Ig}}}|_{q_{\mathrm{ST}} = 1} \in \hat{\mathcal{O}}_{\mathcal{Y}^{\mathrm{Ig}}}^+(\mathcal{Y}^{\mathrm{Ig}})$$
is continuous. 
\end{proposition}

\begin{proof}In view of Proposition \ref{prop1}, this follows from the analysis of \cite[Section 4.5]{Brooks}. Note that $\flat$ here corresponds to $p$-depletion in loc. cit.

\end{proof}

The following Theorem is one of our key applications of Theorem \ref{pintegraltheorem}. %For the statement of the Theorem, recall we are using Convention \ref{jconvention}, and that $\mathbf{j}$ will be denoted by $j$ in the rest of the paper. %We will not need part (2) of the following Theorem in the rest of the paper. 

\begin{theorem}\label{pintegraltheorem2}Recall we are in the setting of Assumption \ref{nepsilonassumption}. Suppose $G \in \mathbb{B}_{\mathrm{dR},\mathcal{V}_x}(\mathcal{V}_x)\llbracket q_{\mathrm{dR}}-1\rrbracket$ is a generalized $p$-adic modular form of weight $k \in \mathbb{Z}_{\ge 0}$, and that it satisfies the identity (\ref{Ftransformationidentity}) for all $\gamma \in \Gamma_{1,p}(p^{n(\epsilon_0/p^{\alpha})})$ and $W = \mathcal{Y}^{\mathrm{Ig}}(\epsilon_0/p^{\alpha})$ for some $\alpha \in \mathbb{Z}_{\ge 0}$. Then:
\begin{enumerate}
\item $d_k^{j}G \in \mathbb{B}_{\mathrm{dR},\mathcal{V}_x}(\mathcal{V}_x)\llbracket q_{\mathrm{dR}}-1\rrbracket$ is a generalized $p$-adic modular form of weight $k +2j$, and moreover 
\begin{equation}\label{integralthetaG}\theta_t(\theta_q(d_k^{j}G))|_{\mathcal{Y}^{\mathrm{Ig}}(\epsilon_0/p^{\alpha})} \in \hat{\mathcal{O}}_{\mathcal{Y}^{\mathrm{Ig}}(\epsilon_0/p^{\alpha})}^+(\mathcal{Y}^{\mathrm{Ig}}(\epsilon_0/p^{\alpha})).
\end{equation}
%\item Suppose moreover that $w' = G \cdot w_{\mathrm{can}}^{\otimes k}$ is an eigenvector for $U_p^*$ with eigenvalue $\alpha_p \in \mathbb{C}_p$, i.e. $U_p^*w' = \alpha_p\cdot w'$.\footnote{This assumption will be satisfied when $w'$ is the pullback of a $U_p$-eigenform on $Y$, for example.} Thus, by Lemma \ref{Uplemma}, $\partial_k^{j}w'$ is an eigenvector for $U_p^*$ with eigenvalue $p^{j}\alpha_p$. Then for any integer $0 \le \alpha' \le \alpha$, 
%\begin{equation}\label{integralthetaG'}p^{(j+1)\alpha'}\alpha_p^{\alpha'} \cdot \theta_t(\theta_q(d_k^{j}G))|_{\mathcal{Y}^{\mathrm{Ig}}(\epsilon_0/p^{\alpha- \alpha'})} \in \hat{\mathcal{O}}_{\mathcal{Y}^{\mathrm{Ig}}(\epsilon_0/p^{\alpha - \alpha'})}^+(\mathcal{Y}^{\mathrm{Ig}}(\epsilon_0/p^{\alpha - \alpha'})).
%\end{equation}
\item %%Recall that $G|_{\mathcal{Y}^{\mathrm{Ig}}} \in \hat{\mathcal{O}}_{\mathcal{Y}^{\mathrm{Ig}}}(\mathcal{Y}^{\mathrm{Ig}})$ by Theorem \ref{Katzpadicmodularformtheorem}. Assume that $G|_{\mathcal{Y}^{\mathrm{Ig}}} \in \hat{\mathcal{O}}_{\mathcal{Y}^{\mathrm{Ig}}}^+(\mathcal{Y}^{\mathrm{Ig}})$ (which can always be ensured after multiplying $G$ by an element of $\mathcal{O}_k$), Suppose moreover that the function
%%$$\mathbb{Z}_{\ge 0} \ni j \mapsto \theta_{\mathrm{AS}}^{j}G|_{\mathcal{Y}^{\mathrm{Ig}}}|_{q_{\mathrm{ST}} = 1} \in \hat{\mathcal{O}}_{\mathcal{Y}^{\mathrm{Ig}}}^+(\mathcal{Y}^{\mathrm{Ig}})$$
%%is continuous. 
Recall the notation of Proposition \ref{prop4}. Then the function
\begin{equation}\label{integralthetaGcontinuous}\mathbb{Z}_{\ge 0} \ni j \mapsto\theta_t(\theta_q(d_k^{j}G^{\flat}))|_{\mathcal{Y}^{\mathrm{Ig}}(\epsilon_0/p^{\alpha})} \in \hat{\mathcal{O}}_{\mathcal{Y}^{\mathrm{Ig}}(\epsilon_0/p^{\alpha})}^+(\mathcal{Y}^{\mathrm{Ig}}(\epsilon_0/p^{\alpha}))
\end{equation}
is continuous with respect to the $p$-adic topologies on the source and target. It extends continuously to a function $\mathbb{Z}/(p-1) \times \mathbb{Z}_p \rightarrow \hat{\mathcal{O}}_{\mathcal{Y}^{\mathrm{Ig}}(\epsilon_0/p^{\alpha
})}^+(\mathcal{Y}^{\mathrm{Ig}}(\epsilon_0/p^{\alpha}))$. 
\end{enumerate}
\end{theorem}

\begin{remark}\label{assumptionsatisfiedremark2}As in Remark \ref{assumptionsatisfiedremark}, the assumption that $G$ satisfy (\ref{Ftransformationidentity}) in Theorem \ref{pintegraltheorem2} holds if $w' = G \cdot w_{\mathrm{can}}^{\otimes k}$ is the pullback along $\mathcal{Y}^{\mathrm{Ig}}(\epsilon_0/p^{\alpha}) \rightarrow Y_1(\alpha+1,\epsilon_0/p^{\alpha})$ of an element of $\omega^{\otimes k}(Y_1(\alpha+1,\epsilon_0/p^{\alpha}))$. This will be the case in our main application of Theorem \ref{pintegraltheorem2}.
\end{remark}

\begin{proof}[Proof of Theorem \ref{pintegraltheorem2}]\textbf{(1)}: The fact that $d_k^{j}G \in \mathbb{B}_{\mathrm{dR},\mathcal{V}_x}(\mathcal{V}_x)\llbracket q_{\mathrm{dR}}-1\rrbracket$ is a generalized $p$-adic modular form of weight $k + 2j$ follows from Definition \ref{padicMSdefinition}. Corollary \ref{weightcorollary} implies that $d_k^{j}G$ satisfies (\ref{Ftransformationidentity}) with $k + 2j$ in place of $k$ for all $\gamma \in \Gamma_{1,p}(p^{n(\epsilon_0/p^{\alpha})})$.  
%the assumptions of Theorem \ref{weighttheorem} for weight $k + 2j$, $W = \mathcal{Y}^{\mathrm{Ig}}(\epsilon_0/p^{\alpha})$ and $W_0 = Y(\epsilon_0/p^{\alpha})$. 
 Applying Theorem \ref{pintegraltheorem} to $d_k^{j}G$, we get (\ref{integralthetaG}). \\

\textbf{(2)}: By Propositions \ref{prop1} and  \ref{prop4} we have that 
$$\mathbb{Z}_{\ge 0} \ni j \mapsto \theta_t(\theta_q(d_k^{j}G^{\flat}))|_{\mathcal{Y}^{\mathrm{Ig}}} \in \hat{\mathcal{O}}_{\mathcal{Y}^{\mathrm{Ig}}}^+(\mathcal{Y}^{\mathrm{Ig}})$$
is continuous and moreover satisfies the following uniform continuity: if $j \equiv j' \pmod{(p-1)p^{n-1}}$ then
$$\theta_t(\theta_q(d_k^{j}G^{\flat}))|_{\mathcal{Y}^{\mathrm{Ig}}}  \equiv  \theta_t(\theta_q(d_k^{j'}G^{\flat}))|_{\mathcal{Y}^{\mathrm{Ig}}} \pmod{p^n\hat{\mathcal{O}}_{\mathcal{Y}^{\mathrm{Ig}}}^+(\mathcal{Y}^{\mathrm{Ig}})}.$$
Since for all $j \in \mathbb{Z}_{\ge 0}$
$$\theta_t(\theta_q(d_k^{j}G^{\flat}))|_{\mathcal{Y}^{\mathrm{Ig}}(\epsilon_0/p^{\alpha})} \in \hat{\mathcal{O}}_{\mathcal{Y}^{\mathrm{Ig}}(\epsilon_0/p^{\alpha})}^+(\mathcal{Y}^{\mathrm{Ig}}(\epsilon_0/p^{\alpha})) \rightarrow \varinjlim_{0 < \epsilon < \epsilon_0/p^{\alpha}}\hat{\mathcal{O}}_{\mathcal{Y}^{\mathrm{Ig}}(\epsilon)}^+(\mathcal{Y}^{\mathrm{Ig}}(\epsilon)),$$
and by Corollary \ref{completioncorollary} we have that $\hat{\mathcal{O}}_{\mathcal{Y}^{\mathrm{Ig}}}^+(\mathcal{Y}^{\mathrm{Ig}})$ is the $p$-adic completion of the direct limit in the previous displayed expression, then for all $m \in \mathbb{Z}_{\gg \alpha}$ we have that 
$$\theta_t(\theta_q(d_k^{j}G^{\flat}))|_{\mathcal{Y}^{\mathrm{Ig}}(\epsilon_0/p^m)} \equiv \theta_t(\theta_q(d_k^{j'}G^{\flat}))|_{\mathcal{Y}^{\mathrm{Ig}}(\epsilon_0/p^m)} \pmod{p^n\hat{\mathcal{O}}_{\mathcal{Y}^{\mathrm{Ig}}(\epsilon_0/p^m)}^+(\mathcal{Y}^{\mathrm{Ig}}(\epsilon_0/p^m))}.$$
%Recall the rational subset $\hat{\mathcal{V}}_{\alpha}' \subset \hat{Y}_{\infty}$. By (\ref{mathcalYIgepsilonpmVinclusion}) and increasing $m$ if necessary, we may assume that $\hat{\mathcal{Y}}^{\mathrm{Ig}}(\epsilon_0/p^m) \subset \hat{\mathcal{V}}_{\alpha}' \subset \hat{Y}_{\infty}$. 
%Note that $\hat{\mathcal{V}}_{\alpha}' \cap \hat{\mathcal{Y}}^{\mathrm{Ig}}(\epsilon_0)$ is an intersection of two affinoids and is thus affinoid. Then $W = \mathrm{Spf}(\mathbf{\Gamma}(\hat{\mathcal{O}}_{\hat{\mathcal{V}}_{\alpha}' \cap \mathcal{Y}^{\mathrm{Ig}}(\epsilon_0/p^{\alpha})}))$ is an affine formal scheme, and the open embedding of adic spaces $\hat{\mathcal{V}}_{\alpha}' \cap \hat{\mathcal{Y}}^{\mathrm{Ig}}(\epsilon_0/p^{\alpha}) \subset \hat{\mathcal{Y}^{\mathrm{Ig}}}(\epsilon_0)$ induces an open embedding of formal schemes $W \subset \hat{\mathcal{Y}}^{\mathrm{Ig}}(\epsilon_0)^+$. 
By Proposition \ref{prop2}, for all $m \in \mathbb{Z}_{\gg \alpha}$ the open subset $\hat{\mathcal{Y}}^{\mathrm{Ig}}(\epsilon_0/p^m)^+ \subset \hat{\mathcal{Y}}^{\mathrm{Ig}}(\epsilon_0/p^{\alpha})^+$ is dense. However, the locus in $\hat{\mathcal{Y}}^{\mathrm{Ig}}(\epsilon_0/p^{\alpha})^+$ defined by
$$\theta_t(\theta_q(d_k^{j}G^{\flat})) \equiv \theta_t(\theta_q(d_k^{j'}G^{\flat})) \pmod{p^n}$$
is pro-Zariski (and thus adically) closed. Since the dense open subset $\hat{\mathcal{Y}}^{\mathrm{Ig}}(\epsilon_0/p^m)^+$ is contained in this locus, we have that the locus is all of $\hat{\mathcal{Y}}^{\mathrm{Ig}}(\epsilon_0/p^{\alpha})^+$, i.e.
$$\theta_t(\theta_q(d_k^{j}G^{\flat}))|_{\mathcal{Y}^{\mathrm{Ig}}(\epsilon_0/p^{\alpha})} \equiv \theta_t(\theta_q(d_k^{j'}G^{\flat}))|_{\mathcal{Y}^{\mathrm{Ig}}(\epsilon_0/p^{\alpha})} \pmod{p^n\hat{\mathcal{O}}_{\mathcal{Y}^{\mathrm{Ig}}(\epsilon_0/p^{\alpha})}^+(\mathcal{Y}^{\mathrm{Ig}}(\epsilon_0/p^{\alpha}))}.$$
%Thus $\hat{\mathcal{Y}}^{\mathrm{Ig}}(\epsilon_0/p^m)^+ \subset W$ is dense. Since $\mathbf{\Gamma}(\mathcal{O}_W) = \Gamma(\hat{\mathcal{O}}_{\hat{\mathcal{Y}}^{\mathrm{Ig}}(\epsilon_0/p^{\alpha}) \cap \hat{\mathcal{V}}_{\alpha}'})$, 
This implies that
%(\ref{integralthetaGcontinuous0}) is continuous, which by the discussion in the previous paragraph implies 
(\ref{integralthetaGcontinuous}) is continuous. The continuous extension to $\mathbb{Z}/(p-1) \times \mathbb{Z}_p$ assertion now immediately follows. %Since $\hat{\mathcal{Y}}^{\mathrm{Ig}}(\epsilon_0)^+$ is geometrically irreducible over every geometric irreducible component of $Y^+(\epsilon_0)$ by Proposition \ref{prop2}, and $W$then $W$ is geometrically irreducible

\end{proof}

\section{Anticyclotomic $p$-adic $L$-functions for $GL_1/K$}\label{padicLfunctionsection}We continue to work under Convention \ref{Yconvention}.

\begin{assumption}\label{pramifiedassumption}For Sections \ref{padicLfunctionsection} and \ref{padicLfunctionsection2}, let $K/\mathbb{Q}$ be an imaginary quadratic field of class number 1. Hence $K = \mathbb{Q}(i)$ or $\mathbb{Q}(\sqrt{-p})$ for $p = 2, 3, 7, 11, 19, 43, 67$ or $163$. Let $p$ be the unique finite prime which ramifies in $K/\mathbb{Q}$, and let $\frak{p}$ be the unique prime of $\mathcal{O}_K$ above $p$. Thus $\frak{p}^2 = p\mathcal{O}_K$. 
\end{assumption}

\subsection{Notation}Henceforth, let $D_K \in \mathbb{Z}_{< 0}$ denote the fundamental discriminant of $K$. Given any ideal $\frak{n} \subset \mathcal{O}_K$, let $\mathcal{C}\ell(\frak{n})$ denote the ray class group of $K$ of modulus $\frak{n}$. 

\begin{definition}
Let
\begin{equation}\label{ebddefinition}\begin{split}\varepsilon := \begin{cases} 
1 & p = 2 \; \text{and} \; K = \mathbb{Q}(\sqrt{-2})\\
3 & p = 2 \; \text{and} \; K = \mathbb{Q}(i)\\
2 & p = 3\\
1 & p > 3
\end{cases},& \hspace{1cm} \beta := \begin{cases} 3 & p = 2 \; \text{and} \; K = \mathbb{Q}(\sqrt{-2})\\
2 & p = 2 \; \text{or} \; 3, \; \text{and} \;  K \neq \mathbb{Q}(\sqrt{-2}) \\
1 & p > 3\\
\end{cases},\\
&q := \begin{cases} 4 & p = 2\\
p & p > 2
\end{cases}.
\end{split}
\end{equation}
\end{definition}
\begin{remark}The constant $\varepsilon$ is not to be confused with the parameter $\epsilon$ first introduced in Section \ref{formalShimurasection}. 
\end{remark}

Note that the $p$-adic logarithm also induces a non-canonical isomorphism
\begin{equation}\label{logarithmequal}1 + q\mathbb{Z}_p \underset{\sim}{\xrightarrow{\log}}\mathbb{Z}_p.
\end{equation}

\begin{definition}\label{OKpfacts}\begin{enumerate}
\item Let
$$\Gamma := 1+\frak{p}^{\varepsilon}\mathcal{O}_{K_p}.$$
If $K \neq \mathbb{Q}(\sqrt{-2})$, then 
\begin{equation}\label{firstGamma}\Gamma \underset{\sim}{\xrightarrow{\log}} \mathcal{O}_{K_p}
\end{equation}
as $\mathcal{O}_{K_p}$-modules, given by the $p$-adic logarithm $\log$ (one can check that $\log$ is injective, and that $\Gamma$ is torsion-free). If $K = \mathbb{Q}(\sqrt{-2})$ (and thus $p = 2$), then
\begin{equation}\label{secondGamma}\Gamma \cong (1+\pi)^{\mathbb{Z}_p} \times \{\pm 1\} \times (1+q)^{\mathbb{Z}_p} \cong \{\pm 1\} \times \mathbb{Z}_p^{\oplus 2}
\end{equation}
as $\mathbb{Z}_p$-modules, for any uniformizer $\pi$ of $\mathcal{O}_{K_p}$. 

\item Let $\Delta = (\mathcal{O}_{K_p}^{\times})_{\mathrm{tors}}$ be the torsion part of $\mathcal{O}_{K_p}^{\times}$, which is equal to the group of roots of unity in $\mathcal{O}_{K_p}$:
$$\mu(\mathcal{O}_{K_p}) = \begin{cases} \mu_{p-1} & p > 3\\
\mu_6 & p = 3\\
\mu_4 & p = 2 \; \text{and} \; K = \mathbb{Q}(i)\\
\mu_2 & p = 2 \; \text{and} \; K = \mathbb{Q}(\sqrt{-2})\\
\end{cases}.$$
Let $\Delta[p^{\infty}]$ denote the $p$-primary part of $\Delta$. 

%Finally, define
%\begin{equation}\label{Gamma-definition}\Gamma_- := \Gamma/(\Delta[p^{\infty}]\cdot (1+q\mathbb{Z}_p)), \hspace{1cm} \Gamma_{-,n} = \Gamma_-/(\Gamma_-)^{p^n}.
%\end{equation}
\end{enumerate}
\end{definition}

%In fact, we have:

%\begin{lemma}\label{gamma-lemma}$\Gamma_- \cong \mathbb{Z}_p, \Gamma_{-,n} \cong \mathbb{Z}/p^n$.
%\end{lemma}

%\begin{proof}When $K = \mathbb{Q}(\sqrt{-2})$, this is clear from (\ref{secondGamma}). Let $\varpi$ be any local uniformizer of $\mathcal{O}_{K_p}$; in particular $\mathrm{ord}_p(\varpi) = 1/2$ where $\mathrm{ord}_p$ is the $p$-adic valuation on $\mathbb{C}_p$ with $\mathrm{ord}_p(p) = 1$. Note
%$$\Gamma = 1 + \frak{p}^{\varepsilon}\mathcal{O}_{K_p} = (1+\varpi^{\varepsilon})^{\mathcal{O}_{K_p}}.$$
%Then since $\mathcal{O}_{K_p} = \mathbb{Z}_p[\varpi]$, we have 
%$$\mathcal{O}_{K_p} = \mathbb{Z}_p + \frac{\log(1+q)}{\log(1+\varpi^{\varepsilon})}\cdot \mathbb{Z}_p,$$
%from which we see
%$$\Gamma = (1+\varpi^{\varepsilon})^{\mathbb{Z}_p}(1+q)^{\mathbb{Z}_p}.$$
%Thus
%$$\Gamma_- = \Gamma/(1+q\mathbb{Z}_p) \overset{(\ref{logarithmequal})}{=} (1+\varpi^{\varepsilon})^{\mathbb{Z}_p} \cong \mathbb{Z}_p.$$
%The statement now follows. 
%\end{proof}

Given any algebraic Hecke character $\varrho : K^{\times}\backslash \mathbb{A}_K^{\times} \rightarrow \mathbb{C}^{\times}$, let $\frak{f}(\varrho) \subset \mathcal{O}_K$ denote its (exact) conductor. For the rest of this section, let 
$$\lambda : \mathbb{A}_K^{\times}/K^{\times} \rightarrow \mathbb{C}^{\times}$$
denote a fixed algebraic Hecke character of infinity type $(1,0)$. (See \cite[Chapter II.1]{deShalit} for definitions of these notions and background on algebraic Hecke characters.) Let 
$$\frak{f} = \frak{f}(\lambda).$$
In our main application, $\lambda$ will be the Hecke character $\lambda_E$ associated to an elliptic curve $E/\mathbb{Q}$ with CM by $\mathcal{O}_K$, or a twist thereof. Given any fractional ideal $\frak{n}$ of $\mathcal{O}_K$, let $\frak{n}^{(p)}$ denote the prime-to-$p$ part of $\frak{n}$.

\begin{convention}\label{avatarconvention}We will often conflate notation and also let $\lambda$ denote the $p$-adic avatar of $\lambda$. No confusion should arise as the target of $\lambda$ will be obvious from context.  
\end{convention}

\begin{assumption}\label{pconductorassumption}Henceforth, assume the following for $\lambda$, $\frak{f} = \frak{f}(\lambda)$ and $K$:
\begin{enumerate} 
\item Assume that 
$$\overline{\frak{f}} = \frak{f}.$$
Since $p$ is assumed to be ramified, this is equivalent to the equality on prime-to-$p$ parts
$$\overline{\frak{f}}^{(p)} = \frak{f}^{(p)}.$$
Since $p$ is the only finite prime ramified in $K/\mathbb{Q}$, all primes dividing $\frak{f}^{(p)}$ are inert or split in $K/\mathbb{Q}$. Since moreover $\frak{f}^{(p)} = \overline{\frak{f}}^{(p)}$, then $\frak{f}^{(p)} = f_0\mathcal{O}_K$ for a unique positive integer $f_0$. We moreover assume that $f_0 \ge 4$.\footnote{This assumption allows us to apply the main results of Sections \ref{Ysection}, \ref{ShimuraCurveSection} and \ref{zdRqdRqdRexpsection} when $f_0|N$, which we will soon assume (Assumption \ref{Nf0assumption}). Recall $N \ge 4$ was assumed in the aforementioned sections so that the Shimura curve $Y = Y(\Gamma(N))$ represents a fine moduli space, and also to employ the results of \cite{Buzzard}. (See Convention \ref{Yconvention}.) This assumption on $N$ can very likely be removed throughout the paper, for example by working with moduli stacks instead of moduli spaces.}
\item Letting $\mathbb{N}_L : \mathbb{A}_L^{\times}/L^{\times} \rightarrow \mathbb{C}^{\times}$ denote the norm character of a number field $L/\mathbb{Q}$ and letting
$$\eta := \lambda|_{\mathbb{A}_{\mathbb{Q}}^{\times}}\mathbb{N}_{\mathbb{Q}}^{-1}$$
denote the central character of $\lambda$, assume that $\eta$ is the quadratic character attached to $K/\mathbb{Q}$, i.e. $\eta$ is the Jacobi symbol attached to the fundamental discriminant $D_K \in\mathbb{Z}_{< 0}$ of $K$, i.e.
\begin{equation}\label{etaK}\eta = \left(\frac{D_K}{\cdot}\right).
\end{equation}
As we will work under the assumption (\ref{etaK}) for the rest of the paper, we will henceforth let $\eta$ denote $\left(\frac{D_K}{\cdot}\right)$ synonymously. In particular, $f(\eta) = |D_K| = p^{\beta}$ where $f(\eta)$ is the conductor of $\eta$.%\footnote{When $\lambda$ is the infinity type $(1,0)$ Hecke character attached to an elliptic curve over $\mathbb{Q}$ with CM by $\mathcal{O}_K$, this assumption is satisfied.}
%\item Assume that for every finite prime $v$ of $\mathcal{O}_K$, the local character $\lambda_v$ satisfies $\lambda_v(\mathcal{O}_{K_v}^{\times}) \subset \mathcal{O}_K^{\times}$. When $\lambda$ is attached to an elliptic curve $E/\mathbb{Q}$ with CM by $\mathcal{O}_K$, this holds by \cite[Top of p. 170]{Silverman}. 
%\item $p\nmid \#\mathcal{C}\ell(1)$, that is $p$ does not divide the class number of $K$. Clearly this holds if $K$ has class number 1 (such as when $\lambda = \lambda_E$ for the Hecke character $\lambda_E$ associated with $E/\mathbb{Q}$).\footnote{This second assumption is not strictly necessary. It will be easy to see through our proof that we get a $p$-adic $L$-function $\mathcal{L}_{\lambda} \in \mathcal{O}_{\mathbb{C}_p}\llbracket p^t\mathbb{Z}_p\rrbracket [1/p]$, where $p^t$ is the $p$-part of $\#\mathcal{C}\ell(1)$.} 
\item Assume that there is some elliptic curve $E/\mathbb{Q}$ with CM by $\mathcal{O}_K$ with associated infinity type $(1,0)$ Hecke character $\lambda_E$ such that $\frak{f}(\lambda_E)|\frak{f}$. %When $\lambda = \lambda_E$ is itself the Hecke character associated to an elliptic curve $E/\mathbb{Q}$ with CM by $\mathcal{O}_K$, this assumption is clearly satisfied by $E$ itself.
\end{enumerate}
\end{assumption}

\begin{example}\label{Eexample}Let $E/\mathbb{Q}$ be an elliptic curve with CM by $\mathcal{O}_K$, and let $\lambda_E : K^{\times}\backslash\mathbb{A}_K^{\times} \rightarrow \mathbb{C}^{\times}$ be the infinity type $(1,0)$ Hecke character attached to $E$, i.e. with $L(E/\mathbb{Q},s) = L(\lambda_E,s)$. Then 
\begin{equation}\label{centralcharacter}\lambda_E|_{\mathbb{A}_{\mathbb{Q}}^{\times}}\mathbb{N}_{\mathbb{Q}}^{-1} = \left(\frac{D_K}{\cdot}\right) 
\end{equation}
and $\lambda_E(\overline{x}) = \overline{\lambda}_E(x)$, where $x \mapsto \overline{x}$ denotes complex conjugation $\mathbb{A}_K^{\times} \xrightarrow{\sim} \mathbb{A}_K^{\times}$. Thus $\overline{\frak{f}(\lambda_E)} = \frak{f}(\lambda_E)$ and $\frak{f}(\lambda_E)^{(p)} = f_0\cdot \mathcal{O}_K$ for some $f_0 \in \mathbb{Z}_{> 0}$. Assume $f_0 \ge 4$; this further assumption is satisfied for all but finitely many isomorphism classes of $E/\mathbb{Q}$. Then 
$$\lambda = \lambda_E$$
satisfies Assumption \ref{pconductorassumption}. 

%When $\lambda$ is the Hecke character $\lambda_E$ attached to an elliptic curve $E/\mathbb{Q}$ with CM by $\mathcal{O}_K$, it is known (see \cite[Top of p. 170]{Silverman}) that $\lambda_{\frak{p}}(\mathcal{O}_{K_p}^{\times}) \subset \mathcal{O}_K^{\times}$. Hence the first assumption of (1) holds, and moreover $\lambda_E(\overline{x}) = \overline{\lambda}_E(x)$, so that the second assumption holds.

\end{example}

%Let $\mathrm{ord}_p : \mathbb{C}_p \rightarrow \mathbb{Q} \cup \{\infty\}$ be the $p$-adic valuation with $\mathrm{ord}_p(p) = 1$. Let $\mathrm{ord}_{\frak{p}} : \mathbb{C}_p \rightarrow \mathbb{Q} \cup \{\infty\}$ be the valuation with $\mathrm{ord}_{\frak{p}}(p) = 2$. 

\begin{definition}
\begin{enumerate}
\item In the notation of Assumption \ref{pconductorassumption}, let
\begin{equation}\label{edefinition}e := 2 \cdot \mathrm{ord}_p(\frak{f}) = \mathrm{ord}_{\frak{p}}(\frak{f}).
\end{equation}
In other words, $\frak{p}^e$ is the exact power of $\frak{p}$ dividing $\frak{f} = \frak{f}(\lambda)$. Let 
\begin{equation}\label{adefinition}a := \left\lceil \frac{e}{2}\right\rceil.
\end{equation}
\item Consider the character $\lambda|_{\mathcal{O}_{K_{\frak{p}}}^{\times}}$, which by (\ref{ebddefinition}) factors through $\lambda|_{\mathcal{O}_{K_{\frak{p}}}^{\times}} : (\mathcal{O}_{K_{\frak{p}}}/\frak{p}^e)^{\times} \rightarrow \overline{\mathbb{Z}}^{\times}$. By Assumption \ref{pconductorassumption} (2), we have that 
$$\lambda|_{\mathcal{O}_{K_{\frak{p}}}^{\times} \cap \mathbb{A}_{\mathbb{Q}}^{\times}} = \eta|_{\mathbb{Z}_p^{\times}} : (\mathbb{Z}_p/(\mathbb{Z}_p \cap \frak{p}^e\mathcal{O}_{K_{\frak{p}}}))^{\times} = (\mathbb{Z}_p/p^{\lceil e/2\rceil}\mathbb{Z}_p)^{\times} \rightarrow \overline{\mathbb{Z}}^{\times}.$$
Since $\eta|_{\mathbb{Z}_p^{\times}} : \mathbb{Z}_p^{\times} \rightarrow \overline{\mathbb{Z}}^{\times}$ factors through $\mathbb{Z}_p^{\times} \twoheadrightarrow (\mathbb{Z}_p/p^{\beta}\mathbb{Z}_p)^{\times} \rightarrow \overline{\mathbb{Z}}^{\times}$, we thus have
\begin{equation}\label{betadelta}\beta \le \left\lceil \frac{e}{2}\right\rceil = a.
\end{equation}
\end{enumerate}
\end{definition}

%\begin{choice}\label{abstractgammachoice}Fix an element of $\gamma \in \Gamma$ such that
%$$\mathrm{ord}_p(\gamma -1) = \frac{1}{2} + \delta.$$
%Later, we will take the specific choice made in (\ref{gammachoice}). We will often conflate notation and also let $\gamma$ denote the image of $\gamma$ under $\Gamma \rightarrow \Gamma_-$. Using Lemma \ref{gamma-lemma}, there is a unique identification
%$$\Gamma_- = \mathbb{Z}_p, \hspace{1cm} \Gamma_{-,n} = \mathbb{Z}/p^n$$
%such that $\gamma \in \Gamma_-$ maps to $p^{\delta - a}$ under the first identification above, where 
%\begin{equation}\label{adefinition}a := \begin{cases} 0 & K \neq \mathbb{Q}(i) \; \text{or} \;\; \mathbb{Q}(\sqrt{-3})\\
%1 & \text{else}.\\
%\end{cases}.
%\end{equation}
%\end{choice}

\subsection{Supersingular CM points and their associated periods}\label{CMpointsection}
Let $D/\mathbb{Q}$ be a quaternion algebra \emph{split at $p$ and $\infty$}, i.e. $D \otimes_{\mathbb{Q}} \mathbb{Q}_p \cong M_2(\mathbb{Q}_p)$ and $D \otimes_{\mathbb{Q}} \mathbb{R} \cong M_2(\mathbb{R})$. Given a congruence subgroup $\Gamma \subset (\mathcal{O}_D \otimes_{\mathbb{Z}}\hat{\mathbb{Z}})^{\times}$ (Definition \ref{congruencesubgroups}), let $Y(\Gamma)$ be as in Section \ref{adicShimurasection}. Recall the notation of Convention \ref{Yconvention}, so that $Y = Y(\Gamma(N))$ where $(N,p) = 1$ and $N$ is coprime with the discriminant of $D$.

\begin{choice}Recall that $D_K \in \mathbb{Z}_{<0}$ is the fundamental discriminant of $K$. Henceforth, let
\begin{equation}\label{pichoice}\varpi := \frac{\frac{D_K}{\mathrm{ord}_p(q)} + \sqrt{D_K}}{2} \in \mathcal{O}_K.
\end{equation}
(Note that by (\ref{ebddefinition}) we have $\mathrm{ord}_p(q) = 2$ if $p = 2$ and $\mathrm{ord}_p(q) = 1$ if $p > 2$.) Then $\varpi \in \frak{p}$ and is a uniformizer of $\mathcal{O}_{K_p}$ and we have 
$$\mathcal{O}_{K_p} = \mathbb{Z}_p + \mathbb{Z}_p\varpi.$$
%In particular, $\varpi$ is a uniformizer of $\mathcal{O}_{K_p}$.
Viewing $\varpi \in \mathcal{H}^+ = \{\tau \in \mathbb{C} : \mathrm{Im}(\tau) > 0\}$, let
\begin{equation}\label{tau0}\tau_0 := \frac{1}{\varpi} \in \mathcal{H}^+.
\end{equation}
In particular as 
$$\varpi = \begin{cases} -1 + i & p = 2, K = \mathbb{Q}(i)\\
-2+\sqrt{-2} & p = 2, K = \mathbb{Q}(\sqrt{-2})\\
-\frac{p+1}{2} + \frac{1 +\sqrt{D_K}}{2} & p > 2\\
\end{cases},$$
by Assumption \ref{pramifiedassumption} we have
\begin{equation}\label{OKtrivialize}\mathcal{O}_K = \mathbb{Z} + \mathbb{Z}\varpi.
\end{equation}
\end{choice}

Recall our fixed maximal order $\mathcal{O}_D$ (Choice \ref{maximalorderchoice}). Pick an embedding
$$\iota_{\infty} : D \hookrightarrow M_2(\mathbb{R}).$$
Fix an embedding
\begin{equation}\label{Kembeddings} \rho : K \hookrightarrow D
\end{equation}
such that $\rho(\mathcal{O}_K) \subset \mathcal{O}_D$ and for any $\alpha \in K$, 
\begin{equation}\label{Kembeddingscondition}\rho(\alpha)_{\infty}\left(\begin{array}{ccc} \tau_0 \\
1\\
\end{array}\right) = \left(\begin{array}{ccc} \alpha \tau_0 \\
\alpha\\
\end{array}\right).
\end{equation}
Here $\rho(\alpha)_{\infty} = i_{\infty}(\rho(\alpha))$. The previous equation implies
\begin{equation}\label{fixpoint}\rho(\alpha)_{\infty}\cdot \tau_0 = \tau_0
\end{equation}
where ``$\cdot$'' denotes the left modular action
\begin{equation}\label{leftmodularaction}\left(\begin{array}{ccc} a & b\\
c & d\\
\end{array}\right)\cdot \tau = \frac{a\tau + b}{c\tau + d}, \hspace{1cm} \left(\begin{array}{ccc} a & b\\
c & d\\
\end{array}\right) \in GL_2(\mathbb{R}).
\end{equation}
(An embedding (\ref{Kembeddings}) satisfying the above is often called normalized.)
Then (\ref{Kembeddings}) induces an embedding
\begin{equation}\label{Kembeddingsadele}\rho : \mathbb{A}_K \hookrightarrow D(\mathbb{A}_{\mathbb{Q}}).
\end{equation}

Let 
$$\Gamma(N) \subset D^{\times}(\mathbb{A}_{\mathbb{Q}}^{(\infty)})$$
be as in Definition \ref{congruencesubgroups}. Henceforth, let 
$$Y = Y(\Gamma(N))$$
as in Convention \ref{Yconvention}. For the remainder of Section \ref{padicLfunctionsection}, we will operate under the following Assumption regarding $N, f_0$ and $D$.

\begin{assumption}\label{Nf0assumption}In the setting of Convention \ref{Yconvention}, we further assume that $N \in \mathbb{Z}_{\ge 4}$, $(N,p) = 1$ is such that $f_0|N$ where $f_0 \ge 4$ is as in Assumption \ref{pconductorassumption} (1). In particular, the quaternion algebra $D/\mathbb{Q}$ is split at all places dividing $pN\infty$.  
\end{assumption}

%\begin{remark}We still have not fixed $N$, and thus $N$ is allowed to vary subject to the condition $f_0|N$, depending on our purposes. However, one may also choose a single $N$ to work uniformly for all subsequent constructions, see Remark \ref{Ncompatibleremark}. 

%\end{remark}

Let us work with a general open compact $\Gamma' \subset D^{\times}(\mathbb{A}_{\mathbb{Q}}^{(\infty)})$ briefly and recall the algebraic Shimura curve $\mathbb{Y}(\Gamma')$ over $\mathbb{Q}$ from Section \ref{algebraicYsection}. Recall the double coset description (\ref{doublequotient})
$$\mathbb{Y}(\Gamma')(\mathbb{C}) = D^{\times} \backslash \mathcal{H}^{\pm} \times D^{\times}(\mathbb{A}_{\mathbb{Q}}^{(\infty)})/\Gamma'.$$
Denote an element of the double quotient appearing above by $[z,\beta]$, where $z \in \mathcal{H}^+$ and $\beta \in D^{\times}(\mathbb{A}_{\mathbb{Q}}^{(\infty)})$ and $(z,\beta)$ is a double coset representative of $[z,\beta]$. 
For any supernatural number $M$, consider the principal congruence subgroup $\Gamma(M)$ from Definition \ref{congruencesubgroups}. Then (\ref{Kembeddings}) induces
\begin{equation}\label{classgroupembedding}\mathcal{C}\ell(M) = K^{\times}\backslash \mathbb{A}_K^{\times,(\infty)}/(1+M\hat{\mathcal{O}}_K) \hookrightarrow D^{\times}\backslash \mathcal{H}^{\pm} \times D^{\times}(\mathbb{A}_{\mathbb{Q}}^{(\infty)})/\Gamma(M) = \mathbb{Y}(\Gamma(M))(\mathbb{C}), \hspace{.6cm} [\frak{b}] \mapsto [\tau_0,\frak{b}]
\end{equation}
where $\hat{\mathcal{O}}_K = \mathcal{O}_K \otimes_{\mathbb{Z}}\hat{\mathbb{Z}}$ and $[\frak{b}] \in \mathcal{C}\ell(M)$ is an ideal class with representative $\frak{b} \in \mathbb{A}_K^{\times,(\infty)} \overset{(\ref{Kembeddingsadele})}{\hookrightarrow} D^{\times}(\mathbb{A}_{\mathbb{Q}}^{(\infty)})$. Given $\frak{b} \in \mathbb{A}_K^{\times,(\infty)}$, let $\sigma_{\frak{b}} \in \mathrm{Gal}(K^{\mathrm{ab}}/K)$ denote the (arithmetically normalized) Artin symbol (recalling here that $K^{\mathrm{ab}}$ denotes the maximal abelian extension of $K$). By Shimura's reciprocity law (\cite[Theorems 6.31 and 6.38]{Shimura}) we have $[\tau_0,1] \in \mathbb{Y}(\Gamma(M))(K^{\mathrm{ab}})$ and 
\begin{equation}\label{Shimuralaw}[\tau_0,1]^{\sigma_{\frak{b}}} = [\tau_0,\frak{b}] = [\tau_0,1] \cdot \frak{b}
\end{equation}
where the right action of $\frak{b}$ is given through right multiplication by $\mathbb{A}_K^{\times,(\infty)} \overset{(\ref{Kembeddingsadele})}{\hookrightarrow} D^{\times}(\mathbb{A}_{\mathbb{Q}}^{(\infty)})$ and (\ref{classgroupembedding}).

Recall Convention \ref{Yconvention}. In particular, $Y(\epsilon)$, $Y^+(\epsilon)$, etc. are as in Sections \ref{adicShimurasection} and \ref{Hassenbhdsection}. Let $\mathcal{Y}^{\mathrm{Ig}}(\epsilon)$ be as in (\ref{U}) and $\mathcal{Y}^{\mathrm{Ig}}(\epsilon)^+$ be as in (\ref{Uplus}). 

\begin{definition}\label{Ofdefinition}For any integer $f$, let $\mathcal{O}_f = \mathbb{Z} + f\mathcal{O}_K \subset \mathcal{O}_K$ denote the order of conductor $f$. Let $\mathcal{O}_{f,p} = \mathcal{O}_f \otimes_{\mathbb{Z}}\mathbb{Z}_p$. 
\end{definition}

%Note that $Y^+ \rightarrow Y(f)^+$ is a finite \'{e}tale map over $\mathbb{Z}_p$. It will be more convenient to work on $Y^+$ in order to apply the results of Sections \ref{notationsection} through \ref{integralitysection}, as $Y^+$ satisfies Choice \ref{globalHasseassumption}. We have a uniformization
%\begin{equation}\label{seconduniformization}Y(\mathbb{C}) = D^{\times}\backslash \mathcal{H}^{\pm} \times D^{\times}(\mathbb{A}_{\mathbb{Q}}^{(\infty)})/U(N)
%\end{equation}
%where $U(N) \subset D^{\times}(\mathbb{A}_{\mathbb{Q}}^{(\infty)})$ is the open compact attached to $\Gamma(N)$.

\begin{choice}\label{choice}When $D = M_2(\mathbb{Q})$, let $A$ be an elliptic curve defined over $\mathcal{O}_K$ (recall $K$ has class number 1) with
$$A(\mathbb{C}) \cong \mathbb{C}/(\mathbb{Z}\tau_0 + \mathbb{Z}).$$
When $D \neq M_2(\mathbb{Q})$, let $A$ be an false elliptic curve defined over $\mathcal{O}_K$ with 
$$A(\mathbb{C}) \cong \mathbb{C}^{\oplus 2}/\left(\iota_{\infty}(\mathcal{O}_D)\cdot \left(\begin{array}{ccc} \tau_0\\
1\\
\end{array}\right)\right)$$
(this is the false elliptic curve $A_{\tau_0}$ of \cite[Section 2.5]{Brooks}). By the theory of complex multiplication (\cite[Chapter II.1.4]{deShalit}), $A$ has CM by $\mathcal{O}_K$ and is defined up to unique isomorphism over $\mathcal{O}_K$ (since $K$ has class number 1). Suppose $D = M_2(\mathbb{Q})$. Then there is some $\Omega \in \mathbb{C}^{\times}$ such that $\Omega(\mathbb{Z}\tau_0 + \mathbb{Z})$ is the period lattice of $A$. The point $[\tau_0,1] \in Y(\mathbb{C})$ is induced by the algebraic point 
$$(A,\frac{\Omega \tau_0}{N},\frac{\Omega}{N}) \in \mathbb{Y}(K(N)),$$
noting that $(\frac{\Omega\tau_0}{N},\frac{\Omega}{N})$ is a $\Gamma(N)$-level structure under $A(\mathbb{C}) = \mathbb{C}/(\Omega(\mathbb{Z}\tau_0+\mathbb{Z}))$. 

%When $D = M_2(\mathbb{Q})$, let $A'$ be the elliptic curve defined over $\mathcal{O}_K$ with
%$$A'(\mathbb{C}) \cong \mathbb{C}/(\mathbb{Z}(\tau_0/p^a) + \mathbb{Z}) = \mathbb{C}/(\varpi^{-1} p^{-a}\mathcal{O}_{p^a}).$$
%When $D \neq M_2(\mathbb{Q})$, let $A'$ be the false elliptic curve defined over $\mathcal{O}_K$ with 
%$$A'(\mathbb{C}) \cong \left(\mathbb{C}/(\mathbb{Z}(\tau_0/p^a) + \mathbb{Z})\right)^{\oplus 2} = \left(\mathbb{C}/(\varpi^{-1} p^{-a} \mathcal{O}_{p^a})\right)^{\oplus 2}$$
%otherwise. By the theory of complex multiplication (\cite[Chapter II.1.4]{deShalit}), $A'$ has complex multiplication by $\mathcal{O}_{p^a}$ and is defined up to unique isomorphism over $\mathcal{O}_K$. The point $[\tau_0/p^a,1] \in Y(\mathbb{C})$ is induced by the point 
%$$(A',\frac{1}{N},\frac{\tau_0}{p^aN}) \in Y(K(\frak{f}^{(p)}))$$
%under the moduli interpretation. 
\end{choice}

\begin{remark}Note that in the totally split case $D = M_2(\mathbb{Q})$ one can take the choice $A = E$, where $E$ is the CM elliptic curve as in Assumption \ref{pconductorassumption} (3). We will make this particularly convenient choice later in Choice \ref{AEchoice}. 
\end{remark}

Given $m \in \mathbb{Z}_{> 0}$, let $K[m]$ denote the ring class field of conductor $m$. Thus (since $K$ has class number 1), 
$$\mathrm{Gal}(K[m]/K) \cong \mathcal{C}\ell(1)[m] \cong (\mathcal{O}_K/m\mathcal{O}_K)^{\times}/((\mathbb{Z}/m)^{\times}\mathcal{O}_K^{\times})$$
where $\mathcal{C}\ell(1)[m]$ is the ring class group of $K$ of conductor $m$. From the theory of complex multiplication and by examining formal groups, one sees that $A[\varpi]$ lifts the kernel of $p$-power Frobenius on $A$ modulo $\varpi\mathcal{O}_{K[f_0']}$ and that the leading coefficient of the isogeny of formal groups $\hat{A} \rightarrow \hat{A}/\hat{A}[\varpi]$ has $p$-adic absolute value $p^{-1/2}$. Thus (cf. \cite[Chapter 3]{Katzpamf})
$$|\mathrm{Ha}(A)| = p^{-1}/p^{-1/2} = p^{-1/2}.$$
Thus, $A$ has a canonical subgroup of order $p$. 

%Since $A[\varpi]$ lifts the kernel of $p$-power relative Frobenius on $A$ modulo $\varpi \mathcal{O}_K$, we have 
%$$|\mathrm{Ha}(A)| = p^{-1/2}.$$
%Thus $A$ has a canonical subgroup of order $p$. %Let $[\cdot]_A : \mathcal{O}_K \xrightarrow{\sim} \mathrm{End}(A/\mathcal{O}_K)$ (resp. $e\mathrm{End}(A/\mathcal{O}_K)$ if $D \neq M_2(\mathbb{Q})$) denote the CM action on $A$.
Recall that we will often suppress the $\Gamma = \Gamma(N)$-level structure from notation as per Convention \ref{tameconvention}.

\begin{choice}\label{CMpointchoice}
\begin{enumerate}
\item Henceforth fix $\epsilon_0 > 0$ in the valuation group of $k$ (see Definition \ref{kdefinition}) with
\begin{equation}\label{rchoice}1/(p+1) \le \epsilon_0  < p/(p+1).
\end{equation}
Thus $\epsilon_0$ satisfies the assumptions of Definition \ref{Udefinition} and moreover satisfies Assumption \ref{nepsilonassumption}:
$$n(\epsilon_0) \overset{(\ref{nepsilondefinition})}{=} 1.$$
\item Let $\varpi \in \mathcal{O}_K$ be as in (\ref{pichoice}). Choose a trivialization
$$(e_1,e_2) : \mathbb{Z}^{\oplus 2} \xrightarrow{\sim} e^1H_1(A(\mathbb{C}),\mathbb{Z})$$
(where $H_1$ denotes singular homology and $e^1 = \left(\begin{array}{ccc} 1 & 0\\
0 & 0\\
\end{array}\right)$ is as in Convention \ref{idempotentconvention}) such that the corresponding trivialization (induced by taking the dual of the Artin comparison isomorphism between singular and \'{e}tale cohomology)
$$(e_1,e_2) : \mathbb{Z}_p^{\oplus 2} \xrightarrow{\sim} T_pA$$
(see Convention \ref{idempotentconvention} (2) for the definition of $T_pA$) satisfies
\begin{equation}\label{erelation}[\varpi]_A(e_2) = e_1,
\end{equation}
where $[\cdot ]_A : \mathcal{O}_K \xrightarrow{\sim} \mathrm{End}(A/K)$ is the CM action. 
\item In particular, in the notation of (\ref{eindefinition}), we have that $e_{1,n} \in A[p^n]$ is a primitive $p^n/\varpi$-torsion point and $e_{2,n} \in A[p^n]$ is a primitive $p^n$-torsion point. 
%and such that $e_{2,n} = e_2 \pmod{p^n}$ corresponds to $\frac{1}{p^n}$ under the isomorphism $A(\mathbb{C}) \cong \mathbb{C}/(\mathbb{Z}\tau_0 + \mathbb{Z})$ if $D = M_2(\mathbb{Q})$, and to $\frac{1}{p^n}$ under the isomorphism $e^1A(\mathbb{C}) \cong \mathbb{C}/(\mathbb{Z}\tau_0 + \mathbb{Z})$ if $D \neq M_2(\mathbb{Q})$. 
Also, let $P$ be the $\Gamma = \Gamma(N)$-level structure on $A$ from Choice \ref{choice} (which we continue to suppress per Convention \ref{tameconvention}). 
\item Recall $\mathbb{Y}_{\infty}(\Gamma)$ from (\ref{algebraicYinftyGamma}) as well as the notation from Convention \ref{Yconvention} and let
$$y := (A,e_1,e_2) \in \mathbb{Y}_{\infty}(\Gamma)(K^{\mathrm{ab}}) \subset Y_{\infty}(\overline{\mathbb{Q}}_p,\overline{\mathbb{Z}}_p).$$
%Note that $\langle e_1 \pmod{p}\rangle$ is the kernel of the isogeny $A \rightarrow A'$ dual to the natural projection $A' \rightarrow A$. Define a trivialization
%$$(e_1',e_2') : \mathbb{Z}_p^{\oplus 2} \xrightarrow{\sim} T_pA'$$
%by
%$$(A',e_1',e_2') = (A,e_1,e_2) \cdot \left(\begin{array}{ccc} 1 & 0\\
%0 & 1/p^a
%\end{array}\right).$$
%Recall the notation $\mathbb{Y}(\Gamma)$ of Section \ref{algebraicYsection}. Then
%\begin{equation}\label{pointy}y := (A',e_1',e_2') \in \mathbb{Y}(\Gamma(p^{\infty}))(\overline{\mathbb{Q}}) \subset Y_{\infty}(\overline{\mathbb{Q}}_p,\overline{\mathbb{Z}}_p),
%\end{equation}
%and by construction 
%$$y = (A,e_1,e_2) \cdot \left(\begin{array}{ccc} 1 & 0\\
%0 & 1/p^a
%\end{array}\right), \hspace{1cm} [\varpi p^a]_{A'}(e_2') = e_1'.$$
%Since the CM actions of $\mathcal{O}_{K_p}$ on $H^{1,0}(A)$ is equal to multiplication by $\mathcal{O}_{K_p}$, (\ref{erelation}) implies
%\begin{equation}\label{pidefinition}z_{\mathrm{dR}}(y) = \frac{1}{\varpi}.
%\end{equation}

\item Recall $1/(p+1) \le \epsilon_0 < p/(p+1)$ (see (\ref{rchoice})). Since $\mathbb{Z}_p \cdot e_{1,1}  = A[\pi]$ by (\ref{erelation}) (see (\ref{eindefinition}) for the definition of the notation $e_{1,1}$), which is the canonical subgroup of $A$, then by Definition \ref{plusdefinitions} we have 
\begin{equation}\label{YIgin1} y \in \mathcal{Y}^{\mathrm{Ig}}(\epsilon_0)(\overline{\mathbb{Q}}_p,\overline{\mathbb{Z}}_p).
\end{equation}

\item Define
\begin{equation}\label{pointyn}y_n := y \cdot g^{-n} = y \cdot \left(\begin{array}{ccc} 1 & 0 \\
0 & 1/p^n\\
\end{array}\right) \in \mathbb{Y}(\Gamma(p^{\infty}))(K^{\mathrm{ab}}) \subset Y_{\infty}(\overline{\mathbb{Q}}_p,\overline{\mathbb{Z}}_p)
\end{equation}
and write 
%$$y_n = (A_n',e_1'(n),e_2'(n)).$$
$$y_n = (A_n,e_1^n,e_2^n).$$
where $(e_1^n,e_2^n) : \mathbb{Z}_p^{\oplus 2} \xrightarrow{\sim} T_pA_n$. Then $A_n$ has CM by the order $\mathcal{O}_{p^n}$, and from (\ref{pointyn}) and (\ref{leftmodularaction}) one has
\begin{equation}\label{An'complex}A_n(\mathbb{C}) \cong \begin{cases} \mathbb{C}/\left(\mathbb{Z}\frac{\tau_0}{p^n} + \mathbb{Z} \right) & D = M_2(\mathbb{Q})\\
\mathbb{C}/\left(\iota_{\infty}(\mathcal{O}_D)\cdot \left(\begin{array}{ccc} \frac{\tau_0}{p^n}\\
1\\
\end{array}\right)\right) & D \neq M_2(\mathbb{Q})\\
\end{cases},
\end{equation}
which is the false elliptic curve $A_{\tau_0/p^n}$ in \cite[Section 2.5]{Brooks}. 
\end{enumerate}

%Moreover
%$$z_{\mathrm{dR}}(y_n) \overset{(\ref{modulartransformationidentity})}{=} \frac{z_{\mathrm{dR}}(y)}{p^n} = \frac{\tau_0}{p^{a + n}}.$$
%Thus by Lemma \ref{domainlemma}, $\langle e_1'(n) \pmod{p^{n+a + 1}}\rangle$ is the canonical subgroup of $A_n'$ of order $p^{n + a+1}$. From this, one can compute using the fact that the order $p^{n+a + 1}$ canonical subgroup lifts $p^{n + a + 1}$-power relative Frobenius that
%$$|\mathrm{Ha}(A_n')| = p^{-1/(2p^{n + a})}.$$ 
%Note also that 
%$$y_n = y\cdot \left(\begin{array}{ccc} 1 & 0\\
%0 & 1/p^n\\
%\end{array}\right) \overset{(\ref{YIgin1})}{\in} \mathcal{Y}^{\mathrm{Ig}}(\epsilon_0/p^n),
%$$
%and that the isogeny $A' \rightarrow A_n'$ underlying the action of $\left(\begin{array}{ccc} 1 & 0\\
%0 & p^{n+ a}\\
%\end{array}\right)$ is the dual to the natural projection 
%\begin{equation}\label{pin}\varpi_n : A_n' \rightarrow A'
%\end{equation} induced by (\ref{An'complex}). Equivalently, $\varpi_n$ is the isogeny underlying the action $y_n \rightarrow y_n \cdot \left(\begin{array}{ccc} 1 & 0\\
%0 & p^n\\
%\end{array}\right) = y$. 
\end{choice}

\begin{remark}Note that since $p \ge 2$, one can always take the choice $\epsilon_0 = 1/2$ in (\ref{rchoice}). Later in Choice \ref{rchoice2} we will take this choice to simplify certain arguments that follow. 
\end{remark}

%Since $A_n'$ has CM by an order of $K$, it has potentially good reduction everywhere and hence a canonical integral model $A_{n,+}'$ after finite base change, and moreover over $\mathcal{O}_{\mathbb{C}_p}$ the trivialization $(e_1'(n),e_2'(n))$ extends to a Drinfeld level structure 
%$$(e_1'(n)_+,e_2'(n)_+) : \mathbb{Z}_{p,\mathcal{O}_{\mathbb{C}_p}}^{\oplus 2} \rightarrow T_pA_{n,+}'(\overline{\mathbb{Z}}_p).$$
%This gives a point 
%\begin{equation}\label{pointchoice}y_{n,+} \in \mathcal{Y}^{\mathrm{Ig}}(\epsilon_0/p^n)^+(\overline{\mathbb{Z}}_p,\overline{\mathbb{Z}}_p).
%\end{equation}
%We will often simply write $y_n$ for $y_{n,+}$, when no confusion should arise. 

\begin{definition}\label{mixedclassgroupdefinition}
\begin{enumerate}
\item Recall we assumed that $K$ has class number 1. %Recall the notation of Definition \ref{Ofdefinition}. 
 For any integer $m \in \mathbb{Z}_{\ge 1}$ and any $\frak{f}' \subset \mathcal{O}_K$ with $(\frak{f}',m) = 1$, let
$$\mathcal{C}\ell(\frak{f}')[m] = \frac{(\mathcal{O}_K/\frak{f}')^{\times} \times (\mathcal{O}_K/m)^{\times}/(\mathbb{Z}/m)^{\times}}{\mathcal{O}_K^{\times}},$$
where $\mathcal{O}_K^{\times}$ maps diagonally to the product in the numerator. This is the mixed ring/ray class group of conductor $m$ and modulus $\frak{f}'$. 

\item Suppose that $m = p^{a+n}m'$, $(m',p) = 1$, where $a$ is as in (\ref{adefinition}) and $n \in \mathbb{Z}_{\ge 0}$. From (\ref{adefinition}) and (\ref{betadelta}) we see that the natural map $\mathcal{O}_K^{\times} \rightarrow (\mathcal{O}_K/p^{a+n})^{\times}$ is injective and so we have
$$\mathcal{C}\ell(\frak{f}')[p^{a+n}m'] \cong (\mathcal{O}_K/\frak{f}')^{\times} \times \left((\mathcal{O}_K/p^{a+n})^{\times}/((\mathbb{Z}/p^{a+n})^{\times}\mathcal{O}_K^{\times})\right)\times \left((\mathcal{O}_K/m')^{\times}/(\mathbb{Z}/m')^{\times}\right).$$
%The above group is the mixed ring/ray class group of conductor $p^{a+n}m'$ and modulus $\frak{f}'$. %When $w_{\frak{f}^{(p)}} = 1$ (which is the case for all but finitely many $\frak{f}$), this group is isomorphic to 
%$$\mathcal{C}\ell(\frak{f}^{(p)}) \times \mathbb{Z}/p^a.$$
\end{enumerate}
\end{definition}

\begin{definition}\label{Sdefinition}%Let $\frak{f}'$ be any ideal with $(\frak{f}',p) = 1$. 
Recall $\frak{f}$ from Assumption \ref{pconductorassumption}, and that $\frak{f}^{(p)}$ is its prime-to-$p$ part. Fix a set of everywhere integral id\`{e}les
$$S \subset \mathbb{A}_K^{\times,(pN\infty)} \overset{\rho}{\subset} D^{\times}(\mathbb{A}_{\mathbb{Q}}^{(\infty)})$$
(i.e. $S$ is a set of finite id\`{e}les which are prime to $pN$) and such that $S$ is a full set of representatives of $\mathcal{C}\ell(\frak{f}^{(p)})[p^a]$ under the map
$$\mathbb{A}_K^{\times,(\infty)} \twoheadrightarrow K^{\times}\backslash \mathbb{A}_K^{\times,(\infty)}/\left((\mathbb{Z}_p + p^a\mathcal{O}_{K_{\frak{p}}})^{\times}\cdot \prod_{v\nmid \frak{p}\infty}(1+\frak{f}^{(p)}\mathcal{O}_{K_v})\right) \cong \mathcal{C}\ell(\frak{f}^{(p)})[p^a].$$
For simplicity, we will further assume that 
$$1 \in S$$
(which thus represents the trivial class in $\mathcal{C}\ell(\frak{f}^{(p)})[p^a]$). %When $\frak{f}' = \frak{f}^{(p)}$ (where $\frak{f}$ is as in Assumption \ref{pconductorassumption}), we will let 
%$$S = S(\frak{f}^{(p)}).$$
\end{definition}

\begin{choice}\label{choice2}
\begin{enumerate}
\item Given $\frak{b} \in S$ (see Definition \ref{Sdefinition}), let $\sigma_{\frak{b}} \in \mathrm{Gal}(K^{\mathrm{ab}}/K)$ denote the (arithmetically normalized) Artin symbol of $\frak{b}$. Letting $E$, $\lambda_E$ be as in Assumption \ref{pconductorassumption} (3), define
\begin{equation}\label{YIgin}y_{\frak{b}} := y \cdot \frak{b} \overset{(\ref{Shimuralaw})}{=} y^{\sigma_{\frak{b}}}\in \mathbb{Y}(\Gamma(p^{\infty}))(K^{\mathrm{ab}}) \subset Y_{\infty}(\overline{\mathbb{Q}}_p,\overline{\mathbb{Z}}_p).
\end{equation}
For $\frak{b} = 1$, we have $y_1 = y$ by Definition \ref{Sdefinition}. Since the id\`{e}le $\frak{b}$ has trivial component at $\frak{p}$ in $\mathcal{O}_{K_p}^{\times}$, the $p$-component 
$$\rho(\frak{b})_p \in \mathcal{O}_{K_p}^{\times} \overset{\rho}{\hookrightarrow} GL_2(\mathbb{Q}_p).$$ %By (\ref{pidefinition}) and (\ref{fixpoint}) we have 
%\begin{equation}\label{ybz}z_{\mathrm{dR}}(y_{\frak{b}}) = z_{\mathrm{dR}}(y\cdot \frak{b}) = z_{\mathrm{dR}}(y \cdot \rho(\frak{b})_p) = z_{\mathrm{dR}}(y) \overset{(\ref{pidefinition})}{=} \frac{1}{\varpi}.
%\end{equation}

\item For any $\frak{b} \in S$, let $A_{\frak{b}} = A/A[\frak{b}]$ be the elliptic curve or false elliptic curve underlying $y_{\frak{b}}$.  Let 
\begin{equation}\label{pib}\pi_{\frak{b}} : A\rightarrow A_{\frak{b}}
\end{equation}
be the natural projection; this is equal to the isogeny underlying the right action of $\frak{b} \in S$. 

\item Since $\frak{b}$ is prime to $p$, we have 
$$y_{\frak{b}} \in \mathcal{Y}^{\mathrm{Ig}}(\epsilon_0)$$
from (\ref{YIgin1}), which by $\mathcal{Y}^{\mathrm{Ig}}(\epsilon_0/p^n) = \mathcal{Y}^{\mathrm{Ig}}(\epsilon_0) \cdot g^{-n}$ (from (\ref{gisomorphism})) implies 
\begin{equation}\label{YIgin2}y_{\frak{b},n} := y_{\frak{b}} \cdot g^{-n} \overset{(\ref{gdefinition})}{=} y_{\frak{b}} \cdot \left(\begin{array}{ccc} 1 & 0\\
0 & 1/p^n\\
\end{array}\right) \in \mathcal{Y}^{\mathrm{Ig}}(\epsilon_0/p^n)(\overline{\mathbb{Q}}_p,\overline{\mathbb{Z}}_p).
\end{equation}
When $\frak{b} = 1$, we let $y_n = y_{\frak{b},n}$. 

%Since 
%$$\frak{b} \in S \subset \mathbb{A}_K^{\times,(p\infty)} \subset GL_2(\mathbb{A}_{\mathbb{Q}}^{(p\infty)})$$
%s prime to $p$, it acts trivially on the $\Gamma(p^{\infty})$-level structure $(e_1(y_{\frak{b}}),e_2(y_{\frak{b}})) : \mathbb{Z}_p^{\oplus 2} \xrightarrow{\sim} T_pA_{\frak{b}}$. Thus we have 
%\begin{equation}\label{ybz2}z_{\mathrm{dR}}(y_{\frak{b},n}) = z_{\mathrm{dR}}(y_{\frak{b}}\cdot g^{-n}) \overset{(\ref{ybz})}{=} \frac{1}{p^n\varpi}.
%\end{equation}
%Hence
%\begin{equation}\label{zin2}z_{\mathrm{dR}}(y_{\frak{b},n}) = z_{\mathrm{dR}}(y_{\frak{b}} \cdot \left(\begin{array}{ccc} 1 & 0\\
%0 & 1/p^n\\
%\end{array}\right)) \overset{(\ref{modulartransformationidentity})}{=} \frac{z_{\mathrm{dR}}(y_{\frak{b}})}{p^n} = \frac{\tau_0}{p^{n+a}}.
%\end{equation}

\item Let $A_{\frak{b},n}$ be the elliptic curve or false elliptic curve underlying $y_{\frak{b},n}$. When $\frak{b} = 1$, let $A_n = A_{\frak{b},n}$. Let 
\begin{equation}\label{pin2}\pi_{\frak{b},n} : A_{\frak{b},n} \rightarrow A_{\frak{b}}
\end{equation}
denote the natural isogeny given by division by the order-$p^n$ canonical subgroup. 
%For $\frak{b}$ representing the trivial ideal class in $\mathcal{C}\ell(\frak{f}^{(p)})[p^a]$, is simply the isogeny (\ref{pin}) up to isomorphism. 

%Recall $a$ from (\ref{ebddefinition}). We now let 
%\begin{equation}\label{gammachoice}\gamma = 1 + \varpi p^a \in \Gamma.
%\end{equation}
%Define a sequence $(1,\zeta_p,\zeta_{p^2},\ldots,\zeta_{p^n},\ldots)$ of $p^{\mathrm{th}}$-power roots of unit by
%\begin{equation}\label{orientation}\langle e_{1,n},e_{2,n}\rangle = \zeta_{p^n}.
%\end{equation}
%Thus $\zeta_{p^n}^p = \zeta_{p^{n-1}}$ for all $n \in \mathbb{Z}_{\ge 1}$. 
\end{enumerate}
\end{choice}

%In summary, our choice of $\varpi$ from (\ref{pichoice}) and $y$ from (\ref{pointchoice}) determines $\gamma$, $(1,\zeta_p,\zeta_{p^2},\ldots)$ and thus determines the additive character $\Psi$ by (\ref{additivecharacter}). 

\begin{choice}
\begin{enumerate}
\item Let $A$ be the CM (false) elliptic curve as in Choice \ref{CMpointchoice}; then $A$ has a model $A/\overline{\mathbb{Q}}$. We henceforth fix a \emph{generating} invariant differential
\begin{equation}\label{fixdifferential}w_0(A) \in \Omega_{A/\overline{\mathbb{Q}}}.
\end{equation}
\item For any $\frak{b} \in S$, let
\begin{equation}\label{fixdifferentialb}w_0(A_{\frak{b}}) \in \Omega_{A_{\frak{b}}/\overline{\mathbb{Q}}}
\end{equation}
be the unique differential such that for the isogeny $\pi_{\frak{b}} : A \rightarrow A_{\frak{b}}$ from (\ref{pib}),
\begin{equation}\label{bpullback}\pi_{\frak{b}}^*w_0(A_{\frak{b}}) = w_0(A).
\end{equation}

\item Recall the isogeny $\pi_{\frak{b},n} : A_{\frak{b},n} \rightarrow A_{\frak{b}}$ from (\ref{pin2}), and let $\pi_{\frak{b},n}^{\vee}: A_{\frak{b}} \rightarrow A_{\frak{b},n}$ be its dual isogeny. Both $\pi_{\frak{b},n}$ and $\pi_{\frak{b},n}^{\vee}$ are cyclic degree $p^n$ isogenies. Define 
\begin{equation}\label{pullbacknormalization}w_0(A_{\frak{b},n}) := \frac{1}{p^n}\pi_{\frak{b},n}^*w_0(A_{\frak{b}}) \in \Omega_{A_{\frak{b},n}'/\overline{\mathbb{Q}}},
\end{equation}
so that 
$$(\pi_{\frak{b},n}^{\vee})^*w_0(A_{\frak{b},n}) = w_0(A_{\frak{b}}).$$
%When $\frak{b} = 1$, recall that we let $A_n = A_{\frak{b},n}$. 
\end{enumerate}
\end{choice}

%\begin{lemma}\label{pullbackstable}We have
%\begin{equation}\label{pullbackcanonicaldifferentials}(\pi_{\frak{b},n}^{\vee})^*\tilde{w}_{\mathrm{can}}(y_{\frak{b},n}) = \tilde{w}_{\mathrm{can}}(y_{\frak{b}}), \hspace{1cm} (\pi_{\frak{b},n}^{\vee})^*(2\pi idz)(y_{\frak{b},n}) = (2\pi idz)(y_{\frak{b}}).
%\end{equation}
%\end{lemma}

%\begin{proof}From the proof of \cite[Lemma 2.11]{ChojeckiHansenJohansson}, noting that $\pi_{\frak{b},n}^{\vee}$ is the isogeny corresponding to $\left(\begin{array}{ccc} 1 & 0\\
%0 & p^n\\
%\end{array}\right)$, we see that the dual isogeny satisfies
%$$\pi_{\frak{b},n}^*\mathrm{HT}(e_1(y_{\frak{b}})) = p^n\mathrm{HT}(e_1(y_{\frak{b},n}))$$
%which implies
%$$\pi_{\frak{b},n}^*w_{\mathrm{can}}(y_{\frak{b}}) = p^nw_{\mathrm{can}}(y_{\frak{b},n}).$$
%Thus, since $\pi_{\frak{b},n}$ has degree $p^n$, we have
%$$(\pi_{\frak{b},n}^{\vee})^*w_{\mathrm{can}}(y_{\frak{b},n}) = w_{\mathrm{can}}(y_{\frak{b}}).$$
%Recall $\tilde{w}_{\mathrm{can}} = \theta \circ \theta_X(w_{\mathrm{can}})$. Applying $\theta \circ \theta_X$ to the above equation, we get the first equality of (\ref{pullbackcanonicaldifferentials}). 

%Viewing the CM points $y_{\frak{b}}$ and $y_{\frak{b},n}$ as points in $\mathbb{Y}(\overline{\mathbb{Q}})$ briefly, the same argument gives the second equality of (\ref{pullbackcanonicaldifferentials}). 

%\end{proof}

\begin{definition}\label{finalperiodsdefinition}Define the periods
$$\Omega_p(n) := \Omega_p(y_n) \in \mathbb{C}_p^{\times}, \hspace{1cm} \Omega_{\infty}(n) := \Omega_{\infty}(y_n) \in \mathbb{C}^{\times}$$
to be those from (\ref{defineomegaperiods}) defined with respect to the point 
$$y_n = (A_n,e_1^n,e_2^n) := (A,e_1,e_2) \cdot g^{-n}$$
from (\ref{YIgin2}) and the differential $w_0(A_n)$ from (\ref{pullbacknormalization}) (with $\frak{b} = 1 \in S$). 
\end{definition}

Consider the equations (\ref{defineomegaperiods}) for the point $y_{\frak{b},n}$ and the differential $w_0(A_{\frak{b},n})$:
\begin{equation}\label{diffeq1}\tilde{w}_{\mathrm{can}}(y_{\frak{b},n}) = \Omega_p(y_{\frak{b},n}) \cdot w_0(A_{\frak{b},n}), \hspace{1cm} (2\pi idz)(y_{\frak{b},n}) = \Omega_{\infty}(y_{\frak{b},n}) \cdot w_0(A_{\frak{b},n}).
\end{equation}
%Note that since 1 and $\frak{b}$ have the same ideal class (since $K$ has class number 1), there is a unique $K[p^n]$-isomorphism $i_{\frak{b},n} : A_{\frak{b},n} \xrightarrow{\sim} A_n^{\sigma_{\frak{b},n}}$, where $K[p^n]$ denotes the ring class field of conductor $p^n$ over $K$ and $\sigma_{\frak{b},n} \in \mathrm{Gal}(K[p^n]/K)$ is the Artin symbol of $\frak{b} \in S$. Hence, suppressing $\Gamma = \Gamma(N)$-level structure (per Convention \ref{tameconvention}) and 
%Writing
%$$y_{\frak{b},n} = (A_{\frak{b},n},e_{1,\frak{b}}^n,e_{2,\frak{b}}^n)$$
%where $(e_{1,\frak{b}}^n,e_{2,\frak{b}}^n)$ is a $\Gamma(p^{\infty})$-level structure on $A_{\frak{b},n}$, by the theory of complex multiplication we have 
Since $\frak{b}$ is prime to $pN$ we have, by \cite[II.1.5 (15)]{deShalit},
\begin{equation}\label{compareeis}y_{\frak{b},n} = (A_{\frak{b},n},\pi_{n,\frak{b}}(e_1),\pi_{n,\frak{b}}(e_2)) = \pi_{n,\frak{b}}(y_n)
%(A_{\frak{b},n},e_{1,\frak{b}}^n,e_{2,\frak{b}}^n) \overset{i_{\frak{b},n}}{\cong} (A_n^{\sigma_{\frak{b},n}},[\alpha_n(\frak{b})]_{A_n^{\sigma_{\frak{b},n}}}(e_1^{\sigma_{\frak{b},n}}),[\alpha_n(\frak{b})]_{A_n^{\sigma_{\frak{b},n}}}(e_2^{\sigma_{\frak{b},n}})) = y_n^{\sigma_{\frak{b},n}} \cdot \alpha_n(\frak{b})
\end{equation}
%$$y_{\frak{b},n} = (A_{\frak{b},n},(e_1^n,e_2^n)\cdot \frak{b}) = (A_{\frak{b},n},[\kappa_{A_n}(\frak{b})^{-1}]_{A_n}(e_1^n),[\kappa_{A_n}(\frak{b})^{-1}]_{A_n}(e_2^n))$$
%by Shimura's reciprocity law 
%, where 
%$$\kappa_{A_n} : \mathbb{A}_K^{\times} \rightarrow \mathrm{Gal}(K(A_n[p^{\infty}])/K) \hookrightarrow \mathrm{End}(A_n/\mathcal{O}_{K_p}) = \mathcal{O}_{p^n,p}^{\times}$$
%is the reciprocity character attached to $A_n$ and 
%for some unique 
%$$\alpha_n(\frak{b}) \in\mathcal{O}_{p^n,p}^{\times} \overset{(\ref{Kembeddingsadele})}{\subset} GL_2(\mathbb{Z}_p).$$
where 
$$\pi_{n,\frak{b}} : A_n \rightarrow A_n/A_n[\frak{b}] \cong A_{\frak{b},n}$$
is the natural projection.
%Here $[\cdot]_{A_n^{\sigma_{\frak{b},n}}} : \mathcal{O}_{p^n,p} \hookrightarrow \mathrm{End}(A_n^{\sigma_{\frak{b},n}}[p^{\infty}]/\mathcal{O}_{K_p})$ is the CM action. %Since $\pi_{\frak{b},n}^* : \Omega_{A_n/\mathcal{O}_{K_p}} \cong \Omega_{A_{\frak{b},n}/\mathcal{O}_{K_p}} \rightarrow \Omega_{A_n/\mathcal{O}_{K_p}}$ is multiplication by $\alpha_n(\frak{b})$, then $\pi_{\frak{b},n}^*$ maps the Hodge filtration to the Hodge filtration and 
From (\ref{compareeis}), the fact that the Hodge filtration (\ref{Hodgefiltrationinclusion}) is preserved by isogenies, and the fact that $z_{\mathrm{dR}}$ measures the position of the Hodge filtration (\ref{iwcan}), we have 
$$z_{\mathrm{dR}}(y_{\frak{b},n}) = z_{\mathrm{dR}}(y_n).$$
Thus by (\ref{iwcan}) we have
$$\pi_{\frak{b},n}^*(w_{\mathrm{can}}(y_{\frak{b},n})) = w_{\mathrm{can}}(y_n).$$
Therefore, by (\ref{tildewcan}) we have
$$\pi_{\frak{b},n}^*(\tilde{w}_{\mathrm{can}}(y_{\frak{b},n})) = \tilde{w}_{\mathrm{can}}(y_n).$$
Similarly, $\pi_{\frak{b},n} : A_n(\mathbb{C}) = \mathbb{C}/\mathcal{O}_{p^n} \rightarrow \mathbb{C}/\frak{b}^{-1}\mathcal{O}_{p^n} = A_{\frak{b},n}(\mathbb{C})$ is the natural projection, and so we have
$$\pi_{\frak{b},n}^*((2\pi idz)(y_{\frak{b},n})) = (2\pi idz)(y_n).$$
Thus, applying $\pi_{\frak{b},n}^*$ to (\ref{diffeq1}) and using (\ref{bpullback}), we get
$$\tilde{w}_{\mathrm{can}}(y_n) = \Omega_p(y_{\frak{b},n}) \cdot w_0(A_n), \hspace{1cm} (2\pi idz)(y_n) = \Omega_{\infty}(y_{\frak{b},n}) \cdot w_0(A_n).$$
Thus the equations (\ref{defineomegaperiods}) for the point $y_n$ and the differential $w_0(A_n)$ imply:
\begin{equation}\label{Omegaynequal}\Omega_p(y_{\frak{b},n}) = \Omega_p(n), \hspace{1cm} \Omega_{\infty}(y_{\frak{b},n}) = \Omega_{\infty}(n)
\end{equation}
for all $\frak{b} \in S$.

\subsection{Form of the fundamental power series}Our $p$-adic $L$-functions will be constructed through considering $p$-adic Maass-Shimura deriviatves of a certain ``fundamental power series'' in $q_{\mathrm{dR}}-1$ coming from $q_{\mathrm{dR}}$-expansions of Eisenstein series at CM points. Let $\lambda$ be the $(1,0)$ Hecke over $K$ satisfying Assumption \ref{pconductorassumption}. Recall the full set of representatives $S \subset \mathbb{A}_K^{\times,(\infty)}$ of $\mathcal{C}\ell(\frak{f}^{(p)})[p^a]$ from Definition \ref{Sdefinition}. We will ultimately define the fundamental power series 
$$\mathcal{L}_{\lambda} \in \mathcal{O}_{\mathbb{C}_p} + (q_{\mathrm{dR}}-1)\mathbb{C}_p\llbracket q_{\mathrm{dR}}-1\rrbracket$$
as a sum
\begin{equation}\label{sum1}\mathcal{L}_{\lambda} = \sum_{\frak{b} \in S}\lambda^{-1}(\frak{b})\cdot\mu_{\frak{b}},
\end{equation}
where each $\mu_{\frak{b}} \in \mathcal{O}_{\mathbb{C}_p} + (q_{\mathrm{dR}}-1) \mathbb{C}_p\llbracket q_{\mathrm{dR}}-1\rrbracket$ is the $q_{\mathrm{dR}}$-expansion of an appropriate weight 1 Eisenstein series (see Section \ref{Eisensteinsection}) specialized to an appropriate CM point $y_{\frak{b},a}$ (see (\ref{YIgin2}) with $n = a$). %From $\mu_{\lambda}$, the $p$-adic $L$-function $\mathcal{L}_{\lambda}$ will be obtained via pushforward along the natural quotient $\mathcal{C}\ell(\frak{f}^{(p)})[p^a]  \times \mathbb{Z}_p \rightarrow \mathbb{Z}_p$. 

\subsection{Eisenstein series}\label{Eisensteinsection}
%In this section, we recall the construction of weight 1 Eisenstein series from Kato-Siegel $\Theta$-functions, which will be needed for the construction of $p$-adic $L$-functions of Hecke characters.

Assume that the quaternion algebra $D = M_2(\mathbb{Q})$ in this section so that in particular $Y$ is a modular curve and $\mathcal{E} \rightarrow Y$ is the universal elliptic curve with $\Gamma = \Gamma(N)$-level structure (see Convention \ref{Yconvention}). Let $L \subset \mathbb{C}$ be a rank 2 $\mathbb{Z}$-lattice. For $k \in \mathbb{Z}_{\ge 1}$, let $E_k(z,L)$ be the Eisenstein series defined in \cite[II.3.1 (5)]{deShalit} (for $k = 1, 2$, one defines $E_k(z,L)$ as $\left(-\frac{d}{dz}\right)^k\Theta(z,L)$ where $\Theta(z,L)$ is defined as in Chapter II.2.1 (8) of op. cit.). Recall the quadratic character $\eta$ attached to the imaginary quadratic field $K$, which has conductor $|D_K| = p^{\beta}$. 

\begin{definition}[Weight 1 Eisenstein series]\label{Eisensteindefinition}Recall $\eta = \lambda|_{\mathbb{A}_{\mathbb{Q}}^{\times}}\mathbb{N}_{\mathbb{Q}}^{-1}$ is the central character of $\lambda$ (see Assumption \ref{pconductorassumption} (2)). %By (\ref{betadelta}), we have $a \le \beta$, and so $\eta$ can be viewed as an imprimitive Dirichlet character of conductor $p^a$. 

Let 
\begin{equation}\label{bracketpdefinition}[p] := \left(\begin{array}{ccc} p & 0\\
0 & p\\
\end{array}\right) \in GL_2(\mathbb{Q})
\end{equation}
and $[p^{-1}] = [p]^{-1}$. This is compatible with the notation of $[p]$ from \cite[Section 3.8]{BDP}. Recall 
$$V_p := \left(\begin{array}{ccc} 1 & 0\\
0 & p \\
\end{array}\right) \in GL_2(\mathbb{Q})$$
from Definition \ref{flatdefinition}. 

\begin{enumerate}
\item Let 
$$L = \mathbb{Z}\tau + \mathbb{Z}$$
for $\tau \in \mathcal{H}^+$. Write
$$E_{1,\eta}(z;L) = \sum_{i = 0}^{p^a-1}\eta(i)\cdot E_1(iz,L) = \sum_{i = 0, (i,p) = 1}^{p^a-1}\eta(i)\cdot E_1(iz,L)$$
recalling for the second equality that $\eta(i) = 0$ if $(i,p) \neq 1$ by convention since $\eta$ has conductor $p^a$. Let
$$\mathbf{w}_{1,\eta}^{\mathrm{an}} = E_{1,\eta}(z;L)\cdot dz.$$
%Here, the $p^a\tau$ appears due to the usual normalization of the $q$-expansion of a form of level $\Gamma_1(p^a)$ (see \cite[2.4.3]{KatzImQuad}), which we will use below.  

\item Let $\pi : \mathcal{E}|_{\mathbb{Y}(\Gamma_1(p^a))} \rightarrow \mathbb{Y}(\Gamma_1(p^a))$ denote the universal object (i.e. universal elliptic curve with $\Gamma_1(p^a)$-level structure). By GAGA and the $q$-expansion principle (\cite[(2.4.2)-(2.4.3)]{KatzImQuad}), $\mathbf{w}_{1,\eta}^{\mathrm{an}}$ is the complex analytification of an algebraic section 
$$\mathbf{w}_{1,\eta} \in \mathbf{\Gamma}(\Omega_{\mathcal{E}|_{\mathbb{Y}(\Gamma_1(p^a))}/\mathbb{Y}(\Gamma_1(p^a))}),$$
where $\mathcal{E}|_{\mathbb{Y}(\Gamma_1(p^a))} \rightarrow \mathbb{Y}(\Gamma_1(p^a))$ is the universal object (i.e. universal elliptic curve with $\Gamma_1(p^a)$-level structure).
%Let $e_N$ denote the first piece of the $\Gamma(N)$-level structure on $Y = Y(\Gamma(N))$. 
\item Let $e_{2,a} : \mathbb{Y}(\Gamma_1(p^a)) \rightarrow \mathcal{E}|_{\mathbb{Y}(\Gamma_1(p^a))}[p^a]$ be the second element of the $\Gamma(p^a)$-level structure as in (\ref{eindefinition}), and consider 
\begin{equation}\label{tildew1eta}\tilde{w}_{1,\eta} := e_{2,a}^*\mathbf{w}_{1,\eta} \in \mathbf{\Gamma}(\omega_{\mathbb{Y}(\Gamma_1(p^a))})
\end{equation}
where 
$$\omega_{\mathbb{Y}(\Gamma_1(p^a))} := \pi_*\Omega_{\mathcal{E}|_{\mathbb{Y}(\Gamma_1(p^a))}/\mathbb{Y}(\Gamma_1(p^a))}.$$
In fact, letting $\overline{\mathbb{Y}(\Gamma_1(p^a))}$ denote the compactification of $\mathbb{Y}(\Gamma_1(p^a))$, it is known (\cite[p. 1043]{BDP}) that $\omega_{\mathbb{Y}(\Gamma_1(p^a))}$ extends to a line bundle $\omega_{\overline{\mathbb{Y}(\Gamma_1(p^a))}}$ on $\overline{\mathbb{Y}(\Gamma_1(p^a))}$ and that $w_{1,\eta}$ extends to a section
\begin{equation}\label{tildewintegral0}\tilde{w}_{1,\eta} \in \mathbf{\Gamma}(\omega_{\overline{\mathbb{Y}(\Gamma_1(p^a))}})
\end{equation}
(\cite[(2.5), Chapter III]{KatzImQuad}). Thus $\tilde{w}_{1,\eta}$ has a well-defined $q$-expansion at the cusp $\infty$, which we denote by $\tilde{w}_{1,\eta}(q)$. 

\item By a direct calculation on $q$-expansions using (\ref{tildew1eta}) (see \cite[Chapter 1.1]{BCDDPR}), we see that 
$$\tilde{w}_{1,\eta}(q) \in \mathbb{Z}_p\llbracket q\rrbracket.$$
Thus, by the $q$-expansion principle (\cite[2.2.5]{KatzImQuad}) we see that 
\begin{equation}\label{tildewintegral}\tilde{w}_{1,\eta} \in \mathbf{\Gamma}(\omega_{Y(\Gamma_1(p^a))^+}).
\end{equation}
Here, $\pi_+: \mathcal{E}^+|_{Y(\Gamma_1(p^a))^+} \rightarrow Y(\Gamma_1(p^a))^+$ is the universal object (universal elliptic curve with $\Gamma_1(p^a)$-Drinfeld level structure) and 
$$\omega_{Y(\Gamma_1(p^a))^+} := \pi_{+,*}\Omega_{\mathcal{E}^+|_{Y(\Gamma_1(p^a))^+}/Y(\Gamma_1(p^a))^+}$$
is the Hodge bundle on $Y(\Gamma_1(p^a))^+$ (cf. (\ref{omegaY})). 

\item From the $q$-expansion $\tilde{w}_{1,\eta}(q)$ (\cite[Section 1.1]{BCDDPR}), we also see that $\tilde{w}_{1,\eta}$ is a $U$-eigenvector of $U$-eigenvalue 1, where $U = [p]U_p$ is the Hecke operator from Definition \ref{flatdefinition}. Thus
\begin{equation}\label{Ueigenvalue1}U^*\tilde{w}_{1,\eta}|_{Y_{\infty}} = \tilde{w}_{1,\eta}|_{Y_{\infty}}.
\end{equation}

\item Let $\mathbb{Y}_{\infty}(\Gamma)$ be as in (\ref{algebraicYinftyGamma}) for $\Gamma = \Gamma(N)$ as in Convention \ref{Yconvention}. Similarly, recalling that $V_p = g$ as in Definition \ref{UVdefinition}, and recalling that the Hecke operator $[p]$ (\ref{bracketpdefinition}) corresponding to multiplication by $p$ commutes with both $U_p$ and $V_p$, we define 
\begin{equation}\label{tildewflat}\begin{split}\tilde{w}_{1,\eta}^{\flat} &:= (1 - V_p^*U_p^*)\tilde{w}_{1,\eta}|_{\mathbb{Y}_{\infty}(\Gamma)} = (1-V_p^*[p^{-1}]^*([p]U_p)^*)\tilde{w}_{1,\eta} = (1-V_p^*[p^{-1}]^*U^*)\tilde{w}_{1,\eta} \\
&\overset{(\ref{Ueigenvalue1})}{=} (1-(g[p^{-1}])^*)(g^a)^*w_{1,\eta}|_{\mathbb{Y}_{\infty}(\Gamma)} = ((g^a)^*-[p^{-1}]^*(g^{a+1})^*)w_{1,\eta}|_{\mathbb{Y}_{\infty}(\Gamma)}\\
&\hspace{8cm}\in \omega_{\mathbb{Y}(\Gamma_1(p^{a+1}))}(\mathbb{Y}_{\infty}(\Gamma)) \subset \omega(Y_{\infty})
\end{split}
\end{equation}
as in Definition \ref{flatdefinition}, where $\omega$ is as in (\ref{omegaY}). Then (\ref{tildewintegral0}), (\ref{tildewintegral}) and (\ref{tildewflat}) show
\begin{equation}\label{tildewintegral2}\tilde{w}_{1,\eta}^{\flat}  \in \omega_{\mathbb{Y}(\Gamma_1(p^a))}(\mathbb{Y}_{\infty}(\Gamma)) \cap \omega_+(Y_{\infty}^+)
\end{equation}
where $\omega_+$ is as in (\ref{omegaY}). (Here the above intersections are taken in $\omega(Y_{\infty})$.)

\item Recall $g \in GL_2(\mathbb{Q}_p)$ from (\ref{gdefinition}), and recall that $GL_2(\mathbb{Q}_p)$ defines a right action on $\mathbb{Y}_{\infty}(\Gamma)$. Define
%$$w_{1,\eta} \in \mathbf{\Gamma}(\omega_{\mathbb{Y}(\Gamma_1(p^a))})$$
%to be the unique section with $q$-expansion ((2.4.3) of op. cit.) given by 
\begin{equation}\label{qexpansionrelate}\begin{split}w_{1,\eta} &:= (g^a)^*(\tilde{w}_{1,\eta}|_{Y_{\infty}}) \overset{(\ref{tildewintegral0}), (\ref{tildewintegral})}{\in} \omega_{\mathbb{Y}(\Gamma_1(p^a))}(\mathbb{Y}_{\infty}(\Gamma))\cap \omega_+(Y_{\infty}^+), \\
w_{1,\eta}^{\flat} &:= (g^a)^*(\tilde{w}_{1,\eta}^{\flat}|_{Y_{\infty}}) \overset{(\ref{tildewintegral2})}{\in} \omega_{\mathbb{Y}(\Gamma_1(p^a))}(\mathbb{Y}_{\infty}(\Gamma))\cap \omega_+(Y_{\infty}^+).
\end{split}
\end{equation}
%where $\mathrm{Tate}(q^{p^a}) = \mathbb{G}_{m,\mathbb{Z}(\!(q)\!)}/q^{p^a}$ is the Tate curve, $\omega_{\mathrm{can}} \in \Omega_{\mathrm{Tate}(q^{p^a})/\mathbb{Z}(\!(q)\!)}$ is the unique section pulling back to the standard differential $dT/(1+T)$ on $\mathbb{G}_m$ ($T$ the standard coordinate on $\mathbb{G}_{m,\mathbb{Z}(\!(q)\!)}$) under the quotient $\mathbb{G}_{m,\mathbb{Z}(\!(q)\!)} \rightarrow \mathbb{G}_{m,\mathbb{Z}(\!(q)\!)}/q^{p^a}$, and $q \in \mathrm{Tate}(q^{p^a})[p^a]$ is the $\Gamma_1(p^a)$-level structure (a $p^a$-torsion point) corresponding to $e_{1,a}$. (Note that our $\Gamma_1(p^a)$-level structure $q$ is equal to the $\Gamma_{00}(p^a)$-level structure $j_{\mathrm{can}}$ from (2.2.3) of op. cit.) 
%\item Recall $g \in GL_2(\mathbb{Q}_p)$ from (\ref{gdefinition}). By the $q$-expansion principle (\cite[(2.4.2)-(2.4.3)]{KatzImQuad}), (\ref{qexpansionrelate}) implies
%\begin{equation}\label{qexpansionrelate2}(g^{-a})^*\left(w_{1,\eta}|_{Y_{\infty}}\right) = \tilde{w}_{1,\eta}|_{Y_{\infty}}\overset{(\ref{tildew1eta})}{=} \left(e_{1,a}^*\mathbf{w}_{1,\eta}\right)|_{Y_{\infty}}.
%\end{equation}
\end{enumerate}
\end{definition}

\begin{remark}We caution the reader that, despite what the notation might suggest, $w_{1,\eta}^{\flat}$ is \emph{not} defined by applying Definition \ref{flatdefinition} to $w_{1,\eta}|_{Y_{\infty}}$, rather by the formula (\ref{qexpansionrelate}). %Note also that, since $V_p = g$,
%$$w_{1,\eta} = (g^a)^*\tilde{w}_{1,\eta}^{\flat} \overset{(\ref{tildewflat})}{=} (g^a)^*(1 - g^*U_p^*)\tilde{w}_{1,\eta} = ((g^a)^* - (g^{a+1})^*U_p^*)\tilde{w}_{1,\eta}.$$
\end{remark}

\begin{proposition}\label{wVproposition}$w_{1,\eta}, w_{1,\eta}^{\flat} \in \omega_+(Y_{\infty}^+)$ are pullbacks along $Y_{\infty}^+ \rightarrow Y(\Gamma_1(p^a))^+$ of sections (which we will denote also by $w_{1,\eta}$ and $w_{1,\eta}^{\flat}$)
\begin{equation}\label{qexpansionrelate2}w_{1,\eta} \in \mathbf{\Gamma}(\omega_{Y(\Gamma_1(p^a))}) \hspace{1cm} \text{and} \hspace{1cm} w_{1,\eta}^{\flat} \in \mathbf{\Gamma}(\omega_{Y(\Gamma_1(p^{a+1}))}),
\end{equation}
and thus from (\ref{qexpansionrelate}) we have 
\begin{equation}\label{qexpansionrelate3}w_{1,\eta} \in \mathbf{\Gamma}(\omega_{Y(\Gamma_1(p^a))}) \cap \omega_+(Y_{\infty}^+) \hspace{1cm} \text{and} \hspace{1cm} w_{1,\eta}^{\flat} \in \mathbf{\Gamma}(\omega_{Y(\Gamma_1(p^{a+1}))}) \cap \omega_+(Y_{\infty}^+).
\end{equation}
%\item One sees by direct calculation that $w_{1,\eta}(q)$ has $p$-integral coefficients as a power series in $q$ (cf. Lemma 3.4.1 of op. cit.), which by the $q$-expansion principle implies that $w_{1,\eta}|_{Y(\Gamma_1(p^a))}$ is induced by a section 
%$$w_{1,\eta} \in \mathbf{\Gamma}(\omega_{Y(\Gamma_1(p^a))^+})$$
%It is well-known that $w_{1,\eta}$ is a weight 1 Eisenstein series that is an eigenform of level $\Gamma_1(p^a)$, nebentype $\eta : (\mathbb{Z}/p^a)^{\times} \twoheadrightarrow (\mathbb{Z}/p^{\beta})^{\times} \rightarrow \{\pm 1\}$ and Atkin-Lehner $U = [p]U_p$-eigenvalue 1 (see Definition \ref{flatdefinition} for the definitions of $U$ and $U_p$); see \cite[Chapter 1]{BCDDPR} and \cite[Section 3.3.4]{KatzImQuad} with $N = p^a$ and auxiliary function $F = (\eta,\mathbf{1}) : \mathbb{Z}/p^a \times \mathbb{Z}/p^a \rightarrow \overline{\mathbb{Q}}$ in the notation of loc. cit. (where $\mathbf{1}$ denotes the trivial character). 
%\item
%\end{enumerate}
\end{proposition}

\begin{proof}The second inclusion of (\ref{qexpansionrelate2}) follows immediately from the first inclusion and (\ref{qexpansionrelate}). To prove the first inclusion, we will show that for any text object $(A,P,e_1,e_2) \in \mathbb{Y}_{\infty}(\Gamma)(R)$, $R$ a $k$-algebra, $P$ a $\Gamma = \Gamma(N)$-level structure over $R$ and $(e_1,e_2)$ a $\Gamma_1(p^a)$-level structure over $R$, we have that 
$$w_{1,\eta}(A,P,e_1,e_2) = w_{1,\eta}((A,P,e_1,e_2) \cdot \gamma)$$
for every $\gamma \in \Gamma_{1,p}(p^a) \subset GL_2(\mathbb{Z}_p)$ (where $\Gamma_{1,p}(p^a)$ is as in (\ref{Gamma1pdefinition})). Since, by Definition \ref{congruencesubgroups},
$$\Gamma_{1,p}(p^{\alpha}) \cdot \prod_{\ell \neq p}GL_2(\mathbb{Z}_{\ell}) = \Gamma_1(p^a),$$
this will give the assertion.

Write $\gamma = \left(\begin{array}{ccc} a' & b'\\
c' & d'\\
\end{array}\right) \in \Gamma_{1,p}(p^a)$. Thus 
\begin{equation}\label{cdcongruences}c' \equiv 0 \pmod{p^a}, \hspace{1cm} d' \equiv 1 \pmod{p^a}.
\end{equation}
This implies 
\begin{equation}\label{subgroupequality}\langle a'e_{1,a} + c'e_{2,a}\rangle = \langle e_{1,a}\rangle \subset A[p^n],
\end{equation}
where $\langle \cdot \rangle$ as in Definition \ref{YGamma+definition} (5) (where $e_{i,a}$ is as in (\ref{eindefinition})). Recall that 
$$w_{1,\eta} \overset{(\ref{qexpansionrelate})}{=} (g^a)^*\tilde{w}_{1,\eta} \overset{(\ref{tildew1eta})}{=} (g^a)^*e_{2,a}^*\mathbf{w}_{1,\eta}$$
where $e_{2,a} : \mathbb{Y}_{\infty}(\Gamma) \rightarrow \mathcal{E}|_{\mathbb{Y}_{\infty}(\Gamma)}$ is the order-$p^a$ torsion section given by the second piece $e_2$ of the $\Gamma(p^{\infty})$-level structure on the universal object $\mathcal{E}|_{\mathbb{Y}_{\infty}(\Gamma)} \rightarrow \mathbb{Y}_{\infty}(\Gamma)$. Thus it suffices to show that the values of $e_{2,a}$ at $(A,P,e_1,e_2) \cdot g^a$ and $(A,P,e_1,e_2)\cdot \gamma g^a$ are equal, i.e.
\begin{equation}\label{needtoshowe}e_{2,a}(A,P,e_1,e_2)\cdot g^a = e_{2,a}((A,P,e_1,e_2) \cdot \gamma g^a).
\end{equation}
%However, note that 
%$$\Gamma_{1,p}(p^a)\cdot g^a = \left\{\gamma \in GL_2(\mathbb{Z}_p) : \gamma \equiv \left(\begin{array}{ccc} 1 & 0\\
%0 & 0 \\
%\end{array}\right) \pmod{p^a}\right\}.$$
By the moduli interpretation of the $GL_2(\mathbb{Q}_p)$-action (\cite[Section 2.3]{ChojeckiHansenJohansson}) we have that 
$$(A,P,e_1,e_2) \cdot \gamma = (A,P,a'e_1 + c'e_2, b'e_1 + d'e_2).$$
Thus, letting 
$$f : A \rightarrow A/\langle a'e_{1,a} + c'e_{2,a}\rangle \overset{(\ref{subgroupequality})}{=} A/\langle e_{1,a}\rangle$$
denote the natural projection and $f^{\vee}$ denote its dual isogeny, we have
$$(A,P,e_1,e_2) \cdot g^a = (A/\langle e_{1,a}\rangle,f(P),\frac{1}{p^a},f(e_1),f(e_2))$$
and
$$(A,P,e_1,e_2) \cdot \gamma g^a = (A,P,a'e_1 + c'e_2,b'e_1 + d'e_2) \cdot g^a = (A/\langle e_{1,a}\rangle,f(P),\frac{1}{p^a}f(a'e_1 + c'e_2),f(b'e_1+d'e_2)).$$
Since 
\begin{equation}\label{fequality}f(b'e_1+d'e_2) \pmod{p^a} = f(b'e_{1,a} + d'e_{2,a}) = f(d'e_{2,a})\overset{(\ref{cdcongruences})}{=} f(e_{2,a}) = f(e_2) \pmod{p^a},
\end{equation}
we have 
\begin{align*}e_{2,a}((A,P,e_1,e_2) \cdot g^a) &= e_{2,a}(A,P,e_1,e_2) = (A,f(e_2) \pmod{p^a}) \overset{(\ref{fequality})}{=} (A,f(b'e_1+d'e_2) \pmod{p^a}) \\
&= e_{2,a}(A,P,\frac{1}{p^a}f(a'e_1 + c'e_2),f(c'e_1+d'e_2)) = e_{2,a}((A,P,e_1,e_2) \cdot \gamma g^a)
\end{align*}
which gives (\ref{needtoshowe}).

\end{proof}

\begin{definition}
\begin{enumerate}
\item Recall $\omega$ and $\omega_+$ from (\ref{omegaY}). Pulling back along $Y_{\infty}^+ \rightarrow Y(\Gamma_1(p^a))^+$, we get a section
$$w_{1,\eta} \in \omega_+(Y_{\infty}^+) \subset \omega(Y_{\infty}).$$
Let
$$w_{1,\eta}^{\flat} := (U_p^*V_p^* - V_p^*U_p^*)w_{1,\eta} \in \omega_+(Y_{\infty}^+) \subset \omega(Y_{\infty})$$
be defined as in Definition \ref{flatdefinition} for 
$$w' = w_{1,\eta} \in \omega_+(Y_{\infty}^+) \subset \omega(Y_{\infty}).$$
Note that since $U_p, V_p$ are in the local Hecke algebra at $p$ of $Y_{\infty}$, and $w_{1,\eta} \in \mathbf{\Gamma}(\omega_{Y(\Gamma_1(p^a))})$ from (\ref{qexpansionrelate2}), then in fact $w_{1,\eta}^{\flat}$ is the pullback along $Y_{\infty} \rightarrow Y(\Gamma_1(p^{a+1}))$ of a section
$$w_{1,\eta}^{\flat} \in \mathbf{\Gamma}(\omega_{Y(\Gamma_1(p^{a+1}))}).$$
%See Proposition \ref{wVproposition} below. 
%Since $w_{1,\eta}$ is a $U_p$-eigenvector with eigenvalue $\alpha_p = 1$, we have the formula
%$$w_{1,\eta}^{\flat} = (1-V_p^*)w_{1,\eta}.$$
\item For any $s \in \mathbb{Z}_{\ge 0}$, let
$$\pi_0(\mathbb{Y}(\Gamma_1(p^s))(\mathbb{C})^{\mathrm{an}})$$
denote the connected component set of the complex analytic space 
$$\mathbb{Y}(\Gamma_1(p^s))(\mathbb{C})^{\mathrm{an}}$$
attached to $\mathbb{Y}(\Gamma_1(p^s))(\mathbb{C})$. The complex analytic universal cover of each connected component of $\mathbb{Y}(\Gamma_1(p^s))(\mathbb{C})^{\mathrm{an}}$ is $\mathcal{H}^+ = \{\tau \in \mathbb{C} : \mathrm{Im}(\tau) > 0\}$. Define the complex analytic modular forms of weight 1
\begin{equation}\label{E1eta}E_{1,\eta} := \frac{w_{1,\eta}}{2\pi i dz} \in \prod_{\pi_0(\mathbb{Y}(\Gamma_1(p^a))(\mathbb{C})^{\mathrm{an}})}\mathcal{O}_{\mathcal{H}^+}(\mathcal{H}^+), \hspace{1cm} E_{1,\eta}^{\flat} := \frac{w_{1,\eta}^{\flat}}{2\pi i dz} \in \prod_{\pi_0(\mathbb{Y}(\Gamma_1(p^{a+1}))(\mathbb{C})^{\mathrm{an}})}\mathcal{O}_{\mathcal{H}^+}(\mathcal{H}^+).
\end{equation}
%Then as in \cite[Section 1]{BCDDPR}, we have the following identity for the $q$-expansion:
%$$E_{1,\eta}^{\flat}(q) = -Np^a\cdot \sum_{n = 1, p\nmid n}^{\infty}\sigma_{0,\eta}(n)q^n, \hspace{1cm} \sigma_{0,\eta}(n) = \sum_{0 < d|n}\eta(d).$$
\end{enumerate}
\end{definition}

We will henceforth make frequent use of the following notation for certain open subsets $W \subset \mathcal{V}_x$ (see (\ref{Vz})), denoting the power series in $\hat{\mathcal{O}}_{W}(W)\llbracket q_{\mathrm{dR}}-1\rrbracket$ obtained from an element of $\hat{\mathcal{O}}_{W}(W)(\!(t)\!)\llbracket q_{\mathrm{dR}}-1\rrbracket$ (e.g. a $q_{\mathrm{dR}}$-expansion) by applying the $\hat{\mathcal{O}}_{W}(W)$-module homomorphism $\theta_t : \hat{\mathcal{O}}_{W}(W)(\!(t)\!) \rightarrow \hat{\mathcal{O}}_{W}(W)$ (see (\ref{thetat})) to the $\hat{\mathcal{O}}_{W}(W)(\!(t)\!)$-coefficients.

\begin{definition}\label{thetatpowerseriesdefinition}Let $W = U_m$ for some $m \in \mathbb{Z}$ (see (\ref{Um})) or $\mathcal{V}_x$ (see (\ref{Vz})). Thus in any case $W \subset \mathcal{V}_x$ is open. Recall $\theta_t : \hat{\mathcal{O}}_{W}(W)(\!(t)\!) \rightarrow \hat{\mathcal{O}}_{W}(W)$ from (\ref{thetat}). Given any element
$$F(q_{\mathrm{dR}}) \in \hat{\mathcal{O}}_{W}(W)(\!(t)\!)\llbracket q_{\mathrm{dR}}-1\rrbracket \overset{(\ref{BdecompositionUm})}{=} \mathbb{B}_{\mathrm{dR},W}(W)\llbracket q_{\mathrm{dR}}-1\rrbracket,$$
let 
$$\theta_t(F(q_{\mathrm{dR}})) \in \hat{\mathcal{O}}_{W}(W)\llbracket q_{\mathrm{dR}}-1\rrbracket$$
denote the power series obtained by applying $\theta_t$ to the coefficients of $F(q_{\mathrm{dR}})$, and fixing $q_{\mathrm{dR}}-1$. 
\end{definition}

\begin{definition}
Using (\ref{Gflatgeneraldefinition}), Definition \ref{'zqexpansions} and Definition \ref{thetatpowerseriesdefinition}, define
\begin{equation}\label{G1etadefinition}G_{1,\eta}^{\flat}(q_{\mathrm{dR}}) := \theta_t\left(w_{1,\eta}^{\flat}(q_{\mathrm{dR}})\right) \in \hat{\mathcal{O}}_{\mathcal{V}_x}(\mathcal{V}_x)\llbracket q_{\mathrm{dR}}-1\rrbracket,
\end{equation}
where $\mathcal{V}_x \subset Y_{\infty}$ is the open subset from (\ref{VxVy}). In other words, $G_{1,\eta}^{\flat}(q_{\mathrm{dR}})$ is the power series obtained by applying $\theta_t$ from Definition \ref{thetatpowerseriesdefinition} to the coefficients of 
$$w_{1,\eta}^{\flat}(q_{\mathrm{dR}}) \in \mathbb{B}_{\mathrm{dR},\mathcal{V}_x}(\mathcal{V}_x)\llbracket q_{\mathrm{dR}}-1\rrbracket  \overset{(\ref{Bdecomposition''})}{=}  \hat{\mathcal{O}}_{\mathcal{V}_x}(\!(t)\!)\llbracket q_{\mathrm{dR}}-1\rrbracket,$$
as defined in Definition \ref{thetatpowerseriesdefinition}. Note that by construction $w_{1,\eta}(q_{\mathrm{dR}})$ and $w_{1,\eta}^{\flat}(q_{\mathrm{dR}})$ are generalized $p$-adic modular forms of weight 1 in the sense of Definition \ref{generalizedpadicmodularformdefinition}.

%Finally, let
%$$G_{1,\eta}^{\flat} := G_{1,\eta}|_{1-V_p} = (1 - \left(\begin{array}{ccc} 1 & 0\\
%0 & p\\
%\end{array}\right)^*)G_{1,\eta}.$$

%using Proposition \ref{coordinateproposition} to write in terms of the coordinate $X$ (note that $X$ is not a free variable, and satisfies the relation $f^{(n)}(X) = 0$), and where $\Omega' \in \mathbf{\Gamma}(\mathcal{O}_{\hat{\mathcal{Y}}^{\mathrm{Ig}}(r)}^+) \overset{(\ref{completesheaves})}{=} \Gamma(\hat{\mathcal{O}}_{\mathcal{Y}^{\mathrm{Ig}}(r)}^+)$ be defined by 
%$$\Omega' = \frac{e_{1,\beta}^*\frak{S}}{e_{1,\beta}^*dX}.$$
%This period $\Omega'$ is only involved in our discussion in order to define (\ref{GqdRexpansion}) (where $\Omega' \in \Gamma(\hat{\mathcal{O}}_{\mathcal{Y}^{\mathrm{Ig}}(r)}^+)$ is viewed in the coefficients of the $q_{\mathrm{dR}}$-expansion). 
%Pulling back by the rigid-analytic section $Y(r) \hookrightarrow Y(\Gamma_0(f))$  of the natural projection $Y_f \rightarrow Y$ provided by the order $p^{\mathrm{ord}_p(f)+1}$-canonical subgroup (which exists by (\ref{rassumption})), we can view 
%$$G_{1,\eta}^{\flat} \in \mathbf{\Gamma}(\mathcal{O}_{Y_{\infty}^+(r)}).$$

%We note that the values of the $p$-adic and complex-analytic Eisenstein series are related via (\ref{derivativecomparison}).  
\end{definition}

Note that $GL_2(\mathbb{Q})$ acts on the left of $\mathcal{H}^+$ via the modular action $\gamma \mapsto \gamma \cdot \tau$ from (\ref{leftmodularaction}). We then define a right action on $\mathcal{H}^+$ by 
$$\tau \cdot \gamma = \gamma^{\vee} \cdot \tau,$$
recalling $\gamma^{\vee} := \mathrm{det}(\gamma)\gamma^{-1}$. Thus we have a right action of $GL_2(\mathbb{Q})$ on
$$E_{1,\eta}\in \prod_{\pi_0(\mathbb{Y}(\Gamma_1(p^a))(\mathbb{C})^{\mathrm{an}})}\mathcal{O}_{\mathcal{H}^+}(\mathcal{H}^+)$$
via pullback, and acting on $w_{1,\eta} \in \omega(Y_{\infty})$ via the right $GL_2(\mathbb{Q}_p)$-action on $Y_{\infty}$. Then letting $V$ be the usual Atkin-Lehner $V$ operator (\cite[Section 3.7]{BDP}), we have $V = [p^{-1}]V_p$ as easily follows from the definitions of loc. cit. (see also Definition \ref{flatdefinition}). Recall $[p]$ from (\ref{bracketpdefinition}). 

\begin{proposition}\label{Vproposition}Then
$$E_{1,\eta}^{\flat} = (1-([p^{-1}]V_p)^*)E_{1,\eta}, \hspace{1cm} w_{1,\eta}^{\flat} = (1-([p^{-1}]V_p)^*)w_{1,\eta}.$$
\end{proposition}

\begin{proof}Recall (Definition \ref{flatdefinition})
$$E_{1,\eta}^{\flat} = (U_p^*V_p^* - V_p^*U_p^*)E_{1,\eta} = (1 - V_p^*U_p^*)E_{1,\eta}.$$
Note that $([p]U_p)^*E_{1,\eta} = E_{1,\eta}$, since $E_{1,\eta}$ is a $U$-eigenvector of $U = [p]U_p$-eigenvalue 1 (recall $U$ and $U_p$ are as in Definition \ref{flatdefinition}). Thus, since $[p]$ and $[p^{-1}]$ are in the center of $GL_2(\mathbb{Q}_p)$, 
$$(1-V_p^*U_p^*)E_{1,\eta} = (1- ([p^{-1}]V_p)^*([p]U_p)^*)E_{1,\eta} = (1-([p^{-1}]V_p)^*)E_{1,\eta}.$$
The proof of the identity for $w_{1,\eta}^{\flat}$ is identical.
\end{proof}

%Thus $w_{1,\eta}^{\flat} = (1-V_p^*)w_{1,\eta}$

\subsection{The fundamental power series}

Recall the notation of Convention \ref{Gammaconvention}. By (\ref{qexpansionrelate3}) and Proposition \ref{wVproposition}, we have  
$$w_{1,\eta}^{\flat} \in \mathbf{\Gamma}(\omega_{Y(\Gamma_1(p^{a+1}))^+}) \cap \omega_+(Y_{\infty}^+) \subset \omega(Y_{\infty}).$$
% and via pullback along $Y_{\infty} \rightarrow Y(\Gamma_1(p^{a+1}))$ can be viewed as a section
%$$w_{1,\eta}^{\flat} \in \omega(Y_{\infty}).$$
%Note that the $q$-expansion of $E_{1,\eta}^{\flat}$ has $p$-integral coefficients by Definition \ref{flatdefinition} (2) and Definition \ref{Eisensteindefinition}. (One may also show this $p$-integrality using \cite[Chapter I (1) and discussion after (36)]{BCDDPR}.) Moreover, by Theorem \ref{Katzpadicmodularformtheorem}, $w_{1,\eta}^{\flat}|_{\mathcal{Y}^{\mathrm{Ig}}}$ is a Katz $p$-adic modular form. Thus the $q$-expansion principle for Katz $p$-adic modular forms (\cite[Section 2.7-8]{Katzpamf}) implies
In particular, we have
$$w_{1,\eta}^{\flat}|_{\mathcal{Y}^{\mathrm{Ig}}} \in \mathcal{O}_{\mathcal{Y}^{\mathrm{Ig}}(0)^+}(\mathcal{Y}^{\mathrm{Ig}}(0)^+).$$

Recall the open subset 
$$Y_1(a+1,\epsilon_0/p^a) \subset Y(\Gamma(N) \cap \Gamma_1(p^{a+1}))$$
from Convention \ref{Y1convention}. Observe that by pulling back $w_{1,\eta}^{\flat} \in \mathbf{\Gamma}(\omega_{Y(\Gamma_1(p^{a+1}))})$ along the map 
$$Y_1(a+1,\epsilon_0/p^a) \subset Y(\Gamma(N) \cap \Gamma_1(p^{a+1})) \rightarrow Y(\Gamma_1(p^{a+1}))$$
we can view $w_{1,\eta}^{\flat} \in \omega(Y_1(a+1,\epsilon_0/p^a))$. Thus by Remark \ref{assumptionsatisfiedremark2} and Proposition \ref{wVproposition}, $w_{1,\eta}^{\flat}$ satisfies the assumptions of Theorem \ref{pintegraltheorem2} with $\alpha = a$ and $k = 1$. Invoking Theorem \ref{pintegraltheorem2} (1) we get
\begin{equation}\label{wqdRexpansion}\theta_t(\theta_q(w_{1,\eta}^{\flat})) \in \hat{\mathcal{O}}_{\mathcal{Y}^{\mathrm{Ig}}(\epsilon_0/p^a)}^+(\mathcal{Y}^{\mathrm{Ig}}(\epsilon_0/p^a)).
\end{equation}
By (\ref{G1etadefinition}) and the fact that 
$$\theta_t \circ \theta_q(G_{1,\eta}^{\flat}(q_{\mathrm{dR}})) = \theta_q \circ \theta_t(G_{1,\eta}^{\flat}(q_{\mathrm{dR}}))$$
(see (\ref{thetaq}) and (\ref{thetat})), this implies
\begin{equation}\label{GqdRexpansion}\theta_q(G_{1,\eta}^{\flat}(q_{\mathrm{dR}}))  \in \mathcal{O}_{\mathcal{Y}^{\mathrm{Ig}}(\epsilon_0/p^a)}^+(\mathcal{Y}^{\mathrm{Ig}}(\epsilon_0/p^a)).
\end{equation}
%By (\ref{YIgin2})  we have
%\begin{equation}\label{ybin}y_{\frak{b},a} \in \mathcal{Y}^{\mathrm{Ig}}(\epsilon_0/p^a).
%\end{equation}

\begin{definition}
Recall the point $y_{\frak{b},a} \in \mathcal{Y}^{\mathrm{Ig}}(\epsilon_0/p^a)(\overline{\mathbb{Q}}_p,\overline{\mathbb{Z}}_p)$ from (\ref{YIgin2}). Let 
$$G_{1,\eta}^{\flat}(y_{\frak{b,a}})(q_{\mathrm{dR}}) \in \mathbb{C}_p\llbracket q_{\mathrm{dR}}-1\rrbracket$$
denote the image of 
$$G_{1,\eta}^{\flat}(q_{\mathrm{dR}}) \in \hat{\mathcal{O}}_{\mathcal{V}_x}(\mathcal{V}_x)\llbracket q_{\mathrm{dR}}-1\rrbracket$$
under the specialization of coefficients 
$$\hat{\mathcal{O}}_{\mathcal{V}_x}(\mathcal{V}_x) \rightarrow \hat{\mathcal{O}}_{\mathcal{V}_x}(y_{\frak{b},a}) \subset \mathbb{C}_p.$$
Recall that $\theta_q(q_{\mathrm{dR}}) = 1$ (see (\ref{thetaq})), and so 
$$\theta_q(G_{1,\eta}^{\flat}(q_{\mathrm{dR}})) = G_{1,\eta}^{\flat}(q_{\mathrm{dR}})|_{q_{\mathrm{dR}} = 1}.$$
Thus the constant term of $G_{1,\eta}^{\flat}(y_{\frak{b},a})(q_{\mathrm{dR}})$ is 
$$G_{1,\eta}^{\flat}(y_{\frak{b},a})(q_{\mathrm{dR}})|_{q_{\mathrm{dR}} = 1} = \theta_q(G_{1,\eta}^{\flat}(y_{\frak{b},a})(q_{\mathrm{dR}})) = \theta_q(G_{1,\eta}^{\flat}(q_{\mathrm{dR}}))(y_{\frak{b},a}) \overset{(\ref{GqdRexpansion})}{\in} \mathcal{O}_{\mathbb{C}_p}.$$
Define
\begin{equation}\label{localmeasure}\mu_{\frak{b}} := G_{1,\eta}^{\flat}(y_{\frak{b},a})(q_{\mathrm{dR}}) \in \mathcal{O}_{\mathbb{C}_p} + (q_{\mathrm{dR}}-1)\mathbb{C}_p\llbracket q_{\mathrm{dR}}-1\rrbracket.
\end{equation}
Now following (\ref{sum1}) we define
\begin{equation}\label{globalmeasure}\mathcal{L}_{\lambda} := \sum_{\frak{b} \in S}\lambda^{-1}(\frak{b})\cdot \mu_{\frak{b}}\in \mathcal{O}_{\mathbb{C}_p} + (q_{\mathrm{dR}}-1)\mathbb{C}_p\llbracket q_{\mathrm{dR}} - 1\rrbracket.
\end{equation}
%Now we let $r :  \mathcal{C}\ell(\frak{f}^{(p)})[p^{\beta}] \times \mathbb{Z}_p \rightarrow \mathbb{Z}_p$ denote the natural quotient. Via pushforward, we get a map
%$$r_* : \mathcal{O}_{\mathbb{C}_p}\llbracket  \mathcal{C}\ell(\frak{f}^{(p)})[p^{\beta}] \times \mathbb{Z}_p \rrbracket \rightarrow \mathcal{O}_{\mathbb{C}_p}\llbracket \mathbb{Z}_p\rrbracket, \hspace{1cm} r_*\mu(f) = \mu(f \circ r).$$
%We define
%$$\mathcal{L}_{\lambda} := r_*\mu_{\lambda} \in \mathcal{O}_{\mathbb{C}_p}\llbracket \mathbb{Z}_p\rrbracket.$$
\end{definition}

%\begin{remark}\label{qdR1remark}Recall that $\theta_q(q_{\mathrm{dR}}) = 1$ (see (\ref{thetaq})), and so 
%$$\theta_q(G_{1,\eta}^{\flat}(q_{\mathrm{dR}})) = G_{1,\eta}^{\flat}(q_{\mathrm{dR}})|_{q_{\mathrm{dR}} = 1}.$$
%This implies
%$$\theta_q(\mathcal{L}_{\lambda}) = \mathcal{L}_{\lambda}|_{q_{\mathrm{dR}} = 1}.$$
%\end{remark}

\subsection{The $p$-adic $L$-function}

We will use the differential operator $d_k^j$ from (\ref{dkjformula2}) in order to construct ``derivatives'' of our fundamental power series, which will in turn be used to construct our $p$-adic $L$-function by $p$-adically varying $j$; $j$ will thus be the variable of the $p$-adic $L$-function. By invoking Theorem \ref{pintegraltheorem2}, we will show that the $p$-adic $L$-function varies continuously in $j \in \mathbb{Z}/(p-1) \times \mathbb{Z}_p$. 
%\begin{choice}Choose and fix any $\xi$ in the valuation group of $k$ (see Definition \ref{kdefinition}) with
%\begin{equation}\label{xichoice}0 \le \xi < 1/2.
%\end{equation}
%\end{choice}
 %Recall the power series $G_{1,\eta}^{\flat}(q_{\mathrm{dR}})$ from (\ref{G1etadefinition}). By (\ref{wqdRexpansion}), the definition of $W_{\xi}$ from (\ref{W'definition}) and (\ref{dkjformula2}), have
%$$d_1^jw_{1,\eta}^{\flat}(q_{\mathrm{dR}}) \in \mathcal{O}_{\hat{\mathcal{Y}}^{\mathrm{Ig}}(\epsilon_0)^+}(\hat{\mathcal{Y}}^{\mathrm{Ig}}(\epsilon_0)^+)\llbracket q_{\mathrm{dR}}-1\rrbracket [1/(z_{\mathrm{dR}}-z_{\mathrm{HT}})] \subset \hat{\mathcal{O}}_{\mathcal{Y}^{\mathrm{Ig}}(\epsilon_0)}^+(W_{\xi})\llbracket q_{\mathrm{dR}}-1\rrbracket [1/(z_{\mathrm{dR}}-z_{\mathrm{HT}})].$$
%Here, the last inclusion above follows from the fact that, since $W_{\xi} \subset \mathcal{Y}^{\mathrm{Ig}} \subset \mathcal{Y}^{\mathrm{Ig}}(\epsilon_0)^+$ is an inclusion of pro-adic spaces, the restriction 
%$$\mathcal{O}_{\hat{\mathcal{Y}}^{\mathrm{Ig}}(\epsilon_0)^+}(\hat{\mathcal{Y}}^{\mathrm{Ig}}(\epsilon_0)^+) \rightarrow \mathcal{O}_{\hat{\mathcal{Y}}^{\mathrm{Ig}}(\epsilon_0)}(W_{\xi}) = \hat{\mathcal{O}}_{\mathcal{Y}^{\mathrm{Ig}}(\epsilon_0)}(W_{\xi})$$
%factors through
%$$\mathcal{O}_{\hat{\mathcal{Y}}^{\mathrm{Ig}}(\epsilon_0)^+}(\hat{\mathcal{Y}}^{\mathrm{Ig}}(\epsilon_0)^+) \rightarrow \hat{\mathcal{O}}_{\mathcal{Y}^{\mathrm{Ig}}(\epsilon_0)}^+(W_{\xi}).$$

Recalling Definition \ref{thetatpowerseriesdefinition}, define for $j \in \mathbb{Z}_{\ge 0}$
\begin{equation}\label{derivativeintegral}D_1^jG_{1,\eta}^{\flat}(q_{\mathrm{dR}}) := \theta_t\left(d_1^jw_{1,\eta}^{\flat}(q_{\mathrm{dR}})\right) \in \hat{\mathcal{O}}_{\mathcal{V}_x}(\mathcal{V}_x)\llbracket q_{\mathrm{dR}}-1\rrbracket.
\end{equation}
By Theorem \ref{pintegraltheorem2} (1) with $\alpha = a$ (as in (\ref{adefinition})) and $k = 1$ and Remark \ref{assumptionsatisfiedremark2} applied to $w_{1,\eta}^{\flat}$ (viewing $w_{1,\eta}^{\flat} \in \omega(Y(\Gamma_1(Np^a)))$ as in the discussion before (\ref{wqdRexpansion})), we have
\begin{equation}\label{derivativeintegral2}\theta_q(D_1^jG_{1,\eta}^{\flat}(q_{\mathrm{dR}})) = D_1^jG_{1,\eta}^{\flat}(q_{\mathrm{dR}})|_{q_{\mathrm{dR}} = 1} \in \hat{\mathcal{O}}_{\mathcal{Y}^{\mathrm{Ig}}(\epsilon_0/p^a)}^+(\mathcal{Y}^{\mathrm{Ig}}(\epsilon_0/p^a)). 
\end{equation}
Moreover, viewing $\mathbb{Z}_{\ge 0} \subset \mathbb{Z}/(p-1) \times \mathbb{Z}_p$ embedded diagonally, Theorem \ref{pintegraltheorem2} (3) implies that (\ref{derivativeintegral2}) extends to a continuous function 
$$\mathbb{Z}/(p-1) \times \mathbb{Z}_p \rightarrow \hat{\mathcal{O}}_{\mathcal{Y}^{\mathrm{Ig}}(\epsilon_0/p^a)}^+(\mathcal{Y}^{\mathrm{Ig}}(\epsilon_0/p^a)).$$

\begin{definition}For any $\frak{b} \in S$, let
\begin{equation}\label{derivativelocalmeasure}D_1^j\mu_{\frak{b}} = D_1^jG_{1,\eta}^{\flat}(y_{\frak{b},a})(q_{\mathrm{dR}}) \overset{(\ref{derivativeintegral}), (\ref{derivativeintegral2})}{\in} \mathcal{O}_{\mathbb{C}_p} + (q_{\mathrm{dR}}-1)\mathbb{C}_p\llbracket q_{\mathrm{dR}}-1\rrbracket,
\end{equation}
where $D_1^jG_{1,\eta}^{\flat}(y_{\frak{b},a})(q_{\mathrm{dR}})$ denotes the image of $D_1^jG_{1,\eta}^{\flat}(q_{\mathrm{dR}})$ under the specialization of coefficients
\begin{equation}\label{ybspecialize2}\hat{\mathcal{O}}_{\mathcal{Y}^{\mathrm{Ig}}(\epsilon_0/p^a)}(\mathcal{Y}^{\mathrm{Ig}}(\epsilon_0/p^a)) \rightarrow \hat{\mathcal{O}}_{\mathcal{Y}^{\mathrm{Ig}}(\epsilon_0)}(y_{\frak{b},a}) \subset \mathbb{C}_p,
\end{equation}
recalling $y_{\frak{b},a} \in \mathcal{Y}^{\mathrm{Ig}}(\epsilon_0/p^a)(\overline{\mathbb{Q}}_p,\overline{\mathbb{Z}}_p)$ from (\ref{YIgin2}). 
\end{definition}

Let $E/\mathbb{Q}$ be the elliptic curve with CM by $\mathcal{O}_K$ as in Assumption \ref{pconductorassumption} (3). Let $\lambda_E$ be its associated infinity type $(1,0)$ Hecke character, so that $\frak{f}(\lambda_E)|\frak{f}$. 

\begin{definition}For any $j \in \mathbb{Z}_{\ge 0}$, we define
\begin{equation}\label{globalmeasure2}D_1^j\mathcal{L}_{\lambda} = \sum_{\frak{b} \in S}(\lambda^{-1}\lambda_E^{-2j})(\frak{b})\cdot D_1^j\mu_{\frak{b}} \in \mathcal{O}_{\mathbb{C}_p} + (q_{\mathrm{dR}}-1)\mathbb{C}_p\llbracket q_{\mathrm{dR}}-1\rrbracket.
\end{equation}
Note that $D_1^j\mathcal{L}_{\lambda}$ implicitly depends on $E/\mathbb{Q}$. However, as the choice of $E/\mathbb{Q}$ will be obvious from context, we suppress it from notation. 

\end{definition}

Recall (e.g. from Definition \ref{OKpfacts} (2)) that
\begin{equation}\label{unitdecomposition}\mathcal{O}_{K_p}^{\times} = \begin{cases} \mathbb{Z}/(p-1) \times \mathbb{Z}_p^{\oplus 2} & p > 3\\
\mathbb{Z}/6 \times \mathbb{Z}_3^{\oplus 2} & K = \mathbb{Q}(\sqrt{-3})\\
\mathbb{Z}/4 \times \mathbb{Z}_2^{\oplus 2} & K = \mathbb{Q}(i)\\
\mathbb{Z}/2 \times \mathbb{Z}_2^{\oplus 2} & K = \mathbb{Q}(\sqrt{-2})\\
\end{cases}.
\end{equation}

%The natural embedding $\mathbb{Z} \subset \mathrm{Aut}(\mathcal{O}_{K_p}^{\times})$ given by exponentiation on $\mathcal{O}_{K_p}^{\times}$ corresponds to the diagonal embedding under the above identification. We may also naturally embed $\mathbb{Z} \subset \mathbb{Z}/(p-1) \times \mathbb{Z}_p$ diagonally. 

We will need an elementary lemma demonstrating the continuity of $j \mapsto \lambda_E^j(\frak{b})$. Recall we embed $\mathbb{Z}_{\ge 0} \subset \mathbb{Z}/(p-1) \times \mathbb{Z}_p$ diagonally, and the image is dense with respect to the product of the discrete topology on the first factor and the $p$-adic topology on the second factor. 

\begin{lemma}\label{sequencelemma}For every $\frak{b} \in S$, the function $\mathbb{Z}_{\ge 0} \ni j \mapsto \lambda_E^j(\frak{b}) \in \mathcal{O}_{K_p}^{\times}$ extends to a continuous function $\mathbb{Z}/(p-1) \times \mathbb{Z}_p \rightarrow \mathcal{O}_{K_p}^{\times}$.

%Suppose a sequence $\{j_m\} \in \mathbb{Z} \subset \mathbb{Z}/(p-1) \times \mathbb{Z}_p$ tends to an element $j \in \mathbb{Z}/(p-1) \times \mathbb{Z}_p$. Then $\{j_m\}$ tends to an element in $\mathrm{Aut}(\mathcal{O}_{K_p}^{\times})$. 
\end{lemma}

\begin{proof}Given a sequence $j_m \in \mathbb{Z}_{\ge 0}$ which tends to $j \in \mathbb{Z}/(p-1) \times \mathbb{Z}_p$. Thus for all $k \in \mathbb{Z}_{\ge 0}$, there exists $M \in \mathbb{Z}_{\ge 0}$ such that for all $m \ge M$, 
$$j-j_m \in \{0\} \times p^k\mathbb{Z}_p \subset \mathbb{Z}/(p-1) \times \mathbb{Z}_p.$$ Since $\lambda_E(\frak{b}) \in \mathcal{O}_{K_p}^{\times}$, (\ref{unitdecomposition}), (\ref{firstGamma}) and (\ref{secondGamma}) imply $\lambda_E^{j-j_m}(\frak{b}) \in 1 + p^{k-c}\cdot \mathbb{Z}_p^{\oplus 2}$ for some $c \in \mathbb{Z}_{\ge 0}$ independent of $k$. This gives the assertion.

%Write $j_m = (k_m,\ell_m)$ in terms of the decomposition $\mathbb{Z}/(p-1) \times \mathbb{Z}_p$; explicitly, $k_m = j_m \pmod{p-1}$, $\ell_m = j_m$. Then $k_m = a$ for some fixed $a \in \mathbb{Z}/(p-1)$ for all $m \gg 0$, and $\ell_m \rightarrow b$ in $\mathbb{Z}_p$ as $m \rightarrow \infty$ for some fixed $b \in \mathbb{Z}_p$. Hence in $\mathrm{Aut}(\mathcal{O}_{K_p}^{\times})$, under the identification (\ref{automorphismembedding}) we have
%$$j_m \rightarrow \begin{cases} (a,b) & p > 3\\
%(a,b,b) & p = 2,3, K \neq \mathbb{Q}(\sqrt{-2})\\
%(b,b,b) & p = 2, K = \mathbb{Q}(\sqrt{-2}).
%\end{cases}
%$$
\end{proof}

By Theorem \ref{pintegraltheorem2} (3), 
$$\mathbb{Z}_{\ge 0} \ni j \mapsto \theta_q(D_1^jG_{1,\eta}^{\flat}(q_{\mathrm{dR}})) = D_1^jG_{1,\eta}^{\flat}(q_{\mathrm{dR}})|_{q_{\mathrm{dR}} = 1} \in \hat{\mathcal{O}}_{\mathcal{Y}^{\mathrm{Ig}}(\epsilon_0/p^a)}^+(\mathcal{Y}^{\mathrm{Ig}}(\epsilon_0/p^a))$$
extends to a continuous function
\begin{equation}\label{intermediatecontinuous}\mathbb{Z}/(p-1) \times \mathbb{Z}_p \ni j \mapsto \theta_q(D_1^jG_{1,\eta}^{\flat}(q_{\mathrm{dR}})) = D_1^jG_{1,\eta}^{\flat}(q_{\mathrm{dR}})|_{q_{\mathrm{dR}} = 1} \in \hat{\mathcal{O}}_{\mathcal{Y}^{\mathrm{Ig}}(\epsilon_0/p^a)}^+(\mathcal{Y}^{\mathrm{Ig}}(\epsilon_0/p^a)).
\end{equation}
From Lemma \ref{sequencelemma}, we see that for all $\frak{b} \in S$,
$$\mathbb{Z}_{\ge 0} \ni j \mapsto \lambda_E^{-j}(\frak{b}) \in \mathcal{O}_{K_p}^{\times}$$
extends to a continuous function 
\begin{equation}\label{intermediatecontinuous2}\mathbb{Z}/(p-1)\times \mathbb{Z}_p \ni j \mapsto \lambda_E^{-j}(\frak{b}).
\end{equation}
By the continuity of (\ref{intermediatecontinuous2}) and the continuity of the specialization of (\ref{intermediatecontinuous}) along 
$$\hat{\mathcal{O}}_{\mathcal{Y}^{\mathrm{Ig}}(\epsilon_0/p^a)}^+(\mathcal{Y}^{\mathrm{Ig}}(\epsilon_0/p^a)) \rightarrow \hat{\mathcal{O}}_{\mathcal{Y}^{\mathrm{Ig}}(\epsilon_0/p^a)}^+(\mathcal{Y}^{\mathrm{Ig}}(\epsilon_0/p^a))(y_{\frak{b},a}) \subset \mathcal{O}_{\mathbb{C}_p}$$
for every $\frak{b} \in S$ (where $y_{\frak{b},a} \in \mathcal{Y}^{\mathrm{Ig}}(\epsilon_0/p^a)(\overline{\mathbb{Q}}_p,\overline{\mathbb{Z}}_p)$ is as in (\ref{YIgin2})), we see that
$$\mathbb{Z}_{\ge 0} \ni j \mapsto \theta_q(D_1^j\mathcal{L}_{\lambda}) = D_1^j\mathcal{L}_{\lambda}|_{q_{\mathrm{dR}} = 1} \in \mathcal{O}_{\mathbb{C}_p}$$
extends to a continuous function
\begin{equation}\label{Econtinuous}\mathbb{Z}/(p-1) \times \mathbb{Z}_p \ni j \mapsto \theta_q(D_1^j\mathcal{L}_{\lambda}) = D_1^j\mathcal{L}_{\lambda}|_{q_{\mathrm{dR}} = 1} \in \mathcal{O}_{\mathbb{C}_p}.
\end{equation}
%This fact is used in the proof of Lemma \ref{continuitylemma}.

\subsection{Anticyclotomic $GL_1/K$ $p$-adic $L$-functions}Let
\begin{equation}\label{cnotation}\varrho^c = \varrho \circ c
\end{equation}
for any Hecke character $\varrho$ over $K$, where $c : \mathbb{A}_K^{\times} \xrightarrow{\sim} \mathbb{A}_K^{\times}$ is complex conjugation. We will assume that the quaternion algebra $D = M_2(\mathbb{Q})$ for the rest of Section \ref{padicLfunctionsection} which in particular implies $Y$ is a modular curve (see Convention \ref{Yconvention}). 

\begin{theorem}\label{interpolation2}Suppose $\lambda$ satisfies Assumption \ref{pconductorassumption}. Consider the function
$$\mathbb{Z}/(p-1) \times \mathbb{Z}_p \ni j \mapsto D_1^j\mathcal{L}_{\lambda}|_{q_{\mathrm{dR}} = 1} \in \mathcal{O}_{\mathbb{C}_p}$$
from (\ref{Econtinuous}). Then we have the following interpolation property: there exists a constant $C_{\lambda} \in \overline{\mathbb{Q}}^{\times}$ depending only on $\lambda$ and the choice of $\varpi \in \mathcal{O}_K$ made in (\ref{pichoice}), such that for any $j \in \mathbb{Z}_{\ge 0} \subset \mathbb{Z}/(p-1) \times \mathbb{Z}_p$ (embedded diagonally), we have the following equality of elements of $\overline{\mathbb{Q}}$:
\begin{equation}\label{interp1twist}\Omega_p(a)^{1+2j}\cdot (D_1^j\mathcal{L}_{\lambda})|_{q_{\mathrm{dR}} = 1} = \Omega_{\infty}(a)^{1+2j}p^{aj}\varpi^{2j}\cdot \frac{(2\pi)^jj!}{\sqrt{|D_K|^j}}\cdot C_{\lambda}\cdot L((\lambda_E^c/\lambda_E)^j\lambda^{-1},0)\\.
\end{equation}
%Here, $L^{(\frak{n})}$ denotes the $L$-function with all the Euler factors at primes dividing $\frak{n}$ removed. 
The quantities $\Omega_p(a) \in \mathbb{C}_p^{\times}$ and $\Omega_{\infty}(a) \in \mathbb{C}^{\times}$ are as in Definition \ref{finalperiodsdefinition}, and where $a \in \mathbb{Z}_{\ge 0}$ is as in (\ref{adefinition}). 
\end{theorem}

\begin{remark}By Assumption \ref{pconductorassumption} (2), we have that $\frak{p}|\frak{f}$ which implies $\frak{p}|\frak{f}((\lambda_E^c/\lambda_E)^j\lambda^{-1})$. Thus the Euler factor of $L((\lambda_E^c/\lambda_E)^j\lambda^{-1},0)$ at $\frak{p}$ is 1, and there is no Euler factor to remove as in \cite[Theorem II.4.14]{deShalit}. 
\end{remark}

\begin{remark}[Analysis of the interpolation formula]\label{interpolationcomparisonremark}We make a few remarks comparing the interpolation formula (\ref{interp1twist}) with the interpolation formula of \cite[Theorem II.4.14]{deShalit}, in particular explaining the appearance of the factor $p^{aj}\varpi^{2j}$ on the right-hand side of (\ref{interp1twist}).

\begin{enumerate}
\item The analogue of the complex period $\Omega \in \mathbb{C}^{\times}$ defined in \cite[(5.1.16)]{BDP}, which we here call $\Omega_{\infty} \in \mathbb{C}^{\times}$, is defined by the relation
\begin{equation}\label{period1}w_0 = \Omega_{\infty}\cdot (2 \pi i dz)(\mathbb{C}/(\mathbb{Z} + \mathbb{Z}\varpi p^a))
\end{equation}
on $\mathbb{C}/(\mathbb{Z}+ \mathbb{Z}\varpi p^a)$ (see loc. cit.); here $(2\pi idz)(\mathbb{C}/(\mathbb{Z} + \mathbb{Z}\varpi p^a))$ is the unique invariant differential on $\mathbb{C}/(\mathbb{Z} + \mathbb{Z}\varpi p^a)$ with
$$\int_{e_2}(2\pi i dz)(\mathbb{C}/(\mathbb{Z} + \mathbb{Z}\varpi p^a)) = 1,$$
where $(e_1,e_2) : \mathbb{Z}^{\oplus 2} \xrightarrow{\sim} H_1(\mathbb{C}/(\mathbb{Z}+ \mathbb{Z}\varpi p^a ),\mathbb{Z})$ is the basis $(1,\varpi p^a)$ of the period lattice $\mathbb{Z}\varpi p^a + \mathbb{Z}$. (Recall here that $H_1$ denotes singular homology.)

\item On the other hand, $\Omega_{\infty}(a) \in \mathbb{C}^{\times}$ in Theorem \ref{interpolation2} arises by the relation
\begin{equation}\label{period2}\Omega_{\infty}(a) \cdot w_0 = (2\pi idz)(\mathbb{C}/(\mathbb{Z}\tau_0/p^a + \mathbb{Z}))
\end{equation}
on $\mathbb{C}/(\mathbb{Z}\tau_0/p^a + \mathbb{Z}) \overset{(\ref{tau0})}{=} \mathbb{C}/(\mathbb{Z}(1/(\varpi p^a)) + \mathbb{Z})$ (note that $\Omega_{\infty}(a)$ is on the opposite side of $\Omega_{\infty}$ in (\ref{period1})); here $(2\pi idz)(\mathbb{C}/(\mathbb{Z}\tau_0/p^a + \mathbb{Z}))$ is the differential with
$$\int_{e_2'}(2\pi idz)(\mathbb{C}/(\mathbb{Z}\tau_0/p^a + \mathbb{Z})) = 1,
$$
where $(e_1',e_2') :  \mathbb{Z}^{\oplus 2} \xrightarrow{\sim} H_1(\mathbb{C}/(\mathbb{Z}\tau_0/p^a + \mathbb{Z}),\mathbb{Z})$ is the basis $(\tau_0/p^a,1)$ of the period lattice $\mathbb{Z}\tau_0/p^a + \mathbb{Z}$. 
\item Thus 
$$(2 \pi idz)(\mathbb{C}/(\mathbb{Z}\tau_0/p^a + \mathbb{Z})) = \frac{(2\pi idz)(\mathbb{C}/(\mathbb{Z}\varpi p^a + \mathbb{Z}))}{p^a\varpi},$$
and so (\ref{period1}) and (\ref{period2}) imply
$$\Omega_{\infty}^{-1} = \Omega_{\infty}(a)\cdot p^a\varpi.$$
Thus we have 
\begin{equation}\label{Omegainftys}\Omega_{\infty}^{-(1+2j)} = \Omega_{\infty}(a)^{1+2j} \cdot (p^a\varpi)^{1+2j}.
\end{equation}
\item Note that 
$$\mathrm{vol}(\mathcal{O}_{p^a}) \overset{(\ref{OKtrivialize}),\; \text{Definition} \; \ref{Ofdefinition}}{=} \mathrm{vol}(\mathbb{Z} + \mathbb{Z}\varpi p^a) = p^a\mathrm{Im}(\varpi) \overset{(\ref{pichoice})}{=} p^a\frac{\sqrt{|D_K|}}{2}.$$
Thus the term $\Omega_{\infty}(a)^{1+2j}p^{aj}\varpi^{2j}$ appearing on the right-hand side of (\ref{interp1twist}) can be rewritten as
\begin{align*}\Omega_{\infty}(a)^{1+2j}p^{aj}\varpi^{2j} \overset{(\ref{Omegainftys})}{=} \Omega_{\infty}^{-(1+2j)}\cdot p^{-a(1+j)}\varpi^{-1} &=  \Omega_{\infty}^{-(1+2j)} \cdot \mathrm{vol}(\mathcal{O}_{p^a})^{-j}\cdot \left(\frac{\sqrt{|D_K|}}{2}\right)^jp^{-a}\varpi^{-1}.
\end{align*}
\item Therefore we can rewrite (\ref{interp1twist}) as
\begin{equation}\label{interp1twistalternate}\Omega_p(a)^{1+2j}\cdot (D_1^j\mathcal{L}_{\lambda})|_{q_{\mathrm{dR}} = 1} = \Omega_{\infty}^{-(1+2j)}\cdot \mathrm{vol}(\mathcal{O}_{p^a})^{-j}\pi^jj!p^{-a}\varpi^{-1}C_{\lambda}\cdot L((\lambda_E^c/\lambda_E)^j\lambda^{-1},0).
\end{equation}
This is in direct analogy with the interpolation formula of \cite[Theorem II.4.14]{deShalit}. (Note that our $\Omega_p(a)$ is the analogue of $\Omega_p^{-1}$ in op. cit., where $\Omega_p$ is defined in Chapter II.4.4 of op. cit.)%, except that a factor of $\left(2/\sqrt{|D_K|}\right)^j$ has been removed. %This latter phenomenon can be explained by the fact that $p|D_K$ our $p$ ramified in $K$ setting, and $4|D_K$ when $p = 2$; thus we have canceled out ``extra $p$-denominators'' that would otherwise occur. 
\end{enumerate}
\end{remark}

%\begin{remark}\label{interpolation2remark}We will show in Lemma \ref{continuitylemma} that the assignment
%$$\mathbb{Z}_{\ge 0} \ni j \mapsto D_1^j\mathcal{L}_{\lambda}|_{q_{\mathrm{dR}} = 1} \in \mathcal{O}_{\mathbb{C}_p}$$ extends to a continuous function $\mathbb{Z}/(p-1) \times \mathbb{Z}_p \rightarrow \mathcal{O}_{\mathbb{C}_p}$ (viewing $\mathbb{Z}_{\ge 0} \subset \left(\mathbb{Z}/(p-1) \times \mathbb{Z}_p\right)$ as embedded diagonally). This continuous extension is our $p$-adic $L$-function attached to the Hecke character $\lambda$, which we continue to denote by $D_1^j\mathcal{L}_{\lambda}|_{q_{\mathrm{dR}}=1}$.
%\end{remark}

\begin{proof}[Proof of Theorem \ref{interpolation2}]Recall
$$D_1^j\mu_{\frak{b}} \overset{(\ref{derivativelocalmeasure})}{=} D_1^jG_{1,\eta}^{\flat}(y_{\frak{b},a})(q_{\mathrm{dR}}).$$
We have
\begin{equation}\label{interpcalc'}\begin{split}D_1^j\mathcal{L}_{\lambda}|_{q_{\mathrm{dR}} = 1} &\overset{(\ref{globalmeasure2})}{=} \sum_{\frak{b} \in S}(\lambda^{-1}\lambda_E^{-2j})(\frak{b})\cdot D_1^j\mu_{\frak{b}}|_{q_{\mathrm{dR}} = 1} \\
&\overset{(\ref{derivativelocalmeasure})}{=} \sum_{\frak{b} \in S}(\lambda^{-1}\lambda_E^{-2j})(\frak{b})\cdot D_1^jG_{1,\eta}^{\flat}(y_{\frak{b},a})(q_{\mathrm{dR}})|_{q_{\mathrm{dR}}= 1} \\
%&\overset{(\ref{chainidentity})}{=} \frac{1}{p^n}\sum_{\frak{b}}\sum_{a = 0}^{p^n-1}(\lambda^{-1}\lambda_E^{-2j})(\frak{b})\chi(\gamma^{r_p(a)})p^{-jn}\theta_t\left(\gamma_{-a/p^n,n}^*(\partial_1^jw_{1,\eta}^{\flat})(y_{\frak{b}})(q_{\mathrm{dR}})\right)|_{q_{\mathrm{dR}} = 1} \\
%&\overset{(\ref{specializechain})}{=} \frac{1}{p^{(j+1)n}}\sum_{\frak{b}}\sum_{a = 0}^{p^n-1}(\lambda^{-1}\lambda_E^{-2j})(\frak{b})\chi(\gamma^{r_p(a)})D_1^jG_{1,\eta}^{\flat}(y_{\frak{b}}\cdot \gamma_{-a/p^n,n})(q_{\mathrm{dR}})|_{q_{\mathrm{dR}} = 1}\\
&\overset{(\ref{qexpansionrelate}), (\ref{compareMSvalues})}{=} \frac{\Omega_{\infty}(a)^{1+2j}}{\Omega_p(a)^{1+2j}}\sum_{\frak{b} \in S}(\lambda^{-1}\lambda_E^{-2j})(\frak{b})\cdot\partial_1^j\left((g^a)^*E_{1,\eta}^{\flat}\right)(y_{\frak{b},a})\\
&= \frac{\Omega_{\infty}(a)^{1+2j}}{\Omega_p(a)^{1+2j}}\sum_{\frak{b} \in S}(\lambda^{-1}\lambda_E^{-2j})(\frak{b})\cdot \partial_1^j\left((g^a)^*E_{1,\eta}^{\flat}\right)(y_{\frak{b},a})
%&= \Omega_p^{-(1+2j)}\sum_{\frak{b}}(\lambda^{-1}\lambda_E^{-2j}\chi)(\frak{b}')\left(\Omega_{\infty}^{1+2j}\partial_1^jE_{1,\eta}^{\flat}(y_{\frak{b}})\right),
\end{split}
\end{equation}
Here when invoking (\ref{compareMSvalues}) (applied to $w' = w_{1,\eta}^{\flat} :=  (g^a)^*\tilde{w}_{1,\eta}^{\flat} \overset{(\ref{qexpansionrelate})}{\in} \omega_{\mathbb{Y}(\Gamma_1(p^a))}(\mathbb{Y}_{\infty}(\Gamma))$), we use that 
$$y_{\frak{b},a} = y_{\frak{b}} \cdot g^{-a}\overset{(\ref{YIgin2})}{\in} \mathcal{Y}^{\mathrm{Ig}}(\epsilon_0/p^a) \subset \mathcal{Y}^{\mathrm{Ig}}(\epsilon_0) = U \overset{(\ref{Uminclusion})}{\subset} U_{n(\epsilon_0)-1}$$
so that Assumption (2) of Theorem \ref{CMcoincidetheorem} is satisfied.  
%where the final equality follows from the same reasoning as in the proof of (\ref{interp1}) for the presence of the factor $\varpi^{1+2j} p^{jn+(1+2j)\beta}$ (note $\lambda^{-1}\lambda_E^{-2j}$ appears in the above expression as opposed to $\lambda^{-1}$), with the final sum ranging over 
%$$y_{\frak{b}'} = y_{\frak{b}}\cdot \gamma_{-a/p^n,n}\rho(\varpi p^{n+\beta}), \hspace{.25cm} \frak{b} \in \mathcal{C}\ell(\frak{f}^{(p)})[p^{\beta}], \hspace{.25cm} 0 \le a \le p^n-1, \hspace{.25cm}\frak{b}' = \frak{b}\left(\mathbb{Z}\gamma^{r_p(a)} + p^{n+\beta}\mathcal{O}_K\right).$$

Recall 
$$y_{\frak{b},a} = (A_{\frak{b},a},e_1^a,e_2^a) \overset{(\ref{YIgin})}{\in} \mathcal{Y}^{\mathrm{Ig}}(\epsilon_0/p^a).$$
In particular, $\mathbb{Z}_p \cdot (e_{1,1}^a \cdot \frak{b})$ (see (\ref{eindefinition})) is the canonical subgroup $A_{\frak{b},a}[\frak{p}]$ of $A_{\frak{b},a}$. Thus by Proposition \ref{Vproposition} and an analogous argument to that in \cite[proof of Theorem 5.9]{BDP}, we see that 
\begin{align*}\sum_{\frak{b} \in S}(\lambda^{-1}\lambda_E^{-2j})(\frak{b})\partial_1^j\left((g^a)^*E_{1,\eta}^{\flat}\right)(y_{\frak{b},a}) &= (1-(\lambda^{-1}\lambda_E^{-2j})(\frak{p}))\sum_{\frak{b} \in S}(\lambda^{-1}\lambda_E^{-2j})(\frak{b})\partial_1^j\left((g^a)^*E_{1,\eta}\right)(y_{\frak{b},a})\\
&=\sum_{\frak{b} \in S}(\lambda^{-1}\lambda_E^{-2j})(\frak{b})\partial_1^j\left((g^a)^*E_{1,\eta}\right)(y_{\frak{b},a}),
\end{align*}
the last equality using the fact that $\lambda^{-1}\lambda_E^{-2j}$ is ramified at $\frak{p}$ by Assumption \ref{pconductorassumption} (2), so that $\lambda^{-1}\lambda_E^{-2j}(\frak{p}) = 0$. 
We will show that
\begin{align*}\sum_{\frak{b} \in S}(\lambda^{-1}\lambda_E^{-2j})(\frak{b})\partial_1^j\left((g^a)^*E_{1,\eta}\right)(y_{\frak{b},a}) &= \sum_{\frak{b} \in S}(\lambda^{-1}\lambda_E^{-2j})(\frak{b})\partial_1^j\left((g^a)^*E_{1,\eta}\right)(y_{\frak{b}}) \\
&= \frac{2^jj!}{\sqrt{|D_K|^j}}p^{aj}\varpi^{2j}L^{(\frak{p})}((\lambda_E^c/\lambda_E)^j\lambda^{-1},0)C_{\lambda}
\end{align*}
for some constant $C_{\lambda}$ depending only on $\lambda$, $\lambda_E$ and $\varpi$.
%, where the last equality follows from (\ref{An'complex}) with $n = a$ and the calculations of \cite{KatzCM} (see also \cite[Chapter II.4.14]{deShalit}). 
The first equality is immediate. For the second equality, we use the computation (\cite[II.3.1 (5)-(6)]{deShalit}. The key step is the following: note that in the notation of II.2.1 (4) and II.3.1 (4) of op. cit., $A(L) = \pi^{-1} \cdot \mathrm{vol}(L)$ where $\mathrm{vol}(L) = \mathrm{Area}(\mathbb{C}/L)$, and $j$ in II.3.1 (5)-(6) of op. cit. is equal to our $-j$, and thus $A(L)^j$ in loc. cit. is $A(L)^{-j}$ in our notation)
\begin{align*}&\partial_1^j\left((g^a)^*E_{1,\eta}\right)(y_{\frak{b},a}) \\
&= \mathrm{vol}\left(\mathbb{C}/(\frak{b}^{-1}(\mathbb{Z}+\mathbb{Z}\tau_0/p^a))\right)^{-j}\pi^j j!\\
&\hspace{3.65cm}\sum_{i = 0}^{p^a-1}\sum_{(n,m) \in \mathbb{Z}^{\oplus 2}}\eta(i)\frac{\overline{\left(\frac{i\tau_0}{p^a} + \frak{b}^{-1}(n + m\tau_0)\right)}^j}{\left(\frac{i\tau_0}{p^a} + \frak{b}^{-1}(n + m\tau_0)\right)^{1+j}}\left|\left(\frac{i\tau_0}{p^a} + \frak{b}^{-1}(n + m\tau_0)\right)\right|^{-2s}|_{s = 0}\\
&\overset{(\ref{volumecomputation})}{=} |\mathbb{N}(\frak{b}\varpi)|^jp^{aj}\frac{\pi^j j!}{\mathrm{Im}(\varpi)^j}\\
&\hspace{3.42cm}\cdot \sum_{i = 0}^{p^a-1}\sum_{(n,m) \in \mathbb{Z}^{\oplus 2}}\eta(i)\frac{\overline{\left(\frac{i\tau_0}{p^a} + \frak{b}^{-1}(n + m\tau_0)\right)}^j}{\left(\frac{i\tau_0}{p^a} + \frak{b}^{-1}(n + m\tau_0)\right)^{1+j}}\left|\left(\frac{i\tau_0}{p^a} + \frak{b}^{-1}(n + m\tau_0)\right)\right|^{-2s}|_{s = 0}\\
&= |\mathbb{N}(\frak{b})|^j\frac{\pi^j j!}{\mathrm{Im}(\varpi)^j}p^{aj+1}\varpi^{1+2j}\\
&\hspace{2.5cm}\cdot |p^a\varpi|^{2s}\sum_{i = 0}^{p^a-1}\sum_{(n,m) \in \mathbb{Z}^{\oplus 2}}\eta(i)\frac{\overline{\left(i + \frak{b}^{-1}p^a(n\varpi + m)\right)}^j}{\left(i +  \frak{b}^{-1}p^a(n\varpi + m)\right)^{1+j}}\left|\left(i + \frak{b}^{-1}p^a(n\varpi + m)\right)\right|^{-2s}|_{s = 0}\\
&= |\mathbb{N}(\frak{b})|^j\frac{\pi^j j!}{\mathrm{Im}(\varpi)^j}p^{aj+1}\varpi^{1+2j}\\
&\hspace{2.5cm}\cdot|p^a\varpi|^{2s}\sum_{i = 0}^{p^a-1}\sum_{(n,m) \in \mathbb{Z}^{\oplus 2}}\eta(i)\frac{\overline{\left(i + \frak{b}^{-1}p^a(n\varpi + m)\right)}^j}{\left(i +  \frak{b}^{-1}p^a(n\varpi + m)\right)^{1+j}}\left|\left(i + \frak{b}^{-1}p^a(n\varpi + m)\right)\right|^{-2s}|_{s = 0}
\end{align*}
where ``$\overline{x}$'' denotes the complex conjugate of $x \in \mathbb{C}$, $\mathrm{Im}(\varpi)$ denotes the imaginary part of $\varpi \in \mathcal{H}^+$ (from (\ref{pichoice})), we recall that $\tau_0 = 1/\varpi$ (\ref{tau0}), and the right-hand side of the above expression is evaluated using Hecke's trick of analytic continuation (see \cite[Chapter III.3.4]{KatzImQuad}). Here, since 
$$\frak{b}^{-1}(\mathbb{Z} + \mathbb{Z}\tau_0/p^a) = \frak{b}^{-1}p^{-a}\varpi^{-1}(\mathbb{Z}\varpi p^a+ \mathbb{Z}),$$
the volume of $\mathbb{C}/(\frak{b}^{-1}(\mathbb{Z}+\mathbb{Z}\tau_0/p^a))$ is 
\begin{equation}\label{volumecomputation}\begin{split}\mathrm{vol}\left(\mathbb{C}/(\frak{b}^{-1}(\mathbb{Z}+\mathbb{Z}\tau_0/p^a))\right) = |\mathbb{N}(\frak{b}p^a\varpi)|^{-1} \cdot \mathrm{vol}(\mathbb{C}/(\mathbb{Z}\varpi p^a + \mathbb{Z})) &= |\mathbb{N}(\frak{b}p^a\varpi)|^{-1}\cdot\mathrm{Im}(p^a\varpi) \\
&= |\mathbb{N}(\frak{b}\varpi)|^{-1}p^{-a}\mathrm{Im}(\varpi),
\end{split}
\end{equation}
where $\mathbb{N}(\frak{b}p^a\varpi) \subset \mathbb{Z}$ denotes the ideal norm and $|\mathbb{N}(\frak{b}p^a\varpi)|$ denotes the unique generator in $\mathbb{Z}_{>0}$ of $\mathbb{N}(\frak{b}\varpi)$, and $\mathrm{Im}(\tau) \in \mathbb{R}$ denotes the imaginary part of $\tau \in \mathbb{C}$. Moreover 
$$\mathrm{Im}(\varpi) = \frac{\sqrt{|D_K|}}{2}$$
(see (\ref{pichoice})), recalling that $D_K \in \mathbb{Z}_{< 0}$ denotes the fundamental discriminant of $K$. Also note that the $p^a$-torsion point $ie_{2,a}$ corresponds to 
$$\frac{i\tau_0}{p^a} \in \mathbb{C}/(\mathbb{Z} +\mathbb{Z}\tau_0) \cong A(\mathbb{C}) \ni ie_{2,a}$$
under the complex uniformization. Thus, since $\eta = \lambda|_{\mathbb{A}_{\mathbb{Q}}^{\times}}\mathbb{N}_{\mathbb{Q}}^{-1} = \lambda_E|_{\mathbb{A}_{\mathbb{Q}}^{\times}}\mathbb{N}_{\mathbb{Q}}^{-1}$, we have 
\begin{align*}(\lambda^{-1}\lambda_E^{-2j})&(\frak{b})\partial_1^jE_{1,\eta}(y_{\frak{b}}) \\
&= \frac{(2\pi)^jj!}{\sqrt{|D_K|}^j}p^{aj+1}\varpi^{1+2j}|p^a\varpi|^{2s}\sum_{\frak{a} \subset \mathcal{O}_K, [\frak{a}] = [\frak{b}]}\lambda^{-1}(\frak{a})(\lambda_E^c/\lambda_E)^j(\frak{a})\left|\mathbb{N}(\frak{a})\right|^{-2s}|_{s = 0},\\
%&= \frac{(2\pi)^jj!}{\sqrt{|D_K|}^j}\sum_{\frak{a} \subset \mathcal{O}_K, [\frak{a}] = [\frak{b}]}(\lambda^{-1}(\lambda_E^c/\lambda_E)^j)(\frak{a})\left|\mathbb{N}(\frak{a})\right|^{-2s}|_{s = 0} ,
\end{align*}
where $[\frak{a}]$ denotes the class of an ideal $\frak{a} \subset \mathcal{O}_K$ in $\mathcal{C}\ell(\frak{f}^{(p)})[p^a]$, and $[\frak{b}] \in \mathcal{C}\ell(\frak{f}^{(p)})[p^a]$ is the ideal class of $\frak{b} \in S \subset \mathbb{A}_K^{\times,(pN\infty)}$. Together with (\ref{interpcalc'}), this finishes the proof. 

\end{proof}

\section{Anticyclotomic $p$-adic $L$-functions for $GL_2/\mathbb{Q}$}\label{padicLfunctionsection2}

Let $\lambda$ be as in Assumption \ref{pconductorassumption} and continue to assume that $(p,K)$ is as in Assumption \ref{pramifiedassumption}.

\begin{assumption}Assume for the remainder of Section \ref{padicLfunctionsection2} that the Hecke $L$-function $L(\lambda,s)$ has global root number -1 (in the sense of \cite[Section 4.5]{Tatesthesis}). 
\end{assumption}

\subsection{A ``good'' twist of $\lambda$}\label{goodtwistsection}%For later applications, it will be necessary to consider a certain twist of our measure by a moment character for $\Gamma_-$. This moment character is of the form $\lambda/\lambda^c$, where $\psi$ is a twist of $\lambda$ (meaning $\psi/\lambda$ is a finite order character) with good reduction at $\frak{p}$. 

For our purposes, we will need to consider a certain finite order twist $\psi = \lambda/\chi_0$ of $\lambda$ as well as its associated theta series $\theta_{\psi}$. Our strategy for choosing a suitable twist follows closely the strategy of \cite[Section 3E]{BDP2}, and as in loc. cit. we will make use of this twist in order to circumvent the Heegner hypothesis. Recall we are working under Assumption \ref{Nf0assumption}; in particular $f_0|N$ where $f_0\mathcal{O}_K = \frak{f}^{(p)}$ from Assumption \ref{pconductorassumption} (1), and $(N,p) = 1$. For any finite prime $v$ of $\mathcal{O}_K$, consider the local character $\lambda|_{\mathcal{O}_{K_v}^{\times}} : \mathcal{O}_{K_v}^{\times} \rightarrow \overline{\mathbb{Q}}^{\times}$. If  $\lambda = \lambda_E$ as in Example \ref{Eexample}, by the theory of complex multiplication we have that $\lambda_{v}(\mathcal{O}_{K_v}^{\times}) \subset \mathcal{O}_K^{\times}$ (see \cite[Top of p. 170]{Silverman}). 

Recall the notation of $\varrho^c = \varrho \circ c$ from (\ref{cnotation}). Recall that given an algebraic Hecke character $\chi : K^{\times} \backslash \mathbb{A}_K^{\times} \rightarrow \mathbb{C}^{\times}$, we let $\frak{\frak{f}}(\chi) \subset\mathcal{O}_K$ denote its conductor. Let $\varepsilon(\varrho)$ denote the global root number of $\varrho$ and given a place $v$ of $K$ let $\varepsilon_v(\varrho)$ denote the local root number at $v$ (see \cite[Theorem 2.4.1, Section 4.5]{Tatesthesis}). Then $\varepsilon(\varrho) = \prod_v\varepsilon_v(\varrho)$ by loc. cit. We will need the following Lemma in order to later ensure (in particular) that the quaternion algebra $D/\mathbb{Q}$ we consider later (the unique $D/\mathbb{Q}$ satisfying (\ref{sign2})) is split at $p$. 

\begin{lemma}\label{modifylemma}%For every finite prime $v$ of $\mathcal{O}_K$, choose $\varepsilon_v \in \{ \pm 1\}$. 
%Suppose we are given %an ideal $\mathcal{S} \subset \mathcal{O}_K$ with $(\mathcal{S},\frak{p}) = 1$ and a finite order 
%a finite order Hecke character 
%$$\chi_1 : K^{\times} \backslash\mathbb{A}_K^{\times} \rightarrow \overline{\mathbb{Z}}^{\times}$$
%and 
%any prime-to-$p$ ideal $\frak{n} \subset \mathcal{O}_K$. 
%such that
%\begin{equation}\label{oddpconductorassumption}\mathrm{ord}_{\frak{p}}(\frak{f}(\lambda^c\chi_1/\chi_1^c)) \hspace{.25cm} \text{is odd}
%\end{equation}
%and if $p = 2$ and $K = \mathbb{Q}(i)$, further assume that
%\begin{equation}\label{furtherpconductorassumption}\mathrm{ord}_{\frak{p}}(\frak{f}(\lambda^c\chi_1/\chi_1^c)) \ge 7 \hspace{.25cm} \text{and} \hspace{.25cm} \mathrm{ord}_{\frak{p}}(\frak{f}(\lambda^c\chi_1/\chi_1^c)) \equiv 1 \pmod{4}.
%\end{equation}
Choose any element 
$$\varepsilon \in \{\pm 1\}.$$
%such that
%$$\varepsilon_{\frak{p}} = \varepsilon_{\frak{p}}(\lambda^c)$$
%if $p = 2$ and $K = \mathbb{Q}(i)$. 
Then there exist infinitely many finite order Hecke characters 
$$\chi : K^{\times}\backslash\mathbb{A}_K^{\times} \rightarrow \overline{\mathbb{Z}}^{\times}$$ such that, letting $\chi_v = \chi|_{K_v^{\times}}$, %and $\frak{g}_{\frak{p}} = \frak{g}/\frak{g}^{(p)}$ denote the $\frak{p}$-primary part of an ideal $\frak{g} \subset \mathcal{O}_K$, we have 
\begin{enumerate}
\item $\chi_v|_{\mathcal{O}_{K_v}^{\times}} = \lambda_v|_{\mathcal{O}_v}^{\times}$ for all $v|pN$, 
%\item $\frak{f}(\chi)_{\frak{p}}|\frak{f}(\lambda)_{\frak{p}}$,
%\item $(\frak{f}(\chi)^{(p)},\frak{f}(\lambda)^{(p)}) = 1$,
%\item $\varepsilon_{\frak{p}}(\lambda^c\chi/\chi^c)\chi_{\frak{p}}(-1) = \varepsilon_{\frak{p}}$,
%\item $\chi_{\frak{p}}(-1) = \varepsilon_{\frak{p}}$,
\item $\varepsilon(\lambda^c\chi/\chi^c) = \varepsilon$.
\end{enumerate}
%If moreover $\mathrm{ord}_{\frak{p}}(\frak{f}(\chi_1)) > \mathrm{ord}_{\frak{p}}(\frak{f}(\lambda))$, we may further find infinitely many $\chi$ satisfying the above as well as 
%$$\mathrm{ord}_{\frak{p}}(\frak{f}(\chi)) > \mathrm{ord}_{\frak{p}}(\frak{f}(\lambda)).$$
\end{lemma}

\begin{proof}

Any continuous character 
$$K^{\times}\backslash \mathbb{A}_K^{\times}/\mathbb{C}^{\times} = \left(\prod_{v\nmid \infty}\mathcal{O}_{K_v}^{\times}\right)/\mathcal{O}_K^{\times} \rightarrow \overline{\mathbb{Z}}^{\times}$$
(where the equality above uses the fact that $K$ has class number 1) induces a finite order Hecke character $K^{\times}\backslash\mathbb{A}_K^{\times} \rightarrow \overline{\mathbb{Z}}^{\times}$ (and vice versa). Thus there exist infinitely many finite order Hecke characters $\varrho : K^{\times}\backslash\mathbb{A}_K^{\times} \rightarrow \overline{\mathbb{Z}}^{\times}$ such that letting $\varrho_v = \varrho|_{K_v^{\times}}$, we have 
$$\varrho_v|_{\mathcal{O}_{K_v}^{\times}} = \lambda_v|_{\mathcal{O}_{K_v}^{\times}}$$
for all $v|pN$. For any such $\varrho$, we will show that for any prime $\ell\nmid pN\frak{f}(\varrho)$ inert in $K/\mathbb{Q}$, there are infinitely many finite order Hecke characters $\chi_{\ell} : K^{\times} \backslash \mathbb{A}_K^{\times} \rightarrow \overline{\mathbb{Z}}^{\times}$ of $\ell$-power conductor such that letting 
$$\chi = \varrho\chi_{\ell},$$
we have
\begin{equation}\label{rootnumberchangerho}\varepsilon(\lambda^c\chi/\chi^c) = \varepsilon.
\end{equation}
This will show that there are infinitely many $$\chi : K^{\times}\backslash\mathbb{A}_K^{\times} \rightarrow \overline{\mathbb{Z}}^{\times}$$
satisfying (1) and (2) in the statement of the Lemma.

Let $v$ be any finite place of $K$. By local root number stability (\cite[Section 4.5]{Deligne}), for all finite order $\chi_{1v}: K_v^{\times} \rightarrow \overline{\mathbb{Z}}^{\times}$, there exists $a_{\chi_{1,v}} \in K_v^{\times}$ such that for all finite order $\chi_{2,v}: K_v^{\times} \rightarrow \overline{\mathbb{Z}}^{\times}$ with $\mathrm{ord}_v(\frak{f}(\chi_{1,v})) \ge 2\mathrm{ord}_v(\frak{f}(\chi_{2,v}))$ we have
\begin{equation}\label{rootnumberstability}\varepsilon_v(\chi_{1,v}\chi_{2,v}) = \chi_{2,v}^{-1}(a_{\chi_{1,v}})\varepsilon_v(\chi_{1,v}).
\end{equation}
Let $\pi_v$ be a uniformizer of $K_v$ and write $a_{\chi_{1,v}} = \pi_v^{e'}u_v$ where $e' \in \mathbb{Z}$ and $u_v \in \mathcal{O}_{K_v}^{\times}$. If $\chi_{2,v}$ is unramified then $\chi_{2,v}^{-1}(a_{\chi_{1,v}}) = \chi_{2,v}^{-1}(\pi_v^{e'})$. By (4.5.5) of op. cit., we have that $e' = -\mathrm{ord}_v(\frak{f}(\chi_{1,v}))$, and hence 
\begin{equation}\label{unramifiedformula}\chi_{2,v}^{-1}(a_{\chi_{1,v}}) = \chi_{2,v}^{-1}(\pi_v^{e'}) = \chi_{2,v}(\frak{f}(\chi_{1,v})).
\end{equation}

Now let $\chi_1,\chi_2 : K^{\times}\backslash \mathbb{A}_K^{\times} \rightarrow \mathbb{C}^{\times}$ be two finite order Hecke characters and let $\chi_{1,v} = \chi_1|_{K_v^{\times}}$ and $\chi_{2,v} = \chi_2|_{K_v^{\times}}$. We have
$$\frac{\varepsilon(\chi_1\chi_2)}{\varepsilon(\chi_1)\varepsilon(\chi_2)} = \prod_v \frac{\varepsilon_v(\chi_1\chi_2)}{\varepsilon_v(\chi_1)\varepsilon_v(\chi_2)}.$$
Suppose $(\frak{f}(\chi_1),\frak{f}(\chi_2)) = 1$. Then from (\ref{rootnumberstability}) and (\ref{unramifiedformula}) we have 
\begin{equation}\label{globalrootformula}\varepsilon(\chi_1\chi_2) = \varepsilon(\chi_1)\varepsilon(\chi_2)\chi_1(\frak{f}(\chi_2))\chi_2(\frak{f}(\chi_1)).
\end{equation}

Now choose such a prime $\ell > 2$ inert in $K/\mathbb{Q}$ such that $(\ell,pN\frak{f}(\varrho)) = 1$. Note that $K_{\ell}^{\times} = \ell^{\mathbb{Z}} \times \mathcal{O}_{K_{\ell}}^{\times}$ and $\mathcal{O}_{K_{\ell}}^{\times} = \Delta_{\ell} \times \Gamma_{\ell}$ where $\Delta_{\ell} = \mathcal{O}_{K_{\ell},\mathrm{tors}}^{\times}$ is the torsion submodule and $\Gamma_{\ell} \cong \mathbb{Z}_{\ell}^{\oplus 2}$. Write $\Gamma_{\ell,+} \times \Gamma_{\ell,-}$ where $\Gamma_{+,\ell} \cong \mathbb{Z}_{\ell}$ is the factor on which conjugation acts via $\mathrm{Gal}(K_{\ell}/\mathbb{Q}_{\ell})$ acts as the identity and $\Gamma_{-,\ell} \cong \mathbb{Z}_{\ell}$ is the factor on which this conjugation acts via the inverse map $\gamma \mapsto \gamma^{-1}$. (Note that this splitting exists since $\ell > 2$.) Now consider any finite order character 
$$\tilde{\chi}_{\ell} : K^{\times} \backslash \mathbb{A}_K^{\times} \rightarrow \overline{\mathbb{Z}}^{\times}$$
such that $\tilde{\chi}_{\ell}|_{\Delta_{\ell} \times \Gamma_{+,\ell}} = 1$ and $\tilde{\chi}_{\ell}(\ell) = 1$. Thus $\tilde{\chi}_{\ell} = \tilde{\chi}_{\ell}|_{\Gamma_{-,\ell}}$ and we have $\tilde{\chi}_{\ell}^c = \tilde{\chi}_{\ell}^{-1}$ i.e. $\tilde{\chi}_{\ell}$ is anticyclotomic. 
Then by the first paragraph of this proof and the fact that $\mathcal{O}_K^{\times} \subset \Delta_{\ell}$, $\tilde{\chi}_{\ell}$ is in fact a global finite order Hecke character of $\ell$-power conductor
$$\tilde{\chi}_{\ell} : K^{\times}\backslash\mathbb{A}_K^{\times} \rightarrow \overline{\mathbb{Z}}^{\times}.$$

Note that since $\frak{f}(\tilde{\chi}_{\ell})$ is a power of $\ell$ and $\ell \nmid \frak{f}(\lambda\varrho^c/\varrho) = \frak{f}(\lambda^c\varrho/\varrho^c)$ (the last equality following from Assumption \ref{pconductorassumption} (1) and the fact that $(\varrho/\varrho^c)^c = (\varrho/\varrho^c)^{-1}$), we have $(\frak{f}(\tilde{\chi}_{\ell}),\frak{f}(\lambda^c\varrho/\varrho^c)) = 1$. Hence we may specialize (\ref{globalrootformula}) to $\chi_1 = \tilde{\chi}_{\ell}$ and $\chi_2 = \lambda^c\varrho/\varrho^c$. By Assumption \ref{pconductorassumption} (2) we have that 
$$(\lambda^c\varrho/\varrho^c)_{\ell}(\ell^n) = \lambda_{\ell}^c(\ell^n)\varrho_{\ell}(\ell)/\varrho_{\ell}(\ell^c) = \lambda_{\ell}((\ell^c)^n) = \lambda_{\ell}(\ell^n) = \eta(\ell^n) = (-1)^n$$
where the last equality follows because $\ell$ is inert in $K/\mathbb{Q}$ and $\eta$ is the quadratic character attached to $K/\mathbb{Q}$. Moreover, since $\frak{f}(\lambda\varrho/\varrho^c)^c = \frak{f}(\lambda\varrho^c/\varrho)$ by Assumption \ref{pconductorassumption} (1), we have 
$$\tilde{\chi}_{\ell}(\frak{f}(\lambda^c\varrho/\varrho^c)) = \tilde{\chi}_{\ell}(\frak{f}(\lambda^c\varrho/\varrho^c)^c) = \tilde{\chi}_{\ell}^c(\frak{f}(\lambda^c\varrho/\varrho^c)) = \tilde{\chi}_{\ell}^{-1}(\frak{f}(\lambda^c\varrho/\varrho^c))$$
which implies $\tilde{\chi}_{\ell}(\frak{f}(\lambda^c\varrho/\varrho^c))^2 = 1$. Since $\tilde{\chi}_{\ell} = \tilde{\chi}_{\ell}|_{\Gamma_{\ell}}$ has $\ell$-power order and $\ell > 2$, this implies 
$$\tilde{\chi}_{\ell}(\frak{f}(\lambda^c\varrho/\varrho^c)) = 1.$$
Moreover, since $\tilde{\chi}_{\ell}$ is anticyclotomic its functional equation of its Hecke $L$-series $L(\tilde{\chi}_{\ell},s)$ simply relates $L(\tilde{\chi}_{\ell},s)$ to itself and so we have 
$$\varepsilon(\tilde{\chi}_{\ell}) = 1.$$
Hence (\ref{globalrootformula}) with $\chi_1 = \tilde{\chi}_{\ell}$ and $\chi_2 =\lambda^c\varrho/\varrho^c$ reads
$$\varepsilon(\lambda^c\varrho/\varrho^c\tilde{\chi}_{\ell}) = \varepsilon(\lambda^c\varrho/\varrho^c)(-1)^{\mathrm{ord}_{\ell}(\frak{f}(\tilde{\chi}_{\ell}))}.$$
It is clear that there are infinitely many $\tilde{\chi}_{\ell}$ as above such that $\tilde{\chi}_{\ell}|_{\Gamma_{\ell,+}} = 1$, $\tilde{\chi}_{\ell} = \tilde{\chi}_{\ell}|_{\Gamma_{\ell,-}}$ and  $\mathrm{ord}_{\ell}(\frak{f}(\tilde{\chi}_{\ell})) \equiv 0 \pmod{2}$ (resp. $1 \pmod{2}$), and hence there are infinitely many such $\tilde{\chi}_{\ell}$ with $\varepsilon(\lambda^c\chi/\chi^c) = \varepsilon$. 
%$$\chi_{\ell}' : K_{\ell}^{\times} = \ell^{\mathbb{Z}} \times \Delta_{\ell} \times \Gamma_{\ell,+} \times \Gamma_{\ell,-} \rightarrow \overline{\mathbb{Z}}^{\times}$$
%such that $\chi_{\ell}'(\ell) = 1$, $\chi_{\ell}'|_{\Delta_{\ell}} = 1$ and $\chi_{\ell}'|_{\Gamma_{\ell}}$ is finite order. 
 Recall that $\Gamma_{-,\ell} \cong \mathbb{Z}_{\ell}$ and $\ell > 2$, and hence $\Gamma_{-,\ell}$ is 2-divisible. Since $\Gamma_{\ell,-} \cong \mathbb{Z}_{\ell}$ and $\tilde{\chi}_{\ell} = \tilde{\chi}_{\ell}|_{\Gamma_{-,\ell}}$, we may therefore find a unique $\chi_{\ell} : \Gamma_{\ell,-} \rightarrow \overline{\mathbb{Z}}^{\times}$ such that $\chi_{\ell}^2 = \tilde{\chi}_{\ell}$. Since $c$ acts on $\Gamma_{\ell,-}$ by $\gamma \mapsto \gamma^{-1}$, we see that $\chi_{\ell}/\chi_{\ell}^c = \chi_{\ell}^2 = \tilde{\chi}_{\ell}$. By the first paragraph of the proof, we see that $\chi_{\ell}$ is in fact a global Hecke character of $\ell$-power conductor $\chi_{\ell} : K^{\times} \backslash \mathbb{A}_K^{\times} \rightarrow \overline{\mathbb{Z}}^{\times}$. Hence there are infinitely many $\chi_{\ell}$ such that (\ref{rootnumberchangerho}) holds. 
\end{proof}

%Choose any finite order Hecke character 
%$\chi_0 : \mathbb{A}_K^{\times} \rightarrow \overline{\mathbb{Z}}^{\times}$ with 
%\begin{equation}\label{chi0localconditions}\chi_0|_{\mathcal{O}_{K_v}^{\times}} = \lambda|_{\mathcal{O}_{K_v}^{\times}}
%\end{equation}
%for all $v|pN$.

%By the argument of \cite[Lemma 2.5]{BurungaleDisegni}, we may find a finite order Hecke character $\chi_1 : K^{\times}\backslash\mathbb{A}_K^{\times} \rightarrow \overline{\mathbb{Z}}^{\times}$ such that
%$$\mathrm{ord}_{\frak{p}}(\frak{f}(\chi_1)) > \mathrm{ord}_{\frak{p}}(\frak{f}(\lambda_E)),$$
%$$\mathrm{ord}_{\frak{p}}(\frak{f}(\lambda^c\chi_1/\chi_1^c)) \hspace{.25cm} \text{is odd}$$
%and
%$$\varepsilon(\lambda^c\chi_1/\chi_1^c) = +1.$$
%As loc. cit. shows, we may arbitrarily modify $\chi_0$ at primes $v$ not dividing $\frak{f}(\lambda)\frak{f}(\chi_0)$ in order to obtain new $\chi_0$ satisfying (\ref{+1rootnumber}), and thus we have infinitely many $\chi_0$ satisfying (\ref{chi0localconditions}) and (\ref{+1rootnumber}). Using the results of \cite{Rohrlich}, we can thus choose $\chi_0$ as above such that 
%Moreover, the argument of loc. cit. shows that we may further choose $\chi_1$ to satisfy (\ref{oddpconductorassumption}) and (\ref{furtherpconductorassumption}). 
By Lemma \ref{modifylemma} with $\varepsilon = 1$, we may find infinitely many $\chi$ such that 
$$\chi_v|_{\mathcal{O}_{K_v}^{\times}} = \lambda_v|_{\mathcal{O}_{K_v}^{\times}}$$
for all $v|pN$ and 
$$\varepsilon(\lambda^c\chi/\chi^c) = 1.$$
%and 
%$$\varepsilon_{\frak{p}}(\lambda^c\chi/\chi^c)\chi_{\frak{p}}(-1) = \varepsilon_{\frak{p}}(\lambda).$$
Hence by the main theorem of \cite{Rohrlich}, from the set of infinitely many $\chi$ as above we may find one element, which we call $\chi_0$ from now on, such that
\begin{equation}\label{nonvanishingLvalue}L(\lambda^c\chi_0/\chi_0^c,1) \neq 0.
\end{equation}
%(\ref{+1rootnumber}) holds and
%\begin{equation}\label{localrootnumbercondition}\varepsilon_{\frak{p}}(\lambda^c\chi_0/\chi_0^c) = \varepsilon_{\frak{p}}(\lambda)\chi_{0,\frak{p}}(-1).
%\end{equation}

We summarize all the properties our choice of $\chi_0$ satisfies.

\begin{choice}\label{chi0choice}Suppose we are given $\lambda$ as in Assumption \ref{pconductorassumption}. Let $\chi_0 : K^{\times}\backslash\mathbb{A}_K^{\times} \rightarrow \overline{\mathbb{Z}}^{\times}$ be a finite order Hecke character as in Lemma \ref{modifylemma}, which thus satisfies
%\begin{equation}\label{bigconductor}\mathrm{ord}_{\frak{p}}(\frak{f}(\chi_0)) > \mathrm{ord}_{\frak{p}}(\frak{f}(\lambda_E)),
%\end{equation}
\begin{equation}\label{localvcondition}\chi_{0,v}|_{\mathcal{O}_{K_v}^{\times}} = \lambda_v|_{\mathcal{O}_{K_v}^{\times}}
\end{equation}
for all places $v|pN$ of $K$, 
\begin{equation}\label{+1rootnumber}\varepsilon(\lambda^c\chi_0/\chi_0^c) = 1
\end{equation}
as well as (\ref{nonvanishingLvalue}).
% and (\ref{localrootnumbercondition}).
\end{choice}

Recall that $\eta$ is the quadratic character attached to $K/\mathbb{Q}$. In particular, since $\lambda|_{\mathbb{A}_{\mathbb{Q}}^{\times}}\mathbb{N}_{\mathbb{Q}}^{-1} = \eta$ by Assumption \ref{pconductorassumption} (3), we have
\begin{equation}\label{pNsign}\chi_{0,\ell}(-1) \overset{(\ref{localvcondition})}{=} \lambda_{\ell}(-1) = \eta_{\ell}(-1)
\end{equation}
for all $\ell|pN$, where $\lambda_{\ell} = \prod_{v|\ell}\lambda_v$ and $\chi_{0,\ell} = \prod_{v|\ell} \chi_{0,v}$ with $v|\ell$ running over all places of $K$ above $\ell$. Note that since $\eta$ is an odd Dirichlet character with conductor $p^{\beta}$, the above expression is $1$ if $\ell \neq p$ and $-1$ if $\ell = p$.  Given $\chi_0$ as in Choice \ref{chi0choice}, we will consider the following twist of $\lambda$ by $\chi_0^{-1}$.

\begin{definition}Given $\lambda$ as in Assumption \ref{pconductorassumption}, define the twist 
\begin{equation}\label{psichoice}\psi := \lambda/\chi_0 : K^{\times}\backslash\mathbb{A}_K^{\times} \rightarrow \mathbb{C}^{\times}
\end{equation}
where $\chi_0$ is as in Choice \ref{chi0choice}.
\end{definition}
Thus $\psi$ is an algebraic Hecke character of infinity type $(1,0)$, and following Convention \ref{avatarconvention} we will also use $\psi$ to denote its $p$-adic avatar. By (\ref{localvcondition}) we have that $\frak{f}(\psi)$ is prime to $p$. 

\subsection{Rankin-Selberg pairs}\label{RSsection}Continue to suppose $\chi_0$ is as in Choice \ref{chi0choice}. Let $g$ be any $GL_2$-modular form %of level $\Gamma(M')$ where $M'$ is only divisible by primes dividing $pN$ and 
 such that $(g,\chi_0)$ is a \emph{Rankin-Selberg pair}, i.e. a pair satisfying the compatibility 
$$\omega_g = \chi_0^{-1}|_{\mathbb{A}_{\mathbb{Q}}^{\times}}$$
where $\omega_g$ is the central character of $g$ (see \cite[Chapter 1.4.1]{YZZ}). Let $\pi_g$ denote the automorphic representation of $D^{\times}(\mathbb{A}_{\mathbb{Q}})$ associated with $g$. %, and let $(\pi_g)_K$ denote the base change to $K$. 
Let $\varepsilon(\pi_g,\chi_0,1/2)$ denote the global root number of the Rankin-Selberg $L$-function $L(\mathrm{JL}(\pi_g) \times \pi_{\chi_0},s)$ attached to the $GL_2 \times GL_2$-product $\mathrm{JL}(\pi_g) \times \pi_{\chi_0}$ over $\mathbb{Q}$ of the Jacquet-Langlands lifting $\mathrm{JL}(\pi_g)$ of $\pi_g$ and the automorphic representation $\pi_{\chi_0}$ of $GL_2(\mathbb{A}_{\mathbb{Q}})$ attached the theta series $\theta_{\chi_0}$ of $\chi_0$.\footnote{This is a different convention than that in \cite{YZZ}, which instead considers the base change $(\pi_g)_K$ to $K$ and the twist $(\pi_g)_K \otimes \chi_0$ by $\chi_0$. In particular, letting $\varepsilon((\pi_g)_K \otimes \chi_0,1/2)$ denote the global root number of $(\pi_g)_K \otimes \chi_0$ and $\varepsilon_v((\pi_g)_K \otimes \chi_0,1/2)$ denote the local root number of $(\pi_g)_K\otimes \chi_0$ at a place $v$ of $K$, we have $\varepsilon(\pi_g,\chi_0,1/2) = \varepsilon((\pi_g)_K \otimes \chi_0,1/2)\eta(-1)$ and $\varepsilon_v(\pi_g,\chi_0,1/2) = \varepsilon((\pi_g)_K\otimes \chi_0,1/2)\eta_v(-1)$, recalling here that $\eta$ is the quadratic character attached to $K/\mathbb{Q}$.}

Let $D$ be a quaternion algebra over $\mathbb{Q}$ which satisfies the Saito-Tunnell criterion (\cite[Theorem 1.3]{YZZ}, \cite{YZZerratum}) for every place $\ell$ of $\mathbb{Q}$:
\begin{equation}\label{sign2}\varepsilon_{\ell}(\pi_g,\chi_0,1/2) = \chi_{0,\ell}(-1)\eta_{\ell}(-1)\varepsilon_{\ell}(D),
\end{equation}
where 
$$\varepsilon_{\ell}(\pi_g,\chi_0,1/2) = \prod_{v|\ell}\varepsilon_v(\pi_g,\chi_0,1/2)$$
with $v|\ell$ running over places of $K$ above $\ell$, $\chi_{0,\ell} = \prod_{v|\ell}\chi_{0,v}$ also with $v|\ell$ running over places of $K$ above $\ell$, and $\varepsilon_{\ell}(D)$ is the local Hasse invariant at $\ell$ of $D/\mathbb{Q}$. 

In the case that $g = \theta_{\psi}$ is the weight 2 theta series attached to the infinity type $(1,0)$ Hecke character $\psi$ from (\ref{psichoice}), the Galois representation attached to $\pi_g$ is $\varrho_g|_{\mathrm{Gal}(\overline{K}/K)} \cong \psi \oplus \psi^c$. Hence by the Artin formalism,
\begin{equation}\label{Artinfactorization}L(\mathrm{JL}(\pi_g) \times \pi_{\chi_0},s-1/2) = L(\psi\chi_0,s)L(\psi^c\chi_0,s) = L(\lambda,s)L(\lambda^c\chi_0/\chi_0^c,s)
\end{equation}
and by \cite[Equations below (1.4.1) and (2.2.5)]{Gelbart}, we have 
\begin{equation}\label{Artinfactorizationrootnumber}\varepsilon(\pi_g,\chi_0,1/2) = \varepsilon(\lambda,1/2)\varepsilon(\lambda^c\chi_0/\chi_0^c,1/2), \hspace{1cm} \varepsilon_{\ell}(\pi_g,\chi_0,1/2) = \prod_{v|\ell}\varepsilon_v(\lambda,1/2)\varepsilon_v(\lambda^c\chi_0/\chi_0^c,1/2)
\end{equation}
where $v|\ell$ runs over all places of $K$ above a (finite or infinite) place $\ell$ of $\mathbb{Q}$. In particular, by our assumption $\varepsilon(\lambda) = -1$ and (\ref{+1rootnumber}), we have that the global root number attached to $\mathrm{JL}(\pi_g) \times \pi_{\chi_0}$ satisfies
\begin{equation}\label{globalsign}\varepsilon(\pi_g,\chi_0,1/2) = -1.
\end{equation}
%Now further choose the finite order Hecke character $\chi_0 : \mathbb{A}_K^{\times}/K^{\times} \rightarrow \mathcal{O}_K^{\times}$ such that 
%$$\chi_{0,\ell}(-1) = \varepsilon_{\ell}((\pi_g)_K,\chi_0,1/2)$$
%for all $\ell|N$, and such that $\chi_0|_{\mathcal{O}_{K_p}^{\times}} = \lambda|_{\mathcal{O}_{K_p}}^{\times}$, and also satisfying
%$$\varepsilon(\lambda^c\chi_0/\chi^c,1) = +1.$$
%This can always be arranged, see \cite[Lemma 2.5]{BurungaleDisegni}.
%\footnote{In fact, the above equality occurs even half the time, as was communicated to the author by W. Zhang \cite{WZhangrootnumber}.} 
%Note that for this $\chi_0$, (\ref{sign2}) implies $\varepsilon_{\ell}(D) = 1$ for every $\ell|N$, i.e. the discriminant of $D$ is prime-to-$N$. 

%Since the base change $(\pi_{g,p})_{K_p}$ of the local representation $\pi_{g,p}$ (of $GL_2(\mathbb{Q}_p)$) to $K_p$ is an unramified principal series, and $\chi_0$ has $\frak{p}$-conductor equal to $(\frak{f}(\lambda),\frak{p})$ by Assumption \ref{pconductorassumption}, we have that $\varepsilon_{\frak{p}}(g,\chi_0) = +1$ by standard local root number results. 
Moreover, by (\ref{nonvanishingLvalue}), we have 
\begin{equation}\label{analyticrankequality}\mathrm{ord}_{s = 1/2}L(\mathrm{JL}(\pi_g) \times \pi_{\chi_0},s) = \mathrm{ord}_{s = 1}L(\lambda,s).
\end{equation}  
When $\lambda = \lambda_E$ we have $L(\lambda_E,s) = L(E/\mathbb{Q},s)$ and so this is equal to
$$\mathrm{ord}_{s = 1}L(\lambda_E,s) = \mathrm{ord}_{s=1}L(E/\mathbb{Q},s) =: r_{\mathrm{an}}(E/\mathbb{Q}).$$

\begin{proposition}\label{Dsplitproposition}Suppose $g = \theta_{\psi}$, where $\psi = \lambda/\chi_0$ is as in (\ref{psichoice}) (with $\chi_0$ as in Choice \ref{chi0choice}). Any $D/\mathbb{Q}$ satisfying (\ref{sign2}) has $\varepsilon_{\ell}(D) = 1$ for all places $\ell|pN\infty$. In particular, the discriminant of $D$ is prime to $pN\infty$. 
\end{proposition}

\begin{proof}%From (\ref{localrootnumbercondition}) and (\ref{sign2}) with $\ell = p$, we see that $\epsilon_p(D) = +1$. 
Suppose $v|pN$. Then $\chi_{0,v}|_{\mathcal{O}_{K_v}^{\times}} = \lambda_v|_{\mathcal{O}_{K_v}^{\times}}$ for all $v|pN$; in particular, at each prime $v|pN$ we have that $\chi_{0,v}$ and $\lambda_v$ are equivalent characters in the sense of \cite[Discussion after Definition 2.4.1]{Tatesthesis}. Thus by the definition of local root numbers (\cite[Section 2.5]{Tatesthesis}, \cite[Section 3]{Tate2}) we have %We first consider the case when $v|\frak{f} = \frak{f}(\lambda)$, or equivalently (by Assumption \ref{pconductorassumption}) $v|pf_0$. %By our above choices, $\chi_0$ is ramified at $v$ (because $\lambda$ is), and $(\pi_g)_{K,v}$ is an unramified principal series because $\psi$ is unramified at $v$. Hence %standard properties of root numbers (\cite{Gelbart}) imply
\begin{equation}\label{sign1}\varepsilon_{\ell}(\pi_g,\chi_0,1/2) \overset{(\ref{Artinfactorizationrootnumber})}{=} \prod_{v|\ell}\varepsilon_v(\lambda)\varepsilon_v(\lambda^c\chi_0/\chi_0^c) = \prod_{v|\ell}\varepsilon_v(\lambda)\varepsilon_v(\lambda) = \prod_{\ell|v}\varepsilon_v(\lambda)^2 = 1.
\end{equation}

Thus we have, for all $\ell|pN$, 
$$\varepsilon_{\ell}(D) \overset{(\ref{sign2})}{=} \chi_{0,\ell}^{-1}(-1)\eta_{\ell}(-1)\varepsilon_{\ell}(\pi_g,\chi_0,1/2) \overset{(\ref{pNsign})}{=} \varepsilon_{\ell}(\pi_g,\chi_0,1/2) \overset{(\ref{sign1})}{=} 1.$$

%Now suppose $v\nmid \frak{f}$. Then $\lambda$ is unramified at $v$ and so by (\ref{chi0localconditions}) $\chi_0$ is unramified at $v$, and thus so is $\psi = \lambda/\chi_0$. Hence 
%\begin{equation}\label{sign1'}\varepsilon_v(\pi_g,\chi_0,1/2) = 1.
%\end{equation}
When $\ell = \infty$ we have $\varepsilon_{\infty}(\pi_g,\chi_0,1/2) = -1$ since the weight of $g$ (which is 2) is greater than the weight of $\theta_{\chi_0}$ (which is 1), see \cite[p. 3]{Brooks}. Thus since $\chi_{0,\infty}(-1) = 1$ (since $\chi_0$ is a finite order Hecke character over the imaginary quadratic field $K$, which forces $\chi_{0,\infty} = 1$) and $\eta_{\infty}(-1) = -1$ (since $\eta$ is the quadratic character attached to the imaginary quadratic field $K/\mathbb{Q}$, which implies it is odd), we have 
$$\varepsilon_{\infty}(E) = \overset{(\ref{sign2})}{=} \chi_{0,\infty}^{-1}(-1)\eta_{\infty}(-1)\varepsilon_{\infty}(\pi_g,\chi_0,1/2) = 1.$$

Thus $\varepsilon_{\ell}(D) = 1$ for all $\ell|pN\infty$, which gives the assertion. 

\end{proof}

%We henceforth also impose the assumption that $\varepsilon((\pi_g),\chi_0,1/2) = -1$, which one can always ensure by adding a prime power to the conductor of $\chi_0$ at an inert prime coprime with the conductor of $g$ and modifying the local character at that prime if necessary. 

In summary, we have the following properties of our $\chi_0$ and $D$. 

\begin{choice}\label{twistchoice}When $g = \theta_{\psi}$ ($\psi$ as in (\ref{psichoice})) we chose a finite order Hecke character $\chi_0 : \mathbb{A}_K^{\times}/K^{\times} \rightarrow \overline{\mathbb{Z}}^{\times}$ as in Choice \ref{chi0choice} such that: 
\begin{enumerate}
%\item $D/\mathbb{Q}$ is a quaternion algebra split at $p$ and $\infty$,
\item $(g,\chi_0)$ is a Rankin-Selberg pair (as in the first paragraph of Section \ref{RSsection}) satisfying (\ref{sign2}),
\item $\varepsilon(\psi^c\chi_0,1) = 1$ and
\item $L(\psi^c\chi_0,1) \neq 0$.
\end{enumerate}
%By Proposition \ref{Dsplitproposition}, we have that $D$ is split at all $\ell|pN$. 
\end{choice}
Note that by these assumptions, the infinity type $(1,0)$ Hecke character $\psi \chi_0^c = (\lambda/\chi_0)\chi_0^c$ satisfies Assumption \ref{pconductorassumption} by (\ref{localvcondition}) and the fact that $(\chi_0/\chi_0^c)|_{\mathbb{A}_{\mathbb{Q}}^{\times}} = 1$. 

\begin{definition}\label{f'notation}Let
$$\frak{f}' := \frak{f}(\chi_0).$$
Let $f' \in \mathbb{Z}_{\ge 0}$ be the minimal positive integer such that 
$$\frak{f}'|f'.$$
Let $f_0'$ denote the prime-to-$p$ part of $f'$. Given a quaternion algebra $D/\mathbb{Q}$, write
$$f_0' = f_+'f_-'$$
where $f_+'$ is supported on primes prime to $\mathrm{disc}(D)$ and $f_-'$ is supported on primes dividing $\mathrm{disc}(D)$. 
\end{definition}

\begin{assumption}\label{f'assumption}For the rest of Section \ref{padicLfunctionsection2}, assume that 
$$f_0'p^a = f',$$
i.e. $a$ as in (\ref{adefinition}) is the exact power of $p$ dividing $f'$. 
\end{assumption}

Now assume $\chi_0$ is as in Choice \ref{chi0choice}. In particular, by (\ref{ebddefinition}) and Choice \ref{chi0choice} we have 
$$f_0\frak{p}^e = \frak{f}(\lambda) = \frak{f}|\frak{f}'.$$
In particular, from (\ref{adefinition}) we see that
$$f_0p^a|f', \hspace{1cm} f_0'p^a = f'.$$
Thus Assumption \ref{f'assumption} is satisfied in this case. As we will see in Proposition \ref{Dsplitproposition} below, we have that 
$$f_0|f_+'$$
when $g = \theta_{\psi}$ as in Choice \ref{twistchoice}.

For the remainder of Section \ref{padicLfunctionsection2} we make the following assumption on $N, f'$ and $D$. 

\begin{assumption}\label{Nf0assumption2}In the setting of Convention \ref{Yconvention}, we let $N \in \mathbb{Z}_{\ge 4}$,  be such that $f_+'|N$ and $(pf_-',N) = 1$ (in the notation of Definition \ref{f'notation}). In particular, the quaternion algebra $D/\mathbb{Q}$ is split at all places dividing $pN\infty$. 
\end{assumption}

\subsection{Yuan-Zhang-Zhang Heegner points}Recall $\mathbb{Y}_{\infty}$ from (\ref{mathbbYinfty}) and fix a point as in \cite[Section 1.3 of Introduction]{YZZ}
$$P \in \mathbb{Y}_{\infty}^{K^{\times}},$$
i.e. a CM point for $K$, which lies above the point $y \in \mathbb{Y}_{\infty}(\Gamma)(\overline{\mathbb{Q}})$ from Choice \ref{CMpointchoice} (4) under the natural projection $\mathbb{Y}_{\infty} \rightarrow \mathbb{Y}_{\infty}(\Gamma)$ (where $\mathbb{Y}_{\infty}(\Gamma)$ is as in (\ref{algebraicYinftyGamma})). Viewing 
$$\pi_g \in \mathrm{Hom}(\mathbb{Y}_{\infty},A_g) \otimes_{\mathbb{Z}} \mathbb{Q},$$
by (\ref{globalsign}) and (\ref{sign2}), from loc. cit. we get a Heegner point
$$P_{g,\chi_0} := \int_{\mathrm{Gal}(K^{\mathrm{ab}}/K)}\pi_g(P^{\sigma})\otimes \chi_0(\sigma) d\sigma \in (A_g \otimes \chi_0)^{\mathrm{Gal}(\overline{K}/K)} \otimes_{\mathbb{Z}} \mathbb{Q}.$$
As these points are more general than the usual notion of Heegner point found in the literature, we call such a point a \emph{Yuan-Zhang-Zhang Heegner point}. The nontriviality of this point will ultimately depend on property (3) of Choice \ref{twistchoice}. As explained above, the existence of a $\chi_0$ with this property ultimately follows from the \emph{ineffective} results of Rohrlich \cite{Rohrlich}, and hence the point $P_{g,\chi_0}$ is itself ineffective. 

\subsection{Anticyclotomic Rankin-Selberg $p$-adic $L$-functions}\label{GL2setupsection}Let $E/\mathbb{Q}$ be an elliptic curve with CM by $\mathcal{O}_K$ as in Assumption \ref{pconductorassumption} (3) so that letting $\lambda_E$ be its associated infinity type $(1,0)$ Hecke character, we have $\frak{f}(\lambda_E)^{(p)}|N$ (or equivalently $f_0|N$) by Assumption \ref{Nf0assumption}. %Let $g$ be any $GL_2$-eigenform  (not necessarily a CM form) of level $\Gamma_1(Mp^{\alpha})$ where $(M,p) = 1$, which is a $U = [p]U_p$-eigenvector (see Definition \ref{flatdefinition} for the definitions of $U$ and $U_p$), in particular meaning $p$ divides the level of $g$, i.e. $\alpha \ge 1$. 
Suppose that $(g,\chi_0)$ is a Rankin-Selberg pair and $D/\mathbb{Q}$ is a quaternion algebra split at all places dividing $pN\infty$, and recall that we are working under Assumption \ref{Nf0assumption2} for the remainder of Section \ref{padicLfunctionsection2}. %Let $D$ be a quaternion algebra over $\mathbb{Q}$ satisfying (\ref{sign2}) with discriminant prime to $Np$, as in Section \ref{CMpointsection}. In the case where $g = \theta_{\psi}$ ($\psi$ as in (\ref{psichoice})), Proposition \ref{Dsplitproposition} already ensures that such a $D$ satisfies $(\mathrm{disc}(D),Np) = 1$. 

Our main application of the constructions of this section (see the end of Section \ref{factorizationsection}) will concern the CM form case $g = \theta_{\psi}$ as in Choice \ref{twistchoice} with respect to the type $(1,0)$ character $\lambda = \lambda_E$ attached to a given elliptic curve $E/\mathbb{Q}$ with CM by $\mathcal{O}_K$ as in Example \ref{Eexample}. 

The next Theorem asserts the existence of a Rankin-Selberg $p$-adic $L$-function 
$$\mathcal{L}_{g \times \chi_0} \in \mathcal{O}_{\mathbb{C}_p} + (q_{\mathrm{dR}}-1)\mathbb{C}_p\llbracket q_{\mathrm{dR}}-1\rrbracket$$
whose $p$-adic Maass-Shimura derivatives 
$$D_0^j\mathcal{L}_{g \times \chi_0}\in \mathcal{O}_{\mathbb{C}_p} + (q_{\mathrm{dR}}-1)\mathbb{C}_p\llbracket q_{\mathrm{dR}}-1\rrbracket$$
have constant terms (i.e. values at $q_{\mathrm{dR}} = 1$) which interpolate square roots of positive Hodge-Tate weight-anticyclotomic twists of the central $L$-value of $(g,\chi_0)$. 

\begin{theorem}\label{GL2maintheorem}Let $(g,\chi_0)$ be as in Section \ref{RSsection} where $g$ satisfies Assumption \ref{gassumption} below. Consider the continuous function
$$\mathbb{Z}/(p-1) \times \mathbb{Z}_p \ni j \mapsto D_0^j\mathcal{L}_{g \times \chi_0}|_{q_{\mathrm{dR}} = 1} \in \mathcal{O}_{\mathbb{C}_p}$$
from (\ref{Gcontinuous}) below. Then we have the following interpolation formula. There exist $C_{g,\chi_0}, c_{g,\chi_0}, c_{g,\chi_0}' \in \overline{\mathbb{Q}}_p^{\times}$ which depend only on $(g,\chi_0)$ and the choice of $\varpi \in \mathcal{O}_K$ made in (\ref{pichoice}) such that for any $j \in \mathbb{Z}_{\ge 1} \subset \mathbb{Z}/(p-1) \times \mathbb{Z}_p$ (where $\mathbb{Z}_{\ge 1}$ is embedded diagonally), we have the following equality of elements of $\overline{\mathbb{Q}}$:
\begin{enumerate}
\item \begin{equation}\label{interpRankinSelberg}\Omega_p(a)^{4j} \cdot(D_0^j\mathcal{L}_{g\times \chi_0})^2|_{q_{\mathrm{dR}} = 1}= \Omega_{\infty}(a)^{4j}p^{2aj}\varpi^{4j} \cdot C_{g,\chi_0} \cdot L^{\mathrm{alg}}(g,\chi_0(\lambda_E^c/\lambda_E)^j,1).
\end{equation}
\item Letting $\mathcal{L}_{g \times \chi_0} := D_0^0\mathcal{L}_{g \times \chi_0}$, we have
 \begin{equation}\label{padicWaldspurger}\mathcal{L}_{g \times \chi_0}|_{q_{\mathrm{dR}} = 1} = c_{g,\chi_0} \cdot \log_{w_g^{\flat}}(P_{g,\chi_0})= c_{g,\chi_0}'\cdot \log_{w_g}(P_{g,\chi_0}) . 
\end{equation}
\end{enumerate}
Here, $\Omega_p(a) \in \mathbb{C}_p^{\times}$, $\Omega_{\infty}(a) \in \mathbb{C}^{\times}$ are as in Definition \ref{finalperiodsdefinition} (with $n = a$ and $a \in \mathbb{Z}_{\ge 0}$ as in (\ref{adefinition})). Moreover,
$$L(g,\chi,s) = L(\mathrm{JL}(\pi_g) \times \pi_{\chi},s-\frac{1}{2})$$
is the Rankin-Selberg $L$-function centered at 1, where 
$$L^{\mathrm{alg}}(g,\chi_0(\lambda_E^c/\lambda_E)^j,1) = \frac{2^{2j}\pi^{2j-1}\Gamma(j+1)\Gamma(j)}{\sqrt{|D_K|}^{2j}} L(g,\chi_0(\lambda_E^c/\lambda_E)^j,1)$$
denotes the algebraic normalization of the central $L$-value as in \cite[Paragraph after Definition 3.5]{CastellaHsieh} (note that in loc. cit., the central value is at $s = 1/2$ instead of at $s = 1$, corresponding to the unitary normalization $\mathrm{JL}(\pi_g) \times \pi_{\chi_0(\lambda_E^c/\lambda_E)^j}$, recalling theta $\pi_{\chi}$ denotes the automorphic representation of $GL_2(\mathbb{A}_{\mathbb{Q}})$ attached to the theta series $\theta_{\chi}$ of $\chi$).%, and ``$L^{\mathrm{alg},(\frak{p})}$'' denotes $L^{\mathrm{alg}}$ with the Euler factor at $\frak{p}$ removed.
\end{theorem}

\begin{remark}Suppose $g = \theta_{\psi}$ where $\psi$ is as in Choice \ref{twistchoice}. Given Assumption \ref{pconductorassumption} we have that $\lambda$ and $\lambda^c\chi_0/\chi_0^c$ are both ramified at $\frak{p}$. Thus (\ref{Artinfactorization}) implies that the Euler factor of $L^{\mathrm{alg}}(g,\chi_0(\lambda_E^c/\lambda_E)^j,1)$ at $\frak{p}$ is 1. Thus there is no Euler factor at $\frak{p}$ to remove as in \cite[Theorem 5.13]{BDP}.
\end{remark}

\begin{remark}We explain the appearance of the factor $p^{2aj}\varpi^{4j}$ on the right-hand side of (\ref{interpRankinSelberg}). From the discussion of Remark \ref{interpolationcomparisonremark} and (\ref{OmegaOmega'}), one can rewrite the right-hand side of (\ref{interpRankinSelberg}) as 
$$\Omega_p(a)^{4j} \cdot(D_0^j\mathcal{L}_{g\times \chi_0})^2|_{q_{\mathrm{dR}} = 1}= \Omega_{\infty}^{-4j} \cdot \mathrm{vol}(\mathcal{O}_{p^a})^{-2j} \pi^{2j-1}\Gamma(j+1)\Gamma(j)C_{g,\chi_0}p^{-2a}\varpi^{-2} \cdot L(g,\chi_0(\lambda_E^c/\lambda_E)^j,1)$$
with $\Omega_{\infty}$ as in (\ref{period1}). (Cf. with (\ref{interp1twistalternate}).) This is in direct analogy with \cite[(5.2.3-4)]{BDP} (see also Theorem 4.6 and 5.5 of op. cit.). (Note that our $p$-adic period $\Omega_p(a)$ in (\ref{interpRankinSelberg}) is the analogue of the $p$-adic period $\Omega_p^{-1}$ of op. cit., where $\Omega_p$ is defined in (5.2.2) of op. cit.)
\end{remark}

We will prove Theorem \ref{GL2maintheorem} in Section \ref{GL2maintheoremproofsection}.

%\begin{remark}\label{GL2maintheoremremark}We will show in Lemma \ref{continuitylemma} that the assignment
%$$\mathbb{Z}_{\ge 1} \ni j \mapsto D_d^j\mathcal{L}_{g \times \chi_0}|_{q_{\mathrm{dR}} = 1} \in \mathcal{O}_{\mathbb{C}_p}$$ extends to a continuous function $\mathbb{Z}/(p-1) \times \mathbb{Z}_p \rightarrow \mathcal{O}_{\mathbb{C}_p}$ (viewing $\mathbb{Z}_{\ge 1} \subset \left(\mathbb{Z}/(p-1) \times \mathbb{Z}_p\right)$ as embedded diagonally). This continuous extension is our $p$-adic $L$-function attached to the Rankin-Selberg pair $(g,\chi_0)$, which we continue to denote by $D_0^0\mathcal{L}_{g \times \chi_0}|_{q_{\mathrm{dR}}=1}$. 
%\end{remark}

\subsection{Form of the $p$-adic $L$-function}Following (\ref{sum1}), we will construct a ``fundamental power series''
$$\mathcal{L}_{g \times \chi_0}\in \mathcal{O}_{\mathbb{C}_p} + (q_{\mathrm{dR}}-1)\mathbb{C}_p\llbracket q_{\mathrm{dR}}-1 \rrbracket$$
as a sum 
\begin{equation}\label{sum1'}\mathcal{L}_{g \times \chi_0} = \sum_{\frak{b}\in S'}\chi_0(\frak{b})\cdot \mu_{g \times \chi_0, \frak{b}}
\end{equation}
for some power series 
$$\mu_{g \times \chi_0, \frak{b}} \in \mathcal{O}_{\mathbb{C}_p} + (q_{\mathrm{dR}}-1) \mathbb{C}_p\llbracket q_{\mathrm{dR}}-1\rrbracket$$
which are the $q_{\mathrm{dR}}$-expansions of $G^{\flat}$ from (\ref{Gflatdefinition}) specialized to the point $y_{\frak{b},a}$ from (\ref{YIgin2}).

%$$\mu = \sum_{\frak{a}}\mu_{\frak{a}}$$
%for some measures $\mu \in \mathcal{O}_{\mathbb{C}_p}\llbracket G_-\rrbracket [1/p]$. Recall our splitting $G_-  \cong \Delta_- \times \Gamma_-$, and let $a \in \Delta_-$ be elements such that the $\frak{a}a$ represent the quotient $\mathrm{Gal}(K(\frak{f})[p^{\infty}]/K)/\Gamma_- = \mathcal{C}\ell(\frak{f})\Delta_-$. 

%Then for each $\frak{a}$, we will define $\mu_{\frak{a}}$ as
%$$\mu_{\frak{a}} = \sum_a \chi_0^{-1}(a)\mu_{\frak{a}a}$$
%where $\mu_{\frak{a}a} \in \mathcal{O}_{\mathbb{C}_p}\llbracket \Gamma_-\rrbracket [1/p]$. Hence, in all we can write
%$$\mu = \sum_{\frak{a}}\sum_a (\psi\chi_0)^{-1}(\frak{a}a)\mu_{\frak{a}a} = \sum_{\frak{b}}\lambda^{-1}(\frak{b})\mu_{\frak{b}}$$
%for $\mu_{\frak{b}} \in \mathcal{O}_{\mathbb{C}_p}\llbracket \Gamma_-\rrbracket [1/p]$, where $\frak{b}$ runs over integral, prime-to-$p$ representatives of $\mathcal{C}\ell(\frak{f})\Delta_-$, and for $\frak{b} = \frak{a}a$, we have $\mu_{\frak{b}} = \mu_{\frak{a}a}$. 

\subsection{Construction of the $p$-adic $L$-function}%In this section, we apply the previous results for $Y$ with associated quaternion algebra $D$ split at $p$ and $\infty$ but not necessarily totally split (i.e. equal to $M_2(\mathbb{Q})$). 

\begin{assumption}\label{gassumption}
\begin{enumerate}
\item Let $g$ be any $GL_2$-eigenform  (not necessarily a CM form) of level $\Gamma_1(Mp^{\alpha})$ where $(M,p) = 1$, which is a $U = [p]U_p$-eigenvector (see Definition \ref{flatdefinition} for the definitions of $U$ and $U_p$), in particular meaning $p$ divides the level of $g$, i.e. $\alpha \ge 1$. Let $\alpha_p \in \overline{\mathbb{Z}}$ be the $U$-eigenvalue. 
\item Assume that 
$$\alpha \le a$$
where $a$ is as in (\ref{adefinition}).
\item Assume that $g$ is normalized (has $q$-expansion $\sum_{n = 1}^{\infty}a_nq^n \in \overline{\mathbb{Z}}\llbracket q\rrbracket, a_1 = 1$).  
\end{enumerate}
\end{assumption}

If $g = \theta_{\psi}$ is as in Choice \ref{twistchoice}, then $\psi$ is in particular unramified at $\frak{p}$ by (\ref{localvcondition}) and so \cite[Proposition 3.13]{BDP2} implies 
$$\alpha = \beta \overset{(\ref{betadelta})}{\le} a,$$
and thus Assumption \ref{gassumption} is satisfied. 

\begin{definition}\label{wgdefinitions}
\begin{enumerate}
\item Let $\pi : \mathcal{E}^+ \rightarrow Y^+$ denote the universal object (universal false elliptic curve with $\mathcal{O}_D$-endomorphism structure, $\Gamma = \Gamma(N)$-level structure and $\Gamma(p^{\infty})$-level structure) and recall $\omega_+ = \pi_{+,*}\Omega_{\mathcal{E}^+/Y^+}$ from (\ref{omegaY}). Let
$$w_g \in \omega_+^{\otimes 2}(Y_{\infty}^+)$$
be the 1-form associated with the normalized Jacquet-Langlands correspondence of $g$ (where here normalized means $w_g \in \omega_+^{\otimes 2}(Y_{\infty}^+)$ and $p^{-r}w_g \not\in \omega_+^{\otimes 2}(Y_{\infty}^+)$ for any $r \in \mathbb{Q}_{> 0}$ and any element $p^{-r} \in \mathbb{C}_p$ with $\mathrm{ord}_p(p^{-r}) = -r$).  %By \cite[Proposition 3.13]{BDP2}, 
By Assumption \ref{gassumption}, it is the pullback of a modular form of weight 2 and level $\Gamma_1(Mp^{\alpha})$ for some integer $M$ prime to $p$ and $\alpha \in \mathbb{Z}_{\ge 0} \cap \mathbb{Z}_{\le a}$, where $a$ is as in (\ref{adefinition}).  %, and is a $U = [p]U_p$-eigenvector with eigenvalue $\alpha_p = \psi(\frak{p})$ (see Definition \ref{flatdefinition} for the definitions of $U$ and $U_p$, and for the computation of $\alpha_p$ see \cite[Proposition 3.13]{BDP2}). 
\item Let 
$$w_g^{\flat} = (1-V_p^*U_p^*)w_g \in \omega_+^{\otimes 2}(Y_{\infty}^+)$$
as in Definition \ref{flatdefinition}; it is thus the pullback of a modular form of weight 2 and level $\Gamma_1(Mp^{\alpha+2})$. If moreover $g$ is a $U = [p]U_p$-eigenform of eigenvalue $\alpha_p$, then 
$$V_p^*U_p^*g = V_p^*[p^{-1}]^*(\alpha_p g) = \alpha_p[p^{-1}]^*V_p^*g$$
and so $w_g^{\flat}$ is the pullback of a modular form of weight 2 and level $\Gamma_1(Mp^{\alpha+1})$. %Since $a \ge \beta$ (see (\ref{betadelta})), then we can view $w_g$ as a form of level $\Gamma_1(Mp^a)$ and $w_g^{\flat}$ as a form of level $\Gamma_1(Mp^{a+1})$. %By Theorem \ref{STanalyticcontinuationtheorem} and the $p$-integrality of the local Serre-Tate expansions of $w_g^{\flat}$ on $Y^{\mathrm{Ig}}$ (see \cite[Lemma 8.2]{Brakocevic} and \cite[Section 4]{Brooks}), we have 
We thus have
$$w_g^{\flat}|_{\mathcal{Y}^{\mathrm{Ig}}} \in \omega_+^{\otimes 2}|_{\mathcal{Y}^{\mathrm{Ig}}}(\mathcal{Y}^{\mathrm{Ig}}).$$
\end{enumerate}
\end{definition}

%Observe that the order $p^{a+1}$-canonical subgroup provides a rigid analytic map $Y(\epsilon_0/p^a) \rightarrow Y(\Gamma_1(Np^{a+1}))$, and thus pulling back $w_g^{\flat} \in \omega^{\otimes 2}(Y(\Gamma_1(Mp^{a+1})))$ by the map $$Y(\epsilon_0/p^a) \rightarrow Y(\Gamma_1(Np^{a+1})) \rightarrow Y(\Gamma_1(Mp^{a+1}))$$
%we can view $w_g^{\flat} \in \omega^{\otimes 2}(Y(\epsilon_0/p^a))$. By Theorem \ref{pintegraltheorem2} (1) with $\alpha = a$ and $k = 2$ and Remark \ref{assumptionsatisfiedremark2} applied to $w_g^{\flat} \in\omega^{\otimes 2}(\mathcal{V}_x)$, we have 
%$$\theta_t(\theta_q(w_g^{\flat})) \in \hat{\mathcal{O}}_{\mathcal{Y}^{\mathrm{Ig}}(\epsilon_0)}^+(\mathcal{Y}^{\mathrm{Ig}}(\epsilon_0/p^{\alpha}))$$
%for any integer $\alpha \ge \beta$ (in particular, we will consider $\alpha = a \ge \beta$ for $a$ as in (\ref{adefinition})). 

Recall Definition \ref{'zqexpansions}, Theorem \ref{U'Utheorem} and Definition \ref{thetatpowerseriesdefinition}. 

\begin{definition}\begin{enumerate}
\item Let
$$G_2^{\flat}(q_{\mathrm{dR}}) := \theta_t(w_g^{\flat}(q_{\mathrm{dR}})) \in \hat{\mathcal{O}}_{\mathcal{V}_x}(\mathcal{V}_x)\llbracket q_{\mathrm{dR}}-1\rrbracket.$$
Here, recall by Definition \ref{'zqexpansions} that $w_g^{\flat}(q_{\mathrm{dR}})$ is the $q_{\mathrm{dR}}$-expansion of the generalized $p$-adic modular form of weight 2 
\begin{equation}\label{Ggeneralizedpadicmodularform}\tilde{G}_2^{\flat} := \frac{w_g^{\flat}}{w_{\mathrm{can}}^{\otimes 2}} \in \mathbb{B}_{\mathrm{dR},\mathcal{V}_x}(\mathcal{V}_x)\llbracket q_{\mathrm{dR}}-1\rrbracket \overset{(\ref{Bdecomposition''})}{=} \hat{\mathcal{O}}_{\mathcal{V}_x}(\mathcal{V}_x)(\!(t)\!)\llbracket q_{\mathrm{dR}}-1\rrbracket
\end{equation}
attached to $w_g^{\flat}$, and $\theta_t(w_g^{\flat}(q_{\mathrm{dR}}))$ is written in the notation of Definition \ref{thetatpowerseriesdefinition}. 
\item Thus we may also write, in the notation of Definition \ref{thetatpowerseriesdefinition},
\begin{equation}\label{alsohave}G_2^{\flat}(q_{\mathrm{dR}}) = \theta_t(\tilde{G}_2^{\flat}(q_{\mathrm{dR}})) \in \hat{\mathcal{O}}_{\mathcal{V}_x}(\mathcal{V}_x)\llbracket q_{\mathrm{dR}}-1\rrbracket.
\end{equation}
\item Let $\mathbb{Y}_{\infty}(\Gamma)$ be as in (\ref{algebraicYinftyGamma}) with $\Gamma = \Gamma(N)$ as in Convention \ref{Yconvention}. Since $w_g \in \omega^{\otimes 2}(\mathbb{Y}_{\infty}(\Gamma))$, in fact $w_g^{\flat} \in \omega^{\otimes 2}(\mathbb{Y}_{\infty}(\Gamma))$.
\item Define
$$F_g := \frac{w_g}{2\pi id\tau} \in \mathcal{O}_{\mathcal{H}^+}(\mathcal{H}^+), \hspace{1cm} F_g^{\flat} = \frac{w_g^{\flat}}{2\pi id\tau} \in \mathcal{O}_{\mathcal{H}^+}(\mathcal{H}^+),$$
where $\tau$ is the standard coordinate on $\mathcal{H}^+$.
%For a sufficiently small open compact subgroup $U \subset D^{\times}(\mathbb{A}_{\mathbb{Q}}^{(\infty)})$, we can view $g$ as a weight 2 eigenform on a Shimura curve $Y = Y_U$ associated with $U$, and let $Y(r)$ be the locus cut out by $|\mathrm{Ha}| \ge p^{-r}$ with $0 \le r < 1/(p+1)$. Then the canonical subgroup of order $p$ exists on $Y(r)$, and since $g$ is a $U_p$-eigenvector (say with eigenvalue $\alpha_p$), we can define 
%$$g^{\flat} = g|_{1-U_pV_p}  = g - \alpha_p g|_{V_p}.$$
%This is also a well-defined form on $Y_{\infty}(r)$.
\end{enumerate}
\end{definition}

%Recall $w_{\mathrm{can}}$ from (\ref{zqw2}). Then
%$$\frac{w_g^{\flat}}{w_{\mathrm{can}}^{\otimes 2}} \in \mathbb{B}_{\mathrm{dR},U}^+(U)\llbracket X/t\rrbracket.$$
Henceforth, assume $g$ is a $U = [p]U_p$-eigenvector (see Definition \ref{flatdefinition} for the definitions of $U$ and $U_p$) of eigenvalue $\alpha_p$. Consider the $p$-adic logarithm $\log_{w_g}$, which by \cite{Coleman} is a locally analytic primitive of $w_g$. As in \cite[Section 3.8]{BDP}, one can verify that $(1-V_p^*U_p^*)g = (1 - \alpha_p([p^{-1}]V_p)^*)$ is the Coleman polynomial of $\log_w$ on $\mathcal{Y}^{\mathrm{Ig}}(\epsilon_0)$ (note that the order-$p$ canonical subgroup exists on $\mathcal{Y}^{\mathrm{Ig}}(\epsilon_0)$ since $n(\epsilon_0) = 1$, see (\ref{rchoice})), and hence
\begin{equation}\label{logintegral}\log_{w_g^{\flat}} \in \mathcal{O}_{\mathcal{Y}^{\mathrm{Ig}}(\epsilon_0)}^+(\mathcal{Y}^{\mathrm{Ig}}(\epsilon_0))
\end{equation}
%Moreover, it is the primitive of $w_g^{\flat}$. By Theorem \ref{STanalyticcontinuationtheorem} and the $p$-integrality of the local Serre-Tate expansions of $\log_{w_g^{\flat}}$ on $Y^{\mathrm{Ig}}$ (see \cite[Section 7]{Brooks}), we have 
%$$\log_{w_g^{\flat}}|_{\mathcal{Y}^{\mathrm{Ig}}} \in \mathcal{O}_{\mathcal{Y}^{\mathrm{Ig}}}^+(\mathcal{Y}^{\mathrm{Ig}}).$$
%Thus applying Theorem \ref{pintegraltheorem2} (1) to $\log_{w_g^{\flat}} \in \mathcal{O}_{\mathcal{V}_x}(\mathcal{V}_x)$, we have
is a rigid analytic primitive of $w_g^{\flat}$. Thus, since $\partial$ from (\ref{partialkjdefinition}) is just the exterior derivative $\mathcal{O}_{\mathcal{V}_x} \rightarrow \Omega_{\mathcal{V}_x}$, we have for any $j \in \mathbb{Z}_{\ge 1}$
\begin{equation}\label{primitiveequation}\partial_0\log_{w_g^{\flat}} = w_g^{\flat}, \hspace{1cm} d_0^j\log_{w_g^{\flat}} = d_2^{j-1}\tilde{G}_2.
\end{equation}

Moreover,
\begin{equation}\label{logwgintegral}\theta_t(\theta_q(\log_{w_g^{\flat}}))  \overset{(\ref{thetaq}), (\ref{thetat})}{=} \log_{w_g^{\flat}} \overset{(\ref{logintegral})}{\in} \hat{\mathcal{O}}_{\mathcal{Y}^{\mathrm{Ig}}(\epsilon_0)}^+(\mathcal{Y}^{\mathrm{Ig}}(\epsilon_0)).
\end{equation}

\begin{definition}
\begin{enumerate}
\item Let
\begin{equation}\label{Gflatdefinition}G^{\flat}(q_{\mathrm{dR}}) := \theta_t\left(\log_{w_g^{\flat}}(q_{\mathrm{dR}})\right) \in \hat{\mathcal{O}}_{\mathcal{Y}^{\mathrm{Ig}}(\epsilon_0)}(\mathcal{Y}^{\mathrm{Ig}}(\epsilon_0))\llbracket q_{\mathrm{dR}}-1\rrbracket.
\end{equation}
\item For any $j \in \mathbb{Z}_{\ge 0}$, define
\begin{equation}\label{derivativeGflatdefinition}D_0^jG^{\flat}(q_{\mathrm{dR}}) := \theta_t\left(d_0^j\left(\log_{w_g^{\flat}}(q_{\mathrm{dR}})\right)\right)
\end{equation}
so that $D_0^0G^{\flat}(q_{\mathrm{dR}}) = G^{\flat}(q_{\mathrm{dR}})$. We will usually restrict to $\mathcal{Y}^{\mathrm{Ig}}(\epsilon_0/p^{a}) \subset \mathcal{Y}^{\mathrm{Ig}}(\epsilon_0)$ unless otherwise noted for notational convenience and view 
$$D_0G^{\flat}(q_{\mathrm{dR}}) \in \hat{\mathcal{O}}_{\mathcal{Y}^{\mathrm{Ig}}(\epsilon_0/p^{a})}(\mathcal{Y}^{\mathrm{Ig}}(\epsilon_0/p^{a}))\llbracket q_{\mathrm{dR}}-1\rrbracket.$$
\end{enumerate}
\end{definition}

%Let $W_{\xi}$ be as in (\ref{W'definition}). 
Recall $w_g^{\flat} \in \omega_+^{\otimes 2}(Y_{\infty}^+)$ (see Definition \ref{wgdefinitions} (2)) is a form of level $\Gamma_1(Mp^{\alpha+1})$ where $\alpha \in \mathbb{Z}$ with $0 \le \alpha \le a$ ($a$ as in (\ref{adefinition})). Recall the generalized $p$-adic modular form of weight 2 $\tilde{G}_2^{\flat}$ (see (\ref{Ggeneralizedpadicmodularform})) attached to $w_g^{\flat}$. Recall the open subset 
$$Y_1(a+1,\epsilon_0/p^{a}) \subset Y(\Gamma(N) \cap \Gamma_1(p^{a+1}))$$
from Convention \ref{Y1convention}. Observe that by pulling back $w_g^{\flat} \in \omega^{\otimes 2}(Y(\Gamma_1(Mp^{a+1})))$ along the map 
$$Y_1(a+1,\epsilon_0/p^{a}) \subset Y(\Gamma(N) \cap \Gamma_1(p^{a+1})) \rightarrow Y(\Gamma_1(Mp^{a+1}))$$
we can view 
$$w_g^{\flat} \in \omega^{\otimes 2}(Y_1(a+1,\epsilon_0/p^{a})).$$
Then by Remark \ref{assumptionsatisfiedremark2}, $\tilde{G}_2^{\flat}$ satisfies the assumptions of Theorem \ref{pintegraltheorem2} (1) with $k = 2$. Thus by Theorem \ref{pintegraltheorem2} (1) we have for all $j \in \mathbb{Z}_{\ge 1}$
\begin{equation}\label{logtildeGintegral}\begin{split}\theta_t(d_0^j\log_{w_g^{\flat}}(q_{\mathrm{dR}}))|_{q_{\mathrm{dR}} = 1} \overset{(\ref{thetaq})}{=} \theta_q(\theta_t(d_0^j\log_{w_g^{\flat}}(q_{\mathrm{dR}}))) &\overset{(\ref{Ggeneralizedpadicmodularform}), (\ref{primitiveequation})}{=} \theta_q(\theta_t(d_2^{j-1}\tilde{G}_2^{\flat}(q_{\mathrm{dR}}))) \\
&\hspace{.3cm}\overset{(\ref{integralthetaG})}{\in} \hat{\mathcal{O}}_{\mathcal{Y}^{\mathrm{Ig}}(\epsilon_0/p^{a})}^+(\mathcal{Y}^{\mathrm{Ig}}(\epsilon_0/p^{a})).
\end{split}
\end{equation}
%Since 
%$$d_2^{j-1}G_2^{\flat} = d_0^j\log_{w_g^{\flat}},$$
%We thus have for all $j \in \mathbb{Z}_{\ge 1}$
%$$\theta_q\left(\theta_t\left(d_0^j\left(\log_{w_g^{\flat}}(q_{\mathrm{dR}})\right)\right)\right) \overset{ (\ref{Ggeneralizedpadicmodularform}), (\ref{primitiveequation})}{=} \theta_q\left(\theta_t\left(d_2^{j-1}\left(\tilde{G}_2^{\flat}(q_{\mathrm{dR}})\right)\right)\right) \in \hat{\mathcal{O}}_{\mathcal{Y}^{\mathrm{Ig}}(\epsilon_0/p^{a})}^+(\mathcal{Y}^{\mathrm{Ig}}(\epsilon_0/p^{a})).$$
%Thus for any $j \in \mathbb{Z}_{\ge 1}$,
%\begin{equation}\label{beforeintegral}D_0^jG^{\flat}(q_{\mathrm{dR}}) \overset{(\ref{derivativeGflatdefinition})}{=} \theta_t\left(d_0^j\left(\log_{w_g^{\flat}}(q_{\mathrm{dR}})\right)\right) \in \hat{\mathcal{O}}_{\mathcal{Y}^{\mathrm{Ig}}(\epsilon_0/p^{a})}^+(\mathcal{Y}^{\mathrm{Ig}}(\epsilon_0/p^{a})) + (q_{\mathrm{dR}}-1) \cdot \hat{\mathcal{O}}_{\mathcal{Y}^{\mathrm{Ig}}(\epsilon_0/p^{a})}(\mathcal{Y}^{\mathrm{Ig}}(\epsilon_0/p^{a}))\llbracket q_{\mathrm{dR}}-1\rrbracket.
%\end{equation}

Thus for any $j \in \mathbb{Z}_{\ge 1}$,
\begin{equation}\label{derivativeGflatintegral}\begin{split}D_0^jG^{\flat}(q_{\mathrm{dR}}) &\overset{(\ref{derivativeGflatdefinition})}{=} \theta_t\left(d_0^j\left(\log_{w_g^{\flat}}(q_{\mathrm{dR}})\right)\right) \\
&\overset{(\ref{logtildeGintegral})}{\in} \hat{\mathcal{O}}_{\mathcal{Y}^{\mathrm{Ig}}(\epsilon_0/p^{a})}^+(\mathcal{Y}^{\mathrm{Ig}}(\epsilon_0/p^{a})) + (q_{\mathrm{dR}}-1)\cdot \hat{\mathcal{O}}_{\mathcal{Y}^{\mathrm{Ig}}(\epsilon_0/p^{a})}(\mathcal{Y}^{\mathrm{Ig}}(\epsilon_0/p^{a})) \llbracket q_{\mathrm{dR}}-1\rrbracket.
\end{split}
\end{equation}
By Theorem \ref{pintegraltheorem2} (3), the function 
$$\mathbb{Z}_{\ge 1} \ni j \mapsto \theta_q(D_0^jG^{\flat}(q_{\mathrm{dR}})) \overset{(\ref{logtildeGintegral})}{=} \theta_q(\theta_t(d_0^j\left(\log_{w_g^{\flat}}(q_{\mathrm{dR}})\right))) \in \hat{\mathcal{O}}_{\mathcal{Y}^{\mathrm{Ig}}(\epsilon_0/p^{a})}^+(\mathcal{Y}^{\mathrm{Ig}}(\epsilon_0/p^{a}))$$
extends to a continuous function
\begin{equation}\label{intermediatecontinuous3}\mathbb{Z}_{\ge 1} \ni j \mapsto D_0^jG^{\flat}(q_{\mathrm{dR}})|_{q_{\mathrm{dR}} = 1} = \theta_q(D_0^jG^{\flat}(q_{\mathrm{dR}}))  \in \hat{\mathcal{O}}_{\mathcal{Y}^{\mathrm{Ig}}(\epsilon_0/p^{a})}^+(\mathcal{Y}^{\mathrm{Ig}}(\epsilon_0/p^{a})).
\end{equation}

%We have from (\ref{primitiveequation})
%\begin{equation}\label{logarithmicderivative}D_0G^{\flat} = G_2^{\flat}.
%\end{equation}
%The follows from the fact that $\log_{w_g^{\flat}}$ is a rigid analytic primitive of $w_g^{\flat}$ and (\ref{'KS}). 

%For simplicity, we slightly abuse notation and let $Y_{\infty} = Y_{U,\infty}$ denote the $\Gamma(p^{\infty})$-level Shimura curve over $Y_U$ (recall $D$ is split at $p$). Let also let $\mathcal{Y}^{\mathrm{Ig}}, \mathcal{Y}^{\mathrm{Ig}}(r)$ be the corresponding Igusa towers as before. Then we take the $q_{\mathrm{dR}}$-expansion to get
%$$G^{\flat}(q_{\mathrm{dR}}-1) \in \Gamma(\hat{\mathcal{O}}_{\mathcal{Y}^{\mathrm{Ig}}(r)}^+)\llbracket q_{\mathrm{dR}}-1\rrbracket [1/p] = \Gamma(\hat{\mathcal{O}}_{\mathcal{Y}^{\mathrm{Ig}}(r)}^+)\llbracket \mathbb{Z}_p\rrbracket [1/p].$$
%As before, we let $G^{\flat}(y)(q_{\mathrm{dR}}-1)$ denote the image under the specialization of coefficients $\Gamma(\hat{\mathcal{O}}_{\mathcal{Y}^{\mathrm{Ig}}(r)}^+) \rightarrow \Gamma(\hat{\mathcal{O}}_{\mathcal{Y}^{\mathrm{Ig}}(r)}^+)(y) \subset \mathcal{O}_{\mathbb{C}_p}$, and we get
%$$G^{\flat}(y)(q_{\mathrm{dR}}-1) \in \mathcal{O}_{\mathbb{C}_p}\llbracket q_{\mathrm{dR}}-1\rrbracket [1/p] \subset \mathcal{O}_{\mathbb{C}_p}\llbracket \mathbb{Z}_p\rrbracket [1/p].$$

Analogously as in Definition \ref{Sdefinition}, we define the following set. 

\begin{definition}\label{Sdefinition2}%Let $\frak{f}'$ be any ideal with $(\frak{f}',p) = 1$. 
%Recall $\frak{f}$ from Assumption \ref{pconductorassumption}, and that $\frak{f}^{(p)}$ is its prime-to-$p$ part.
Recall $N$ from Assumption \ref{Nf0assumption2} and $f_0'$ from Definition \ref{f'notation}. Fix a set of everywhere integral id\`{e}les
$$S' \subset \mathbb{A}_K^{\times,(pN\infty f_0')} \overset{\rho}{\subset} D^{\times}(\mathbb{A}_{\mathbb{Q}}^{(\infty)})$$
(i.e. $S'$ is a set of finite id\`{e}les which are prime to $pNf_0'$) and such that $S'$ is a full set of representatives of $\mathcal{C}\ell(N)[p^af_0']$ (recall the notation of Definition \ref{mixedclassgroupdefinition}) under the map
$$\mathbb{A}_K^{\times,(\infty)} \twoheadrightarrow K^{\times}\backslash \mathbb{A}_K^{\times,(\infty)}/\left((\mathbb{Z}_p + p^a\mathcal{O}_{K_{\frak{p}}})^{\times}\cdot \prod_{v\nmid \frak{p}\infty}(1+Nf_0'\mathcal{O}_{K_v})\right) \cong \mathcal{C}\ell(N)[p^af_0'].$$
For simplicity, we will further assume that 
$$1 \in S'$$
(which thus represents the trivial class in $\mathcal{C}\ell(N)[p^af_0']$). %When $\frak{f}' = \frak{f}^{(p)}$ (where $\frak{f}$ is as in Assumption \ref{pconductorassumption}), we will let 
%$$S = S(\frak{f}^{(p)}).$$
\end{definition}

We now will need to consider CM points for $\mathcal{O}_{p^af_0'} \subset \mathcal{O}_K$ (see Definition \ref{Ofdefinition}). Recall $\varpi$ from (\ref{pichoice}) and $\tau_0$ from (\ref{tau0}).  

\begin{choice}\label{choice}When $D = M_2(\mathbb{Q})$, let $A'$ be an elliptic curve defined over $\mathcal{O}_{K[f_0']}$ (where $K[m]/K$ denotes the ring class field of conductor $m$) with
$$A'(\mathbb{C}) \cong \mathbb{C}/(\mathbb{Z}\frac{\tau_0}{f_0'} + \mathbb{Z}).$$
When $D \neq M_2(\mathbb{Q})$, let $A'$ be an false elliptic curve defined over $\mathcal{O}_{K[f_0']}$ with 
$$A'(\mathbb{C}) \cong \mathbb{C}^{\oplus 2}/\left(\iota_{\infty}(\mathcal{O}_D)\cdot \left(\begin{array}{ccc} \frac{\tau_0}{f_0'}\\
1\\
\end{array}\right)\right)$$
(this is the false elliptic curve $A_{\tau_0/f_0'}$ of \cite[Section 2.5]{Brooks}). By the theory of complex multiplication (\cite[Chapter II.1.4]{deShalit}), $A'$ has CM by $\mathcal{O}_{f_0'}$. Suppose $D = M_2(\mathbb{Q})$. Then there is some $\Omega' \in \mathbb{C}^{\times}$ such that $\Omega'(\mathbb{Z}\tau_0/f_0' + \mathbb{Z})$ is the period lattice of $A'$. The point $[\tau_0/f_0',1] \in Y(\mathbb{C})$ is induced by the algebraic point 
$$(A,\frac{\Omega' \tau_0}{f_0'N},\frac{\Omega'}{N}) \in \mathbb{Y}(K(N)[f_0']),$$
noting that $(\frac{\Omega'\tau_0}{f_0'N},\frac{\Omega'}{N})$ is a $\Gamma(N)$-level structure under $A'(\mathbb{C}) = \mathbb{C}/(\Omega'(\mathbb{Z}\tau_0/f_0'+\mathbb{Z}))$. 

%When $D = M_2(\mathbb{Q})$, let $A'$ be the elliptic curve defined over $\mathcal{O}_K$ with
%$$A'(\mathbb{C}) \cong \mathbb{C}/(\mathbb{Z}(\tau_0/p^a) + \mathbb{Z}) = \mathbb{C}/(\varpi^{-1} p^{-a}\mathcal{O}_{p^a}).$$
%When $D \neq M_2(\mathbb{Q})$, let $A'$ be the false elliptic curve defined over $\mathcal{O}_K$ with 
%$$A'(\mathbb{C}) \cong \left(\mathbb{C}/(\mathbb{Z}(\tau_0/p^a) + \mathbb{Z})\right)^{\oplus 2} = \left(\mathbb{C}/(\varpi^{-1} p^{-a} \mathcal{O}_{p^a})\right)^{\oplus 2}$$
%otherwise. By the theory of complex multiplication (\cite[Chapter II.1.4]{deShalit}), $A'$ has complex multiplication by $\mathcal{O}_{p^a}$ and is defined up to unique isomorphism over $\mathcal{O}_K$. The point $[\tau_0/p^a,1] \in Y(\mathbb{C})$ is induced by the point 
%$$(A',\frac{1}{N},\frac{\tau_0}{p^aN}) \in Y(K(\frak{f}^{(p)}))$$
%under the moduli interpretation. 
\end{choice}

From the theory of complex multiplication and by examining formal groups, one sees that $A'[\varpi]$ lifts the kernel of $p$-power Frobenius on $A'$ modulo $\varpi\mathcal{O}_{K[f_0']}$ and that the leading coefficient of the isogeny of formal groups $\hat{A}' \rightarrow \hat{A}'/\hat{A}'[\varpi]$ has $p$-adic absolute value $p^{-1/2}$. Thus (cf. \cite[Chapter 3]{Katzpamf})
$$|\mathrm{Ha}(A')| = p^{-1}/p^{-1/2} = p^{-1/2}.$$
Thus, $A'$ has a canonical subgroup of order $p$. Let $[\cdot]_{A'} : \mathcal{O}_{f_0'} \xrightarrow{\sim} \mathrm{End}(A'/\mathcal{O}_{K[f_0']})$ denote the CM action.

\begin{definition}Recall $1/(p+1) \le \epsilon_0 < p/(p+1)$ is as in (\ref{rchoice}). 
\begin{enumerate}
\item Choose a trivialization 
$$(e_1',e_2') : \mathbb{Z}^{\oplus 2} \xrightarrow{\sim} e^1H_1(A'(\mathbb{C}),\mathbb{C})$$
(where $H_1$ denotes singular homology and $e^1 = \left(\begin{array}{ccc} 1 & 0\\
0 & 0\\
\end{array}\right)$ is as in Convention \ref{idempotentconvention}) such that the corresponding trivialization
$$(e_1',e_2') : \mathbb{Z}_p^{\oplus 2} \xrightarrow{\sim} T_pA'$$
satisfies 
$$[\varpi]_{A'}(e_2) = e_1.$$
\item Thus $e_{1,1}$ trivializes the canonical subgroup of $A'$, and thus by Definition \ref{plusdefinitions} we have a point
$$y' = (A',\frac{\Omega' \tau_0}{f_0'N},\frac{\Omega'}{N},e_1',e_2') \in \mathcal{Y}^{\mathrm{Ig}}(\epsilon_0).$$
\item For $\frak{b} \in S'$ (see Definition \ref{Sdefinition2}), define 
\begin{equation}\label{YIgin2'}y_{\frak{b}}' = y' \cdot \frak{b} \in \mathcal{Y}^{\mathrm{Ig}}(\epsilon_0), \hspace{1cm} y_{\frak{b},a}' := y_{\frak{b}}' \cdot g^{-a} \overset{(\ref{gisomorphism})}{\in} \mathcal{Y}^{\mathrm{Ig}}(\epsilon_0/p^{a}).
\end{equation}
Here the first inclusion holds since $\frak{b} \in S'$ is prime to $p$, and thus the action $\cdot \frak{b}$ (see (\ref{Shimuralaw})) preserves canonical subgroups. 
%defined with respect to the $N$ from Choice \ref{NMf0choice}. 
\item Define
\begin{equation}\label{localmeasure'}\mu_{g \times \chi_0, \frak{b}} := G^{\flat}(y_{\frak{b},a}')(q_{\mathrm{dR}}) \overset{(\ref{derivativeGflatintegral})}{\in}\mathcal{O}_{\mathbb{C}_p} + (q_{\mathrm{dR}}-1)\mathbb{C}_p\llbracket q_{\mathrm{dR}}-1\rrbracket,
\end{equation}
where $G^{\flat}(y_{\frak{b},a}')(q_{\mathrm{dR}})$ denotes the specialization of the coefficients of the power series $$G^{\flat}(q_{\mathrm{dR}}) \in\hat{\mathcal{O}}_{\mathcal{Y}^{\mathrm{Ig}}(\epsilon_0/p^{a})}(\mathcal{Y}^{\mathrm{Ig}}(\epsilon_0/p^{a}))\llbracket q_{\mathrm{dR}}-1\rrbracket$$
from (\ref{Gflatdefinition}) along 
\begin{equation}\label{ybspecialize3}\hat{\mathcal{O}}_{\mathcal{Y}^{\mathrm{Ig}}(\epsilon_0/p^{a})}(\mathcal{Y}^{\mathrm{Ig}}(\epsilon_0/p^{a})) \rightarrow \hat{\mathcal{O}}_{\mathcal{Y}^{\mathrm{Ig}}(\epsilon_0/p^{a})}(y_{\frak{b},a}') \subset \mathbb{C}_p.
\end{equation}
\item Now define the Rankin-Selberg $p$-adic $L$-function as
\begin{equation}\label{globalmeasure'}\mathcal{L}_{g \times \chi_0} := \sum_{\frak{b} \in S'}\chi_0(\frak{b})\cdot\mu_{g \times \chi_0,\frak{b}}\in \mathcal{O}_{\mathbb{C}_p} + (q_{\mathrm{dR}}-1)\mathbb{C}_p\llbracket  q_{\mathrm{dR}}-1\rrbracket.
\end{equation}
%Now we let $r : \mathrm{Gal}(K(\frak{f})[p^{\infty}]/K) \rightarrow \Gamma_- = \mathbb{Z}_p$ denote the natural quotient. Via pushforward, we get a map
%$$r_* : \mathcal{O}_{\mathbb{C}_p}\llbracket \mathcal{C}\ell(\frak{f}^{(p)})[p^{\beta}] \times \mathbb{Z}_p\rrbracket [1/p] \rightarrow \mathcal{O}_{\mathbb{C}_p}\llbracket \mathbb{Z}_p\rrbracket[1/p], \hspace{1cm} r_*\mu(f) = \mu(f \circ r).$$
%We define
%$$\mathcal{L}_{g \times \chi_0} := r_*\mu_{g \times \chi_0} \in \mathcal{O}_{\mathbb{C}_p}\llbracket \mathbb{Z}_p\rrbracket.$$
\end{enumerate}
\end{definition}

We will also need to consider $p$-adic Maass-Shimura derivatives of $\mathcal{L}_{g \times \chi_0}$. 

\begin{definition}For $\frak{b} \in S'$ and $j \in \mathbb{Z}_{\ge 0}$, let
\begin{equation}\label{derivativelocalmeasure'}D_0^j\mu_{g \times \chi_0,\frak{b}} := D_0^jG^{\flat}(y_{\frak{b},a}')(q_{\mathrm{dR}}) \overset{(\ref{derivativeGflatintegral})}{\in} \mathcal{O}_{\mathbb{C}_p} + (q_{\mathrm{dR}}-1)\mathbb{C}_p\llbracket q_{\mathrm{dR}}-1\rrbracket,
\end{equation}
where $D_0^jG^{\flat}(y_{\frak{b},a}')(q_{\mathrm{dR}})$ denotes the specialization of the coefficients of the power series $$D_0^jG^{\flat}(q_{\mathrm{dR}}) \in \hat{\mathcal{O}}_{\mathcal{Y}^{\mathrm{Ig}}(\epsilon_0/p^{a})}(\mathcal{Y}^{\mathrm{Ig}}(\epsilon_0/p^{a}))\llbracket q_{\mathrm{dR}}-1\rrbracket$$
from (\ref{derivativeGflatdefinition}) along (\ref{ybspecialize3}). With $E/\mathbb{Q}$ the elliptic curve with CM by $\mathcal{O}_K$ as above, let
\begin{equation}\label{globalmeasureRS}D_0^j\mathcal{L}_{g \times \chi_0} := \sum_{\frak{b} \in S'}(\lambda_E^{-2j}\chi_0)(\frak{b}) \cdot D_0^j\mu_{g \times \chi_0,\frak{b}}.
\end{equation}
In particular, $D_0^0\mathcal{L}_{g \times \chi_0} = \mathcal{L}_{g \times \chi_0}$. 
\end{definition}

By (\ref{intermediatecontinuous2}) and the specialization of (\ref{intermediatecontinuous3}) along 
$$\hat{\mathcal{O}}_{\mathcal{Y}^{\mathrm{Ig}}(\epsilon_0/p^{a})}^+(\mathcal{Y}^{\mathrm{Ig}}(\epsilon_0/p^{a})) \rightarrow \hat{\mathcal{O}}_{\mathcal{Y}^{\mathrm{Ig}}(\epsilon_0/p^{a})}^+(y_{\frak{b},a}') \subset \mathcal{O}_{\mathbb{C}_p}$$
for every $\frak{b} \in S'$, we get that
\begin{equation}\label{Gcontinuous}\mathbb{Z}/(p-1)\times \mathbb{Z}_p \ni j \mapsto D_0^j\mathcal{L}_{g \times \chi_0}|_{q_{\mathrm{dR}} = 1} \in \mathcal{O}_{\mathbb{C}_p}
\end{equation}
is a continuous function. 

\subsection{Proof of Theorem \ref{GL2maintheorem}}\label{GL2maintheoremproofsection}

\begin{proof}[Proof of Theorem \ref{GL2maintheorem}]%The fact that $\mathcal{L}_{g \times \chi_0}$ is supported on $\mathbb{Z}_p^{\times}$ follows from an analogous argument to that in the beginning of the proof of Theorem \ref{interpolation2}, using (\ref{globalmeasure'}), (\ref{localmeasure'}), (\ref{Gflatdefinition}) and (\ref{flatformula2}).

\textbf{(1)}: Recall that $\theta_q(q_{\mathrm{dR}} -1) = 0$, and so
\begin{equation}\label{recallthetaX}\theta_q(D_0^jG^{\flat}(y_{\frak{b},a}')(q_{\mathrm{dR}})) = D_0^jG^{\flat}(y_{\frak{b},a}')(q_{\mathrm{dR}})|_{q_{\mathrm{dR}} = 1}.
\end{equation}
Analogously to (\ref{interpcalc'}), we have
\begin{equation}\label{interpcalcRankinSelberg}\begin{split}D_0^j\mathcal{L}_{g \times \chi_0}|_{q_{\mathrm{dR}} = 1} &= \sum_{\frak{b} \in S'}(\lambda_E^{-2j}\chi_0)(\frak{b})\cdot D_0^j\mu_{g \times \chi_0,\frak{b}}|_{q_{\mathrm{dR}} = 1} \\
&\overset{(\ref{derivativelocalmeasure'})}{=} \sum_{\frak{b} \in S'}(\lambda_E^{-2j}\chi_0)(\frak{b})\cdot D_0^jG^{\flat}(y_{\frak{b},a}')(q_{\mathrm{dR}})|_{q_{\mathrm{dR}} = 1} \\
%&\overset{(\ref{chainidentity'})}{=} \sum_{\frak{b} \in S}(\lambda_E^{-2j}\chi_0)(\frak{b})\theta_t\left(\left(\partial_0^j\log_{w_g^{\flat}}\right)(y_{\frak{b}})(q_{\mathrm{dR}})\right)|_{q_{\mathrm{dR}} = 1}\\
&\overset{(\ref{recallthetaX})}{=} \sum_{\frak{b} \in S'}(\lambda_E^{-2j}\chi_0)(\frak{b})\cdot D_0^jG^{\flat}(y_{\frak{b},a}')(q_{\mathrm{dR}})|_{q_{\mathrm{dR}} = 1}\\
&\overset{(\ref{compareMSvalues})}{=} \sum_{\frak{b} \in S'}(\lambda_E^{-2j}\chi_0)(\frak{b})\frac{\Omega_{\infty}(y_{\frak{b},a}')^{2j}}{\Omega_p(y_{\frak{b},a}')^{2j}}\left(\partial_2^{j-1}F_g^{\flat}(y_{\frak{b},a}')\right),\\
%&= \Omega_p^{-2j}\sum_{\frak{b}'}(\lambda_E^{-2j}\chi_0\chi)(\frak{b}')\left(\Omega_{\infty}^{2j}\partial_0^jG^{\flat}(y_{\frak{b}'})\right)
\end{split}
\end{equation}
where in the last line, we use $D_0^jG^{\flat} = D_2^jw_g^{\flat}$, which we compare with $D_2F_g^{\flat}$ using (\ref{compareMSvalues}). Here, $\Omega_p(y_{\frak{b},a}')$ and $\Omega_{\infty}(y_{\frak{b},a}')$ are defined as in (\ref{defineomegaperiods}) with respect to the point 
$$y_{\frak{b},a}'   \overset{(\ref{YIgin2'})}{\in} \mathcal{Y}^{\mathrm{Ig}}(\epsilon_0/p^{a}) \subset \mathcal{Y}^{\mathrm{Ig}}(\epsilon_0) = U \overset{(\ref{rchoice}), (\ref{Um})}{=} U_{n(\epsilon_0)-1}$$
and, letting $A_{\frak{b},a}'$ be the underlying false elliptic curve of $y_{\frak{b},a}'$, with respect to any fixed generator 
\begin{equation}\label{fixdifferential3}w_0(A_{\frak{b},a}') \in \Omega_{A_{\frak{b},a}'/\overline{\mathbb{Q}}}
\end{equation}
Here note that we can also apply (\ref{compareMSvalues}) to the CM point $y_{\frak{b},a}'$, as it evidently satisfies Assumption (2) of Theorem \ref{CMcoincidetheorem}. 
%where $y_{\frak{b}'}$ ranges over 
%$$y_{\frak{b}} \cdot \gamma_{-a/p^n,n}\rho(\pi p^{n+\beta}), \hspace{.25cm} 0\le a \le p^n-1, \hspace{.25cm} \frak{b}' = \frak{b}\left(\mathbb{Z}\gamma^{r_p(a)} + p^{n+\beta}\mathcal{O}_K\right).$$
%Note that the factor $\pi^{2j}p^{(j-1)n + 2j\beta}$ comes from $\pi^{2j}p^{2j(n+\beta)} = \lambda_{E,\infty}^{-2j}(\pi p^{n+\beta})$ times $1/p^{(j+1)n}$. 

We claim that there exists $\Omega'' \in \overline{\mathbb{Q}}^{\times}$ such that 
\begin{equation}\label{OmegaOmega'}\Omega_p(y_{\frak{b},a}') = \Omega''\cdot \Omega_p(a) \hspace{.5cm} \text{and} \hspace{.5cm} \Omega_{\infty}(y_{\frak{b},a}') = \Omega'' \cdot \Omega_{\infty}(a),
\end{equation}
where $\Omega_p(a)$ and $\Omega_{\infty}(a)$ are as in Definition \ref{finalperiodsdefinition}. To see this, note that the isogeny 
\begin{equation}\label{complexpiba'}\mathbb{C}^{\oplus 2}/\iota_{\infty}(\mathcal{O}_D)\cdot \left(\begin{array}{ccc} \frac{\tau_0}{f_0'p^a} \\
1\end{array}\right) \rightarrow \mathbb{C}^{\oplus 2}/\iota_{\infty}(\mathcal{O}_D)\cdot \left(\begin{array}{ccc} \frac{\tau_0}{p^a} \\
1\end{array}\right)\end{equation}
induces a degree $f_0'$ isogeny 
$$\pi_{\frak{b}',a}' : A_{\frak{b}',a}' \rightarrow A_{\frak{b},a}$$
defined over $K[f_0']$ for every $\frak{b}' \in S'$ mapping to the image of $\frak{b} \in S \subset \mathbb{A}_K^{\times,(\infty)} \rightarrow \mathcal{C}\ell(\frak{f}^{(p)})[p^a]$ under the map 
$$\mathbb{A}_K^{\times,(\infty)} \rightarrow \mathcal{C}\ell(\frak{f}^{(p)})[f_0'p^a] \rightarrow \mathcal{C}\ell(\frak{f}^{(p)})[p^a].$$ 
Let $y_{\frak{b},a} \in \mathcal{Y}^{\mathrm{Ig}}(\epsilon_0)(\overline{\mathbb{Q}}_p,\overline{\mathbb{Z}}_p)$ as in (\ref{YIgin2}). Since $(f_0',p) = 1$, we have that 
\begin{equation}\label{follows1}\pi_{\frak{b}',a}'^*(w_{\mathrm{can}}(y_{\frak{b},a})) = w_{\mathrm{can}}(y_{\frak{b}',a}').
\end{equation}
Similarly, since the differential $2\pi idz$ on $\mathbb{C}^{\oplus 2}/\iota_{\infty}(\mathcal{O}_D)\cdot \left(\begin{array}{ccc} \tau \\
1\end{array}\right)$ has period lattice $\mathbb{Z}\tau + \mathbb{Z}$, we see from (\ref{complexpiba'}) that the preimage the period lattice $\mathbb{Z}\tau_0/p^a + \mathbb{Z}$ of $(2\pi idz)(y_{\frak{b},a})$ under $\pi_{\frak{b}',a}'$ is the period lattice $\mathbb{Z}\tau_0/(f_0'p^a) + \mathbb{Z} $ of $(2\pi idz)(y_{\frak{b}',a}')$, and hence
\begin{equation}\label{follows2}\pi_{\frak{b}',a}'^*((2\pi idz)(y_{\frak{b},a})) = (2\pi idz)(y_{\frak{b}',a}').
\end{equation}
Finally, recall $w_0(A_{\frak{b}})$ from (\ref{fixdifferentialb}) and $w_0(A_{\frak{b}',a}')$ from (\ref{fixdifferential3}) and define $\Omega'' \in \overline{\mathbb{Q}}^{\times}$ by 
$$\Omega'' \cdot \pi_{\frak{b},a}'^*w_0(A_{\frak{b}}) = w_0(A_{\frak{b}',a}').$$
Now (\ref{OmegaOmega'}) follows from (\ref{compareeis}), (\ref{Omegaynequal}), (\ref{follows1}), (\ref{tildewcan}) and (\ref{follows2}). Note that although the quaternion algebras $D/\mathbb{Q}$ underlying $Y$ may be different in Sections \ref{padicLfunctionsection} and \ref{padicLfunctionsection2}, Definition \ref{finalperiodsdefinition} and (\ref{Omegaynequal}) do not depend on $D/\mathbb{Q}$ so long as the latter is split at $p$ and $\infty$, which is satisfied in both Sections (see Assumptions \ref{Nf0assumption} and \ref{Nf0assumption2}).

Now from (\ref{interpcalcRankinSelberg}) and (\ref{OmegaOmega'}) we have 
\begin{align*}
D_0^j\mathcal{L}_{g \times \chi_0}|_{q_{\mathrm{dR}} = 1} = \sum_{\frak{b} \in S'}(\lambda_E^{-2j}\chi_0)(\frak{b})\frac{\Omega_{\infty}(a)^{2j}}{\Omega_p(a)^{2j}}\left(\partial_2^{j-1}F_g^{\flat}(y_{\frak{b},a}')\right).
\end{align*}

%Recall that $\psi = \lambda_E/\chi_0$ and $\lambda_E$ is ramified at $\frak{p}$, so 
%Note that $(\lambda_E^{-2j}\chi_0)_{\frak{p}}|_{\mathcal{O}_{K_{\frak{p}}}^{\times}} \overset{(\ref{localvcondition})}{=} \lambda_{E,\frak{p}}^{-2j+1}|_{\mathcal{O}_{K_{\frak{p}}}^{\times}}$ is ramified at $\frak{p}$ and thus by convention $(\lambda_E^{-2j}\chi_0)(\frak{p}) = 0$. 
Since $F_g^{\flat}$ is an eigenform for $U = [p]U_p$, say of eigenvalue $\alpha_p$, we have, analogously to Proposition \ref{Vproposition}, that $\partial_2^jF_g$ is a $U$-eigenform of eigenvalue $p^{j-1}\alpha_p$ and thus 
$$\partial_2^{j-1}(F_g^{\flat}) = (1-V_p^*U_p^*)\partial_2^{j-1}(F_g) = (1-[p^{-1}]V_p^*U^*)\partial_2^{j-1}(F_g) = (1-p^{j-1}\alpha_p ([p^{-1}]V_p)^*)\partial_2^{j-1}(F_g).$$
Thus by an argument analogous to \cite[Theorem 5.9]{BDP}, one can show that 
\begin{align*}&\left(\sum_{\frak{b} \in S'}(\lambda_E^{-2j}\chi_0)(\frak{b})\partial_2^{j-1}(F_g^{\flat})(y_{\frak{b},a}')\right)^2 = (1-p^{j-1}\alpha_p (\lambda_E^{-2j}\chi_0)(\frak{p}))^2\left(\sum_{\frak{b} \in S'}(\lambda_E^{-2j}\chi_0)(\frak{b})\partial_2^{j-1}(F_g)(y_{\frak{b},a}')\right)^2 \\
&=(1-p^{j-1}\alpha_p (\lambda_E^{-2j}\chi_0)(\frak{p}))^2 \left(\sum_{\frak{b} \in S'}(\lambda_E^{-2j}\chi_0)(\frak{b})\partial_2^{j-1}(F_g)(y_{\frak{b},a}')\right)^2 \\
&\overset{(\ref{An'complex})}{=}(1-p^{j-1}\alpha_p (\lambda_E^{-2j}\chi_0)(\frak{p}))^2 \left(\sum_{\frak{b} \in S'}(\lambda_E^{-2j}\chi_0)(\frak{b})\partial_2^{j-1}(F_g)\left(\mathbb{C}/i_{\infty}(\mathcal{O}_D)\left(\mathbb{Z}\tau_0/p^{a} + \mathbb{Z}\right),P,e_1,e_2\right)\right)^2\\
&=(1-p^{j-1}\alpha_p (\lambda_E^{-2j}\chi_0)(\frak{p}))^2(p^{a}\varpi)^{4j} \left(\sum_{\frak{b} \in S'}(\lambda_E^{-2j}\chi_0)(\frak{b})\partial_2^{j-1}(F_g)\left(\mathbb{C}/i_{\infty}(\mathcal{O}_D)\left(\mathbb{Z}p^{a}\varpi + \mathbb{Z}\right),P,e_1,e_2\right)\right)^2
\end{align*}
where the last equality follows because $\partial_2^{j-1}F_g$ has weight $2j$ and $\tau_0 = 1/\varpi$ (\ref{tau0}).

Finally, from the explicit Waldspurger formula (\cite[Theorem 1.4]{YZZ} and \cite[Theorem 1.8]{explicitGZ} with $c_1$ in loc. cit. taken to be $f_0'p^a$ where $f_0'$ is as in Definition \ref{f'notation} (recall also Assumption \ref{f'assumption}), cf. also \cite[Theorem 4.6]{BDP}) we have
\begin{align*}%\Omega_p^4\pi^4p^{4\beta}\vartheta\left(p^n\sum_{\frak{b}'}(\lambda^{-2}\chi_0\chi)(\frak{b}')d_0G^{\flat}(y_{\frak{b}'})\right)^2 &= \pi^4p^{2n+4\beta}\left(\tilde{\Omega}_p^2\sum_{\frak{b}'}(\lambda^{-2}\chi_0\chi)(\frak{b}')d_0G^{\flat}(y_{\frak{b}'})\right)^2 \\
%&\hspace{-1cm}\overset{(\ref{comparerealperiods2})}{=}\
&\left(\sum_{\frak{b} \in S'}(\lambda_E^{-2j}\chi_0)(\frak{b})\partial_2^{j-1}(F_g)\left(\mathbb{C}/i_{\infty}(\mathcal{O}_D)\left(\mathbb{Z}p^{a}\varpi + \mathbb{Z}\right),P,e_1,e_2\right)\right)^2 \\
&= \mathrm{vol}\left(\mathbb{C}/(\mathbb{Z}p^{a}\varpi + \mathbb{Z})\right)^{-2j} \cdot \pi^{2j-1}j!(j-1)!L(g,\chi_0(\lambda_E^c/\lambda_E)^j,1)\cdot C_{g,\chi_0}'\\
&= (p^{a}\mathrm{Im}(\varpi))^{-2j}\cdot \pi^{2j-1}j!(j-1)!L(g,\chi_0(\lambda_E^c/\lambda_E)^j,1)\cdot C_{g,\chi_0}'\\
&= p^{-2aj}\cdot L^{\mathrm{alg}}(g,\chi_0(\lambda_E^c/\lambda_E^j),1)\cdot C_{g,\chi_0}'
\end{align*}
for some constant $C_{g,\chi_0}' \in \overline{\mathbb{Q}}^{\times}$ depending only on $g$ and $\chi_0$. Here, we use that 
$$\mathrm{Im}(\varpi) \overset{(\ref{pichoice})}{=} \frac{\sqrt{|D_K|}}{2},$$
as well as the fact that $M$ is the prime-to-$p$ part of the conductor of $g$, so that $\overline{\lambda}_E(M) = \overline{\eta}(M) = \eta(M) = \lambda_E(M)$ and thus $(\lambda_E/\overline{\lambda}_E)^j(M) = 1$. 
%Moreover, $L^{\mathrm{alg},(\frak{p})} = L^{\mathrm{alg}}$ since one can verify using a $(\lambda_E^c/\lambda_E)^j$-twisted version of the factorization (\ref{Artinfactorization}) and the ramifiedness of $\lambda_E$ at $\frak{p}$ that the Euler factor at $\frak{p}$ is 1. 
 Together with the previous paragraph, this gives (\ref{interpRankinSelberg}). \\ %Note that the square of the toric period is equal to 
%$$p^n C_g' L^{\mathrm{alg},(\frak{p})}(g,\chi_0(\lambda_E^c/\lambda_E)^j\chi,1),$$
%with the factor of $p^n$ coming from the conductor of $\chi$ (cf. \cite[Proposition 3.6]{CastellaHsieh}). This explains the power of $2(j-1)n$ on the left-hand side and $(2j-1)n$ on the right-hand side of the above equality. Together with (\ref{interpcalcRankinSelberg}), this finishes the proof of (1). 

\textbf{(2)}: Recall 
\begin{equation}\label{recallintegral}\log_{w_g^{\flat}} \overset{(\ref{logwgintegral})}{\in} \mathcal{O}_{\mathcal{Y}^{\mathrm{Ig}}(\epsilon_0)}^+(\mathcal{Y}^{\mathrm{Ig}}(\epsilon_0)).
\end{equation}
%and the fact that 
%\begin{equation}\label{derivativeidentity}d_0^j\log_{w_g^{\flat}} = d_2^{j-1}\tilde{G}_2^{\flat}
%\end{equation}
%for all $j \in \mathbb{Z}_{\ge 1}$. %Let $G^{\flat}(q_{\mathrm{dR}}) \in \hat{\mathcal{O}}_{\mathcal{Y}^{\mathrm{Ig}}(\epsilon_0)}(\mathcal{Y}^{\mathrm{Ig}}(\epsilon_0))\llbracket q_{\mathrm{dR}}-1\rrbracket$ be as in (\ref{Gflatdefinition}) and let $G_2(q_{\mathrm{dR}}) \in \hat{\mathcal{O}}_{\mathcal{V}_x}(\mathcal{V}_x)\llbracket q_{\mathrm{dR}}-1\rrbracket$ be as in (\ref{derivativeGflatdefinition}). p^{a}
Thus
\begin{equation}\label{derivativeidentity2}\theta_t(\theta_q(d_0^j\log_{w_g^{\flat}}(q_{\mathrm{dR}}))) \overset{(\ref{Gcontinuous})}{=} \theta_q(D_0^jG^{\flat}(q_{\mathrm{dR}}))
\end{equation}
By (\ref{globalmeasure'}) and the continuity of (\ref{Gcontinuous}) in $j$, the first equality in the $p$-adic Waldspurger formula (\ref{padicWaldspurger}) will follow if we can show, letting $j = (p-1)p^{n-1}$ in (\ref{derivativeidentity2}) and taking the limit $\mathbb{Z}_{\ge 1} \ni n \rightarrow \infty$, that the limit exists and is equal to 
$$\log_{w_g^{\flat}} \in \hat{\mathcal{O}}_{\mathcal{Y}^{\mathrm{Ig}}(\epsilon_0)}(\mathcal{Y}^{\mathrm{Ig}}(\epsilon_0)).$$
The proof of this is based on an analogous argument to the proof of Theorem \ref{pintegraltheorem2} (2), which we give in the next paragraph.

For all $j \in \mathbb{Z}_{\ge 0}$, note that
$$\theta_t(\theta_q(d_0^{j}\log_{w_g^{\flat}}))|_{\mathcal{Y}^{\mathrm{Ig}}(\epsilon_0/p^{a})} \overset{(\ref{recallintegral}), (\ref{derivativeGflatintegral})}{\in} \hat{\mathcal{O}}_{\mathcal{Y}^{\mathrm{Ig}}(\epsilon_0/p^{a})}^+(\mathcal{Y}^{\mathrm{Ig}}(\epsilon_0/p^{a}))$$
where the case $j = 0$ follows from (\ref{recallintegral}) and the case $j \in \mathbb{Z}_{\ge 1}$ follows from (\ref{derivativeGflatintegral}) and (\ref{derivativeidentity2}). By Propositions \ref{prop1} and  \ref{prop4} we have that 
$$\mathbb{Z}_{\ge 0} \ni j \mapsto \theta_t(\theta_q(d_0^{j}\log_{w_g^{\flat}}))|_{\mathcal{Y}^{\mathrm{Ig}}} \in \hat{\mathcal{O}}_{\mathcal{Y}^{\mathrm{Ig}}}^+(\mathcal{Y}^{\mathrm{Ig}})$$
is continuous and moreover satisfies the following uniform continuity: if $j \equiv j' \pmod{(p-1)p^{n-1}}$ then
$$\theta_t(\theta_q(d_k^{j}\log_{w_g^{\flat}}))|_{\mathcal{Y}^{\mathrm{Ig}}}  \equiv  \theta_t(\theta_q(d_k^{j'}\log_{w_g^{\flat}}))|_{\mathcal{Y}^{\mathrm{Ig}}} \pmod{p^n\hat{\mathcal{O}}_{\mathcal{Y}^{\mathrm{Ig}}}^+(\mathcal{Y}^{\mathrm{Ig}})}.$$
Since
$$\theta_t(\theta_q(d_0^{j}\log_{w_g^{\flat}}))|_{\mathcal{Y}^{\mathrm{Ig}}(\epsilon_0/p^{a})} \in \hat{\mathcal{O}}_{\mathcal{Y}^{\mathrm{Ig}}(\epsilon_0/p^{a})}^+(\mathcal{Y}^{\mathrm{Ig}}(\epsilon_0/p^{a})) \rightarrow \varinjlim_{0 < \epsilon < \epsilon_0/p^{a}}\hat{\mathcal{O}}_{\mathcal{Y}^{\mathrm{Ig}}(\epsilon)}^+(\mathcal{Y}^{\mathrm{Ig}}(\epsilon))$$
and by Corollary \ref{completioncorollary} we have that $\hat{\mathcal{O}}_{\mathcal{Y}^{\mathrm{Ig}}}^+(\mathcal{Y}^{\mathrm{Ig}})$ is the $p$-adic completion of the direct limit in the previous displayed expression, then for any $n \in \mathbb{Z}_{\ge 1}$ we have for all $m \in \mathbb{Z}_{\ge a}$ that 
\begin{align*}\log_{w_g^{\flat}}|_{\mathcal{Y}^{\mathrm{Ig}}(\epsilon_0/p^m)} &\overset{(\ref{thetaq}), (\ref{thetat})}{=} \theta_t(\theta_q(\log_{w_g^{\flat}}))|_{\mathcal{Y}^{\mathrm{Ig}}(\epsilon_0/p^m)} \\
&\equiv \theta_t(\theta_q(d_0^{(p-1)p^{n-1}}\log_{w_g^{\flat}}))|_{\mathcal{Y}^{\mathrm{Ig}}(\epsilon_0/p^m)} \pmod{p^n\hat{\mathcal{O}}_{\mathcal{Y}^{\mathrm{Ig}}(\epsilon_0/p^m)}^+(\mathcal{Y}^{\mathrm{Ig}}(\epsilon_0/p^m))}.
\end{align*}
By Proposition \ref{prop2}, for all $m \in \mathbb{Z}_{\ge a}$ the open subset $\hat{\mathcal{Y}}^{\mathrm{Ig}}(\epsilon_0/p^m)^+ \subset \hat{\mathcal{Y}}^{\mathrm{Ig}}(\epsilon_0/p^{a})^+$ is dense. However, the locus in $\hat{\mathcal{Y}}^{\mathrm{Ig}}(\epsilon_0/p^{a})^+$ defined by
$$\log_{w_g^{\flat}}\equiv \theta_t(\theta_q(d_0^{(p-1)p^{n-1}}\log_{w_g^{\flat}})) \pmod{p^n}$$
is pro-Zariski (and thus adically) closed. Since the dense open subset $\hat{\mathcal{Y}}^{\mathrm{Ig}}(\epsilon_0/p^m)^+$ is contained in this locus, we have that the locus is all of $\hat{\mathcal{Y}}^{\mathrm{Ig}}(\epsilon_0/p^{a})^+$, i.e.
$$\log_{w_g^{\flat}}|_{\mathcal{Y}^{\mathrm{Ig}}(\epsilon_0/p^{a})} \equiv \theta_t(\theta_q(d_0^{(p-1)p^{n-1}}\log_{w_g^{\flat}}))|_{\mathcal{Y}^{\mathrm{Ig}}(\epsilon_0/p^{a})} \pmod{p^n\hat{\mathcal{O}}_{\mathcal{Y}^{\mathrm{Ig}}(\epsilon_0/p^{a})}^+(\mathcal{Y}^{\mathrm{Ig}}(\epsilon_0/p^{a}))}.$$
Letting $n \rightarrow \infty$ gives the assertion. 

Finally, we show the second equality in (\ref{padicWaldspurger}). Note that as $\log_{w_g}$ is a $U = [p]U_p$-eigenvector (see Definition \ref{flatdefinition} for the definitions of $U$ and $U_p$) of eigenvalue $\psi(\frak{p})/p$, we have
$$\log_{w_g^{\flat}} = (1-\frac{\psi(\frak{p})}{p}[p^{-1}]^*V_p^*)\log_{w_g},$$
from which, analogously to the computation in \cite[Theorem 5.9]{BDP}, we can compute as
$$\log_{w_g^{\flat}}(P_{g,\chi_0}) = (1-\frac{(\psi\chi_0)(\frak{p})}{p})\log_{w_g}(P_{g,\chi_0}) = (1-\frac{\lambda(\frak{p})}{p})\log_{w_g}(P_{g,\chi_0}) = \log_{w_g}(P_{g,\chi_0})$$
where the last equality follows since $\lambda$ is ramified at $\frak{p}$ (by Assumption \ref{pconductorassumption} (2)). 

\end{proof}

\subsection{A factorization of $p$-adic $L$-functions}\label{factorizationsection}In this section, we relate the $GL_1/K$ and $GL_2/\mathbb{Q}$ $p$-adic $L$-functions in the case where $g = \theta_{\psi}$. 

We will need a Lemma asserting continuity of the $j$-fold derivative.

\begin{lemma}\label{continuitylemma}The maps $\mathbb{Z}_{\ge 0} \rightarrow \mathcal{O}_{\mathbb{C}_p}$ defined by 
$$j \mapsto D_0^j\mathcal{L}_{g \times \chi_0}|_{q_{\mathrm{dR}} = 1}, \hspace{.5cm} j \mapsto D_1^j\mathcal{L}_{\lambda}|_{q_{\mathrm{dR}} = 1}, \hspace{.5cm} j \mapsto D_1^j\mathcal{L}_{\psi\chi_0^c}|_{q_{\mathrm{dR}} = 1}$$
extend to continuous functions $\mathbb{Z}/(p-1) \times \mathbb{Z}_p \rightarrow \mathcal{O}_{\mathbb{C}_p}$, where $\mathbb{Z} \subset \mathbb{Z}/(p-1) \times \mathbb{Z}_p$ is embedded diagonally.
\end{lemma}

\begin{proof}%Recall (\ref{globalmeasureRS})
%$$D_0^j\mathcal{L}_{g \times \chi_0}:= \sum_{\frak{b} \in S'}(\lambda_E^{-2j}\chi_0)(\frak{b})\cdot D_0^jG^{\flat}(y_{\frak{b},a}')(q_{\mathrm{dR}}).$$
 The continuity statement for $D_0^j\mathcal{L}_{g \times \chi_0}$ immediately follows from (\ref{globalmeasureRS}) and the continuity of $j \mapsto D_0^jG^{\flat}(y_{\frak{b},a}')(q_{\mathrm{dR}})|_{q_{\mathrm{dR}} = 1}$ (\ref{Gcontinuous}) and Lemma \ref{sequencelemma}. The proof of continuity of $j \mapsto D_1^j\mathcal{L}_{\lambda}|_{q_{\mathrm{dR}} = 1}$ and $j \mapsto D_1^j\mathcal{L}_{\psi\chi_0^c}|_{q_{\mathrm{dR}} = 1}$ is completely analogous, using (\ref{globalmeasure2}) and (\ref{Econtinuous}) (with $\lambda = \lambda_E$ and $\lambda = \psi\chi_0^c$) in place of (\ref{Gcontinuous}).

\end{proof}

Suppose now that $\lambda = \lambda_E$ is the infinity type $(1,0)$ Hecke character attached to an elliptic curve $E/\mathbb{Q}$ with CM by $\mathcal{O}_K$ (see Example \ref{Eexample}). Let $\psi, \chi_0$ be as in Choice \ref{twistchoice} and (\ref{chi0choice}) defined with respect to $\lambda = \lambda_E$. Consider the Rankin-Selberg pair $(g,\chi_0) = (\theta_{\psi},\chi_0)$. Then $\mathcal{L}_{g \times \chi_0}$ is well-defined. Moreover, $\lambda_E$ and $\psi\chi_0^c$ both satisfy Assumption \ref{pconductorassumption}, and so $\mathcal{L}_{\lambda_E}$ and $\mathcal{L}_{\psi\chi_0^c}$ are well-defined. In this case, we prove a factorization identity involving special values of our $p$-adic $L$-functions, which will be essential in our proof of the rank 1 $p$-converse theorem. 

\begin{theorem}For $(g,\chi_0)$ as above, we have the following identity: for all $j \in \mathbb{Z}/(p-1) \times \mathbb{Z}_p$,
\begin{equation}\label{factorization}D_1^{j-1}\mathcal{L}_{\lambda_E}|_{q_{\mathrm{dR}} = 1}\cdot D_1^j\mathcal{L}_{\psi\chi_0^c}|_{q_{\mathrm{dR}} = 1} = \left(D_0^j\mathcal{L}_{g\times \chi_0}|_{q_{\mathrm{dR}} = 1}\right)^2 \cdot C_{g \times \chi_0},
\end{equation}
where $C_{g\times \chi_0} \in \overline{\mathbb{Q}}_p^{\times}$ is some constant depending only on the pair $(g,\chi_0)$ (and independent of $j$). 
\end{theorem}

\begin{proof}From the interpolation identities (\ref{interp1twist}) and (\ref{interpRankinSelberg}), (\ref{factorization}) holds for all $j \in \mathbb{Z}_{\ge 1}$. Now the assertion follows from Lemma \ref{continuitylemma}.
\end{proof}

The following proposition uses the non-vanishing $L$-value assumption of Choice \ref{twistchoice} on the pair $(g,\chi_0)$; recall that the existence of such a $\chi_0$ used the deep (and ineffective) theorem of Rohrlich \cite{Rohrlich}. 

\begin{proposition}\label{nonzerotheorem}We have
\begin{equation}\label{nonzero}\mathcal{L}_{\psi\chi_0^c}|_{q_{\mathrm{dR} = 1}} \neq 0.
\end{equation}
\end{proposition}

\begin{proof}By Choice \ref{twistchoice}, 
$$L(\psi^{-1}(\chi_0^c)^{-1},0) = L(\psi^c\chi_0,1) \neq 0.$$
By (\ref{interp1twist}) with $j = 0$ and $\lambda = \psi\chi_0^c$, we have for some $C \in \mathbb{C}_p^{\times}$,
$$\mathcal{L}_{\psi\chi_0^c}|_{q_{\mathrm{dR}} = 1} = C \cdot \Omega_{\infty}(a) \cdot L^{(\frak{p})}(\psi^{-1}(\chi_0^c)^{-1},0)\cdot \Omega_p(a)^{-1}C_{\psi\chi_0^c} \neq 0.
$$
\end{proof}

The following Corollary will be key in linking the non-vanishing of a Heegner point to the corank 1 assumption on the Selmer group in the proof of the rank 1 $p$-converse theorem.

\begin{corollary}\label{corollaryHeegnerformula}We have
\begin{equation}\label{Heegnerpointidentity2}D_1^{-1}\mathcal{L}_{\lambda_E}|_{q_{\mathrm{dR}}=1} = \left(\mathcal{L}_{g\times \chi_0}\right)^2|_{q_{\mathrm{dR}} = 1} \cdot C' = \log_{w_g}^2(P_{g,\chi_0}) \cdot C''
\end{equation}
for some $C', C'' \in \mathbb{C}_p^{\times}$. Hence $D_1^{-1}\mathcal{L}_{\lambda_E}|_{q_{\mathrm{dR}} = 1} \neq 0$ if and only if the Heegner point 
$$P_{g,\chi_0} \in (A_g \otimes \chi_0)^{\mathrm{Gal}(\overline{K}/K)} \otimes \mathbb{Q}$$
is nontrivial. 
\end{corollary}

\begin{proof}Plugging $j = 0$ into (\ref{factorization}) we have 
$$(\mathcal{L}_{g \times \chi_0}|_{q_{\mathrm{dR}} = 1})^2  \cdot C_{g \times \chi_0}
= D_1^{-1}\mathcal{L}_{\lambda_E}|_{q_{\mathrm{dR}} = 1} \cdot \mathcal{L}_{\psi\chi_0^c}|_{q_{\mathrm{dR}} = 1}.$$
Now (\ref{Heegnerpointidentity2}) follows from (\ref{nonzero}) and (\ref{padicWaldspurger}). 

%$$\mathcal{L}_{\lambda}\cdot \left(d_0\mathcal{L}_{\psi\chi_0^c} + \frac{1}{\vartheta(z_{\mathrm{dR}}-z_{\mathrm{HT}})(y)}\mathcal{L}_{\psi\chi_0^c}\right) = \left(d_0\mathcal{L}_{g\times \chi}\right)^2 \cdot C_{g \times \chi_0}$$
%for some $C_{g \times \chi_0} \in \mathbb{C}_p^{\times}$, and by (\ref{nonzero}), we see that the specialization to $\kappa_-^{-1}$ 
%$$\mathcal{L}_{\lambda}(\kappa_-^{-1}) \cdot C'''  = \left(d_0\mathcal{L}_{g\times \chi}\right)^2(\kappa_-^{-1}) \cdot C_{g \times \chi_0} \overset{(\ref{logarithmicderivative}),(\ref{padicWaldspurger})}{=} \log_{w_g}^2(P_{g,\chi_0})  \cdot C_{g\times \chi_0}'$$
%for $C''', C_{g\times \chi_0} \in \mathbb{C}_p^{\times}$. 

\end{proof}

\begin{remark}The identity (\ref{Heegnerpointidentity2}) can be viewed as a supersingular analogue of Rubin's formula (see \cite{Rubinformula}, also \cite{BDP2}). Under the assumption $\mathrm{corank}_{\mathbb{Z}_p}\mathrm{Sel}_{p^{\infty}}(E/\mathbb{Q}) = 1$, we will show that the left-hand side of (\ref{Heegnerpointidentity2}) is nonzero (see the proof of Theorem \ref{BSDrank1theorem}) using a ``Rubin-type main conjecture'' (\ref{RMC}), which we formulate and prove in the next sections.
\end{remark}

\section{Preparations from Classical Iwasawa Theory}\label{backgroundsection}

For this section, let $K/\mathbb{Q}$ be an arbitrary imaginary quadratic field. We do not assume that $K$ has class number 1 (as in Assumption \ref{pramifiedassumption}) until we make Choice \ref{Echoice} in Section \ref{Galoisextensionsection}. Continue to assume that $p$ is a finite prime that ramifies in $K/\mathbb{Q}$. For the rest of the paper, when given $\lambda$ as in Assumption \ref{pconductorassumption}, we will work under Assumption \ref{Nf0assumption} when invoking the objects of Convention \ref{Yconvention}.

\subsection{Coleman power series}\label{Colemansection}Let $K_p/\mathbb{Q}_p$ be an extension of degree $2$, let $L_p/K_p$ a finite unramified extension. (Here, $L_p$ does not yet refer to the $p$-adic completion under (\ref{fixembeddings}) of a number field $L/K$; later in Definition \ref{varphidefinition} we will take $L = K(\frak{f}_0)$ and $L_p = K(\frak{f}_0)_p$, where $\frak{f}_0 \subset \mathcal{O}_K$ is some ideal prime to $p$.) Then $\mathrm{Gal}(L_p/K_p)$ is cyclic and generated by the Frobenius element $\phi$ generating the Galois group of residue fields of $L_p/K_p$. Given a power series $g(X) \in L_p\llbracket X\rrbracket$. Let $g^{\phi}(X) \in L_p\llbracket X\rrbracket$ denote the power series obtained by applying $\phi$ to the coefficients of $g(X)$.

Let $F$ be a height 2 relative Lubin-Tate formal group with respect to $L_p/K_p$, in the sense of \cite[Chapter I.1]{deShalit}. In particular, $F$ is a formal group defined over $\mathcal{O}_{L_p}$ with $\mathrm{End}(F/\mathcal{O}_{L_p}) = \mathcal{O}_{K_p}$ and equipped with special $\phi$-linear endomorphism 
$$f : F \rightarrow F^{\phi},$$
where for any $m\in \mathbb{Z}$, $F^{\phi^m}$ denotes the formal group obtained by applying $\phi^m$ to the coefficients of $F$. Let 
$$f^{\phi^m} : F^{\phi^m} \rightarrow F^{\phi^{m+1}}$$
denote the $\phi$-linear endomorphism obtained by applying $\phi^m$ to the coefficients of $f$. Let 
$$[+] : F \times F \rightarrow F$$
denote the formal group law on $F$, and let $[a]$ denote formal multiplication by $a \in \mathcal{O}_{K_p}$. Let 
$$\varpi' := f'(0)$$
which is a uniformizer of $\mathcal{O}_{L_p}$, let $d := [L_p:K_p]$, and let 
\begin{equation}\label{fncomposition}f^n = f^{\phi^{n-1}} \circ f^{\phi^{n-2}} \circ \cdots \circ f^{\phi} \circ f,
\end{equation}
and note that $f^d = [\xi]$ for 
$$\xi := \Nm_{L_p/K_p}(\varpi') \in \mathcal{O}_{K_p}.$$
Given $\nu \in \mathcal{O}_{K_p}$, let $F[\nu]$ denote the $\nu$-torsion group scheme. Let $F[f^n]$ denote the $f^n$-torsion group scheme. Note that for any local parameter $\varpi$ of $\mathcal{O}_{K_p}$, we have $F[f^n] = F[\varpi^n]$.

A prototypical example of a relative Lubin-Tate group arises from the formal group of a CM elliptic curve. We will work in this setting extensively later. 

\begin{example}\label{CMexample}Suppose $A$ is an elliptic curve with CM by $\mathcal{O}_K$ defined over $\mathcal{O}_L$, where $L/K$ is a finite extension. Suppose $A$ has good reduction at all places above $p$, and suppose $L_p/K_p$ is unramified. Then the formal group $\hat{A}$ over $\mathcal{O}_{L_p}$ is an example of a height 2 relative Lubin-Tate group with respect to $L_p/K_p$. Let $\frak{p}$ be the prime of $K$ determined by (\ref{fixembeddings}). The special $\phi$-linear endomorphism is the projection 
$$f : \hat{A} \rightarrow \hat{A}/\hat{A}[\frak{p}] \cong \hat{A}^{\phi}.$$ 
\end{example}

Let $\varpi$ denote a uniformizer of $\mathcal{O}_{K_p}$.  Let $L_{p,n} = L_p(F[\varpi^n])$. Then the isomorphism
\begin{equation}\label{kappadefinition}\kappa : \mathrm{Gal}(L_{p,\infty}/L_p) \xrightarrow{\sim} \mathcal{O}_{K_p}^{\times}, \hspace{1cm} \sigma(x) = [\kappa(\sigma)](x) \hspace{.25cm} \forall x \in F[p^{\infty}],
\end{equation}
as defined as in \cite[Chapter I.3.3 (9)]{deShalit}, is the \emph{reciprocal} of the inverse of the local Artin symbol. Let 
$$\mathcal{U} = \varprojlim_n \mathcal{O}_{L_{p,n}}^{\times}, \hspace{1cm} \mathcal{U}^1 = \varprojlim_n \mathcal{O}_{L_{p,n}}^{\times,1}$$
with transition maps given by norm, where $\mathcal{O}_{L_{p,n}}^{\times,1}$ is the principal part of $\mathcal{O}_{L_{p,n}}^{\times}$. 

\begin{choice}\label{fbasis}Fix an $f$-compatible sequence 
$$\alpha_n \in F^{\phi^{-n}}[\varpi^n]$$
of generators, i.e. each $\alpha_n$ generates $F^{\phi^{-n}}[\varpi^n]$ as an $\mathcal{O}_{K_p}$-module, and 
\begin{equation}\label{fbasiscondition}f^{\phi^{-n}}(\alpha_n) = \alpha_{n-1}.
\end{equation}
\end{choice}
Sometimes $\alpha = (\alpha_n)$ is called a generator the Tate module, see \cite[Theorem I.2.2]{deShalit} and the ensuing discussion. We prefer the terminology ``$f$-basis'' in our discussion, as we will also simultaneously consider bases of usual Tate modules. 

Fixing a formal group parameter $X$ on $F$, letting $\mathbf{\Gamma}(F)$ denote the coordinate ring of $F$, we have 
$$\mathbf{\Gamma}(F) \cong \mathcal{O}_{L_p}\llbracket X\rrbracket,$$
using Convention \ref{Gammaconvention}. We have a multiplicative norm operator (see \cite[Theorem I.2.1]{deShalit})
$$N_f : \mathbf{\Gamma}(F) \rightarrow \mathbf{\Gamma}(F),$$
which is uniquely characterized by 
\begin{equation}\label{normidentitycharacterize}(N_fg) \circ f = \prod_{\varpi \in F[f]}g(X[+]\varpi).
\end{equation}
Here, $\mathbf{\Gamma}(F)^{\times}$ denotes the units in $\mathbf{\Gamma}(F)$. Let $\mathbf{\Gamma}(F)^{\times,N_f = \phi}$ denote the subgroup of $\mathbf{\Gamma}(F)^{\times}$ consisting of elements $g$ such that $N_fg = g^{\phi}$. 

The following fundamental theorem is due to Coleman \cite{Colemanpowerseries}, and gives a functorial way to convert norm-compatible units into power series. The construction is often referred to as ``Coleman power series''. 

\begin{theorem}[Coleman Power Series, see Chapter I.1.2 of \cite{deShalit}]\label{Colemanpowerseries}Fix a choice of $\{\alpha_n\}_n$ as in Choice \ref{fbasis}. There is a $\mathrm{Gal}(L_{p,\infty}/K_p)$-equivariant isomorphism
$$\mathcal{U} \xrightarrow{\sim} \mathbf{\Gamma}(F)^{\times,N_f = \phi} \cong \mathcal{O}_{L_p}\llbracket X\rrbracket^{\times,N_f = \phi},$$
characterized by
$$\beta = (\beta_n) \mapsto g_{\beta}(X), \hspace{1cm} g_{\beta}^{\phi^{-n}}(\alpha_n) = \beta_n.$$
Moreover, $g_{\beta}$ satisfies the following properties:
\begin{enumerate}
\item $g_{\beta_1\beta_2} = g_{\beta_1}g_{\beta_2}$,
\item for any $\sigma \in \mathrm{Gal}(L_{p,\infty}/L_p)$, $g_{\beta^{\sigma}} = g_{\beta} \circ [\kappa(\sigma)]$, 
\item $N_fg_{\beta} = g_{\beta}^{\phi}$. 
\end{enumerate}
\end{theorem}

The restriction of the isomorphism $\beta \mapsto g_{\beta}$ of Theorem \ref{Colemanpowerseries} to $\mathcal{U}^1 \subset \mathcal{U}$ factors through
$$\mathcal{U}^1 \rightarrow 1+(\frak{p},X)\mathcal{O}_{L_p}\llbracket X\rrbracket.$$
Recall the power series
\begin{equation}\label{logpowerseries}\log(1+X) = X - \frac{X^2}{2} + \frac{X^3}{3} - \ldots \in L_p\llbracket X\rrbracket.
\end{equation}
Hence for $\beta \in \mathcal{U}^1$, we have a well-defined power series
$$\log(g_{\beta}) \in L_p\llbracket X\rrbracket.$$

\begin{definition}Given a power series $G \in L_p\llbracket X\rrbracket$, we define
$$\widetilde{G}(X) := G(X) - \frac{1}{p}\sum_{\varpi \in F[\varpi]}G(X[+]\varpi) \in L_p\llbracket X\rrbracket.$$
\end{definition}

\begin{proposition}\label{tildelogproposition}
$$\widetilde{\log} g_{\beta} \in \mathcal{O}_{L_p}\llbracket X\rrbracket \cdot \frac{1}{\varpi} \subset \mathcal{O}_{L_p}\llbracket X\rrbracket [1/p].$$
\end{proposition}

\begin{proof}We follow the argument of \cite[Lemma I.3.3]{deShalit}. Since $g_{\beta}^p \equiv g_{\beta}^{\phi} \circ f \pmod{\varpi'\mathcal{O}_{L_p}\llbracket X\rrbracket}$, we have
$$g_{\beta}^p \equiv g_{\beta}^{\phi} \circ f \overset{\text{Theorem \ref{Colemanpowerseries} (3)}}{=} N_fg_{\beta} \circ f \overset{(\ref{normidentitycharacterize})}{=}  \prod_{\varpi \in F[f]}g_{\beta}(X[+]\varpi) \pmod{\varpi'\mathcal{O}_{L_p}\llbracket X\rrbracket}.$$
Taking $\log$ of both sides, noting that $pn|p^n$ for all $n \ge 1$ we get
$$p \cdot \log g_{\beta} \equiv \sum_{\varpi \in F[f]}\log g_{\beta}(X[+]\varpi) \pmod{\varpi'\mathcal{O}_{L_p}\llbracket X\rrbracket}.$$
Dividing by $p$ and recalling that $(\varpi')^2\mathcal{O}_{L_p} = p\mathcal{O}_{L_p}$ since $p$ is ramified in $K_p/\mathbb{Q}_p$, we arrive at the assertion. 
\end{proof}

\begin{definition}\label{completedgroupalgebradefinition}Given any ring $R$ and group $G$, let $R[G]$ denote the group algebra of $G$ over $R$. Let 
$$R\llbracket G\rrbracket := \varprojlim_I R[G]/I$$
where $I \subset R[G]$ runs over all ideals of finite index. If $R$ is a $\mathbb{Z}_p$-algebra, we let 
$$R\llbracket G\rrbracket [1/p] := R\llbracket G\rrbracket \otimes_{\mathbb{Z}_p}\mathbb{Q}_p.$$ 
\end{definition}

Hence we get a $\mathbb{Z}_p\llbracket \mathrm{Gal}(L_{p,\infty}/K_p)\rrbracket$-module homomorphism
\begin{equation}\label{tildelogmap}\widetilde{\log} g : \mathcal{U}^1 \rightarrow \mathcal{O}_{L_p}\llbracket X\rrbracket [1/p], \hspace{1cm} \beta \mapsto \widetilde{\log} g_{\beta} \in \mathcal{O}_{L_p}\llbracket X\rrbracket [1/p].
\end{equation}

Let 
$$\log_F = \log_F(X) \in L_p\llbracket X\rrbracket$$
be the unique normalized formal logarithm of $F$ (see \cite[Chapter I.1.1]{deShalit}), i.e. with $\log_F'(0) = 1$. Let
\begin{equation}\label{wF}w_F := d\log_F(X) = \log_F'(X)dX
\end{equation}
be the normalized invariant differential. 

\begin{definition}Recall the character 
$$\kappa : \mathrm{Gal}(\overline{K}_p/L_p) \twoheadrightarrow \mathrm{Gal}(L_{p,\infty}/L_p) \xrightarrow{\sim}\mathcal{O}_{K_p}^{\times}$$
from (\ref{kappadefinition}). Let $\mathcal{O}_{K_p}(\kappa^{-1})$ denote the rank 1 free $\mathcal{O}_{K_p}$-module with $\mathrm{Gal}(\overline{K}_p/K_p)$-action given by $\kappa^{-1}$. Given any $\mathbb{Z}_p\llbracket \mathrm{Gal}(\overline{K}_p/K_p)\rrbracket$-module $M$, let 
$$M \otimes_{\mathbb{Z}_p}\mathcal{O}_{K_p}(\kappa^{-1}).$$
\end{definition}

We thus get an $\mathcal{O}_{K_p}\llbracket \mathrm{Gal}(L_{p,\infty}/K_p)\rrbracket$-module homomorphism
\begin{equation}\label{dtildelogmap}\frac{d}{w_F}\widetilde{\log} g: \mathcal{U}^1(\kappa^{-1}) \rightarrow \mathcal{O}_{L_p}\llbracket X\rrbracket [1/p], \hspace{1cm} \beta \mapsto \frac{d}{w_F}\widetilde{\log} g_{\beta}.
\end{equation}

\subsection{Thickening Coleman power series (after Tsuji)}\label{thickensection}
%Recall our previously chosen modular curve $Y = Y(N)$, where $N$ satisfies Choice \ref{globalHasseassumption} and $f|N$ where $f$ is the unique positive integer generating $\frak{f}^{(p)}$. Since $p \nmid N$, the supersingular locus $Y^{\mathrm{ss}}$ is a finite \'{e}tale cover of a finite union of height 2 Lubin-Tate deformation spaces $M$. 

Let 
$$F_0 := F \pmod{\varpi'\mathcal{O}_{L_p}}$$
denote the special fiber of $F$. We henceforth view $F_0$ as a height 2 formal group over $\overline{\mathbb{F}}_p$, and let
\begin{equation}\label{Wcompositum}W_{\mathcal{O}_{K_p}} := \mathcal{O}_{K_p}\cdot W(\overline{\mathbb{F}}_p)
\end{equation}
be the compositum of $W(\overline{\mathbb{F}}_p)$ and $\mathcal{O}_{K_p}$ in $\overline{\mathbb{Z}}_p$. Thus $W_{\mathcal{O}_{K_p}}$ is the $p$-adic completion of the maximal unramified extension of $\mathcal{O}_{K_p}$. Note that the residue field of $W_{\mathcal{O}_{K_p}}$ is $\overline{\mathbb{F}}_p$. Let $\mathrm{LT}$ denote the Lubin-Tate deformation space over $W_{\mathcal{O}_{K_p}}$ of $F_0/\overline{\mathbb{F}}_p$ (see \cite{LubinTate}), so that there is a non-canonical isomorphism 
$$\mathrm{LT} \cong \mathrm{Spf}(W_{\mathcal{O}_{K_p}}\llbracket T\rrbracket)$$
for some parameter $T$.  We will henceforth fix such an isomorphism, i.e. fix a parameter $T$ on $\mathrm{LT}$. Let
$$F^{\mathrm{univ}} \rightarrow \mathrm{LT}$$
denote the universal deformation of the height 2 formal group $F_0$. In particular, we have
\begin{equation}\label{universalcongruence}\left(F^{\mathrm{univ}} \pmod{(\varpi,T)W_{\mathcal{O}_{K_p}}\llbracket T\rrbracket}\right) \cong F_0, 
\end{equation}
where $\varpi$ is any uniformizer of $\mathcal{O}_{K_p}$. As explained in \cite[Section 1.1]{Tsuji}, $F^{\mathrm{univ}}$ has an $\mathcal{O}_{K_p}^{\times}$-action coming from the universal property, and $\mathrm{LT}$ also has an $\mathcal{O}_{K_p}^{\times}$-action that factors through $\mathcal{O}_{K_p}^{\times}/\mathbb{Z}_p^{\times}$ (i.e. $\mathbb{Z}_p^{\times}$ acts trivially on $\mathrm{LT}$). Moreover, these $\mathcal{O}_{K_p}^{\times}$-actions commute with the structure morphism $F \rightarrow \mathrm{LT}$. 

%Let $Y = Y(\Gamma(N))$ as before, which we view here as an adic space over $\mathrm{Spa}(\mathbb{Q}_p,\mathbb{Z}_p)$ with formal model $\hat{Y}^+$ over $\mathrm{Spf}(\mathbb{Z}_p)$. Recall the notation of Section \ref{formalShimurasection} as well as Convention \ref{Yconvention}. Recall that we have an embedding of formal schemes $\mathrm{LT} \subset \hat{Y}^+$ over $\mathrm{Spf}(\mathcal{O}_{L_p})$. 

%\begin{convention}\label{basechangeLT}Henceforth, we will consider the base change of $F^{\mathrm{univ}} \rightarrow \mathrm{LT}$ along $\mathrm{Spec}(W_{\mathcal{O}_{K_p}}) \rightarrow \mathrm{Spec}(W(\overline{\mathbb{F}}_p))$:
%$$F^{\mathrm{univ}} \times_{\mathrm{Spec}(W(\overline{\mathbb{F}}_p))} \mathrm{Spec}(W_{\mathcal{O}_{K_p}}) \rightarrow \mathrm{LT} \times_{\mathrm{Spec}(W(\overline{\mathbb{F}}_p))}\mathrm{Spec}(W_{\mathcal{O}_{K_p}}),$$
%which by a slight abuse of notation we will denote by 
%$$F^{\mathrm{univ}} \rightarrow \mathrm{LT}.$$
%No confusion should arise, as we will not work over the base $W(\overline{\mathbb{F}}_p)$ for the rest of our discussion. In particular, we note that 
%$$\mathcal{O}_{L_p} \subset W_{\mathcal{O}_{K_p}}.$$

%\end{convention}

%Adopting this Convention (\ref{universalcongruence}) then implies
%$$\left(F^{\mathrm{univ}} \pmod{(\varpi,T)W_{\mathcal{O}_{K_p}}\llbracket T\rrbracket}\right) \cong F_0.$$
By (\ref{universalcongruence}) and the results of \cite[Chapter III.3.6-3.8]{Katzpamf} applied to the height 2, 1-dimensional formal group $F^{\mathrm{univ}} \rightarrow \mathrm{LT}$, we have that $F^{\mathrm{univ}}$ has a canonical subgroup $C\subset F^{\mathrm{univ}}[p]$ of order $p$.
%Recall that there is an \'{e}tale map of formal schemes $\mathrm{LT} \rightarrow Y^+$ (recall $Y^+$ has level prime-to-$p$, see Convention \ref{Yconvention}). Let
%$$\mathrm{LT}(1/2) := \mathrm{LT} \times_{Y^+} Y^+(1/2),$$
%so that $\mathrm{LT}(1/2) \subset \mathrm{LT}$ is a subformal scheme defined over $\mathrm{Spf}(\mathcal{O}_{K_p})$. Since $F$ is a Lubin-Tate $\mathcal{O}_{K_p}$-module where $p$ ramifies in $K$, one can compute $|\mathrm{Ha}(F)| = p^{-1/2}$. Thus the $\mathrm{Spf}(\mathcal{O}_{L_p})$-point on $\mathrm{LT}$ corresponding to $F$ (i.e. the point $T = 0$) belongs to $\mathrm{LT}(1/2)$. Recall the lift of relative Frobenius $f : F \rightarrow F^{\phi}$. 
%Let 
%$$F^{\mathrm{univ}}(1/2) := F^{\mathrm{univ}} \times_\mathrm{LT} \mathrm{LT}(1/2).$$
%The universal canonical subgroup on $Y(1/2)$ restricts to a canonical subgroup for $F^{\mathrm{univ}}(1/2)$, which we call $C$ here. 
Note that $F^{\mathrm{univ}}/C \rightarrow \mathrm{LT}$ is a deformation of 
$$F_0^{\phi} := F_0 \times_{\mathrm{Spec}(\overline{\mathbb{F}}_p),\phi}\mathrm{Spec}(\overline{\mathbb{F}}_p).$$
Thus, since
$$F^{\mathrm{univ}} \times_{\mathrm{Spec}(W_{\mathcal{O}_{K_p}}),\phi}\mathrm{Spec}(W_{\mathcal{O}_{K_p}}) =: F^{\mathrm{univ},\phi}\rightarrow \mathrm{LT}^{\phi} := \mathrm{LT} \times_{\mathrm{Spec}(W_{\mathcal{O}_{K_p}}),\phi}\mathrm{Spec}(W_{\mathcal{O}_{K_p}})$$
is the universal deformation of $F_0^{\phi}$, the universal property of $F^{\mathrm{univ},\phi} \rightarrow \mathrm{LT}^{\phi}$ gives maps 
$$F^{\mathrm{univ}}/C \rightarrow F^{\mathrm{univ},\phi}, \hspace{1cm} \mathrm{LT} \rightarrow \mathrm{LT}^{\phi}$$
fitting into the following diagram:
\begin{equation}\label{isogenydiagramF}
\begin{tikzcd}[column sep = huge]
F^{\mathrm{univ}} \arrow{dr} \arrow{r} & F^{\mathrm{univ}}/C \arrow{d} \arrow{r} & F^{\mathrm{univ},\phi} \arrow{d}\\
 & \mathrm{LT} \arrow{r} & \mathrm{LT}^{\phi}
\end{tikzcd}
\end{equation}
where the right square is a pullback diagram. 
Let
$$f^{\mathrm{univ}} : F^{\mathrm{univ}} \rightarrow F^{\mathrm{univ},\phi}$$
denote the top row of (\ref{isogenydiagramF}). From this morphism of formal groups, we get an associated multiplicative norm operator (cf. \cite[cf. p. 114]{Tsuji}, \cite[Chapter I.2.1]{deShalit})
$$N_{f^{\mathrm{univ}}} : \mathbf{\Gamma}(F^{\mathrm{univ}}) \rightarrow \mathbf{\Gamma}(F^{\mathrm{univ}})$$
satisfying
$$N_{f^{\mathrm{univ}}}g \circ f^{\mathrm{univ}} = \prod_{\alpha \in C}\tau_{\alpha}^*g,$$
where $\tau_{\alpha}^* : \mathbf{\Gamma}(F^{\mathrm{univ}}) \rightarrow \mathbf{\Gamma}(F^{\mathrm{univ}})$ denotes translation by $\alpha$. Note that although such sections $\alpha$ exist only on a cover or $F^{\mathrm{univ}}$, the right-hand side expression descends to an element of $\mathbf{\Gamma}(F^{\mathrm{univ}})$.

\begin{theorem}[cf. Section 4 of \cite{Tsuji}]\label{thickenproposition}
\begin{enumerate}
\item Given a Coleman power series 
$$g_{\beta} \in \mathbf{\Gamma}(F)^{\times,N_f = \phi},$$
there is a unique element 
$$g_{\beta}^{\mathrm{univ}} \in \mathbf{\Gamma}(F^{\mathrm{univ}})^{\times,N_{f^{\mathrm{univ}}} = \phi}$$
specializing to $g_{\beta}$ under pullback by $F \rightarrow F^{\mathrm{univ}}$. Hence we get a $\mathrm{Gal}(L_{p,\infty}/K_p)$-equivariant isomorphism 
\begin{equation}\label{thick}\mathrm{thicken} : \mathbf{\Gamma}(F)^{\times,N_f = \phi}\hat{\otimes}_{\mathcal{O}_{L_p}}W_{\mathcal{O}_{K_p}} \xrightarrow{\sim}\mathbf{\Gamma}(F^{\mathrm{univ}})^{\times, N_{f^{\mathrm{univ}}} = \phi}
\end{equation}
where ``$\hat{\otimes}$'' denotes $p$-adically completed tensor product.
\item Combining this with the isomorphism of Theorem \ref{Colemanpowerseries}, we get a morphism
\begin{equation}\label{thick2}
\mathrm{thicken} \circ g : \mathcal{U}\xrightarrow{\sim} \mathbf{\Gamma}(F)^{\times,N_f = \phi} \rightarrow \mathbf{\Gamma}(F^{\mathrm{univ}})^{\times, N_{f^{\mathrm{univ}}} = \phi}.
\end{equation}
\item Moreover, choosing any formal parameter $X$ on the formal group $F^{\mathrm{univ}}$, the isomorphism (\ref{thick2}) factors through
\begin{equation}\label{principalthick2}\mathcal{U}^1 \rightarrow \left(1 + (\varpi,T,X)\mathbf{\Gamma}(F^{\mathrm{univ}})\right)^{N_{f^{\mathrm{univ}}} = \phi}.
\end{equation}
\end{enumerate}
\end{theorem}

\begin{proof}Taking the $m$-fold composition of $\phi^{-1}N_{f^{\mathrm{univ}}}$, one has the operator 
$$(\phi^{-1}N_{f^{\mathrm{univ}}})^m : \mathbf{\Gamma}(F^{\mathrm{univ}}) \rightarrow \mathbf{\Gamma}(F^{\mathrm{univ}})$$
for all $m \ge 0$. As in \cite[Section 4]{Tsuji}, one chooses an arbitrary lift $\tilde{g}_{\beta}$ of $g_{\beta}$ to $\mathbf{\Gamma}(F^{\mathrm{univ}})$ and shows using standard properties of the norm operator that 
$$(\phi^{-1}N_{f^{\mathrm{univ}}})^{\infty}(\tilde{g}_{\beta}) = \lim_{m \rightarrow \infty}(\phi^{-1}N_{f^{\mathrm{univ}}})^{m}(\tilde{g}_{\beta})$$
is a well-defined element of $\mathbf{\Gamma}(F^{\mathrm{univ}})^{\times,N_{f^{\mathrm{univ}}} = \phi}$ independent of the choice of $\tilde{g}_{\beta}$. 
\end{proof}

We also thicken the constructions (\ref{tildelogmap}) and (\ref{dtildelogmap}). 

\begin{definition}Given any $\beta \in \mathcal{U}^1$, we can apply the power series $\log$ from (\ref{logpowerseries}) to (\ref{principalthick2}) to get  
$$\log g \in \mathbf{\Gamma}(F^{\mathrm{univ}}_{\eta}),$$
where $F^{\mathrm{univ}}_{\eta}$ is the adic generic fiber of $F^{\mathrm{univ}}$. Similarly, given any $G \in \mathbf{\Gamma}(F^{\mathrm{univ}}_{\eta})$, we denote
\begin{equation}\label{tildedefinition}\tilde{G} = G - \frac{1}{p}\sum_{\alpha \in F^{\mathrm{univ}}[f^{\mathrm{univ}}]}\tau_{\alpha}^*G.
\end{equation}
\end{definition}

\begin{proposition}Given any $\beta \in \mathcal{U}^1$, we have 
$$\widetilde{\log} (\mathrm{thicken} \circ g_{\beta}) \in \mathbf{\Gamma}(F^{\mathrm{univ}})[1/p].$$
\end{proposition}

\begin{proof}This is the same argument as in the proof of Proposition \ref{tildelogproposition}, replacing $f$ by $f^{\mathrm{univ}}$. 
\end{proof}

\begin{convention}\label{tildelogconvention}For notational convenience, given a power series $G$ such that $\log(G)$ is well-defined, we will denote $\widetilde{\log(G)}$ by $\widetilde{\log}(G)$. 
\end{convention}

\begin{definition}
\begin{enumerate}
\item Define the $\mathbb{Z}_p\llbracket \mathrm{Gal}(L_{p,\infty}/K_p)\rrbracket$-module homomorphism
\begin{equation}\label{thickentildelogmap} \widetilde{\log} (\mathrm{thicken} \circ g) : \mathcal{U}^1 \rightarrow \mathbf{\Gamma}(F^{\mathrm{univ}})[1/p].
\end{equation}
%Let 
%\begin{equation}\label{wFuniv}w_{F^{\mathrm{univ}}} \in \Omega_{F^{\mathrm{univ}}/\mathrm{LT}}
%\end{equation}
%denote the unique normalized invariant differential on $F^{\mathrm{univ}}$ (which thus restricts to the unique invariant normalized differential on $F^{\mathrm{univ}}$). 

%Recall the section 

\item Using the $\mathbf{\Gamma}(\mathcal{O}_{\mathrm{LT}})[1/p]$-linear exterior derivative $d : \mathbf{\Gamma}(F^{\mathrm{univ}})[1/p] \rightarrow \mathbf{\Gamma}(\Omega_{F^{\mathrm{univ}}/\mathrm{LT}})[1/p]$, define the $\mathcal{O}_{K_p}\llbracket \mathrm{Gal}(L_{p,\infty}/K_p)\rrbracket$-module homomorphism
\begin{equation}\label{thickendtildelogmap} d\widetilde{\log} (\mathrm{thicken} \circ g) : \mathcal{U}^1(\kappa^{-1}) \rightarrow \mathbf{\Gamma}(\Omega_{F^{\mathrm{univ}}/\mathrm{LT}})[1/p].
\end{equation}
\end{enumerate}
\end{definition}

%Hence from Proposition \ref{thickenproposition} we get a $\mathrm{Gal}(L_{p,\infty}/K_p)$-equivariant map 
%\begin{equation}\label{thick}\mathrm{thicken} : \mathbf{\Gamma}(F)^{\times,N_f = \phi} \rightarrow \mathbf{\Gamma}(F^{\mathrm{univ}})^{N_{f^{\mathrm{univ}}} = \phi}.
%\end{equation}
We end this section with some notation that we will need later. Recall the characteristic $p$ Lubin-Tate group $F_0$ from the previous section. Recall $\mathrm{LT}$ is the Lubin-Tate deformation space over $W_{\mathcal{O}_{K_p}}$ classifying deformations of $F_0$. Given any height 2 formal group $F_0'$ over $\overline{\mathbb{F}}_p$, let $F^{\mathrm{univ}}(F_0') \rightarrow \mathrm{LT}(F_0')$ denote the associated universal deformation, where $\mathrm{LT}(F_0')$ is the Lubin-Tate deformation space defined over $W_{\mathcal{O}_{K_p}}$. 

\begin{convention}\label{infiniteleveladicgenericfiberconvention}We will let $\mathrm{LT}_{\infty}(F_0')$ denote the \emph{adic generic fiber} of the usual infinite-level Lubin-Tate deformation space of $F_0'$ classifying deformations of $F_0'$ together with $\Gamma(p^{\infty})$-Drinfeld level structure (see \cite{Fargues}). We will use this convention for $\mathrm{LT}_{\infty}(F_0')$ as we will predominantly work with the adic generic fiber and seldom refer to the formal model considered in op. cit. Thus we have a morphism of pro-adic spaces
$$\mathrm{LT}_{\infty}(F_0') \rightarrow \mathrm{Spa}(W_{\mathcal{O}_{K_p}}[1/p],W_{\mathcal{O}_{K_p}}),$$
recalling that $W_{\mathcal{O}_{K_p}}$ is as in (\ref{Wcompositum}).
\end{convention}

\subsection{Embedding Lubin-Tate towers in towers of Shimura curves}Let $Y$ be as in Convention \ref{Yconvention}. Recall that $Y_{\infty}^{\mathrm{ss}} \subset Y_{\infty}$ is the supersingular locus of $Y_{\infty}$. By the standard Rapoport-Zink uniformization $Y_{\infty}^{\mathrm{ss}}$ (see \cite[discussion after Theorem III.1.2]{ScholzeTorsion}), we can write \begin{equation}\label{YssLT}Y_{\infty}^{\mathrm{ss}} = \bigsqcup_{F_0'} \mathrm{LT}_{\infty}(F_0')
\end{equation}
where $F_0'$ ranges over finitely many copies of the finitely many isomorphism classes of reductions of formal groups of supersingular elliptic curves, and $\mathrm{LT}_{\infty}(F_0')$ for each such $F_0'$ appears with finite multiplicity in the above disjoint union. In particular, $Y_{\infty}^{\mathrm{ss}}$ is a finite disjoint union of infinite-level Lubin-Tate towers $\mathrm{LT}_{\infty}(F_0')$.

Let
$$F_{\infty}^{\mathrm{univ}}(F_0') \rightarrow \mathrm{LT}_{\infty}(F_0')$$
denote $F^{\mathrm{univ}}(F_0') \rightarrow \mathrm{LT}(F_0')$ pulled back along $\mathrm{LT}_{\infty}(F_0') \rightarrow \mathrm{LT}(F_0')$. We will mainly be concerned with $F_0' = F_0^{\phi^m}$, where $F_0^{\phi^m}$ denotes applying $\phi^m$ to the coefficients of $F_0$. %Given a formal scheme $X \rightarrow \mathrm{Spf}(\mathcal{O}_L,\mathcal{O}_L)$ for some $p$-adically complete subring $\mathcal{O}_L \subset \mathcal{O}_{\mathbb{C}_p}$, let $X^{\mathrm{ad}} \rightarrow \mathrm{Spa}(\mathcal{O}_L[1/p],\mathcal{O}_L)$ denote its adic generic fiber. 

\begin{definition}\label{esectiondefinitions}%Recall we have a map of formal schemes $\mathrm{LT}(F_0') \rightarrow Y^+(\epsilon_0)$ (recall the notation of Convention \ref{Yconvention}), and $\mathrm{LT}_{\infty}(F_0') \subset Y_{\infty}$ is an open subset. 
For any $m \in \mathbb{Z}$, let 
$$\mathrm{LT}^{\phi^m} := \mathrm{LT}(F_0^{\phi^m}),$$
%$$\mathrm{LT}(\epsilon_0)^{\phi^m} := \mathrm{LT}(F_0^{\phi^m}) \times_{Y^+} Y^+(\epsilon_0),$$
$$\mathrm{LT}_{\infty}^{\phi^m} := \mathrm{LT}_{\infty}(F_0^{\phi^m}),$$
%$$\mathrm{LT}_{\infty}(\epsilon_0)^{\phi^m} := \mathrm{LT}_{\infty}(F_0^{\phi^m}) \cap Y_{\infty}(\epsilon_0),$$  
$$F^{\mathrm{univ},\phi^m} := F^{\mathrm{univ}}(F_0^{\phi^m}),$$
and
$$F_{\infty}^{\mathrm{univ},\phi^m} := F_{\infty}^{\mathrm{univ}}(F_0^{\phi^m}).$$
%$$F_{\infty}^{\mathrm{univ}}(\epsilon_0)^{\phi^m} := F_{\infty}^{\mathrm{univ}}(F_0^{\phi^m}) \times_{\mathrm{LT}_{\infty}^{\phi^m}}\mathrm{LT}_{\infty}(\epsilon_0)^{\phi^m}.$$

%Recall $a$ from (\ref{adefinition}), and let
%\begin{equation}\label{esection}e_{1,a} : \mathrm{LT}_{\infty}^{\phi^{-2a+1}} \rightarrow F_{\infty}^{\mathrm{univ},\phi^{-2a+1}}
%\end{equation}
%denote the first piece $e_1$ of the universal level structure modulo $p^a$ (recall (\ref{eindefinition})); note this is equal to the restriction of the $e_{1,a}$ from the $\Gamma(p^{\infty})$-level structure $(e_1,e_2)$ on $Y_{\infty}$ to $\mathrm{LT}_{\infty}^{\phi^{-2a+1}} \subset Y_{\infty}^{\mathrm{ss}}$. 
\end{definition}

%Let $M_{\infty}$ denote the adic generic fiber of infinite-level Lubin-Tate space (classifying deformations with $\Gamma(p^{\infty})$-level structure), so that $M_{\infty} \in M_{\text{pro\'{e}t}}^{\mathrm{ad}}$. Then let 
%$$M_{\infty}(r) := M_{\infty} \times_{M^{\mathrm{ad}}} M(r)^{\mathrm{ad}}.$$  
%and let $F_{\infty}^{\mathrm{univ}}(r) \rightarrow M_{\infty}(r)$. Let
%$$e = e_{1,\beta} : M_{\infty}(r) \rightarrow F_{\infty}^{\mathrm{univ}}(r)$$
%denote the first piece $e_1$ of the universal level structure modulo $p^{\beta}$. 

\subsection{Semi-local Coleman map}\label{semilocalColemansection}We explain how to ``globalize'' the construction of Coleman power series by constructing a map from norm-compatible systems of semi-local units to power series, following \cite[Chapter II.4.5]{deShalit}.

\begin{definition}For any $\mathcal{O}_K$-ideal $\frak{n} \subset \mathcal{O}_K$, let $K(\frak{n})$ denote the ray class field over $K$ of modulus $\frak{n}$. Thus $\mathrm{Gal}(K(\frak{n})/K) \cong \mathcal{C}\ell(\frak{n})$ under the Artin reciprocity map.
Let 
$$R_n(\frak{n}) = \mathcal{O}_{K(\frak{n}\frak{p}^n)} \otimes_{\mathcal{O}_K}\mathcal{O}_{K_p} \cong \bigoplus_{\frak{P}|\frak{p}}\mathcal{O}_{K(\frak{n}\frak{f}^n),\frak{P}}$$
where $\frak{P}|\frak{p}$ runs over prime ideals of $\mathcal{O}_{K(\frak{n}\frak{p}^n)}$ above $\frak{p}$. We have norm maps $\mathrm{Nm}_n : R_n(\frak{n}) \rightarrow R_{n-1}(\frak{n})$. Let $R_n^{\times}(\frak{n})$ denote the group of units of $R_n(\frak{n})$. 
Let
\begin{equation}\label{semi-localunitsdefinition}\mathbb{U}(\frak{n}) := \varprojlim_{\mathrm{Nm}_n}R_n^{\times}(\frak{n}), \hspace{1cm} \mathbb{U}^1(\frak{n}) := \varprojlim_{\mathrm{Nm}_n}R_n^1(\frak{n})
\end{equation}
where $R_n^1(\frak{n})$ denotes the pro-$p$ part of $R_n^{\times}(\frak{n})$. In other words, $R_n^1(\frak{n})$ is the submodule of $R_n^{\times}(\frak{n})$ of semi-local principal units, i.e. elements whose projection to $\mathcal{O}_{K(\frak{n}\frak{f}^n),\frak{P}}^{\times}$ for each $\frak{P}|\frak{p}$ is principal.  %When $\frak{n} = 1$, we write 
%$$R_n(1) = R_n, \hspace{.25cm} R_n^{\times}(1) = R_n^{\times}, \hspace{.25cm} R_n^1(1) = R_n^1, \hspace{1cm} \mathbb{U} = \mathbb{U}(1), \hspace{1cm} \mathbb{U}^1 = \mathbb{U}^1(1)$$
%for brevity. 
\end{definition}

For any $\mathcal{O}_K$-ideal $\frak{n} \subset \mathcal{O}_K$, let
$$w_{\frak{n}} := \#\{x \in \mathcal{O}_K^{\times} : x \equiv 1 \pmod{\frak{n}}\}.$$

\begin{choice}\label{FixCMdefinition}
Let $\frak{f}_0 \subset \mathcal{O}_K$ be an integral ideal prime to $p$. Further suppose that 
$$w_{\frak{f}_0} = 1.$$
Henceforth fix an elliptic curve $A_0$ such that
\begin{enumerate}
\item $A_0$ has CM by $\mathcal{O}_K$,
\item $A_0$ is defined over $K(\frak{f}_0)$, 
\item $A_0^{\mathrm{tors}} \subset A_0(K^{\mathrm{ab}})$,
\item $A_0$ has good reduction at all primes of $K(\frak{f}_0)$ not dividing $\frak{f}_0$. 
\end{enumerate}
Such an elliptic curve exists under our assumptions on $\frak{f}_0$ by \cite[Lemma II.1.4 (i)]{deShalit}.
%As remarked in \cite[p. 4]{OukhabaViguie}, the existence of such an elliptic curve is proven in \cite[p. 216]{Shimura} and \cite[Chapter II, Lemma 1.4]{deShalit}. Let $\mathcal{O}_{K,(\frak{p})}$ denote the localization of $\mathcal{O}_K$ at $\frak{p}$ (i.e., inverting all elements outside of $\frak{p}$). Then let $R$ be the integral closure of $\mathcal{O}_{K,(\frak{p})}$ in $L$. We then fix a minimal generalized Weierstrass model of $E$ over $R$,
%$$y^2 + a_1xy + a_3y = x^3 + a_2x^2 + a_4x + a_6.$$
%We let 
%$$\omega_A = \frac{dx}{2y+a_1x+a_3} \in \Omega_{A/R}^1$$
%be the usual holomorphic invariant differential attached to the above Weierstrass model. The pair $(A,\omega_A)$ determines a unique $\mathcal{O}_K$-lattice with
%$$\theta_{\infty,A} : \mathbb{C}/L \xrightarrow{\sim} A(\mathbb{C}).$$
%Here, for all $z \in \mathbb{C}\setminus L$, $\theta_{\infty,A}(z)$ is the unique point with coordinates
%$$x(z) := \wp_L(z) - b_2/12, \hspace{1cm} y(z) := (\wp'(z) - a_1x(z) - a_3)/2$$
%where $b_2 := a_1^2 + 4a_2$.  
%Let $\hat{A}$ be the formal group of $A$, with parameter $t := -x/y$. We note that the function field of $E/\overline{\mathbb{Q}}$ is $\overline{\mathbb{Q}}(\wp_L,\wp_{L'})$, and we denote the natural map $\overline{\mathbb{Q}}(\wp_L,\wp_L') \rightarrow \mathbb{C}_p\llbracket t\rrbracket$ (taken with respect to the embedding $i_p$ fixed in (\ref{fixembeddings})) obtained by expanding a rational function on $E$ in terms of the formal parameter $t$ by 
%$$f \mapsto \hat{f}.$$
\end{choice}

We will often work with the following choice of relative Lubin-Tate group $F$.
\begin{example}\label{Fchoice}
Let $\hat{A}_0/\mathcal{O}_{K(\frak{f}_0)_p}$ denote the formal group of $A_0/\mathcal{O}_{K(\frak{f}_0)_p}$. Then 
$$F = \hat{A}_0$$
is a relative Lubin-Tate group for the unramified extension $L_p = K(\frak{f}_0)_p/K_p$ (see Example \ref{CMexample}). 
\end{example}

Thus in any case,
\begin{equation}\label{Ufdecomposition}\mathbb{U}(\frak{f}_0) = \bigoplus_{\frak{P}|\frak{p}}\mathcal{U}, \hspace{1cm} \mathbb{U}^1(\frak{f}_0) = \bigoplus_{\frak{P}|\frak{p}} \mathcal{U}^1
\end{equation}
where $\frak{P}|\frak{p}$ ranges over all primes of $\mathcal{O}_{K(\frak{f_0})}$ above $\frak{p}$. 

Let $R = \mathcal{O}_{K(\frak{f}_0)} \otimes_{\mathcal{O}_K}\mathcal{O}_{K,\frak{p}}$, and note that $R\llbracket T\rrbracket  = \bigoplus_{\frak{P}|\frak{p}}\mathcal{O}_{K(\frak{f}_0),\frak{P}}\llbracket T\rrbracket $. By applying Theorem \ref{Colemanpowerseries} to each factor of (\ref{Ufdecomposition}), we get a semi-local Coleman map
\begin{equation}\label{globalColemanpowerseries}\mathbb{U}(\frak{f}_0) \rightarrow R\llbracket X\rrbracket^{\times}, \hspace{1cm} \beta \mapsto g_{\beta},
\end{equation}
which factors through a map
$$\mathbb{U}^1(\frak{f}_0) \rightarrow 1 + (\frak{p},X)R\llbracket X\rrbracket.$$
Applying (\ref{tildelogmap}) to each factor of (\ref{Ufdecomposition}), we get a map
$$\widetilde{\log} g : \mathbb{U}^1(\frak{f}_0) \rightarrow R\llbracket X\rrbracket [1/p].$$
Our fixed embedding $\overline{\mathbb{Q}} \hookrightarrow \overline{\mathbb{Q}}_p$ from (\ref{fixembeddings}) specifies exactly one $\frak{P}|\frak{p}$, and so projecting onto the corresponding component we get an $\mathcal{O}_{K_p}\llbracket \mathrm{Gal}(K(\frak{f}_0p^{\infty})/K(\frak{f}_0))\rrbracket$-module homomorphism
\begin{equation}\label{tildelogmap2}\widetilde{\log} g : \mathbb{U}^1(\frak{f}_0) \rightarrow \mathcal{O}_{K(\frak{f}_0)_p}\llbracket X\rrbracket [1/p].
\end{equation}

Composing (\ref{thickentildelogmap}) with the projection $\mathbb{U}^1(\frak{f}_0) \rightarrow \mathcal{U}^1$ from (\ref{Ufdecomposition}) to the component $\mathcal{U}$ corresponding to the above distinguished $\frak{P}|\frak{p}$, we get an $\mathcal{O}_{K_p}\llbracket \mathrm{Gal}(K(\frak{f}_0p^{\infty})/K(\frak{f}_0))\rrbracket$-module homomorphism
\begin{equation}\label{globalthickentildelogmap}\widetilde{\log}\left(\mathrm{thicken}(g)\right) : \mathbb{U}^1(\frak{f}_0) \rightarrow \mathbf{\Gamma}(\Omega_{F^{\mathrm{univ}}/\mathrm{LT}})[1/p].
\end{equation}%Given $\beta \in \mathbb{U}^1(\frak{f}_0)$, we will continue to denote the corresponding Coleman power series by $g_{\beta}$. 

\begin{definition}\label{varphidefinition}Retain the setting and notation of Choice \ref{FixCMdefinition}. 
\begin{enumerate}
\item By \cite[Lemma II.1.4 (ii)]{deShalit} there exists a Hecke character 
$$\lambda_{A_0} : K^{\times}\backslash \mathbb{A}_K^{\times} \rightarrow \mathbb{C}^{\times}$$
of infinity type $(1,0)$ and conductor $\frak{f}_0$ such that
$$\psi_{A_0/K(\frak{f}_0)} = \lambda_{A_0} \circ \mathrm{Nm}_{K(\frak{f}_0)/K}$$
where $\psi_{A_0/K(\frak{f}_0)} : K(\frak{f}_0)^{\times}\backslash\mathbb{A}_{K(\frak{f}_0)}^{\times} \rightarrow \mathbb{C}^{\times}$ is the Hecke character associated with $A_0/K(\frak{f}_0)$ by the theory of complex multiplication (see \cite[Chapter II.1.4]{deShalit}). 
\item Moreover, since $\frak{f}_0 = \frak{f}(\lambda_{A_0})$, by \cite[Proposition II.1.6]{deShalit} we have 
\begin{equation}\label{raytorsionequal}K(\frak{f}_0\frak{p}^n) = K(\frak{f}_0)(A_0[\frak{f}_0\frak{p}^n]).
\end{equation}

\item The $p$-adic avatar of $\lambda_{A_0}$ (which we continue to denote by $\lambda_{A_0}$, following Convention \ref{avatarconvention}) can be viewed, via Artin reciprocity, as a character
$$\lambda_{A_0} : \mathrm{Gal}(K(\frak{f}_0)(A_0[\frak{f}_0p^{\infty}])/K(\frak{f}_0)) = \mathrm{Gal}(K(\frak{f}_0p^{\infty})/K) \rightarrow \mathcal{O}'^{\times}$$
for some finite extension $\mathcal{O}'/\mathcal{O}_{K_p}$, which restricts to an isomorphism
$$\lambda_{A_0}|_{\mathrm{Gal}(K(\frak{f}_0p^{\infty})/K(\frak{f}_0))} : \mathrm{Gal}(K(\frak{f}_0p^{\infty})/K(\frak{f}_0)) \xrightarrow{\sim} \mathcal{O}_{K_p}^{\times}.$$
\item For the rest of the this definition, let $L = K(\frak{f}_0)$ and $F = \hat{A}_0$. Under the natural equality 
$$\mathrm{Gal}(K(\frak{f}_0p^{\infty})/K(\frak{f}_0)) = \mathrm{Gal}(L_{p,\infty}/L_p)$$
induced by (\ref{fixembeddings}) and using the fact that $K(\frak{f}_0p^{\infty}) = K(\frak{f}_0)(A[\frak{f}_0p^{\infty}])/K(\frak{f}_0)$ is totally ramified at all places above $p$, we have 
\begin{equation}\label{kappavarphiequal}\lambda_{A_0}|_{\mathrm{Gal}(K(\frak{f}_0p^{\infty})/K(\frak{f}_0))} = \kappa
\end{equation}
from (\ref{kappadefinition}).
\end{enumerate}
\end{definition}

%Let $\mathcal{O}_{K_p}(\lambda_{A_0}_0^{-1})$ denote the rank 1 free $\mathcal{O}_{K_p}$-module with $\mathrm{Gal}(K(\frak{f}_0p^{\infty})/K)$-action given by $\lambda_{A_0}_0^{-1}$. Define an $\mathcal{O}_{K_p}\llbracket \mathrm{Gal}(K(\frak{f}_0p^{\infty})/K)\rrbracket$-module
%$$\mathbb{U}^1(\frak{f}_0)(\lambda_{A_0}_0^{-1}) = \mathbb{U}^1(\frak{f}_0) \otimes_{\mathbb{Z}_p}\mathcal{O}_{K_p}(\lambda_{A_0}_0^{-1}) \overset{(\ref{kappavarphiequal})} = \mathbb{U}^1.$$
Let 
$$\mathbb{U}^1(\frak{f}_0) \rightarrow \mathcal{U}^1$$ be the projection onto the component of (\ref{Udecomposition}) corresponding to the $\frak{P}|\frak{p}$ induced by (\ref{fixembeddings}). Tensoring this projection with $\otimes_{\mathbb{Z}_p}\mathcal{O}_{K_p}(\kappa^{-1})$, we get a projection
$$\mathbb{U}^1(\frak{f}_0)(\kappa^{-1}) \rightarrow \mathcal{U}^1(\kappa^{-1}).$$
Precomposing (\ref{thickendtildelogmap}) with this projection, we get an $\mathcal{O}_{K_p}\llbracket \mathrm{Gal}(K(\frak{f}_0p^{\infty})/K(\frak{f}_0))\rrbracket$-module homomorphism
\begin{equation}\label{globalthickendtildelogmap}d\widetilde{\log}\left(\mathrm{thicken}(g)\right) : \mathbb{U}^1(\frak{f}_0)(\kappa^{-1}) \rightarrow \mathbf{\Gamma}(\Omega_{F^{\mathrm{univ}}/\mathrm{LT}})[1/p].
\end{equation}
 
\subsection{Functoriality of norm-compatible systems of units}

Let $\frak{g}|\frak{h}$ be integral $\mathcal{O}_K$-ideals, which are not necessarily prime to $p$ and with no assumptions on $w_{\frak{g}}, w_{\frak{h}}$. Clearly, we have $K(\frak{g}\frak{p}^n)\subset K(\frak{h}\frak{p}^n)$ and $R_n(\frak{g}) \subset R_n(\frak{h})$ for all $n \ge 0$. It is clear from Artin reciprocity that for all $n \gg 0$, the degree $[K(\frak{h}\frak{p}^n):K(\frak{g}\frak{p}^n)]$ stabilizes, and $[K(\frak{f}_0\frak{p}^n):K(\frak{f}_0\frak{p}^{n-1})] = p$ and $[K(\frak{g}_0\frak{p}^n):K(\frak{g}_0\frak{p}^{n-1})] = p$ for all $n \gg 0$. 
Thus we get commutative diagrams for all $n \gg 1$
 \begin{equation}\label{Unormdiagram1}
\begin{tikzcd}[column sep = huge]
R_n^{\times}(\frak{g}) \arrow{d}{\mathrm{Nm}_n} \arrow[hook]{r} & R_n^{\times}(\frak{h}) \arrow{d}{\mathrm{Nm}_n}\\
  R_{n-1}^{\times}(\frak{g}) \arrow[hook]{r} & R_{n-1}^{\times}(\frak{h})\\
   \end{tikzcd}, \hspace{1cm} 
   \begin{tikzcd}[column sep = huge]
R_n^1(\frak{g}) \arrow{d}{\mathrm{Nm}_n} \arrow[hook]{r} & R_n^1(\frak{h}) \arrow{d}{\mathrm{Nm}_n}\\
  R_{n-1}^1(\frak{g}) \arrow[hook]{r} & R_{n-1}^1(\frak{h})\\
   \end{tikzcd},
\end{equation}
and commutative diagrams for all $n \ge 1$
\begin{equation}\label{Unormdiagram2}
\begin{tikzcd}[column sep = huge]
R_n^{\times}(\frak{h}) \arrow{d}{\mathrm{Nm}_n} \arrow{rr}{\mathrm{Nm}_{K(\frak{h}\frak{p}^n)/K(\frak{g}\frak{p}^n)}} & & R_n^{\times}(\frak{g}) \arrow{d}{\mathrm{Nm}_n}\\
  R_{n-1}^{\times}(\frak{h}) \arrow{rr}{\mathrm{Nm}_{K(\frak{h}\frak{p}^{n-1})/K(\frak{g}\frak{p}^{n-1})}} & & R_{n-1}^{\times}(\frak{g})\\
   \end{tikzcd}, \hspace{1cm} 
   \begin{tikzcd}[column sep = huge]
R_n^1(\frak{h}) \arrow{d}{\mathrm{Nm}_n} \arrow{rr}{\mathrm{Nm}_{K(\frak{h}\frak{p}^n)/K(\frak{g}\frak{p}^n)}} && R_n^1(\frak{g}) \arrow{d}{\mathrm{Nm}_n}\\
  R_{n-1}^1(\frak{h}) \arrow{rr}{\mathrm{Nm}_{K(\frak{h}\frak{p}^{n-1})/K(\frak{g}\frak{p}^{n-1})}} && R_{n-1}^1(\frak{g})\\
   \end{tikzcd}.
\end{equation}
Thus, from (\ref{Unormdiagram1}) we get natural inclusions
\begin{equation}\label{Uinclusion}i_{\frak{g},\frak{h}} : \mathbb{U}(\frak{g}) \hookrightarrow \mathbb{U}(\frak{h}), \hspace{1cm} i_{\frak{g},\frak{h}} : \mathbb{U}^1(\frak{g}) \hookrightarrow \mathbb{U}^1(\frak{h}),
\end{equation}
and from (\ref{Unormdiagram2}) we get maps
\begin{equation}\label{Unorm}\mathrm{Nm}_{\frak{h}/\frak{g}} : \mathbb{U}(\frak{h}) \rightarrow \mathbb{U}(\frak{g}), \hspace{1cm} \mathrm{Nm}_{\frak{h}/\frak{g}} : \mathbb{U}^1(\frak{h}) \rightarrow \mathbb{U}^1(\frak{g}).
\end{equation}
%The composition $\mathrm{Nm}_{\frak{h}/\frak{g}} \circ i_{\frak{g},\frak{h}} : \mathbb{U}^1(\frak{g}) \rightarrow \mathbb{U}^1(\frak{g})$ is equal to multiplication by a constant (where here we are referring to additive notation: ``multiplication'' on units corresponds to exponentiation, and ``addition'' corresponds to multiplication).  %Thus, we get an isomorphism
%$$\mathbb{U}^1(\frak{h}) \otimes_{\mathbb{Z}_p}\mathbb{Q}_p \cong \mathbb{U}^1(\frak{f}_0) \otimes_{\mathbb{Z}_p}\mathbb{Q}_p.$$

\subsection{Galois extensions associated to CM elliptic curves}\label{Galoisextensionsection}

\begin{choice}\label{Echoice}Assume from now on that $K$ has class number 1, and let $E/\mathbb{Q}$ be a fixed elliptic curve with CM by $\mathcal{O}_K$. We continue to assume that $p$ is a finite prime ramified in $K/\mathbb{Q}$; in fact, there is a unique such prime since $K$ is one of the fields in Assumption \ref{pramifiedassumption}.
\end{choice}

\begin{assumption}\label{lambdalambdaEassumption}Let $\lambda_E : \mathbb{A}_K^{\times}/K^{\times} \rightarrow \mathbb{C}^{\times}$ be the type (1,0) Hecke character associated with $E$ by the theory of complex multiplication (\cite[Chapter II.1.4]{deShalit}), and let
$$\frak{f}_E = \frak{f}(\lambda_E).$$
As explained in Example \ref{Eexample}, $\frak{f}_E^{(p)} = f_0\mathcal{O}_K$ for some $f_0 \in \mathbb{Z}_{> 0}$. Assume henceforth that $f_0 \ge 4$ so that $\lambda_E$ satisfies Assumption \ref{pconductorassumption}. 
\end{assumption}

Note the assumption $f_0 \ge 4$ is satisfied for all but finitely many isomorphism classes of $E/\mathbb{Q}$.

\begin{definition}\label{mathcalKndefinition}For any $n \in \mathbb{Z}_{\ge 0} \cup \{\infty\}$, let 
$$\mathcal{K}_n := K(E[\frak{p}^n]).$$
 %We have the module of norm-compatible systems of principal local units
%$$\mathbb{U}_E = \varprojlim_n \mathcal{O}_{\mathcal{K}_n}^{\times}, \hspace{1cm} \mathbb{U}_E^1 = \varprojlim_n \mathcal{O}_{\mathcal{K}_n}^{\times,1}.$$
\end{definition}
Using Artin reciprocity, we can view the $p$-adic avatar of $\lambda_E$ (which we continue to denote by $\lambda_E$, following Convention \ref{avatarconvention}) as a character
$$\lambda_E : \mathrm{Gal}(K(\frak{f}_0p^{\infty})/K) \rightarrow \mathcal{O}_{K_p}^{\times}.$$
The fact that the image of $\lambda_E$ is contained in $\mathcal{O}_{K_p}^{\times}$ follows from \cite[II.1.5 (15)]{deShalit} (see also Proposition \ref{lambdafactorproposition} below). Recall the notation from (\ref{edefinition})
$$e := \mathrm{ord}_{\frak{p}}(\frak{f}_E).$$
Then $K(\frak{f}_0\frak{p}^n) = K(E[\frak{f}_0\frak{p}^n])$ for all $n \ge e$ by Proposition II.1.6 of op. cit. Thus
\begin{equation}\label{Kninclusion}\mathcal{K}_n = K(E[\frak{p}^n]) \subset K(E[\frak{f}_0\frak{p}^n]) = K(\frak{f}_0\frak{p}^n)
\end{equation}
for all $n \ge e$.

\begin{proposition}\label{lambdafactorproposition}The character $\lambda_E : \mathrm{Gal}(K(\frak{f}_0p^{\infty})/K) \rightarrow \mathcal{O}_{K_p}^{\times}$ factors through the restriction $\mathrm{Gal}(K(\frak{f}_0p^{\infty})/K) \twoheadrightarrow \mathrm{Gal}(\mathcal{K}_{\infty}/K)$ to induce a character
\begin{equation}\label{lambdaE}\lambda_E : \mathrm{Gal}(\mathcal{K}_{\infty}/K) \rightarrow \mathcal{O}_{K_p}^{\times}.
\end{equation}
\end{proposition}

\begin{proof}Given any ideal $\frak{b} \subset \mathcal{O}_K$ with $(\frak{b},\frak{f}_E) = 1$, let $\sigma_{\frak{b}} = (\cdot,K(\frak{f}_0p^{\infty})/K) \in \mathrm{Gal}(K(\frak{f}_0p^{\infty})/K)$ denote the Artin symbol. By \cite[II.1.5 (11)]{deShalit} with $F = K$ and $\frak{f} = \frak{f}_E$, we have  
$$\sigma_{\frak{b}}(u) = [\lambda_E(\sigma_{\frak{b}})]_E(u) \hspace{1cm} \forall u \in E[p^{\infty}]$$
where $[\cdot]_E : \mathcal{O}_K \xrightarrow{\sim} \mathrm{End}(E/K)$ is the CM action. However, the left-hand side is equal to $\sigma_{\frak{b}}|_{\mathcal{K}_{\infty}}(u)$, where $|_{\mathcal{K}_{\infty}}$ denotes restriction to $\mathcal{K}_{\infty} = K(E[p^{\infty}]) \subset K(\frak{f}_0p^{\infty})$. This shows that $\lambda_E(\sigma_{\frak{b}})$ only depends on $\sigma_{\frak{b}}|_{\mathcal{K}_{\infty}}$, which gives the assertion.
\end{proof}

We now define modules of norm-compatible systems of semi-local units attached to $E$ that will be of utmost importance to our discussion.
\begin{definition}\label{UEdefinition}Let
$$\mathbb{U} = \varprojlim_n (\mathcal{O}_{\mathcal{K}_n}\otimes_{\mathbb{Z}}\mathbb{Z}_p)^{\times}, \hspace{1cm} \mathbb{U}^1 = \varprojlim_n (\mathcal{O}_{\mathcal{K}_n}\otimes_{\mathbb{Z}}\mathbb{Z}_p)^{\times,1},$$
where $(\mathcal{O}_{\mathcal{K}_n}\otimes_{\mathbb{Z}}\mathbb{Z}_p)^{\times,1}$ denotes the pro-$p$ part of $(\mathcal{O}_{\mathcal{K}_n}\otimes_{\mathbb{Z}}\mathbb{Z}_p)^{\times}$.
\end{definition}

\begin{definition}Let $K_{\infty}/K$ be the unique maximal $\mathbb{Z}_p^{\oplus 2}$-extension of $K$. Define
\begin{equation}\label{GammaK}\Gamma_K := \mathrm{Gal}(K_{\infty}/K) \cong \mathbb{Z}_p^{\oplus 2}.
\end{equation}
Let
$$\Gamma_{K,n} = \Gamma_K/\Gamma_K^{p^n}.$$
Let $K_n/K$ be the unique subextension of $K_{\infty}/K$ with
$$\mathrm{Gal}(K_n/K) \cong (\mathbb{Z}/p^n)^{\oplus 2}.$$
In other words, $K_n$ is the fixed field of $\Gamma_{K,n}$. 
\end{definition}

\begin{proposition}\label{GammaKtotallyramified}\
\begin{enumerate}
\item Assume that $M$ is any totally imaginary number field with class number prime to $p$. Then letting $M_{\infty}/M$ be the compositum of all $\mathbb{Z}_p$-extensions of $M$, we have that $M_{\infty}/M$ is totally ramified at all primes above $p$ and totally unramified at all finite primes not above $p$. 
\item In particular, for $M$ an imaginary quadratic field with class number prime to $p$, $M_{\infty}/M$ is totally ramified at all primes above $p$ and totally unramified at all primes not above $p$. Thus $\Gamma_M = \mathrm{Gal}(M_{\infty}/M)$ is equal to its decomposition subgroup at any prime of $\mathcal{O}_{M_{\infty}}$ above $p$.
\end{enumerate}
\end{proposition}

\begin{proof}It is clear that (2) immediately follows from (1). We now prove (1). Let $M \subset L \subset M_{\infty}$ be the maximal subextension totally unramified at all places above $p$; in particular, $[L:M] = p^n$ for some $n \in \mathbb{Z}_{\ge 0}$. Since $\mathrm{Gal}(M_{\infty}/M) \cong \mathbb{Z}_p^{\oplus 2}$, and $\mathbb{Z}_p$-extensions are totally unramified at all finite primes not dividing $p$ (\cite[Proposition 13.2]{Washington}), and $L/M$ is totally unramified at all archimedean places since $M$ is totally imaginary, we thus have that $L/M$ is totally unramified at all places. Thus $L$ is contained in the Hilbert class field of $M$ which by our  assumption has degree over $M$ prime to $p$. Thus $L = M$. 

\end{proof}

\begin{proposition}The composition 
$$\mathrm{Gal}(K_{\infty}K(\frak{f}_0)/K(\frak{f}_0)) \twoheadrightarrow \mathrm{Gal}(K_{\infty}/(K(\frak{f}_0)\cap K_{\infty})) \subset \Gamma_K$$
is an isomorphism. Thus, it induces an identification
\begin{equation}\label{Kinftyrestriction}\mathrm{Gal}(K_{\infty}K(\frak{f}_0)/K(\frak{f}_0)) = \Gamma_K.
\end{equation}
%Moreover, $K_{\infty}K(\frak{f}_0)/K(\frak{f}_0)$ is the maximal $\mathbb{Z}_p^{\oplus 2}$-subextension of $K(\frak{f}_0p^{\infty})/K(\frak{f}_0)$. 
\end{proposition}

\begin{proof}For the first assertion, we need to show that $K(\frak{f}_0) \cap K_{\infty} = K$. First, since $\Gamma_K = \mathrm{Gal}(K_{\infty}/K) \cong \mathbb{Z}_p^{\oplus 2}$, the image of $\mathcal{O}_{K_v}^{\times}$ under the Artin reciprocity map $\mathbb{A}_K^{\times} \twoheadrightarrow \Gamma_K$ is trivial for all $v \nmid p$. Thus $K_{\infty}/K$ is unramified outside $p$, and so since $(\frak{f}_0,p) = 1$, we have $K(\frak{f}_0)\cap K_{\infty} = K(1) = K$, where the last equality follows from our assumption that $K$ has class number 1. 

%For the second assertion, let $K_{\infty}'/K(\frak{f}_0)$ be any $\mathbb{Z}_p^{\oplus 2}$-subextension of $K(\frak{f}_0p^{\infty})/K(\frak{f}_0)$. Then $\mathrm{Gal}(K_{\infty}'/K(\frak{f}_0))$
\end{proof}

Recall that, by reciprocity (since $w_{\frak{f}_0} = 1$),
\begin{equation}\label{reciprocityOKp}\mathrm{Gal}(K(\frak{f}_0p^{\infty})/K(\frak{f}_0)) \cong \mathcal{O}_{K_p}^{\times} \cong \mathcal{O}_{K_p}^{\times,\mathrm{tors}} \times \mathbb{Z}_p^{\oplus 2}
\end{equation}
where $\mathcal{O}_{K_p}^{\times,\mathrm{tors}}\subset \mathcal{O}_{K_p}^{\times}$ is the torsion subgroup. The subgroup $$\mathrm{Gal}(K(\frak{f}_0p^{\infty})/K_{\infty}K(\frak{f}_0)) \subset \mathrm{Gal}(K(\frak{f}_0p^{\infty})/K(\frak{f}_0))$$
maps to the trivial element through the surjection 
$$\mathrm{Gal}(K(\frak{f}_0p^{\infty})/K(\frak{f}_0)) \twoheadrightarrow \mathrm{Gal}(K_{\infty}K(\frak{f}_0)/K(\frak{f}_0)) \overset{(\ref{Kinftyrestriction})}{=} \Gamma_K \cong \mathbb{Z}_p^{\oplus 2}.$$
Thus
\begin{equation}\label{reciprocityOKp2}\mathrm{Gal}(K(\frak{f}_0p^{\infty})/K_{\infty}K(\frak{f}_0)) \subset \mathcal{O}_{K_p}^{\times,\mathrm{tors}}
\end{equation}
under (\ref{reciprocityOKp}). On the other hand, we have an exact sequence
$$1 \rightarrow \mathrm{Gal}(K(\frak{f}_0p^{\infty})/K(p^{\infty})K(\frak{f}_0)) \rightarrow \mathrm{Gal}(K(\frak{f}_0p^{\infty})/K_{\infty}K(\frak{f}_0)) \rightarrow \mathrm{Gal}(K(p^{\infty})K(\frak{f}_0)/K_{\infty}K(\frak{f}_0)) \rightarrow 1.$$
By Artin reciprocity we have 
$$\mathrm{Gal}(K(\frak{f}_0p^{\infty})/K(p^{\infty})K(\frak{f}_0)) \cong \mathcal{O}_{K_p}^{\times}, \hspace{1cm}\mathrm{Gal}(K(p^{\infty})K(\frak{f}_0)/K_{\infty}K(\frak{f}_0)) \cong \mathcal{O}_{K_p}^{\times,\mathrm{tors}}/\mathcal{O}_K^{\times}.$$
Thus (\ref{reciprocityOKp2}) must in fact be an isomorphism. This implies that there exists a decomposition
\begin{equation}\label{existdecomposition}\mathrm{Gal}(K(\frak{f}_0p^{\infty})/K(\frak{f}_0)) \cong \mathrm{Gal}(K(\frak{f}_0p^{\infty})/K_{\infty}K(\frak{f}_0)) \times \mathrm{Gal}(K_{\infty}K(\frak{f}_0)/K(\frak{f}_0)).
\end{equation}

\begin{definition}\label{decompositionchoices}
\begin{enumerate}
\item Let
\begin{equation}\label{DeltaA}\Delta_{A_0} := \mathrm{Gal}(K(\frak{f}_0p^{\infty})/K_{\infty}). 
\end{equation}
\item Using the existence of (\ref{existdecomposition}), choose and fix forever identifications
\begin{equation}\label{fixA}\begin{split}
%\mathrm{Gal}(K(\frak{f}_0p^{\infty})/K) &= \mathrm{Gal}(K(\frak{f}_0p^{\infty})/K(\frak{f}_0)) \times \mathrm{Gal}(K(\frak{f}_0)/K),\\
\mathrm{Gal}(K(\frak{f}_0p^{\infty})/K(\frak{f}_0)) &= \mathrm{Gal}(K(\frak{f}_0p^{\infty})/K_{\infty}K(\frak{f}_0)) \times \mathrm{Gal}(K_{\infty}K(\frak{f}_0)/K(\frak{f}_0)) \\
&\overset{(\ref{Kinftyrestriction})}{=} \mathrm{Gal}(K(\frak{f}_0p^{\infty})/K_{\infty}K(\frak{f}_0)) \times \Gamma_K.
\end{split}
\end{equation}
\item Let $\Gamma' \subset \mathrm{Gal}(K(\frak{f}_0p^{\infty})/K)$ be the image of 
$$\Gamma_K \overset{x \mapsto (1,x)}{\hookrightarrow} \mathrm{Gal}(K(\frak{f}_0p^{\infty})/K_{\infty}K(\frak{f}_0)) \times\Gamma_K \overset{(\ref{fixA})}{=} \mathrm{Gal}(K(\frak{f}_0p^{\infty})/K(\frak{f}_0)) \subset \mathrm{Gal}(K(\frak{f}_0p^{\infty})/K).$$
Since $\Gamma' \cong \mathbb{Z}_p^{\oplus 2}$ is torsion free, we have $\Delta_{A_0} \cap \Gamma' = \{1\}$ and so we have a natural inclusion 
$$\Delta_{A_0} \times\Gamma' \subset \mathrm{Gal}(K(\frak{f}_0p^{\infty})/K).$$
We claim that this inclusion is in fact an equality. To see this, recall the natural surjection 
$$\mathrm{Gal}(K(\frak{f}_0p^{\infty})/K) \twoheadrightarrow \mathrm{Gal}(K_{\infty}/K) = \Gamma_K$$
induced by restriction from $K(\frak{f}_0p^{\infty})$ to $K_{\infty}$, whose kernel is $\Delta_{A_0}$. By restriction to $\Gamma'$ and using the fact that $\Delta_{A_0} \cap \Gamma' = \{1\}$, we get a map
$$\Gamma' \subset \mathrm{Gal}(K(\frak{f}_0p^{\infty})/K)/\Delta_{A_0} \xrightarrow{\sim} \Gamma_K.$$
By (\ref{Kinftyrestriction}),  the above map is surjective. This gives 
$$\Delta_{A_0} \times\Gamma' = \mathrm{Gal}(K(\frak{f}_0p^{\infty})/K).$$ 
Identifying $\Gamma' \overset{(\ref{Kinftyrestriction})}{=} \Gamma_K$, we thus get a decomposition
\begin{equation}\label{fixA2}
\mathrm{Gal}(K(\frak{f}_0p^{\infty})/K) = \Delta_{A_0} \times \Gamma_K.
\end{equation}
By construction, (\ref{fixA}) and (\ref{fixA2}) are compatible under the inclusion $\mathrm{Gal}(K(\frak{f}_0p^{\infty})/K(\frak{f}_0)) \subset \mathrm{Gal}(K(\frak{f}_0p^{\infty})/K)$.

%these decompositions satisfy the following compatibility: we require that the restriction of
%$$\mathrm{Gal}(K(\frak{f}_0p^{\infty})/K_{\infty}K(\frak{f}_0)) \times \Gamma_K = \mathrm{Gal}(K(\frak{f}_0p^{\infty})/K(\frak{f}_0)) \subset \mathrm{Gal}(K(\frak{f}_0p^{\infty})/K) = \Delta_{A_0} \times \Gamma_K$$
%to $\Gamma_K \subset \Delta_{A_0} \times \Gamma_K$ factors through
%$$\Gamma_K \rightarrow \Gamma_K \subset \Delta_{A_0} \times \Gamma_K$$
%where $\Gamma_K \rightarrow \Gamma_K$ is the identity map and $\Gamma_K \subset \Delta_{A_0} \times \Gamma_K$ is inclusion into the second factor. This is possible by (\ref{Kinftyrestriction}).

\item Viewing $\lambda_{A_0}$ from Definition \ref{varphidefinition} as a character $\lambda_{A_0} : \mathrm{Gal}(K(\frak{f}_0p^{\infty})/K) \rightarrow \overline{\mathbb{Q}}_p^{\times}$ via Artin reciprocity, and using (\ref{fixA2}) to view $\Delta_{A_0} \subset \mathrm{Gal}(K(\frak{f}_0p^{\infty})/K)$, we let
\begin{equation}\label{chiA}\chi_{A_0} := \lambda_{A_0}|_{\Delta_{A_0}}.
\end{equation}

\item Let
\begin{equation}\label{DeltaK}\Delta_K := \mathrm{Gal}(\mathcal{K}_{\infty}/K_{\infty}).
\end{equation}
Applying the natural projection $\mathrm{Gal}(K(\frak{f}_0p^{\infty})/K_{\infty}) \twoheadrightarrow \mathrm{Gal}(\mathcal{K}_{\infty}/K)$ to (\ref{fixA2}), we get a decomposition
\begin{equation}\label{fixK}\mathrm{Gal}(\mathcal{K}_{\infty}/K) = \Delta_K \times \Gamma_K.
\end{equation} 

%Note that since $K$ has class number 1, we have a natural identification through restriction $\mathrm{res}_{\mathcal{K}_{\infty}/K_{\infty}} : \Gamma_K \xrightarrow{\sim} \mathrm{Gal}(K_{\infty}/K)$, where $K_{\infty}/K$ is the $\mathbb{Z}_p^{\oplus 2}$-extension. 
\item Viewing $\lambda_E$ as a character $\lambda_E : \mathrm{Gal}(\mathcal{K}_{\infty}/K) \rightarrow \overline{\mathbb{Q}}_p^{\times}$ via Artin reciprocity, and using (\ref{fixK}) to view $\Delta_K \subset \mathrm{Gal}(\mathcal{K}_{\infty}/K)$, let
\begin{equation}\label{chiE}\chi_E := \lambda_E|_{\Delta_K}.
\end{equation}
\end{enumerate}
\end{definition}

We now fix a particular choice of $\frak{f}_0$ for our applications. 
\begin{choice}\label{f0choice}Choose
$$\frak{f}_0 = \frak{f}_E^{(p)},$$
which we recall is the prime-to-$p$ part of $\frak{f}_E = \frak{f}(\lambda_E)$. Since $\lambda_E$ satisfies Assumption \ref{pconductorassumption} by Assumption \ref{lambdalambdaEassumption}, we have $\frak{f}_0 = f_0\mathcal{O}_K$ with $f_0 \ge 4$. Easy casework shows that this implies $w_{\frak{f}_0} = 1$.
\end{choice} %In particular, if $N$ is the conductor of $E/\mathbb{Q}$ then $N = \mathrm{Nm}_{K/\mathbb{Q}}(\frak{f})|D_K|$ where $D_K$ is the fundamental discriminant of $K$. 

%Precomposing  (\ref{twistColeman}) with (\ref{1ftwist}), we get an $\mathcal{O}_{K_p}\llbracket \mathrm{Gal}(\mathcal{K}_{\infty}/K)\rrbracket$-module homomorphism
%\begin{equation}\label{EtwistColeman}\frac{d}{w_F}\widetilde{\log} g : \mathbb{U}^1(\lambda_E^{-1}) \rightarrow \mathcal{O}_{K(\frak{f}_0)_p}\llbracket X\rrbracket [1/p].
%\end{equation}

We now show that under the decompositions of Definition \ref{decompositionchoices}, we have $\lambda_E/\chi_E = \lambda_{A_0}/\chi_{A_0}$.

\begin{proposition}\label{sameCMcharacterproposition}We have 
\begin{equation}\label{twistsame}\lambda_E/\chi_E = \lambda_{A_0}/\chi_{A_0}
\end{equation}
as continuous characters $\Gamma_K \rightarrow \mathcal{O}_{K_p}^{\times}$, using (\ref{fixA2}) and (\ref{fixK}) to view these characters on $\Gamma_K$. 
\end{proposition}

\begin{proof}By the theory of complex multiplication (\cite[Theorem 5.11]{Rubin2}), we see that $\lambda_E/\chi_E$ maps $\Gamma_K = \mathrm{Gal}(K_{\infty}/K)$ into $\Gamma'$ with finite cokernel, where $\Gamma' \subset \mathcal{O}_{K_p}^{\times}$ is the maximal free $\mathbb{Z}_p$-submodule. Similarly $\lambda_{A_0}/\chi_{A_0}$ maps 
$$\Gamma_K \overset{(\ref{Kinftyrestriction})}{=} \mathrm{Gal}(K_{\infty}K(\frak{f}_0)/K(\frak{f}_0)) \rightarrow \Gamma'$$
with finite cokernel, where the identity $\mathrm{Gal}(K_{\infty}K(\frak{f}_0)/K(\frak{f}_0)) = \Gamma_K$ is induced by restriction from the compositum $K_{\infty}K(\frak{f}_0)$ to $K_{\infty}$. (The fact that this restriction is an isomorphism follows from the fact that $K(\frak{f}_0)$ and $K_{\infty}$ are linearly disjoint, which in turn follows from our assumption that $K$ has class number 1.) Moreover, since $\lambda_E$ and $\lambda_{A_0}$ differ by a finite order character (both are algebraic Hecke characters of infinity type $(1,0)$), we have that $\lambda_E/\chi_E$ and $\lambda_{A_0}/\chi_{A_0}$ differ by a finite order character. Since the images of both $\lambda_E/\chi_E$ and $\lambda_{A_0}/\chi_{A_0}$ lie in the torsion-free group $\Gamma'$, this finite order character must be trivial, and so we have $\lambda_E/\chi_E = \lambda_{A_0}/\chi_{A_0}$ on all of $\Gamma_K$. 

\end{proof}

Finally, we will need the following variants of (\ref{Unormdiagram1}) and (\ref{Unormdiagram2}). Recall the inclusion $\mathcal{K}_n \subset K(\frak{f}_0\frak{p}^n)$ for all $n \ge e$ from (\ref{Kninclusion}). %Recall that $a \ge e/2$ by (\ref{betadelta}), and so $\frak{p}^e|p^a$ and $\mathcal{K}_{n+2a} \subset K(\frak{f}_0p^a\frak{p}^n)$ and for all $n \ge 0$. 
Therefore
$$(\mathcal{O}_{\mathcal{K}_n} \otimes_{\mathbb{Z}}\mathbb{Z}_p)^{\times} \subset R_n^{\times}(\frak{f}_0), \hspace{1cm} (\mathcal{O}_{\mathcal{K}_n} \otimes_{\mathbb{Z}}\mathbb{Z}_p)^{\times,1} \subset R_n^1(\frak{f}_0)$$
for all $n \ge e$. Moreover, the degree $[K(\frak{f}_0\frak{p}^n) : \mathcal{K}_n]$ stabilizes for all $n \gg e$. Thus we get commutative diagrams for all $n \gg e$
 \begin{equation}\label{UEnormdiagram1}
\begin{tikzcd}[column sep = huge]
(\mathcal{O}_{\mathcal{K}_n}\otimes_{\mathbb{Z}}\mathbb{Z}_p)^{\times} \arrow{d}{\mathrm{Nm}_n} \arrow[hook]{r} & R_n^{\times}(\frak{f}_0) \arrow{d}{\mathrm{Nm}_n}\\
  (\mathcal{O}_{\mathcal{K}_{n-1}}\otimes_{\mathbb{Z}}\mathbb{Z}_p)^{\times} \arrow[hook]{r} & R_{n-1}^{\times}(\frak{f}_0)\\
   \end{tikzcd}, \hspace{1cm} 
   \begin{tikzcd}[column sep = huge]
(\mathcal{O}_{\mathcal{K}_n}\otimes_{\mathbb{Z}}\mathbb{Z}_p)^{\times,1} \arrow{d}{\mathrm{Nm}_n} \arrow[hook]{r} & R_n^1(\frak{f}_0) \arrow{d}{\mathrm{Nm}_n}\\
 (\mathcal{O}_{\mathcal{K}_{n-1}}\otimes_{\mathbb{Z}}\mathbb{Z}_p)^{\times,1} \arrow[hook]{r} & R_{n-1}^1(\frak{f}_0)\\
   \end{tikzcd},
\end{equation}
and for all $n \ge 1$, commutative diagrams
\begin{equation}\label{UEnormdiagram2}
\begin{tikzcd}[column sep = huge]
R_n^{\times}(\frak{f}_0) \arrow{d}{\mathrm{Nm}_n} \arrow{rr}{\mathrm{Nm}_{K(\frak{f}_0\frak{p}^n)/\mathcal{K}_n}} & & (\mathcal{O}_{\mathcal{K}_n}\otimes_{\mathbb{Z}}\mathbb{Z}_p)^{\times}  \arrow{d}{\mathrm{Nm}_n}\\
  R_{n-1}^{\times}(\frak{f}_0) \arrow{rr}{\mathrm{Nm}_{K(\frak{f}_0\frak{p}^{n-1})/\mathcal{K}_{n-1}}} & & (\mathcal{O}_{\mathcal{K}_{n-1}}\otimes_{\mathbb{Z}}\mathbb{Z}_p)^{\times} \\
   \end{tikzcd},\\
\end{equation}
\begin{equation}\label{UEnormdiagram3}
   \begin{tikzcd}[column sep = huge]
R_n^1(\frak{f}_0) \arrow{d}{\mathrm{Nm}_n} \arrow{rr}{\mathrm{Nm}_{K(\frak{f}_0\frak{p}^n)/\mathcal{K}_n}} && (\mathcal{O}_{\mathcal{K}_n}\otimes_{\mathbb{Z}}\mathbb{Z}_p)^{\times,1}  \arrow{d}{\mathrm{Nm}_n}\\
  R_{n-1}^1(\frak{f}_0) \arrow{rr}{\mathrm{Nm}_{K(\frak{f}_0\frak{p}^{n-1})/\mathcal{K}_{n-1}}} && (\mathcal{O}_{\mathcal{K}_{n-1}}\otimes_{\mathbb{Z}}\mathbb{Z}_p)^{\times,1} \\
   \end{tikzcd}.
\end{equation}
%Note that there is a diagram analogous to (\ref{UEnormdiagram2}) without the ``$p^a$'' for all $n \gg e$, 
%\begin{equation}\label{UEnormdiagram4}
%\begin{tikzcd}[column sep = huge]
%R_n^{\times}(\frak{f}_0) \arrow{d}{\mathrm{Nm}_n} \arrow{rr}{\mathrm{Nm}_{K(\frak{f}_0\frak{p}^n)/\mathcal{K}_{n}}} & & (\mathcal{O}_{\mathcal{K}_{n}}\otimes_{\mathbb{Z}}\mathbb{Z}_p)^{\times}  \arrow{d}{\mathrm{Nm}_{n}}\\ R_{n-1}^{\times}(\frak{f}_0) \arrow{rr}{\mathrm{Nm}_{K(\frak{f}_0\frak{p}^{n-1})/\mathcal{K}_{n-1}}} & & (\mathcal{O}_{\mathcal{K}_{n-1}}\otimes_{\mathbb{Z}}\mathbb{Z}_p)^{\times} \\
%   \end{tikzcd}.\\
%\end{equation}
%Similarly one has such a diagram analogous to (\ref{UEnormdiagram3}). However, we will particularly need the diagram (\ref{UEnormdiagram2}) for Definition \ref{xiEdefinition} below. 

The diagrams (\ref{UEnormdiagram1}) fit together to give a map of $\mathbb{Z}_p\llbracket \mathrm{Gal}(K(\frak{f}_0p^{\infty})/K)\rrbracket$-modules
\begin{equation}\label{UEinclusion}i_{\frak{f}_0} : \mathbb{U}^1 \hookrightarrow \mathbb{U}^1(\frak{f}_0)
\end{equation}
and the diagrams (\ref{UEnormdiagram2}) and (\ref{UEnormdiagram3}) give a map of $\mathbb{Z}_p\llbracket \mathrm{Gal}(K(\frak{f}_0p^{\infty})/K)\rrbracket$-modules
\begin{equation}\label{UEnorm}\mathrm{Nm}_{\frak{f}_0} : \mathbb{U}^1(\frak{f}_0) \rightarrow \varprojlim_{n \ge 0}(\mathcal{O}_{\mathcal{K}_n} \otimes_{\mathbb{Z}}\mathbb{Z}_p)^{\times,1} = \mathbb{U}^1.
\end{equation}
%where the last isomorphism is defined by mapping the $n^{\mathrm{th}}$-indexed term of the source inverse limit to the $(n+2a)^{\mathrm{th}}$-indexed term of the target inverse limit. 
% The diagram (\ref{UEnormdiagram4}) gives
%\begin{equation}\label{UEnorm2}\mathrm{Nm}_{\frak{f}_0} : \mathbb{U}^1(\frak{f}_0) \rightarrow\mathbb{U}^1.
%\end{equation}

\subsection{General Iwasawa modules and their twists}Let us now introduce notation to denote the Iwasawa algebras that will feature most prominently in our discussion. 

\begin{definition}\label{Lambdadefinition}Recall the notation of Definition \ref{completedgroupalgebradefinition} and $\Gamma_K \cong \mathbb{Z}_p^{\oplus 2}$ from (\ref{GammaK}). 
\begin{enumerate}
\item Henceforth, given a $p$-adically complete $\mathbb{Z}_p$-algebra $R$, let 
$$\widetilde{\Lambda}_R = R\llbracket \mathrm{Gal}(\mathcal{K}_{\infty}/K)\rrbracket, \hspace{1cm} \Lambda_R[1/p] = \Lambda \otimes_{\mathbb{Z}_p}\mathbb{Q}_p$$
and
$$\Lambda_R = R\llbracket \Gamma_K\rrbracket, \hspace{1cm} \Lambda_R[1/p] = \Lambda_R \otimes_{\mathbb{Z}_p}\mathbb{Q}_p.$$
There are thus a natural projections
$$\widetilde{\Lambda}_R \twoheadrightarrow \Lambda_R, \hspace{1cm} \widetilde{\Lambda}_R[1/p] \rightarrow \Lambda_R[1/p].$$
When $R = \mathbb{Z}_p$, we will simply write $\widetilde{\Lambda}_{\mathbb{Z}_p} = \widetilde{\Lambda}$, $\widetilde{\Lambda}_{\mathbb{Z}_p}[1/p] = \widetilde{\Lambda}[1/p]$, $\Lambda_{\mathbb{Z}_p} = \Lambda$ and $\Lambda_{\mathbb{Z}_p}[1/p] = \Lambda[1/p]$. 

\item Given a character $\rho : \Gamma_K \rightarrow R^{\times}$, let $(\Gamma_K-\rho(\Gamma_K)) \subset \Lambda_R[1/p]$ denote the kernel of the map
$$\Lambda_R[1/p] \rightarrow R[1/p], \hspace{1cm} \gamma \mapsto \rho(\gamma),\hspace{.25cm} \forall \gamma \in \Gamma_K.$$
When $\rho = \mathbf{1}$, the trivial character on $\Gamma_K$, we let $(\Gamma_K - \mathbf{1}(\Gamma_K)) = (\Gamma_K - 1)$. Given a $\Lambda_R[1/p]$-module $M$, let 
$$M/(\Gamma_K-\rho(\Gamma_K)) = M \otimes_{\Lambda_R[1/p]}\Lambda_R[1/p]/(\Gamma_K-\rho(\Gamma_K)).$$

%\item When $M \subset \Lambda_R[1/p]$ is an ideal, we will let $M/(\Gamma_K-\rho(\Gamma_K))$ denote the image of 
%$$M \subset \Lambda_R[1/p] \rightarrow \Lambda_R[1/p]/(\Gamma_K-\rho(\Gamma_K)) = R[1/p].$$
\end{enumerate}
\end{definition}

We now define our notation isotypic components. Note that we invert $p$ when defining $\chi$-isotypic components. Recall the notation of Definition \ref{completedgroupalgebradefinition}. 

\begin{definition}\label{isotypicdefinition}\begin{enumerate}
\item Given a group $G$, let $\hat{G}$ denote the group of $\mathbb{C}_p^{\times}$-valued characters on it.

\item Recall $\Delta_A$ and $\Delta_K$ from Definition \ref{decompositionchoices}. Given a $p$-adically complete ring $R$ and an $\widetilde{\Lambda}_R$-module $M$ and a character $\chi \in \hat{\Delta}_K$, let $R[\chi]$ denote the extension of $R$ generated by the values of $\chi$ and let
$$M_{\chi} := \{x \in M \otimes_{R} R[\chi][1/p] : \delta(x) = \chi(x), \hspace{.25cm} \forall \delta \in \Delta_K\}$$
denote the $\chi$-isotypic component. $M_{\chi}$ can be viewed both as an $\widetilde{\Lambda}_{R[\chi]}[1/p]$-module and as an $\Lambda_{R[\chi]}[1/p]$-module using (\ref{fixK}). 

\item Let $R_{\Delta_K}/R$ be the extension generated by the values of all $\chi \in \hat{\Delta}_K$. 
We have a canonical isotypic decomposition
$$M \otimes_R R_{\Delta_K}[1/p] = \bigoplus_{\chi \in \hat{\Delta}_K}M_{\chi}\otimes_{R[\chi][1/p]}R_{\Delta_K}[1/p]$$
given by 
$$m = \sum_{\chi \in \hat{\Delta}_K}e_{\chi}(m) \mapsto (e_{\chi}(m) \otimes 1)_{\chi \in \hat{\Delta}_K},$$
where 
\begin{equation}\label{echiK}e_{\chi} = \frac{1}{\#\Delta_K}\sum_{\delta \in \Delta_K}\chi^{-1}(\delta)\delta.
\end{equation}

\item Given a $p$-adically complete ring $R$ and an $R\llbracket \mathrm{Gal}(K(\frak{f}_0p^{\infty})/K)\rrbracket$-module $M$ and a character $\chi \in \hat{\Delta}_{A_0}$,  let 
$$M_{\chi} := \{x \in M \otimes_R R[\chi][1/p] : \delta(x) = \chi(x), \hspace{.25cm} \forall \delta \in \Delta_{A_0}\}$$
denote the $\chi$-isotypic component. $M_{\chi}$ can be viewed both as an $R[\chi]\llbracket \mathrm{Gal}(K(\frak{f}_0p^{\infty})/K)\rrbracket [1/p]$-module and as an $\Lambda_{R[\chi]}[1/p]$-module. 

\item Let $R_{\Delta_{A_0}}[\chi]/R$ be the extension generated by the values of all $\chi \in \hat{\Delta}_{A_0}$. We have a canonical isotypic decomposition
$$M \otimes_R R_{\Delta_{A_0}}[1/p] = \bigoplus_{\chi \in \hat{\Delta}_{A_0}}M_{\chi}\otimes_{R[\chi][1/p]}R_{\Delta_{A_0}}[1/p]$$
given by 
$$m = \sum_{\chi \in \hat{\Delta}_{A_0}}e_{\chi}(m) \mapsto (e_{\chi}(m) \otimes 1)_{\chi \in \hat{\Delta}_{A_0}},$$
where 
\begin{equation}\label{echiA}e_{\chi} = \frac{1}{\#\Delta_{A_0}}\sum_{\delta \in \Delta_{A_0}}\chi^{-1}(\delta)\delta.
\end{equation}

%Recall the type $(1,0)$ Hecke character $\lambda_E$ over $K$ attached to $E$, and viewing $\lambda_E : \mathrm{Gal}(\mathcal{K}_{\infty}/K) \cong \Gamma_K \times \Delta_K \rightarrow \overline{\mathbb{Q}}_p^{\times}$ as a $p$-adic Galois character, let $\chi_E$ be the restriction to $\Delta_K$ of $\lambda_E$. 

%Given an $\mathcal{O}_{K_p}\llbracket \Delta_K \times \Gamma_K\rrbracket [1/p]$-module $M$, let 
%$$M(\lambda_E^{-1}) := M \otimes_{K_p}K_p(\lambda_E^{-1})$$
%where $K_p(\lambda_E^{-1})$ is the one-dimensional vector space with $\mathrm{Gal}(\overline{K}/K)$ acting by $\lambda_E^{-1}$. 
\end{enumerate}
\end{definition}

Recall we view 
$$\mathrm{Gal}(K(\frak{f}_0p^{\infty})/K_{\infty}) = \Delta_{A_0} \subset \mathrm{Gal}(K(\frak{f}_0p^{\infty})/K)$$
via (\ref{fixA2}). By Proposition \ref{lambdafactorproposition}, $\lambda_E|_{\mathrm{Gal}(K(\frak{f}_0p^{\infty})/K_{\infty})}$ factors through the character $\chi_E : \mathrm{Gal}(\mathcal{K}_{\infty}/K_{\infty}) \rightarrow \mathcal{O}_{K_p}^{\times}$ from (\ref{chiE}). Thus (\ref{UEinclusion}) descends to the $\chi_E$-isotypic component to give a map of $\Lambda_{\mathcal{O}_{K_p}}[1/p]$-modules
\begin{equation}\label{UEinclusionchi}i_{\frak{f}_0,\chi_E} : \mathbb{U}_{\chi_E}^1 \hookrightarrow \mathbb{U}^1(\frak{f}_0)_{\chi_E}
\end{equation}
and
\begin{equation}\label{UEnormchi}\mathrm{Nm}_{\frak{f}_0,\chi_E} : \mathbb{U}^1(\frak{f}_0)_{\chi_E} \rightarrow \mathbb{U}_{\chi_E}^1.
\end{equation}

We will need the following notion of twisting by a character. Recall the notation of Definition \ref{completedgroupalgebradefinition}. 

\begin{definition}\label{globalcharactertwistdefinition}
\begin{enumerate}
\item Let $G$ be any group. For any character $\varrho : G \rightarrow \mathcal{O}'^{\times}$ for some ring $\mathcal{O}'$, let $\mathcal{O}'(\varrho)$ denote the rank 1 free $\mathcal{O}'$-module with $G$-action given by $\varrho$. 
\item Given any $p$-adically complete $\mathbb{Z}_p$-algebra $R$, a profinite group $G$, and any $R\llbracket G\rrbracket$-module $M$ and a $p$-adically complete $R$-algebra $\mathcal{O}'$, define the $\mathcal{O}'\llbracket G\rrbracket$-module
$$M(\varrho) := M\hat{\otimes}_{R}\mathcal{O}'(\varrho)$$
where ``$\hat{\otimes}$'' denotes $p$-adically completed tensor product. 
\end{enumerate}
\end{definition}

As the following combination of Definitions \ref{isotypicdefinition} and \ref{globalcharactertwistdefinition} will appear, we explain it in detail. 

\begin{definition}\label{Mtwistdefinition}
\begin{enumerate}
\item Given a $\widetilde{\Lambda}$-module $M$, as $\chi_E$ is a character on $\Delta_K = \mathrm{Gal}(\mathcal{K}_{\infty}/K_{\infty})$ and the extension of $\mathbb{Z}_p$ generated by $\chi_E$ is $\mathcal{O}_{K_p}$, we get a well-defined $e_{\chi_E}\widetilde{\Lambda}$-module 
$$M_{\chi_E} := e_{\chi_E}M$$
from Definition \ref{isotypicdefinition}, where $e_{\chi_E}$ is as in (\ref{echiK}) with $\chi = \chi_E$. 
\item Viewing $M_{\chi_E}$ as a $\widetilde{\Lambda}_{\mathcal{O}_{K_p}}[1/p]$-module and viewing $\lambda_E$ as a character 
$$\lambda_E : \mathrm{Gal}(\mathcal{K}_{\infty}/K) \overset{(\ref{fixK})} = \Delta_K \times \Gamma_K \rightarrow \mathcal{O}_{K_p}^{\times},$$
we can twist $M_{\chi_E}$ as in Definition \ref{globalcharactertwistdefinition} to get a $\widetilde{\Lambda}_{\mathcal{O}_{K_p}}[1/p]$-module
$$M_{\chi_E}(\lambda_E^{-1}) := M_{\chi_E} \otimes_{\widetilde{\Lambda}_{\mathcal{O}_{K_p}}} \otimes_{\mathcal{O}_{K_p}} \mathcal{O}_{K_p}(\lambda_E^{-1}).$$
Since $\chi_E = \lambda_E|_{\Delta_K}$, the $\widetilde{\Lambda}_{\mathcal{O}_{K_p}}[1/p]$-action on $M_{\chi_E}(\lambda_E^{-1})$ factors through the projection $\widetilde{\Lambda}_{\mathcal{O}_{K_p}}[1/p] \twoheadrightarrow \Lambda_{\mathcal{O}_{K_p}}[1/p]$, we can also view $M_{\chi_E}(\lambda_E^{-1})$ as a $\Lambda_{\mathcal{O}_{K_p}}[1/p]$-module (which we often prefer). 
\end{enumerate}
\end{definition}

\subsection{Coleman map on $\mathbb{U}^1$}

Let $\mathbf{1}_{A_0}$ denote the trivial character on $\Delta_{A_0}$. %Taking the trivial isotypic components of the second map of (\ref{Uinclusion}) we get 
%$$i_{\frak{f}_0} : \mathbb{U}_{1_K}^1 \rightarrow \mathbb{U}^1(\frak{f}_0)_{\mathbf{1}_{A_0}}.$$
Tensoring (\ref{UEinclusionchi}) with $\otimes_{\mathcal{O}_{K_p}}\mathcal{O}_{K_p}((\lambda_E/\chi_E)^{-1})$, we get a map of $\Lambda_{\mathcal{O}_{K_p}} [1/p]$-modules 
 \begin{equation}\label{1ftwist}i_{\frak{f}_0,\chi_E} : \mathbb{U}_{\chi_E}^1(\lambda_E^{-1}) \rightarrow \mathbb{U}^1(\frak{f}_0)_{\chi_E}(\lambda_E^{-1}) = \mathbb{U}^1(\frak{f}_0)_{\mathbf{1}_{A_0}}((\lambda_E/\chi_E)^{-1}) \overset{(\ref{twistsame})}{=}  \mathbb{U}^1(\frak{f}_0)_{\mathbf{1}_{A_0}}((\lambda_{A_0}/\chi_{A_0})^{-1}).
\end{equation}
The first equality in the first equation follows immediately after noting that the values of each of $\lambda_E$ and $\lambda_E/\chi_E$ generate $\mathcal{O}_{K_p}$ over $\mathbb{Z}_p$. Similarly tensoring (\ref{UEnormchi}) with $\otimes_{\mathcal{O}_{K_p}} \mathcal{O}_{K_p}((\lambda_E/\chi_E)^{-1})$, we get a map of $\Lambda_{\mathcal{O}_{K_p}}[1/p]$-modules
\begin{equation}\label{UEnormchitwist}\mathrm{Nm}_{\frak{f}_0,\chi_E} : \mathbb{U}^1(\frak{f}_0)_{\mathbf{1}_{A_0}}((\lambda_{A_0}/\chi_{A_0})^{-1}) \overset{(\ref{twistsame})}{=} \mathbb{U}^1(\frak{f}_0)_{\mathbf{1}_{A_0}}((\lambda_E/\chi_E)^{-1}) = \mathbb{U}^1(\frak{f}_0)_{\chi_E}(\lambda_E^{-1}) \rightarrow  \mathbb{U}_{\chi_E}^1(\lambda_E^{-1}).
\end{equation}

\begin{remark}\label{isotypicdifferremark}The reader may find the notation ``$\mathbb{U}^1(\frak{f}_0)_{\mathbf{1}_{A_0}}((\lambda_{A_0}/\chi_{A_0})^{-1}) $'' strange, and ask why we do not write it as $\mathbb{U}^1(\frak{f}_0)_{\chi_{A_0}}(\lambda_{A_0}^{-1})$ in parallel with $\mathbb{U}^1(\frak{f}_0)_{\chi_E}(\lambda_E^{-1})$. This is because in the notation of Definition \ref{isotypicdefinition} and \ref{globalcharactertwistdefinition}, these objects differ by an extension of scalars: Let $\mathcal{O}_{K_p}[\chi_{A_0}]$ be the finite extension of $\mathcal{O}_{K_p}$ generated by the values of $\chi_{A_0}$. Then in the notation of Definition \ref{isotypicdefinition}, we have
$$\mathbb{U}^1(\frak{f}_0)_{\mathbf{1}_{A_0}}((\lambda_{A_0}/\chi_{A_0})^{-1}) \otimes_{\mathcal{O}_{K_p}}\mathcal{O}_{K_p}[\chi_{A_0}] = \mathbb{U}^1(\frak{f}_0)_{\chi_{A_0}}(\lambda_{A_0}^{-1}).$$

\end{remark}

Note under that the decomposition (\ref{fixA}) we view $\Gamma_K \subset \mathrm{Gal}(K(\frak{f}_0p^{\infty})/K(\frak{f}_0))$. Thus (\ref{kappavarphiequal}) we have 
$$\lambda_{A_0}/\chi_{A_0} = \kappa|_{\Gamma_K}$$
as characters on $\Gamma_K$. From this and (\ref{globalthickendtildelogmap}), we get a map
\begin{equation}\label{globalthickendtildelogmap'}d\widetilde{\log}\left(\mathrm{thicken}(g)\right) : \mathbb{U}^1(\frak{f}_0)_{\mathbf{1}_{A_0}}((\lambda_{A_0}/\chi_{A_0})^{-1}) \rightarrow \mathbf{\Gamma}(\Omega_{F^{\mathrm{univ}}/\mathrm{LT}})[1/p].
\end{equation}
Precomposing (\ref{globalthickendtildelogmap'}) with (\ref{1ftwist}), we get an $\Lambda_{\mathcal{O}_{K_p}}[1/p]$-module homomorphism
\begin{equation}\label{EtwistColeman}d\mathrm{Log} : \mathbb{U}_{\chi_E}^1(\lambda_E^{-1}) \rightarrow \mathbf{\Gamma}(F^{\mathrm{univ}})[1/p], \hspace{1cm} \beta \mapsto d\left(\widetilde{\log}(\mathrm{thicken}(g_{i_{\frak{f}_0,\chi_E}(\beta)}))\right). 
\end{equation}
This map will be of fundamental interest to us. 

\subsection{Elliptic units and their Coleman power series}%Let $N$ be the smallest positive integer in $\frak{f}^{(p)}$, so that $N$ prime to $p$, and assume $N > 3$ and let $Y = Y(N)$, which has a universal object $E \rightarrow Y$. 

%For any $1\le a \le N-1$ have the Siegel unit
%$$g(q) = q^{1/12}(1-\zeta_N^a)\prod_{n = 1}^{\infty} (1-q^n\zeta_N^a)(1-q^n\zeta_N^{-a}) \in \Gamma(Y)^{\times}.$$
%Recall our fixed basis of Tate module $\alpha$, which for $F = \hat{A}$ induces a corresponding basis for $\varprojlim_n A^{\phi^{-n}}[\varpi^n]$. Then for each $n$, we get a CM point $(A^{\phi^{-n}},\alpha_n)$ (suppressing the $\Gamma_1(N)$-level structure) in some cover of $Y$.
Continue to assume $\frak{f}_0$ is as in Choice \ref{f0choice}.

\begin{definition}\label{siegelthetadefinition}\begin{enumerate}
\item Let $q_z = e^{2\pi i z}$, let $L$ be a lattice $\mathbb{Z}\tau + \mathbb{Z}$, $\tau \in \mathcal{H}^+$ and let $q_{\tau} = \mathrm{exp}(2 \pi i \tau)$. Let 
$$A(L) = \pi^{-1}\mathrm{Area}(\mathbb{C}/L) = \pi^{-1}\mathrm{Im}(\tau)$$
as in \cite[Chapter II.2.1 (4)]{deShalit}. 
\item As in Chapter II.2.1 (8) of op. cit., we have the Siegel $\theta$-function 
$$\theta(z,L) = e^{6A(L)^{-1}z(z-\bar{z})}\cdot q_{\tau}(q_z^{1/2}-q_z^{-1/2})^{12} \cdot \prod_{n = 1}^{\infty}\left((1-q_{\tau}^nq_z)(1-q_{\tau}^nq_z^{-1})\right)^{12}.$$
\item Let $\frak{a} \subset \mathcal{O}_K$ be an ideal prime to $6\frak{f}_0\frak{p}$. Let $L_0$ be the period lattice of $A_0$ (see Choice \ref{FixCMdefinition}). Thus, after possibly replacing $A_0$ by one of its Galois conjugates, we may assume
$$L_0 = \Omega_0\frak{f}_0$$
for some $\Omega_0 \in \mathbb{C}^{\times}$ (see \cite[II.4.3 (6)]{deShalit}). 
\item Let
$$\Theta(z;L_0,\frak{a}) := \theta(z,L_0)^{\mathbb{N}\frak{a}}/\theta(z,\frak{a}^{-1}L_0)$$
where $\mathbb{N}\frak{a}$ is the positive generator of the ideal $\mathrm{Nm}_{K/\mathbb{Q}}(\frak{a}) \subset \mathbb{Z}$.
\end{enumerate} 
\end{definition}

\begin{definition}\label{xiEdefinition}As explained in Definition \ref{siegelthetadefinition}, we may assume without loss of generality that $L_0 = \Omega_0 \frak{f}_0$ for some $\Omega_0 \in \mathbb{C}^{\times}$. In particular, $\Omega_0$ is a primitive $\frak{f}_0$-torsion point of $\mathbb{C}/L_0$ (see Choice \ref{f0choice} and the discussion above). Define norm-compatible systems of elliptic units
\begin{equation}\label{ellipticunitea}e(\frak{a};\frak{f}_0) = \varprojlim_n \Theta(\Omega_0;\frak{p}^nL_0,\frak{a}) \in \mathbb{U}(\frak{f}_0), \hspace{1cm} e^1(\frak{a};\frak{f}_0) \in\mathbb{U}^1(\frak{f}_0)
\end{equation}
where $e^1(\frak{a};\frak{f}_0)$ is the image of $e(\frak{a};\frak{f}_0)$ under the pro-$p$ projection $\mathbb{U}(\frak{f}_0) \rightarrow \mathbb{U}^1(\frak{f}_0)$. %Recalling $a$ from (\ref{adefinition}), let
%\begin{equation}\label{ellipticuniteaa}e(\frak{a};\frak{f}_0p^a) = \varprojlim_n \Theta(\Omega_0;\frak{p}^np^aL_0,\frak{a}) \in \mathbb{U}(\frak{f}_0p^a), \hspace{1cm} e^1(\frak{a};\frak{f}_0p^a) \in \mathbb{U}^1(\frak{f}_0p^a)
%\end{equation}
%where $e^1(\frak{a};\frak{f}_0p^a)$ is the image of $e(\frak{a};\frak{f}_0p^a)$ under the natural projection $\mathbb{U}(\frak{f}_0p^a) \rightarrow \mathbb{U}^1(\frak{f}_0p^a)$. Here all inverse limits are taken with respect to norm (cf. \cite[Chapter II.4.9]{deShalit}). See Proposition II.2.4 of op. cit. for a proof of norm-compatibility. 
\end{definition}

%\begin{proposition} Recall the maps 
%$$\mathrm{Nm}_{\frak{f}_0/\frak{f}_E^{(p)}} : \mathbb{U}^1(\frak{f}_0) \rightarrow \mathbb{U}^1(\frak{f}_E^{(p)}),\hspace{1cm} \mathrm{Nm}_{\frak{f}_0p^{\beta}/\frak{f}_E^{(p)}p^{\beta}} : \mathbb{U}^1(\frak{f}_0p^{\beta}) \rightarrow \mathbb{U}^1(\frak{f}_E^{(p)}p^{\beta})$$
%from (\ref{Unorm}) and $e(\frak{a};\frak{f}_0) \in \mathbb{U}^1(\frak{f}_0)$ and $e(\frak{a};\frak{f}_E^{(p)}) \in \mathbb{U}^1(\frak{f}_E^{(p)})$ from (\ref{ellipticunitea}), and $e_{\beta}(\frak{a};\frak{f}_0) \in \mathbb{U}^1(\frak{f}_0p^{\beta})$ and $e_{\beta}(\frak{a};\frak{f}_E^{(p)}p^{\beta}) \in \mathbb{U}^1(\frak{f}_E^{(p)}p^{\beta})$ from (\ref{ellipticuniteadelta}). We have 
%\begin{align*}\mathrm{Nm}_{\frak{f}_0/\frak{f}_E^{(p)}}(e(\frak{a};\frak{f}_0)^{w_{\frak{f}_E^{(p)}}}) = e(\frak{a};\frak{f}_E^{(p)}) \in \mathbb{U}^1(\frak{f}_E^{(p)}),\\
%\mathrm{Nm}_{\frak{f}_0p^{\beta}/\frak{f}_E^{(p)}p^{\beta}}(e_{\beta}(\frak{a};\frak{f}_0)^{w_{\frak{f}_E^{(p)}}}) = e_{\beta}(\frak{a};\frak{f}_E^{(p)}) \in \mathbb{U}^1(\frak{f}_E^{(p)}p^{\beta}).
%\end{align*}
%\end{proposition}

%\begin{proof}The first identity follows immediately from Proposition II.2.5 of op. cit. with $\frak{f} = \frak{f}_E^{(p)}$, $\frak{g} = \frak{f}_0$ and $\frak{l} = (\frak{f}_E^{(p)})^{r-1}$. Similarly, the second identity follows from loc. cit. with $\frak{f} = \frak{f}_E^{(p)}p^{\beta}$, $\frak{g} = \frak{f}_0p^{\beta}$ and $\frak{l} = (\frak{f}_E^{(p)})^{r-1}$. 
%\end{proof}

We have the following Proposition (which is ultimately a consequence of the definition (\ref{ellipticunitea})) describing the Coleman power series of $e(\frak{a};\frak{f}_0)$. Recall the integer $f_0 \in \mathbb{Z}_{> 0}$ with $\frak{f}_0 = f_0\mathcal{O}_K$ (Choice \ref{f0choice}).

\begin{proposition}[Proposition II.4.9 of \cite{deShalit}]\label{e1aThetaproposition}Recall $L_0 = \Omega_0\frak{f}_0$ so that $\Omega_0$ is a $\frak{f}_0$-torsion (or equivalently $f_0$-torsion) point on $\mathbb{C}/L_0$. Let $P(z)$ be the Taylor series expansion of 
$$\Theta(\Omega_0-z;L_0,\frak{a}) \in K(\frak{f}_0)\llbracket z\rrbracket,$$
and let $\log_{A_0} : A_0(\mathbb{C}_p) \rightarrow \mathbb{C}_p$ be the $p$-adic formal group logarithm.
The Coleman power series $g_{e(\frak{a};\frak{f}_0)}$ from (\ref{globalColemanpowerseries}) of $e(\frak{a};\frak{f}_0)$ is 
$$g_{e(\frak{a};\frak{f}_0)}(X) = P(\log_{A_0}(X)).$$
Moreover, 
$$g_{e(\frak{a};\frak{f}_0)}/g_{e^1(\frak{a};\frak{f}_0)}$$
is a constant in $\mathcal{O}_{K(\frak{f}_0)_p}^{\times}$. 
\end{proposition}

For convenience (in order to compute the thickening of the Coleman power series of $e^1(\frak{f}_0;\frak{a})$), we will henceforth make the following choice of $\frak{a}$.

\begin{choice}\label{achoice}Henceforth, fix an ideal $\frak{a} \subset \mathcal{O}_K$ prime to $6\frak{f}_0\frak{p}N$ (where $N$ is as in Assumption \ref{Nf0assumption}) such that $\frak{a} = \mathbf{a} \cdot \mathcal{O}_K$ for some $\frak{a} \in \mathbb{Z}_{> 0}$. In particular, $\mathbb{N}\frak{a} = \mathbf{a}^2 \cdot \mathbb{Z}$. 
\end{choice}

%For later purposes, we will seek to relate the Coleman power series of elliptic units to Eisenstein numbers, following Section II.3 of op. cit. In the notation of Section \ref{Eisensteinsection}, let 
%$$E_1(z;L_0,\frak{a}) = \mathbb{N}\frak{a}\cdot E_1(z,L_0) - E_1(z,\frak{a}^{-1}L_0).$$

%Recall the level structure
%$$e_{1,a} : \mathrm{LT}_{\infty}^{\phi^{-2a+1}} \rightarrow F_{\infty}^{\mathrm{univ},\phi^{-2a+1}}.$$
%from (\ref{esection}). Note that the specialization of $e_{1,a}$ to the Lubin-Tate $\mathcal{O}_{K_p}$-module $F^{\phi^{-2a+1}}$ (viewed as a fiber of $F_{\infty}^{\mathrm{univ},\phi^{-2a+1}} \rightarrow \mathrm{LT}_{\infty}^{\phi^{-2a+1}}$) is a primitive $\frak{p}^{2a-1} = \frac{p^a}{\varpi}$-torsion point on $F^{\phi^{-2a+1}}$ (where $\varpi$ is as in (\ref{pichoice})). 

We now define a universal version of the $f$-basis from Choice \ref{fbasis} appearing in the statement of Theorem \ref{Colemanpowerseries}.

\begin{definition}\label{universalalphadefinition}
\begin{enumerate}
\item Recall that, from (\ref{YssLT}), for any $m \in\mathbb{Z}$ we may view $\mathrm{LT}_{\infty}^{\phi^m} \subset Y_{\infty}$ and identify the universal objects $F_{\infty}^{\mathrm{univ},\phi^m}\rightarrow \mathrm{LT}_{\infty}^{\phi^m}$ and $\hat{\mathcal{E}}|_{\mathrm{LT}_{\infty}^{\phi^m}} \rightarrow \mathrm{LT}_{\infty}^{\phi^m}$:
\begin{equation}\label{FEidentify}F_{\infty}^{\mathrm{univ},\phi^m} = \hat{\mathcal{E}}|_{\mathrm{LT}_{\infty}^{\phi^m}}
\end{equation}
where $\hat{E}$ is the formal group of the universal object $\mathcal{E} \rightarrow Y_{\infty}$, i.e. the universal elliptic curve with $\Gamma = \Gamma(N)$-level structure (recall Assumption \ref{Nf0assumption}, in particular $(N,p) = 1$ and $f_0|N$ so that the $\Gamma$-level structure determines a $\Gamma(f_0)$-level structure) and $\Gamma(p^{\infty})$-level structure. 
\item Let 
$$P = (P_1,P_2) : (\mathbb{Z}/f_0)_Y^{\oplus 2} \xrightarrow{\sim} \mathcal{E}[f_0]$$
denote the universal $\Gamma(f_0)$-level structure and let 
$$(e_1,e_2) : \hat{\mathbb{Z}}_{p,Y_{\infty}}^{\oplus 2} \xrightarrow{\sim} T_p\mathcal{E}$$
denote the universal $\Gamma(p^{\infty})$-level structure, and as in (\ref{eindefinition}) let $e_{i,n} = e_i \pmod{p^n}$. 
\item Let 
\begin{equation}\label{eNpi}e_{f_0p^i} \in \mathcal{E}[f_0p^i](Y_{\infty})
\end{equation}
denote the unique $f_0p^i$-torsion point corresponding to 
$$(P_1,e_{2,i}) \in \mathcal{E}[f_0](Y_{\infty}) \times \mathcal{E}[p^i](Y_{\infty})$$
under the canonical decomposition $\mathcal{E}[f_0p^i] = \mathcal{E}[f_0] \times \mathcal{E}[p^i]$. 
\item Let 
$$f^{\phi^m,\mathrm{univ}} : \mathcal{E}|_{\mathrm{LT}_{\infty}^{\phi^m}} \rightarrow \mathcal{E}|_{\mathrm{LT}_{\infty}^{\phi^{m+1}}}$$
denote the isogeny given by division by the canonical subgroup of $\mathcal{E}|_{\mathrm{LT}_{\infty}^{\phi^m}} \rightarrow \mathrm{LT}_{\infty}^{\phi^m}$. (Note that the canonical subgroup exists on $\mathrm{LT}_{\infty}^{\phi^m}$ by the discussion at the beginning of Section \ref{thickensection}.) This isogeny induces an isomorphism 
\begin{equation}\label{Ntorsionisomorphism}f^{\phi^m,\mathrm{univ}} : \mathcal{E}|_{\mathrm{LT}_{\infty}^{\phi^m}}[f_0] \xrightarrow{\sim} \mathcal{E}|_{\mathrm{LT}_{\infty}^{\phi^{m+1}}}[f_0]
\end{equation}
since $(f_0,p) = 1$. Let 
$$(f^{\phi^m,\mathrm{univ}})^i := f^{\mathrm{univ},\phi^{m+i-1}} \circ f^{\mathrm{univ},\phi^{m+i-2}} \circ \cdots \circ f^{\mathrm{univ},\phi^{m+1}} \circ f^{\mathrm{univ},\phi^m}$$
(cf. (\ref{fncomposition})). Let 
$$P_1^{\phi^{-m}} \in \mathcal{E}|_{\mathrm{LT}_{\infty}^{\phi^{-m}}}[f_0](\mathrm{LT}_{\infty})$$
be the unique order-$f_0$ point such that 
$$(f^{\mathrm{univ},\phi^{-m}})^m(P^{\phi^{-m}}) = P_1 \in \mathcal{E}|_{\mathrm{LT}_{\infty}}[f_0](\mathrm{LT}_{\infty})$$
under the isomorphism (\ref{Ntorsionisomorphism}). (Recall $P = (P_1,P_1)$ is the universal $\Gamma(f_0)$-level structure on $\mathcal{E}$.)

\item For any $m \in \mathbb{Z}$, consider the open subset 
$$\mathrm{LT}_{\infty}^{\phi^m} \cap \mathcal{Y}^{\mathrm{Ig}}(1/2) \subset \mathrm{LT}_{\infty}^{\phi^m}.$$
Recalling $e_{f_0p^i}$ from (\ref{eNpi}), for any $m \in \mathbb{Z}_{\ge 0}$ define
\begin{equation}\label{alphabasis}\alpha_{2,m} := P_1^{\phi^{-2m}} - e_{f_0p^m}|_{\mathrm{LT}_{\infty}^{\phi^{-2m}}\cap \mathcal{Y}^{\mathrm{Ig}}(1/2)} \in \mathcal{E}|_{\mathrm{LT}_{\infty}^{\phi^{-2m}}}[p^m](\mathrm{LT}_{\infty}\cap \mathcal{Y}^{\mathrm{Ig}}(1/2)).
\end{equation}
%Note that since $e_{1,1}$ trivializes the canonical subgroup on $\mathcal{Y}^{\mathrm{Ig}}(1/2)$, then $e_{1,1}|_{\mathrm{LT}_{\infty}^{\phi^{-2m+1}}\cap \mathcal{Y}^{\mathrm{Ig}}(1/2)}$ generates the kernel of $f^{\mathrm{univ},\phi^{-2m+1}}$. Thus, since 
Thus, since
$$\ker\left((f^{\mathrm{univ},\phi^{-2m}})^2\right) = \mathcal{E}|_{\mathrm{LT}_{\infty}^{\phi^{-2m}}}[p],$$
%we have that $e_{1,m+1}|_{\mathrm{LT}_{\infty}^{\phi^{-2m+1}}\cap \mathcal{Y}^{\mathrm{Ig}}(1/2)}$ generates the kernel of $(f^{\mathrm{univ},\phi^{-2m+1}})^{2m-1}$. Thus, by induction one has
by induction we have 
$$(f^{\mathrm{univ},\phi^{-2m}})^2(\alpha_{2,m}) = \alpha_{2,m-1}$$ 
for all $m \in \mathbb{Z}_{\ge 1}$ and $\alpha_{2,0} = 0$ (cf. (\ref{fbasiscondition})). Since 
$$\mathcal{E}|_{\mathrm{LT}_{\infty}^{\phi^{-2m}}}[p^{\infty}] = \hat{\mathcal{E}}|_{\mathrm{LT}_{\infty}^{\phi^{-2m}}}[p^{\infty}]$$
by the supersingularity of 
$$\mathcal{E}|_{\mathrm{LT}_{\infty}^{\phi^{-2m}}} \rightarrow \mathrm{LT}_{\infty}^{\phi^{-2m}} \subset Y_{\infty}^{\mathrm{ss}},$$
(\ref{FEidentify}) then shows that 
\begin{align*}\alpha_{2,m} &\in F_{\infty}^{\mathrm{univ},\phi^{-2m}}[(f^{\mathrm{univ},\phi^{-2m}})^{2m}](\mathrm{LT}_{\infty} \cap \mathcal{Y}^{\mathrm{Ig}}(1/2)) = F_{\infty}^{\mathrm{univ},\phi^{-2m}}[p^m](\mathrm{LT}_{\infty} \cap \mathcal{Y}^{\mathrm{Ig}}(1/2)),
\end{align*}
where $F_{\infty}^{\mathrm{univ},\phi^{-2m}}[(f^{\mathrm{univ},\phi^{-2m}})^{2m}] \subset F_{\infty}^{\mathrm{univ},\phi^{-2m}}$ denotes the kernel of $(f^{\mathrm{univ},\phi^{-2m}})^{2m}$. 

%\item 
%Under the identification (\ref{FEidentify}), $f^{\mathrm{univ},\phi^m}$ induces an isogeny 
%$$f^{\mathrm{univ},\phi^m} : F_{\infty}^{\mathrm{univ},\phi^m} \rightarrow F_{\infty}^{\mathrm{univ},\phi^{m+1}}.$$
\end{enumerate}
\end{definition}

\begin{remark}The construction of (\ref{alphabasis}) is in parallel to the construction of the $f$-basis $(\omega_n)_n$ in \cite[Chapter II.4.4]{deShalit}. One can view our construction as a thickened version of loc. cit. One may also think of $(\alpha_{2,m})_m$ as a universal version of the choice of $f$-basis on $F$ from Choice \ref{fbasis}.

\end{remark}

%\begin{definition}Henceforth, let
%$$\alpha_m = \alpha_{1,m}(y(2m-1))$$
%where $y(2m-1) \in \mathrm{LT}_{\infty}^{\phi^{-2m+1}}(\overline{\mathbb{Q}}_p,\overline{\mathbb{Z}}_p)$ is the point from (\ref{y(m)}) below. Then $(\alpha_m)_m$ satisfies (\ref{fbasiscondition}). Henceforth take this choice of $(\alpha_n)_n$ in Choice \ref{fbasis}. 
%\end{definition}

\begin{proposition}\label{Eisensteinnumber}We have
$$-\frac{d}{\log_{A_0}'(X)dX}\log g_{e^1(\frak{a};\frak{f}_0)} = -\frac{d}{\log_{A_0}'(X)dX}\log g_{e(\frak{a};\frak{f}_0)}(X) = 12 \cdot E_1(\Omega_0-z;L_0,\frak{a})$$
with $z = \log_{A_0}(X)$.
%, as well as
%\begin{equation}\label{Colemanspecialize}-\alpha_{a}^*\frac{d}{\log_{A_0^{\phi^{-2a+1}}}'(X)dX}\log g_{e^1(\frak{a};\frak{f}_0)}^{\phi^{-2a+1}} = 12\cdot E_1(\Omega_0,\Omega_0\frac{p^a}{\varpi}f_0(\mathbb{Z} + \mathbb{Z}\tau_0)).
%\end{equation}
\end{proposition}

\begin{proof}This follows from Proposition \ref{e1aThetaproposition} and \cite[II.3.1(7)]{deShalit}.%, and the second follows from loc. cit. and the fact that the property $g_{\beta}^{\phi^{-n}}(\alpha_n) = \beta_n$ from Theorem \ref{Colemanpowerseries} implies $g_{e(\frak{a};\frak{f}_0)}^{\phi^{-2a+1}}(\alpha_a) =  e(\frak{a};\frak{f}_0)_{2a+1} \overset{(\ref{ellipticunitea})}{=} \Theta(\Omega_0,\Omega_0\frac{p^a}{\varpi}f_0(\mathbb{Z} + \mathbb{Z}\tau_0))$. 

\end{proof}

We also have thickened versions of these results. Recall the universal object $\mathcal{E} \rightarrow \mathbb{Y}$ (Section \ref{algebraicYsection}). Using (\ref{YssLT}) identify $\mathcal{E}|_{\mathrm{LT}_{\infty}} = F_{\infty}^{\mathrm{univ}}$ and let $z = \log_{F_{\infty}^{\mathrm{univ}}}(X)$ denote the resulting coordinate on $\mathcal{E}|_{\mathrm{LT}_{\infty}}$, where $\log_{F_{\infty}^{\mathrm{univ}}}$ is the normalized formal logarithm on $F_{\infty}^{\mathrm{univ}}$ (cf. \cite[II.4.9]{deShalit}). Let $\mathcal{E}(\mathbb{C})^{\mathrm{an}} \rightarrow \mathbb{Y}(\mathbb{C})^{\mathrm{an}}$ denote the complex analytification of $\mathcal{E}$, and recall that the (componentwise) complex analytic universal cover of $\mathbb{Y}(\mathbb{C})^{\mathrm{an}}$ is 
$$\mathbf{H}^+ := \bigsqcup_{\pi_0(\mathbb{Y}(\mathbb{C})^{\mathrm{an}})}\mathcal{H}^+$$
(where, as before, $\pi_0(\mathbb{Y}(\mathbb{C})^{\mathrm{an}})$ denotes the component set of the complex analytic space $\mathbb{Y}(\mathbb{C})^{\mathrm{an}}$ and $\mathcal{H}^+ = \{\tau \in \mathbb{C} : \mathrm{Im}(\tau) > 0\}$). Let $z$ be the standard coordinate on $\mathbb{C}$, which induces a coordinate $z$ on 
$$\mathcal{E}(\mathbb{C})^{\mathrm{an}}|_{\mathbf{H}^+} = \bigsqcup_{\pi_0(\mathbb{Y}(\mathbb{C})^{\mathrm{an}})}\mathbb{C}/(\mathbb{Z}\tau + \mathbb{Z}),$$
where $\tau$ is the coordinate on $\mathcal{H}^+$.

\begin{proposition}\label{ThickenColemanspecialize}
\begin{enumerate}
\item 
In the setting of Theorem \ref{thickenproposition} and Definition \ref{universalalphadefinition}, we have 
\begin{equation}\label{thickenTheta}\mathrm{thicken}\left(g_{e(\frak{a};\frak{f}_0)}\right)  = \Theta(\Omega_0-z;\Omega_0f_0(\mathbb{Z}\tau + \mathbb{Z}),\frak{a}) = [-1]_{\mathcal{E}}^*\tau_{P_1}^*\left(\frac{\Theta(z;\mathbf{a}^{-1}(\mathbb{Z}\tau + \mathbb{Z}))}{\Theta(z;\mathbb{Z}\tau + \mathbb{Z})^{\mathbb{N}\frak{a}}}\right)
\end{equation}
where $\tau_{P_1} : \mathcal{E} \rightarrow \mathcal{E}$ denotes translation by $P_1 \in \mathcal{E}[f_0]$, and $[-1]_{\mathcal{E}} : \mathcal{E} \rightarrow \mathcal{E}$ denotes multiplication by $-1$ under the group law $[\cdot]_{\mathcal{E}}$ of the universal elliptic curve $\mathcal{E}$. 
(Note that $\Theta(\Omega_0-z;\Omega_0f_0(\mathbb{Z}\tau + \mathbb{Z}),\frak{a}) \in \mathbf{\Gamma}(\mathcal{O}_{\mathcal{E}^{\mathrm{univ}}}^{\times})$ and is not just a section of $\mathcal{O}_{\mathcal{E}(\mathbb{C})^{\mathrm{an}}}^{\times}$, by \cite[Chapter II.1]{Tsuji}.) Moreover, 
$$\mathrm{thicken}\left(g_{e(\frak{a};\frak{f}_0)}\right)/\mathrm{thicken}\left(g_{e^1(\frak{a};\frak{f}_0)}\right)$$
is a constant in $\mathcal{O}_{K(\frak{f}_0)_p}^{\times}$. 

\item Let $\frak{a} \subset \mathcal{O}_K$ and $\mathbf{a} \in \mathbb{Z}_{> 0}$ be as in Choice \ref{achoice} and let 
$$\alpha_{2,a} \in F_{\infty}^{\mathrm{univ},\phi^{-2a}}[p^a](\mathrm{LT}_{\infty} \cap \mathcal{Y}^{\mathrm{Ig}}(1/2))$$
be as in (\ref{alphabasis}). Then we have the following equality in $\mathcal{O}_{\mathcal{Y}^{\mathrm{Ig}}(1/2)}(\mathrm{LT}_{\infty} \cap \mathcal{Y}^{\mathrm{Ig}}(1/2))$:
\begin{equation}\label{thickenTheta2}\alpha_{2,a}^*\mathrm{thicken}\left(g_{e(\frak{a};\frak{f}_0)}^{\phi^{-2a}}\right) = e_{2,a}^*\left(\frac{\Theta(z;\mathbf{a}^{-1}f_0(\mathbb{Z}\tau + \mathbb{Z}))}{\Theta(z;f_0(\mathbb{Z}\tau + \mathbb{Z}))^{\mathbb{N}\frak{a}}}\right).
\end{equation}

\item Recall $U_p$ and $V_p$ from Definition \ref{flatdefinition}. Let $d : F^{\mathrm{univ}} \rightarrow \Omega_{F^{\mathrm{univ}}/\mathrm{LT}}$ denote the exterior derivative. Then
\begin{equation}\label{thickenEisenstein}\begin{split}\alpha_{2,a}^*d\widetilde{\log}\left(\mathrm{thicken}\left(g_{e^1(\frak{a};\frak{f}_0)}^{\phi^{-2a+1}}\right)\right) 
&= \alpha_{2,a}^*d\widetilde{\log}\left(\mathrm{thicken}\left(g_{e(\frak{a};\frak{f}_0)}^{\phi^{-2a}}\right)\right) \\
&= (1-V_p^*U_p^*)\alpha_{2,a}^*d\log\left(\mathrm{thicken}\left(g_{e(\frak{a};\frak{f}_0)}^{\phi^{-2a}}\right)\right) \\
&= (1-V_p^*U_p^*)e_{2,a}^*d\log\left(\frac{\Theta(z;\mathbf{a}^{-1}f_0(\mathbb{Z}\tau + \mathbb{Z}))}{\Theta(z;f_0(\mathbb{Z}\tau + \mathbb{Z}))^{\mathbb{N}\frak{a}}}\right)\\
&\hspace{-1.6cm}=  (1-V_p^*U_p^*)e_{2,a}^*\left(d\log\left(\Theta(z;\mathbf{a}^{-1}f_0(\mathbb{Z}\tau + \mathbb{Z}))\right) - d\log\left(\Theta(z;f_0(\mathbb{Z}\tau + \mathbb{Z}))^{\mathbb{N}\frak{a}}\right)\right).
\end{split}
\end{equation}
Here, $\log (\mathrm{thicken}(g_{e(\frak{a};\frak{f}_0)}))$ is defined using the Iwasawa branch of the $p$-adic logarithm as follows: $\log (\mathrm{thicken}(g_{e^1(\frak{a};\frak{f}_0)}))$ is defined as usual using (\ref{logpowerseries}) and (\ref{principalthick2}), 
$$\log\left(g_{e(\frak{a};\frak{f}_0)}\right) = \log\left(g_{e^1(\frak{a}l\frak{f}_0)}\right) + 
\log\left(g_{e(\frak{a};\frak{f}_0)}/g_{e^1(\frak{a};\frak{f}_0)}\right)$$
where $\log\left(g_{e(\frak{a};\frak{f}_0)}/g_{e^1(\frak{a};\frak{f}_0)}\right) \in \mathcal{O}_{K(\frak{f}_0)_p}$ is well-defined since $g_{e(\frak{a};\frak{f}_0)}/g_{e^1(\frak{a};\frak{f}_0)} \in \mathcal{O}_{K(\frak{f}_0)_p}^{\times}$, and $\widetilde{\log}$ is as in Convention \ref{tildelogconvention}.

\item Then for any $0 \le i \le p^a-1$,
\begin{equation}\label{thickenEisenstein2}\sum_{i = 0}^{p^a-1}\eta(i)(i\alpha_{2,a})^*d\widetilde{\log}\left(\mathrm{thicken}\left(g_{e^1(\frak{a};\frak{f}_0)}\right)\right) = 12\left(\tilde{w}_{1,\eta}^{\flat}(z,\mathbf{a}^{-1}f_0L) - \mathbb{N}\frak{a} \cdot \tilde{w}_{1,\eta}^{\flat}(z,f_0L)\right).
\end{equation}
\end{enumerate}
\end{proposition}

\begin{proof}\textbf{(1)}: The function $\Theta(\Omega_0-z;\Omega_0f_0(\mathbb{Z}\tau + \mathbb{Z}),\frak{a})$ specializes to $\Theta(\Omega_0-z;L_0,\frak{a})$ upon taking $\tau = \tau_0$ (\ref{tau0}) and noting that 
$$\Omega_0f_0(\mathbb{Z}\tau_0 + \mathbb{Z}) = \Omega_0 \frak{f}_0 = L_0.$$
Hence, to prove the first equality of (\ref{thickenTheta}), it suffices to check that the following identity holds in $\mathbf{\Gamma}(F^{\mathrm{univ}})$:
$$N_{f^{\mathrm{univ}}}(\Theta(\Omega_0-z;\Omega_0f_0(\mathbb{Z}+ \mathbb{Z}\tau),\frak{a})) = \Theta(\Omega_0-z;\Omega_0f_0(\mathbb{Z}\tau + \mathbb{Z}),\frak{a})^{\phi},$$
where the superscript ``$\phi$'' denotes applying the lift of Frobenius $\phi \in \mathrm{Gal}(K(\frak{f}_0)_p/K_p)$ to the $\mathcal{O}_{K(\frak{f}_0)_p}$-coefficients in 
$$\mathbf{\Gamma}(F^{\mathrm{univ}}) \cong \mathcal{O}_{K(\frak{f}_0)_p}\llbracket T, X\rrbracket.$$
This identity itself follows from the fact that the kernel of $f^{\mathrm{univ}}$ is the canonical subgroup of $F^{\mathrm{univ}}$, and standard properties of Kato-Siegel $\Theta$-functions (\cite[Chapter II, Proposition 1.1 and p. 150]{Tsuji}). 

The second equality of (\ref{thickenTheta}) follows after observing that $\Omega_0$ maps to the $f_0$-torsion point under the uniformization $\mathcal{E}(\mathbb{C})^{\mathrm{an}}|_{\mathcal{H}^+} \cong \mathbb{C}/(\Omega_0f_0(\mathbb{Z}+ \mathbb{Z}\tau))$. 

The constant statement follows from the unique lifting statement Theorem \ref{thickenproposition} and the constant statement of Proposition \ref{Eisensteinnumber}.\\

\textbf{(2)}: From (\ref{alphabasis}) and (\ref{thickenTheta}) we have 
$$\alpha_{2,a}^*\mathrm{thicken}\left(g_{e(\frak{a};\frak{f}_0)}^{\phi^{-2a}}\right)(X) = e_{f_0p^a}^*\left(\frac{\Theta(z;\mathbf{a}^{-1}(\mathbb{Z}\tau + \mathbb{Z}))}{\Theta(z;(\mathbb{Z}\tau + \mathbb{Z}))^{\mathbb{N}\frak{a}}}\right)$$
where $e_{f_0p^a} \in \mathcal{E}[f_0p^a](Y_{\infty})$ is as in (\ref{eNpi}). Now (\ref{thickenTheta2}) follows from the fact that $[f_0]_{\mathcal{E}}(e_{f_0p^a}) = e_{2,a}$ by definition. \\

\textbf{(3)}: Every equality in (\ref{thickenEisenstein}) is formal except for the second and third equalities. For the second equality, first note that 
$$d\widetilde{\log}\left(\mathrm{thicken}(g_{e(\frak{a};\frak{f}_0)}^{\phi^{-2a}})\right) \overset{(\ref{tildedefinition})}{=} \left(1 - \frac{1}{p} \sum_{\alpha \in F^{\mathrm{univ},\phi^{-2a}}[f^{\mathrm{univ},\phi^{-2a}}]}\tau_{\alpha}^*\right)d\log\left(\mathrm{thicken}(g_{e(\frak{a};\frak{f}_0)}^{\phi^{-2a}})\right)$$
where 
$$\tau_{\alpha} : F^{\mathrm{uhiv},\phi^{-2a}} \rightarrow F^{\mathrm{univ},\phi^{-2a}}$$
denotes translation by $\alpha$, and then observe that by the moduli interpretations of $U_p$ and $V_p$ from Definition \ref{flatdefinition},
$$\alpha_{2,a}^*\left(1 - \frac{1}{p} \sum_{\alpha \in F^{\mathrm{univ},\phi^{-2a}}[f^{\mathrm{univ},\phi^{-2a+1}}]}\tau_{\alpha}^*\right) = (1-V_p^*U_p^*)\alpha_{2,a}^*.$$

The third equality of (\ref{thickenEisenstein}) follows from (\ref{thickenTheta2}). \\

%now follows from \cite[Chapter II.3.1 (7)]{deShalit} and the unique lifting statement of Theorem \ref{thickenproposition}. \\

\textbf{(4)}: The identity (\ref{thickenEisenstein2}) follows from (\ref{thickenEisenstein}), (\ref{tildewflat}) and the relation (cf. \cite[Chapter II.3.1 (7)]{deShalit})
$$\frac{d}{dz}\log\Theta(z;\mathbb{Z}\tau + \mathbb{Z}) = 12 \cdot E_1(z;\mathbb{Z}\tau + \mathbb{Z}).$$

\end{proof}

\subsection{The fundamental elliptic unit $\xi_E$}In this section, we will define a special elliptic unit that will show up in our explicit reciprocity law (\ref{explicitreciprocity}). 
%Let $\Gamma' \subset \mathrm{Gal}(\mathcal{K}_{\infty}/K)$ be the maximal $\mathbb{Z}_p$-free submodule so that we have a decomposition 
%\begin{equation}\label{Gamma'decomposition}\mathrm{Gal}(\mathcal{K}_{\infty}/K) = \Delta_K \times \Gamma'
%\end{equation}
%(see (\ref{DeltaK}) for the definition of $\Delta_K$). 

\begin{proposition}\label{DeltaKproposition} The composition 
$$\Delta_K \subset \mathrm{Gal}(\mathcal{K}_{\infty}/K) \twoheadrightarrow \mathrm{Gal}(\mathcal{K}_{2n+2a}/K)$$
factors through an \emph{isomorphism} 
$$\Delta_K \xrightarrow{\sim} \mathrm{Gal}(\mathcal{K}_{2n+2a}/(\mathcal{K}_{2n+2a} \cap K_{\infty}))$$
for all $n \ge 0$. 
\end{proposition}

\begin{proof}The restriction map
$$\mathrm{res} : \mathrm{Gal}(\mathcal{K}_{\infty}/\mathcal{K}_{2n+2a}) \rightarrow \mathrm{Gal}(K_{\infty}/(\mathcal{K}_{2n+2a}\cap K_{\infty}))$$
is surjective, and so since both the source and target are free $\mathbb{Z}_p$-modules of rank 2 it is an isomorphism. The decomposition (\ref{fixK}) thus induces an identification
$$\mathrm{Gal}(\mathcal{K}_{\infty}/(\mathcal{K}_{2n+2a} \cap K_{\infty})) = \Delta_K \times \mathrm{Gal}(K_{\infty}/(\mathcal{K}_{2n+2a} \cap K_{\infty})) \overset{\mathrm{res}^{-1}}{\cong} \Delta_K \times \mathrm{Gal}(\mathcal{K}_{\infty}/\mathcal{K}_{2n+2a}).$$
Now taking the quotient by $\mathrm{Gal}(\mathcal{K}_{\infty}/\mathcal{K}_{2n+2a})$ of both sides gives the Proposition.
\end{proof}

%Since $a \ge e/2$ (see (\ref{edefinition}) and (\ref{betadelta})), one can check that $[\mathcal{K}_{2(n+1)+2a} : \mathcal{K}_{2n+2a}] = p^2$ for all $n \ge 0$. It follows from an easy induction that the degrees $[\mathcal{K}_{2n+2a}:K_n]$ and $[(\mathcal{K}_{2n+2a} \cap K_{\infty}):K_n]$ are constant for all $n \ge 0$. 

\begin{definition}[The fundamental elliptic unit $\xi_E$]

%Since one can check that $K(\frak{f}_0p^{\beta})K_{\infty} = K(\frak{f}_0p^{\infty})$, it follows that $\mathcal{S}$ must necessarily surject onto $\Delta_{A_0}$ under the projection $\mathrm{Gal}(K(\frak{f}_0p^{\infty})/K) \rightarrow \Delta_{A_0}$ induced by (\ref{fixA}). Thus, without loss of generality, we may choose the set $\mathcal{S}$ to contain the subgroup $\Delta_{A_0}$. 

%View $\lambda_E$ as a $p$-adic Galois character 
%$$\lambda_E : \mathrm{Gal}(K(\frak{f}_0p^{\infty})/K) \overset{(\ref{fixA})}{=} \Delta_{A_0} \times \Gamma_K \rightarrow \overline{\mathbb{Q}}_p^{\times}$$
%using Artin reciprocity. By slight abuse of notation, we will let $\chi_E = \lambda_E|_{\Delta_{A_0}}$ for the remainder of this definition. This is compatible with the notation of Definition \ref{decompositionchoices}, since $\Delta_{A_0} \twoheadrightarrow \Delta_K$ under the natural projection $\mathrm{Gal}(K(\frak{f}_0p^{\infty})/K) \twoheadrightarrow \mathrm{Gal}(\mathcal{K}_{\infty}/K)$. 
\begin{enumerate}
\item Recall the map $\mathrm{Nm}_{\frak{f}_0} : \mathbb{U}^1(\frak{f}_0) \rightarrow \mathbb{U}^1$ from (\ref{UEnorm}) and the element $e^1(\frak{a};\frak{f}_0) \in \mathbb{U}^1(\frak{f}_0)$ from (\ref{ellipticunitea}). 
Recall $\chi_E$ from (\ref{chiE}) and let $e_{\chi_E}$ be the associated projector from (\ref{echiA}).
Then 
$$e_{\chi_E}(e^1(\frak{a};\frak{f}_0)) \in \mathbb{U}^1(\frak{f}_0)_{\chi_E}.$$
\item By Proposition \ref{lambdafactorproposition}, the map (\ref{UEnorm}) descends to the $\chi_E$-isotypic component
$$\mathrm{Nm}_{\frak{f}_0,\chi_E} : \mathbb{U}^1(\frak{f}_0)_{\chi_E} \rightarrow \mathbb{U}_{\chi_E}^1.$$
Now define the \emph{fundamental elliptic unit attached to $E$}
\begin{equation}\label{xiE}\begin{split}\xi_E%&= \left(\mathrm{Nm}_{\frak{f}_0,\chi_E}(e_{\chi_E}(e^1(\frak{a};\frak{f}_0))) \right)\otimes 1\\
&:= \left(\mathrm{Nm}_{\frak{f}_0,\chi_E}(e_{\chi_E}(e^1(\frak{a};\frak{f}_0))) \right)\otimes 1 \in \mathbb{U}_{\chi_E}^1\otimes_{\mathcal{O}_{K_p}}\mathcal{O}_{K_p}(\lambda_E^{-1}) = \mathbb{U}_{\chi_E}^1(\lambda_E^{-1}).
\end{split}
\end{equation}
%Here, the second equality follows from the fact that the $n^{\mathrm{th}}$-indexed term of the inverse limit element $e^1(\frak{a};\frak{f}_0p^a) \in \mathbb{U}^1(\frak{f}_0p^a)$ is by definition equal to the $(n+2a)^{\mathrm{th}}$-indexed term of the inverse limit element $e^1(\frak{a};\frak{f}_0) \in \mathbb{U}^1(\frak{f}_0)$, and the last arrow in the definition (\ref{UEnorm}). 

%Now define
%$$\xi_E^0 := e_{\chi_E}(e_E^1(\frak{a})) \otimes 1 \in \mathbb{U}_{\chi_E}^1 \otimes_{\mathcal{O}_{K_p}}\mathcal{O}_{K_p}(\lambda_E^{-1}) = \mathbb{U}_{\chi_E}^1(\lambda_E^{-1}).$$
%Define
%\begin{equation}\label{xiE0}\begin{split}\xi_E^0 :&= \left(\sum_{\sigma \in \mathcal{S}}\lambda_E^{-1}(\sigma) e_E^1(\frak{a})^{\sigma}\right) \otimes 1 \\
%&= \left(\sum_{m = 0}^{k_1}\sum_{n = 0}^{k_2}\lambda_E^{-1}(\gamma_1^m\gamma_2^n)e_{\chi_E}\left(e_E^1(\frak{a})^{\gamma_1^m\gamma_2^n}\right)\right) \otimes 1
%\\
%& \in \mathbb{U}_{\chi_E}^1 \otimes_{\mathcal{O}_{K_p}}\mathcal{O}_{K_p}(\lambda_E^{-1}) = \mathbb{U}_{\chi_E}^1(\lambda_E^{-1}).
%\end{split}
%\end{equation}
%where the inclusion follows because $\sum_{m = 0}^{k_1}\sum_{n = 0}^{k_2}\lambda_E^{-1}(\gamma_1^m\gamma_2^n)e_{\chi_E}\left(e_E^1(\frak{a})^{\gamma_1^m\gamma_2^n}\right)$ is evidently an $\mathcal{O}_{K_p}$-linear combination of elements of the form $e_{\chi}(\beta)$ for $\beta \in \mathbb{U}^1$, and so belongs to $\mathbb{U}_{\chi_E}^1$.

\item Recall the map $i_{\frak{f}_0,\chi_E} : \mathbb{U}_{\chi_E}^1(\lambda_E^{-1})\rightarrow \mathbb{U}^1(\frak{f}_0)_{\chi_E}(\lambda_E^{-1})$ from (\ref{1ftwist}). We have 
$$i_{\frak{f}_0,\chi_E}\left(\xi_E\right) \in \mathbb{U}^1(\frak{f}_0)_{\chi_E}(\lambda_E^{-1}).
$$

%Applying the isotypic projector $e_{\chi_E}$ from (\ref{echiK}), we have
%$$e_{\chi_E}\left(\mathrm{Nm}_{\frak{f}_0p^a}(e^1(\frak{a};\frak{f}_0p^a))\right) \in \mathbb{U}_{\chi_E}^1,$$
%which under the twisting map $\mathbb{U}_{\chi_E}^1 \rightarrow \mathbb{U}_{\chi_E}^1 \otimes_{\mathcal{O}_{K_p}}\mathcal{O}_{K_p}(\lambda_E^{-1}) = \mathbb{U}_{\chi_E}^1(\lambda_E^{-1})$ gives an element
%$$e_{\chi_E}\left(\mathrm{Nm}_{\frak{f}_0p^a}(e^1(\frak{a};\frak{f}_0p^a))\right) \otimes 1 \in  \mathbb{U}_{\chi_E}^1(\lambda_E^{-1}).$$

%\begin{equation}\begin{split}
%\frac{1}{\#\Delta_{A_0}}\left(\sum_{\sigma \in \mathcal{S}}\lambda_E^{-1}(\sigma)\cdot e(\frak{a};\frak{f}_0p^a)^{\sigma}\right) \otimes 1\\
%&= \left(\sum_{\sigma \in \Delta_{A_0}\backslash \mathcal{S}}(\lambda_E/\chi_E)^{-1}(\sigma)\left(\frac{1}{\#\Delta_{A_0}}\sum_{\delta \in \Delta_{A_0}}\chi_E^{-1}(\delta)e(\frak{a};\frak{f}_0p^a)^{\sigma\delta}\right)\right) \otimes 1 \\
%& = \left(\sum_{\sigma \in \Delta_{A_0}\backslash \mathcal{S}}(\lambda_E/\chi_E)^{-1}(\sigma)e_{\chi_E}\left(e(\frak{a};\frak{f}_0p^a)^{\sigma}\right)\right) \otimes 1\\
%&\in \mathbb{U}_{\chi_E}^1 \otimes_{\mathcal{O}_{K_p}}\mathcal{O}_{K_p}(\lambda_E^{-1}) = \mathbb{U}_{\chi_E}^1(\lambda_E^{-1})
%\end{split}
%\end{equation}
%where $e(\frak{a};\frak{f}_0p^a)^{\sigma} = \sigma(e(\frak{a};\frak{f}_0p^a))$. 
\end{enumerate}
\end{definition}

Recall the map $d\mathrm{Log} : \mathbb{U}_{\chi_E}^1(\lambda_E^{-1}) \rightarrow \mathbf{\Gamma}(\Omega_{F^{\mathrm{univ}}/\mathrm{LT}})[1/p]$ from (\ref{EtwistColeman}) and the section $\mathbf{e} : \mathrm{LT}_{\infty}^{\phi^{-2a}}\rightarrow F_{\infty}^{\mathrm{univ},\phi^{-2a}}$ from (\ref{esection2}) below. We have 
$$d\mathrm{Log}(\xi_E) \in \mathbf{\Gamma}(F^{\mathrm{univ}})[1/p] \subset \mathbf{\Gamma}(F_{\infty}^{\mathrm{univ}}).$$
We will later consider the $q_{\mathrm{dR}}$-expansion of the pulled back section
$$\mathbf{e}^*\left(d\mathrm{Log}(\xi_E)^{\phi^{-2a}}\right) \in \mathbf{\Gamma}(\mathrm{LT}_{\infty}^{\phi^{-2a}}),$$
where $d\mathrm{Log}(\xi_E)^{\phi^{-2a}}$ denotes the extension of scalars of $d\mathrm{Log}(\xi_E)$ by $\phi^{-2a}$ (see (\ref{Frobeniusapplication}) below). In our explicit reciprocity law (\ref{explicitreciprocity}), we will show that after specializing the coefficients of this $q_{\mathrm{dR}}$-expansion to the CM point $y$ from (\ref{usey2}), we essentially obtain the $p$-adic $L$-function $\mathcal{L}_{\lambda_E}$ from Theorem \ref{interpolation2}.

\subsection{Preparations for defining the reciprocity map}
In this section, we make a few preparations for defining the reciprocity map in Section \ref{thereciprocitymapsection}.

\begin{convention}\label{Fconvention}For the rest of the paper, we will work with the choice of $F$ from Example \ref{Fchoice}. Thus the associated Lubin-Tate space $\mathrm{LT}$ classifies deformations of the special fiber $\tilde{\hat{A}}_0 := \hat{A}_0 \pmod{\varpi'\mathcal{O}_{K(\frak{f}_0)_p}}$ of the formal group $\hat{A}_0/\mathcal{O}_{K(\frak{f}_0)_p}$.
\end{convention}

\begin{assumption}\label{Yassumption}For the rest of the paper, let $Y$ be as in Convention \ref{Yconvention} subject to the condition $f_0|N$ from Assumption \ref{Nf0assumption}. %(See also Remark \ref{Ncompatibleremark} if one wishes to impose a further convenient, but not strictly necessary, compatibility on $N$.)
\end{assumption}

\begin{choice}\label{rchoice2}Henceforth, make the choice of $\epsilon_0 = 1/2$ in (\ref{rchoice}). Since $p$ is ramified in $K$, 1/2 is in the valuation group of $K_p \subset k = \widehat{\mathbb{Q}_p(\mu_{p^{\infty}})}$. Note that the last inclusion follows from our assumption that $K$ has class number 1, so that $K = \mathbb{Q}(\sqrt{-p})$ with $p \equiv 3 \pmod{4}$ or $K = \mathbb{Q}(i)$ and thus $K \subset \mathbb{Q}(\mu_{p^{\infty}})$.
\end{choice}

Let 
\begin{equation}\label{usey}y = (A,e_1,e_2) \overset{(\ref{YIgin1})}{\in} \mathcal{Y}^{\mathrm{Ig}}(1/2)(\overline{\mathbb{Q}}_p,\overline{\mathbb{Z}}_p)
\end{equation}
be as in Choice \ref{choice} with the choice of $\epsilon_0 = 1/2$ from Choice \ref{rchoice2} (we continue to suppress the $\Gamma = \Gamma(N)$-level structure per Convention \ref{tameconvention}), and view $y \in \mathrm{LT}_{\infty}$ using the fact that $Y_{\infty}^{\mathrm{ss}} \cong \bigsqcup \mathrm{LT}_{\infty}$, i.e. the supersingular locus $Y_{\infty}^{\mathrm{ss}}$ of $Y_{\infty}$ is a finite union of adic generic fibers $\mathrm{LT}_{\infty}$ of infinite-level Lubin-Tate towers. (See \cite[discussion after Theorem III.1.2]{ScholzeTorsion}.) Under the inclusion $\mathrm{LT}_{\infty} \subset Y_{\infty}^{\mathrm{ss}} \subset \mathcal{V}_x$, we naturally identify the formal group $F_{\infty}^{\mathrm{univ}} \rightarrow \mathrm{LT}_{\infty}$ with the formal group attached to the connected $p$-divisible group $\mathcal{E}[p^{\infty}]|_{\mathrm{LT}_{\infty}}$, where $\mathcal{E} \rightarrow Y_{\infty}$ is the universal object and $\mathcal{E}[p^{\infty}]$ is as in Convention \ref{idempotentconvention} (i.e. $\mathcal{E}[p^{\infty}]$ is the $p$-divisible group obtained by applying $e^1$ to the usual $p$-divisible group of $\mathcal{E}$).

Here, as we are working in the split quaternion algebra case (i.e. $D = M_2(\mathbb{Q})$), $A$ is an elliptic curve defined over $K$ with CM by $\mathcal{O}_K$. Since $A$ was allowed to be any such elliptic curve, we henceforth fix the following convenient choice. 

\begin{choice}\label{AEchoice}Recall our elliptic curve $E/\mathbb{Q}$ with CM by $\mathcal{O}_K$ from Assumption \ref{pconductorassumption} (3). Henceforth take the choice $A = E$ in Choice \ref{choice}. Thus
$$y = (E,e_1,e_2) \in \mathcal{Y}^{\mathrm{Ig}}(1/2)(\overline{\mathbb{Q}}_p,\overline{\mathbb{Z}}_p).$$
\end{choice} 

Letting $E^+$ be a minimal local good integral model of $E$, which is defined over a finite extension $\mathcal{O}'$ of $\mathcal{O}_{K_p}$, we get a point $y^+ = (E^+,e_1^+,e_2^+) \in\mathcal{Y}^{\mathrm{Ig}}(1/2)^+(\overline{\mathbb{Z}}_p,\overline{\mathbb{Z}}_p)$, where $(e_1^+,e_2^+) : \mathbb{Z}_p^{\oplus 2} \rightarrow T_pE^+$ is a Drinfeld level structure. Let $A_0$ be as in Choice \ref{FixCMdefinition} for $\frak{f}_0 = \frak{f}_E^{(p)}$ (see Choice \ref{f0choice}). 
 Then since $\lambda_E/\lambda_{A_0}$ is a finite order character, $E$ and $A_0$ are twists of each other i.e. isomorphic over $\overline{\mathbb{Q}}$. Thus the $p$-adic integral model $E^+$ and the $p$-adic integral model $A_0^+$ of $A_0$ reduce to the same point in $Y^+(\overline{\mathbb{F}}_p)$. Equivalently, $E$ and $A_0$ lie in the same $p$-adic residue disc in $Y$, and thus belong to the same Lubin-Tate deformation space $\mathrm{LT}(F_0)$ where 
 $$F_0 = \tilde{\hat{A}}_0 := \hat{A}_0 \pmod{\varpi'\mathcal{O}_{K(\frak{f}_0)_p}}.$$
 In the notation of Definition \ref{esectiondefinitions} for $F_0 = \tilde{\hat{A}}_0$, we thus have
$$y \in \mathrm{LT}_{\infty}(\overline{\mathbb{Q}}_p,\overline{\mathbb{Z}}_p),$$
and so by (\ref{usey}) we have 
\begin{equation}\label{usey2}y \in (\mathrm{LT}_{\infty} \cap \mathcal{Y}^{\mathrm{Ig}}(1/2))(\overline{\mathbb{Q}}_p,\overline{\mathbb{Z}}_p).
\end{equation}

Recall that 
$$F^{\mathrm{univ}} \rightarrow \mathrm{LT}$$ 
denotes the universal deformation of $F_0 = F \pmod{\varpi'} = \hat{A}_0 \pmod{\varpi'\mathcal{O}_{K(\frak{f}_0)_p}} = \tilde{\hat{A}}_0$ (see Example \ref{Fchoice}). Recall the notation of Definition \ref{esectiondefinitions}. 
In particular,
$$F^{\mathrm{univ},\phi^m} \rightarrow \mathrm{LT}^{\phi^m}$$
is the universal deformation of $F_0^{\phi^m} = \tilde{\hat{A}}_0^{\phi^m}$. 

\begin{definition}\label{lambdaFunivdefinition} For any $m \in \mathbb{Z}$, define an inclusion
\begin{equation}\label{bracketFuniv}[\cdot ]_{F^{\mathrm{univ},\phi^m}} : \mathrm{Gal}(K(\frak{f}_0p^{\infty})/K(\frak{f}_0)) \hookrightarrow \mathrm{Aut}(F^{\mathrm{univ},\phi^m}).
\end{equation}
as follows. 

\begin{enumerate}
\item By the main theorem of complex multiplication (\cite[Chapter II.1.3 (11)]{deShalit}), we have an inclusion
$$\mathrm{Gal}(K(\frak{f}_0p^{\infty})/K(\frak{f}_0)) = \mathrm{Gal}(K(\frak{f}_0)(A_0^{\phi^m}[\frak{f}_0p^{\infty}])/K(\frak{f}_0)) \subset \mathrm{Aut}(A_0^{\phi^m}) \hookrightarrow \mathrm{Aut}(\tilde{\hat{A}}_0^{\phi^m}),$$
where the last inclusion follows as endomorphisms of $A_0^{\phi^m}$ are determined by their special fibers by the N\'{e}ron mapping property. As $F^{\mathrm{univ},\phi^m}$ is the universal deformation of $\tilde{\hat{A}}_0^{\phi^m}$, from the universal property (\cite[Section 1.1]{Tsuji}) we get inclusions
$$\mathrm{Aut}(\tilde{\hat{A}}_0^{\phi^m}) \hookrightarrow \mathrm{Aut}(F^{\mathrm{univ},\phi^m}/W_{\mathcal{O}_{K_p}}), \hspace{1cm} \mathrm{Aut}(\tilde{\hat{A}}_0^{\phi^m}) \hookrightarrow \mathrm{Aut}(\mathrm{LT}^{\phi^m}/W_{\mathcal{O}_{K_p}})$$
where $\mathrm{Aut}(X/W_{\mathcal{O}_{K_p}})$ denotes the group of automorphisms of a formal scheme $X \rightarrow \mathrm{Spf}(W_{\mathcal{O}_{K_p}})$ that preserve the structure map $X \rightarrow \mathrm{Spf}(W_{\mathcal{O}_{K_p}})$. 
\item The above inclusions are compatible (i.e. are $W_{\mathcal{O}_{K_p}}$-linear) with respect to the natural map 
$$\mathrm{Aut}(\mathrm{LT}^{\phi^m}/W_{\mathcal{O}_{K_p}}) \rightarrow \mathrm{Aut}(F^{\mathrm{univ},\phi^m}/W_{\mathcal{O}_{K_p}})$$
induced by the $W_{\mathcal{O}_{K_p}}$-linear map $F^{\mathrm{univ},\phi^m} \rightarrow \mathrm{LT}^{\phi^m}$.

\item The composition of the last two inclusions is our desired inclusion $[\cdot]_{F^{\mathrm{univ},\phi^m}}$. Using
$$\Gamma_K \overset{(\ref{fixA})}{\subset}  \mathrm{Gal}(K(\frak{f}_0p^{\infty})/K(\frak{f}_0)),$$
we may also consider the restrictions
$$[\cdot ]_{F^{\mathrm{univ},\phi^m}} : \Gamma_K \hookrightarrow \mathrm{Aut}(F^{\mathrm{univ},\phi^m}/W_{\mathcal{O}_{K_p}}), \hspace{1cm} [\cdot]_{\mathrm{LT}^{\phi^m}} : \Gamma_K  \hookrightarrow \mathrm{Aut}(\mathrm{LT}^{\phi^m}/W_{\mathcal{O}_{K_p}}).$$
\end{enumerate}
\end{definition}

In fact, under the global Artin reciprocity isomorphism $\mathcal{O}_{K_p}^{\times} \cong \mathrm{Gal}(K(\frak{f}_0p^{\infty})/K(\frak{f}_0))$, the action $[\cdot]_{F^{\mathrm{univ},\phi^{m}}}$ coincides with the $\mathcal{O}_{K_p}^{\times}$ action on $F^{\mathrm{univ},\phi^{m}}$ defined by the universal property (see \cite[Section 1.1]{Tsuji}). Since the action of $\phi$ on $F^{\mathrm{univ}}$ commutes with the $\mathcal{O}_{K_p}^{\times}$-action on $F^{\mathrm{univ}}$ (see loc. cit.; one can also see this commutativity using \cite[I.3.7 (14)]{deShalit}), we have 
\begin{equation}\label{bracketFunivsequal}[\cdot ]_{F^{\mathrm{univ}}} = [\cdot ]_{F^{\mathrm{univ},\phi^m}}
\end{equation}
for all $m \in \mathbb{Z}$.

\begin{definition}\label{gamma12definition}Let $\gamma_1,\gamma_2$ be a $\mathbb{Z}_p$-basis of $\Gamma_K \cong \mathbb{Z}_p^{\oplus 2}$. 
\end{definition}
Using Propoposition \ref{DeltaKproposition}, we have an isomorphism 
\begin{equation}\label{K2adecomposition}\mathrm{Gal}(\mathcal{K}_{2a}/K) = \Delta_K \times \mathrm{Gal}((\mathcal{K}_{2a} \cap K_{\infty})/K).
\end{equation}
Using (\ref{fixK}), we thus see $\mathrm{Gal}((\mathcal{K}_{2a} \cap K_{\infty})/K)$ as a quotient of $\Gamma_K$ and so can write 
\begin{equation}\label{K2adecompositionbasis}\mathrm{Gal}((\mathcal{K}_{2a} \cap K_{\infty})/K) = \frac{\gamma_1^{\mathbb{Z}_p}}{\gamma_1^{p^{k_1}\mathbb{Z}_p}} \times \frac{\gamma_2^{\mathbb{Z}_p}}{\gamma_2^{p^{k_2}\mathbb{Z}_p}} \cong (\mathbb{Z}/p^{k_1}) \times (\mathbb{Z}/p^{k_2})
\end{equation}
for some $k_1,k_2 \in \mathbb{Z}_{\ge 0}$. (Here, recall $\mathcal{K}_n$ is defined in Definition \ref{mathcalKndefinition} and $a$ is as in (\ref{adefinition}).)

\subsection{The reciprocity map}\label{thereciprocitymapsection}
In this section, we use the material developed in the previous sections to construct a $\Lambda_{\mathcal{O}_{K_p}}[1/p]$-linear map 
$$\delta : \mathbb{U}_{\chi_E}^1(\lambda_E^{-1}) \rightarrow \mathbb{C}_p\llbracket q_{\mathrm{dR}}-1\rrbracket,$$
where the action of $\Lambda_{\mathcal{O}_{K_p}}[1/p]$ on the target is through 
$$\Lambda_{\mathcal{O}_{K_p}}[1/p] \twoheadrightarrow \Lambda_{\mathcal{O}_{K_p}}[1/p]/(\Gamma_K-1) = K_p \subset \mathbb{C}_p$$
(see Definition \ref{Lambdadefinition}). 
%Here $\Lambda_{\mathcal{O}_{K_p}}[1/p]$ acts on the target through $\Lambda_{\mathcal{O}_{K_p}}[1/p] \twoheadrightarrow \Lambda_{\mathcal{O}_{K_p}}[1/p]/(\Gamma_K-1) = K_p$. 
One may think of this $\delta$ as a ``$q_{\mathrm{dR}}$-thickening'' of the reciprocity map $\delta$ from \cite{RubinMC}, or else that of the logarithmic derivative map $\delta$ from \cite{Rubin3}. (The two maps of both op. cit. are in fact equal up to a nonzero multiple by Wiles's reciprocity law, \cite[Chapter I.4]{deShalit}.) We will also define certain $p$-adic Maass-Shimura derivatives $D_1^j\delta$ of our $\delta$.

For brevity, let $\mathbf{\Gamma}(X)$ denote the coordinate ring of an affinoid adic space $X$. Also recall the notation of Convention \ref{Gammaconvention}; i.e., for sheaf $\mathcal{F}$, $\mathbf{\Gamma}(\mathcal{F})$ denotes its global sections.  %For brevity, let $\mathbf{\Gamma}(X)$ denote the coordinate ring of an affinoid adic space $X$. 

\begin{definition}Recall the torsion point 
$$\alpha_{2,a} : \mathrm{LT}_{\infty} \rightarrow F_{\infty}^{\mathrm{univ},\phi^{-2a}}$$
from (\ref{alphabasis}), so that using Definition (\ref{bracketFuniv}) we get a torsion point 
$$[\gamma]_{F^{\mathrm{univ},\phi^{-2a}}}(\alpha_{2,a}) : \mathrm{LT}_{\infty} \rightarrow F_{\infty}^{\mathrm{univ},\phi^{-2a}}$$
for any $\gamma \in \Gamma_K$. Let $\gamma_1, \gamma_2 \in \Gamma_K$ from Definition \ref{gamma12definition}. Define
\begin{equation}\label{esection2}\begin{split} \mathbf{e} &:= \sum_{m = 0}^{p^{k_1}-1}\sum_{n = 0}^{p^{k_2}-1}[\gamma_1^m\gamma_2^n]_{F^{\mathrm{univ},\phi^{-2a}}}(\alpha_{2,a}) \otimes_{\mathbb{Z}_p} \lambda_E^{-1}(\gamma_1^m\gamma_2^n) \in F_{\infty}^{\mathrm{univ},\phi^{-2a}}[p^a](\mathrm{LT}_{\infty}) \otimes_{\mathbb{Z}_p}\mathcal{O}_{K_p}.
\end{split}
\end{equation}
Using the pullback 
$$[\gamma_1^m\gamma_2^n]_{F^{\mathrm{univ},\phi^{-2a}}}(\alpha_{2,a})^* : \mathbf{\Gamma}(F_{\infty}^{\mathrm{univ},\phi^{-2a}}) \rightarrow \mathbf{\Gamma}(\mathrm{LT}_{\infty}),$$
we get an induced map
$$\mathbf{e}^* : \mathbf{\Gamma}(F_{\infty}^{\mathrm{univ},\phi^{-2a}}) \rightarrow \mathbf{\Gamma}(\mathrm{LT}_{\infty})$$
which in turn induces
$$\mathbf{e}^* : \mathbf{\Gamma}(\Omega_{F^{\mathrm{univ},\phi^{-2a}}/\mathrm{LT}_{\infty}^{\phi^{-2a}}}|_{F_{\infty}^{\mathrm{univ},\phi^{-2a}}}) \rightarrow \mathbf{\Gamma}(\omega|_{\mathrm{LT}_{\infty}})$$
where $\omega$ is the Hodge bundle from (\ref{omegaY}). 
\end{definition}

In the above Definition, we employ the specific power $\phi^{-2a}$ of Frobenius $\phi$ in order to make use of Proposition \ref{ThickenColemanspecialize}.

%Recall the unique normalized differential $w_{F^{\mathrm{univ}}} \in \Omega_{F^{\mathrm{univ}}/\mathrm{LT}}$ from (\ref{wFuniv}). Recall our previously fixed CM point $y = (A,e_1,e_2) \in Y_{\infty}(\epsilon_0)$ where $1/2 \le \epsilon_0 < p/(p+1)$. Suppose $\mathrm{LT} \subset Y$ is the unique Lubin-Tate space with $y \in \mathrm{LT}_{\infty}(\epsilon_0)$, and let $w_0$ in (\ref{fixdifferential}) be such that
%$$w_0 = w_{F^{\mathrm{univ}}}(y).$$

\begin{definition}
\begin{enumerate}
\item Recall the notation of Definition \ref{esectiondefinitions}, and recall that 
$$\frac{1}{p} \in \mathrm{\Gamma}(F_{\infty}^{\mathrm{univ}})$$
by Convention \ref{infiniteleveladicgenericfiberconvention}, so that we have a natural map 
$$\mathbf{\Gamma}(F^{\mathrm{univ}})[1/p] \rightarrow \mathbf{\Gamma}(F_{\infty}^{\mathrm{univ}}).$$
Recall the map $$d\mathrm{Log} : \mathbb{U}^1(\frak{f}_0)_{\chi_E}(\lambda_E^{-1})  \rightarrow \mathbf{\Gamma}(F^{\mathrm{univ}})[1/p]$$ from (\ref{EtwistColeman}). Recalling $W_{\mathcal{O}_{K_p}}$ from (\ref{Wcompositum}), let 
\begin{equation}\label{Frobeniusapplication}\begin{split}\phi^{-2a} : \mathbf{\Gamma}(\Omega_{F^{\mathrm{univ}}/\mathrm{LT}}|_{F_{\infty}^{\mathrm{univ}}}) &\rightarrow \mathbf{\Gamma}(\Omega_{F^{\mathrm{univ}}/\mathrm{LT}}|_{F_{\infty}^{\mathrm{univ}}})\otimes_{W_{\mathcal{O}_{K_p}},\phi^{-2a}} W_{\mathcal{O}_{K_p}}\\
&\cong \mathbf{\Gamma}(\Omega_{F^{\mathrm{univ},\phi^{-2a}}/\mathrm{LT}^{\phi^{-2a}}}|_{F_{\infty}^{\mathrm{univ},\phi^{-2a}}})
\end{split}
\end{equation}
denote the map $f \mapsto f \otimes 1$. In other words, one applies $\phi^{-2a} : W_{\mathcal{O}_{K_p}} \xrightarrow{\sim} W_{\mathcal{O}_{K_p}}$ to the $W_{\mathcal{O}_{K_p}}$-coefficients of $w \in  \mathbf{\Gamma}(\Omega_{F^{\mathrm{univ}}/\mathrm{LT}}|_{F_{\infty}^{\mathrm{univ}}})$. 
\item Define a composition of maps
\begin{equation}\label{delta0}\begin{split}\delta_0 : \mathbb{U}_{\chi_E}^1(\lambda_E^{-1}) & \xrightarrow{d\mathrm{Log}} \mathbf{\Gamma}(\Omega_{F^{\mathrm{univ}}/\mathrm{LT}})[1/p] \rightarrow \mathbf{\Gamma}(\Omega_{F^{\mathrm{univ}}/\mathrm{LT}}|_{F_{\infty}^{\mathrm{univ}}})\\
&\xrightarrow{\phi^{-2a}} \mathbf{\Gamma}(\Omega_{F^{\mathrm{univ},\phi^{-2a}}/\mathrm{LT}^{\phi^{-2a}}}|_{F_{\infty}^{\mathrm{univ},\phi^{-2a}}})\xrightarrow{\mathbf{e}^*}\mathbf{\Gamma}(\omega|_{\mathrm{LT}_{\infty}}) \\&\xrightarrow{w \mapsto w/w_{\mathrm{can}}}\mathbb{B}_{\mathrm{dR},\mathcal{V}_x}^+(\mathrm{LT}_{\infty})\llbracket q_{\mathrm{dR}}-1\rrbracket.
\end{split}
\end{equation}
Here $\omega$ is the Hodge bundle from (\ref{omegaY}) and the map $w \mapsto w/w_{\mathrm{can}}$ is defined by dividing any $w \in \mathbf{\Gamma}(\omega|_{\mathrm{LT}_{\infty}})$ by the restriction $w_{\mathrm{can}}|_{\mathrm{LT}_{\infty}}$ to the open subset $\mathrm{LT}_{\infty}\subset \mathcal{V}_x$ of the generator 
$$w_{\mathrm{can}} \in \omega \otimes_{\mathcal{O}_Y}\mathbb{B}_{\mathrm{dR},\mathcal{V}_x}^+(\mathcal{V}_x)\llbracket q_{\mathrm{dR}}-1\rrbracket$$
from Proposition \ref{wcangeneratorproposition3}. 
%where the last arrow is the extension of scalars
%$$\mathbf{\Gamma}(\mathrm{LT}_{\infty}(1/2))  \rightarrow \mathbf{\Gamma}(\mathrm{LT}_{\infty}(1/2))  \otimes_{\mathcal{O}_{K(\frak{f})_p},\phi^{-2a}}\mathcal{O}_{K(\frak{f}_0)_p} \xrightarrow{\sim} \mathbf{\Gamma}(\mathrm{LT}_{\infty}(1/2)),$$
%where $\phi^{-2a} : \mathcal{O}_{K(\frak{f}_0)_p} \xrightarrow{\sim} \mathcal{O}_{K(\frak{f}_0)_p}$ is the $(-2a)^{\mathrm{th}}$ power of the Frobenius generator $\phi \in \mathrm{Gal}(K(\frak{f}_0)_p/K_p)$ from Section \ref{Colemansection} (with $L_p = K(\frak{f}_0)_p$, see Choice \ref{Fchoice}).
\end{enumerate}
\end{definition}

We now take the $q_{\mathrm{dR}}$-expansion of $\delta_0$. Recall the definition of $q_{\mathrm{dR}}$-expansions from Definition \ref{'zqfunctionexpansions}.  % and that 
%$$\mathbf{U}' = \mathcal{V}_x \supset U = \mathcal{Y}^{\mathrm{Ig}}(1/2)$$
%by Theorem \ref{U'Utheorem} and Definition \ref{Udefinition} (with the choice of $\epsilon_0 = 1/2$ from Choice \ref{rchoice2}). Given $\beta \in \mathbb{U}_{\chi_E}^1(\lambda_E^{-1})$, we have $\delta_0(\beta) \in \mathbf{\Gamma}(\mathrm{LT}_{\infty}^{\phi^{-2a+1}})$. Applying Definition \ref{'zqexpansions} to the affinoid open
Applying this definition to the section $\delta_0(\beta)$ defined on the open subset
$$W = \mathrm{LT}_{\infty} \subset \mathcal{V}_x \overset{(\ref{U'U})}{=} \mathbf{U}',$$
we get 
$$\delta_0(\beta)(q_{\mathrm{dR}}) \in \mathbb{B}_{\mathrm{dR},\mathcal{V}_x}^+(\mathrm{LT}_{\infty})\llbracket q_{\mathrm{dR}}-1\rrbracket.$$
Recalling the $p$-adic Maass-Shimura operator $d_k^j : \mathbb{B} \rightarrow \mathbb{B}$ from (\ref{dkjdefinition}) (noting $\mathrm{LT}_{\infty} \subset \mathcal{V}_x$ is open), we get
$$d_k^j : \mathbb{B}_{\mathrm{dR},\mathcal{V}_x}(\mathrm{LT}_{\infty})\llbracket q_{\mathrm{dR}}-1\rrbracket \rightarrow  \mathbb{B}_{\mathrm{dR},\mathcal{V}_x}(\mathrm{LT}_{\infty})\llbracket q_{\mathrm{dR}}-1\rrbracket.$$
Thus we have 
$$d_1^j\delta_0(\beta)(q_{\mathrm{dR}}) \in \mathbb{B}_{\mathrm{dR},\mathcal{V}_x}(\mathrm{LT}_{\infty})\llbracket q_{\mathrm{dR}}-1\rrbracket$$
for every $j \in \mathbb{Z}_{\ge 0}$. %Multiplying this by 
%$$\Omega_{\mathrm{LT}} \in  \mathbb{B}_{\mathrm{dR},\mathcal{V}_x}^+(\mathrm{LT}_{\infty}^{\phi^{-2a+1}})\llbracket q_{\mathrm{dR}}-1\rrbracket \subset \mathbb{B}_{\mathrm{dR},\mathcal{V}_x}(\mathrm{LT}_{\infty}^{\phi^{-2a+1}})\llbracket q_{\mathrm{dR}}-1\rrbracket$$
%from (\ref{OmegaLT}), we define
%\begin{equation}\label{deltadelta0}\delta(\beta)(q_{\mathrm{dR}}) := \delta_0(\beta)(q_{\mathrm{dR}}) \cdot \Omega_{\mathrm{LT}} \in \mathbb{B}_{\mathrm{dR},\mathcal{V}_x}(\mathrm{LT}_{\infty}^{\phi^{-2a+1}})\llbracket q_{\mathrm{dR}}-1\rrbracket
%\end{equation}
%and thus have
%$$d_1^j\left(\delta(\beta)(q_{\mathrm{dR}})\right) \in \mathbb{B}_{\mathrm{dR},\mathcal{V}_x}(\mathrm{LT}_{\infty}^{\phi^{-2a+1}})\llbracket q_{\mathrm{dR}}-1\rrbracket.$$

Now let
$$\theta_t\left(d_1^j\left(\delta(\beta)(q_{\mathrm{dR}})\right)\right) \in \hat{\mathcal{O}}_{\mathcal{V}_x}(\mathrm{LT}_{\infty})\llbracket q_{\mathrm{dR}}-1\rrbracket$$
be as in Definition \ref{thetatpowerseriesdefinition}. Now recall our point $y \in (\mathrm{LT}_{\infty} \cap \mathcal{Y}^{\mathrm{Ig}}(1/2))(\overline{\mathbb{Q}}_p,\overline{\mathbb{Z}}_p)$ from (\ref{usey2}). Let
$$\theta_t(d_1^j\delta(\beta)(y)(q_{\mathrm{dR}})) \in \mathbb{C}_p\llbracket q_{\mathrm{dR}}-1\rrbracket$$
be obtained by evaluating the coefficients of $\theta_t(d_1^j\delta(\beta)(q_{\mathrm{dR}}))$ along the specialization 
$$\hat{\mathcal{O}}_{\mathrm{LT}_{\infty}}(\mathrm{LT}_{\infty}) \rightarrow \hat{\mathcal{O}}_{\mathrm{LT}_{\infty}}(y) \subset \mathbb{C}_p.$$

%We will need to consider twists by powers of a certain Galois character $\rho : \Gamma_K \rightarrow \mathcal{O}_{K_p}^{\times}$ defined below. Recall the notation $\lambda_E^c = \lambda_E \circ c$, where $c : \mathbb{A}_K^{\times} \xrightarrow{\sim} \mathbb{A}_K^{\times}$ is complex conjugation. View $\lambda_E$ as a $p$-adic Hecke character $\lambda_E : \mathbb{A}_K^{\times}/K^{\times} \rightarrow \mathcal{O}_{K_p}^{\times}$ and then as a Galois character $\lambda_E : \mathrm{Gal}(\mathcal{K}_{\infty}/K) \rightarrow \mathcal{O}_{K_p}^{\times}$ using Artin reciprocity. %By restriction to $\Gamma_K \overset{(\ref{fixK})}{\subset} \mathrm{Gal}(\mathcal{K}_{\infty}/K)$, we may view $\lambda_E : \Gamma_K \rightarrow \mathcal{O}_{K_p}^{\times}$. 

%\begin{definition}\label{rhodefinition}Define a character
%$$\rho : \mathrm{Gal}(\mathcal{K}_{\infty}/K) \overset{(\ref{fixK})} \Gamma_K \times \Delta_K \rightarrow \mathcal{O}_{K_p}^{\times}, \hspace{1cm} \gamma \mapsto \frac{\lambda_E(\gamma)}{\lambda_E^c(\gamma)}.$$
%\end{definition}

Recall that we view $\lambda_E : \mathrm{Gal}(\mathcal{K}_{\infty}/K) \rightarrow \mathcal{O}_{K_p}^{\times}$ as a character $\lambda_E : \Gamma_K \rightarrow \mathcal{O}_{K_p}^{\times}$ by restricting to $\Gamma_K \overset{(\ref{fixK})}{\subset} \mathrm{Gal}(\mathcal{K}_{\infty}/K)$. In fact, this restriction is equal to the character $\lambda_E/\chi_E : \Gamma_K \rightarrow \mathcal{O}_{K_p}^{\times}$, where $\chi_E$ is from (\ref{chiE}). However, we will prefer the shorter notation $\lambda_E$ for the sake of brevity. %When $F(q_{\mathrm{dR}})$ is a power series in $q_{\mathrm{dR}} - 1$, recall ``$F(q_{\mathrm{dR}})|_{q_{\mathrm{dR}} = 1}$'' denotes the value obtained by plugging in $q_{\mathrm{dR}} - 1 = 0$. 

\begin{definition} Define, for every $j \in \mathbb{Z}_{\ge 0}$
\begin{equation}\label{delta}\begin{split}&D_1^j\delta : \mathbb{U}_{\chi_E}^1(\lambda_E^{-1}) \rightarrow \mathbb{C}_p\llbracket q_{\mathrm{dR}}-1\rrbracket,\\
&D_1^j\delta(\beta) := \sum_{0 \le m \le p^{k_1}-1, 0 \le n \le p^{k_1}-1, \sigma \in \Delta_{A_0}}\lambda_E^{-2j}(\gamma_1^m\gamma_2^n\sigma)\cdot \theta_t(d_1^j\delta_0(\beta^{\gamma_1^m\gamma_2^n\sigma}))(y)(q_{\mathrm{dR}}).
\end{split}
\end{equation}
When $j = 0$, we simply write $\delta = D_1^0\delta$. 
\end{definition}

%Given $\beta \in \mathcal{U}_{\chi_F}^1(\kappa^{-1})$, one thus gets 
%$$\delta(\beta) \in \mathbb{C}_p\llbracket q_{\mathrm{dR}}-1\rrbracket.$$
%One can then extend $\delta_0(\beta)$ to a distribution on $\Delta_0 \times \mathbb{Z}_p$ as follows. Given $U \subset \Delta_0 \times \mathbb{Z}_p$ and $\sigma \in \Delta_0$ such that $\sigma U \subset 1 \times \mathbb{Z}_p \subset \Delta_0 \times \mathbb{Z}_p$, we define 
%$$\delta_0(\beta)(U) = \delta_0(\sigma(\beta))(\sigma U).$$
%Then letting $r : \Delta_0 \times \mathbb{Z}_p \rightarrow \mathbb{Z}_p$ denote the projection, we get a pushforward distribution $r_*\delta_0 \in \mathrm{Dist}(\mathbb{Z}_p,\mathbb{C}_p)$; in fact, by construction $r_*\delta_0 \in \mathbb{C}_p\llbracket q_{\mathrm{dR}}-1\rrbracket$. 
%Let $\sigma^*\delta_0(\beta) = \delta_0(\beta^{\sigma})$. We get a map
%\begin{equation}\label{split'}\delta := \sum_{\sigma \in \Delta_0}\sigma^*\delta_0 : \mathcal{U}_{\chi_F}^1(\kappa^{-1}) \rightarrow \mathbb{C}_p\llbracket q_{\mathrm{dR}}-1\rrbracket.
%\end{equation}

For our purposes, we will soon define a ``Coleman primitive'' of the value $\delta|_{q_{\mathrm{dR}}=1}$ of the above map $\delta$ at $q_{\mathrm{dR}} = 1$ as the following limit:
$$\delta'(\cdot) = \lim_{m \rightarrow \infty}D_1^{-1 + p^m(p-1)}\delta(\cdot)|_{q_{\mathrm{dR}} = 1}.$$
However, it is not \emph{a priori} clear that $j \mapsto D_1^j\delta(\cdot)|_{q_{\mathrm{dR}} = 1}$ has good enough $p$-adic continuity properties so that this limit converges. It will turn out that the restriction of $D_1^j\delta(\cdot)$ to a rank 1 subspace $U_0 \subset \mathbb{U}_{\chi_E}^1(\lambda_E^{-1})$ will have such continuity properties, and thus $\delta'$ will be well-defined on $U_0$. (See (\ref{delta'}) below.) The subspace $U_0$ is defined as the $\Lambda_{\mathcal{O}_{K_p}}[1/p]$-saturation in $\mathbb{U}_{\chi_E}^1(\lambda_E^{-1})$ of the module of elliptic units $\overline{\mathcal{C}}_{\chi_E}^1(\lambda_E^{-1})$ (see (\ref{saturation})), and thus is intrinsic to $\mathbb{U}_{\chi_E}^1(\lambda_E^{-1})$. The continuity of $D_1^j\delta|_{U_0}(\cdot)|_{q_{\mathrm{dR}} = 1}$ in $j$ is ``inherited'' from the continuity of $D_1^j\delta|_{\overline{\mathcal{C}}_{\chi_E}^1(\lambda_E^{-1})}(\cdot)|_{q_{\mathrm{dR}} = 1}$ in $j$; this latter continuity is a consequence of the explicit reciprocity law (\ref{explicitreciprocity}) and the continuity in $j$ of the $p$-adic $L$-function $j \mapsto D_1^j\mathcal{L}_{\lambda_E}|_{q_{\mathrm{dR}} = 1}$ (see Lemma \ref{continuitylemma}).

\section{Rubin-type Main Conjectures and Explicit Reciprocity Laws}\label{RMCsection}
%For this remainder of the paper, assume $E/\mathbb{Q}$ is an elliptic curve with CM by $\mathcal{O}_K$ and with $p$ a prime ramified in $K$. In particular, $K$ has class number 1.

In this section, we formulate the Rubin-type main conjectures we need and relate them to the $p$-adic $L$-function of Section \ref{padicLfunctionsection} via an explicit reciprocity law. Recall that $E/\mathbb{Q}$ is our previously fixed elliptic curve with CM by $\mathcal{O}_K$ (Choice \ref{Echoice}). In particular, $K$ has class number 1 and $p$ continues to be the unique finite prime ramified in $K/\mathbb{Q}$.  

\subsection{Review of the relevant Iwasawa modules}%Given an ideal $\frak{m} \subset \mathcal{O}_K$, let $K(\frak{m})$ denote the ray class field of modulus $\frak{m}$. Thus $\mathrm{Gal}(K(\frak{m})/K) \cong \mathcal{C}\ell(\frak{m})$ via Artin reciprocity. Note that $K(\frak{p}^n)/K$ is totally ramified above $\frak{p}$, since $K$ has class number 1; let $\frak{p}_n$ denote the unique prime of $\mathcal{O}_{K(\frak{p}^n)}$ above $\frak{p}$. 
Recall (see Definition \ref{UEdefinition})
$$\mathbb{U} = \varprojlim_n (\mathcal{O}_{\mathcal{K}_n}\otimes_{\mathbb{Z}}\mathbb{Z}_p)^{\times} \;, \hspace{1cm} \mathbb{U}^1 = \varprojlim_n (\mathcal{O}_{\mathcal{K}_n}\otimes_{\mathbb{Z}}\mathbb{Z}_p)^{\times,1}$$
denote the module of norm-compatible systems of local units attached to $E$. Let
$$\mathcal{E} = \varprojlim_n \mathcal{O}_{\mathcal{K}_n}^{\times}\;, \hspace{1cm} \mathcal{E}^1 = \varprojlim_n \mathcal{O}_{\mathcal{K}_n}^{\times,1}$$
denote the module of norm-compatible systems of global units, where 
$$\mathcal{O}_{\mathcal{K}_n}^{\times,1} = \mathcal{O}_{\mathcal{K}_n}^{\times} \cap (\mathcal{O}_{\mathcal{K}_n}\otimes_{\mathbb{Z}}\mathbb{Z}_p)^{\times,1}$$
denotes the principal part of $\mathcal{O}_{\mathcal{K}_n}^{\times}$. Let $\overline{\mathcal{E}}^1$ denote the $p$-adic closure of $\mathcal{E}^1$ inside $\mathbb{U}^1$. Let $M_{\infty}/\mathcal{K}_{\infty}$ be the maximal pro-$p$ abelian extension of $\mathcal{K}_{\infty}$ unramified at all places of $\mathcal{K}_{\infty}$ not above $\frak{p}$, let $N_{\infty}/\mathcal{K}_{\infty}$ denote the maximal pro-$p$ abelian extension of $\mathcal{K}_{\infty}$ unramified everywhere, and let $\mathcal{X} = \mathrm{Gal}(M_{\infty}/\mathcal{K}_{\infty})$ and $\mathcal{Y} = \mathrm{Gal}(N_{\infty}/\mathcal{K}_{\infty})$. 
Recall our fundamental exact sequence
$$0 \rightarrow \overline{\mathcal{E}}^1\rightarrow \mathbb{U}^1 \xrightarrow{\mathrm{rec}} \mathcal{X} \rightarrow \mathcal{Y} \rightarrow 0.$$

\begin{convention}\label{UEconvention}Note that $\mathbb{U}$, $\mathcal{E}$, etc. all depend on the CM elliptic curve $E$. In the rare occasion when we wish to emphasize the dependence on $E$, we will add the subscript ``$E$'', e.g. $\mathbb{U}_E$, $\mathcal{E}_E$, etc.  
\end{convention}

%Now let $\frak{f}_0 \subset \mathcal{O}_K$ be the integral ideal from Choice \ref{f0choice}.
\begin{definition}[Module of elliptic units]\label{ellipticunitsdefinition}Recall $a$ from (\ref{adefinition}). Let 
$\mathcal{C}^1(\frak{f}_0)$ be the module $\mathcal{C}_{\frak{f}_0}$ defined as in \cite[Chapter III.1.3 (4)]{deShalit}. Then 
$$\mathcal{C}^1(\frak{f}_0) \subset \mathbb{U}^1(\frak{f}_0).$$
Let 
$$\mathcal{C}^1 \subset \mathcal{E}^1$$
be the image of the map
$\mathrm{Nm}_{\frak{f}_0} : \mathbb{U}^1(\frak{f}_0) \rightarrow \mathbb{U}^1$ from (\ref{UEnorm}); one easily sees that indeed the image of $\mathcal{C}^1$ is contained in $\mathcal{E}^1$. Note that this definition of $\mathcal{C}^1$ conforms with the module of elliptic units associated with the tower $\mathcal{K}_{\infty}/K$ as defined in \cite[Definition 3.2, Section 5.2]{JohnsonLeungKings}.
%\cite[Section 1]{RubinMC} (see also; note that Rubin's definition is obtained by taking norms of the elliptic units in ray class towers constructed in loc. cit., see \cite[p. 29]{RubinMC}). 
 Let 
 $$\overline{\mathcal{C}}^1 \subset \mathbb{U}^1$$
 denote the $p$-adic closure of $\mathcal{C}^1$ in $\mathbb{U}^1$. 
\end{definition} %We have a norm map $\mathbb{U}^1 \rightarrow \mathbb{U}_E^1$ induced by the norm of the extension $L/K$ (unramified at $\frak{p}$). Define $\overline{\mathcal{E}}_E^1$ and $\overline{\mathcal{C}}_E^1$ as the images of $\overline{\mathcal{C}}^1$ and $\overline{\mathcal{E}}^1$, respectively. 
Recall the notation $\widetilde{\Lambda}$ from Definition \ref{Lambdadefinition}. In our setting where $p$ is ramified in $K$, $\overline{\mathcal{E}}^1$ has $\widetilde{\Lambda}$-rank 1, $\overline{\mathcal{C}}^1$ has $\widetilde{\Lambda}$-rank 1, $\mathbb{U}^1$ has $\widetilde{\Lambda}$-rank 2, $\mathcal{X}$ has $\widetilde{\Lambda}$-rank 1, and $\mathcal{Y}$ has $\widetilde{\Lambda}$-rank 0. Dividing by $\overline{\mathcal{C}}^1$, we get an exact sequence of $\widetilde{\Lambda}$-modules
\begin{equation}\label{fundamentalES}0 \rightarrow \overline{\mathcal{E}}^1/\overline{\mathcal{C}}^1 \rightarrow \mathbb{U}^1/\overline{\mathcal{C}}^1 \rightarrow \mathcal{X} \rightarrow \mathcal{Y} \rightarrow 0,
\end{equation}
where $\overline{\mathcal{E}}^1/\overline{\mathcal{C}}^1$ has $\widetilde{\Lambda}$-rank 0, $\mathbb{U}^1/\overline{\mathcal{C}}^1$ has $\widetilde{\Lambda}$-rank 1, $\mathcal{X}$ has $\widetilde{\Lambda}$-rank 1, and $\mathcal{Y}$ has $\widetilde{\Lambda}$-rank 0. 
%The nonzero terms of this exact sequence have ranks 0, 1, 1 and 0, respectively. 

\subsection{Rubin-type main conjectures}Recall that $\chi_E = \lambda_E|_{\Delta_K}$ (see Definition \ref{decompositionchoices}). In the following discussion, we will extensively use the notation $M_{\chi_E}(\lambda_E^{-1})$ of Definition \ref{Mtwistdefinition}.  We will often consider the cases $M = \overline{\mathcal{E}}^1, \overline{\mathcal{C}}^1, \mathbb{U}^1, \mathcal{X}, \mathcal{Y}$. 

\begin{proposition}\label{ranksproposition}$\overline{\mathcal{E}}_{\chi_E}^1(\lambda_E^{-1})$ is a torsion-free $\Lambda_{\mathcal{O}_{K_p}}[1/p]$-module of rank 1 and $\overline{\mathcal{C}}_{\chi_E}^1(\lambda_E^{-1})$ is a free $\Lambda_{\mathcal{O}_{K_p}}[1/p]$-modules of rank 1, and $\mathbb{U}_{\chi_E}^1(\lambda_E^{-1})$ is a free $\Lambda_{\mathcal{O}_{K_p}}[1/p]$-module of rank 2. 
\end{proposition}

\begin{proof}See \cite[Corollary 7.8, Lemma 11.8]{RubinMC}. Note that in our setting, we have inverted $p$, and so there is no issue with taking $\chi_E$-isotypic components using (\ref{echiK}). 
\end{proof}

Recall the notation $\Lambda_{\mathcal{O}_{K_p}}[1/p]$ from Definition \ref{Lambdadefinition}. Taking the $\chi_E$-isotypic component of (\ref{fundamentalES}) and twisting by $\lambda_E^{-1}$ (see Definition \ref{Mtwistdefinition}), we get an exact sequence of $\Lambda_{\mathcal{O}_{K_p}}[1/p]$-modules
\begin{equation}\label{fundamentalEStwist}0 \rightarrow \overline{\mathcal{E}}_{\chi_E}^1(\lambda_E^{-1})/\overline{\mathcal{C}}_{\chi_E}^1(\lambda_E^{-1}) \rightarrow \mathbb{U}_{\chi_E}^1(\lambda_E^{-1})/\overline{\mathcal{C}}_{\chi_E}^1(\lambda_E^{-1}) \rightarrow \mathcal{X}_{\chi_E}(\lambda_E^{-1}) \rightarrow \mathcal{Y}_{\chi_E}(\lambda_E^{-1}) \rightarrow 0.
\end{equation}

In Section \ref{MCproofsection} of the Appendix, we use work of Johnson-Leung-Kings (\cite{JohnsonLeungKings}) on equivariant main conjectures to prove:

\begin{theorem}[Rational Elliptic Units Main Conjecture]\label{EMC}$$\mathrm{char}_{\Lambda_{\mathcal{O}_{K_p}}[1/p]}\left((\overline{\mathcal{E}}^1/\overline{\mathcal{C}}^1)_{\chi_E}(\lambda_E^{-1})\right) = \mathrm{char}_{\Lambda_{\mathcal{O}_{K_p}}[1/p]}(\mathcal{Y}_{\chi_E}(\lambda_E^{-1})).$$
\end{theorem}

Admitting Theorem \ref{EMC}, we would then expect to get a main conjecture relating $\mathbb{U}_{\chi_E}^1(\lambda_E^{-1})/\overline{\mathcal{C}}_{\chi_E}^1(\lambda_E^{-1})$ and $\mathcal{X}_{\chi_E}(\lambda_E^{-1})$. However, these latter $\Lambda_{\mathcal{O}_{K_p}}[1/p]$-modules have rank 1, and so have vanishing characteristic ideals and thus do not admit an obvious nontrivial main conjecture relating them. To remedy this, one can formulate a main conjecture in the following way, suggested by Rubin's strategy in the latter half of \cite[Section 11]{RubinMC}. \emph{Choose} a decomposition (cf. Lemma 11.9 of op. cit.)
\begin{equation}\label{Udecomposition}\mathbb{U}_{\chi_E}^1(\lambda_E^{-1}) \cong U_1 \oplus U_2
\end{equation}
where each $U_i \cong \Lambda_{\mathcal{O}_{K_p}}[1/p]$ and such that $U_2 \cap \overline{\mathcal{E}}_{\chi_E}^1(\lambda_E^{-1}) = 0$. Then $\mathbb{U}_{\chi_E}^1(\lambda_E^{-1})/(\overline{\mathcal{C}}_{\chi_E}^1(\lambda_E^{-1}),U_2)$ and $\mathcal{X}_{\chi_E}(\lambda_E^{-1})/\mathrm{rec}(U_2)$ are torsion $\Lambda_{\mathcal{O}_{K_p}}[1/p]$-modules, and (\ref{fundamentalEStwist}) gives an exact sequence of torsion $\Lambda_{\mathcal{O}_{K_p}}[1/p]$-modules
\begin{equation}\label{fundamentalEStwist2}0 \rightarrow \overline{\mathcal{E}}_{\chi_E}^1(\lambda_E^{-1})/\overline{\mathcal{C}}_{\chi_E}^1(\lambda_E^{-1}) \rightarrow \mathbb{U}_{\chi_E}^1(\lambda_E^{-1})/(\overline{\mathcal{C}}_{\chi_E}^1(\lambda_E^{-1}),U_2) \rightarrow \mathcal{X}_{\chi_E}(\lambda_E^{-1})/\mathrm{rec}(U_2) \rightarrow \mathcal{Y}_{\chi_E}(\lambda_E^{-1}) \rightarrow 0.
\end{equation}
Theorem \ref{EMC} now implies the following, which we call a ``Rubin-type main conjecture''.

\begin{corollary}[Rubin-Type Main Conjecture]\label{RMC}Fixing a decomposition as in (\ref{Udecomposition}) satisfying the above conditions, we have
$$\mathrm{char}_{\Lambda_{\mathcal{O}_{K_p}}[1/p]}\left(\mathbb{U}_{\chi_E}^1(\lambda_E^{-1})/(\overline{\mathcal{C}}_{\chi_E}^1(\lambda_E^{-1}),U_2)\right) = \mathrm{char}_{\Lambda_{\mathcal{O}_{K_p}}[1/p]}\left(\mathcal{X}_{\chi_E}(\lambda_E^{-1})/\mathrm{rec}(U_2)\right).$$
\end{corollary}

\begin{proof}This follows from (\ref{fundamentalEStwist2}), Theorem \ref{EMC} and additivity of $\Lambda_{\mathcal{O}_{K_p}}[1/p]$-characteristic ideals in exact sequences. 
\end{proof}

\subsection{Relating Rubin-type main conjectures to $L$-values and $p$-adic $L$-functions: outline}

The flexibility to vary the choice of $U_2$ in (\ref{Udecomposition}) is crucial in our proofs of rank 0 and rank 1 $p$-converse theorems. For the rank 0 case (see Section \ref{rank0section}), we will take $U_2$ to be the kernel of the map $\delta|_{q_{\mathrm{dR}} = 1}$ for $\delta$ in (\ref{delta}). This is essentially the same choice of $U_2$ taken in \cite[Section 11]{RubinMC}, where $U_2$ is the kernel of a certain logarithmic derivative of Coleman power series map (denoted by $\delta$ in Proposition 11.7 of op. cit.). In fact, our $\delta|_{q_{\mathrm{dR}} = 1}$ is essentially equal to $\delta$ of loc. cit. up to a nonzero multiple. One can show, using Wiles's explicit reciprocity law (\cite{Wiles}, \cite[Chapter I.4]{deShalit}), that 
$$\delta|_{q_{\mathrm{dR}} = 1}(\overline{\mathcal{C}}_{\chi_E}^1(\lambda_E^{-1})) = C \cdot L(\lambda_E^{-1},0) = C \cdot L(E/\mathbb{Q},1).$$
The assumption $\mathrm{corank}_{\mathbb{Z}_p}\mathrm{Sel}_{p^{\infty}}(E/\mathbb{Q}) = 0$ implies, via Wiles's explicit reciprocity law, that $U_2 \cap \overline{\mathcal{C}}_{\chi_E}^1(\lambda_E^{-1}) = 0$ and hence $U_2 \cap \mathcal{E}_{\chi_E}^1(\lambda_E^{-1}) = 0$, so that $U_2$ satisfies the assumptions of (\ref{Udecomposition}). Moreover, Theorem \ref{RMC} implies 
$$\delta|_{q_{\mathrm{dR}} = 1}(\overline{\mathcal{C}}_{\chi_E}^1(\lambda_E^{-1})) \neq 0,$$
and as a result we obtain the rank 0 $p$-converse theorem (cf. \cite[Theorem 11.18]{RubinMC}). 

For the rank 1 case, we choose $U_2$ with $U_2 \cap \overline{\mathcal{E}}_{\chi_E}^1(\lambda_E^{-1}) = 0$ and $U_1$ to be the kernel of $\delta|_{q_{\mathrm{dR}} = 1}$; note that this was the choice of $U_2$ in the rank 0 situation described in the previous paragraph. Then from (\ref{delta'}) below, we have a map
$$\delta' : U_1 \rightarrow \mathbb{C}_p$$
satisfying the ``explicit reciprocity law''
$$\delta'(\overline{\mathcal{C}}_{\chi_E}^1(\lambda_E^{-1})) = C_E'\cdot D_1^{-1}\mathcal{L}_{\lambda_E}|_{q_{\mathrm{dR}} = 1}$$
for some $C_E' \in \mathbb{C}_p^{\times}$ (see Theorem \ref{Colemanmapthm} below). The assumption $\mathrm{corank}_{\mathbb{Z}_p}\mathrm{Sel}_{p^{\infty}}(E/\mathbb{Q}) = 1$ implies (see proof of Theorem \ref{BSDrank1theorem})
\begin{align*}D_1^{-1}\mathcal{L}_{\lambda_E}|_{q_{\mathrm{dR}} = 1} &= \delta'(\overline{\mathcal{C}}_{\chi_E}^1(\lambda_E^{-1}))/(\Gamma_K-\lambda_E^{-2}(\Gamma_K)) \\
&\overset{\text{Theorem \ref{RMC}}}{=} \mathrm{char}_{\Lambda_{\mathcal{O}_{K_p}}[1/p]}\left(\mathcal{X}_{\chi_E}(\lambda_E^{-1})/\mathrm{rec}(U_2)\right)/(\Gamma_K-\lambda_E^{-2}(\Gamma_K)) \neq 0
\end{align*}
where $M/(\Gamma_K-\lambda_E^{-2}(\Gamma_K))$ is written in the notation of Definition \ref{Lambdadefinition} for $\rho = \lambda_E^{-2}$. Now (\ref{Heegnerpointidentity2}) gives the $p$-converse theorem.

\subsection{Choice of decomposition of module of norm-compatible systems of semi-local units: rank 1 case}\label{rank1U2choicesection}We seek to find an appropriate submodule $U' \subset \mathbb{U}_{\chi_E}^1(\lambda_E^{-1})$ to quotient out by to formulate our main conjecture. As mentioned above, the choice of $U'$ depends on whether $\mathrm{corank}_{\mathbb{Z}_p}\mathrm{Sel}_{p^{\infty}}(E/\mathbb{Q})$ is 0 or 1. In the rank 0 case, $U'$ will be a direct summand such that $\mathbb{U}_{\chi_E}^1(\lambda_E^{-1}) = U' \oplus U_2$ where $U_2$ is the kernel of the Kummer map $\delta$ of \cite[Proposition 11.4]{RubinMC} (see also Definition \ref{deltaEdefinition}); thus $\mathbb{U}_{\chi_E}^1(\lambda_E^{-1})/U'$ will reduce modulo the augmentation ideal to the dual of the usual Selmer group (which is torsion under the rank 0 assumption). In the rank 1 case, $U'$ will itself be the kernel of the above map $\delta$ of loc. cit., and $\mathbb{U}_{\chi_E}^1(\lambda_E^{-1})/U'$ will specialize to the dual of a smaller Selmer group; this $U'$ will also have an intrinsic description as the saturation in $\mathbb{U}_{\chi_E}^1(\lambda_E^{-1})$ of the submodule of elliptic units. We will first discuss the choice of $U'$ in the rank 1 case, and discuss the rank 0 case in Section \ref{rank0preview}. 

\begin{definition}Let 
\begin{equation}\label{saturation}U_0 = \{x \in \mathbb{U}_{\chi_E}^1(\lambda_E^{-1}): \exists \alpha \in \Lambda_{\mathcal{O}_{K_p}}[1/p], \alpha x \in \xi_E\cdot \Lambda_{\mathcal{O}_{K_p}}[1/p] = \overline{\mathcal{C}}_{\chi_E}^1(\lambda_E^{-1})\}.
\end{equation}
That is, $U_0$ is the $\Lambda_{\mathcal{O}_{K_p}}[1/p]$-saturation of the $\Lambda_{\mathcal{O}_{K_p}}[1/p]$-submodule
$$\Lambda_{\mathcal{O}_{K_p}}[1/p]\cdot \xi_E \subset \mathbb{U}_{\chi_E}^1(\lambda_E^{-1}).$$
\end{definition}

\begin{proposition}\label{freeproposition}$U_0$ is a free rank 1 $\Lambda_{\mathcal{O}_{K_p}}[1/p]$-module.
\end{proposition}

\begin{proof}By definition, $\mathbb{U}_{\chi_E}^1(\lambda_E^{-1})/U_0$ is $\Lambda_{\mathcal{O}_{K_p}}[1/p]$-torsion-free, and by freeness of $\mathbb{U}_{\chi_E}^1(\lambda_E^{-1})$ we see that $U_0$ has rank 1. 

We now proceed along the same lines as in the proof of \cite[Lemma 4.1]{Rubin3}. Let $\beta_1,\beta_2$ be a $\Lambda_{\mathcal{O}_{K_p}}[1/p]$-basis of $\mathbb{U}_{\chi_E}^1(\lambda_E^{-1})$. Since $\Gamma_K \cong \mathbb{Z}_p^{\oplus 2}$, we have by the Amice transform (\cite[I.3.1 (1)]{deShalit})
$$\Lambda_{\mathcal{O}_{K_p}}[1/p] \cong \mathcal{O}_{K_p}\llbracket T_1, T_2\rrbracket [1/p],$$
which is a UFD. Let $Z \subset \mathbb{U}_{\chi_E}^1(\lambda_E^{-1})$ be a maximal free submodule of $U_0$. Then since $\Lambda_{\mathcal{O}_{K_p}}[1/p]$ is a UFD, $\mathbb{U}_{\chi_E}^1(\lambda_E^{-1})/Z$ is torsion-free by maximality of $Z$, and hence so is the submodule $U_0/Z$. However, $U_0/Z$ has rank 0 and so is torsion and torsion-free, meaning $U_0 = Z$. 
\end{proof}

%Recall that $K_{\infty}$ is the $\mathbb{Z}_p^{\oplus 2}$-extension of $K$. Note that since $K$ has class number 1, the restriction $\mathrm{Gal}(\mathcal{K}_{\infty}/K) \rightarrow \mathrm{Gal}(K_{\infty}/K)$ induces an identification $\Gamma_K \xrightarrow{\sim} \mathrm{Gal}(K_{\infty}/K)$. Hence we have a natural anticyclotomic projection $\Gamma_K \rightarrow \Gamma_-$.  Given a $\Lambda_{\mathcal{O}_{K_p}}[1/p]$-module $M$, let 
%$$M^- := M\otimes_{\Lambda_{\mathcal{O}_{K_p}}[1/p]}\mathcal{O}_{K_p}\llbracket \Gamma_-\rrbracket [1/p].$$

\begin{choice}\label{beta0choice}Using Proposition \ref{freeproposition}, choose and fix a $\Lambda_{\mathcal{O}_{K_p}}[1/p]$-basis $\beta_0$ of $U_0$. 
\end{choice}

\begin{choice}\label{Udecompositionchoice}Choose any decomposition of $\Lambda_{\mathcal{O}_{K_p}}[1/p]$-modules $\mathbb{U}_{\chi_E}^1(\lambda_E^{-1}) = U_1 \oplus U_2$ such that each $U_i$ is $\Lambda_{\mathcal{O}_{K_p}}[1/p]$-free of rank 1 and
\begin{equation}\label{globalintersect0}U_2 \cap \overline{\mathcal{E}}_{\chi_E}^1(\lambda_E^{-1}) = 0.
\end{equation}
Such a decomposition with $U_2$ as above will be fixed later in (\ref{kernelequalities}) when $r_p = 0$ and in Choice \ref{U2finallyfixed} when $r_p = 1$.% Note that this exists since $\overline{\mathcal{E}}_{\chi_E}^1(\lambda_E^{-1}) \cong \Lambda_{\mathcal{O}_{K_p}}[1/p]$ and $\mathbb{U}_{\chi_E}^1(\lambda_E^{-1}) \cong \Lambda_{\mathcal{O}_{K_p}}[1/p]^{\oplus 2}$. 
\end{choice} 

\begin{proposition}\label{reciprocityU2}
$$\mathrm{rec}(U_2) \subset \mathcal{X}_{\chi_E}(\lambda_E^{-1})$$
is $\Lambda_{\mathcal{O}_{K_p}}[1/p]$-free of rank 1. 
\end{proposition}

\begin{proof}This follows from the identity $\ker(\mathrm{rec}) =  \overline{\mathcal{E}}_{\chi_E}^1(\lambda_E^{-1})$ and (\ref{globalintersect0}). 
\end{proof}

Let $\beta_i$ be a $\Lambda_{\mathcal{O}_{K_p}}[1/p]$-basis of $U_i$ for $i = 1,2$. Then we can write
$$\beta_0 = \lambda_1\beta_1 + \lambda_2\beta_2.$$
Let 
$$\mathrm{pr}_1 : \mathbb{U}_{\chi_E}^1(\lambda_E^{-1}) = U_1 \oplus U_2 \rightarrow U_1$$
denote projection onto the first factor. 
%By the definition (\ref{saturation}) we have that $\lambda_1$ and $\lambda_2$ have no common factor. Hence we may assume that 
%\begin{equation}\label{lambda1nonvanishing}\lambda_1(1) := \lambda \pmod{(\Gamma_K-1)} \neq 0.
%\end{equation}
Then 
\begin{equation}\label{multiple1}\mathrm{pr}_1(\beta_0) = \lambda_1\beta_1.
\end{equation}
In the rank 1 situation, we may take $U_1 = U_0$, and hence $\lambda_1 = 1$, see Choice \ref{U2finallyfixed} below.%\footnote{This is one of the places where the $r_p = 1$ condition is essentially used in the proof of the rank 1 $p$-converse theorem.} 

Recall the element $\xi_E \in \mathbb{U}_{\chi_E}^1(\lambda_E^{-1})$ from (\ref{xiE}). Clearly, by construction, $\xi_E \in \overline{\mathcal{C}}_{\chi_E}^1(\lambda_E^{-1})$. 

\begin{proposition}\label{xiEgeneratorproposition}$\xi_E \in \overline{\mathcal{C}}_{\chi_E}^1(\lambda_E^{-1})$ is a $\Lambda_{\mathcal{O}_{K_p}}[1/p]$-generator.
\end{proposition}

\begin{proof}This will follow from the construction of $\xi_E$. First, note that $e_{\chi_E}(e^1(\frak{a};\frak{f}_0))$ is a $\Lambda_{\mathcal{O}_{K_p}}[1/p]$-generator of $\overline{\mathcal{C}}_{\chi_E}^1(\frak{f}_0)$, where $\overline{\mathcal{C}}^1(\frak{f}_0)$ is the $p$-adic closure in $\mathbb{U}^1(\frak{f}_0)$ of $\mathcal{C}^1(\frak{f}_0)$ (see Definition \ref{ellipticunitsdefinition}). This follows from the definition of $\mathcal{C}_{\frak{f}_0}$ from \cite[Chapter III.1.4]{deShalit} and the fact that $\chi_E$ is nontrivial on the decomposition group of all primes above $\frak{p}$ of $\Delta_K$ (cf. \cite[Lemma 11.5 (ii), Lemma 11.8]{RubinMC}). Now by Definition \ref{ellipticunitsdefinition},
$$\overline{\mathcal{C}}_{\chi_E}^1 = \mathrm{Nm}_{\frak{f}_0,\chi_E}(\overline{\mathcal{C}}_{\chi_E}^1(\frak{f}_0)).$$
Thus, since $\xi_E = \mathrm{Nm}_{\frak{f}_0,\chi_E}(e_{\chi_E}(e^1(\frak{a};\frak{f}_0)))$ by definition (see (\ref{xiE})), we are done.
\end{proof}

Using Proposition \ref{xiEgeneratorproposition}, write
\begin{equation}\label{multiple2}\xi_E = \lambda_0\beta_0
\end{equation}
for $\lambda_0 \in \Lambda_{\mathcal{O}_{K_p}}[1/p]$. Then 
$$\mathrm{pr}_1(\xi_E) = \lambda_0\mathrm{pr}_1(\beta_0).$$

\begin{corollary}\label{explicitcharacteristicidealcorollary}$$\lambda_0\lambda_1 = \mathrm{char}_{\Lambda_{\mathcal{O}_{K_p}}[1/p]}\left(\mathcal{X}_{\chi_E}(\lambda_E^{-1})/\mathrm{rec}(U_2)\right).$$
\end{corollary}

\begin{proof}We have 
\begin{align*}\lambda_0\lambda_1 \overset{(\ref{multiple1}), (\ref{multiple2})}{=} \mathrm{char}_{\Lambda_{\mathcal{O}_{K_p}}[1/p]}\left(U_1/(\mathrm{pr}_1(\xi_E))\right) &= \mathrm{char}_{\Lambda_{\mathcal{O}_{K_p}}[1/p]}\left(\mathbb{U}_{\chi_E}^1(\lambda_E^{-1})/(\overline{\mathcal{C}}_{\chi_E}^1(\lambda_E^{-1}),U_2)\right) \\
&\overset{\text{Theorem \ref{RMC}}}{=} \mathrm{char}_{\Lambda_{\mathcal{O}_{K_p}}[1/p]}\left(\mathcal{X}_{\chi_E}(\lambda_E^{-1})/\mathrm{rec}(U_2)\right).
\end{align*}
\end{proof}

\subsection{Explicit reciprocity laws}
\begin{definition}\label{characterevaluation}For any continuous character $\varrho : \Gamma_K \rightarrow \mathcal{O}_{\mathbb{C}_p}^{\times}$, we have a natural evaluation map
$$\varrho : \Lambda_{\mathcal{O}_{\mathbb{C}_p}}[1/p] \rightarrow \mathbb{C}_p, \hspace{1cm} \gamma \mapsto \gamma(\varrho) := \varrho(\gamma),\hspace{.25cm} \forall \gamma \in \Gamma_K.$$
Given $\alpha \in \Lambda_{\mathcal{O}_{\mathbb{C}_p}}[1/p]$ we denote the image of $\alpha$ under this map by $\alpha(\varrho)$. 
\end{definition}

%Recall that we view $\lambda_E$ as a character $\lambda_E : \Gamma_K \rightarrow \mathcal{O}_{K_p}^{\times}$ by restricting to $\Gamma_K$ using (\ref{fixK}), and thus apply Definition \ref{characterevaluation} to $\rho = \lambda_E$. In fact, this restriction is equal to the character $\lambda_E/\chi_E : \Gamma_K \rightarrow \mathcal{O}_{K_p}^{\times}$, where $\chi_E$ is from (\ref{chiE}). However, we will prefer the shorter notation as it is simpler. When $F(q_{\mathrm{dR}})$ is a power series in $q_{\mathrm{dR}} - 1$, recall ``$F(q_{\mathrm{dR}})|_{q_{\mathrm{dR}} = 1}$'' denotes the value obtained by plugging in $q_{\mathrm{dR}} - 1 = 0$. 

\begin{theorem}[Explicit Reciprocity Law]\label{Colemanmapthm}The maps (\ref{delta}) induce maps (which we will also denote by $D_1^j\delta$)
$$D_1^j\delta : \mathbb{U}_{\chi_E}^1(\lambda_E^{-1}) \hat{\otimes}_{\mathcal{O}_{K_p}}\mathcal{O}_{\mathbb{C}_p} \rightarrow \mathbb{C}_p\llbracket q_{\mathrm{dR}}-1\rrbracket$$
satisfying the following. 
\begin{enumerate}
\item For all $\alpha \in \Lambda_{\mathcal{O}_{\mathbb{C}_p}} [1/p]$, all $\beta \in \mathbb{U}_{\chi_E}^1(\lambda_E^{-1}) \hat{\otimes}_{\mathcal{O}_{K_p}}\mathcal{O}_{\mathbb{C}_p}$ and all $j \in \mathbb{Z}_{\ge 0}$,
\begin{equation}\label{equivariance}D_1^j\delta(\alpha\beta) = \alpha(\lambda_E^{2j})\cdot D_1^j\delta(\beta)
\end{equation}
where $\alpha(\lambda_E^{2j})$ is written in the notation of Definition \ref{characterevaluation}.
\item Recall that the free $\Lambda_{\mathcal{O}_{K_p}}[1/p]$-module $\overline{\mathcal{C}}_{\chi_E}^1(\lambda_E^{-1})$ has a generator $\xi_E$ (Proposition \ref{xiEgeneratorproposition}). Recall $D_1^j\mathcal{L}_{\lambda_E}$ defined in (\ref{globalmeasure2}). There exists $C_E \in \mathbb{C}_p^{\times}$ such that all $j \in \mathbb{Z}_{\ge 0}$, we have  
\begin{equation}\label{explicitreciprocity}D_1^j\delta(\xi_E)|_{q_{\mathrm{dR}} = 1} = C_E\cdot (\lambda_E^{1+2j}(\frak{a}) - \mathbb{N}\frak{a})\cdot D_1^j\mathcal{L}_{\lambda_E}|_{q_{\mathrm{dR}} = 1}.
\end{equation} 
\end{enumerate}
\end{theorem}

%\begin{remark}(1) in the statement of Theorem \ref{Colemanmapthm} can be compared with the $\Lambda_{\mathcal{O}_{K_p}}[1/p]$-equivariance of the Coates-Wiles homomorphisms (\cite[Chapter I.3.5 (ii)]{deShalit}). %In a standard $\Lambda_{\mathcal{O}_{K_p}}[1/p]$-equivariance (as in the height 1 case of \cite{deShalit}), one would instead replace the ``$r_{\chi}$'' on both sides of the identity with ``$\chi$''. %However, in our case $r_{\chi}$ is necessary, as the coordinate $q_{\mathrm{dR}}-1$ does not have a natural $\Gamma_K$-action as in the height 1 case, but rather a natural $\mathbb{Z}_p^{\times}$-action; it is a coincidence in the height 1 case that the $\mathbb{Z}_p^{\times}$-action coincides with the anticyclotomic Galois action induced by the $\Gamma_K$-action. 
%(2) in the statement is the ``explicit reciprocity law'' alluded to previously.
%\end{remark}

\begin{proof}

%Recall $\mathbf{e}$ from (\ref{esection2}), recall the point $y = (E,e_1,e_2) \in \mathrm{LT}_{\infty}(1/2)$ from Choice \ref{AEchoice}, and recall $\mathcal{K}_n = K(E[\frak{p}^n])$. %Let 
%$$\mathbf{e}(y) \in E[p^a] \otimes_{\mathbb{Z}_p}\mathcal{O}_{K_p}, \hspace{1cm} e_{1,a}(y) \in E[p^a],$$
%denote the specializations of $\mathbf{e}$ and $e_{1,a}$ to $y$, i.e. $\mathbf{e}(y)$ is the composition 
%$$\mathrm{Spa}(\mathbb{C}_p,\mathcal{O}_{\mathbb{C}_p}) \xrightarrow{y} \mathrm{LT}_{\infty}(1/2) \xrightarrow{\mathbf{e}} F_{\infty}^{\mathrm{univ}}(1/2)[p^a],$$
%and similarly for $e_{1,a}(y)$. 

\textbf{(1)}: Recall the $p$-adic Maass-Shimura operator $d_k^j$ from (\ref{dkjformula2}). Recall $U = \mathcal{Y}^{\mathrm{Ig}}(\epsilon_0)$ from Definition \ref{Udefinition} (as well as our choice of $\epsilon_0 = 1/2$ from Choice \ref{rchoice2}). For any $F \in \mathcal{O}_{U}(W)$ for a pro\'{e}tale open $W \rightarrow U$, recall 
$$\theta_t\left(d_k^jF(q_{\mathrm{dR}})\right) \in \hat{\mathcal{O}}_{U}(W)\llbracket q_{\mathrm{dR}}-1\rrbracket$$
from Definition \ref{thetatpowerseriesdefinition}. We abbreviate this by $\theta_td_k^j(F)$ in what follows.   %= (\mathbb{C}/\Omega(\mathbb{Z}\tau_0 + \mathbb{Z}),\frac{\Omega}{N},\frac{\Omega\tau_0}{N},e_1,e_2)$$
%be the CM point as in Choice \ref{CMpointchoice} (see also Choice \ref{choice} for the second equality, recalling we are in the $D = M_2(\mathbb{Q})$ case), and recall that we made the specific choice $A = E$ in Choice \ref{AEchoice}. Recall also the notation $y_n = y \cdot g^{-n}$ from (\ref{YIgin2}). 
 For shorthand, when $W \subset \mathcal{Y}^{\mathrm{Ig}}(\epsilon_0)$ is an open and $y' \in W(\mathbb{C}_p,\mathcal{O}_{\mathbb{C}_p})$ (as well as our choice $\epsilon_0 = 1/2$ from Definition \ref{rchoice2}), and let 
$$\theta_{t,y'}d_k^j(F) \in \mathbb{C}_p\llbracket q_{\mathrm{dR}}-1\rrbracket$$
denote 
%$$\theta_t\left(d_k^jF(y')(q_{\mathrm{dR}})\right),$$
the power series obtained by specializing the coefficients of 
$$\theta_td_k^jF \in \hat{\mathcal{O}}_{U}(W)\llbracket q_{\mathrm{dR}}-1\rrbracket$$
along 
$$\hat{\mathcal{O}}_{U}(W) \rightarrow \hat{\mathcal{O}}_{U}(y') \subset \mathbb{C}_p.$$
Recall the point $y \in \mathrm{LT}_{\infty}(\overline{\mathbb{Q}}_p,\overline{\mathbb{Z}}_p)$ from (\ref{usey2}). We will consider 
$$W = \mathrm{LT}_{\infty} \cap \mathcal{Y}^{\mathrm{Ig}}(1/2)$$
and $\theta_{t,y'}$ for $y' = y$ in what follows.

Recall our basis $\gamma_1,\gamma_2$ of $\Gamma_K$. %Note that for all $\gamma \in \Gamma_K$, using (\ref{K2adecomposition}) and (\ref{K2adecompositionbasis}), for every $0 \le m \le p^{k_1}-1, 0 \le n \le p^{k_2}-1$ we can write
%\begin{equation}\label{rewritegamma}\gamma\gamma_1^m\gamma_2^n = \gamma_{m,n}'\gamma_1^{m'}\gamma_2^{n'}
%\end{equation}
%such that $0\le m'\le p^{k_1}-1, 0\le n' \le p^{k_2}-1$ for some unique $\gamma_{m,n}' \in \mathrm{Gal}(K_{\infty}/(K_{\infty} \cap \mathcal{K}_{2a})) \subset \Gamma_K$. Given $a,b \in \mathbb{Z}$, $a \le b$, let 
%$$[a,b]_{\mathbb{Z}} := \{x \in \mathbb{Z} : a \le x\le b\} \subset \mathbb{Z}.$$
%The assignment 
%\begin{equation}\label{integerbijection}[0,p^{k_1}-1]_{\mathbb{Z}} \times [0,p^{k_2}-1]_{\mathbb{Z}} \rightarrow [0,p^{k_1}-1]_{\mathbb{Z}} \times [0,p^{k_2}-1]_{\mathbb{Z}} , \hspace{1cm} (m,n) \mapsto (m',n')
%\end{equation}
%is easily seen to be a bijection of sets. 
Let $d\mathrm{Log}(\beta)^{\phi^{-2a}}$ denote the image of $d\mathrm{Log}(\beta) \in \mathbf{\Gamma}(F_{\infty}^{\mathrm{univ}})$ under the map 
$$\mathbf{\Gamma}(F_{\infty}^{\mathrm{univ}}) \xrightarrow{\phi^{-2a}} \mathbf{\Gamma}(F_{\infty}^{\mathrm{univ},\phi^{-2a}})$$
from (\ref{Frobeniusapplication}).

Recall our embedding 
$$\Gamma_K \overset{(\ref{fixA})}{\hookrightarrow} \mathrm{Gal}(K(\frak{f}_0p^{\infty})/K(\frak{f}_0))$$
as well as our embedding $\rho : K \hookrightarrow D$ from (\ref{Kembeddings}) determined by $\tau_0 = \varpi \overset{(\ref{pichoice})}{\in} \frak{p}$. From this latter inclusion and (\ref{Kembeddingscondition}) we have 
\begin{equation}\label{varpiGamma0}\rho(1+\varpi) \in \Gamma_{0,p}(p).
\end{equation}
Thus, given
$$\gamma \in \Gamma_K \overset{[\cdot]_{F^{\mathrm{univ}}}}{\hookrightarrow} 1+\varpi\mathcal{O}_{K_p} \overset{(\ref{varpiGamma0})}{\hookrightarrow}\Gamma_{0,p}(p),$$
we have 
\begin{equation}\label{gammalambdaFuniv}\begin{split}\theta_{t,y}([\gamma]_{F^{\mathrm{univ}}}^*w_{\mathrm{can}}|_{\mathrm{LT}_{\infty}}) = \theta_t(\theta_q([\gamma]_{F^{\mathrm{univ},}}^*w_{\mathrm{can}}))(y) \overset{(\ref{Hodgethetacommute})}{=} \theta_t(\theta_q(w_{\mathrm{can}}))(y \cdot \gamma) 
%&= [\lambda_{F^{\phi^{-2a+1}}}(\gamma)]_{F^{\phi^{-2a+1}}}^*\theta_t(\theta_q(w_{\mathrm{can}}))(y) \\&= \kappa_{F^{\phi^{-2a+1}}}(\gamma) \cdot  \theta_t(\theta_q(w_{\mathrm{can}}))(y(2a-1)) \\
%&\overset{(\ref{kappasequal})}{=} \kappa_F(\gamma) \cdot \theta_t(\theta_q(w_{\mathrm{can}}))(y(2a-1))\\
%&\overset{(\ref{lambdakappaequal})}{=} \lambda_{A_0}(\gamma) \cdot \theta_t(\theta_q(w_{\mathrm{can}}))(y(2a-1))\\
= \lambda_E(\gamma) \cdot \theta_t(\theta_q(w_{\mathrm{can}}))(y)
\end{split}
\end{equation}
where the last equality follows since $y = (E,e_1,e_2)$ (see (\ref{usey})) and $E$ has CM by $\mathcal{O}_K$. 

Let $\beta \mapsto \beta^{\gamma}$ denote the action of $\gamma \in \Gamma_K$ on $\beta \in \mathbb{U}_{\chi_E}^1(\lambda_E^{-1})$; recall this action is the twist by $\lambda_E^{-1}$ of the usual Galois action of $\Gamma_K$ on $\mathbb{U}_{\chi_E}^1$. Then we have  
\begin{equation}\label{esectionlambda}\begin{split}
%&[\gamma]_{F^{\mathrm{univ},\phi^{-2a+1}}}^*\delta(\beta)\\
D_1^j\delta(\beta^{\gamma})&\overset{(\ref{delta})}{=} \sum_{0\le m \le p^{k_1}-1}\sum_{0 \le n \le p^{k_2}-1}\sum_{\sigma \in \Delta_{A_0}}\theta_{t,y}\left(d_1^j\left(\frac{\mathbf{e}^*d\mathrm{Log}(\beta^{\gamma})^{\phi^{-2a}}}{w_{\mathrm{can}}|_{\mathrm{LT}_{\infty}}}\right)\right)\\
&\overset{(\ref{partialdkjrelate}), (\ref{esection2})}{=} \sum_{m = 0}^{p^{k_1}-1}\sum_{n = 0}^{p^{k_2}-1}\sum_{\sigma \in \Delta_{A_0}}\lambda_E^{-2j}(\gamma_1^m\gamma_2^n\sigma)\lambda_E^{-1}(\gamma_1^m\gamma_2^n\gamma)\\
&\hspace{4cm}\cdot \theta_{t,y}\left(\frac{\partial_1^j\left([\gamma\gamma_1^m\gamma_2^n\sigma]_{F^{\mathrm{univ},\phi^{-2a}}}(\alpha_{2,a})^*\left(d\mathrm{Log}(\beta)^{\phi^{-2a}}\right)\right)}{w_{\mathrm{can}}^{\otimes 1 + 2j}|_{\mathrm{LT}_{\infty}}}\right)\\
&= \lambda_{E}^{2j}(\gamma)\cdot \sum_{m = 0}^{p^{k_1}-1}\sum_{n = 0}^{p^{k_2}-1}\sum_{\sigma \in \Delta_{A_0}}\lambda_{E}^{-2j}(\gamma\gamma_1^m\gamma_2^n\sigma)\lambda_E^{-1}(\gamma_1^m\gamma_2^n\gamma)\\
&\hspace{4cm}\cdot \theta_{t,y}\left(\frac{\partial_1^j\left([\gamma\gamma_1^m\gamma_2^n\sigma]_{F^{\mathrm{univ},\phi^{-2a}}}(\alpha_{2,a})^*\left(d\mathrm{Log}(\beta)^{\phi^{-2a}}\right)\right)}{w_{\mathrm{can}}^{\otimes 1 + 2j}|_{\mathrm{LT}_{\infty}}}\right)\\
&\overset{(\ref{gammalambdaFuniv})}{=} \lambda_{E}^{2j}(\gamma)\cdot \sum_{m = 0}^{p^{k_1}-1}\sum_{n = 0}^{p^{k_2}-1}\sum_{\sigma \in \Delta_{A_0}}\lambda_{E}^{-2j}(\gamma_1^m\gamma_2^n\sigma)\lambda_E^{-1}(\gamma_1^m\gamma_2^n)\\
&\hspace{4cm}\cdot \theta_{t,y}\left(\frac{\partial_1^j\left([\gamma\gamma_1^m\gamma_2^n\sigma]_{F^{\mathrm{univ},\phi^{-2a}}}(\alpha_{2,a})^*\left(d\mathrm{Log}(\beta)^{\phi^{-2a}}\right)\right)}{[\gamma]_{F^{\mathrm{univ}}}^*w_{\mathrm{can}}^{\otimes 1 + 2j}|_{\mathrm{LT}_{\infty}}}\right)\\
&\overset{(\ref{bracketFunivsequal})}{=} \lambda_{E}^{2j}(\gamma)\cdot \sum_{m = 0}^{p^{k_1}-1}\sum_{n = 0}^{p^{k_2}-1}\sum_{\sigma \in \Delta_{A_0}}\lambda_{E}^{-2j}(\gamma_1^m\gamma_2^n\sigma)\lambda_E^{-1}(\gamma_1^m\gamma_2^n)\\
&\hspace{4cm} \cdot \theta_{t,y}\left(\frac{\partial_1^j\left([\gamma_1^m\gamma_2^n\sigma]_{F^{\mathrm{univ},\phi^{-2a}}}(\alpha_{2,a})^*\left(d\mathrm{Log}(\beta)^{\phi^{-2a}}\right)\right)}{w_{\mathrm{can}}^{\otimes 1 + 2j}|_{\mathrm{LT}_{\infty}}}\right)\\
&\overset{(\ref{partialdkjrelate}), (\ref{esection2}), (\ref{delta})}{=} \lambda_E^{2j}(\gamma) \cdot \delta(\beta).
%&= \lambda_{E}(\gamma)\cdot \sum_{m = 0}^{p^{k_1}-1}\sum_{n = 0}^{p^{k_2}-1}\lambda_{E}^{-1}(\gamma\gamma_1^m\gamma_2^n)\cdot \theta_{t,y(2a-1)}\left(\left(\frac{[\gamma\gamma_1^m\gamma_2^n]_{F^{\mathrm{univ},\phi^{-2a+1}}}(\alpha_{1,a})^*\left(d\widetilde{\log}g_{\beta}^{\phi^{-2a+1}}\right)}{\alpha_{1,a}^*w_{F^{\mathrm{univ},\phi^{-2a+1}}}}\right)\right)\\
%&\overset{(\ref{rewritegamma})}{=} \lambda_{E}(\gamma)\cdot \sum_{m = 0}^{p^{k_1}-1}\sum_{n = 0}^{p^{k_2}-1}\lambda_{E}^{-1}(\gamma_{m,n}'\gamma_1^{m'}\gamma_2^{n'})\cdot \theta_{t,y(2a-1)}\left(\left([\gamma_{m,n}'\gamma_1^{m'}\gamma_2^{n'}]_{F^{\mathrm{univ},\phi^{-2a+1}}}(\alpha_{1,a})^*\left(d\mathrm{Log}(\beta)^{\phi^{-2a+1}}\right)\right)\right)\\
%&=\lambda_{E}(\gamma) \cdot \sum_{m = 0}^{p^{k_1}-1}\sum_{n = 0}^{p^{k_2}-1}\lambda_{E}^{-1}(\gamma_1^{m'}\gamma_2^{n'})\cdot \theta_{t,y(2a-1)}\left(\left([\gamma_1^{m'}\gamma_2^{n'}]_{F^{\mathrm{univ},\phi^{-2a+1}}}(\alpha_{1,a})^*\left(d\mathrm{Log}(\beta)^{\phi^{-2a+1}}\right)\right)\right)\\
%&\overset{(\ref{integerbijection})}{=}  \lambda_{E}(\gamma) \cdot \sum_{m = 0}^{p^{k_1}-1}\sum_{n = 0}^{p^{k_2}-1}\lambda_{E}^{-1}(\gamma_1^{m}\gamma_2^{n})\cdot \theta_{t,y(2a-1)}\left(\left([\gamma_1^{m}\gamma_2^{n}]_{F^{\mathrm{univ},\phi^{-2a+1}}}(\alpha_{1,a})^*\left(d\mathrm{Log}(\beta)^{\phi^{-2a+1}}\right)\right)\right) \\
%&= \lambda_{E}(\gamma)\cdot \theta_{t,y(2a-1)}\left(\left(\mathbf{e}^*\left(d\mathrm{Log}(\beta)^{\phi^{-2a+1}}\right)\right)\right) \overset{(\ref{delta0})}{=} \lambda_{E}(\gamma) \cdot \delta(\beta),
\end{split}
\end{equation}
where in the penultimate equality we use the fact that 
$$[\gamma_1^m\gamma_2^n]_{F^{\mathrm{univ},\phi^{-2a}}}(\alpha_{2,a}) \mapsto [\gamma\gamma_1^m\gamma_2^n]_{F^{\mathrm{univ},\phi^{-2a}}}(\alpha_{2,a})$$
is a permutation of the set of torsion points 
$$\{[\gamma_1^m\gamma_2^n]_{F^{\mathrm{univ},\phi^{-2a}}}(\alpha_{2,a})\}_{0 \le m \le p^{k_1}-1, 0 \le n \le p^{k_2}-1} \subset F^{\mathrm{univ},\phi^{-2a}}[p^a](\mathrm{LT}_{\infty}).$$
Now (1) follows for $\alpha = \gamma$.
%$$D_1^j\delta(\gamma\cdot \beta)= \lambda_E^{-1}(\gamma) \cdot [\gamma]_{F^{\mathrm{univ},\phi^{-2a+1}}}^*D_1^j\delta(\beta) \overset{(\ref{finalgammaidentity})}{=} D_1^j\delta(\beta).$$
 The case for general $\alpha \in \Lambda_{\mathcal{O}_{\mathbb{C}_p}}[1/p]$ follows immediately from the additivity of 
 $$\Lambda_{\mathcal{O}_{\mathbb{C}_p}}[1/p] \ni \alpha \mapsto D_1^j\delta(\alpha\cdot\beta) \in \mathbb{C}_p\llbracket q_{\mathrm{dR}}-1\rrbracket.$$\\

\textbf{(2)}: Recall $\Delta_{A_0} = \mathrm{Gal}(K(\frak{f}_0p^{\infty})/K_{\infty})$ from (\ref{DeltaA}) and $\Delta_K =\mathrm{Gal}(\mathcal{K}_{\infty}/K_{\infty})$ from (\ref{DeltaK}), so that there is a natural map $\Delta_{A_0} \twoheadrightarrow \Delta_K$. Recall $\gamma_1, \gamma_2 \in \Gamma_K$ from Definition \ref{gamma12definition}, and consider the set
\begin{equation}\label{considertheset}\Delta_{A_0} \times \{\gamma_1^m\gamma_2^n\}_{0 \le m \le p^{k_1}-1, 0\le n \le p^{k_2}-1} \overset{(\ref{fixA2})}{\subset} \mathrm{Gal}(K(\frak{f}_0p^{\infty})/K).
\end{equation}
Recall that the global Artin reciprocity map gives a surjection 
$$\mathbb{A}_K^{\times,(\infty)} \twoheadrightarrow \mathrm{Gal}(K(\frak{f}_0p^{\infty})/K);$$
thus from (\ref{fixA}) we get a surjection $\mathbb{A}_K^{\times,(\infty)} \twoheadrightarrow \Delta_{A_0}$. Let $\mathbf{S} \subset \mathbb{A}_K^{\times,(\infty)}$ be a set of integral prime-to-$pN$ id\`{e}les such that 
\begin{equation}\label{Sbijection}\mathbf{S} \subset \mathbb{A}_K^{\times,(\infty)} \twoheadrightarrow \Delta_{A_0}  \times \{\gamma_1^m\gamma_2^n\}_{0 \le m \le p^{k_1}-1, 0\le n \le p^{k_2}-1}
\end{equation}
is a bijection as a map of sets. 

Given $\sigma \in \Delta_K$, let 
$$\frak{b}_{\sigma}\in \mathbf{S}$$
be any element representing an element of $\Delta_{A_0}$ lying above $\sigma$ under $\Delta_{A_0} \twoheadrightarrow \Delta_K$. Similarly, given any 
$$\gamma \in \{\gamma_1^m\gamma_2^n\}_{0 \le m \le p^{k_1}-1, 0\le n \le p^{k_2}-1},$$
let 
$$\frak{b}_{\gamma} \in \mathbf{S}$$
be the unique element representing it. 

We need the following intermediate Lemma. 
\begin{lemma}\label{Srayclasslemma}Recall $a$ from (\ref{adefinition}). $\mathbf{S} \subset \mathbb{A}_K^{\times,(\infty)}$ surjects onto $\mathrm{Gal}(K(\frak{f}_0p^a)/K)$ under the global Artin reciprocity map 
$$\mathbb{A}_K^{\times,(\infty)} \twoheadrightarrow \mathrm{Gal}(K^{\mathrm{ab}}/K) \twoheadrightarrow \mathrm{Gal}(K(\frak{f}_0p^a)/K).$$
\end{lemma}

\begin{proof}Throughout the proof we will let ``$\overset{\mathrm{rec}}{\cong}$'' denote an isomorphism induced by the Artin reciprocity map 
$$\mathbb{A}_K^{\times,(\infty)} \twoheadrightarrow \mathrm{Gal}(K^{\mathrm{ab}}/K)$$
and let ``$\overset{\mathrm{rec}^{-1}}{\cong}$'' denote the inverse of such an isomorphism. Note that
$$\Delta_{A_0} \times \Gamma_K \overset{(\ref{fixA2})}{=} \mathrm{Gal}(K(\frak{f}_0p^{\infty})/K) \overset{\mathrm{rec}^{-1}}{\cong} \left((\mathcal{O}_K/\frak{f}_0)^{\times}/\mathcal{O}_K^{\times}\right) \times \mathcal{O}_{K_p}^{\times}$$
where the second isomorphism uses the fact that $w_{\frak{f}_0} = 1$ by Choice \ref{f0choice} so that $\mathcal{O}_K^{\times} \subset (\mathcal{O}_K/\frak{f}_0)^{\times}$. This gives 
\begin{equation}\label{reciprocitytors}\Delta_{A_0} \overset{\mathrm{rec}^{-1}}{\cong} \left((\mathcal{O}_K/\frak{f}_0)^{\times}/\mathcal{O}_K^{\times}\right) \times (\mathcal{O}_{K_p}^{\times})_{\mathrm{tors}}
\end{equation}
under the reciprocity map, $M_{\mathrm{tors}}$ denotes the torsion submodule of $M$. 

Since 
$$\mathrm{Gal}(K(\frak{f}_0p^a)/K) \overset{\mathrm{rec}^{-1}}{\cong} \left((\mathcal{O}_K/\frak{f}_0)^{\times}/\mathcal{O}_K^{\times}\right) \times (\mathcal{O}_{K_p}/p^a\mathcal{O}_{K_p})^{\times},$$
the assertion follows if we can show that the composition 
\begin{align*}\Delta_{A_0} \times \{\gamma_1^m\gamma_2^n\}_{0 \le m \le p^{k_1-1}, 0 \le n \le p^{k_2-1}} \subset \mathrm{Gal}(K(\frak{f}_0p^{\infty})/K) &\twoheadrightarrow \mathrm{Gal}(K(\frak{f}_0p^a)/K) \\
&\overset{\mathrm{rec}^{-1}}{\cong} \left((\mathcal{O}_K/\frak{f}_0)^{\times}/\mathcal{O}_K^{\times}\right) \times (\mathcal{O}_{K_p}/p^a\mathcal{O}_{K_p})^{\times}
\end{align*}
is surjective. By (\ref{reciprocitytors}), we see that 
$$\Delta_{A_0} \hspace{.5cm} \text{surjects onto} \hspace{.5cm} \left((\mathcal{O}_K/\frak{f}_0)^{\times}/\mathcal{O}_K^{\times}\right) \times (\mathcal{O}_{K_p}^{\times})_{\mathrm{tors}}$$
under the above map. It remains to show that 
$$\{\gamma_1^m\gamma_2^n\}_{0 \le m \le p^{k_1-1}, 0 \le n \le p^{k_2-1}} \hspace{.5cm} \text{surjects onto} \hspace{.5cm} (\mathcal{O}_{K_p}/p^a\mathcal{O}_{K_p})^{\times}/(\mathcal{O}_{K_p}^{\times})_{\mathrm{tors}}$$
under the above map. By Definition \ref{gamma12definition}, it surjects onto $\mathrm{Gal}((\mathcal{K}_{2a}\cap K_{\infty})/K)$, which, since $$K(p^a) \subset \mathcal{K}_{2a} = K(E[p^a])$$
by the theory of complex multiplication, in turn surjects onto $\mathrm{Gal}((K(p^{2a})\cap K_{\infty})/K)$. We will show that 
\begin{equation}\label{Kinftyrayclass}\mathrm{Gal}((K(p^{2a}) \cap K_{\infty})/K) \overset{(\mathrm{rec}^{-1})}{\cong} ((\mathcal{O}_{K_p}/p^a\mathcal{O}_{K_p})^{\times})/(\mathcal{O}_{K_p}^{\times})_{\mathrm{tors}},
\end{equation}
which then gives the assertion. 

By Proposition \ref{GammaKtotallyramified} with $M = K$ we have $K_{\infty} \subset K(p^{\infty})$. Since 
$$\mathcal{O}_{K_p}^{\times} \cong (\mathcal{O}_{K_p}^{\times})_{\mathrm{tors}} \times \mathbb{Z}_p^{\oplus 2},$$
the map 
$$\mathcal{O}_{K_p}^{\times}/\mathcal{O}_K^{\times} \overset{\mathrm{rec}}{\cong} \mathrm{Gal}(K(p^{\infty})/K) \twoheadrightarrow \mathrm{Gal}(K_{\infty}/K) \cong \mathbb{Z}_p^{\oplus 2}$$
factors through
$$\mathcal{O}_{K_p}^{\times}/(\mathcal{O}_{K_p}^{\times})_{\mathrm{tors}} \overset{\mathrm{rec}}{\cong} \mathrm{Gal}(K_{\infty}/K).$$
Now applying the projection $\mathrm{Gal}(K_{\infty}/K) \twoheadrightarrow \mathrm{Gal}((K(p^a) \cap K_{\infty})/K)$ to the right-hand side, which corresponds under reciprocity to reducing the left-hand side modulo $p^a\mathcal{O}_{K_p}$, we get (\ref{Kinftyrayclass}).

\end{proof}

Continuing the proof of Theorem \ref{Colemanmapthm} (2), by Lemma \ref{Srayclasslemma} $\mathbf{S}$ surjects onto $\mathrm{Gal}(K(\frak{f}_0p^a)/K)$ under Artin reciprocity, and thus it surjects onto $\mathrm{Gal}(K(\frak{f}_0)[p^a]/K)$ where $K(\frak{f}_0)[p^a]/K$ is the class field of the mixed ring/ray class group $\mathcal{C}\ell(\frak{f}_0)[p^a]$ (Definition \ref{mixedclassgroupdefinition}). Fix a subset $S \subset \mathbf{S}$ such that the composition 
$$S \subset \mathbf{S} \xrightarrow{\mathrm{rec}} \mathrm{Gal}(K(\frak{f}_0p^{\infty})/K) \twoheadrightarrow \mathrm{Gal}(K(\frak{f}_0)[p^a]/K)$$
is a bijection of sets. Take this choice of $S$ in Definition \ref{Sdefinition}. We may write
\begin{equation}\label{reciprocityrelate}\mathrm{rec}(\mathbf{S}) = \bigsqcup_{i = 0, (i,p) = 1}^{p^a-1}\sigma_i\cdot \mathrm{rec}(S)
\end{equation}
where $\sigma_i \in \mathrm{Gal}(K(\frak{f}_0p^a)/K)$ is the Artin symbol of the principal ideal $(i) = i\mathcal{O}_K$.

%From (\ref{lambdaFlambdaE}) we have $\lambda_{F^{\mathrm{univ}}}(\gamma) = \lambda_E(\gamma)$ for any $\gamma \in \Gamma_K$. 
We will use the notation of Definition \ref{thetatpowerseriesdefinition} below. Observe that we may view $\chi_E : \Delta_K \rightarrow \mathcal{O}_{K_p}^{\times}$ (see (\ref{chiE})) as a character $\chi_E : \Delta_{A_0} \rightarrow \mathcal{O}_{K_p}^{\times}$ via pullback along $\Delta_{A_0} \twoheadrightarrow \Delta_K$ (see (\ref{fixA2}) and (\ref{fixK})). Recall $\partial_k^j$ from (\ref{partialkjdefinition}). By the definition (\ref{xiE}), Proposition \ref{Eisensteinnumber} and (\ref{EtwistColeman}), we have 
\begin{align*}&D_1^j\delta(\xi_E)|_{q_{\mathrm{dR}} = 1} \overset{(\ref{delta}), (\ref{xiE})}{=}\sum_{m = 0}^{p^{k_1}-1}\sum_{n = 0}^{p^{k_2}-1}\sum_{\sigma \in \Delta_{A_0}}\lambda_E^{-(1+2j)}(\gamma_1^m\gamma_2^n\sigma)\\
&\hspace{3cm}\cdot\theta_{t,y}d_1^j\left(\frac{[\gamma_1^m\gamma_2^n\sigma]_{F^{\mathrm{univ},\phi^{-2a}}}(\alpha_{2,a})^*\left(d\widetilde{\log}\left(\mathrm{thicken}(g_{e^1(\frak{a};\frak{f}_0)}^{\phi^{-2a}})\right)\right)}{w_{\mathrm{can}}|_{\mathrm{LT}_{\infty}}}\right)|_{q_{\mathrm{dR}} = 1} \\
&\overset{(\ref{Sbijection})}{=} \sum_{\frak{b} \in \mathbf{S}}\lambda_E^{-(1+2j)}(\frak{b})\cdot\theta_{t,y}d_1^j\left(\frac{[\sigma_{\frak{b}}]_{F^{\mathrm{univ},\phi^{-2a}}}(\alpha_{2,a})^*\left(d\widetilde{\log}\left(\mathrm{thicken}(g_{e^1(\frak{a};\frak{f}_0)}^{\phi^{-2a}})\right)\right)}{w_{\mathrm{can}}|_{\mathrm{LT}_{\infty}}}\right)|_{q_{\mathrm{dR}} = 1} \\
&\overset{(\ref{reciprocityrelate})}{=} \sum_{\frak{b} \in S} \sum_{i = 0, (i,p) = 1}^{p-1}\lambda_E^{-(1+2j)}(i\frak{b})\cdot\theta_{t,y}d_1^j\left(\frac{[\sigma_i\sigma_{\frak{b}}]_{F^{\mathrm{univ},\phi^{-2a}}}(\alpha_{2,a})^*\left(d\widetilde{\log}\left(\mathrm{thicken}(g_{e^1(\frak{a};\frak{f}_0)}^{\phi^{-2a}})\right)\right)}{w_{\mathrm{can}}|_{\mathrm{LT}_{\infty}}}\right)|_{q_{\mathrm{dR}} = 1} \\
&\overset{(\ref{centralcharacter})}{=} \sum_{\frak{b} \in S}\lambda_E^{-(1+2j)}(\frak{b})\\
&\hspace{2.1cm}\cdot \sum_{i = 0, (i,p) = 1}^{p-1}\eta^{-(1+2j)}(i)\theta_{t,y}d_1^j\left(\frac{[\sigma_i\sigma_{\frak{b}}]_{F^{\mathrm{univ},\phi^{-2a}}}(\alpha_{2,a})^*\left(d\widetilde{\log}\left(\mathrm{thicken}(g_{e^1(\frak{a};\frak{f}_0)}^{\phi^{-2a}})\right)\right)}{w_{\mathrm{can}}|_{\mathrm{LT}_{\infty}}}\right)|_{q_{\mathrm{dR}} = 1} \\
&\overset{(\ref{etaK})}{=} \sum_{\frak{b} \in S}\lambda_E^{-(1+2j)}(\frak{b})\sum_{i = 0, (i,p) = 1}^{p-1}\eta(i)\cdot\theta_{t,y}d_1^j\left(\frac{[\sigma_i\sigma_{\frak{b}}]_{F^{\mathrm{univ},\phi^{-2a}}}(\alpha_{2,a})^*\left(d\widetilde{\log}\left(\mathrm{thicken}(g_{e^1(\frak{a};\frak{f}_0)}^{\phi^{-2a}})\right)\right)}{w_{\mathrm{can}}|_{\mathrm{LT}_{\infty}}}\right)|_{q_{\mathrm{dR}} = 1} \\
&\overset{(\ref{thickenEisenstein2})}{=} 12\frac{\#\Delta_{A_0}}{\#\Delta_K}\cdot \sum_{\frak{b} \in S}\lambda_E^{-(1+2j)}(\frak{b})\cdot \theta_t\left(d_1^j\left(\frac{[\sigma_{\frak{a}\frak{b}}]_{F^{\mathrm{univ}}}^*\tilde{w}_{1,\eta}^{\flat}}{w_{\mathrm{can}}|_{\mathrm{LT}_{\infty}}} - \mathbb{N}\frak{a}\frac{[\sigma_{\frak{b}}]_{F^{\mathrm{univ}}}^*\tilde{w}_{1,\eta}^{\flat}}{w_{\mathrm{can}}|_{\mathrm{LT}_{\infty}}}\right)(y)(q_{\mathrm{dR}})\right)|_{q_{\mathrm{dR}} = 1}\\
&= 12\frac{\#\Delta_{A_0}}{\#\Delta_K}(\lambda_E^{1+2j}(\frak{a}) - \mathbb{N}\frak{a})\cdot \sum_{\frak{b} \in S}\lambda_E^{-(1+2j)}(\frak{b})\cdot \theta_t\left(d_1^j\left(\frac{[\sigma_{\frak{b}}]_{F^{\mathrm{univ}}}^*\tilde{w}_{1,\eta}^{\flat}}{w_{\mathrm{can}}|_{\mathrm{LT}_{\infty}}}\right)(y)(q_{\mathrm{dR}})\right)|_{q_{\mathrm{dR}} = 1}\\
&\overset{(\ref{YIgin})}{=} 12\frac{\#\Delta_{A_0}}{\#\Delta_K}(\lambda_E^{1+2j}(\frak{a}) - \mathbb{N}\frak{a})  \cdot \sum_{\frak{b} \in S}\lambda_E^{-(1+2j)}(\frak{b})\cdot \theta_t\left(d_1^j\left(\frac{\tilde{w}_{1,\eta}^{\flat}}{w_{\mathrm{can}}|_{\mathrm{LT}_{\infty}}}\right)(y_{\frak{b}})(q_{\mathrm{dR}})\right)|_{q_{\mathrm{dR}} = 1}\\
&= 12\frac{\#\Delta_{A_0}}{\#\Delta_K}(\lambda_E^{1+2j}(\frak{a}) - \mathbb{N}\frak{a})  \cdot \sum_{\frak{b} \in S}\lambda_E^{-(1+2j)}(\frak{b})\cdot \theta_t\left(d_1^j\left(\frac{(g^a)^*\tilde{w}_{1,\eta}^{\flat}}{(g^a)^*w_{\mathrm{can}}}\right)(y_{\frak{b}}\cdot g^{-a})(q_{\mathrm{dR}})\right)|_{q_{\mathrm{dR}} = 1}\\
&\overset{(\ref{tildewflat}), (\ref{tildew1eta})}{=}12\frac{\#\Delta_{A_0}}{\#\Delta_K}(\lambda_E^{1+2j}(\frak{a}) - \mathbb{N}\frak{a})  \cdot \sum_{\frak{b} \in S}\lambda_E^{-(1+2j)}(\frak{b})\cdot \theta_t\left(d_1^j\left(\frac{w_{1,\eta}^{\flat}}{(g^a)^*w_{\mathrm{can}}}\right)(y_{\frak{b}} \cdot g^{-a})(q_{\mathrm{dR}})\right)|_{q_{\mathrm{dR}} = 1}\\
&\overset{(\ref{gwcan})}{=} \frac{12}{p^a}\frac{\#\Delta_{A_0}}{\#\Delta_K}(\lambda_E^{1+2j}(\frak{a}) - \mathbb{N}\frak{a})  \cdot \sum_{\frak{b} \in S}\lambda_E^{-(1+2j)}(\frak{b})\cdot \theta_t\left(d_1^j\left(\frac{w_{1,\eta}^{\flat}}{w_{\mathrm{can}}}\right)(y_{\frak{b}} \cdot g^{-a})(q_{\mathrm{dR}})\right)|_{q_{\mathrm{dR}} = 1}\\
&\overset{(\ref{YIgin2})}{=}\frac{12}{p^a}\frac{\#\Delta_{A_0}}{\#\Delta_K}(\lambda_E^{1+2j}(\frak{a}) - \mathbb{N}\frak{a})  \cdot \sum_{\frak{b} \in S}\lambda_E^{-(1+2j)}(\frak{b})\cdot \theta_t\left(d_1^j\left(\frac{w_{1,\eta}^{\flat}}{w_{\mathrm{can}}}\right)(y_{\frak{b},a})(q_{\mathrm{dR}})\right)|_{q_{\mathrm{dR}} = 1}\\
%&=-12(\lambda_E^{1+2j}(\frak{a}) - \mathbb{N}\frak{a})\frac{\#\Delta_{A_0}}{\#\Delta_K}C'' \cdot \sum_{\frak{b} \in S}\lambda_E^{-(1+2j)}(\frak{b})\\
%&\hspace{7cm}\cdot \theta_t\left(d_1^j\left(\frac{[\sigma_{\frak{b}}]_{F^{\mathrm{univ},\phi^{-2a+1}}}^*w_{1,\eta}^{\flat}}{w_{\mathrm{can}}|_{\mathrm{LT}_{\infty}^{\phi^{-2a+1}}}}\right)(y_{\frak{b},a})(q_{\mathrm{dR}})\right)|_{q_{\mathrm{dR}} = 1}\\
&\overset{(\ref{derivativeintegral})}{=}\frac{12}{p^a}\frac{\#\Delta_{A_0}}{\#\Delta_K}(\lambda_E^{1+2j}(\frak{a}) - \mathbb{N}\frak{a})  \cdot \sum_{\frak{b} \in S}\lambda_E^{-(1+2j)}(\frak{b})\cdot D_1^jG_{1,\eta}^{\flat}(y_{\frak{b},a})(q_{\mathrm{dR}})|_{q_{\mathrm{dR}} = 1}\\
&\overset{(\ref{globalmeasure2})}{=}\frac{12}{p^a}\frac{\#\Delta_{A_0}}{\#\Delta_K}(\lambda_E^{1+2j}(\frak{a}) - \mathbb{N}\frak{a})  \cdot D_1^j\mathcal{L}_{\lambda_E}|_{q_{\mathrm{dR}} = 1}.
\end{align*}

\end{proof}

Let $U_0 \subset \mathbb{U}_{\chi_E}^1(\lambda_E^{-1})$ be as in (\ref{saturation}). Recall the basis $\beta_0$ of $U_0$, and $\lambda_0$ from Choice \ref{beta0choice}.

\begin{proposition}\label{integrateproposition}Recall our basis $\beta_0 \in U_0$ (see Choice \ref{beta0choice}), and suppose that $\lambda_0 \in \Lambda_{\mathcal{O}_{K_p}}[1/p]$ from (\ref{multiple2}) satisfies
\begin{equation}\label{1notzero}\lambda_0(\lambda_E^{2j}) \neq 0
\end{equation}
where 
$$\lambda_E : \mathrm{Gal}(\mathcal{K}_{\infty}/K) \overset{(\ref{fixK})}{=} \Delta_K \times \Gamma_K \rightarrow \mathcal{O}_{K_p}^{\times}$$
is the character from (\ref{lambdaE}), and $\lambda_0(\lambda_E^{2j})$ is written in the notation of Definition \ref{characterevaluation}. For any $j \in \mathbb{Z}/(p-1)\times \mathbb{Z}_p$, let 
$$\{j(i)\} \subset \mathbb{Z}_{\ge 0} \subset \mathbb{Z}/(p-1) \times \mathbb{Z}_p$$
be any sequence converging to $j$. %Further suppose that 
%$$(\lambda_E^{1+2j}(\frak{a}) - \mathbb{N}\frak{a}) \neq 0.$$
Then have that 
\begin{equation}\label{alimit}\lim_{j(i) \rightarrow j}D_1^{j(i)}\delta|_{U_0}(\beta_0)|_{q_{\mathrm{dR}} = 1}
\end{equation}
converges to a value in $\mathbb{C}_p$.
\end{proposition}

\begin{proof}Recall $\lambda_0 \beta = \xi_E$ (see (\ref{multiple2})). Thus for any $j \in \mathbb{Z}_{\ge 0}$, 
\begin{equation}\label{equivariancecalculation}\begin{split}\lambda_0(\lambda_E^{2j})\cdot D_1^j\delta|_{U_0}(\beta)|_{q_{\mathrm{dR}} = 1} = \lambda_0(\lambda_E^{2j})\cdot D_1^j\delta|_{U_0}(\beta)|_{q_{\mathrm{dR}} = 1} &\overset{(\ref{equivariance})}{=} D_1^j\delta|_{U_0}(\xi_E)|_{q_{\mathrm{dR}} = 1} \\
&\overset{(\ref{explicitreciprocity})}{=} C_E \cdot (\lambda_E^{1+2j}(\frak{a}) - \mathbb{N}\frak{a})\cdot D_1^j\mathcal{L}_{\lambda_E}|_{q_{\mathrm{dR}} = 1}
\end{split}
\end{equation}
with $C_E \in \mathbb{C}_p^{\times}$.

%We need a lemma.

%\begin{lemma}\label{1not0lemma}$$1(\alpha) \neq 0.$$
%\end{lemma}
%\begin{proof}Suppose $1(\alpha) = 0$. Then (\ref{equivariancecalculation}) shows 
%$$D_1^j\mathcal{L}_{\lambda_E}(1) = 0$$
%for all $j \ge 0$. By (\ref{dkj}), this implies 
%$$\left(\frac{q_{\mathrm{dR}}d}{dq_{\mathrm{dR}}}\right)^j\mathcal{L}_{\lambda_E}(1) = 0$$
%for all $j \ge 0$. By Amice's transform, this implies $\mathcal{L}_{\lambda_E} = 0$ as a measure on $\mathbb{Z}_p$. However, from (\ref{interp1}), the fact that the root number of $L(\lambda_E^{-1}\chi,0)$ is 0 for infinitely many $\chi$ (\cite{WZhangrootnumber}), and the main result of Rohrlich \cite{Rohrlich}, we see that $\mathcal{L}_{\lambda_E} \neq 0$, a contradiction. 
%\end{proof}

From (\ref{equivariancecalculation}) and (\ref{1notzero}), we have
\begin{equation}\label{equivariancecalculation2}\lambda_0(\lambda_E^{2j}) \cdot D_1^j\delta|_{U_0}(\beta_0)|_{q_{\mathrm{dR}} = 1} = C_E \cdot (\lambda_E^{1+2j}(\frak{a}) - \mathbb{N}\frak{a})\cdot D_1^j\mathcal{L}_{\lambda_E}|_{q_{\mathrm{dR}} = 1}.
\end{equation}
Now dividing both sides by $\lambda_0(\lambda_E^{2j}) \neq 0$, the assertion follows from Lemma \ref{continuitylemma}.
%Now letting $\{j_0\} \subset \mathbb{Z}_{\ge 0}$ converge to $a$ in $\mathbb{Z}/(p-1) \times \mathbb{Z}_p$ so that $\kappa_-^{j_0} \rightarrow \kappa_-^a$ uniformly as functions $\mathbb{Z}_p^{\times} \rightarrow \mathbb{Z}_p^{\times}$, and invoking (\ref{support}) and Lemma \ref{1not0lemma}, we see that the right-hand side of (\ref{equivariancecalculation2}) converges and so (\ref{alimit}) converges.

\end{proof}

\begin{definition}Define\footnote{The notation suggests that one view $\delta'$ as the ``anticyclotomic derivative'' of $\delta$. In fact, $\delta'$ is a limit of derivatives given by the formula 
$$\delta' = \lim_{m \rightarrow \infty}D_1^{-1+p^m(p-1)}\delta|_{q_{\mathrm{dR}} = 1}.$$
On the other hand the above formula suggests that $\delta'$ can be viewed as a kind of $p$-adic primitive of $\delta|_{q_{\mathrm{dR}} = 1}$.
This is reflected in the literature in the terminologies ``$p$-adic Gross-Zagier formula'' (cf. \cite{BDP}) and ``$p$-adic Waldspurger formula'' (cf. \cite{LZZ}).}
\begin{equation}\label{delta'}\delta' := D_1^{-1}\delta|_{U_0}|_{q_{\mathrm{dR}} = 1} := \lim_{m \rightarrow \infty}D_1^{-1+(p-1)p^m}\delta|_{U_0}|_{q_{\mathrm{dR}} = 1} : U_0\hat{\otimes}_{\mathcal{O}_{K_p}}\mathcal{O}_{\mathbb{C}_p}  \rightarrow \mathbb{C}_p.
\end{equation}
%(Note that unlike for $\delta$, we have plugged in ``$q_{\mathrm{dR}}-1 = 0$'', i.e. evaluated at the trivial character 1, in the definition of $\delta'$.) 
This is well-defined by Proposition \ref{integrateproposition} for $j = -1$.%This definition makes sense because 
%\begin{equation}\label{inverselog}\left(\frac{q_{\mathrm{dR}}d}{dq_{\mathrm{dR}}}\right)^{-1}\delta|_{U_0} = (\kappa_-^{-1})^*\delta|_{U_0},
%\end{equation}
%where we recall $\kappa_- : \mathbb{Z}_p \rightarrow \mathbb{Z}_p$ is the moment function $x \mapsto x$; since $\delta|_{U_0}$ is supported on $\mathbb{Z}_p^{\times} \subset \mathbb{Z}_p$ by (\ref{support}), we have that for any continuous function $f$ on $\mathbb{Z}_p$, 
%$$(\kappa_-^{-1})^*\delta|_{U_0}(f) = (\kappa_-^{-1})^*\delta|_{U_0}(f|_{\mathbb{Z}_p^{\times}}) = \delta|_{U_0}(\kappa_-^{-1}|_{\mathbb{Z}_p^{\times}}f|_{\mathbb{Z}_p^{\times}}),$$
%which is well-defined since $\kappa_-^{-1} : \mathbb{Z}_p^{\times} \rightarrow \mathbb{Z}_p^{\times}$ is the continuous function $x \mapsto x^{-1}$. 
\end{definition}

We summarize some key properties of $\delta'$.

\begin{corollary}The map $\delta'$ satisfies:
\begin{enumerate}
\item For all $\alpha \in \Lambda_{\mathcal{O}_{\mathbb{C}_p}}[1/p]$ all $\beta \in U_0\hat{\otimes}_{\mathcal{O}_{K_p}}\mathcal{O}_{\mathbb{C}_p}$, 
\begin{equation}\label{equivariance2}\delta'(\alpha\beta) = \alpha(\lambda_E^{-2}) \cdot \delta'(\beta).
\end{equation}
\item We also have
\begin{equation}\label{explicitreciprocity2}\delta'(\xi_E) = C_E \cdot (\lambda_E^{-1}(\frak{a}) - \mathbb{N}\frak{a})\cdot   D_1^{-1}\mathcal{L}_{\lambda_E}|_{q_{\mathrm{dR}} = 1}.
\end{equation}
for some $C_E \in \mathbb{C}_p^{\times}$. Moreover,
\begin{equation}\label{Lvalueiff}\delta'(\xi_E) \neq 0 \iff D_1^{-1}\mathcal{L}_{\lambda_E}|_{q_{\mathrm{dR}} = 1} \neq 0.
\end{equation}
\end{enumerate}
\end{corollary}

\begin{proof}This follows from (1) and (2) of Theorem \ref{Colemanmapthm}, (\ref{equivariancecalculation}) and taking $j = -1 + p^m(p-1)$ and the limit $m \rightarrow \infty$ in (\ref{explicitreciprocity}).  The equivalence (\ref{Lvalueiff}) follows from the observation that $\lambda_E^{-1}(\frak{a}) \neq \mathbb{N}\frak{a}$ because $\frak{a} \subset \mathcal{O}_K$ is a proper integral ideal. 

\end{proof}

%\begin{remark}\label{equivariantat1remark} Note that (\ref{equivariance}) is not a standard $\Lambda_{\mathcal{O}_{K_p}}[1/p]$-equivariance property that one often sees in classical (height 1) Coleman maps (in particular, the ``$r_{\chi}$''s would be replaced by essentially $\chi$ divided by a Gauss sum, see (\ref{newinterpolationremark})). However, at the (primitive) trivial character $\chi = 1$, we have $r_1 = 1$, and hence the maps
%$$\delta|_{q_{\mathrm{dR}} = 1} : \mathbb{U}_{\chi_E}^1(\lambda_E^{-1}) \rightarrow \mathbb{C}_p, \hspace{1cm} \delta'|_{q_{\mathrm{dR}} = 1} : \mathbb{U}_{\chi_E}^1(\lambda_E^{-1}) \rightarrow \mathbb{C}_p$$
%\emph{are} $\Lambda_{\mathcal{O}_{K_p}}[1/p]$-equivariant in the obvious sense. 
%\end{remark}

\subsection{Choice of decomposition of module of norm-compatible systems of semi-local units: rank 0 case}\label{rank0preview}Our choice of $U_0$ in the rank 0 case will follow that of \cite[Section 11]{RubinMC}. We will choose 
$$U_2 = \ker \delta|_{q_{\mathrm{dR}} = 1}.$$
As in loc. cit., using Wiles's explicit reciprocity law describing the local Kummer pairing, one can relate $U_2$ to the Kummer local condition at $p$ describing the global Selmer group $\mathrm{Sel}_{p^{\infty}}(E/\mathbb{Q})$. In particular, under the assumption $\mathrm{corank}_{\mathbb{Z}_p}\mathrm{Sel}_{p^{\infty}}(E/\mathbb{Q}) = 0$ we will show in Section \ref{rank0section} that $U_2 \cap \overline{\mathcal{E}}_{\chi_E}^1(\lambda_E^{-1}) = 0$ and that there is a map
$$\mathrm{Hom}(\mathcal{X}_{\chi_E}/\mathrm{rec}(U_2),K_p/\mathcal{O}_{K_p})^{\mathrm{Gal}(\mathcal{K}_{\infty}/K)} \rightarrow \mathrm{Sel}_{p^{\infty}}(E/\mathbb{Q})$$
with finite kernel and cokernel. By Theorem \ref{RMC} and (\ref{explicitreciprocity}) with $j = 0$ (noting that 
$$\lambda_E^0(\frak{a}) - \mathbb{N}\frak{a}  = 1-\mathbb{N}\frak{a} \neq 0$$
since $\frak{a} \subset \mathcal{O}_K$ is a proper ideal), this implies 
$$C \cdot L(E/\mathbb{Q},1) = \mathcal{L}_{\lambda_E}|_{q_{\mathrm{dR}} = 1} \neq 0,$$
for some $C \in \mathbb{C}_p^{\times}$, and hence $L(E/\mathbb{Q},1) \neq 0$ which gives the rank 0 $p$-converse theorem. 

\section{Rank 0 $p$-Converse Theorem}\label{rank0section}
%Let $K$ be an imaginary quadratic field of class number 1, and $E/K$ any elliptic curve with CM by $\mathcal{O}_K$. Denote its associated algebraic Hecke character of type $(1,0)$ by $\lambda$ and denote its conductor by $\frak{g}$. We now apply the results of the previous section to $L = K(1) = K$ (see Proposition \ref{extendtoconductor1}), $L_n = K(p^n)$, so that $G_{\infty} = \mathrm{Gal}(L_{\infty}/K) \cong \Gamma \times \Delta$, where $\Gamma \cong \mathrm{Gal}(L_{\infty}/L')$ and $\Delta = \mathrm{Gal}(L'/K)$. Note that $K(p^{\infty}) = K(E[p^{\infty}])$ (where the last equality follows from the theory of complex multiplication). Let $\lambda : \mathrm{Gal}(L_{\infty}/K)\rightarrow \mathcal{O}_{K_p}^{\times}$ be the local reciprocity character, and let $\chi_E : \mathrm{Gal}(L'/K) \rightarrow \mathcal{O}_{K_p}^{\times}$ be its restriction to $\Delta$.

Let $E/\mathbb{Q}$ be an elliptic curve with CM by $\mathcal{O}_K$ as in Choice \ref{Echoice} (thus $K$ has class number 1), and $p$ a prime that ramifies in $K/\mathbb{Q}$. Descent in the rank 0 case essentially entails replacing the map ``$\delta$'' in \cite[Proposition 11.7]{RubinMC} with the map 
$$\delta|_{U_0}|_{q_{\mathrm{dR}} = 1} : U_0 \rightarrow \mathbb{C}_p,$$ and following largely the same arguments as in Section 11 of op. cit.. In essence, our map $\delta|_{U_0}$ is the $q_{\mathrm{dR}}$-expansion of a thickening of $\delta$ from op. cit., which is in turn the logarithmic derivative of Coleman power series. 

\subsection{The Selmer group via local class field theory}For certain reasons of convenience, we will emphasize the $GL_1/K$ viewpoint in this section rather than the $GL_2/\mathbb{Q}$-viewpoint (for example, viewing $E[p^{\infty}]$ as the $\mathrm{Gal}(\overline{K}/K)$-module $(K_p/\mathcal{O}_{K_p})(\lambda_E)$ (where $\lambda_E : \mathrm{Gal}(\mathcal{K}_{\infty}/K) \rightarrow \mathcal{O}_{K_p}^{\times}$ is the character from (\ref{lambdaE})) rather than as the standard $\mathrm{Gal}(\overline{\mathbb{Q}}/\mathbb{Q})$-module). 

\begin{convention}\label{groupcohomologyconvention}Let $M$ be a field and let $W$ be a topological $\mathrm{Gal}(\overline{M}/M)$-module. For any $i \in \mathbb{Z}_{\ge 0}$, we will let
$$H^i(M,W) := H^i(\mathrm{Gal}(\overline{M}/M),W).$$
Similarly, given a Galois extension $L/M$ and a topological $\mathrm{Gal}(L/M)$-module $W$, for any $i \in \mathbb{Z}_{\ge 0}$ let
$$H^i(L/M,W) := H^i(\mathrm{Gal}(L/M),W).$$
\end{convention}

%\begin{definition}\label{ccharacterdefinition}\begin{enumerate}
%\item Let $c : \mathrm{Gal}(\mathcal{K}_{\infty}/K) \xrightarrow{\sim} \mathrm{Gal}(\mathcal{K}_{\infty}/K)$ denote the automorphism given by conjugation by any element of $\mathrm{Gal}(\mathcal{K}_{\infty}/\mathbb{Q})$ lifting the nontrivial element of $\mathrm{Gal}(K/\mathbb{Q})$ (i.e. ``complex conjugation''). 
%\item Given any character $\rho : \mathrm{Gal}(\mathcal{K}_{\infty}/K) \rightarrow \overline{\mathbb{Q}}_p^{\times}$, let $\rho^c = \rho \circ c$. 
%\end{enumerate}
%\end{definition}

Let $\varpi \in \mathcal{O}_K$ from (\ref{pichoice}), which is in particular a uniformizer of $\mathcal{O}_{K_p}$. For any place $v$ of $K$, consider the short exact sequence of $\mathcal{O}_K[\mathrm{Gal}(\overline{K}_v/K_v)]$-modules 
$$0 \rightarrow E[\varpi^n] \rightarrow E\xrightarrow{\varpi^n} E \rightarrow 0.$$
From the associated long exact sequence we get a short exact sequence
\begin{equation}\label{localdescent}0 \rightarrow E(K_v) \otimes_{\mathcal{O}_K} \mathcal{O}_K/\varpi^n\mathcal{O}_K \rightarrow  H^1(K_v,E[\varpi^n]) \rightarrow H^1(K_v,E)[\varpi^n] \rightarrow 0.
\end{equation}
We also extend this definition to the case $n = \infty$ (obtained by taking a direct limit as $n \rightarrow \infty$ of (\ref{localdescent})), in which case we get a sequence
$$0 \rightarrow E(K_v) \otimes_{\mathcal{O}_K} K_p/\mathcal{O}_{K_p} \rightarrow H^1(K_v,W(\lambda_E)) \rightarrow H^1(K_v,E)[p^{\infty}] \rightarrow 0.$$

%(This comes from the long exact sequence attached to the limit as $n \rightarrow \infty$ of the short exact sequences of $0 \rightarrow E[\varpi^n] \rightarrow E\xrightarrow{\varpi^n} E \rightarrow 0$, recalling that $\varpi$ denotes a uniformizer of $\mathcal{O}_{K_p}$.)

%Briefly consider the $\mathrm{Gal}(\overline{\mathbb{Q}}/\mathbb{Q})$-module $E[p^{\infty}]$ (recall $E/\mathbb{Q}$ is our elliptic curve with CM by $\mathcal{O}_K$). Let $E[p^{\infty}]|_{G_K}$ denote the $\mathrm{Gal}(\overline{K}/K)$-module obtained by restricting the $\mathrm{Gal}(\overline{\mathbb{Q}}/\mathbb{Q})$-action to $\mathrm{Gal}(\overline{K}/K) \subset \mathrm{Gal}(\overline{\mathbb{Q}}/\mathbb{Q})$. Then we have a decomposition of $\mathrm{Gal}(\overline{K}/K)$-modules (in the notation of Definition \ref{ccharacterdefinition})
%\begin{equation}\label{decompositionGK}E[p^{\infty}]|_{G_K} \cong W(\lambda_E) \oplus W(\lambda_E^c).
%\end{equation}

\begin{definition}\label{GL1Selmerdefinition}%For this definition, let $\lambda : \Gamma = \mathrm{Gal}(L_{\infty}/L') \rightarrow \overline{\mathbb{Q}}_p^{\times}$ be any $p$-adic Galois character (as opposed to the $\lambda$ fixed above). Let $\mathcal{O}_{\lambda}$ denote the finite extension of $\mathcal{O}_{L_p}$ generated by values of $\lambda$, and let $\mathcal{K}_{\lambda}$ denote its fraction field.
\begin{enumerate}
\item Let $\rho : \mathrm{Gal}(\mathcal{K}_{\infty}/K) \rightarrow \mathcal{O}_{K_p}^{\times}$ be any continuous character. Let $$T(\rho) := \mathcal{O}_{K_p}(\rho),$$
which is the free rank 1 $\mathcal{O}_{K_p}$-module with $\mathrm{Gal}(\mathcal{K}_{\infty}/K)$ acting through $\rho$. Let
$$V(\rho) := T(\rho)\otimes_{\mathbb{Z}_p}\mathbb{Q}_p \hspace{1cm} W(\rho) := V(\rho)/T(\rho) = T(\rho) \otimes_{\mathbb{Z}_p} \mathbb{Q}_p/\mathbb{Z}_p.$$
\item For every place $v$ of $K$, fix an algebraic closure $\overline{K}_v$ of the $v$-adic completion $K_v$ and fix an embedding $\overline{K} \subset \overline{K}_v$ (when $v = \frak{p}$, take the embedding determined by (\ref{fixembeddings})). Let $G_v \subset \mathrm{Gal}(\mathcal{K}_{\infty}/K)$ denote the decomposition group corresponding to these choices of embeddings.
\item %Recall the notation of Convention \ref{groupcohomologyconvention}. Then the taking the limit as $n \rightarrow \infty$ of the long exact sequences attached to $0 \rightarrow E[p^n] \rightarrow E\xrightarrow{p^n} E \rightarrow 0$, we get a map for every place $v$ of $K$
%\begin{equation}\label{KKummer}H^1(K_v,W(\lambda_E)) \oplus H^1(K_v,W(\lambda_E^c)) \overset{(\ref{decompositionGK})}{=} H^1(K_v,E[p^{\infty}]|_{G_K}) \rightarrow H^1(K_v,E)[p^{\infty}].
%\end{equation}
We define the \emph{Selmer group} to be 
$$S(\lambda_E) := \ker\left(\prod_v\mathrm{res}_v' : H^1(K,W(\lambda_E)) \rightarrow \prod_v H^1(G_v,E)\right)$$
where the product runs over all places $v$ of $K$. Here, the maps $\mathrm{res}_v' : H^1(K,W(\lambda_E)) \rightarrow H^1(G_v,E)$ are defined to be the compositions
$$H^1(K,W(\lambda_E)) \xrightarrow{\mathrm{res}_v} H^1(G_v,W(\lambda_E)) \rightarrow H^1(G_v,E)$$
where 
$$\mathrm{res}_v : H^1(K,W(\lambda_E)) \rightarrow H^1(G_v,W(\lambda_E))$$
is the restriction homomorphism induced by $G_v \subset \mathrm{Gal}(\overline{K}/K)$ and 
$$H^1(G_v,W(\lambda_E)) \rightarrow H^1(G_v,E)$$
is the 
second arrow of (\ref{localdescent}) in the case where $n = \infty$. %restriction of the map (\ref{KKummer}) to $H^1(K_v,W(\lambda_E))$ (recall $\rho = \lambda_E$ or $\lambda_E^c$). %induced by the taking the limit as $n \rightarrow \infty$ of the long exact sequences attached to the short exact sequences of $G_v$-modules $0 \rightarrow E[p^n] \rightarrow E \xrightarrow{p^n} E \rightarrow 0$. 
\item %For any field extension $L/K$ with $L \subset \overline{K}$ and every place $v'$ of $L$, fix an algebraic closure $\overline{L}_{v'}$ of the $v'$-adic completion $L_{v'}$ and fix an embedding $\overline{L} \subset \overline{L}_{v'}$ (when $v'|\frak{p}$ is the unique prime fixed by (\ref{fixembeddings}), we choose the embedding $\overline{L} \subset \overline{L}_{v'}$ to be the one induced by (\ref{fixembeddings})). Let $G_{v'} \subset \mathrm{Gal}(\overline{L}/L)$ denote the decomposition group corresponding to these choices of embeddings. Let $\rho : \mathrm{Gal}(\mathcal{K}_{\infty}/K) \rightarrow \mathcal{O}_{K_p}^{\times}$ be any continuous character. 
We define the \emph{relaxed Selmer group} to be 
$$S'(\rho) := \ker\left(\prod_{v\nmid\frak{p}}\mathrm{res}_{v} : H^1(K,W(\rho)) \rightarrow \prod_{v\nmid\frak{p}}H^1(G_{v},W(\rho))\right),$$
where the product runs over all places $v$ of $K$ not dividing $\frak{p}$. Here $\mathrm{res}_{v} : H^1(K,W(\rho)) \rightarrow H^1(G_{v},W(\rho))$ is the restriction homomorphism induced by $G_{v} \subset \mathrm{Gal}(\overline{K}/K)$. %When $L = K$, we will simply write 
%$$S'(\rho)(K) = S'(\rho).$$
\end{enumerate}
\end{definition}

In this section, we will primarily be interested in the case $\rho = \lambda_E$ from (\ref{lambdaE}), in which case $W(\lambda_E) = E[p^{\infty}]$ as $\mathrm{Gal}(\overline{K}/K)$-modules. 

\begin{remark}Our notation conforms with that of \cite[Section 11]{RubinMC}, where the Selmer groups $S$ and $S'$ in op. cit. correspond to $S(\lambda_E)$ and $S'(\lambda_E)$ in our setting.
\end{remark}

\begin{remark}One can also define analogues $S(\rho)$ of the Selmer group $S(\lambda_E)$ for general $\rho$ using Bloch-Kato's Galois-theoretic descriptions of Selmer groups \cite{BlochKato}.
\end{remark}

Recall our notation for all $n \in \mathbb{Z}_{\ge 0} \cup \{\infty\}$ 
$$\mathcal{K}_n = K(E[\frak{p}^n])$$
and let
$$\mathcal{K}_{n,p} := K_p(E[\frak{p}^n]).$$
Recall the identification 
$$\mathrm{Gal}(\mathcal{K}_{\infty}/K) \overset{(\ref{fixK})}{=} \Gamma_K \times \Delta_K.$$
Similarly, fix an identification 
\begin{equation}\label{fixKp}\mathrm{Gal}(\mathcal{K}_{\infty,p}/K_p) = \Gamma_K \times \Delta_{K,p}
\end{equation}
such that under the inclusion $\mathrm{Gal}(\mathcal{K}_{\infty,p}/K_p) \overset{(\ref{fixembeddings})}{\subset} \mathrm{Gal}(\mathcal{K}_{\infty}/K)$, we have $\Delta_{K,p} \subset \Delta_K$ and $\Gamma_K$ maps isomorphically onto $\Gamma_K$. Note that this is possible since $\Gamma_K = \mathrm{Gal}(K_{\infty}/K)$ is equal to its own decomposition subgroup at the prime of $\mathcal{O}_{K_{\infty}}$ above $\frak{p}$ fixed by (\ref{fixembeddings}), by Proposition \ref{GammaKtotallyramified} with $M = K$.

\begin{definition}\label{slashisotypicdefinition}\begin{enumerate}
\item Given a group $H$ and an $H$-module $M$, let $M^H \subset M$ denote the submodule consisting of elements fixed by every element of $H$. 
\item Given a $p$-adically complete $\mathbb{Z}_p$-algebra $R$ and an $R\llbracket \mathrm{Gal}(\mathcal{K}_{\infty}/K)\rrbracket$-module and a character $\chi \in \hat{\Delta}_K$ (recall $\hat{\Delta}_K$ is the group of $\mathbb{C}_p^{\times}$-valued characters on $\Delta_K$, per Definition \ref{isotypicdefinition}), let
$$M|_{\chi} := M \otimes_{R\llbracket \mathrm{Gal}(\mathcal{K}_{\infty}/K)\rrbracket, \chi} R[\chi]\llbracket \Gamma_K\rrbracket,$$
where $R[\chi]$ is the extension of $R$ obtained by adjoining the values of $\chi$, and the tensor product maps $\chi : \Delta_K \rightarrow R[\chi]^{\times}$. In other words, $M|_{\chi}$ is the maximal subquotient of $M$ on which $\Delta_K$ acts through $\chi$. 

\item Note that $M|_{\chi}$ is different from $M_{\chi}$ from Definition \ref{isotypicdefinition}, as $M|_{\chi}$ is an $R\llbracket \Gamma_K\rrbracket$-module and subquotient of $M$ and $M_{\chi}$ is an $R\llbracket \Gamma_K\rrbracket [1/p]$-submodule and subquotient of $M \otimes_{\mathbb{Z}_p} \mathbb{Q}_p$. It is easy to see that there is a natural identification 
$$M|_{\chi} \otimes_{\mathbb{Z}_p}\mathbb{Q}_p \xrightarrow{\sim} M_{\chi}.$$ 

\item Similarly, given an $R\llbracket \mathrm{Gal}(\mathcal{K}_{\infty,p}/K_p)\rrbracket$-module and a character $\chi \in \hat{\Delta}_{K,p}$, let 
$$M|_{\chi} := M \otimes_{R\llbracket \mathrm{Gal}(\mathcal{K}_{\infty,p}/K_p)\rrbracket, \chi} R[\chi]\llbracket \Gamma_K\rrbracket,$$
where $R[\chi]$ is as above.

\item Given an $R\llbracket \mathrm{Gal}(\mathcal{K}_{\infty,p}/K_p)\rrbracket$-module $M$ and a character $\chi \in \hat{\Delta}_{K,p}$ (recall (\ref{fixKp})), we similarly define $M|_{\chi}$ (using the same definition as in (2) with $\Delta_K$ replaced by $\Delta_{K,p}$). We can similarly define $M_{\chi}$ as in Definition \ref{isotypicdefinition}, and we have a natural identification $M|_{\chi} \otimes_{\mathbb{Z}_p}\mathbb{Q}_p \xrightarrow{\sim} M_{\chi}$. 
\end{enumerate}
\end{definition}

%Recall our previously fixed decomposition (\ref{fixK}), which identifies $\mathrm{Gal}(\mathcal{K}_{\infty}/K) \cong \Gamma_K \times \Delta_K$. Note also that since $\frak{p}$ is totally ramified in $\mathcal{K}_{\infty}/K$, there is a natural identification $\mathrm{Gal}(\mathcal{K}_{\infty}/K) = \mathrm{Gal}(\mathcal{K}_{\infty,p}/K_p)$ with respect to our fixed embedding (\ref{fixembeddings}). 

For the remainder of this section, all $\mathrm{Hom}$s are in the category of $\mathcal{O}_{K_p}$-modules unless otherwise indicated.

\begin{proposition}[Theorem 11.2 of \cite{RubinMC}]\label{relaxSelmertheorem}
The restriction map 
$$H^1(K,W(\lambda_E)) \rightarrow H^1(K_{\infty},W(\lambda_E))$$
induces a map with finite kernel and cokernel
$$S'(\lambda_E) \rightarrow \mathrm{Hom}(\mathcal{X},W(\lambda_E))^{\mathrm{Gal}(\mathcal{K}_{\infty}/K)} = \mathrm{Hom}(\mathcal{X}|_{\chi_E},W(\lambda_E))^{\mathrm{Gal}(\mathcal{K}_{\infty}/K)}.$$
\end{proposition}

\begin{proof}This is well-known, and follows from the following observations: Note that $\mathcal{X} = \mathrm{Gal}(M_{\infty}/\mathcal{K}_{\infty})$ and that $\mathcal{K}_{\infty} = K(E[p^{\infty}])$ and $M_{\infty}/\mathcal{K}_{\infty}$ is the maximal pro-$p$-abelian extension unramified outside $\frak{p}$, so that $\mathrm{Gal}(\overline{K}/\mathcal{K}_{\infty})$ acts trivially on $W(\lambda_E)$ and so we have the ``Kummer'' isomorphism
\begin{equation}\label{Kummerinclusion}H^1(\mathcal{K}_{\infty},W(\lambda_E)) \cong \mathrm{Hom}(\mathrm{Gal}(\overline{K}/\mathcal{K}_{\infty}),W(\lambda_E)) \supset \mathrm{Hom}(\mathcal{X},W(\lambda_E)) = \mathrm{Hom}(\mathcal{X}|_{\chi_E},W(\lambda_E)).
\end{equation}
Since $M_{\infty}/\mathcal{K}_{\infty}$ is unramified outside $\frak{p}$, by \cite[Proposition 1.1]{RubinCompositio} we have (in the notation of Definition \ref{GL1Selmerdefinition})
$$S'(\lambda_E)(\mathcal{K}_{\infty}) = \mathrm{Hom}(\mathcal{X},W(\lambda_E)) = \mathrm{Hom}(\mathcal{X}|_{\chi_E},W(\lambda_E)).$$
Now the assertion follows from a standard descent argument using the fact that $\mathrm{Gal}(\overline{K}/\mathcal{K}_{\infty})$ acts trivially on $W(\lambda_E)$ and using the inflation-restriction exact sequence 
\begin{align*}0 \rightarrow H^1(\mathrm{Gal}(\mathcal{K}_{\infty}/K),W(\lambda_E)) \xrightarrow{\mathrm{inf}} H^1(K,W(\lambda_E)) &\xrightarrow{\mathrm{res}} H^1(\mathcal{K}_{\infty},W(\lambda_E))^{\mathrm{Gal}(\mathcal{K}_{\infty}/K)} \\
&\rightarrow H^2(\mathrm{Gal}(\mathcal{K}_{\infty}/K),W(\lambda_E)).
\end{align*}
See \cite[Lemma 1.2]{RubinCompositio} for details on this last step.

\end{proof}

\begin{lemma}[cf. Lemma 11.6 of \cite{RubinMC}]\label{restrictionlemma}The natural restriction map 
$$H^1(K_p,W(\lambda_E)) \xrightarrow{\mathrm{res}} H^1(\mathcal{K}_{\infty,p},W(\lambda_E))^{\mathrm{Gal}(\mathcal{K}_{\infty,p}/K_p)}$$
induces, via Artin reciprocity, a map
\begin{equation}\label{localHom'}H^1(K_p,W(\lambda_E)) \rightarrow \mathrm{Hom}(\mathbb{U}^1|_{\chi_E},W(\lambda_E))^{\mathrm{Gal}(\mathcal{K}_{\infty}/K)}
\end{equation}
of $\Lambda_{\mathcal{O}_{K_p}}$-modules with finite kernel and cokernel.
\end{lemma}

\begin{remark}The above map is not necessarily an isomorphism (unlike in loc. cit.) because we do not necessarily have $E(K_{\frak{p}})[\frak{p}] = 0$ in our setting (whereas the assumptions of loc. cit. force this). 
\end{remark}

\begin{proof}[Proof of Lemma \ref{restrictionlemma}]We will use the notation of Convention \ref{groupcohomologyconvention} throughout the proof. By the inflation-restriction exact sequence, noting that $\mathrm{Gal}(\overline{K}_p/\mathcal{K}_{\infty,p})$ acts trivially on $W(\lambda_E)$ (recall that $\lambda_E : \mathrm{Gal}(\mathcal{K}_{\infty}/K) \rightarrow \mathcal{O}_{K_p}^{\times}$ is the character from (\ref{lambdaE})), we have an exact sequence
\begin{align*}0 &\rightarrow H^1(\mathcal{K}_{\infty,p}/K_p,W(\lambda_E)) \xrightarrow{\mathrm{inf}} H^1(K_p,W(\lambda_E)) \\
&\xrightarrow{\mathrm{res}} H^1(\mathcal{K}_{\infty,p},W(\lambda_E))^{\mathrm{Gal}(\mathcal{K}_{\infty,p}/K_p)} \rightarrow H^2(\mathcal{K}_{\infty,p}/K_p,W(\lambda_E)).
\end{align*} 
From \cite[Lemma 2.2]{RubinSha}, we see that $H^1(\mathcal{K}_{\infty,p}/K_p,W(\lambda_E))$ is finite. By local duality and the fact that $W(\lambda_E)$ is isomorphic to $\mathrm{Hom}(W(\lambda_E),\mathcal{O}_{K_p}(1))$ via the Weil pairing (where $\mathcal{O}_{K_p}(1) = \mathcal{O}_{K_p} \otimes_{\mathbb{Z}_p}\mathbb{Z}_p(1) = \mathcal{O}_{K_p}\otimes_{\mathbb{Z}_p}\varprojlim_n\mu_{p^n}$ denotes the Tate twist of $\mathcal{O}_{K_p}$), we have 
$$E[p^{\infty}](K_p) = H^0(\mathcal{K}_{\infty,p}/K_p,W(\lambda_E)) \cong H^2(\mathcal{K}_{\infty,p}/K_p,W(\lambda_E)),$$
which is finite. Hence the map 
\begin{equation}\label{restrictionfinite}\mathrm{res} : H^1(K_p,W(\lambda_E)) \xrightarrow{\mathrm{res}} H^1(\mathcal{K}_{\infty,p},W(\lambda_E))^{\mathrm{Gal}(\mathcal{K}_{\infty,p}/K_p)}
\end{equation}
has finite kernel and cokernel. We will finish by showing that the target of (\ref{restrictionfinite}) is isomorphic, under the Artin reciprocity map, to
$$\mathrm{Hom}(\mathbb{U}^1|_{\chi_E},W(\lambda_E))^{\mathrm{Gal}(\mathcal{K}_{\infty}/K)}.$$

Since $\mathrm{Gal}(\overline{K}_p/\mathcal{K}_{\infty,p})$ acts trivially on $W(\lambda_E)$, we have
\begin{equation}\label{trivialGaloisaction}\begin{split}H^1(\mathcal{K}_{\infty,p},W(\lambda_E))^{\mathrm{Gal}(\mathcal{K}_{\infty,p}/K_p)} &= \mathrm{Hom}(\mathrm{Gal}(\overline{K}_p/\mathcal{K}_{\infty,p}),W(\lambda_E))^{\mathrm{Gal}(\mathcal{K}_{\infty,p}/K_p)} \\
&= \mathrm{Hom}(\mathrm{Gal}(\mathcal{K}_{\infty,p}^{\mathrm{ab}}/\mathcal{K}_{\infty,p})|_{\chi_E},W(\lambda_E))^{\mathrm{Gal}(\mathcal{K}_{\infty,p}/K_p)}
\end{split}
\end{equation}
where $\mathcal{K}_{\infty,p}^{\mathrm{ab}}$ denotes the maximal abelian extension of $\mathcal{K}_{\infty,p}$. Let
$$\mathcal{U}_E^1 := \varprojlim_{\mathrm{Nm}_n} \mathcal{O}_{\mathcal{K}_{n,p}}^{\times,1},$$
where $\mathcal{O}_{\mathcal{K}_{n,p}}^{\times,1}$ denotes the principal units in $\mathcal{O}_{\mathcal{K}_{n,p}}^{\times}$.  Note that $\mathbb{U}_E^1$ is isomorphic to a direct product of finitely many copies of $\mathcal{U}_E^1$ (cf. (\ref{Ufdecomposition})), with one direct factor corresponding to the prime ideal of $\mathcal{O}_{\mathcal{K}_{\infty}}$ above $\frak{p}$ fixed by (\ref{fixembeddings}). Let $\mathbb{U}^1 \rightarrow \mathcal{U}_E^1$ be the projection onto this factor. By local class field theory, since $K$ has class number 1, we have a map
\begin{equation}\label{localCFT}\mathbb{U}^1 \times \hat{\mathbb{Z}} \rightarrow \mathcal{U}_E^1 \times \hat{\mathbb{Z}} \xrightarrow{\sim} \mathrm{Gal}(\mathcal{K}_{\infty,p}^{\mathrm{ab}}/\mathcal{K}_{\infty,p}).
\end{equation}
By Lemma \cite[11.5(ii)]{RubinMC}, we have that $\chi_E$ is nontrivial on the decomposition group of $\frak{p}$, and so $\hat{\mathbb{Z}}|_{\chi_E} = 1$. Moreover by (\ref{fixKp}) the projection $\mathbb{U}^1 \rightarrow \mathcal{U}_E^1$ induces an isomorphism of $\Lambda_{\mathcal{O}_{K_p}}$-modules 
\begin{equation}\label{chiquotientiso}\mathbb{U}^1|_{\chi_E} \xrightarrow{\sim}\mathcal{U}_E^1|_{\chi_E}.
\end{equation} Hence (\ref{localCFT}) gives an isomorphism
\begin{equation}\label{localreciprocityisomorphism}\mathbb{U}^1|_{\chi_E} \xrightarrow{\sim} \mathcal{U}_E^1|_{\chi_E} \xrightarrow{\sim} \mathrm{Gal}(\mathcal{K}_{\infty,p}^{\mathrm{ab}}/\mathcal{K}_{\infty,p})|_{\chi_E}.
\end{equation}
This, along with (\ref{trivialGaloisaction}) and the compatibility of (\ref{fixK}) and (\ref{fixKp}), gives (\ref{localHom'}). 

\end{proof}

This has the following consequence. Given a group $G$, let $G_{\mathrm{div}} \subset G$ denote the maximal divisible subgroup. 

\begin{corollary}
The natural restriction map 
$$H^1(K_p,W(\lambda_E)) \xrightarrow{\mathrm{res}} H^1(\mathcal{K}_{\infty,p},W(\lambda_E))^{\mathrm{Gal}(\mathcal{K}_{\infty,p}/K_p)}$$
induces an isomorphism
\begin{equation}\label{localHom}H^1(K_p,W(\lambda_E))_{\mathrm{div}} \xrightarrow{\sim} \mathrm{Hom}(\mathbb{U}_{\chi_E}^1,W(\lambda_E))^{\mathrm{Gal}(\mathcal{K}_{\infty}/K)}.
\end{equation}
%of $\mathcal{O}_{K_p}\llbracket \mathrm{Gal}(\overline{K}_p/\mathcal{K}_{\infty,p})\rrbracket$-modules.
\end{corollary}

\begin{proof}From (\ref{localHom'}), we have a map with finite kernel and cokernel
\begin{align*}H^1(K_p,W(\lambda_E)) &\rightarrow  \mathrm{Hom}(\mathbb{U}^1|_{\chi_E}(\lambda_E^{-1}),K_p/\mathcal{O}_{K_p})^{\mathrm{Gal}(\mathcal{K}_{\infty}/K)} \\
&= \mathrm{Hom}(\left(\mathbb{U}^1|_{\chi_E}(\lambda_E^{-1})\right)^{\mathrm{Gal}(\mathcal{K}_{\infty}/K)},K_p/\mathcal{O}_{K_p})
\end{align*}
where 
$$\mathbb{U}^1|_{\chi_E}(\lambda_E^{-1}) := \mathbb{U}^1|_{\chi_E} \otimes_{\mathcal{O}_{K_p}}\mathcal{O}_{K_p}(\lambda_E^{-1}).$$
Hence we get a map of finitely-generated $\mathcal{O}_{K_p}$-modules with finite kernel and cokernel
$$\left(\mathbb{U}^1|_{\chi_E}(\lambda_E^{-1})\right)^{\mathrm{Gal}(\mathcal{K}_{\infty}/K)} \rightarrow \mathrm{Hom}(H^1(K_p,W(\lambda_E)),K_p/\mathcal{O}_{K_p}).$$
Given a finitely generated $\mathcal{O}_{K_p}$-module $M$, let $M_{\mathrm{free}}$ denote the $\mathcal{O}_{K_p}$-free part. Tensoring with $\otimes_{\mathbb{Z}_p}\mathbb{Q}_p$ and using the fact that 
$$\mathbb{U}^1|_{\chi_E} \otimes_{\mathbb{Z}_p}\mathbb{Q}_p = \mathbb{U}_{\chi_E}^1$$
(Definition \ref{slashisotypicdefinition} (3)), we then get an isomorphism
\begin{align*}\left(\mathbb{U}_{\chi_E}^1(\lambda_E^{-1})\right)^{\mathrm{Gal}(\mathcal{K}_{\infty}/K)} &\xrightarrow{\sim} \mathrm{Hom}(H^1(K_p,W(\lambda_E)),K_p/\mathcal{O}_{K_p})_{\mathrm{free}} \\
&= \mathrm{Hom}(H^1(K_p,W(\lambda_E))_{\mathrm{div}},K_p/\mathcal{O}_{K_p}).
\end{align*}
Applying $\mathrm{Hom}(\cdot,K_p/\mathcal{O}_{K_p})$ again, we arrive at (\ref{localHom}). 

\end{proof}

%Let $M_n/K(E[p^n])$ denote the maximal pro-$p$ abelian extension of $K(E[p^n])$ unramified outside of places above $p$. Let $M_{p,\infty} = \bigcup_{n \in \mathbb{Z}_{\ge 0}}M_{n,p}$ and $\mathcal{X}_p = \mathrm{Gal}(M_{p,\infty}/\mathcal{K}_{\infty,p})$. 
The $\chi_E$-isotypic component $\mathbb{U}_{\chi_E}^1 \xrightarrow{\mathrm{rec}} \mathcal{X}_{\chi_E}$ of the reciprocity map $\mathbb{U}^1 \xrightarrow{\mathrm{rec}} \mathcal{X}$ induces a map 
$$\mathrm{Hom}(\mathcal{X}_{\chi_E},W(\lambda_E)) \rightarrow \mathrm{Hom}(\mathbb{U}_{\chi_E}^1,W(\lambda_E)).$$
We denote this map by $f \rightarrow f|_{\mathbb{U}_{\chi_E}^1}$. From (\ref{localHom}) have the following commutative diagram (cf. \cite[p. 62]{RubinMC}):
\begin{equation}\label{Selmergroupdiagram}
\hspace{-.95cm}\begin{tikzcd}[column sep =small]
     &   &  &  \mathrm{Hom}(\mathcal{X}_{\chi_E},W(\lambda_E))^{\mathrm{Gal}(\mathcal{K}_{\infty}/K)}\arrow{d} \arrow{dr}{} & \\
     & 0 \arrow{r} & \left(E(K_p) \otimes_{\mathcal{O}_K}K_p/\mathcal{O}_{K_p}\right)_{\mathrm{div}}\arrow{r}{} \arrow{dr}{\varphi}& H^1(K_p,W(\lambda_E))_{\mathrm{div}} \arrow{d}{\cong} \arrow{r}& \left(H^1(K_p,E)[\frak{p}^{\infty}]\right)_{\mathrm{div}}\arrow{r}& 0\\
     & & & \mathrm{Hom}(\mathbb{U}_{\chi_E}^1,W(\lambda_E))^{\mathrm{Gal}(\mathcal{K}_{\infty}/K)} & 
\end{tikzcd}
\end{equation}
where the middle row is the exact sequence obtained by taking the divisible part of (\ref{localdescent}) with $v = \frak{p}$ and $n = \infty$. 
%$$0 \rightarrow E(K_p) \otimes_{\mathcal{O}_{K_p}}K_p/\mathcal{O}_{K_p} \rightarrow H^1(K_p,W(\lambda_E)) \rightarrow H^1(K_p,E)[\frak{p}^{\infty}] \rightarrow 0.$$
%Given $f \in \mathrm{Hom}((\mathcal{X})_{\chi_E},W(\lambda_E))$, denote by $f_p$ the image under the natural map $\mathrm{Hom}((\mathcal{X})_{\chi_E},W(\lambda_E)) \rightarrow \mathrm{Hom}((\mathcal{X}_p)_{\chi_E},W(\lambda_E))$ given by the natural inclusion $\mathcal{X}_p \subset \mathcal{X}$ (induced by $i_p$ from (\ref{fixembeddings})). 
 By (\ref{Selmergroupdiagram}), and the fact that the image of the Kummer map
$$E(K_p)\otimes_{\mathcal{O}_K}K_p/\mathcal{O}_{K_p} \rightarrow H^1(K_p,W(\lambda_E))$$
is contained in $H^1(K_p,W(\lambda_E))_{\mathrm{div}}$ (see \cite{BlochKato}), we have that it is equal to the image of 
$$\left(E(K_p)\otimes_{\mathcal{O}_K}K_p/\mathcal{O}_{K_p}\right)_{\mathrm{div}} \rightarrow H^1(K_p,W(\lambda_E))_{\mathrm{div}}.$$
Hence by Proposition \ref{relaxSelmertheorem}, 
\begin{equation}\label{Selmercharacterization}S(\lambda_E)_{\mathrm{div}} = \{f \in \mathrm{Hom}(\mathcal{X}_{\chi_E},W(\lambda_E))^{\mathrm{Gal}(\mathcal{K}_{\infty}/K)} : f|_{\mathbb{U}_{\chi_E}^1} \in \mathrm{im}(\varphi)\}.
\end{equation}
%We have a map $(\mu^1)_{\chi_E} := (\mu_{\mathrm{glob}}^1)_{\chi_E} : (\mathcal{U}^1)_{\chi_E}\otimes_{\mathbb{Z}_p} \rightarrow \Lambda(\mathrm{Gal}(\mathcal{K}_{\infty,p}/K_p),\mathcal{O}_{L_p})_{\chi_E}[1/p]$ induced by taking the $\chi_E$-component of (\ref{desiredmap4}) (with $\frak{f} = 1$, so that $\mathbb{U}^1 = \mathcal{U}^1$). 

%\begin{definition}Given a measure $\mu \in \Lambda(G,R)$ and a continuous function $f$ on $G$, we define the \emph{twist of $\mu$ by $f$} by 
%$$f^*\mu(g) = \mu(fg)$$
%for any continuous function $g$ on $G$.
%\end{definition}

%\begin{definition}Define
%\begin{equation}\label{deltadefinition}\delta := (\lambda_E/\chi_E)^*(\mu^1)_{\chi_E} : (\mathcal{U}^1)_{\chi_E}\otimes_{\mathbb{Z}_p}\mathbb{Q}_p \rightarrow \Lambda(\mathrm{Gal}(\mathcal{K}_{\infty,p}/K_p),\mathcal{O}_{K_p})_{\chi_E}[1/p].
%\end{equation}
%Here, recall our notation for pullback: for any function $f$ on $\mathrm{Gal}(L_{p,\infty}/K_p)$, and any $u \in (\mathcal{U}^1)_{\chi_E}$, $(\lambda_E/\chi_E)^*(\mu^1)_{\chi_E}(u)(f) = (\mu^1)_{\chi_E}(u)((\lambda_E/\chi_E) f)$. 
%\end{definition}

\subsection{Kummer maps and twisting}\label{Kummertwistsection}

Recall $\frak{f}_E$ is the conductor of $\lambda_E$. Recall that the positive integer $f_0$ generates $\frak{f}_E^{(p)} = \frak{f}(\lambda_E)^{(p)}$ (see Choice \ref{f0choice}), which by \cite[Proposition 3.13]{BDP2} and the fact that $K$ is one of the imaginary quadratic fields listed in Assumption \ref{pramifiedassumption} implies that $f_0^2$ is the prime-to-$p$ part of the conductor of $E/\mathbb{Q}$. Recall that for many of our results concerning the construction of $p$-adic $L$-functions, we made the technical assumption (Assumption \ref{pconductorassumption} (1), Assumption \ref{lambdalambdaEassumption}) that $f_0 \ge 4$. 

%\begin{assumption}\label{Nassumptionrank0}For the rest of the section until Theorem \ref{BSDrank0theorem}, assume that $f_0 \ge 4$. 
%\end{assumption}

\begin{choice}\label{A0'choice}Fix an ideal $\frak{f}_0' \subset \mathcal{O}_K$ with $(\frak{f}_0',\frak{p}) = 1$, $\frak{f}_E^{(p)}|\frak{f}_0'$ with $w_{\frak{f}_0'} = 1$. 
\begin{enumerate}
\item By \cite[Theorem II.1.4]{deShalit}, there is a type (1,0) algebraic Hecke character $\lambda_{A_0'}$ as well as an associated elliptic curve $A_0'/K(\frak{f}_0')$ with CM by $\mathcal{O}_K$; in the notation of loc. cit., we have
$$\psi_{A_0'} = \lambda_{A_0'} \circ \mathrm{Nm}_{K(\frak{f}_0')/K}.$$
(For example, we could take $A_0 = A_0'$ from Choice \ref{FixCMdefinition} but we desire to work in a slightly more general setting later, specifically in the proof of Theorem \ref{beta0nonzerotheorem}.) 
\item Let $F = \hat{A}_0'$ be a relative Lubin-Tate group for the unramified extension $L_p/K_p$ where 
$$L = K(\frak{f}_0').$$
%For $n \in \mathbb{Z}_{\ge 0} \cup \{\infty\}$, let $L_{p,n} = K(\frak{f}_0')_p(F[\varpi^n])$ as in Section \ref{Colemansection}.

%This ensures that 
%$$K(E[p^n]) \subset K(\frak{f}_0p^n) = L(A_0[p^n]).$$
%Note that since $K(A_0[\frak{f}_0])/K$ is abelian and ramified only at primes dividing $\frak{f}_0$ (by the criterion of N\'{e}ron-Ogg-Shafarevich; note that $A_0$ has good reduction at all primes of $L$ above $\frak{p}$), then $K(A_0[\frak{f}_0]) \subset K(\frak{f}_0)$. On the other hand, $K(\frak{f}_0) \subset K(A_0[\frak{f}_0])$ by the theory of complex multiplication. Hence $K(\frak{f}_0) = K(A_0[\frak{f}_0])$, and so 
%$$K(\frak{f}_0p^n) = K(\frak{f}_0)(A_0[\frak{f}_0p^n]) = K(\frak{f}_0)(A_0[p^n]) = L(A_0[p^n]).$$
%Denote the associated tower of local units by $\mathcal{U}'$, and the associated tower of principal local units by $(\mathcal{U}')^1$. Recalling that $\mathcal{U}^1$ denotes the tower of local units attached to $E$, by the previous sentence, we have $\mathcal{U}^1 \subset (\mathcal{U}')^1$. Moreover, the norm from $L(A_0[p^n]) \rightarrow K(E[p^n])$ induces a norm map $\mathrm{Nm} : (\mathcal{U}')^1 \rightarrow \mathcal{U}^1$. Denote the type $(1,0)$ Hecke character associated with $A_0$ by $\lambda_{A_0}$, and viewed as a Galois character on $\mathrm{Gal}(L(A_0[p^{\infty}])/K)$ 
\item Note that $L(A_0'[\frak{f}_0'\frak{p}^n]) = K(\frak{f}_0'\frak{p}^n)$ for any $n \in \mathbb{Z}_{\ge 0} \cup \{\infty\}$ by \cite[Lemma I.1.3 (ii)]{deShalit}. Fix a decomposition 
\begin{equation}\label{fixA'}\mathrm{Gal}(K(\frak{f}_0'p^{\infty})/K) \cong \Delta_{A_0'} \times \Gamma_K
\end{equation}
as in (\ref{fixA2}) (i.e. $\Delta_{A_0'} = \mathrm{Gal}(K(\frak{f}_0'p^{\infty})/K_{\infty})$) which maps to the decomposition (\ref{fixK}) under the restriction map $\mathrm{Gal}(K(\frak{f}_0'p^{\infty})/K) \twoheadrightarrow \mathrm{Gal}(\mathcal{K}_{\infty}/K)$. 
\item Let 
$$\chi_{A_0'} = \lambda_{A_0'}|_{\Delta_{A_0'}},$$ 
as in (\ref{chiA}), where we recall $K_{\infty}/K$ denotes the unique $\mathbb{Z}_p^{\oplus 2}$-extension of $K$. By the same argument as in the proof of Proposition \ref{sameCMcharacterproposition}, we have
\begin{equation}\label{sameCMcharacter2}\lambda_{A_0'}/\chi_{A_0'} = \lambda_E/\chi_E.
\end{equation}
\end{enumerate}
\end{choice}

\begin{choice}\label{generatorchoice}Henceforth, fix an $\mathcal{O}_{K_p}$-non-torsion element of $E(K_p)$. 
This induces an identification
\begin{equation}\label{generatorchoiceiso}\mathrm{Hom}(E(K_p)\otimes_{\mathcal{O}_K}K_p/\mathcal{O}_{K_p},K_p/\mathcal{O}_{K_p}) \otimes_{\mathbb{Z}_p}\mathbb{Q}_p = \mathrm{Hom}(K_p/\mathcal{O}_{K_p},K_p/\mathcal{O}_{K_p})\otimes_{\mathbb{Z}_p}\mathbb{Q}_p = K_p.
\end{equation}
Also fix an $\mathcal{O}_{K_p}$-non-torsion element of $A_0'(L_p)$. 
This induces an identification
\begin{equation}\label{generatorchoiceiso'}\mathrm{Hom}(A_0'(L_p) \otimes_{\mathcal{O}_K}K_p/\mathcal{O}_{K_p},K_p/\mathcal{O}_{K_p})\otimes_{\mathbb{Z}_p}\mathbb{Q}_p = \mathrm{Hom}(L_p/\mathcal{O}_{L_p},K_p/\mathcal{O}_{K_p})\otimes_{\mathbb{Z}_p}\mathbb{Q}_p = L_p.
\end{equation}
\end{choice}

\begin{definition}\label{deltaEdefinition}\begin{enumerate}
\item Let 
$$\delta_E : \mathbb{U}_{\chi_E}^1(\lambda_E^{-1}) \rightarrow \mathrm{Hom}(E(K_p)\otimes_{\mathcal{O}_K} K_p/\mathcal{O}_{K_p},K_p/\mathcal{O}_{K_p}) \otimes_{\mathbb{Z}_p}\mathbb{Q}_p \overset{(\ref{generatorchoiceiso})}{=} K_p,$$
be induced by pullback via $\varphi$ from (\ref{Selmergroupdiagram}) (cf. \cite[Proposition 11.7]{RubinMC}). %In particular, since by the $p$-adic logarithm $\log_p : E(K_p) \rightarrow K_p$, we have as $K_p$-vector spaces
%$$\mathrm{Hom}(E(K_p)\otimes_{\mathcal{O}_K}K_p/\mathcal{O}_{K_p},W(\lambda_E)) \otimes_{\mathbb{Z}_p}\mathbb{Q}_p  \cong \mathrm{Hom}(K_p/\mathcal{O}_{K_p},K_p/\mathcal{O}_{K_p}) \otimes_{\mathbb{Z}_p}\mathbb{Q}_p = K_p.$$
%$$\mathbb{U}_{\chi_E}^1(\lambda_E^{-1}) \xrightarrow{\text{local rec}} \mathcal{X}_{\chi_E}(\lambda_E^{-1}) \xrightarrow{$\phi$} \mathrm{Hom}(E(K_p)\otimes_{\mathcal{O}_K}K_p/\mathcal{O}_{K_p},K_p/\mathcal{O}_{K_p}) \otimes_{\mathbb{Z}_p}\mathbb{Q}_p,$$

\item Let $A_0'/L_p$ be as in Choice \ref{A0'choice}, where $L = K(\frak{f}_0')$. Viewing $\lambda_{A_0'}$ as a character 
$$\lambda_{A_0'} : \mathrm{Gal}(\overline{K}/K) \rightarrow \mathbb{C}_p^{\times},$$
we have that 
$$\lambda_{A_0'}|_{\mathrm{Gal}(\overline{K}/L)} = \psi_{A_0'}$$
is $\mathcal{O}_{K_p}^{\times}$-valued. Define an $\mathcal{O}_{K_p}\llbracket \mathrm{Gal}(L(A_0'[\frak{f}_0'p^{\infty}])/L)\rrbracket$-module 
$$W(\lambda_{A_0'}) = \left(K_p/\mathcal{O}_{K_p}\right)(\lambda_{A_0'})$$
as $K_p/\mathcal{O}_{K_p}$ with $\mathrm{Gal}(\overline{K}/L)$-action given by $\lambda_{A_0'}$. 

\item Recall $\chi_{A_0'} = \lambda_{A_0'}|_{\Delta_{A_0'}}$, where the restriction to $\Delta_{A_0'}$ is defined using the decomposition (\ref{fixA'}). Then $\chi_{A_0'}|_{\mathrm{Gal}(\overline{K}/L)}$ is also $\mathcal{O}_{K_p}^{\times}$-valued. Define, analogously to Definition \ref{slashisotypicdefinition}, an $\mathcal{O}_{K_p}\llbracket \Gamma_K\rrbracket$-module
$$\mathbb{U}_{A_0'}^1|_{\chi_{A_0'}} = \mathbb{U}_{A_0'}^1 \otimes_{\mathbb{Z}_p\llbracket \mathrm{Gal}(K(\frak{f}_0')(A_0'[\frak{f}_0'p^{\infty}])/L)\rrbracket,\chi_{A_0'}}\mathcal{O}_{K_p}\llbracket \Gamma_K\rrbracket.$$
\item Using a completely analogous argument to the proof of (\ref{localHom'}), one gets a map with finite kernel and cokernel
$$H^1(L_p,W(\lambda_{A_0'})) \rightarrow \mathrm{Hom}(\mathbb{U}_{A_0'}^1|_{\chi_{A_0'}},W(\lambda_{A_0'}))^{\mathrm{Gal}(L_p(A_0'[\frak{f}_0'p^{\infty}])/L_p)}.$$
By essentially the same argument as in the proof of (\ref{localHom}), one gets an isomorphism
$$H^1(L_p,W(\lambda_{A_0'}))_{\mathrm{div}} \xrightarrow{\sim} \mathrm{Hom}(\mathbb{U}_{A_0',\chi_{A_0'}}^1,W(\lambda_{A_0'}))^{\mathrm{Gal}(L_p(A_0'[\frak{f}_0'p^{\infty}])/L_p)}.$$
One gets an induced diagram for $A_0'$ analogous to (\ref{Selmergroupdiagram}), and thus a map 
$$\varphi_{A_0'} : \left(A_0'(L_p) \otimes_{\mathcal{O}_K}K_p/\mathcal{O}_{K_p}\right)_{\mathrm{div}} \rightarrow \mathrm{Hom}(\mathbb{U}_{A_0',\chi_{A_0'}}^1,W(\lambda_{A_0'}))^{\mathrm{Gal}(L_p(A_0'[\frak{f}_0'p^{\infty}])/L_p)}$$
which induces
$$\delta_{A_0'} : \mathbb{U}_{A_0',\chi_{A_0'}}^1(\lambda_{A_0'}^{-1}) \rightarrow   \mathrm{Hom}(A_0'(L_p) \otimes_{\mathcal{O}_K} K_p/\mathcal{O}_{K_p},K_p/\mathcal{O}_{K_p}) \otimes_{\mathbb{Z}_p}\mathbb{Q}_p \overset{(\ref{generatorchoiceiso'})}{=} L_p \xrightarrow{\mathrm{trace}_{L_p/K_p}} K_p.$$
%and the map of $K_p$-vector spaces
%\begin{align*}\mathrm{Hom}(A_0'(L_p)\otimes_{\mathcal{O}_K}K_p/\mathcal{O}_{K_p},W(\lambda_{A_0'})) \otimes_{\mathbb{Z}_p} \cong \mathbb{Q}_p = \mathrm{Hom}(A_0'(L_p) \otimes_{\mathcal{O}_K}&K_p/\mathcal{O}_{K_p},K_p/\mathcal{O}_{K_p}) \otimes_{\mathbb{Z}_p}\mathbb{Q}_p  \\
%&= L_p\xrightarrow{\mathrm{trace}_{L_p/K_p}} K_p,
%\end{align*}
%we get a map 

%Note that since the local reciprocity map is nonzero and the Kummer map is injective, $\delta_E$ and $\delta_{A_0'}$ are nonzero maps. and similarly with $\delta_{A_0'}$. Hence the targets of $\delta_E$ and $\delta_{A_0'}$ are isomorphic to $K_p$, and hence these maps are surjective.

%there is a trace map 
%$$\mathrm{Tr}_{L_{p,\infty}/K_{p,\infty}} : A_0'(L_{p,\infty}) = E(L_{p,\infty}) \rightarrow E(K_{p,\infty}).$$ 
\end{enumerate}
\end{definition}

Observe that since 
\begin{equation}\label{f'equal}K(\frak{f}_0')(A_0'[\frak{f}_0'\frak{p}^n]) = K(\frak{f}_0'\frak{p}^n)
\end{equation}
for all $n \gg 0$ by \cite[Proposition II.1.6]{deShalit}, then 
\begin{equation}\label{U'sequal}\mathbb{U}^1(\frak{f}_0') = \mathbb{U}_{A_0'}^1.
\end{equation}
By (\ref{sameCMcharacter2}), we have that $\lambda_{A_0'}/\chi_{A_0'}$ takes values in $\mathcal{O}_{K_p}^{\times}$. Let $\mathbf{1}_{A_0'}$ be the trivial character on $\Delta_{A_0'}$. Recall $\Lambda_{\mathcal{O}_{K_p}}[1/p]$ from Definition \ref{Lambdadefinition}. Then we have a well-defined $\Lambda_{\mathcal{O}_{K_p}}[1/p]$-module
$$\mathbb{U}_{A_0',\chi_{A_0'}}^1(\lambda_{A_0'}^{-1})_0 := \mathbb{U}_{A_0',\mathbf{1}_{A_0'}}^1((\lambda_{A_0'}/\chi_{A_0'})^{-1}) :=  \mathbb{U}_{A_0'}^1 \otimes_{\mathbb{Z}_p\llbracket \Delta_{A_0'} \times \Gamma_K\rrbracket} \mathbb{Z}_p\llbracket \Gamma_K\rrbracket \otimes_{\mathcal{O}_{K_p}}K_p(\chi_{A_0'}/\lambda_{A_0'}).$$
Note that 
$$\mathbb{U}_{A_0',\chi_{A_0'}}^1(\lambda_{A_0'}^{-1})_0 = \mathbb{U}^1(\frak{f}_0')_{\mathbf{1}_{A_0'}}((\lambda_{A_0'}/\chi_{A_0'})^{-1})$$
in light of (\ref{U'sequal}), and moreover
$$\mathbb{U}_{A_0',\chi_{A_0'}}^1(\lambda_{A_0'}^{-1})_0 \otimes_{K_p}L_p = \mathbb{U}_{A_0',\chi_{A_0'}}^1(\lambda_{A_0'}^{-1})$$
(cf. Remark \ref{isotypicdifferremark}). 

Similarly, define a $\Lambda_{\mathcal{O}_{K_p}}[1/p]$-module 
\begin{align*}\mathrm{Gal}(K^{\mathrm{ab}}/K(\frak{f}_0')(A_0'[\frak{f}_0'p^{\infty}])&)_{\chi_{A_0'}}(\lambda_{A_0'}^{-1})_0 := \mathrm{Gal}(K^{\mathrm{ab}}/K(\frak{f}_0')(A_0'[\frak{f}_0'p^{\infty}]))_{\mathbf{1}_{A_0'}}((\lambda_{A_0'}/\chi_{A_0'})^{-1}) \\
&\hspace{-.3cm}:= \mathrm{Gal}(K^{\mathrm{ab}}/K(\frak{f}_0')(A_0'[\frak{f}_0'p^{\infty}])) \otimes_{\mathbb{Z}_p\llbracket \Delta_{A_0'} \times \Gamma_K\rrbracket} \mathbb{Z}_p\llbracket \Gamma_K\rrbracket  \otimes_{\mathcal{O}_{K_p}}K_p(\chi_{A_0'}/\lambda_{A_0'})
\end{align*}
which satisfies 
$$\mathrm{Gal}(K^{\mathrm{ab}}/K(\frak{f}_0')(A_0'[\frak{f}_0'p^{\infty}]))_{\chi_{A_0'}}(\lambda_{A_0'}^{-1})_0 \otimes_{K_p}L_p = \mathrm{Gal}(K^{\mathrm{ab}}/K(\frak{f}_0')(A_0'[\frak{f}_0'p^{\infty}]))_{\chi_{A_0'}}(\lambda_{A_0'}^{-1}).$$

Let 
$$i_{\frak{f}_0',\chi_E} : \mathbb{U}_{\chi_E}^1(\lambda_E^{-1}) \rightarrow \mathbb{U}^1(\frak{f}_0')_{\mathbf{1}_{A_0'}}((\lambda_{A_0'}/\chi_{A_0'})^{-1}) = \mathbb{U}_{A_0',\chi_{A_0'}}^1(\lambda_{A_0'}^{-1})_0$$
be as in (\ref{1ftwist}). Let 
$$\mathrm{Nm}_{\frak{f}_0',\chi_E} : \mathbb{U}_{A_0',\chi_{A_0'}}^1(\lambda_{A_0'}^{-1})_0 = \mathbb{U}^1(\frak{f}_0')_{\mathbf{1}_{A_0'}}((\lambda_{A_0'}/\chi_{A_0'})^{-1}) \rightarrow \mathbb{U}_{\chi_E}^1(\lambda_E^{-1})$$
be as in (\ref{UEnormchitwist}).

\begin{proposition}\label{reciprocitydiagramproposition}\begin{enumerate}
\item There are $\alpha_1,\alpha_2 \in K_p^{\times}$ such that there is commutative diagram of $\Lambda_{\mathcal{O}_{K_p}}[1/p]$-equivariant maps
\begin{equation}\label{reciprocitydiagram}
\begin{tikzcd}[column sep =large]
     & \mathbb{U}_{\chi_E}^1(\lambda_E^{-1})\hspace{.5cm} \arrow{r}{\delta_E} \arrow[hook]{d}{i_{\frak{f}_0',\chi_E}}&K_p \arrow{d}{\alpha_1} \\
     & \mathbb{U}_{A_0',\chi_{A_0'}}^1(\lambda_{A_0'}^{-1})_0 \arrow{r}{\delta_{A_0'}} \arrow{d}{\mathrm{Nm}_{\frak{f}_0',\chi_E}}& K_p\arrow{d}{\alpha_2} \\
      & \mathbb{U}_{\chi_E}^1(\lambda_E^{-1})\hspace{.5cm}\arrow{r}{\delta_E}  & K_p
     \end{tikzcd}.
\end{equation}
Here the right vertical arrows denote multiplication by $\alpha_1$ and $\alpha_2$ respectively
% the first right vertical arrow is pullback by $\mathrm{Tr}_{L_{p,\infty}/K_{p,\infty}} : A_0'(L_{p,\infty}) \rightarrow E(K_{p,\infty})$, the second right vertical arrow is pullback by the inclusion $\mathrm{incl} : E(K_{p,\infty}) \subset E(L_{p,\infty}) = A_0'(L_{p,\infty})$, and the 
%and the left vertical arrows and right vertical arrows are isomorphisms of $\Lambda_{\mathcal{O}_{K_p}}[1/p]$-modules. 
 and $\Lambda_{\mathcal{O}_{K_p}}[1/p]$ acts on the target $K_p$ of $\delta_E$ and $\delta_{A_0'}$ through 
 $$\Lambda_{\mathcal{O}_{K_p}}[1/p] \twoheadrightarrow \Lambda_{\mathcal{O}_{K_p}}[1/p]/(\Gamma_K-1) = K_p.$$

\item Similarly, we have a commutative diagram of $\Lambda_{\mathcal{O}_{K_p}}[1/p]$-modules induced by the global Artin reciprocity maps 
$$\mathrm{rec}_E: \mathbb{U}^1 \rightarrow \mathrm{Gal}(K^{\mathrm{ab}}/K(E[p^{\infty}])) \hspace{.25cm} \text{and} \hspace{.25cm} \mathrm{rec}_{A_0'} : \mathbb{U}_{A_0'}^1 \rightarrow \mathrm{Gal}(K^{\mathrm{ab}}/K(\frak{f}_0')(A_0'[\frak{f}_0'p^{\infty}])):$$
\begin{equation}\label{reciprocitydiagram2}
\begin{tikzcd}[column sep =large]
     & \mathbb{U}_{\chi_E}^1(\lambda_E^{-1}) \hspace{.5cm} \arrow{r}{\mathrm{rec}_E} \arrow[hook]{d}{i_{\frak{f}_0',\chi_E}}& \hspace{.5cm}\mathrm{Gal}(K^{\mathrm{ab}}/K(E[p^{\infty}]))_{\chi_E}(\lambda_E^{-1})\arrow{d}{\text{transfer}}\\
     & \mathbb{U}_{A_0',\chi_{A_0'}}^1(\lambda_{A_0'}^{-1})_0 \hspace{.5cm}\arrow{r}{\mathrm{rec}_{A_0'}} \arrow{d}{\mathrm{Nm}_{\frak{f}_0',\chi_E}}& \hspace{.5cm}\mathrm{Gal}(K^{\mathrm{ab}}/K(\frak{f}_0')(A_0'[\frak{f}_0'p^{\infty}]))_{\chi_{A_0'}}(\lambda_{A_0'}^{-1})_0 \arrow[hook]{d} \\
      & \mathbb{U}_{\chi_E}^1(\lambda_E^{-1}) \hspace{.5cm}\arrow{r}{\mathrm{rec}_E}  & \hspace{.5cm} \mathrm{Gal}(K^{\mathrm{ab}}/K(E[p^{\infty}]))_{\chi_E}(\lambda_E^{-1})
     \end{tikzcd}.
\end{equation}

\item Moreover, each arrow in the sequence of morphisms given by the sequence of left vertical arrows in both diagrams
\begin{equation}\label{norminclusionsequence}\mathbb{U}_{\chi_E}^1(\lambda_E^{-1}) \xrightarrow{i_{\frak{f}_0',\chi_E}} \mathbb{U}_{A_0',\chi_{A_0'}}^1(\lambda_{A_0'}^{-1}) \xrightarrow{\mathrm{Nm}} \mathbb{U}_{\chi_E}^1(\lambda_E^{-1})
\end{equation}
is an isomorphism of $\Lambda_{\mathcal{O}_{K_p}}[1/p]$-modules. Since $i_{\frak{f}_0',\chi_E}$ is induced by the natural inclusion 
$$i_{\frak{f}_0'} : \mathbb{U}^1 \hookrightarrow \mathbb{U}^1(\frak{f}_0') \overset{(\ref{f'equal})}{=} \mathbb{U}_{A_0'}^1$$ from (\ref{UEinclusion}), then we get a natural identification of $\Lambda_{\mathcal{O}_{K_p}}[1/p]$-modules
\begin{equation}\label{norminclusionidentification}\mathbb{U}_{\chi_E}^1(\lambda_E^{-1}) = \mathbb{U}_{A_0',\chi_{A_0'}}^1(\lambda_{A_0'}^{-1}).
\end{equation}
\end{enumerate}
\end{proposition}

\begin{proof}First we show (1) and (2). Observe that the composition of the left vertical arrows in both diagrams is multiplication by a constant in $K_p^{\times}$; this implies that the top left vertical arrows in (\ref{reciprocitydiagram}) and (\ref{reciprocitydiagram2}) are injective. Now the commutativity of both diagrams follows immediately from the restriction/corestriction functorialities of the local and Artin reciprocity maps and (\ref{sameCMcharacter2}).

Now we show (3). Recall that by Proposition \ref{ranksproposition}, $\mathbb{U}_{\chi_E}^1(\lambda_E^{-1}) \cong \Lambda_{\mathcal{O}_{K_p}}[1/p]^{\oplus 2}$. By \cite[Corollary 7.8, Lemma 11.8]{RubinMC} we also have $\mathbb{U}_{\chi_{A_0'}}^1(\lambda_{A_0'}^{-1}) \cong \Lambda_{\mathcal{O}_{K_p}}[1/p]^{\oplus 2}$. In particular, each arrow of (\ref{norminclusionsequence}) has a well-defined $\Lambda_{\mathcal{O}_{K_p}}[1/p]$-module determinant after fixing $\Lambda_{\mathcal{O}_{K_p}}[1/p]$-bases of the sources and targets. Since the composition of the arrows in (\ref{norminclusionsequence}) is multiplication by an element in $K_p^{\times}$, the product of the determinants of each arrow of (\ref{norminclusionsequence}) is an element in $K_p^{\times}$. This implies that the determinant of each arrow is a unit in $\Lambda_{\mathcal{O}_{K_p}}[1/p]^{\times}$, and thus each arrow is an isomorphism of $\Lambda_{\mathcal{O}_{K_p}}[1/p]$-modules. Now (\ref{norminclusionidentification}) immediately follows. 

\end{proof}

We will need the following Lemma in the next section. First, recall that $E/\mathbb{Q}$ is an elliptic curve with CM by $\mathcal{O}_K$, with associated infinity type $(1,0)$ Hecke character $\lambda_E$ of conductor $\frak{f} = \frak{f}_0\frak{p}^e$ (\ref{edefinition}) where $(\frak{f}_0,\frak{p}) = 1$, and suppose $\frak{f}_0' \subset \mathcal{O}_K$ is an ideal with $(\frak{f}_0',\frak{p}) = 1$, $\frak{f}_0|\frak{f}_0'$ and $w_{\frak{f}_0'} = 1$. Recall our notation $\mathcal{K}_n = K(E[\frak{p}^n])$. Recall that if $n \ge e$ then $K(\frak{f}_0'\frak{p}^n) = K(E[\frak{f}_0'\frak{p}^n])$ by \cite[Proposition II.1.6]{deShalit}. Thus when $n \ge e$ we get an induced restriction map
$$\mathrm{Gal}(K(\frak{f}_0'\frak{p}^n)/K) = \mathrm{Gal}(K(E[\frak{f}_0'\frak{p}^n])/K) \twoheadrightarrow \mathrm{Gal}(\mathcal{K}_n/K).$$

From the Artin reciprocity map and the fact that $w_{\frak{f}_0'} = 1$, we have an isomorphism for any $n \in \mathbb{Z}_{\ge 0} \cup \{\infty\}$
\begin{equation}\label{rayisomorphism}(\mathcal{O}_{K_p}/\frak{p}^n\mathcal{O}_{K_p})^{\times} \times \left((\mathcal{O}_K/\frak{f}_0')^{\times}/\mathcal{O}_K^{\times}\right) \cong \mathrm{Gal}(K(\frak{f}_0'\frak{p}^n)/K)
\end{equation}
such that the second factor $\left((\mathcal{O}_K/\frak{f}_0')^{\times}/\mathcal{O}_K^{\times}\right)$ maps isomorphically onto $\mathrm{Gal}(K(\frak{f}_0')/K)$ under the restriction map $\mathrm{Gal}(K(\frak{f}_0'\frak{p}^n)/K) \twoheadrightarrow \mathrm{Gal}(K(\frak{f}_0')/K)$. Thus we get induced splittings
\begin{equation}\label{rayclasssplitting}\mathrm{Gal}(K(\frak{f}_0'\frak{p}^n)/K(\frak{f}_0')) \times \mathrm{Gal}(K(\frak{f}_0')/K) \cong \mathrm{Gal}(K(\frak{f}_0'\frak{p}^n)/K).
\end{equation}
Now consider the subgroup $\mathrm{Gal}(K(\frak{f}_0')_p/K_p) \subset \mathrm{Gal}(K(\frak{f}_0')/K)$ (where the inclusion is induced by (\ref{fixembeddings})). Since the image of this inclusion is the decomposition group of the prime of $\mathcal{O}_{K(\frak{f}_0')}$ above $\frak{p}$ fixed by (\ref{fixembeddings}), then we have that 
$$\mathrm{Gal}(K(\frak{f}_0')_p/K_p) \subset \mathrm{Gal}(K(\frak{f}_0')/K) \subset \mathrm{Gal}(K(\frak{f}_0'\frak{p}^n)/K(\frak{f}_0')) \times \mathrm{Gal}(K(\frak{f}_0')/K) \overset{(\ref{rayclasssplitting})}{\cong} \mathrm{Gal}(K(\frak{f}_0'\frak{p}^n)/K)$$
factors through
\begin{equation}\label{rayfactors}\mathrm{Gal}(K(\frak{f}_0')_p/K_p) \subset \mathrm{Gal}(K(\frak{f}_0'\frak{p}^n)_p/K_p).
\end{equation}

\begin{lemma}\label{intermediatelemma'}Suppose $n \ge e$. Then the image of the composition
\begin{equation}\label{raycomposition}\begin{split}\mathrm{Gal}(K(\frak{f}_0')_p/K_p) &\subset \mathrm{Gal}(K(\frak{f}_0')/K) \subset \mathrm{Gal}(K(\frak{f}_0'\frak{p}^n)/K(\frak{f}_0')) \times \mathrm{Gal}(K(\frak{f}_0')/K) \\
&\overset{(\ref{rayclasssplitting})}{\cong} \mathrm{Gal}(K(\frak{f}_0'\frak{p}^n)) = \mathrm{Gal}(K(E[\frak{f}_0'\frak{p}^n])/K) \twoheadrightarrow \mathrm{Gal}(\mathcal{K}_n/K)
\end{split}
\end{equation}
is trivial.

%the map
%\begin{equation}\label{restrictionsurjective}\mathrm{Gal}(K(\frak{f}_0'\frak{p}^n)/K(\frak{f}_0')) = \mathrm{Gal}(K(E[\frak{f}_0'\frak{p}^n])/K(\frak{f}_0')) \rightarrow \mathrm{Gal}(\mathcal{K}_n/K)
%\end{equation}
%induced by restriction from $K(\frak{f}_0'\frak{p}^n)$ to $\mathcal{K}_n$ is surjective.
\end{lemma}

\begin{proof}By (\ref{rayfactors}), the map (\ref{raycomposition}) factors as 
$$\mathrm{Gal}(K(\frak{f}_0')_p/K_p) \rightarrow \mathrm{Gal}(\mathcal{K}_{n,p}/K_p) \subset \mathrm{Gal}(\mathcal{K}_n/K).$$ Now we claim that $\mathcal{K}_{n,p}/K_p$ is totally ramified; admitting this claim, we see that since $K(\frak{f}_0)_p/K_p$ is totally unramified, then the image of the first arrow of the previous displayed map must be trivial, which shows (\ref{raycomposition}) has trivial image. To prove the claim, view $E$ as defined over $\mathcal{O}_{K_p}$ and observe that since $E[\frak{p}^n] = \hat{E}[\frak{p}^n]$, where $\hat{E}$ is the formal group of $E$, and the formal group is the kernel of the reduction map $E \rightarrow E \pmod{\frak{p}\mathcal{O}_{K_p}}$, then the residue field of $\mathcal{K}_{n,p}$ is $\mathcal{O}_{K_p}/\frak{p}\mathcal{O}_{K_p} = \mathbb{F}_p$.

\end{proof}

\subsection{The value on elliptic units}

%Let
%$$\delta_E(\lambda_E^{-1}) : \mathbb{U}_{\chi_E}^1(\lambda_E^{-1}) \rightarrow K_p$$
%denote the twist of $\delta_E$ by $\lambda_E^{-1}$. 
Recall our map $\delta$ from (\ref{delta}), which induces a map
$$\delta|_{q_{\mathrm{dR}} = 1} : \mathbb{U}_{\chi_E}^1(\lambda_E^{-1}) \rightarrow \mathbb{C}_p,$$
where as before ``$|_{q_{\mathrm{dR}} = 1}$'' denotes plugging in $q_{\mathrm{dR}} = 1$. 

\begin{proposition}[cf. Theorem 11.17 of \cite{RubinMC}]\label{babyinterpolationproposition}Let $\xi_E \in \overline{\mathcal{C}}_{\chi_E}^1$ denote the generator from (\ref{xiE}), Proposition \ref{xiEgeneratorproposition}, and let $\Omega_{\infty}(E)$ denote the N\'{e}ron period of $E/\mathbb{Q}$. We have 
\begin{equation}\label{deltascoincide}\delta_E = C_0\cdot \delta|_{q_{\mathrm{dR}} = 1}
\end{equation}
for some $C_0 \in \mathbb{C}_p^{\times}$,
\begin{equation}\label{kernelscoincide}\ker(\delta_E) = \ker(\delta|_{q_{\mathrm{dR}} = 1}),
\end{equation}
and
\begin{equation}\label{babyinterpolation}\delta_E(\xi_E) = \mathcal{L}_{\lambda_E}|_{q_{\mathrm{dR}} = 1} \cdot C_1 = \frac{L(E/\mathbb{Q},1)}{\Omega_{\infty}(E)} \cdot C_2
\end{equation}
for some constants $C_1, C_2 \in \mathbb{C}_p^{\times}$.
% and $C_2 \in K_p^{\times}$. 
\end{proposition}

\begin{remark}\label{babyinterpolationremark}In fact, the equality 
$$\delta_E(\xi_E)  = \frac{L(E/\mathbb{Q},1)}{\Omega_{\infty}(E)} \cdot C_2$$
can be proven without the assumption $f_0 \ge 4$ from Assumption \ref{lambdalambdaEassumption}, see \cite[Theorem 11.17]{RubinMC}. Note that one does not need the assumption $E(K_p)[p] = 0$ as in loc. cit., since \cite[Theorem 5.1]{RubinSha} gives a general formula for computing $C_2$. 

\end{remark}

\begin{proof}[Proof of Proposition \ref{babyinterpolationproposition}]It is clear that (\ref{kernelscoincide}) follows immediately from (\ref{deltascoincide}). Recall $i_{\frak{f}_0}$ from (\ref{UEinclusion}) and $i_{\frak{f}_0,\chi_E}$ from (\ref{1ftwist}). By (\ref{reciprocitydiagram}) for $A_0' = A_0$ where $A_0$ is as in Choice \ref{FixCMdefinition}, we have that 
$$\delta_E = \alpha_1 \cdot \delta_{A_0} \circ i_{\frak{f}_0,\chi_E}$$
for some $\alpha_1 \in K_p^{\times}$. Therefore to prove (\ref{deltascoincide}) it suffices to show
$$\delta_{A_0} \circ i_{\frak{f}_0,\chi_E} = C_0' \cdot \delta_{q_{\mathrm{dR}} = 1}$$
for some $C_0' \in \mathbb{C}_p^{\times}$, and to prove
(\ref{babyinterpolation}) it suffices to show 
$$\delta_{A_0}(i_{\frak{f}_0,\chi_E}(\xi_E)) = \mathcal{L}_{\lambda_E} \cdot C_1' = L(E/\mathbb{Q},1) \cdot C_2'$$
for some $C_1', C_2' \in \mathbb{C}_p^{\times}$. Recall $w_F$ from (\ref{wF}) is the normalized invariant differential on $F = \hat{A}_0$.  

Recall the Frobenius 
\begin{equation}\label{phiinclusion}\phi \in \mathrm{Gal}(K(\frak{f}_0)_p/K_p) \subset \mathrm{Gal}(K(\frak{f}_0)/K) \overset{(\ref{rayclasssplitting})}{\hookrightarrow} \mathrm{Gal}(K(\frak{f}_0\frak{p}^{\infty})/K)
\end{equation}
(taking $n = \infty$ and $\frak{f}_0' = \frak{f}_0$ in (\ref{rayclasssplitting})). We will use (\ref{phiinclusion}) to view $\phi \in \mathrm{Gal}(K(\frak{f}_0\frak{p}^{\infty})/K)$. %For any $n \ge e = \mathrm{ord}_{\frak{p}}(\frak{f}(\lambda_E))$ we have a decomposition induced by the local-global compatibility of Artin reciprocity and the assumptions that $K$ has class number 1 and $w_{\frak{f}_0} = 1$
%$$\mathrm{Gal}(K(\frak{f}_0\frak{p}^n)/K(\frak{f}_0)) \times \mathrm{Gal}(K(\frak{f}_0)/K) \cong \mathrm{Gal}(K(\frak{f}_0\frak{p}^n)/K).$$
By Lemma \ref{intermediatelemma'} with $\frak{f}_0' = \frak{f}_0$, we have that $\phi \overset{(\ref{phiinclusion})}{\in} \mathrm{Gal}(K(\frak{f}_0\frak{p}^{\infty})/K)$ maps to the trivial element under $\mathrm{Gal}(K(\frak{f}_0\frak{p}^{\infty})/K) \twoheadrightarrow \mathrm{Gal}(\mathcal{K}_n/K)$. Hence for any
$$\beta \in \mathbb{U}^1 \overset{i_{\frak{f}_0}}{\hookrightarrow} \mathbb{U}_{A_0}^1$$
the element $\phi \overset{(\ref{phiinclusion})}{\in} \mathrm{Gal}(K(\frak{f}_0\frak{p}^{\infty})/K)$ satisfies
$$i_{\frak{f}_0}(\beta)^{\phi} = i_{\frak{f}_0}(\beta).$$
Thus the Coleman power series $g_{i_{\frak{f}_0}(\beta)} \in \mathcal{O}_{K(\frak{f}_0)_p}\llbracket X\rrbracket$ satisfies 
\begin{equation}\label{Colemanequality}g_{i_{\frak{f}_0}(\beta)}^{\phi}(X) = g_{i_{\frak{f}_0}(\beta)^{\phi}}([\kappa^{-1}(\phi)]_{f,\phi(f)}(X)) = g_{i_{\frak{f}_0}(\beta)}([\kappa^{-1}(\phi)]_{f,\phi(f)}(X))
\end{equation}
where 
$$f : \hat{A}_0 \rightarrow \hat{A}_0/\hat{A}_0[\varpi] = \hat{A}_0^{\phi}$$
is the natural quotient viewed as an morphism of formal groups, the 1-cocycle $\kappa$ and its image under $[\cdot]_{f,\phi(f)}$
$$\kappa : \mathrm{Gal}(K(\frak{f}_0\frak{p}^{\infty})_p/K_p) \rightarrow \mathcal{O}_{K(\frak{f}_0)_p}^{\times}, \hspace{1cm} [\kappa]_{f,\phi(f)} : \hat{A}_0 \rightarrow \hat{A}_0^{\phi}$$
are as in \cite[1.3.7 (14)]{deShalit}, and the first equality of (\ref{Colemanequality}) follows from I.3.7 (15) of op. cit. and the fact that 
$$\phi \overset{(\ref{phiinclusion})}{\in}\mathrm{Gal}(K(\frak{f}_0\frak{p}^{\infty})/K)$$
has trivial projection onto the factor $\mathrm{Gal}(K(\frak{f}_0\frak{p}^n)/K(\frak{f}_0))$ of (\ref{rayclasssplitting}). Thus for any
$$\beta \in \mathbb{U}_{\chi_E}^1(\lambda_E^{-1}) \overset{i_{\frak{f}_0,\chi_E}}{\hookrightarrow} \mathbb{U}_{A_0,\mathbf{1}_{A_0}}((\lambda_{A_0}/\chi_{A_0})^{-1}) \rightarrow \mathbb{U}_{A_0,\chi_{A_0}}^1(\lambda_{A_0}^{-1}),$$
%one sees by direct calculation (cf. \cite[Proposition I.2.3(iii)]{deShalit}) that the constant term of the power series $\frac{d}{w_F}g_{i_{\frak{f}_0,\chi_E}(\beta)}(X)$ is invariant under the action of $\mathrm{Gal}(K(\frak{f}_0)/K)$. Hence in the notation of (\ref{tildelogmap2}), and viewing $\phi \in \mathrm{Gal}(K(\frak{f}_0)_p/K_p) \subset \mathrm{Gal}(K(\frak{f}_0)/K)$, 
we have 
\begin{equation}\label{nophieffect}\frac{d}{w_F}\log g_{i_{\frak{f}_0,\chi_E}(\beta)}^{\phi}(0) \overset{(\ref{Colemanequality})}{=} \frac{d}{w_F}\log g_{i_{\frak{f}_0,\chi_E}(\beta)}([\kappa^{-1}(\phi)]_{f,\phi(f)}(0)) = \kappa^{-1}(\phi)\cdot\frac{d}{w_F}\log g_{i_{\frak{f}_0,\chi_E}(\beta)}(0).
\end{equation}
Thus by \cite[Lemma I.3.3]{deShalit}, we have
\begin{equation}\label{nophieffect2}\begin{split}\frac{d}{w_F}\widetilde{\log} g_{i_{\frak{f}_0,\chi_E}(\beta)}(0) &= \frac{d}{w_F}\log g_{i_{\frak{f}_0,\chi_E}(\beta)}(0) - \frac{f'(0)}{p}\cdot\frac{d}{w_F}\log g_{i_{\frak{f}_0,\chi_E}(\beta)}^{\phi}(0) \\
&\overset{(\ref{nophieffect})}{=} \left(1 - \frac{f'(0)\kappa^{-1}(\phi)}{p}\right)\cdot\frac{d}{w_F}\log g_{i_{\frak{f}_0,\chi_E}(\beta)}(0).
\end{split}
\end{equation}

By Wiles's explicit reciprocity law (\cite[Chapter I.4]{deShalit}) applied to the formal group $F = \hat{A}_0$ (which has good reduction over $\mathcal{O}_{K(\frak{f}_0)_p}$), we have
$$\delta_{A_0}(\beta) = C_0' \cdot \frac{d}{w_F}\log g_{\beta}$$
for some constant $C_0' \in \mathbb{C}_p^{\times}$ independent of $\beta$. 
By (\ref{delta}) and (\ref{nophieffect2}), we thus have 
$$\delta_{A_0} \circ i_{\frak{f}_0,\chi_E} = C_0''\cdot \delta|_{q_{\mathrm{dR}} = 1}$$
for some constant $C_0'' \in \mathbb{C}_p^{\times}$, which gives (\ref{deltascoincide}). Moreover, one computes
\begin{align*}\begin{split}\delta_{A_0}(i_{\frak{f}_0,\chi_E}(\xi_E)) = C_3'\frac{d}{w_F}\log g_{i_{\frak{f}_0}(\xi_E)}(0) = C_4'\frac{d}{w_F}\widetilde{\log} g_{i_{\frak{f}_0}(\xi_E)}(0) \overset{(\ref{delta})}{=} C_4''\delta|_{q_{\mathrm{dR}} = 1}\delta(\xi_E) \overset{(\ref{explicitreciprocity})}{=} C_4''\mathcal{L}_{\lambda_E}|_{q_{\mathrm{dR}} = 1}
\end{split}
\end{align*}
for some constants $C_3', C_4',C_4'' \in \mathbb{C}_p^{\times}$. Here in the last line, we apply (\ref{explicitreciprocity}) with $j = 0$, noting that 
$$\lambda_E^0(\frak{a}) - \mathbb{N}\frak{a} = 1- \mathbb{N}\frak{a} \neq 0$$
since $\frak{a} \subset \mathcal{O}_K$ is a proper ideal. This gives the first equality in (\ref{babyinterpolation}).

The second equality in (\ref{babyinterpolation}) now follows from (\ref{interp1twist}).%, except we only know $C_2 \in \mathbb{C}_p^{\times}$ a priori. If the left-most term $\delta_E(\xi_E)$ of (\ref{babyinterpolation}) is zero then $C_2 \in \mathbb{C}_p^{\times}$ implies $L(E/\mathbb{Q})/\Omega_{\infty}(E) = 0$, which means we may replace $C_2$ by arbitrary element of $K_p^{\times}$ so that (\ref{babyinterpolation}). Now assume that recall that $\delta_E(\xi_E) \neq 0$. By (\ref{deltaEdefinition}) we thus have $\delta_E(\xi_E) \in K_p^{\times}$. Moreover, $L(E/\mathbb{Q},1)/\Omega_{\infty}(E) = L(\lambda_E^{-1},0)/\Omega_{\infty}(E) \in K_p$ (see \cite[Discussion before Theorem 5.1]{RubinSha}). Thus $C_2 \in \mathbb{C}_p^{\times}$ implies $L(E/\mathbb{Q},1)/\Omega_{\infty}(E) \in K_p^{\times}$. Hence have $C_2 = \delta_E(\xi_E)/(L(E/\mathbb{Q},1)/\Omega_{\infty}(E)) \in K_p^{\times}$. Thus $C_2 \in K_p^{\times}$ in all cases, finishing the proof.  
                                                                                                                                                                                                                                                                                                                               
\end{proof}

\subsection{Description of the Selmer group via Wiles's explicit reciprocity law}Continue to retain the choice of Choice \ref{generatorchoice}. We now choose an appropriate decomposition of $\mathbb{U}_{\chi_E}^1(\lambda_E^{-1})$, as in Section \ref{rank0preview}. 

\begin{proposition}[cf. Lemma 11.9 of \cite{RubinMC}]\label{splitproposition}Suppose we are in the setting of Section \ref{Kummertwistsection}, and $A_0'$ is as in Choice \ref{A0'choice}. There exist decompositions
$$\mathbb{U}_{\chi_E}^1(\lambda_E^{-1}) = U_1 \oplus U_2, \hspace{1cm} \mathbb{U}_{A_0',\chi_{A_0'}}^1(\lambda_{A_0'}^{-1})_0 = U_{1,A_0'} \oplus U_{2,A_0'}$$
where $U_i \cong \Lambda_{\mathcal{O}_{K_p}}[1/p]$ and $U_{i,A_0'} \cong \Lambda_{\mathcal{O}_{K_p}}[1/p]$ for $i = 1,2$, and such that 
\begin{equation}\label{kernelequalities}U_2 = \mathrm{ker}(\delta_E), \hspace{1cm} U_{2,A_0'} = \mathrm{ker}(\delta_{A_0'}).
\end{equation}
\end{proposition}

\begin{proof}As recalled in the proof of Proposition \ref{reciprocitydiagramproposition}, both $\mathbb{U}_{\chi_E}^1(\lambda_E^{-1}) \cong \Lambda_{\mathcal{O}_{K_p}}[1/p]^{\oplus 2}$ and $\mathbb{U}_{\chi_{A_0'}}^1(\lambda_{A_0'}^{-1}) \cong \Lambda_{\mathcal{O}_{K_p}}[1/p]^{\oplus 2}$. Now choose an arbitrary $\Lambda_{\mathcal{O}_{K_p}}[1/p]$-basis $\beta_1, \beta_2$ of $\mathbb{U}_{\chi_E}^1(\lambda_E^{-1})$ and an arbitrary $\Lambda_{\mathcal{O}_{K_p}}[1/p]$-basis $\beta_{1,A_0'}, \beta_{2,A_0'}$ of $\mathbb{U}_{A_0',\chi_{A_0'}}^1(\lambda_{A_0'}^{-1})$. Note that the morphisms of $\Lambda_{\mathcal{O}_{K_p}}[1/p]$-modules
$$\delta_E : \mathbb{U}_{\chi_E}^1(\lambda_E^{-1}) \rightarrow K_p, \hspace{1cm} \delta_{A_0'} : \mathbb{U}_{A_0',\chi_{A_0'}}^1(\lambda_{A_0'}^{-1}) \rightarrow K_p$$
are surjective since $1/p \in \Lambda_{\mathcal{O}_{K_p}}[1/p]$. Hence 
$$\delta_E(\beta_1) = \alpha\cdot \delta_E(\beta_2), \hspace{1cm} \delta_{A_0'}(\beta_{1,A_0'}) = \alpha_{A_0'}\cdot\delta_{A_0'}(\beta_{2,A_0'})$$
for some $\alpha, \alpha_{A_0'} \in K_p$. Now let $U_1$ be the $\Lambda_{\mathcal{O}_{K_p}}[1/p]$-module generated by $\beta_1$ and $U_2$ be the $\Lambda_{\mathcal{O}_{K_p}}[1/p]$-module generated by $\beta_1 - \alpha \beta_2$, and let $U_{1,A_0'}$ be the $\Lambda_{\mathcal{O}_{K_p}}[1/p]$-module generated by $\beta_{1,A_0'}$ and $U_{2,A_0'}$ be the $\Lambda_{\mathcal{O}_{K_p}}[1/p]$-module generate by $\beta_{1,A_0'} - \alpha_{A_0'}\beta_{2,A_0'}$. Then $\mathbb{U}_{\chi_E}^1(\lambda_E^{-1}) = U_1 \oplus U_2$ and $\mathbb{U}_{A_0',\chi_{A_0'}}^1(\lambda_{A_0'}^{-1})_0 = U_{1,A_0'} \oplus U_{2,A_0'}$ and these decompositions satisfy (\ref{kernelequalities}). 

\end{proof}

\begin{corollary}The identification $\mathbb{U}_{\chi_E}^1(\lambda_E^{-1}) = \mathbb{U}_{A_0',\chi_{A_0'}}^1(\lambda_{A_0'}^{-1})$ from (\ref{norminclusionidentification}) induces identifications of $\Lambda_{\mathcal{O}_{K_p}}[1/p]$-modules
\begin{equation}\label{kernelidentity}\ker(\delta_E) = \ker(\delta_{A_0'}), \hspace{1cm} \ker(\mathrm{rec}_E) =\ker(\mathrm{rec}_{A_0'}).
\end{equation}
\end{corollary}

\begin{proof}By the commutativity of (\ref{reciprocitydiagram}), the sequence of maps (\ref{norminclusionsequence}) induces a sequence of maps of $\Lambda_{\mathcal{O}_{K_p}}[1/p]$-modules
$$\ker(\delta_E) \xrightarrow{i_{\frak{f}_0',\chi_E}} \ker(\delta_{A_0'}) \xrightarrow{\mathrm{Nm}_{\frak{f}_0',\chi_E}} \ker(\delta_E)$$
whose composition is multiplication by an element in $K_p^{\times}$. Since all the $\Lambda_{\mathcal{O}_{K_p}}[1/p]$-modules appearing in this sequence of maps are free of rank 1, the determinants of all these maps of $\Lambda_{\mathcal{O}_{K_p}}[1/p]$ is well-defined after fixing $\Lambda_{\mathcal{O}_{K_p}}[1/p]$-bases of the sources and targets. The product of the determinants of these maps is a constant in $K_p^{\times}$, and hence each determinant must lie in $\Lambda_{\mathcal{O}_{K_p}}[1/p]^{\times}$ which implies all the maps are isomorphisms. Under (\ref{norminclusionidentification}), this gives the first equality of (\ref{kernelidentity}). 

Similarly, by the commutativity of (\ref{reciprocitydiagram}), we get a sequence of maps of $\Lambda_{\mathcal{O}_{K_p}}[1/p]$-modules
\begin{equation}\label{kernelintermediate}\ker(\mathrm{rec}_E) \xrightarrow{i_{\frak{f}_0',\chi_E}} \ker(\mathrm{rec}_{A_0'}) \xrightarrow{\mathrm{Nm}_{\frak{f}_0',\chi_E}} \ker(\mathrm{rec}_E)
\end{equation}
whose composition is multiplication by an element $\alpha' \in K_p^{\times}$; in particular, the first arrow $i_{\frak{f}_0',\chi_E}$ is injective. We will show that $i_{\frak{f}_0,',\chi_E}$ is also surjective. 

Suppose $\beta \in \ker(\mathrm{rec}_{A_0'})$ and assume for sake of contradiction that 
$$\beta \not\in i_{\frak{f}_0',\chi_E}(\ker(\mathrm{rec}_E)).$$
Observe $\ker(\mathrm{rec}_E) = \overline{\mathcal{E}}_{\chi_E}^1(\lambda_E^{-1})$ is a torsion-free $\Lambda_{\mathcal{O}_{K_p}}[1/p]$-module of rank 1 by Proposition \ref{ranksproposition}. Moreover, \cite[Corollary 7.8, Lemma 11.8]{RubinMC} also implies that $\ker(\mathrm{rec}_{A_0'})= \overline{\mathcal{E}}_{A_0',\chi_{A_0'}}^1(\lambda_{A_0'}^{-1})$ is a torsion-free $\Lambda_{\mathcal{O}_{K_p}}[1/p]$-module of rank 1. Thus $\ker(\mathrm{rec}_{A_0'})/i_{\frak{f}_0',\chi_E}(\ker(\mathrm{rec}_E))$ is torsion. Hence there exists a non-zero divisor $\lambda \in \Lambda_{\mathcal{O}_{K_p}}[1/p]$ (which just means $\lambda \neq 0$ since $\Lambda_{\mathcal{O}_{K_p}} \cong \mathcal{O}_{K_p}\llbracket X\rrbracket [1/p]$ is an integral domain) such that 
$$\lambda\cdot\beta \in i_{\frak{f}_0',\chi_E}(\ker(\mathrm{rec}_E)).$$
Applying $\mathrm{Nm}_{\frak{f}_0',\chi_E}$ from (\ref{kernelintermediate}), we thus get an equality of elements of $\ker(\mathrm{rec}_E)$
$$\lambda \cdot \mathrm{Nm}_{\frak{f}_0',\chi_E}(\beta) = \mathrm{Nm}_{\frak{f}_0',\chi_E}(\lambda\cdot\beta) = \alpha'\lambda \cdot \beta,$$
which implies 
$$\lambda\cdot\mathrm{Nm}_{\frak{f}_0',\chi_E}(\beta) - \alpha'\lambda\cdot\beta = 0.$$
Applying $i_{\frak{f}_0',\chi_E}$ from (\ref{kernelintermediate}) and observing that $i_{\frak{f}_0',\chi_E}(\lambda \cdot \beta) = \lambda \cdot \beta$, we then get
\begin{align*}\lambda\cdot \left(i_{\frak{f}_0',\chi_E}(\mathrm{Nm}_{\frak{f}_0',\chi_E}(\beta)) - \alpha'\cdot \beta\right) &= \lambda\cdot i_{\frak{f}_0',\chi_E}(\mathrm{Nm}_{\frak{f}_0',\chi_E}(\beta)) - \alpha' \cdot i_{\frak{f}_0',\chi_E}(\lambda\cdot\beta) \\
&= i_{\frak{f}_0',\chi_E}(\lambda\cdot\mathrm{Nm}_{\frak{f}_0',\chi_E}(\beta) - \alpha'\lambda\cdot\beta) = 0
\end{align*}
in $\ker(\mathrm{rec}_{A_0'})$. Since $\ker(\mathrm{rec}_{A_0'})$ is torsion-free and $\lambda \neq 0$, we thus have
$$\alpha'\cdot \beta = i_{\frak{f}_0',\chi_E}(\mathrm{Nm}_{\frak{f}_0',\chi_E}(\beta)) \in i_{\frak{f}_0',\chi_E}(\ker(\mathrm{rec}_E)).$$
Dividing by $\alpha' \in K_p^{\times}$ we thus get 
$$\beta = \alpha'^{-1} \cdot  i_{\frak{f}_0',\chi_E}(\mathrm{Nm}_{\frak{f}_0',\chi_E}(\beta)) \in i_{\frak{f}_0',\chi_E}(\ker(\mathrm{rec}_E)),$$
a contradiction. Thus we have shown that the first arrow of (\ref{kernelintermediate}) is an isomorphism, which under (\ref{norminclusionidentification}) gives the second equality of (\ref{kernelidentity}).

%Thus we get an induced map of $\Lambda_{\mathcal{O}_{K_p}}[1/p]$-modules
%$$\ker(\mathrm{rec}_{A_0'})/i_{\frak{f}_0',\chi_E}(\ker(\mathrm{rec}_E)) \hookrightarrow \ker(\mathrm{rec}_E)$$
%where the source has $\Lambda_{\mathcal{O}_{K_p}}[1/p]$-rank 0, and is thus torsion, and the target is torsion-free. Thus the source must in fact be 0, which implies that the first arrow of (\ref{kernelintermediate}) is surjective. Thus it is an isomorphism, which gives the second equality of (\ref{kernelidentity}).

\end{proof}

\begin{proposition}[cf. Proposition 11.10 of \cite{RubinMC}]Let 
$$\delta W(\lambda_E) := \{\text{maps} \; u \mapsto \delta_E(u)v : v \in K_p/\mathcal{O}_{K_p}\}.$$
%We have
%\begin{equation}\label{keyreciprocityequality}\mathrm{im}(\phi) \supset \delta W(\lambda_E) = \mathrm{Hom}((\mathcal{U}^1)_{\chi_E}/\ker(\delta),W(\lambda_E))^{\Delta}.
%\end{equation}
%%In fact, the second set has cokernel of order $\#\left(1-\frac{\lambda_E(\frak{p})}{\mathbb{N}(\frak{p})}\right)$ in the first set.
In the setting of (\ref{Selmergroupdiagram}), we have 
\begin{equation}\label{keyreciprocityequality2}\mathrm{im}(\varphi) =  \delta W(\lambda_E) = \mathrm{Hom}(\mathbb{U}_{\chi_E}^1(\lambda_E^{-1})/\ker(\delta_E),K_p/\mathcal{O}_{K_p})^{\mathrm{Gal}(\mathcal{K}_{\infty}/K)}.
\end{equation}
\end{proposition}

\begin{proof}From the definitions we have
\begin{equation}\label{preliminaryinclusions}\mathrm{im}(\varphi) =  \delta W(\lambda_E) \subset \mathrm{Hom}(\mathbb{U}_{\chi_E}^1(\lambda_E^{-1})/\ker(\delta_E),K_p/\mathcal{O}_{K_p})^{\mathrm{Gal}(\mathcal{K}_{\infty}/K)}
\end{equation}
As abelian groups, we have
$$\mathrm{im}(\varphi) \cong \left(E(K_p)\otimes_{\mathcal{O}_K}K_p/\mathcal{O}_{K_p}\right)_{\mathrm{div}} \cong K_p/\mathcal{O}_{K_p},$$
$$\delta W(\lambda_E) \cong K_p/\mathcal{O}_{K_p}$$
and
\begin{align*}&\mathrm{Hom}(\mathbb{U}_{\chi_E}^1(\lambda_E^{-1})/\ker(\delta_E),K_p/\mathcal{O}_{K_p})^{\mathrm{Gal}(\mathcal{K}_{\infty}/K)} \cong \mathrm{Hom}(\mathcal{O}_{K_p},K_p/\mathcal{O}_{K_p}) = K_p/\mathcal{O}_{K_p}.
\end{align*}
So all terms of (\ref{preliminaryinclusions}) are isomorphic to $K_p/\mathcal{O}_{K_p}$, and hence all the inclusions are actually equalities.

\end{proof}

A key consequence is the following.

\begin{theorem}\label{Selmertheorem}There is a natural identification
\begin{equation}\label{Selmergroupcharacterization}S(\lambda_E)_{\mathrm{div}} = \mathrm{Hom}(\mathcal{X}_{\chi_E}(\lambda_E^{-1})/\mathrm{rec}(\ker(\delta_E)),K_p/\mathcal{O}_{K_p})^{\mathrm{Gal}(\mathcal{K}_{\infty}/K)}.
\end{equation}
Hence,
\begin{equation}\label{Selmergroupcharacterizationorder}\#S(\lambda_E)_{\mathrm{div}} = \#\mathrm{Hom}(\mathcal{X}_{\chi_E}(\lambda_E^{-1})/\mathrm{rec}(\ker(\delta_E)),K_p/\mathcal{O}_{K_p})^{\mathrm{Gal}(\mathcal{K}_{\infty}/K)}.
\end{equation}
%In fact, 
%$$\#\left(\frac{S_{\infty}(\lambda_E)^{\Delta}}{\mathrm{Hom}((\mathcal{X}')_{\chi_E}W(\lambda_E))^{\Delta}}\right) = \#\left(1-\frac{\lambda_E(\frak{p})}{\mathbb{N}(\frak{p})}\right).$$
\end{theorem}

\begin{proof}
This follows immediately from (\ref{keyreciprocityequality2}) and (\ref{Selmercharacterization}).
\end{proof}

\begin{convention}Given $\alpha, \beta \in \mathbb{Q} \cup \{\infty\}$, we will let 
$$\alpha \sim \beta$$
denote the relation
$$\alpha < \infty \hspace{1cm} \text{if and only if} \hspace{1cm} \beta < \infty.$$
\end{convention}

\begin{proposition}[cf. Theorem 11.16 of \cite{RubinMC}]We have
\begin{equation}\label{Selmerorderequality}\begin{split}&\#(\mathrm{Hom}(\mathcal{X}_{\chi_E}(\lambda_E^{-1})/\mathrm{rec}(\ker(\delta_E)),K_p/\mathcal{O}_{K_p})^{\mathrm{Gal}(\mathcal{K}_{\infty}/K)}) \\
&\hspace{4cm}\sim \#(\mathrm{Hom}(\mathbb{U}_{\chi_E}^1(\lambda_E^{-1})/(\overline{\mathcal{C}}_{\chi_E}^1(\lambda_E^{-1}),\ker(\delta_E)),K_p/\mathcal{O}_{K_p})^{\mathrm{Gal}(\mathcal{K}_{\infty}/K)}).
\end{split}
\end{equation}
\end{proposition}

\begin{proof}This follows from the same arguments as in \cite[proof of Theorem 11.16]{RubinMC}, using the Rubin-type Main Conjecture (Theorem \ref{RMC}), Lemma 6.2 of op. cit. (tensored with $\otimes_{\mathbb{Z}_p}\mathbb{Q}_p$) and the fact that $\mathcal{X}$ has no nonzero pseudonull submodules by \cite[Theorem 2]{GreenbergRank} (see also \cite[Theorem 5.3 (v) and Lemma 11.12]{RubinMC}).
\end{proof}

%We have the following consequence of a standard control theorem.

%\begin{proposition}Suppose that $\#S_0(\lambda) < \infty$. Then we have
%\begin{equation}\label{0descent}\begin{split}%%\left(\prod_{v}c_v(E/K)[p^{\infty}]\right)\cdot \left(\#E(K)[p^{\infty}]\right)^{-1}\#S_0(\lambda) &= \#\mathrm{Sel}_{\infty}(\lambda)^{\mathrm{Gal}(L_{\infty}/K)} \\
%c(E)\cdot \#S_0(\lambda) &= \#\mathrm{Sel}_{\infty}(\lambda)^{\mathrm{Gal}(L_{\infty}/K)} \\
%&\ge  \#\mathrm{Hom}((\mathcal{X}')_{\chi_E},W(\lambda))^{\mathrm{Gal}(L_{\infty}/K)}
%\end{split}
%\end{equation}
%where $c(E) \in \mathbb{Q}$ is some nonzero constant depending only on $E$. 
%%where the product runs over all places of $K$ and $c_v(E/K) := [E(K_v) : E^0(K_v)]$ denotes the Tamagawa number of $E/K$, and $c_v(E/K)[p^{\infty}]$ denotes its $p$-primary part. 
%%In fact, we have $c_v(E/K) = 4$ for all places $v$ where $E/K$ has bad reduction.
%\end{proposition}

%\begin{proof}The equality follows from the argument of \cite[Proposition 3.5, Equation (8)]{Aviles}, which itself is a standard calculation using \cite[Theorem A]{Bashmakov} to describe the cokernel of the localization map. The inequality follows immediately from (\ref{Selmergroupcharacterization}).
%%The calculation of Tamagawa numbers is also well-known, and details can be found in \cite[Section 2]{Aviles}. 
%\end{proof}

%\begin{remark}It is possible to make the constant $c(E)$ in (\ref{0descent}) explicit in terms of Tamagawa numbers attached to $E$. 
%\end{remark}

\subsection{The rank 0 $p$-converse theorem for CM elliptic curves with $p$ ramified in the CM field}

Theorem \ref{Selmertheorem} gives us the rank 0 $p$-converse theorem. In fact, we prove a slightly stronger statement, showing that both (\ref{GZKimplication}) and (\ref{pconverseimplication}) hold for $r = 0$ (in particular recovering part of the main result of \cite{CoatesWiles}).

%In the setting of that theorem, let $\chi = \chi_E$, where $\chi_E : \Delta = \mathrm{Gal}(L'/K) \hookrightarrow \mathrm{Aut}(E[\frak{p}^{\varepsilon}]) = (\mathcal{O}_{K_p}/\frak{p}^{\varepsilon})^{\times}$ is given by the theory of complex multiplication, and let $\lambda = \lambda_E$, where $\lambda_E : \Gamma = \mathrm{Gal}(L_{\infty}/L') \xrightarrow{\sim} 1 + \frak{p}^{\varepsilon}\mathcal{O}_{K_p}$ is also given by the theory of complex multiplication. Then $\kappa_E = \lambda_E\chi_E : \mathrm{Gal}(L_{\infty}/K) \hookrightarrow \mathcal{O}_{K_p}^{\times}$ is the usual character (with Hodge-Tate weight $(1,0)$) associated with $E$ given by the theory of complex multiplication. 

\begin{theorem}[Rank 0 $p$-Converse Theorem]\label{BSDrank0theorem}Suppose $E/\mathbb{Q}$ is an elliptic curve with CM by $\mathcal{O}_K$ and $p$ is ramified in $K$. Then 
$$r_p(E/\mathbb{Q}) := \mathrm{corank}_{\mathbb{Z}_p}\mathrm{Sel}_{p^{\infty}}(E/\mathbb{Q}) = 0 \iff r_{\mathrm{an}}(E/\mathbb{Q}) := \mathrm{ord}_{s = 1}L(E/\mathbb{Q},s) = 0.$$

%In the setting of Theorem \ref{Selmertheorem}, write 
%$$\mathcal{L} = \frac{L(\overline{\lambda},1)}{\Omega_{\infty}(E)},$$
%where $\Omega_{\infty}(E)$ is the usual N\'{e}ron period associated with $E$, and in particular $\mathcal{L} \in \overline{\mathbb{Q}}$. Then we have
%\begin{equation}\label{BSDprank0}\#S_0(\lambda) \sim \#\left(\mathcal{O}_{K_p}/\mathcal{L}\mathcal{O}_{K_p}\right),
%\left(\prod_{v}c_v(E/K)[p^{\infty}]\right)\cdot \left(\#E(K)[p^{\infty}]\right)^{-1}\#S_0(\lambda) = \#\left(\mathcal{O}_{K_p}/\mathcal{L}\mathcal{O}_{K_p}\right).
%\end{equation}
%In other words, since $E(\mathbb{Q})\otimes_{\mathbb{Z}}\mathbb{Q}_p/\mathbb{Z}_p = \mathrm{Sel}_{p^{\infty}}(E/\mathbb{Q}) = \mathrm{Sel}_0(\lambda)$,
%$$\mathrm{ord}_p(E(\mathbb{Q})) = \mathrm{ord}_p\left(\#\mathcal{L}\right),$$
%i.e. the $p$-part of the BSD formula is true for $E/K$.
%where ``$\sim$'' denotes equality up to a finite power of $p$. In particular, 
%$$L(E/\mathbb{Q},1) = L(\overline{\lambda},1) \neq 0 \iff \#S_0(\lambda) < \infty.$$
\end{theorem}

\begin{proof}Recall that $f_0^2$ (where $f_0 \in \mathbb{Z}_{> 0}$ is as in Assumption \ref{pconductorassumption} and Example \ref{Eexample}) is the prime-to-$p$ part of the conductor of $E/\mathbb{Q}$ (\cite[Proposition 3.13]{BDP2}). When $f_0 \le 3$, there are finitely many isomorphism classes of elliptic curves for which to check the assertion. Namely, for $p = 2$ all elliptic curves with CM by $\mathbb{Z}[i]$ or $\mathbb{Z}[\sqrt{-2}]$ of conductor $\le 3^2\cdot 2^8 = 2304$, $p = 3$ all elliptic curves with CM by $\mathbb{Z}[\zeta_3]$ of conductor $\le 2^2 \cdot 3^5 = 972$, when $p = 7$ all elliptic curves with CM by $\mathcal{O}_{\mathbb{Q}(\sqrt{-7})}$ of conductor $\le 3^2\cdot 7^2 = 441$, when $p = 11$ all elliptic curves with CM by $\mathcal{O}_{\mathbb{Q}(\sqrt{-11})}$ of conductor $\le 3^2\cdot 11^2 = 1089$, when $p = 19$ all elliptic curves with CM by $\mathcal{O}_{\mathbb{Q}(\sqrt{-19})}$ of conductor $\le 3^2\cdot 19^2 = 3249$, when $p = 43$ all elliptic curves with CM by $\mathcal{O}_{\mathbb{Q}(\sqrt{-43})}$ of conductor $\le 3^2\cdot 43^2 = 16641$, $p = 67$ all elliptic curves with CM by $\mathcal{O}_{\mathbb{Q}(\sqrt{-67})}$ of conductor $\le 3^2\cdot 67^2 = 40401$, and $p = 163$ all elliptic curves with CM by $\mathcal{O}_{\mathbb{Q}(\sqrt{-163})}$ of conductor $\le 3^2\cdot 163^2 = 239121$. The assertion in all of these cases has been verified by the LMFDB Collaboration, \cite{LMFDB}, which calculates the Selmer, algebraic and analytic ranks of all elliptic curves over $\mathbb{Q}$ up to conductor 500000.

Now assume $f_0 \ge 4$. We have
\begin{align*}\#S(\lambda_E)_{\mathrm{div}} &\overset{(\ref{Selmergroupcharacterizationorder})}{=} \#(\mathrm{Hom}(\mathcal{X}_{\chi_E}(\lambda_E^{-1})/(\mathrm{rec}(\mathrm{ker}(\delta_E))),K_p/\mathcal{O}_{K_p})^{\mathrm{Gal}(\mathcal{K}_{\infty}/K)}) \\
&\overset{(\ref{Selmerorderequality})}{\sim}\#(\mathrm{Hom}(\mathbb{U}_{\chi_E}^1(\lambda_E^{-1})/(\ker(\delta_E),\overline{\mathcal{C}}_{\chi_E}^1(\lambda_E^{-1})),K_p/\mathcal{O}_{K_p})^{\mathrm{Gal}(\mathcal{K}_{\infty}/K)}).
\end{align*}

Recall $\mathbb{U}_{\chi_E}^1(\lambda_E^{-1}) = U_1 \oplus U_2$ from Proposition \ref{splitproposition}, where $U_i \cong \Lambda_{\mathcal{O}_{K_p}}[1/p]$ for $i = 1,2$ and $U_2 = \mathrm{ker}(\delta_E)$, and let $\mathrm{pr}_i : \mathbb{U}_{\chi_E}^1(\lambda_E^{-1}) \rightarrow U_i$ be the natural projection. Recall also the $\Lambda_{\mathcal{O}_{K_p}}[1/p]$-equivariant map 
$$\delta_E : \mathbb{U}_{\chi_E}^1(\lambda_E^{-1}) \rightarrow K_p$$
from Definition \ref{deltaEdefinition} (where $\Lambda_{\mathcal{O}_{K_p}}[1/p]$ acts on the target $K_p$ through $\Lambda_{\mathcal{O}_{K_p}}[1/p] \twoheadrightarrow \Lambda_{\mathcal{O}_{K_p}}[1/p]/(\Gamma_K-1) = K_p$), which is surjective since $1/p \in \Lambda_{\mathcal{O}_{K_p}}[1/p]$. Consider the composition 
$$U_1\overset{\delta_E}{\twoheadrightarrow} K_p \twoheadrightarrow K_p/\left(C_2 \cdot L(E/\mathbb{Q},1)/\Omega_{\infty}(E)\cdot K_p\right).$$
By (\ref{babyinterpolation}) and Proposition \ref{xiEgeneratorproposition}, this factors through a $\Lambda_{\mathcal{O}_{K_p}}[1/p]$-equivariant map
$$U_1/\mathrm{pr}_1(\overline{\mathcal{C}}_{\chi_E}^1(\lambda_E^{-1})) = U_1/(\Lambda_{\mathcal{O}_{K_p}}[1/p]\cdot \mathrm{pr}_1(\xi_E)) \twoheadrightarrow K_p/\left(C_2 \cdot L(E/\mathbb{Q},1)/\Omega_{\infty}(E)\cdot K_p\right),$$
where $\Lambda_{\mathcal{O}_{K_p}}[1/p]$ acts on the target $K_p$ through $\Lambda_{\mathcal{O}_{K_p}}[1/p] \twoheadrightarrow \Lambda_{\mathcal{O}_{K_p}}[1/p]/(\Gamma_K-1) = K_p$. By the $\Lambda_{\mathcal{O}_{K_p}}[1/p]$-equivariance, the fact that $\Gamma_K$ acts trivially on the target, and $U_1 \cong \Lambda_{\mathcal{O}_{K_p}}[1/p]$, we thus see that the previous map induces an isomorphism of $K_p$-vector spaces
$$\left(U_1/\mathrm{pr}_1(\overline{\mathcal{C}}_{\chi_E}^1(\lambda_E^{-1}))\right)/(\Gamma_K-1) \xrightarrow{\sim} K_p/\left(C_2 \cdot L(E/\mathbb{Q},1)/\Omega_{\infty}(E)\cdot K_p\right)$$
(where the left-hand side uses the notation of Definition \ref{Lambdadefinition}). From this we see that 
\begin{equation}\label{descentLvalue}\left(U_1/\mathrm{pr}_1(\overline{\mathcal{C}}_{\chi_E}^1(\lambda_E^{-1}))\right)/(\Gamma_K-1) \cong \begin{cases} 0 & L(E/\mathbb{Q},1) \neq 0\\
K_p & L(E/\mathbb{Q},1) = 0\\
\end{cases}.
\end{equation}

Now we have
\begin{align*}&\#(\mathrm{Hom}((\mathbb{U}_{\chi_E}^1(\lambda_E^{-1})/(\ker(\delta_E),\overline{\mathcal{C}}_{\chi_E}^1(\lambda_E^{-1})),K_p/\mathcal{O}_{K_p})^{\mathrm{Gal}(\mathcal{K}_{\infty}/K)})) \\
&\overset{(\ref{kernelequalities})}{=} \#(\mathrm{Hom}(U_1/\mathrm{pr}_1(\overline{\mathcal{C}}_{\chi_E}^1(\lambda_E^{-1})),K_p/\mathcal{O}_{K_p})^{\mathrm{Gal}(\mathcal{K}_{\infty}/K)})\\
&= \#\left(\mathrm{Hom}\left(\left(U_1/\mathrm{pr}_1(\overline{\mathcal{C}}_{\chi_E}^1(\lambda_E^{-1}))\right)/(\Gamma_K-1),K_p/\mathcal{O}_{K_p}\right)\right)\overset{(\ref{descentLvalue})}{=} \begin{cases} 1 & L(E/\mathbb{Q},1) \neq 0\\
\infty & L(E/\mathbb{Q},1) = 0\\
\end{cases}.
%&= \#\{f \in \Hom(\mathbb{U}_{\chi_E}^1(\lambda_E^{-1})/\ker(\delta_E),K_p/\mathcal{O}_{K_p})^{\mathrm{Gal}(\mathcal{K}_{\infty}/K)} : f(\overline{\mathcal{C}}_{\chi_E}^1(\lambda_E^{-1})) = 0\}\\
%&\overset{(\ref{kernelequalities})}{=}  \#\{f \in \Hom(U_1,K_p/\mathcal{O}_{K_p})^{\mathrm{Gal}(\mathcal{K}_{\infty}/K)} : f(\mathrm{pr}_1(\overline{\mathcal{C}}_{\chi_E}^1(\lambda_E^{-1}))) = 0\}\\
%&=  \#\{f \in \Hom(\Lambda_{\mathcal{O}_{K_p}}[1/p],K_p/\mathcal{O}_{K_p})^{\mathrm{Gal}(\mathcal{K}_{\infty}/K)} : f(\overline{\mathcal{C}}_{\chi_E}^1(\lambda_E^{-1})) = 0\}\\
%&= \#\{v \in K_p/\mathcal{O}_{K_p} : \delta_E(\overline{\mathcal{C}}_{\chi_E}^1(\lambda_E^{-1}))v = 0\}\\
%&\overset{(\ref{babyinterpolation})}{=} \#\{v \in K_p/\mathcal{O}_{K_p} : C_2\frac{L(E/\mathbb{Q},1)}{\Omega_{\infty}(E)}v = 0\}= \#\left(\mathcal{O}_{K_p}/\left(C_2\frac{L(E/\mathbb{Q},1)}{\Omega_{\infty}(E)}\mathcal{O}_{K_p}\right)\right)
\end{align*}
Hence in all we have 
$$\#S(\lambda_E)_{\mathrm{div}} \sim \begin{cases} 1 & L(E/\mathbb{Q},1) \neq 0\\
\infty & L(E/\mathbb{Q},1) = 0\\
\end{cases}.$$
Now the assertion immediately follows since 
$$S(\lambda_E) = \mathrm{Sel}_{p^{\infty}}(E/\mathbb{Q}) \otimes_{\mathbb{Z}_p}\mathcal{O}_{K_p}$$
by Shapiro's lemma. For this last step, note that since $E/\mathbb{Q}$ has CM by $\mathcal{O}_K$, the $\mathcal{O}_{K_p}[\mathrm{Gal}(\overline{\mathbb{Q}}/\mathbb{Q})]$-module $E[p^{\infty}] \otimes_{\mathbb{Z}_p}\mathcal{O}_{K_p}$ is the induced module from the $\mathcal{O}_{K_p}[\mathrm{Gal}(\overline{K}/K)]$-module $W(\lambda_E)$, i.e. we have an isomorphism of $\mathcal{O}_{K_p}[\mathrm{Gal}(\overline{\mathbb{Q}}/\mathbb{Q})]$-modules
\begin{equation}\label{inducedformula}E[p^{\infty}] \otimes_{\mathbb{Z}_p}\mathcal{O}_{K_p} \cong W(\lambda_E) \otimes_{\mathcal{O}_{K_p}[\mathrm{Gal}(\overline{K}/K)]}\mathcal{O}_{K_p}[\mathrm{Gal}(\overline{\mathbb{Q}}/\mathbb{Q})].
\end{equation}
(Here, $\mathrm{Gal}(\overline{\mathbb{Q}}/\mathbb{Q})$ acts trivially on $\mathcal{O}_{K_p}$ on the left-hand side.)

\end{proof}

%\begin{remark}Note that $S_0(\lambda) = \mathrm{Sel}_{p^{\infty}}(E/\mathbb{Q})$. Hence Theorem \ref{BSDrank0theorem} implies that if $\#\mathrm{Sel}_{\frak{p}^{\infty}}(E/\mathbb{Q}) < \infty$, i.e. $\mathrm{corank}_{\mathbb{Z}_p}(\mathrm{Sel}_{p^{\infty}}(E/\mathbb{Q})) = 0$, then $L(E/\mathbb{Q},1) = L(\overline{\lambda},1) \neq 0$, i.e. $\mathrm{ord}_{s=1}L(E/\mathbb{Q},s) = 0$. 
%\end{remark}

\section{Rank 1 $p$-Converse Theorem}\label{rank1section}

\subsection{The rank 1 $p$-converse theorem for CM elliptic curves with $p$ ramified in the CM field}

Recall $E/\mathbb{Q}$ is an elliptic curve with CM by $\mathcal{O}_K$ as in Choice \ref{Echoice}, where $K$ is an imaginary quadratic fields of class number 1. We continue to assume that $p$ is the unique finite prime ramified in $K/\mathbb{Q}$. We recall the definition of the usual and relaxed Selmer groups for $E/\mathbb{Q}$.

\begin{definition}\label{unramifiedrelaxedSelmerdefinition}
\begin{enumerate}
\item For each place $\ell$ of $\mathbb{Q}$, fix an algebraic closure $\overline{\mathbb{Q}}_{\ell}$ of the $\ell$-adic completion $\mathbb{Q}_{\ell}$ and fix an embedding $\overline{\mathbb{Q}} \subset \overline{\mathbb{Q}}_{\ell}$ (for $\ell = p$, we continue to use (\ref{fixembeddings})). Let $G_{\ell} = \mathrm{Gal}(\overline{\mathbb{Q}}_{\ell}/\mathbb{Q}_{\ell})$ be the corresponding decomposition group.
% and let $I_{\ell}$ be the inertial subgroup of $G_{\ell}$. 
\item Recall the notation of Convention \ref{groupcohomologyconvention}. Define for all $n \in \mathbb{Z}_{\ge 0} \cup \{\infty\}$,
\begin{align*}\mathrm{Sel}_{p^n}(E/\mathbb{Q}) := \ker\left(\prod_{\ell} \mathrm{res}_{\ell}' : H^1(\mathbb{Q},E[p^n]) \rightarrow \prod_{\ell} H^1(G_{\ell},E)\right)
\end{align*}
and
$$\mathrm{Sel}_{p^n}^{\mathrm{rel}}(E/\mathbb{Q}) := \ker\left(\prod_{\ell \neq p} \mathrm{res}_{\ell}' : H^1(\mathbb{Q},E[p^n]) \rightarrow \prod_{\ell \neq p} H^1(G_{\ell},E)\right).
$$
%and
%\begin{align*}\mathrm{Sel}_{p^n}^{\mathrm{str}}(E/\mathbb{Q}) := \ker\left(\prod_{\ell} \mathrm{res}_{\ell} : H^1(\mathbb{Q},E[p^n]) \rightarrow H^1(G_p,E[p^n]) \times \prod_{\ell \neq p} H^1(G_{\ell},E) \right).
%\end{align*}
Here, the maps $\mathrm{res}_{\ell}' : H^1(\mathbb{Q},E[p^n]) \rightarrow H^1(G_{\ell},E)$ are defined to be the compositions
\begin{equation}\label{Kummerdefinecomposition}H^1(\mathbb{Q},E[p^n]) \xrightarrow{\mathrm{res}_{\ell}} H^1(G_{\ell},E[p^n]) \rightarrow H^1(G_{\ell},E),
\end{equation}
where 
$$\mathrm{res}_{\ell} : H^1(\mathbb{Q},E[p^n]) \rightarrow H^1(G_{\ell},E[p^n])$$ is the restriction homomorphism induced by $G_{\ell} \subset \mathrm{Gal}(\overline{\mathbb{Q}}/\mathbb{Q})$ and 
$$H^1(G_{\ell},E[p^n]) \rightarrow H^1(G_{\ell},E)$$
is induced by the long exact sequence associated with the short exact sequence of $G_{\ell}$-modules 
$$0 \rightarrow E[p^n] \rightarrow E \xrightarrow{p^n} E \rightarrow 0.$$
%\item For $\bullet \in \{\hspace{.25cm},\mathrm{rel}\}$, let
%$$\mathrm{Sel}_{p^{\infty}}^{\bullet}(E/\mathbb{Q})^{\vee} := \mathrm{Hom}_{\mathbb{Z}_p}(\mathrm{Sel}_{p^{\infty}}^{\bullet}(E/\mathbb{Q}),\mathbb{Q}_p/\mathbb{Z}_p).$$
\end{enumerate}
\end{definition}

Immediately from Definition \ref{unramifiedrelaxedSelmerdefinition}, we have for any $n \in \mathbb{Z}_{\ge 0} \cup \{\infty\}$
\begin{equation}\label{Selmerinclusions}%\mathrm{Sel}_{p^n}^{\mathrm{str}}(E/\mathbb{Q}) \subset 
\mathrm{Sel}_{p^n}(E/\mathbb{Q}) \subset \mathrm{Sel}_{p^n}^{\mathrm{rel}}(E/\mathbb{Q}).
%, \hspace{1cm} \mathrm{Sel}_{p^{\infty}}^{\mathrm{str}}(E/\mathbb{Q}) \subset \mathrm{Sel}_{p^{\infty}}^{\mathrm{ur}}(E/\mathbb{Q}) \subset \mathrm{Sel}_{p^{\infty}}^{\mathrm{rel}}(E/\mathbb{Q}).
\end{equation}

\begin{remark}\label{Kummerremark}Let 
$$\delta_{\ell,n} : E(\mathbb{Q}_{\ell}) \otimes_{\mathbb{Z}}\mathbb{Z}/p^n \rightarrow H^1(\mathbb{Q}_{\ell},E[p^n])$$
be the usual Kummer map (here, when $n = \infty$ we let $\mathbb{Z}/p^{\infty} = \mathbb{Q}_p/\mathbb{Z}_p$). Let $\mathrm{im}(\delta_{\ell,n})$ denote the image of $\delta_{\ell,n}$. Then the long exact sequence attached to the short exact sequence of $G_{\ell}$-modules 
$$0 \rightarrow E[p^n] \rightarrow E \xrightarrow{p^n} E \rightarrow 0$$
implies 
$$\mathrm{im}(\delta_{\ell,n}) = \ker\left(H^1(G_{\ell},E[p^n]) \rightarrow H^1(G_{\ell},E)\right)$$
for all places $\ell$. Therefore, (\ref{Kummerdefinecomposition}) implies 
\begin{align*}\mathrm{Sel}_{p^n}(E/\mathbb{Q}) := \ker\left(\prod_{\ell} \mathrm{res}_{\ell} : H^1(\mathbb{Q},E[p^n]) \rightarrow \prod_{\ell} H^1(G_{\ell},E[p^n])/\mathrm{im}(\delta_{\ell,n})\right)
\end{align*}
and
$$\mathrm{Sel}_{p^n}^{\mathrm{rel}}(E/\mathbb{Q}) := \ker\left(\prod_{\ell \neq p} \mathrm{res}_{\ell} : H^1(\mathbb{Q},E[p^n]) \rightarrow \prod_{\ell \neq p} H^1(G_{\ell},E[p^n])/\mathrm{im}(\delta_{\ell,n})\right).
$$
%and
%\begin{align*}\mathrm{Sel}_{p^n}^{\mathrm{str}}(E/\mathbb{Q}) := \ker\left(\prod_{\ell} \mathrm{res}_{\ell} : H^1(\mathbb{Q},E[p^n]) \rightarrow H^1(G_p,E[p^n]) \times \prod_{\ell \neq p} H^1(G_{\ell},E[p^n])/\mathrm{im}(\delta_{\ell,n}) \right).
%\end{align*}

In particular, $\mathrm{Sel}_{p^n}(E/\mathbb{Q})$ is the usual $p^n$-Selmer group of $E/\mathbb{Q}$ (\cite[Chapter X.4]{Silverman}) defined by the Kummer local conditions.
\end{remark}

We also have the following alternate descriptions of the $p^{\infty}$-Selmer groups from Definition \ref{unramifiedrelaxedSelmerdefinition}. 
\begin{proposition}\label{alternatedescriptionproposition}
$$\mathrm{Sel}_{p^{\infty}}(E/\mathbb{Q}) = \ker\left(\prod_{\ell} \mathrm{res}_{\ell} : H^1(\mathbb{Q},E[p^{\infty}]) \rightarrow H^1(G_p,E[p^{\infty}])/\mathrm{im}(\delta_{p,\infty}) \times \prod_{\ell \neq p} H^1(G_{\ell},E[p^{\infty}])\right)$$
and
$$\mathrm{Sel}_{p^{\infty}}^{\mathrm{rel}}(E/\mathbb{Q}) := \ker\left(\prod_{\ell \neq p} \mathrm{res}_{\ell} : H^1(\mathbb{Q},E[p^{\infty}]) \rightarrow \prod_{\ell \neq p} H^1(G_{\ell},E[p^{\infty}])\right).
$$
%and
%$$\mathrm{Sel}_{p^{\infty}}^{\mathrm{str}}(E/\mathbb{Q}) := \ker\left(\prod_{\ell} \mathrm{res}_{\ell} : H^1(\mathbb{Q},E[p^{\infty}]) \rightarrow \prod_{\ell} H^1(G_{\ell},E[p^{\infty}]) \right).$$
\end{proposition}

\begin{proof}Since $E(\mathbb{Q}_{\ell})$ contains a pro-$\ell$ subgroup of finite index, if $\ell \neq p$ we have
$$E(\mathbb{Q}_{\ell}) \otimes_{\mathbb{Z}} \mathbb{Q}_p/\mathbb{Z}_p = 0.$$
In particular, $\mathrm{im}(\delta_{\ell,\infty})  = 0$. Now the assertion follows from the descriptions of the $p^n$-Selmer groups in Remark \ref{Kummerremark} with $n = \infty$.
\end{proof}

Recall the notation of $S(\lambda_E)$ and $S'(\rho)$ from Definition \ref{GL1Selmerdefinition}. Without loss of generality, we may choose the algebraic closures $\overline{\mathbb{Q}}_{\ell}$ and the embeddings $\overline{\mathbb{Q}} \subset \overline{\mathbb{Q}}_{\ell}$ in Definition \ref{unramifiedrelaxedSelmerdefinition} to be equal to the algebraic closures $\overline{K}_v$ and the embeddings $\overline{K} \subset \overline{K}_v$ from Definition \ref{GL1Selmerdefinition}. From (\ref{inducedformula}) and Shapiro's lemma we have that 
$$H^1(K,(K_p/\mathcal{O}_{K_p})(\lambda_E)) = H^1(\mathbb{Q},E[p^{\infty}] \otimes_{\mathbb{Z}_p}\mathcal{O}_{K_p}).$$
Thus Proposition \ref{alternatedescriptionproposition} implies 
\begin{equation}\label{Shapiro}\mathrm{Sel}_{p^{\infty}}(E/\mathbb{Q})\otimes_{\mathbb{Z}_p}\mathcal{O}_{K_p} = S(\lambda_E), \hspace{1cm} \mathrm{Sel}_{p^{\infty}}^{\mathrm{rel}}(E/\mathbb{Q})\otimes_{\mathbb{Z}_p}\mathcal{O}_{K_p} = S'(\lambda_E).
\end{equation}
%These identifications hold in the category of $\mathcal{O}_{K_p}$-modules (where the $\mathcal{O}_{K_p}$-modules structure on $\mathrm{Sel}_{p^{\infty}}(E/\mathbb{Q})$ is given by the $\mathcal{O}_{K_p}$-action on $E[p^{\infty}]$ induced by the CM-action of $\mathcal{O}_K$ on $E$). The identifications (\ref{Shapiro}) also hold in the category of $\mathbb{Z}_p$-modules via restriction of scalars along $\mathbb{Z}_p \subset \mathcal{O}_{K_p}$. 

%\begin{definition}Given any $\mathbb{Z}_p$-module $M$, let $M^{\vee}$ denote the Pontryagin dual, i.e.
%$$M^{\vee} := \mathrm{Hom}_{\mathbb{Z}_p}(M,\mathbb{Q}_p/\mathbb{Z}_p).$$
%\end{definition}

%By Proposition \ref{relaxSelmertheorem}, we have 
%\begin{equation}\label{relaxcontrol}\mathrm{Sel}_{p^{\infty}}^{\mathrm{rel}}(E/\mathbb{Q})^{\vee} = \mathrm{Hom}_{\mathcal{O}_{K_p}}(S'(\lambda_E)_{\mathrm{div}},K_p/\mathcal{O}_{K_p}) = \mathcal{X}_{\chi_E}(\lambda_E^{-1})/(\Gamma_K-1),
%\end{equation}
%where here we use the notation from Definition \ref{Lambdadefinition} (2) for $\rho = \mathbf{1}$. 
%%Similarly, by class field theory and essentially the same argument as for Proposition \ref{relaxSelmertheorem} \emph{mutatis mutandis}, we have
%%\begin{equation}\label{unramifiedcontrol}\mathrm{Sel}_{p^{\infty}}^{\mathrm{ur}}(E/\mathbb{Q})^{\vee} \otimes_{\mathbb{Z}_p}K_p = \mathcal{Y}_{\chi_E}(\lambda_E^{-1})/(\Gamma_K - 1).
%%\end{equation}

We have the following descent results. Below, we will use the notation $M/(\Gamma_K - \rho(\Gamma_K))$ of Definition \ref{Lambdadefinition} (2) for $\rho = \lambda_E^{-2}$. For the remainder of this section, $\mathrm{Hom}$ will be taken in the category of $\mathcal{O}_{K_p}$-modules unless otherwise indicated. 

\begin{proposition}\label{relaxtwistproposition}
\begin{enumerate}
\item We have an isomorphism of $\mathcal{O}_{K_p}$-modules
\begin{equation}\label{relaxtwist}S'(\lambda_E) \cong S'(\lambda_E^{-1}).
\end{equation}
\item We have an isomorphism of $K_p$-vector spaces
$$\mathrm{Hom}(S'(\lambda_E^{-1})_{\mathrm{div}},K_p/\mathcal{O}_{K_p})\otimes_{\mathbb{Z}_p}\mathbb{Q}_p = \mathcal{X}_{\chi_E}(\lambda_E^{-1})/(\Gamma_K - \lambda_E^{-2}(\Gamma_K)).$$
\item We have an isomorphism of $K_p$-vector spaces
\begin{equation}\label{relaxcontrol2}\mathrm{Hom}(\mathrm{Sel}_{p^{\infty}}^{\mathrm{rel}}(E/\mathbb{Q})_{\mathrm{div}}\otimes_{\mathbb{Z}_p}\mathcal{O}_{K_p},K_p/\mathcal{O}_{K_p})\otimes_{\mathbb{Z}_p}\mathbb{Q}_p \cong \mathcal{X}_{\chi_E}(\lambda_E^{-1})/(\Gamma_K - \lambda_E^{-2}(\Gamma_K)).
\end{equation}
\end{enumerate}
\end{proposition}

\begin{proof}

\textbf{(1)}: Let $G_K = \mathrm{Gal}(\overline{K}/K)$ and for each place $v$ of $K$ let $G_v = \mathrm{Gal}(\overline{K}_v/K_v)$. Given a group $G$, let $G^{\mathrm{ab}}$ denote its maximal abelian quotient. Note that $\lambda_E : G_K \rightarrow \mathcal{O}_{K_p}^{\times}$ factors through $\lambda_E : G_K^{\mathrm{ab}} \rightarrow \mathcal{O}_{K_p}^{\times}$, and recall that in the notation of Definition \ref{GL1Selmerdefinition}, $W(\lambda_E^n)$ is the $G_K$-module $K_p/\mathcal{O}_{K_p}$ on which $G_K$ acts through $G_K \twoheadrightarrow \mathrm{Gal}(\mathcal{K}_{\infty}/K) \xrightarrow{\lambda_E^n} \mathcal{O}_{K_p}^{\times}$. Then 
$$\mathbf{i} : G_K^{\mathrm{ab}} \rightarrow G_K^{\mathrm{ab}}, \hspace{1cm} \gamma \mapsto \gamma^{-1}$$
is an isomorphism, and induces an isomorphism
\begin{equation}\label{groupcohomologyisos}H^1(K,W(\lambda_E)) = H^1(G_K^{\mathrm{ab}},W(\lambda_E)) \overset{\mathbf{i}}{\cong} H^1(G_K^{\mathrm{ab}},W(\lambda_E^{-1})) = H^1(K,W(\lambda_E^{-1})).
\end{equation}
Similarly, we get isomorphisms
\begin{align*}H^1(G_K^{\mathrm{ab}}/G_v^{\mathrm{ab}},W(\lambda_E)) &\overset{\mathbf{i}}{\cong} H^1(G_K^{\mathrm{ab}}/G_v^{\mathrm{ab}},W(\lambda_E^{-1})),\\
H^1(G_v,W(\lambda_E)) = H^1(G_v^{\mathrm{ab}},W(\lambda_E)) &\overset{\mathbf{i}}{\cong} H^1(G_v^{\mathrm{ab}},W(\lambda_E^{-1})) = H^1(G_v,W(\lambda_E^{-1})).
\end{align*}
From the inflation-restriction sequence attached to $G_v^{\mathrm{ab}} \subset G_K^{\mathrm{ab}} \rightarrow G_K^{\mathrm{ab}}/G_v^{\mathrm{ab}}$, we get identifications
\begin{equation}\label{unramifiedlocalconditionisos}\ker\left(\mathrm{res}_v : H^1(K,W(\lambda_E)) \rightarrow H^1(G_v,W(\lambda_E))\right) \overset{\mathbf{i}}{\cong} \ker\left(\mathrm{res}_v : H^1(K,W(\lambda_E^{-1})) \rightarrow H^1(G_v,W(\lambda_E^{-1}))\right)
\end{equation}
%It is known (\cite{BlochKato}) that for $v \neq \frak{p}$, the kernel of $\mathrm{res}_v : H^1(K,W(\rho)) \rightarrow H^1(G_v,W(\rho))$ is equal to $H^1(G_v/I_v,W(\rho)^{I_v})$ (the unramified local condition). 
which are compatible with (\ref{groupcohomologyisos}). Now the assertion follows from Definition \ref{GL1Selmerdefinition} and (\ref{unramifiedlocalconditionisos}). 
\\

\textbf{(2)}: This follows from the same argument as that of Proposition \ref{relaxSelmertheorem}, using the fact that $\mathrm{Gal}(\overline{K}/\mathcal{K}_{\infty})$ acts trivially on $W(\lambda_E^{-1})$, and so
\begin{align*}H^1(\mathcal{K}_{\infty},W(\lambda_E^{-1})) &\cong \mathrm{Hom}(\mathrm{Gal}(\overline{K}/\mathcal{K}_{\infty}),W(\lambda_E^{-1})) = \mathrm{Hom}(\mathrm{Gal}(\overline{K}/\mathcal{K}_{\infty}),W(\lambda_E)) \otimes_{\mathcal{O}_{K_p}}\mathcal{O}_{K_p}(\lambda_E^{-2}) \\
&\hspace{-1.8cm}\overset{(\ref{Kummerinclusion})}{\supset} \mathrm{Hom}(\mathcal{X}|_{\chi_E},W(\lambda_E)) \otimes_{\mathcal{O}_{K_p}}\mathcal{O}_{K_p}(\lambda_E^{-2}) = \mathrm{Hom}(\mathcal{X}|_{\chi_E},W(\lambda_E)\otimes_{\mathcal{O}_{K_p}}\mathcal{O}_{K_p}((\lambda_E/\chi_E)^{-2})),
\end{align*}
recalling, in the last equality, that $\lambda_E/\chi_E = \lambda_E|_{\Gamma_K}$ by (\ref{chiE}). 
\\

\textbf{(3)}: This follows from $\mathrm{Sel}_{p^{\infty}}^{\mathrm{rel}}(E/\mathbb{Q}) \otimes_{\mathbb{Z}_p}\mathcal{O}_{K_p} = S'(\lambda_E)$ and parts (1) and (2). 

\end{proof}

Recall 
$$r_p(E/\mathbb{Q}) := \mathrm{corank}_{\mathbb{Z}_p}\mathrm{Sel}_{p^{\infty}}(E/\mathbb{Q}).$$

We continue to assume $E/\mathbb{Q}$ is as in Assumption \ref{lambdalambdaEassumption}, so that $f_0 \ge 4$. 

%\begin{assumption}\label{Nassumptionrank1}For the rest of the section until Theorem \ref{BSDrank1theorem}, assume that $f_0 \ge 4$. 
%\end{assumption}

\begin{proposition}\label{kernelsubspaceproposition}Suppose $r_p(E/\mathbb{Q}) = 1$. Recall $U_0$ from (\ref{saturation}), and $\delta_E$ from Definition \ref{deltaEdefinition}. We have
\begin{equation}\label{kernelsubspace}U_0 = \ker(\delta|_{q_{\mathrm{dR}}  = 1}) \overset{(\ref{kernelscoincide})}{=} \ker(\delta_E).
\end{equation}
In particular, $U_0$ is a direct summand of $\mathbb{U}_{\chi_E}^1(\lambda_E^{-1})$. 
\end{proposition}

\begin{proof}Since $r_p(E/\mathbb{Q}) = \mathrm{corank}_{\mathbb{Z}_p}\mathrm{Sel}_{p^{\infty}}(E/\mathbb{Q}) = 1$, by the $p^{\infty}$-Selmer corank/analytic rank parity conjecture (or ``$p$-parity conjecture'', known for all elliptic curves over $\mathbb{Q}$ by \cite{Nekovar}, \cite{Kim} and \cite{DokchitserDokchitser}, see \cite{Dokchitser} for an overview), we have $w(E/\mathbb{Q}) = -1$ and so $L(E/\mathbb{Q},1) = 0$. Thus, by (\ref{babyinterpolation}) we have $\delta(\xi_E)|_{q_{\mathrm{dR}} = 1} = 0$.\footnote{We could also invoke Theorem \ref{BSDrank0theorem} to prove this.}  Hence 
$$\xi_E \in \ker(\delta|_{q_{\mathrm{dR}} = 1}). $$
By Proposition \ref{splitproposition} (or \cite[Lemma 11.9]{RubinMC}), $\ker(\delta|_{q_{\mathrm{dR}} = 1}) = \ker(\delta_E)$ is a direct summand of $\mathbb{U}_{\chi_E}^1(\lambda_E^{-1})$ (recall our $\delta|_{q_{\mathrm{dR}} = 1}$ is a multiple of $\delta_E$ by (\ref{deltascoincide}), and the latter is the twist by $\lambda_E^{-1}$ of the reciprocity map ``$\delta$'' in Proposition 11.7 of op. cit.). Thus by (\ref{saturation}) we have 
$$\ker(\delta|_{q_{\mathrm{dR}} = 1}) = \ker(\delta_E) \subset U_0.$$
Choose a decomposition $\mathbb{U}_{\chi_E}^1(\lambda_E^{-1}) = \ker(\delta_E) \oplus U'$, where $U' \cong \Lambda_{\mathcal{O}_{K_p}}[1/p]$. Then the image of $U_0 \subset \mathbb{U}_{\chi_E}^1(\lambda_E^{-1}) \twoheadrightarrow U'$ is $U_0/\ker(\delta_E)$. By Proposition \ref{freeproposition}, this latter module is $\Lambda_{\mathcal{O}_{K_p}}[1/p]$-torsion. Since $U'$ is $\Lambda_{\mathcal{O}_{K_p}}[1/p]$-torsion-free we have $U_0/\ker(\delta_E) = 0$, which gives (\ref{kernelsubspace}) and the direct summand statement. 

\end{proof}

\begin{choice}\label{U2finallyfixed}Suppose $r_p(E/\mathbb{Q}) = 1$. By Proposition \ref{kernelsubspaceproposition}, $U_0$ is a direct summand of $\mathbb{U}_{\chi_E}^1(\lambda_E^{-1})$. Henceforth, fix a choice of decomposition
\begin{equation}\label{U0U2decomposition}\mathbb{U}_{\chi_E}^1(\lambda_E^{-1}) = U_0 \oplus U_2,
\end{equation}
so that in particular $U_2 \cap \overline{\mathcal{E}}_{\chi_E}^1(\lambda_E^{-1}) = 0$, and $U_2$ satisfies the condition (\ref{globalintersect0}) of Choice \ref{Udecompositionchoice}. In particular, we may and do take $U_1 = U_0$, and so $\lambda_1 = 1$, $\lambda_2 = 0$ in the notation of Section \ref{rank1U2choicesection}.
\end{choice}

\begin{remark}\label{U2remark}We remark that $\ker(\delta_E)$ plays different roles in the rank 0 situation (see \cite[Section 11]{RubinMC}, where $\ker(\delta_E)$ is denoted by ``$U_2$'' in the notation of Lemma 11.8 of op. cit., also noting that this notation differs from our use of the notation ``$U_2$'') and in the rank 1 situation. In the former situation, it plays the role of the (Kummer) local condition defining the usual Selmer group (\cite[Chapter X.4]{Silverman}), and hence takes the role of $U_2$ in Corollary \ref{RMC}. In the latter situation, it plays the role of $U_1 = U_0$, the domain of our ``derivative'' operator $\delta'$ from (\ref{delta'}), and a different local condition $U_2$ is chosen for Corollary \ref{RMC}. 
\end{remark}

%\begin{corollary}We have 
%$$\mathbb{U}_{\chi_E}^1(\lambda_E^{-1}) \cong U_0 \oplus U_2.$$
%\end{corollary}

%\begin{proof}Recall by Choice \ref{Udecompositionchoice}, we have $U_2 \cap \overline{\mathcal{E}}_{\chi_E}^1(\lambda_E^{-1}) = 0$. Hence since $U_2$ is free of rank 1, we have $U_2 \cap U_0 = 0$. 
%\end{proof}

The following Theorem makes use of our $r_p(E/\mathbb{Q}) = 1$ assumption. In essence, we use this assumption and (\ref{kernelsubspace}) to identify $U_0$ with the kernel $\ker(\delta_E)$ of the reciprocity map $\delta_E$ from Definition \ref{deltaEdefinition} and thus show $\delta'(\beta_0) \neq 0$ where $\beta_0$ is our previously fixed $\Lambda_{\mathcal{O}_{K_p}}[1/p]$-basis of $U_0$ (see Choice \ref{beta0choice}). The argument proceeds via proof by contradiction:  if $\delta'(\beta_0) = 0$ then the functoriality of the reciprocity map (\ref{reciprocitydiagram}) implies $\delta'$ would be the zero map for \emph{every} elliptic curve $E'/\mathbb{Q}$ with CM by $\mathcal{O}_K$. However, by previous formulas (namely the Rubin-type formula (\ref{Heegnerpointidentity2})), this would imply that every such $E'/\mathbb{Q}$ has rank not equal to 1, which is clearly known to be false (we exhibit a list of examples known to have rank 1).

\begin{theorem}\label{beta0nonzerotheorem}Suppose $r_p(E/\mathbb{Q}) = 1$. 
\begin{enumerate}
\item Letting $\delta_E$ be as in Definition \ref{deltaEdefinition}, we have 
\begin{equation}\label{Ukerdeltainclusion}\overline{\mathcal{E}}_{\chi_E}^1(\lambda_E^{-1}) \subset \ker(\delta_E)
\end{equation}
and 
\begin{equation}\label{UE}\left(\ker(\delta_E)/\overline{\mathcal{E}}_{\chi_E}^1(\lambda_E^{-1})\right)/(\Gamma_K-\lambda_E^{-2}(\Gamma_K)) = 0.
\end{equation}
\item Moreover for $\beta_0$ as in Choice \ref{beta0choice} and $\delta'$ as in (\ref{delta'}),
\begin{equation}\label{beta0nonzero}\delta'(\beta_0)\neq 0.
\end{equation}
\end{enumerate}
\end{theorem}

\begin{proof}%Note that since $K$ has class number 1, $\mathcal{K}_n \otimes_K K_{\frak{p}} = \mathcal{K}_{n,p}$, and so we have an identification 
%$$\mathrm{loc}_{\frak{p}} : \mathbb{U}_{\chi_E}^1(\lambda_E^{-1}) \xrightarrow{\sim} \mathcal{U}_{F,\chi_F}^1(\kappa^{-1})$$
%of $\Lambda_{\mathcal{O}_{K_p}}[1/p]$-modules. In particular, for \emph{any} elliptic curve $E'/\mathbb{Q}$ with CM by $\mathcal{O}_K$, we have 
%$$\mathrm{loc}_{\frak{p}} : \mathbb{U}_{\chi_{E'}}^1(\lambda_{E'}^{-1}) \xrightarrow{\sim} \mathcal{U}_{F',\chi_{F'}}^1(\kappa'^{-1})$$ 
%for $F'$ the formal group $\hat{A}'$ of an appropriate twist $A'$ of $E'$, defined over some unramified extension $L_p'/K_p$. By Lubin-Tate theory, the modules $\mathcal{U}_{F',\chi_{F'}}^1(\kappa^{-1})$ form an inverse system indexed by unramified extensions $L_p'' \supset L_p' \supset K_p$, height 2 formal modules $F''$ for $L_p''/K_p$ and $F'$ for $L_p'/K_p$, respectively, with transition maps
%$$\mathrm{Nm}_{L_p''/L_p'} : \mathcal{U}_{F''}^1 \twoheadrightarrow \mathcal{U}_{F'}^1.$$
%On isotypic components, this induces isomorphisms of $\Lambda_{\mathcal{O}_{K_p}}[1/p]$-modules
%\begin{equation}\label{localidentification}\mathcal{U}_{F'',\chi_{F''}}^1 \xrightarrow{\sim} \mathcal{U}_{F',\chi_{F'}}^1.
%\end{equation}

Let $U_0$ be as in (\ref{saturation}), and recall the decomposition \ref{U0U2decomposition}. Since $\overline{\mathcal{E}}_{\chi_E}^1(\lambda_E^{-1})/\overline{\mathcal{C}}_{\chi_E}^1(\lambda_E^{-1})$ is $\Lambda_{\mathcal{O}_{K_p}}[1/p]$-torsion by Proposition \ref{ranksproposition}, from (\ref{saturation}) we see and the fact that $\overline{\mathcal{C}}_{\chi_E}^1(\lambda_E^{-1})$ is free we see that $\overline{\mathcal{E}}_{\chi_E}^1(\lambda_E^{-1}) \subset U_0 \overset{(\ref{kernelsubspace})}{=} \ker(\delta_E)$. This gives (\ref{Ukerdeltainclusion}).

We will prove (\ref{UE}) and (\ref{beta0nonzero}) at the same time. For the remainder of this proof, we will adopt Convention \ref{UEconvention}. Let $E'/\mathbb{Q}$ be \emph{any} elliptic curve with CM by $\mathcal{O}_K$ with $r_p(E'/\mathbb{Q}) = 1$, and let $\lambda_{E'}$ be its associated infinity type $(1,0)$ algebraic Hecke character. Let $f_0'$ be the positive integer generating $\frak{f}_{E'}^{(p)}$ (see Assumption \ref{pconductorassumption} (1)) and assume $f_0' \ge 4$. Let 
$$U_0' \subset \mathbb{U}_{E',\chi_{E'}}^1(\lambda_{E'}^{-1})$$
be as in (\ref{saturation}) with $E'$ in place of $E$, and continue to let $U_0 \subset \mathbb{U}_{E,\chi_E}^1(\lambda_E^{-1})$ be as in (\ref{saturation}) for our given $E$. Letting $\frak{f}_0' \subset \mathcal{O}_K$ be the least common multiple of $\frak{f}_E^{(p)}$ and $\frak{f}_{E'}^{(p)}$, we can find $A_0'$ as in Choice \ref{A0'choice} for both $E$ and $E'$ simultaneously. We thus have a natural identification of $\Lambda_{\mathcal{O}_{K_p}}[1/p]$-modules 
$$U_0 \overset{(\ref{kernelsubspace})}{=} \ker(\delta_E) \overset{(\ref{kernelidentity}) \; \text{for $E$}}{=}\ker(\delta_{A_0'})\overset{(\ref{kernelidentity}) \; \text{for $E'$}}{=} \ker(\delta_{E'}) \overset{(\ref{kernelsubspace})}{=} U_0'.$$
Hence $\beta_0$ induces a basis of $U_0'$. Moreover, since $\overline{\mathcal{E}}_{E,\chi_E}^1(\lambda_E^{-1})$ is the kernel of the Artin reciprocity map $\mathrm{rec}_E : \mathbb{U}_{E,\chi_E}^1(\lambda_E^{-1}) \rightarrow \mathrm{Gal}(K^{\mathrm{ab}}/K(E[p^{\infty}]))$ (here, we are briefly writing ``$\mathrm{rec}_E$'' instead of ``$\mathrm{rec}$'' as before in order to emphasize the dependence on $E$), we have a natural identification of $\Lambda_{\mathcal{O}_{K_p}}[1/p]$-modules
$$\overline{\mathcal{E}}_{E,\chi_E}^1(\lambda_E^{-1}) = \ker(\mathrm{rec}_E) \overset{(\ref{kernelidentity}) \; \text{for $E$}}{=} \ker(\mathrm{rec}_{A_0'}) \overset{(\ref{kernelidentity}) \; \text{for $E'$}}{=} \ker(\mathrm{rec}_{E'}) = \overline{\mathcal{E}}_{E',\chi_{E'}}^1(\lambda_{E'}^{-1}).$$
%Now suppose that the $\delta'(\beta_0) = 0$. By the horizontal norm relations for elliptic units (\cite[Proposition II.2.5]{deShalit}), we see that the saturations (\ref{saturation}) of elliptic units for varying $E'/\mathbb{Q}$, and so $U_0 = U_0'$ under the identifications (\ref{localidentification}), and so $\beta_0$ induces a basis of $U_0'$. 
Write, as in (\ref{multiple2}), $\lambda_0'\beta_0 = \xi_{E'}$ for $\lambda_0' \in \Lambda_{\mathcal{O}_{K_p}}[1/p]$. 

We now prove (\ref{UE}) and (\ref{beta0nonzero}). Assume, for sake of contradiction, that 
\begin{equation}\label{possibilities}\left(\ker(\delta_E)/\overline{\mathcal{E}}_{\chi_E}^1(\lambda_E^{-1})\right)/(\Gamma_K-\lambda_E^{-2}(\Gamma_K)) \neq 0 \hspace{.25cm} \text{or} \hspace{.25cm} \delta'(\beta_0) = 0.
\end{equation}
Assume that the first case of (\ref{possibilities}) holds. Then the discussion at the end of the previous paragraph shows that 
$$\left(\ker(\delta_{E'})/\overline{\mathcal{E}}_{E',\chi_{E'}}^1(\lambda_{E'}^{-1})\right)/(\Gamma_K-\lambda_E^{-2}(\Gamma_K)) = \left(\ker(\delta_E)/\overline{\mathcal{E}}_{E,\chi_E}^1(\lambda_E^{-1})\right)/(\Gamma_K-\lambda_E^{-2}(\Gamma_K)) \neq 0.$$
Now from the surjection
$$\left(U_{0,E'}/\overline{\mathcal{C}}_{E',\chi_{E'}}^1(\lambda_{E'}^{-1})\right)/(\Gamma_K-\lambda_E^{-2}(\Gamma_K)) \twoheadrightarrow \left(\ker(\delta_{E'})/\overline{\mathcal{E}}_{E',\chi_{E'}}^1(\lambda_{E'}^{-1})\right)/(\Gamma_K-\lambda_E^{-2}(\Gamma_K)) \neq 0,$$
we see from Proposition \ref{xiEgeneratorproposition} that 
$$\lambda_0'(\lambda_E^{-2}) = \xi_{E'}(\lambda_E^{-2})/\beta_0(\lambda_E^{-2}) = \left(U_{0,E'}/\overline{\mathcal{C}}_{E',\chi_{E'}}^1/(\lambda_{E'}^{-1})\right)/(\Gamma_K-\lambda_E^{-2}(\Gamma_K)) \neq 0,$$
where $\lambda_0'(\lambda_E^{-2})$, $\beta_0(\lambda_E^{-2})$ and $\xi_{E'}(\lambda_E^{-2})$ are written in the notation of Definition \ref{characterevaluation}, and $\lambda_E$ is as in (\ref{lambdaE}). 

Thus, in either case of (\ref{possibilities}), we see that
$$0 = \lambda_0'(\lambda_E^{-2})\delta'(\beta_0) \overset{(\ref{equivariance2})}{=} \delta'(\xi_{E'}) \overset{(\ref{explicitreciprocity2})}{=} C_{E'}\cdot (\lambda_{E'}^{-1}(\frak{a}) - \mathbb{N}\frak{a})\cdot  D_1^{-1}\mathcal{L}_{\lambda_{E'}}|_{q_{\mathrm{dR}} = 1}.$$
Since $\lambda_{E'}^{-1}(\frak{a}) - \mathbb{N}\frak{a} \neq 0$ (since $\frak{a} \subset \mathcal{O}_K$ is a proper integral ideal), we have 
$$D_1^{-1}\mathcal{L}_{\lambda_{E'}}|_{q_{\mathrm{dR}} = 1} = 0.$$
Let $\lambda_{E'}$ be the infinity type $(1,0)$ Hecke character associated with $E'$, and let $\chi_0'$ be a finite order character chosen as in (\ref{chi0choice}) with respect to $\lambda = \lambda_{E'}$. Hence, by (\ref{Heegnerpointidentity2}) and the previous displayed equation, we see that the Heegner point 
$$P_{\theta_{\lambda_{E'}/\chi_0'},\chi_0'} = 0,$$
and hence by the Gross-Zagier formula for Shimura curves \cite[Theorem 1.2]{YZZ} and (\ref{analyticrankequality}), we see that $\mathrm{ord}_{s = 1}L(E'/\mathbb{Q},s) \neq 1$. 

Since $E'/\mathbb{Q}$ as in the second paragraph was arbitrary, we have 
$$\mathrm{ord}_{s = 1}L(E'/\mathbb{Q},s) \neq 1 \hspace{.25cm} \text{for \emph{any} $E'/\mathbb{Q}$ such that $E/K$ has CM by $\mathcal{O}_K$ and $r_p(E'/\mathbb{Q}) = 1$}.$$
However, this is clearly not true, as for any imaginary quadratic field $K$ of class number 1 and $p$ (the unique finite prime) ramified in $K$, 
$$\text{$\exists \; E'/\mathbb{Q}$ such that $E/K$ has CM by $\mathcal{O}_K$ and such that} \hspace{.3cm} \mathrm{ord}_{s = 1}L(E'/\mathbb{Q},s) = 1 \hspace{.3cm} \text{and} \hspace{.25cm} r_p(E'/\mathbb{Q}) = 1.$$
For example, computations by the LMFDB collaboration (\cite{LMFDB}) supply the following examples (indexed by Cremona's labelling):
\begin{enumerate}
\item For $K = \mathbb{Q}(i)$, $p = 2$, take $E' = 256b1: y^2 = x^3 - 2x$.
\item For $K = \mathbb{Q}(\sqrt{-2})$, $p = 2$, take $E' = 256a1: y^2 = x^3 + x^2 - 3x + 1$.
\item For $K = \mathbb{Q}(\sqrt{-3})$, $p = 3$, take $E' = 225a2: y^2 + y = x^3 + 1$.
\item For $K = \mathbb{Q}(\sqrt{-7})$, $p = 7$, take $E' = 441d1: y^2+xy+y=x^3-x^2-20x+46$.
\item For $K = \mathbb{Q}(\sqrt{-11})$, $p = 11$, take $E' = 121b1: y^2 + y = x^3 - x^2 - 7x + 10$.
\item For $K = \mathbb{Q}(\sqrt{-19})$, $p = 19$, take $E' = 361a1: y^2+y=x^3-38x+90$.
\item For $K = \mathbb{Q}(\sqrt{-43})$, $p = 43$, take $E' = 1849a1: y^2+y=x^3-860x+9707$.
\item For $K = \mathbb{Q}(\sqrt{-67})$, $p = 67$, take $E' = 4489a1: y^2+y=x^3-7370x+243528$.
\item For $K = \mathbb{Q}(\sqrt{-163})$ $p = 163$, take $E' = 26569a1: y^2+y=x^3-2174420x+1234136692$.
\end{enumerate}
Hence we have a contradiction, and (\ref{UE}) and (\ref{beta0nonzero}) are established.

\end{proof}

We will need the following Lemma.

\begin{lemma}\label{Selmercharacterizationtwistlemma}Suppose that $r_p(E/\mathbb{Q}) = 1$. Then
$$\left(\mathcal{X}_{\chi_E}(\lambda_E^{-1})/\mathrm{rec}\left(\ker(\delta_E)\right)\right)/(\Gamma_K-\lambda_E^{-2}(\Gamma_K))$$
is a 1-dimensional $K_p$-vector space.
\end{lemma}

\begin{proof}Throughout the proof, we will let
$$W' := W(\lambda_E)\otimes_{\mathcal{O}_{K_p}}\mathcal{O}_{K_p}((\lambda_E/\chi_E)^{-2}).$$
Analogously to (\ref{Selmergroupcharacterization}), we will show that there is an isomorphism
\begin{equation}\label{Selmergroupcharacterizationtwist}\begin{split}S(\lambda_E)_{\mathrm{div}} &\cong \mathrm{Hom}(\mathcal{X}_{\chi_E}(\lambda_E^{-1})/\mathrm{rec}(\ker(\delta_E)),W((\lambda_E/\chi_E)^{-2}))^{\mathrm{Gal}(\mathcal{K}_{\infty}/K)},
\end{split}
\end{equation}
which implies
$$\mathrm{Hom}(S(\lambda_E),K_p/\mathcal{O}_{K_p}) \otimes_{\mathcal{O}_{K_p}} K_p \cong \left(\mathcal{X}_{\chi_E}(\lambda_E^{-1})/\mathrm{rec}\left(\ker(\delta_E)\right)\right)/(\Gamma_K-\lambda_E^{-2}(\Gamma_K)). 
$$
Given this, the left-hand side of (\ref{Selmergroupcharacterizationtwist}) has $K_p$-dimension $r_p(E/\mathbb{Q}) = 1$ by (\ref{Shapiro}), which gives the assertion. 

%By an analogous argument to that of the proof of Lemma \ref{restrictionlemma} (i.e. twisting the entire argument by $\otimes_{\mathcal{O}_{K_p}}\mathcal{O}_{K_p}(\lambda_E^{-2})$), 

As in the proof of Proposition \ref{relaxtwistproposition}, let $\mathbf{i} : \mathrm{Gal}(\overline{K}_p/K_p)^{\mathrm{ab}} \xrightarrow{\sim} \mathrm{Gal}(\overline{K}_p/K_p)^{\mathrm{ab}}$ be the isomorphism $\gamma \mapsto \gamma^{-1}$. We have isomorphisms
$$H^1(K_p,W(\lambda_E)) = H^1(\mathrm{Gal}(\overline{K}_p/K_p)^{\mathrm{ab}},W(\lambda_E)) \overset{\mathbf{i}}{\cong} H^1(\mathrm{Gal}(\overline{K}_p/K_p)^{\mathrm{ab}},W(\lambda_E^{-1})) = H^1(K_p,W(\lambda_E^{-1})),$$
\begin{align*}H^1(\mathcal{K}_{\infty,p},W(\lambda_E)) &= H^1(\mathrm{Gal}(\overline{K}_p/\mathcal{K}_{\infty,p})^{\mathrm{ab}},W(\lambda_E)) \\
&\overset{\mathbf{i}}{\cong} H^1(\mathrm{Gal}(\overline{K}_p/\mathcal{K}_{\infty,p})^{\mathrm{ab}},W(\lambda_E^{-1})) = H^1(\mathcal{K}_{\infty,p},W(\lambda_E^{-1})).
\end{align*}
These give induced isomorphisms
\begin{equation}\label{givesaninducedisomorphism}\begin{split}H^1(K_p,W(\lambda_E)) &\cong H^1(K_p,W(\lambda_E^{-1})),\\
H^1(\mathcal{K}_{\infty,p},W(\lambda_E))^{\mathrm{Gal}(\mathcal{K}_{\infty,p}/K_p)} &\cong H^1(\mathcal{K}_{\infty,p},W(\lambda_E^{-1}))^{\mathrm{Gal}(\mathcal{K}_{\infty,p}/K_p)}.
\end{split}
\end{equation}
From (\ref{givesaninducedisomorphism}) and (\ref{restrictionfinite}), we see that the restriction map
\begin{align*}H^1(K_p,W(\lambda_E^{-1})) \xrightarrow{\mathrm{res}}H^1(\mathcal{K}_{\infty,p},W(\lambda_E^{-1}))^{\mathrm{Gal}(\mathcal{K}_{\infty,p}/K_p)}
\end{align*}
has finite kernel and cokernel. Given a $\mathrm{Gal}(\mathcal{K}_{\infty,p}/K_p) \overset{(\ref{fixKp})}{=} \Delta_{K,p} \times \Gamma_K$-module $M$, let $M/(\Delta_{K,p}-1)$ denote the $\Delta_{K,p}$-coinvariants of $M$. By replacing $\lambda_E$ with $\lambda_E^{-1}$ in the second paragraph of the proof of Lemma \ref{restrictionlemma}, and recalling that $\lambda_E/\chi_E$ is trivial on $\Delta_{K,p} \overset{(\ref{fixKp})}{\subset} \mathrm{Gal}(\mathcal{K}_{\infty,p}/K_p)$ and $\Delta_K \overset{(\ref{fixK})}{\subset} \mathrm{Gal}(\mathcal{K}_{\infty}/K)$, we see that this restriction map induces an isomorphism
\begin{equation}\label{localHomtwist}\begin{split}H^1(K_p,W(\lambda_E^{-1}))_{\mathrm{div}} &\xrightarrow{\sim} \mathrm{Hom}(\mathcal{U}_E^1,W(\lambda_E^{-1}))^{\mathrm{Gal}(\mathcal{K}_{\infty,p}/K_p)} = \mathrm{Hom}(\mathcal{U}_E^1(\lambda_E),K_p/\mathcal{O}_{K_p})^{\mathrm{Gal}(\mathcal{K}_{\infty,p}/K_p)} \\
&= \mathrm{Hom}(\mathcal{U}_E^1(\lambda_E)/(\Delta_{K,p}-1),K_p/\mathcal{O}_{K_p})^{\mathrm{Gal}(\mathcal{K}_{\infty,p}/K_p)} \\
&= \mathrm{Hom}(\mathcal{U}_{E,\chi_E}^1(\lambda_E^{-1}) \otimes_{\mathcal{O}_{K_p}}\mathcal{O}_{K_p}((\lambda_E/\chi_E)^2),K_p/\mathcal{O}_{K_p})^{\mathrm{Gal}(\mathcal{K}_{\infty,p}/K_p)}\\
&= \mathrm{Hom}(\mathcal{U}_{E,\chi_E}^1,W')^{\mathrm{Gal}(\mathcal{K}_{\infty,p}/K_p)} \overset{(\ref{chiquotientiso})}{\cong} \mathrm{Hom}(\mathbb{U}_{\chi_E}^1,W')^{\mathrm{Gal}(\mathcal{K}_{\infty}/K)}.
\end{split}
\end{equation}
(Here in the last isomorphism, we use the fact that $\mathcal{U}_{E,\chi_E}^1 = \mathcal{U}_E^1|_{\chi_E} \otimes_{\mathbb{Z}_p}\mathbb{Q}_p\overset{(\ref{chiquotientiso})}{\cong} \mathbb{U}^1|_{\chi_E}\otimes_{\mathbb{Z}_p}\mathbb{Q}_p = \mathbb{U}_{\chi_E}^1$, see Definition \ref{slashisotypicdefinition} (3) and (5).)

Precomposing (\ref{localHomtwist}) with the first isomorphism in (\ref{givesaninducedisomorphism}), we get a diagram analogous to (\ref{Selmergroupdiagram})
\begin{equation}\label{Selmergroupdiagramtwist}
\hspace{-.75cm}\begin{tikzcd}[column sep =small]
     &   &  &  \mathrm{Hom}(\mathcal{X}_{\chi_E},W')^{\mathrm{Gal}(\mathcal{K}_{\infty}/K)}\arrow{d} \arrow{dr}{} & \\
     & 0 \arrow{r} & \left(E(K_p) \otimes_{\mathcal{O}_K}K_p/\mathcal{O}_{K_p}\right)_{\mathrm{div}}\arrow{r}{} \arrow{dr}{\varphi'}& H^1(K_p,W(\lambda_E))_{\mathrm{div}} \arrow{d}{\cong} \arrow{r}& \left(H^1(K_p,E)[\frak{p}^{\infty}]\right)_{\mathrm{div}}\arrow{r}& 0\\
     & & & \mathrm{Hom}(\mathbb{U}_{\chi_E}^1,W')^{\mathrm{Gal}(\mathcal{K}_{\infty}/K)} & 
\end{tikzcd}.
\end{equation}
As in the discussion preceding (\ref{Selmergroupdiagram}), we denote the map
$$\mathrm{Hom}(\mathcal{X}_{\chi_E},W') \rightarrow \mathrm{Hom}(\mathbb{U}_{\chi_E}^1,W')$$
by $f \mapsto f|_{\mathbb{U}_{\chi_E}^1}$. By Proposition \ref{relaxtwistproposition} (2), we thus have (cf. (\ref{Selmercharacterization})), 
\begin{equation}\label{Selmercharacterizationtwist}S(\lambda_E)_{\mathrm{div}} \cong \{f \in \mathrm{Hom}(\mathcal{X}_{\chi_E},W')^{\mathrm{Gal}(\mathcal{K}_{\infty}/K)} : f|_{\mathbb{U}_{\chi_E}^1} \in \mathrm{im}(\varphi')\}.
\end{equation}
By an argument analogous to the proof of (\ref{keyreciprocityequality2}), we have
\begin{equation}\label{keyreciprocityequality2twist}\mathrm{im}(\varphi') = \mathrm{Hom}(\mathbb{U}_{\chi_E}^1(\lambda_E^{-1})/\mathrm{ker}(\delta_E),W((\lambda_E/\chi_E)^{-2}))^{\mathrm{Gal}(\mathcal{K}_{\infty}/K)}.
\end{equation}
Now (\ref{Selmergroupcharacterizationtwist}) follows from (\ref{Selmercharacterizationtwist}) and  (\ref{keyreciprocityequality2twist}). 

\end{proof}

Let $\mathrm{Hom}_{\mathbb{Z}_p}$ denote homomorphisms in the category of $\mathbb{Z}_p$-modules. (Recall that throughout this section, we take $\mathrm{Hom}$ to be in the category of $\mathcal{O}_{K_p}$-modules unless otherwise noted.) We prove the following general result on the dimension of $\mathrm{Hom}_{\mathbb{Z}_p}(\mathrm{Sel}_{p^{\infty}}^{\mathrm{rel}}(E/\mathbb{Q}),\mathbb{Q}_p/\mathbb{Z}_p)\otimes_{\mathbb{Z}_p}K_p$, which is mentioned in \cite[Remark before Section 5]{RubinCompositio}. Although we still work under Assumption \ref{lambdalambdaEassumption} as we wish to invoke (\ref{babyinterpolation}) in our proof, this assumption can be removed (see Remark \ref{babyinterpolationremark}). 

\begin{proposition}\label{controlthm}Suppose $r_p(E/\mathbb{Q}) = 1$. Then 
$$\mathcal{X}_{\chi_E}(\lambda_E^{-1})/(\Gamma_K-\lambda_E^{-2}(\Gamma_K)) \overset{(\ref{relaxcontrol2})}{=} \mathrm{Hom}_{\mathbb{Z}_p}(\mathrm{Sel}_{p^{\infty}}^{\mathrm{rel}}(E/\mathbb{Q}),\mathbb{Q}_p/\mathbb{Z}_p)\otimes_{\mathbb{Z}_p}K_p$$
is a 1-dimensional $K_p$-vector space. 
\end{proposition}

\begin{proof}%The case $r_p(E/\mathbb{Q}) = 0$ is just a restatement of (\ref{rp0Xrank1}). Assume now that $r_p(E/\mathbb{Q}) = 1$, and 
Let $\delta_E$ be as in Definition \ref{deltaEdefinition}. 
%Recall 
%\begin{equation}\label{recalldeltaE}\delta_E : \mathbb{U}_{\chi_E}^1(\lambda_E^{-1}) \rightarrow K_p
%\end{equation}
%from Definition \ref{deltaEdefinition} (recall also Choice \ref{generatorchoice}). Further recall that $\delta_E$ is $\Lambda_{\mathcal{O}_{K_p}}[1/p]$-equivariant, where $\Lambda_{\mathcal{O}_{K_p}}[1/p]$ acts on the target $K_p$ through $\Lambda_{\mathcal{O}_{K_p}}[1/p] \twoheadrightarrow \Lambda_{\mathcal{O}_{K_p}}[1/p]/(\Gamma_K-1) = K_p$. This implies that (\ref{recalldeltaE}) factors through the projection 
%\begin{equation}\label{tildeprojection}\mathbb{U}_{\chi_E}^1(\lambda_E^{-1})  \twoheadrightarrow  \mathbb{U}_{\chi_E}^1(\lambda_E^{-1})/(\Gamma_K-1), \hspace{1cm} \beta \mapsto \tilde{\beta}
%\end{equation}
%and thus induces a map of $K_p$-vector spaces
%\begin{equation}\label{tildedelta}\tilde{\delta}_E : \mathbb{U}_{\chi_E}^1(\lambda_E^{-1})/(\Gamma_K-1) \rightarrow K_p,
%\end{equation}
%with kernel the image of $\ker(\delta_E)/(\Gamma_K-1) \rightarrow \mathbb{U}_{\chi_E}^1(\lambda_E^{-1})/(\Gamma_K-1)$. 
%Recall $U_0$ from (\ref{saturation}). 
 From (\ref{fundamentalEStwist}) and (\ref{Ukerdeltainclusion}) we get an exact sequence
$$0 \rightarrow \overline{\mathcal{E}}_{\chi_E}^1(\lambda_E^{-1}) \rightarrow \ker(\delta_E) \xrightarrow{\mathrm{rec}} \mathcal{X}_{\chi_E}(\lambda_E^{-1}) \rightarrow \mathcal{X}_{\chi_E}(\lambda_E^{-1})/\mathrm{rec}(\ker(\delta_E)) \rightarrow 0,$$
recalling that $\mathrm{ker}(\mathrm{rec}) = \overline{\mathcal{E}}_{\chi_E}^1(\lambda_E^{-1})$.
% (For this last inclusion, recall that $\overline{\mathcal{E}}_{\chi_E}^1(\lambda_E^{-1})$ and $\overline{\mathcal{C}}_{\chi_E}^1(\lambda_E^{-1})$ have $\Lambda_{\mathcal{O}_{K_p}}[1/p]$-rank 1 by Proposition \ref{ranksproposition}, and so $\overline{\mathcal{E}}_{\chi_E}^1(\lambda_E^{-1})/\overline{\mathcal{C}}_{\chi_E}^1(\lambda_E^{-1})$ is torsion and thus $\overline{\mathcal{C}}_{\chi_E}^1(\lambda_E^{-1})$ is in the $\Lambda_{\mathcal{O}_{K_p}}[1/p]$-saturation $U_0$.) 
Reducing this exact sequence modulo $(\Gamma_K-\lambda_E^{-2}(\Gamma_K))$, we get an exact sequence of $K_p$-vector spaces
\begin{equation}\label{reducedexactsequence}\begin{split}\overline{\mathcal{E}}_{\chi_E}^1(\lambda_E^{-1})/(\Gamma_K-\lambda_E^{-2}(\Gamma_K)) &\rightarrow \ker(\delta_E)/(\Gamma_K-\lambda_E^{-2}(\Gamma_K)) \xrightarrow{\mathrm{rec}} \mathcal{X}_{\chi_E}(\lambda_E^{-1})/(\Gamma_K-\lambda_E^{-2}(\Gamma_K)) \\
&\rightarrow \left(\mathcal{X}_{\chi_E}(\lambda_E^{-1})/\mathrm{rec}\left(\ker(\delta_E)\right)\right)/(\Gamma_K-\lambda_E^{-2}(\Gamma_K)) \rightarrow 0.
\end{split}
\end{equation}
We have 
$$\left(\ker(\delta_E)/\overline{\mathcal{E}}_{\chi_E}(\lambda_E^{-1})\right)/(\Gamma_K-\lambda_E^{-2}(\Gamma_K)) \overset{(\ref{UE})}{=} 0.$$
Hence (\ref{reducedexactsequence}) induces an exact sequence 
$$0 \rightarrow \mathcal{X}_{\chi_E}(\lambda_E^{-1})/(\Gamma_K-\lambda_E^{-2}(\Gamma_K)) \rightarrow \left(\mathcal{X}_{\chi_E}(\lambda_E^{-1})/\mathrm{rec}\left(\ker(\delta_E)\right)\right)/(\Gamma_K-\lambda_E^{-2}(\Gamma_K)) \rightarrow 0.$$
By Lemma \ref{Selmercharacterizationtwistlemma}, the third term of this exact sequence is 1-dimensional. This gives the assertion. 
\end{proof}

Reducing (\ref{fundamentalEStwist}) modulo $(\Gamma_K-\lambda_E^{-2}(\Gamma_K))$, we get an exact sequence
\begin{equation}\label{reducedfundamentalEStwist}\begin{split}\left(\mathbb{U}_{\chi_E}^1(\lambda_E^{-1})/\overline{\mathcal{C}}_{\chi_E}^1(\lambda_E^{-1})\right)/(\Gamma_K-\lambda_E^{-2}(\Gamma_K))&\xrightarrow{\mathrm{rec}} \mathcal{X}_{\chi_E}(\lambda_E^{-1})/(\Gamma_K-\lambda_E^{-2}(\Gamma_K)) \\
&\rightarrow \mathcal{Y}_{\chi_E}(\lambda_E^{-1})/(\Gamma_K-\lambda_E^{-2}(\Gamma_K)) \rightarrow 0.
\end{split}
\end{equation}

\begin{proposition}\label{nonvanishingU2}Suppose $r_p(E/\mathbb{Q}) = 1$. We have 
$$\mathrm{rec}(U_2)/(\Gamma_K-\lambda_E^{-2}(\Gamma_K)) = \mathcal{X}_{\chi_E}(\lambda_E^{-1})/(\Gamma_K-\lambda_E^{-2}(\Gamma_K)).$$ 
\end{proposition}

\begin{proof}We have $\mathbb{U}_{\chi_E}^1(\lambda_E^{-1}) \overset{(\ref{U0U2decomposition})}{=} U_0 \oplus U_2$. By (\ref{reducedfundamentalEStwist}), we have a map of $K_p$-vector spaces
\begin{equation}\label{vssurjection2}\mathrm{rec}: U_0/(\Gamma_K -\lambda_E^{-2}(\Gamma_K)) \oplus U_2/(\Gamma_K-\lambda_E^{-2}(\Gamma_K)) \rightarrow \mathcal{X}_{\chi_E}(\lambda_E^{-1})/(\Gamma_K-\lambda_E^{-2}(\Gamma_K)).
\end{equation}
Since $U_0 \cong U_2 \cong \Lambda_{\mathcal{O}_{K_p}}[1/p]$, we have 
$$U_i/(\Gamma_K-\lambda_E^{-2}(\Gamma_K)) \cong K_p$$
for $i = 0, 2$. By Proposition \ref{controlthm}, the target of the above map is a 1-dimensional $K_p$-vector space. As in the proof of Proposition \ref{controlthm} in the $r_p(E/\mathbb{Q}) = 1$ case, we see that the kernel of (\ref{vssurjection2}) is 
$$\ker(\delta_E)/(\Gamma_K-\lambda_E^{-2}(\Gamma_K)) \overset{(\ref{kernelsubspace})}{=} U_0/(\Gamma_K-\lambda_E^{-2}(\Gamma_K)).$$
%the kernel of (\ref{vssurjection}) is $\overline{\mathcal{E}}_{\chi_E}^1(\lambda_E^{-1})/(\Gamma_K-\lambda_E^{-2}(\Gamma_K)) \subset U_0/(\Gamma_K-\lambda_E^{-2}(\Gamma_K))$ (the latter inclusion following from (\ref{saturation}) and the fact that $\overline{\mathcal{E}}_{\chi_E}^1(\lambda_E^{-1})/\overline{\mathcal{C}}_{\chi_E}^1(\lambda_E^{-1})$ is torsion by Proposition \ref{ranksproposition}).
Hence we get an injective map 
$$\mathrm{rec} : U_2/(\Gamma_K-\lambda_E^{-2}(\Gamma_K)) \hookrightarrow \mathcal{X}_{\chi_E}(\lambda_E^{-1})/(\Gamma_K-\lambda_E^{-2}(\Gamma_K))$$
of 1-dimensional $K_p$-vector spaces, which is thus also surjective. This gives the assertion.

%By Proposition \ref{Selmertheorem}, we have
%$$\mathrm{dim}_{K_p}\left(\mathcal{X}_{\chi_E}(\lambda_E^{-1})/(\Gamma_K-\lambda_E^{-2}(\Gamma_K))\right) = 1.$$
%By the fact that $U_0 = \ker(\delta_E)$ (since $\delta_E$ equals $\delta|_{q_{\mathrm{dR}} = 1}$ up to some multiple), we have 
%$$\mathrm{dim}_{K_p}\left(\mathcal{X}_{\chi_E}(\lambda_E^{-1})/(\Gamma_K-\lambda_E^{-2}(\Gamma_K))\right)/\left(\mathrm{rec}(U_0)/(\Gamma_K-\lambda_E^{-2}(\Gamma_K))\right) = 1.$$
%Hence $\mathrm{rec}(U_0)/(\Gamma_K-\lambda_E^{-2}(\Gamma_K)) = 0$. Hence $\mathrm{rec}(U_2)/(\Gamma_K-\lambda_E^{-2}(\Gamma_K)) \neq 0$ by (\ref{vssurjection2}), and so $U_2/(\Gamma_K-\lambda_E^{-2}(\Gamma_K)) \neq 0$. 

%From the construction of $\mathbb{U}_E^1$, we have $\mathbb{U}_{\chi_E}^1/(\Gamma_K-\lambda_E^{-2}(\Gamma_K)) \cong K_p^{\oplus 2}$, and so $\mathbb{U}_{\chi_E}^1(\lambda_E^{-1})/(\Gamma_K-\lambda_E^{-2}(\Gamma_K)) \cong K_p(\lambda_E^{-1})^{\oplus 2}$. Recall $\mathbb{U}_{\chi_E}^1(\lambda_E^{-1}) = U_1 \oplus U_2$ with $U_1, U_2$ free. Then $U_i/(\Gamma_K-\lambda_E^{-2}(\Gamma_K))$ is a $K_p$-vector space of dimension at most 1. By the above each has dimension 1, and in particular $U_2/(\Gamma_K-\lambda_E^{-2}(\Gamma_K))$ has dimension 1. 

%The kernel of $\mathrm{rec}/(\Gamma_K-\lambda_E^{-2}(\Gamma_K))$ is $\overline{\mathcal{E}}_{\chi_E}^1(\lambda_E^{-1})/(\Gamma_K-\lambda_E^{-2}(\Gamma_K))$, and so since $U_2 \cap \overline{\mathcal{E}}_{\chi_E}^1(\lambda_E^{-1}) = 0$ we have $\mathrm{rec}(U_2)/(\Gamma_K-\lambda_E^{-2}(\Gamma_K))$ is a 1-dimensional $K_p$-vector space.
\end{proof}

\begin{corollary}\label{nonvanishingcorollary}Suppose $r_p(E/\mathbb{Q}) = 1$. We have 
$$\lambda_0(\lambda_E^{-2}) = \mathrm{char}_{\Lambda_{\mathcal{O}_{K_p}}[1/p]}\left(\mathcal{X}_{\chi_E}(\lambda_E^{-1})/\mathrm{rec}(U_2)\right)/(\Gamma_K-\lambda_E^{-2}(\Gamma_K)) \neq 0,$$
where $\lambda_0(\lambda_E^{-2})$ is written in the notation of Definition \ref{characterevaluation}, where $\lambda_E$ is as in (\ref{lambdaE}). 
\end{corollary}

\begin{proof}The equality follows from Corollary \ref{explicitcharacteristicidealcorollary} with our choice of $\lambda_1 = 1$ from Choice \ref{U2finallyfixed}. For the non-equality, we have an exact sequence
$$U_2 \xrightarrow{\mathrm{rec}} \mathcal{X}_{\chi_E}(\lambda_E^{-1}) \rightarrow \mathcal{X}_{\chi_E}(\lambda_E^{-1})/\mathrm{rec}(U_2) \rightarrow 0.$$
Reducing this modulo $(\Gamma_K -\lambda_E^{-2}(\Gamma_K))$, we get
$$U_2/(\Gamma_K-\lambda_E^{-2}(\Gamma_K)) \xrightarrow{\mathrm{rec}} \mathcal{X}_{\chi_E}(\lambda_E^{-1})/(\Gamma_K-\lambda_E^{-2}(\Gamma_K))\rightarrow \left(\mathcal{X}_{\chi_E}(\lambda_E^{-1})/\mathrm{rec}(U_2)\right)/(\Gamma_K-\lambda_E^{-2}(\Gamma_K))\rightarrow 0.$$
By Proposition \ref{nonvanishingU2}, we thus have 
\begin{equation}\label{XUzero}(\mathcal{X}_{\chi_E}(\lambda_E^{-1})/\mathrm{rec}(U_2))/(\Gamma_K-\lambda_E^{-2}(\Gamma_K)) = 0.
\end{equation}
Now since $\mathcal{X}_{\chi_E}(\lambda_E^{-1})/\mathrm{rec}(U_2)$ is $\Lambda_{\mathcal{O}_{K_p}}[1/p]$-torsion, we have an exact sequence of $\Lambda_{\mathcal{O}_{K_p}}[1/p]$-modules
$$0 \rightarrow \mathcal{M} \rightarrow \mathcal{X}_{\chi_E}(\lambda_E^{-1})/\mathrm{rec}(U_2) \rightarrow \bigoplus_i\frac{\Lambda_{\mathcal{O}_{K_p}}[1/p]}{(f_i)}  \rightarrow \mathcal{N} \rightarrow 0$$
where $\mathcal{M}$ and $\mathcal{N}$ are pseudonull. However, since $\mathcal{X}_{\chi_E}(\lambda_E^{-1})$ has no pseudonull submodules (\cite[Theorem 5.3(v)]{RubinMC}), and $\mathrm{rec}(U_2)$ is a free $\Lambda_{\mathcal{O}_{K_p}}[1/p]$-module of rank 1 (Proposition \ref{reciprocityU2}), then $\mathcal{X}_{\chi_E}(\lambda_E^{-1})/\mathrm{rec}(U_2)$ also has no pseudonull submodules. Thus $\mathcal{M} = 0$. Thus we have an exact sequence
$$0 \rightarrow \mathcal{X}_{\chi_E}(\lambda_E^{-1})/\mathrm{rec}(U_2) \rightarrow \bigoplus_i\frac{\Lambda_{\mathcal{O}_{K_p}}[1/p]}{(f_i)}  \rightarrow \mathcal{N} \rightarrow 0.$$
Now reducing modulo $(\Gamma_K-\lambda_E^{-2}(\Gamma_K))$, we have 
$$0 \overset{(\ref{XUzero})}{=} (\mathcal{X}_{\chi_E}(\lambda_E^{-1})/\mathrm{rec}(U_2))/(\Gamma_K-\lambda_E^{-2}(\Gamma_K)) \rightarrow \bigoplus_i\frac{K_p}{(f_i(\lambda_E^{-2}))}  \rightarrow \mathcal{N}/(\Gamma_K-\lambda_E^{-2}(\Gamma_K)) \rightarrow 0$$
where 
$$f_i(\lambda_E^{-2}) := f_i \pmod{(\Gamma_K-\lambda_E^{-2}(\Gamma_K))}.$$
Note that $\mathcal{N}$ is $\Lambda_{\mathcal{O}_{K_p}}[1/p]$-torsion, since the $(0)$-localization $\mathcal{N}_{(0)} = 0$ by pseudonullity. Thus $\mathcal{N}/(\Gamma_K-\lambda_E^{-2}(\Gamma_K))$ is $\Lambda_{\mathcal{O}_{K_p}}[1/p]/(\Gamma_K-\lambda_E^{-2}(\Gamma_K)) = K_p$-torsion, i.e. is 0. Thus from the previous exact sequence, we have
$$\bigoplus_i\frac{K_p}{(f_i(\lambda_E^{-2}))} = 0,$$
and so
$$\mathrm{char}_{\Lambda_{\mathcal{O}_{K_p}}[1/p]}\left(\mathcal{X}_{\chi_E}(\lambda_E^{-1})/\mathrm{rec}(U_2)\right)/(\Gamma_K-\lambda_E^{-2}(\Gamma_K)) = \prod_i f_i(\lambda_E^{-2}) \neq 0.$$
\end{proof}

\begin{theorem}[Rank 1 $p$-Converse Theorem]\label{BSDrank1theorem}Let $E/\mathbb{Q}$ be an elliptic curve with CM by $\mathcal{O}_K$ and suppose $p$ is ramified in $K$. Then
$$r_p(E/\mathbb{Q}) := \mathrm{corank}_{\mathbb{Z}_p}\mathrm{Sel}_{p^{\infty}}(E/\mathbb{Q}) = 1 \implies r_{\mathrm{an}}(E/\mathbb{Q}) = r_{\mathrm{alg}}(E/\mathbb{Q}) = 1 \hspace{.25cm} \text{and} \hspace{.25cm} \#\Sh(E/\mathbb{Q}) < \infty.$$ 
\end{theorem}

\begin{proof}Recall that $N^2$ is the prime-to-$p$ part of the conductor of $E/\mathbb{Q}$. When $f_0 \le 3$, there are finitely many isomorphism classes of elliptic curves for which to check the assertion. As remarked in the proof of Theorem \ref{BSDrank0theorem}, this has been done by the LMFDB collaboration, \cite{LMFDB}. 

Now assume $f_0 \ge 4$. As recalled in the proof of Proposition \ref{kernelsubspaceproposition}, by previous work on the parity conjecture, we have the global root number of $E/\mathbb{Q}$ is -1, and hence the global root number of $\lambda_E$ is -1. By Corollary \ref{nonvanishingcorollary}, we see that $\lambda_0(\lambda_E^{-2}) \neq 0$. In particular, the assumption (\ref{1notzero}) is satisfied. From (\ref{beta0nonzero}), we have $\delta'(\beta_0) \neq 0$. 
Since $\xi_E = \lambda_0\beta_0$, we see from (\ref{equivariance2}) and (\ref{explicitreciprocity2}), that 
$$C_E \cdot (\lambda_E^{-1}(\frak{a}) - \mathbb{N}\frak{a})\cdot D_1^{-1}\mathcal{L}_{\lambda_E}|_{q_{\mathrm{dR}} = 1} =  \delta'(\xi_E)|_{q_{\mathrm{dR}} = 1} = \lambda_0(\lambda_E^{-2})\delta'(\beta_0) = \lambda_0(\lambda_E^{-2})\delta'(\beta_0) \neq 0.$$
Now the statement follows from (\ref{Heegnerpointidentity2}).  
\end{proof}

\subsection{Sylvester's conjecture}\label{Sylvestersection}

Recall the cubic twist family of elliptic curves
$$E_d : x^3 + y^3 = d.$$
Sylvester conjectured in 1879 (\cite{Sylvester}) that for $d$ prime, if $d \equiv 4, 7, 8 \pmod{9}$ then $E_d$ has a solution over $\mathbb{Q}$. As an immediate corollary of Theorem \ref{BSDrank1theorem}, we prove this conjecture, as well as results on analytic rank. 

\begin{theorem}\label{Sylvesterthm}Suppose $r_3(E_d/\mathbb{Q}) \le 1$. Then 
\begin{equation}\label{cubicconclusion}r_{\mathrm{an}}(E_d/\mathbb{Q}) = r_{\mathrm{alg}}(E_d/\mathbb{Q}) = r_3(E_d/\mathbb{Q}) \hspace{.25cm} \text{and} \hspace{.25cm} \#\Sh(E_d/\mathbb{Q}) < \infty.
\end{equation}
Thus if $r_3(E_d/\mathbb{Q}) = 0$ then there do not exist $x, y \in \mathbb{Q}$ such that $x^3 + y^3 = d$, and if $r_3(E_d/\mathbb{Q}) = 1$ then there exist $x, y \in \mathbb{Q}$ such that $x^3 + y^3 = d$. 

In particular:
\begin{enumerate}
\item For any prime $d \equiv 2,5 \pmod{9}$, there do not exist $x,y \in \mathbb{Q}$ such that $x^3 + y^3 = d$. Moreover, $r_{\mathrm{an}}(E_d/\mathbb{Q}) = 0$. 
\item Sylvester's conjecture is true. That is, for any prime $d \equiv 4,7,8 \pmod{9}$, there exist $x,y \in \mathbb{Q}$ such that $x^3 + y^3 = d$. Moreover, $r_{\mathrm{an}}(E_d/\mathbb{Q}) = 1$. 
\end{enumerate}
\end{theorem}

\begin{remark}Previously, the $d \equiv 4,7 \pmod{9}$ case of Sylvester's conjecture was announced by Elkies \cite{Elkies}, though the full proof remains unpublished. See also the article of Dasgupta-Voight \cite{DasguptaVoight}, which gives another proof of Elkies's result under additional assumptions.
\end{remark}

\begin{proof}[Proof of Theorem \ref{Sylvesterthm}]Note that $E_d$ has CM by $\mathbb{Z}[\sqrt{-3}]$. Now since $r_3(E_d/\mathbb{Q}) \le 1$, (\ref{cubicconclusion}) follows immediately from Theorems \ref{BSDrank0theorem} and \ref{BSDrank1theorem} with $p = 3$ and $K = \mathbb{Q}(\sqrt{-3})$. 

By standard 3-descent (see \cite{DasguptaVoight}, \cite{Satge}), we have for primes $d \neq 3$,
$$r_3(E_d) := \mathrm{corank}_{\mathbb{Z}_3}(\mathrm{Sel}_{3^{\infty}}(E_d)) \le \begin{cases} 0 & d \equiv 2,5 \pmod{9}\\
1 & d \equiv 4,7,8 \pmod{9}\\
2 & d \equiv 1 \pmod{9}\\
\end{cases}.$$
Letting $w_{E_d}$ denote the sign of the functional equation of $L(E_d/\mathbb{Q},s)$, we have (see \cite{DasguptaVoight}) for primes $d \neq 3$,
$$w_{E_d} = \begin{cases} -1 & d \equiv 4,7,8 \pmod{9}\\
1 & \text{else}\\
\end{cases}.$$
Together with known results on the parity conjecture (\cite{Nekovar}, \cite{Kim}, \cite{DokchitserDokchitser}, and \cite{Dokchitser}) the above implies
\begin{equation}\label{r3}r_3(E_d) = \begin{cases} 0 & d \equiv 2,5 \pmod{9}\\
1 & d \equiv 4,7,8 \pmod{9}\\
\end{cases}.
\end{equation}

Now we can prove both assertions.\\

\textbf{(1)}: From the first case of (\ref{r3}), we see that $r_3(E_d) = 0$. Now the assertion follows from (\ref{cubicconclusion}) and the fact that $E_d(\mathbb{Q})_{\mathrm{tors}} = 0$ since $d > 2$ (see \cite{DasguptaVoight}).\\

\textbf{(2)}: From the second case of (\ref{r3}), we see that $r_3(E_d) = 1$. Now the assertion follows from (\ref{cubicconclusion}). 

\end{proof}

\subsection{Goldfeld's conjecture for the congruent number family}\label{Goldfeldsection}

For a general elliptic curve 
$$E : y^2 = x^3 + ax + b,$$ 
recall $r_{\mathrm{an}}(E/\mathbb{Q}) = \mathrm{ord}_{s =1}L(E/\mathbb{Q},s)$ denotes the analytic rank. For any integer $d$, let
$$E^d : y^2 = x^3 + ad^2x + bd^3$$
be the quadratic twist by $\mathbb{Q}(\sqrt{d})$. Then $E$ is isomorphic to $E^d$ over $\mathbb{Q}(\sqrt{d})$, but not over $\mathbb{Q}$ if $\mathbb{Q}(\sqrt{d}) \neq \mathbb{Q}$. The celebrated conjecture of Goldfeld states that:
\begin{conjecture}[Goldfeld's Conjecture \cite{Goldfeld}]\label{Goldfeldconjecture}For $r = 0,1$, 
$$\lim_{X \rightarrow \infty} \frac{\#\left\{0 < |d| < X :  d \; \text{squarefree}, \; r_{\mathrm{an}}(E^d/\mathbb{Q}) = r\right\}}{\#\left\{0 < |d| < X :  d \; \text{squarefree}\right\}} = \frac{1}{2}.$$
\end{conjecture}

To the author's best knowledge, the most general unconditional results towards Goldfeld's conjecture are \cite{Kriz2}, \cite{KrizLi}, \cite{KrizLi2}, \cite{CastellaGrossiLeeSkinner}, \cite{OnoSkinner}, \cite{PerelliPomykala} and \cite{Smith}. In particular, the results of op. cit. imply the ``$2^{\infty}$-Selmer analogue'' of Goldfeld's conjecture, in which the analytic ranks $r_{\mathrm{an}}$ are replaced with $r_2 = \mathrm{corank}_{\mathbb{Z}_2}\mathrm{Sel}_{2^{\infty}}$. As an immediate corollary of the results of op. cit. combined with Theorems \ref{BSDrank0theorem} and \ref{BSDrank1theorem}, we establish Goldfeld's conjecture for certain CM elliptic curves among those considered in op. cit., including the congruent number family. It also resolves the congruent number problem (Conjecture \ref{congruentnumberproblem}) in 100\% of cases.  

\begin{theorem}\label{congruentnumberthm}

Suppose $E/\mathbb{Q}$ is an elliptic curve with $E(\mathbb{Q})[2] \cong (\mathbb{Z}/2)^{\oplus 2}$ and no cyclic 4-isogeny defined over $\mathbb{Q}$. Suppose further that $E$ has CM by $K = \mathbb{Q}(i)$ or $\mathbb{Q}(\sqrt{-2})$. Then:
\begin{enumerate}
\item Goldfeld's conjecture (Conjecture \ref{Goldfeldconjecture}) is true for $E$, and $\#\Sh(E^d) < \infty$ for 100\% of squarefree $d$. 

\item Let $E^d : y^2 = x^3 - d^2x$. If $r_2(E^d/\mathbb{Q}) \le 1$, then 
\begin{equation}\label{congruentconclusion}r_{\mathrm{an}}(E^d/\mathbb{Q}) = r_{\mathrm{alg}}(E^d/\mathbb{Q}) = r_2(E^d/\mathbb{Q}) \hspace{.25cm} \text{and} \hspace{.25cm} \#\Sh(E^d/\mathbb{Q}) < \infty.
\end{equation}
In particular, if $r_2(E^d/\mathbb{Q}) = 0$ then $|d|$ is not a congruent number and if $r_2(E^d/\mathbb{Q}) = 1$ then $|d|$ is a congruent number. 

\item 100\% of squarefree positive integers $d \equiv 1,2, 3 \pmod{8}$ are not congruent numbers and 100\% of squarefree positive integers $d \equiv 5,6,7 \pmod{8}$ are congruent numbers, and Goldfeld's conjecture is true for the congruent number family $E^d: y^2 = x^3 - d^2x$.
\end{enumerate}
\end{theorem}

\begin{proof}The 2-descent result \cite[Theorem 1.2]{Smith} implies that for $r = 0$ or 1,
$$r_2(E^d/\mathbb{Q}) := \mathrm{corank}_{\mathbb{Z}_2}\mathrm{Sel}_{2^{\infty}}(E^d/\mathbb{Q}) = r$$
for $50\%$ of squarefree $d$ (ordered by $|d|$). By assumption, each $E^d$ has CM by $\mathcal{O}_K$. Now applying Theorem \ref{BSDrank0theorem} and Theorem \ref{BSDrank1theorem} for $p = 2$ to each $E^d$ with $r_2(E^d/\mathbb{Q}) = r$, we arrive at Goldfeld's conjecture for $E$ as well as the finiteness of $\#\Sh(E^d)$ for 100\% of squarefree $d$. 

For the congruent number family $E^d : y^2 = x^3 - d^2x$, which has CM by $\mathbb{Z}[i]$, if $r_2(E^d/\mathbb{Q}) \le 1$ then (\ref{congruentconclusion}) immediately follows from Theorems \ref{BSDrank0theorem} and \ref{BSDrank1theorem} with $p = 2$ and $K = \mathbb{Q}(i)$. As discussed in \cite{Tunnell} or \cite{Smith}, root number considerations show that
$$r_2(E^d/\mathbb{Q}) = 0 \hspace{.25cm} \text{for 100\% of squarefree positive $d \equiv 1,2,3 \pmod{8}$},$$
and
$$r_2(E^d/\mathbb{Q}) = 1 \hspace{.25cm} \text{for 100\% of squarefree positive $d \equiv 5,6,7 \pmod{8}$}.$$
Thus, Theorem \ref{BSDrank0theorem} implies 
$$r_{\mathrm{an}}(E^d/\mathbb{Q}) = r_{\mathrm{alg}}(E^d/\mathbb{Q}) = 0 \hspace{.25cm} \text{for 100\% of squarefree positive $d \equiv 1,2,3 \pmod{8}$},$$
and Theorem \ref{BSDrank1theorem} implies 
$$r_{\mathrm{an}}(E^d/\mathbb{Q}) = r_{\mathrm{alg}}(E^d/\mathbb{Q}) = 1 \hspace{.25cm} \text{for 100\% of squarefree positive $d \equiv 5,6,7 \pmod{8}$}.$$

By the discussion of \cite[Introduction]{Tunnell}, we have for all positive integers $d$,
$$r_{\mathrm{alg}}(E^d/\mathbb{Q}) = 0 \iff \; \text{$d$ is not a congruent number}.$$
Thus \emph{a fortiori}
$$r_{\mathrm{alg}}(E^d/\mathbb{Q}) = 0 \implies \text{$d$ is not a congruent number},$$
$$r_{\mathrm{alg}}(E^d/\mathbb{Q}) = 1 \implies \; \text{$d$ is a congruent number}.$$
The results on congruent numbers follow immediately from this and the discussion of the previous paragraph.

\end{proof}

%\begin{corollary}\label{congruentcorollary}Consider $E^d : y^2 = x^3 - d^2x$. Then $r_{\mathrm{an}}(E^d) = 0$ for 100\% of squarefree $d \equiv 1,2,3 \pmod{8}$, ordering by increasing $|d|$, and $r_{\mathrm{an}}(E^d) = 1$ for 100\% of squarefree $d \equiv 5,6,7 \pmod{8}$, ordering by increasing $|d|$. In particular, 100\% of squarefree positive integers $d \equiv 1,2,3 \pmod{8}$ are not congruent numbers and 100\% of squarefree positive integers $d \equiv 5,6,7 \pmod{8}$ are congruent numbers.
%\end{corollary}

%\begin{proof}

%\end{proof}

\section{Appendix: Elliptic Units Main Conjecture via the Equivariant Main Conjecture}\label{EMCsection}
In this Appendix, we prove Theorem \ref{EMC} (the rational elliptic units main conjecture). In \cite{RubinMC}, Rubin proves this conjecture in many cases, under relatively mild assumptions which are however not always satisfied in our situation (e.g. $p|\#\mathcal{O}_K^{\times}$). To circumvent the technical assumptions of Rubin's Euler system argument, we appeal to the results of Johnson-Leung and Kings (\cite{JohnsonLeungKings}) on \emph{equivariant} main conjectures. In particular, the results of Section 5 of op. cit. imply Theorem \ref{EMC} contingent on the vanishing of the total $\mu$-invariant of $\mathrm{Gal}(M_{\infty}'/K_{\infty}')$ for an appropriate extension $K_{\infty}'/K$ where $M_{\infty}'/K_{\infty}'$ is the maximal pro-$p$ abelian extension unramified outside all places above $\frak{p}$. Here, $K_{\infty}'/K$ being an appropriate extension means $\mathrm{Gal}(K_{\infty}'/K) \cong \mathbb{Z}_p \times \Delta'$ where $\Delta'$ finite abelian. Hence we are reduced to showing the vanishing of this $\mu$-invariant. 

In the $p$ split in $K$ situation with $\frak{p}$ the prime of $K$ above $p$ fixed by (\ref{fixembeddings}), the authors of op. cit. take the choice $K_{\infty}' = K(\frak{p}^{\infty})$ and invoke results of Robert (\cite{Robert}) on the vanishing of the total $\mu$-invariant. In the $p$ nonsplit situation, we take $K_{\infty}' = K(\mu_{f_0p^{\infty}})$ where $f_0 \in \mathbb{Z}$ is an $\mathcal{O}_K$-generator of $\frak{f}_E^{(p)}$, and introduce a new argument showing that the $\mu$-invariant of $\mathrm{Gal}(M_{\infty}'/K_{\infty}')$ is equal to the $\mu$-invariant of a product of Kubota-Leopoldt $p$-adic $L$-functions. This latter argument comes down to computing the $\mu$-invariant as the $p$-adic valuation of a sum of logarithms of elliptic units (i.e. proving an ``algebraic $p$-adic Kronecker limit formula'') following the valuation computation of \cite{CoatesWilesHurwitz}, and using Gross-Kersey's ``factorizing'' method from \cite{Gross} and \cite{Kersey} in order to show these logarithms of elliptic units are products of special values of Kubota-Leopoldt $p$-adic $L$-functions.\footnote{We note that Rubin previously extended Gross's argument to the $p$-ramified case in \cite[proof of Corollary 5]{Rubincongruence}.} Then the vanishing of the $\mu$-invariant follows from the theorem of Ferrero-Washington \cite{FerreroWashington} (see also \cite{Sinnott}), which says that the products of Kubota-Leopoldt $p$-adic $L$-functions appearing have vanishing $\mu$-invariant. In the next sections, we carry out this argument. 

\subsection{The cyclotomic algebraic $\mu$-invariant}
We continue the notation of the previous sections. In particular $K$ is an imaginary quadratic field with class number 1 in which $p$ is ramified. %As the argument applies to more general imaginary quadratic fields $K$ than ones with class number one, we will let $K$ be an arbitrary imaginary quadratic field in which $p$ is ramified until further notice. 

\begin{definition}\label{cyclotomicsetting}\begin{enumerate}
\item For the remainder of this section, let $f \in \mathbb{Z}_{> 0}$ be prime to $p$, let $n \in \mathbb{Z}_{\ge 0} \cup \{\infty\}$
$$K \subset L^+ \subset K(\mu_f), \hspace{1cm} L_n^+ = L^+(\mu_{p^n}), \hspace{1cm} \text{and} \hspace{1cm}L_{\infty}^+ = L^+(\mu_{p^{\infty}}).$$
Note that since $K(\mu_f) \subset \mathbb{Q}(\mu_{fp^{\infty}})$ by Assumption \ref{pramifiedassumption}, then $L^+/\mathbb{Q}$ is abelian. Henceforth choose and fix a decomposition
$$\mathrm{Gal}(L_{\infty}^+/K) \cong \Delta_+ \times \Gamma_+'$$
where $\Gamma_+' = \mathrm{Gal}(K_{\infty}^+/K)$ with $K_{\infty}^+$ the maximal $\mathbb{Z}_p$-subextension of $K(\mu_{p^{\infty}})$ and  $\Delta_+ = \mathrm{Gal}(L_{\infty}^+/K_{\infty}^+)$ is a finite abelian group.
\item Let 
$$q := \begin{cases} p & p > 2\\
 4 & p = 2
 \end{cases}, \hspace{1cm} s := \mathrm{ord}_p(q) = \begin{cases} 1 & p > 2\\
 2 & p = 2
 \end{cases}.$$
\item For $n \in \mathbb{Z}_{\ge 0} \cup \{\infty\}$, let $M_n^+$ be the maximal pro-$p$ abelian extension of $L_n^+$ which is unramified outside the places above $p$, let $N_n^+$ be the maximal pro-$p$ abelian extension of $L_n^+$ which is unramified everywhere. In particular, $L_{\infty}^+ \subset M_n^+$. Let $K_n^+/K$ be the unique subextension of $K_{\infty}^+/K$ with $\mathrm{Gal}(K_n^+/K) \cong \mathbb{Z}/p^n$. Let
$$\mathcal{X}^+ := \mathrm{Gal}(M_{\infty}^+/L_{\infty}^+) = \varprojlim_n\mathrm{Gal}(M_n^+/L_n^+).$$
As $L_{\infty}^+/\mathbb{Q}$ is an abelian extension of $\mathbb{Q}$ which is a finite extension of the cyclotomic $\mathbb{Z}_p$-extension of $\mathbb{Q}$, by \cite{GreenbergRank}, $\mathcal{X}^+$ is a torsion $\mathbb{Z}_p\llbracket\mathrm{Gal}(L_{\infty}^+/K)\rrbracket$-module.

\item Let 
$$\Gamma_+ := \mathrm{Gal}(L_{\infty}^+/L_s^+) \subset \mathrm{Gal}(L_{\infty}^+/L^+)$$
which is the maximal subgroup isomorphic to $\mathbb{Z}_p$. Since $K$ has class number 1, then $K_{\infty}^+ \cap K(1) = K$ (where $K(1)$ is the Hilbert class field of $K$) and so the natural restriction map
$$\Gamma_+ \rightarrow \Gamma_+'$$
is the identity. Then since $\Gamma_+^{p^n} = \mathrm{Gal}(L_{\infty}^+/L_{n+s}^+)$ for any $n \in \mathbb{Z}_{\ge 0} \cup \{\infty\}$,
\begin{equation}\label{finitelevel}(\mathcal{X}^+)_{\Gamma_+^{p^n}} = \mathrm{Gal}(M_{n+s}^+/L_{\infty}^+).
\end{equation}
%By the non-vanishing of the $p$-adic regulator (\cite{Brumer}), $\mathrm{Gal}(M_{n+s}^+/L_{\infty}^+)$ is finite. 
Note that $\mathcal{X}^+$ and $(\mathcal{X}^+)_{\Gamma_+^{p^n}}$ are $\mathbb{Z}_p\llbracket \mathrm{Gal}(L_{\infty}^+/K)\rrbracket$-modules. 

\item Given a group $G$, recall $\hat{G}$ denotes its $\mathbb{C}_p^{\times}$-valued character group per Definition \ref{isotypicdefinition}. For any $\chi \in \hat{\Delta}_+$, we can consider the maximal quotient of $\mathcal{X}^+$ through which $\Delta_+$ acts through $\chi$:
$$\mathcal{X}_{\chi}^+ := \mathcal{X}^+ \otimes_{\mathbb{Z}_p[\Delta_+], \chi} \mathbb{Z}_p[\chi]\llbracket \Gamma_+'\rrbracket,$$
where $\mathbb{Z}_p[\chi]$ denotes the extension of $\mathbb{Z}_p$ obtained by adjoining the values of $\chi$.
\end{enumerate}
\end{definition}
%For $0 \le n < \infty$ let 
%$$K \subset K_n^+ \subset K_{\infty}^+$$
%be the unique subfield with $\mathrm{Gal}(K_n^+/K) \cong \mathbb{Z}/p^n$. 

%Let $t \in \mathbb{Z}_{\ge 0}$ such that 
%$$K_{\infty}^+\cap K(1) = K_t^+$$
%(recalling $K(1)$ is the Hilbert class field of $K$) then we have a natural identification $\Gamma_+ = (\Gamma_+')^{p^t}$ given by restriction from $L_{\infty}^+$ to $K_{\infty}^+$. In particular, if the class number of $K$ is prime to $p$, then $\Gamma_+ = \Gamma_+'$. As $L_{\infty}^+/\mathbb{Q}$ is an abelian extension of $\mathbb{Q}$ which is a finite extension of the cyclotomic $\mathbb{Z}_p$-extension of $\mathbb{Q}$, by \cite{GreenbergRank}, $\mathcal{X}^+$ is a torsion $\mathbb{Z}_p\llbracket\mathrm{Gal}(L_{\infty}^+/K)\rrbracket$-module.

%\begin{remark}Note that $\mathcal{X}^+$ is \emph{not} the cyclotomic specialization of $\mathcal{X}$, as the latter would have rank at least 1. In fact, there is a natural map from the latter into the former. 
%\end{remark}

\subsection{Computation of algebraic cyclotomic $\mu$-invariant: a valuation calculation}\label{valuationsection}We follow the strategy of \cite{CoatesWilesHurwitz} for computing 
$$\mathrm{ord}_p(\#(\mathcal{X}^+)_{\Gamma_+^{p^n}}) = \mathrm{ord}_p(\#\mathrm{Gal}(M_{n+s}^+/L_{\infty}^+)).$$
%In this section, we will assume $p > 2$ (as is done in op. cit.), though with work this assumption can be removed.

We need some lemmas and propositions which are refinements of results from op. cit. We follow the strategy of \cite[Chapter III.2]{deShalit} throughout this section.

\begin{definition}[$p$-adic regulator]\label{regulatordefinition}Let $F/K$ be any abelian extension of degree $r$. Let $\sigma_1, \ldots,\sigma_r$ be the embeddings $F \hookrightarrow \mathbb{C}_p$. (Note that all embeddings induce the same embedding $K_p \hookrightarrow \mathbb{C}_p$, since there is only one prime above $p$.) Let $E$ be a subgroup of finite index in $\mathcal{O}_F^{\times}$, and choose generators $e_1,\ldots,e_{r-1}$ for $E/E_{\mathrm{tors}}$. Then the $p$-adic regulator of $E$ is defined to be 
$$R_p(E) = \mathrm{det}(\log \sigma_i(e_j))_{1 \le i,j\le r-1}.$$
This is well-defined up to sign because $\log\mathrm{Nm}_{F/K}(e) = 0$ for $e \in E$. As shorthand, let 
$$R_p(F) = R_p(\mathcal{O}_F^{\times}).$$
\end{definition}

\begin{theorem}[\cite{Brumer}]\label{Brumertheorem}Assume $F/K$ is abelian (as is the case in Definition \ref{regulatordefinition}). Then $R_p(E) \neq 0$.
\end{theorem}

\begin{definition}\label{Ddefinition} We follow the notation and presentation of \cite[Chapter III.2.4]{deShalit}. Fix an abelian extension $F/K$ of degree $r$ as in Definition \ref{regulatordefinition}. For each prime $\frak{P}$ of $F$ above $\frak{p}$ (the unique prime of $K$ above $p$), let $w_{\frak{P}}$ denote the number of $p$-power roots of unity in $F_{\frak{P}}$. 
\begin{enumerate}
\item Let 
$$\Phi = F\otimes_K K_{\frak{p}} = \prod_{\frak{P}|\frak{p}}F_{\frak{P}},$$
and let 
$$U = (\mathcal{O}_F \otimes_{\mathcal{O}_K}\mathcal{O}_{K_{\frak{p}}})^{\times,1} = \prod_{\frak{P}|\frak{p}}\mathcal{O}_{F_{\frak{P}}}^{\times,1}$$
be the group of principal units in $\Phi$. Then the $p$-adic logarithm gives a homomorphism
$$\log : U \rightarrow \Phi$$
whose kernel has order $\prod_{\frak{P}|\frak{p}}w_{\frak{P}}$, and whose image is an open subgroup of $\Phi$. 
\item Let $E \subset \mathcal{O}_F^{\times}$ be a subgroup of finite index, and let 
$$D = E\cdot \langle 1+q\rangle,$$
and let $\overline{D}$ be its closure in $\Phi^{\times}$, and $\langle \overline{D}\rangle$ the projection of $\overline{D}$ to $U$. 

\item We write $\Phi(F), U(F), E(F), D(F)$ and $\overline{D}(F)$ when we with to emphasize the particular $F$ underlying them. Let $\Delta_{\frak{p}}(F/K)$ denote the sum over $\frak{P}|\frak{p}$ of the relative discriminants of $F_{\frak{P}}/K_{\frak{p}}$. 

\item Given any number field $F$, let $h(F)$ denote its class number and let $h(F)[p^{\infty}]$ denote the $p$-primary part of the class number.
\end{enumerate}
\end{definition}

\begin{definition}Suppose we are in the situation of Definition \ref{Ddefinition} with $F = L_n^+$ as in Definition \ref{cyclotomicsetting}. Then letting $U_n = U(L_n^+)$ and $E_n = E(L_n^+)$, Artin reciprocity gives
$$U_n/\langle \overline{E}_n\rangle \cong \mathrm{Gal}(M_n^+/N_n^+).$$
Here the bar denotes $p$-adic closure. Recall that $N_n^+$ (resp. $M_n^+$) is the maximal pro-$p$ abelian extension unramified everywhere (resp. unramified outside places above $p$) of $L_n^+$. Note that $N_n^+ \cap L_{\infty}^+ = L_n^+$ and $L_{\infty}^+ \subset M_n$. Let $Y_n \subset U_n$ be the subgroup such that
\begin{equation}\label{CFTsubgroup}Y_n/\langle \overline{E}_n\rangle \cong \mathrm{Gal}(M_n^+/N_n^+L_{\infty}^+).
\end{equation}
\end{definition}
Recall that $L_n^+ = LK_n^+$, and $K_n^+ \subset K_{\infty}^+$ is such that $[K_n^+:K] = p^n$. Recall that $\frak{p}$ is the unique prime of $\mathcal{O}_K$ above $p$. Consider the norm map $\mathrm{Nm}_{L_n^+/K} : U_n \rightarrow 1+\frak{p}\mathcal{O}_{K_{\frak{p}}}$. 

\begin{lemma}[cf. Lemma III.2.6 of \cite{deShalit}]\label{CWlemma2}We have $Y_n = \ker\mathrm{Nm}_{L_n^+/K}|_{U_n}$.
\end{lemma}

\begin{proof}We first show the inclusion ``$\subset$''. If $u \in Y_n$, then since 
$$U_n/Y_n \cong \mathrm{Gal}(N_n^+L_{\infty}^+/N_n^+) \cong \mathrm{Gal}(L_{\infty}^+/L_n^+),$$
we have $(u,L_{\infty}^+/L_n^+) = 1$ (as usual, denoting the Artin symbol of an abelian extension $E'/E$ by $(\cdot,E'/E)$). Then by the functoriality of the Artin symbol, we get 
$$(\mathrm{Nm}_{L_n^+/K}(u),L_{\infty}^+/\mathbb{Q}) = (u,L_{\infty}^+/L_n^+)  = 1.$$ Hence the id\`{e}le $\mathrm{Nm}_{L_n^+/K}(u)$, which is 1 outside of $p$, is necessarily 1. The argument in reverse shows the inclusion ``$\supset$''.

\end{proof}

Let $D_n = D(L_n^+)$, and similarly let $\overline{D}_n = \overline{D}(L_n^+)$. For any $n \ge 1$, let
\begin{equation}\label{dndefinition}d(n) := \mathrm{ord}_p([L_n^+:L^+]) = \begin{cases}n-1 & p > 3\\
\max(n-2,0) & p = 2, K = \mathbb{Q}(i)\\
n-1 & p = 2, K = \mathbb{Q}(\sqrt{-2}),\; 1 \le n \le 2\\
n-2 & p = 2, K = \mathbb{Q}(\sqrt{-2}),\; n \ge 3
\end{cases}.
\end{equation}

For integers $\alpha, \beta$, let ``$p^{\alpha}||b$'' denote the relation ``$p^{\alpha}$ strictly divides $b$'' (i.e. $p^{\alpha}|b$ but $p^{\alpha+1} \nmid b$).

\begin{lemma}\label{CWlemma1}Assume that we are in the situation of Definition \ref{cyclotomicsetting}, i.e. with $F = L_n^+$. Let $p^e||[L^+:K]$, so that $p^{e+d(n)}||[L_n^+:K]$. Then we have $[\log(U_n) : \log(\langle \overline{D}_n\rangle)] < \infty$, and in fact for any $n \ge 1$,
$$\mathrm{ord}_p\left([\log(U_n) : \log(\langle \overline{D}_n\rangle)]\right) = \mathrm{ord}_p\left(\frac{qp^{e+d(n)}R_p(L_n^+)}{\sqrt{\Delta_{\frak{p}}(L_n^+/K)}} \cdot \prod_{\frak{P}|\frak{p}}(w_{\frak{P}}\mathbb{N}(\frak{P}))^{-1}\right),$$
where $\frak{P}$ runs over all primes of $L_n^+$ above $\frak{p}$.
\end{lemma}

\begin{proof}By Theorem \ref{Brumertheorem} and the abelianness of $L_n^+/K$ we have $R_p(L_n^+) \neq 0$. Now the argument of \cite[Lemma 8]{CoatesWilesHurwitz} goes through using $\varepsilon_d = 1+q$ in place of $\varepsilon_d = 1+p$. 
%The key input of the argument is Lemma 4 of op. cit., which by the remarks in the proof thereof, applies to $K_n^+/K$ since it is totally ramified at $\frak{p}$. 
\end{proof}

\begin{corollary}\label{CWcorollary}Retain the situation of Lemma \ref{CWlemma1}. Let $w(E)$ be the number of roots of unity in $E$. Then for any $n \ge 1$,
$$\mathrm{ord}_p\left([U_n : \langle \overline{D}_n\rangle]\right) = \mathrm{ord}_p\left(\frac{qp^{e+d(n)}R_p(L_n^+)}{w(L_n^+)\sqrt{\Delta_{\frak{p}}(L_n^+/K)}}\cdot\prod_{\frak{P}|\frak{p}}\mathbb{N}(\frak{P})^{-1}\right),$$
where $\frak{P}$ runs over all primes of $L_n^+$ above $\frak{p}$.
\end{corollary}

\begin{proof}The assertion is an immediate consequence of Lemma \ref{CWlemma1} using the snake lemma, see \cite[Lemma 9]{CoatesWilesHurwitz} for details.
\end{proof}

%\begin{lemma}\label{CWlemma2}We have
%$$Y_n = \ker(\mathrm{Nm}_{L_n^+/K}|_{U_n}).$$
%\end{lemma}

%\begin{proof}This is the same argument as in \cite[Lemma 5]{deShalit}. The key input is again Lemma 4 of op. cit., which again by the remarks in the proof of that lemma applies to $K_n^+/K$ since it is totally ramified at $\frak{p}$. 
%\end{proof}

\begin{proposition}[cf. Proposition III.2.7 of \cite{deShalit}]Retain the notation of Lemma \ref{CWcorollary}. We have for all $n \gg 0$
\begin{equation}\label{orderproposition}\mathrm{ord}_p\left([M_n^+:L_{\infty}^+]\right) = \mathrm{ord}_p\left(\frac{qp^{d(n)}h(L_n^+)R_p(L_n^+)}{w(L_n^+) \sqrt{\Delta_{\frak{p}}(L_n^+/K)}}\cdot \prod_{\frak{P}|\frak{p}}(1-\mathbb{N}(\frak{P})^{-1})\right),
\end{equation}
where the product on the right-hand side runs over primes of $L_n^+$ above $\frak{p}$. 
\end{proposition}

\begin{proof}Let $D_n = E_n \cdot \langle 1+q\rangle$. Then since $\mathrm{Nm}_{L_n^+/K}(E_n) = 1$ and $e+d(n) = \mathrm{ord}_p([L_n^+:K])$, we have $\mathrm{Nm}_{L_n^+/K}(\langle \overline{D}_n\rangle) = 1 + qp^{e+d(n)}\mathcal{O}_{K_{\frak{p}}}$. Now we have a diagram
\begin{equation}\label{normdiagram}
\begin{tikzcd}[column sep = large]
     &0 \arrow{r}  & \langle \overline{E}_n\rangle  \arrow{r} \arrow{d}& \langle \overline{D}_n\rangle \arrow{r}{\mathrm{Nm}_{L_n^+/K}} \arrow{d}{}& 1 + qp^{e+d(n)}\mathcal{O}_{K_{\frak{p}}} \arrow{r} \arrow{d} & 0\\
     & 0 \arrow{r} & Y_n \arrow{r}{} & U_n \arrow{r}{\mathrm{Nm}_{L_n^+/K}} & 1+qp^{d(n)}\mathcal{O}_{K_{\frak{p}}} \arrow{r}& 0
\end{tikzcd},
\end{equation}
where the horizontal rows are exact and the vertical arrows are injective, by the above discussion and Lemma \ref{CWlemma2}. Here, the exactness at the fourth term in the bottom row follows from observing that $L_n^+/L^+$ is totally ramified at all primes above $\frak{p}$, and $L^+/K$ is unramified above $\frak{p}$ since $L \subset K(\mu_f)$, $(f,p) = 1$; thus the ramification index of $\frak{p}$ in $L_n^+/K$ is $d(n) = [L_n^+:L^+]$. Since $\mathrm{ord}_p([L_n^+:K]) = e+d(n)$, (\ref{orderproposition}) follows from Lemma \ref{CWcorollary}, (\ref{CFTsubgroup}), the identity
$$[M_n^+:L_{\infty}^+] = [M_n^+:N_n^+L_{\infty}^+][N_n^+L_{\infty}^+:L_{\infty}^+],$$
and the fact that 
$$[N_n^+L_{\infty}^+:L_{\infty}^+] = [N_n^+:L_n^+] = h(L_n^+)[p^{\infty}].$$ 

\end{proof}

\begin{definition}\label{characteristicidealsdefinition}Recall $\Gamma_+'= \mathrm{Gal}(K_{\infty}^+/K) \cong \mathbb{Z}_p$ and $\Delta_+ = \mathrm{Gal}(L_{\infty}^+/K_{\infty}^+)$ as well as our fixed decomposition $\mathrm{Gal}(L_{\infty}^+/K) = \Gamma_+' \times \Delta_+$. For $\chi \in \hat{\Delta}_+$, recall the torsion $\mathbb{Z}_p[\chi]\llbracket \Gamma_+'\rrbracket$-module $(\mathcal{X}^+)_{\chi}$. We let
$$f_{\chi}^+ := \mathrm{char}_{\mathbb{Z}_p[\chi]\llbracket \Gamma_+'\rrbracket}(\mathcal{X}^+_{\chi}).$$%%%, \hspace{1cm} g_{\chi} := \mathrm{char}_{\Lambda(\Gamma',\mathcal{O}_{L_p,\chi})}(\mathbb{U}_{\chi}').$$
%In light of (\ref{globalwithLvalues}), we have 
%$$g_{\chi} = \mu_{\mathrm{glob}}(\frak{g};\chi).$$
%%$$g_{\chi} = \begin{cases} \mu_{\mathrm{glob}}(\frak{g};\chi)\Lambda(\Gamma',\mathcal{O}_{L_p,\chi}) & \chi \neq 1\\
%%(\gamma_0-1)\mu_{\mathrm{glob}}(1;1) & \chi = 1\\
%%\end{cases}.$$
%We consider the specializations, i.e. the images under the maps $\mathcal{X}' \rightarrow (\mathcal{X}')_{\chi}^+$ and $\mathbb{U}' \rightarrow (\mathbb{U}')_{\chi}^+$, which we denote by $f_{\chi}^+$ and $g_{\chi}^+$, respectively.
Denote the $\mu$-invariant by $\mu_{\mathrm{inv}}(f_{\chi}^+)$ and the $\lambda$-invariant by $\lambda_{\mathrm{inv}}(f_{\chi}^+)$.
\end{definition}

\begin{corollary}\label{algebraicasymptoticcorollary}We have for all $n \gg 0$,
\begin{equation}\label{algebraicasymptotic}\sum_{\chi \in \hat{\Delta}_+}\mu_{\mathrm{inv}}(f_{\chi}^+)\cdot (p-1)p^{n-1} + \sum_{\chi \in \hat{\Delta}_+}\lambda_{\mathrm{inv}}(f_{\chi}^+) = 1 + \mathrm{ord}_p\left(\frac{hR_p}{w\sqrt{\Delta}}(L_n^+/K) / \frac{hR_p}{w\sqrt{\Delta}}(L_{n-1}^+/K)\right).
\end{equation}
\end{corollary}

\begin{proof}By (\ref{orderproposition}), the right-hand side of (\ref{algebraicasymptotic}) is $\mathrm{ord}_p([M_n^+:L_{\infty}^+]/[M_{n-1}^+:L_{\infty}^+])$ for all $n \gg 0$. Now (\ref{algebraicasymptotic}) follows from (\ref{finitelevel}), the Weierstrass preparation theorem (using the standard definitions of $\mu_{\mathrm{inv}}$ and $\lambda_{\mathrm{inv}}$, see \cite[Chapter III.2.1, Lemma III.2.2]{deShalit}) and (\ref{dndefinition}). 
\end{proof}

\subsection{Computation of the analytic cyclotomic $\mu$-invariant}We continue to follow the strategy of \cite[Chapter III.2]{deShalit}.

%\begin{lemma}For any finite order character $\nu \in \hat{\Gamma}_+'$, define $m_{\nu} = r$ if $\nu((\Gamma_+')^{r^s}) = 1$ but $\nu((\Gamma_+')^{p^{r-1}}) \neq 1$. Retaining the notation of Corollary \ref{algebraicasymptoticcorollary}, we have for all $n \gg 0$, 
%\begin{equation}\label{auxiliaryanalyticasymptotic}\sum_{\chi \in \hat{\Delta}_+'}\mu_{\mathrm{inv}}(f_{\chi}^+)\cdot p^{t+n-1}(p-1) + \sum_{\chi \in \hat{\Delta}_+'}\lambda_{\mathrm{inv}}(f_{\chi}^+) = \mathrm{ord}\left(\prod_{\nu, m_{\nu} = t+n}\nu(g_{\chi}^+)\right).
%\end{equation}
%\end{lemma}

%\begin{proof}This is just \cite[Lemma III.2.9]{deShalit}.
%\end{proof}

\begin{definition}\label{Kroneckernotation}\begin{enumerate}
\item For any character $\varepsilon \in \widehat{\mathrm{Gal}(L_{\infty}^+/K)}$, let $\frak{g} = \mathrm{cond}(\varepsilon)$, and let $g$ be the smallest positive integer in $\frak{g}$. Then define (cf. \cite[III.2.10 (14)]{deShalit})
$$S_p(\varepsilon) := -\frac{1}{12gw_{\frak{g}}}\sum_{c \in \mathcal{C}\ell(\frak{g})}\varepsilon^{-1}(c)\log \varphi_{\frak{g}}(c)$$
where $\varphi_{\frak{g}}(c)$ is the Robert invariant as in II.2.6 (17) of op. cit. 
%Let $f'$ be the smallest positive generator of the conductor over $\mathbb{Q}$ of the field $L^+ \subset \mathbb{Q}(\mu_f)$ fixed in Definition \ref{cyclotomicsetting} (recall that $L^+/\mathbb{Q}$ is abelian). 
\item Let $f' > 0$ be the smallest positive integer such that $L^+ \subset K(\mu_{f'})$. Thus $f'|f$, $f'$ is prime to $p$, $L_n^+ \subset K(\mu_{f'p^n})$, and $f'p^n$ is the conductor of $L_n^+$ over $\mathbb{Q}$ for all $n \gg 0$. Fix $(\zeta_{f'p^n})_n$ a compatible system of $f'p^n\mathrm{th}$ roots of unity, and for a character $\varepsilon \in \widehat{\mathrm{Gal}(L_{\infty}^+/K)}$ of conductor dividing $f'p^n$ define the Gauss sum by
$$g(\varepsilon) := \frac{1}{f'p^n}\sum_{\sigma \in \mathrm{Gal}(L_n^+/K)}\varepsilon(a)\zeta_{f'p^n}^{-a}.$$
Let
$$S_n := \{\varepsilon \in \widehat{\mathrm{Gal}(L_{\infty}^+/K)} : p^n||\mathrm{cond}(\varepsilon)\}.$$
%where the conductor is computed viewing $\varepsilon : \mathrm{Gal}(K(\frak{f}^{(p)}p^{\infty})/K) \rightarrow \overline{\mathbb{Q}}_p^{\times}$ (so that $\Gamma_+$ is the image of $\Gamma \subset \mathrm{Gal}(K(\frak{f}^{(p)}p^{\infty})/K)$ under the restriction from $K(\frak{f}^{(p)}p^{\infty})$ to $L_{\infty}^+ = K(\mu_{f'p^{\infty}})$, and $\Gamma_+ = (\Gamma_+')^{p^t}$). 
\end{enumerate}
\end{definition}

\begin{proposition}In the notation of Definition \ref{Kroneckernotation}, for all $n \gg 0$ we have
\begin{equation}\label{analyticorder}\mathrm{ord}_p\left(\prod_{\varepsilon \in S_n}g(\varepsilon)S_p(\varepsilon)\right) = \mathrm{ord}_p\left(\frac{hR_p}{w\sqrt{\Delta_{\frak{p}}}}(L_n^+)/\frac{hR_p}{w\sqrt{\Delta_{\frak{p}}}}(L_{n-1}^+)\right).
\end{equation}
\end{proposition}

\begin{proof}This follows from the same argument as in \cite[Chapter III.2.11]{deShalit}, which is a standard calculation using the analytic class number formula. 
\end{proof}

\subsection{Relating to Kubota-Leopoldt $p$-adic $L$-functions}We will show that the total $\mu$-invariant of $\mathcal{X}^+$ is 0 by showing it is equal to the total $\mu$-invariant of a product of Kubota-Leopoldt $p$-adic $L$-functions, which is 0 by the fundamental result of Ferrero-Washington (\cite{FerreroWashington}, see also \cite{Sinnott}). 
%\begin{remark}We note that this vanishing of the $\mu$-invariant on the cyclotomic line is a novel result in and of itself, and traditionally much less-well understood for classical $p$-adic $L$-functions (including for the ordinary analogue of our $p$-adic $L$-function, i.e. the Katz $p$-adic $L$-function). In \cite{Robert} and \cite{OukhabaViguie}, the $\mu$-invariant for the Katz $p$-adic $L$-function on the Coates-Wiles line is studied and naturally falls under the purview of Sinnott's method, while in our setting the cyclotomic line is the natural line of interest. This is because the $p$-adic $L$-function is determined by a measure on the unique $\mathbb{Z}_p$-line contained in the Lubin-Tate tower in the ordinary case, whereas it is determined by a measure on the \emph{diagonal} $\mathbb{Z}_p$-line in the Lubin-Tate tower in our case, which in view or the decomposition (\ref{canonicalsplit}), corresponds to the lift of the image of the norm map $1+p^{\lceil \varepsilon\mathrm{ord}_p(\frak{p}) \rceil}\mathbb{Z}_p \subset \Gamma$, which is the cyclotomic line. We can view the ordinary (height 1) case also in this optic, where the diagonal $\mathbb{Z}_p$ line in the (rank 1) Lubin-Tate tower is the unique $\mathbb{Z}_p$-line.
%\end{remark}
The following Proposition relates the quantities $S_p(\varepsilon)$ to a special value of a product of Kubota-Leopoldt $p$-adic $L$-functions. This is reminiscent of a theorem of Gross \cite{Gross} and Kersey \cite{Kersey} on the factorization of the cyclotomic Katz $p$-adic $L$-function into a product of two Kubota-Leopoldt $p$-adic $L$-functions.

\begin{definition}Recall $\varepsilon \in \widehat{\mathrm{Gal}(L_{\infty}^+/K)}$. As $L_{\infty}^+/\mathbb{Q}$ is an abelian extension, there are two characters 
$$\varepsilon_+, \varepsilon_+\eta \in \widehat{\mathrm{Gal}(L_{\infty}^+/\mathbb{Q})}$$
restricting to $\varepsilon$ on $\mathrm{Gal}(L_{\infty}^+/K) \subset \mathrm{Gal}(L_{\infty}^+/\mathbb{Q})$, recalling here that $\eta$ is the quadratic character attached to $K/\mathbb{Q}$ (\ref{etaK}). Without loss of generality, let $\varepsilon_+$ denote the even extension, i.e. with $\varepsilon_+(-1) = 1$ and let $\varepsilon_- = \varepsilon_+\eta$. 
\end{definition}

\begin{proposition}\label{KLproposition}For all $n \gg 0$ and any $\varepsilon \in S_n$, we have
$$g(\varepsilon)S_p(\varepsilon) = \frac{L_p(\varepsilon_-^{-1}\omega,0)}{2}\frac{L_p(\varepsilon_+,1)}{2},$$
where $L_p(\chi,s)$ denotes the Kubota-Leopoldt $p$-adic $L$-function (see \cite[Chapter 5]{Washington}). 
\end{proposition}

\begin{proof}We note that for all $n \gg 0$, 
$$S_n \subset \widehat{\mathrm{Gal}(L_n^+/K)}.$$
Letting $\varepsilon \in S_n$ and $\frak{g} = \mathrm{cond}(\varepsilon)$ we have $w_{\frak{g}} = 1$ for all $n \gg 0$. By the same argument as in  \cite[proof of Theorem 3.1, (3.6)-(3.7)]{Gross} using the Leopoldt formula for the value $L'(\chi,0)$ of a Dirichlet character $\chi$ (\cite{KatzImQuad}) and Kronecker limit formula (\cite{Stark}), we have
$$g(\varepsilon)S_p(\varepsilon) = -f'p^nB_{1,\varepsilon_-}\frac{g(\varepsilon_+)}{4}\sum_{a \in \mathcal{G}}\varepsilon_+^{-1}(a)\log_p(C^+(a)),$$
where 
$$C^+(a) = (1-\zeta_{f'p^n}^a)(1-\zeta_{f'p^n}^{-a})$$
is the cyclotomic unit of \cite{Gross} and $(\mathbb{Z}/(f'p^n))^{\times} \twoheadrightarrow \mathcal{G}$ is the quotient such that $\mathcal{G} \cong \mathrm{Gal}(L_n^+/\mathbb{Q})$ via the Artin reciprocity map. Alternatively, the above identity follows from \cite[Theorem 2.1]{Kersey}. The right-hand side of the above equation is, by the $p$-adic Leopoldt formula and standard interpolation property for Kubota-Leopoldt $p$-adic $L$-functions (\cite{FerreroGreenberg}), equal to 
$$\frac{L_p(\varepsilon_-^{-1}\omega,0)}{2}\frac{L_p(\varepsilon_+,1)}{2}.$$
\end{proof}

\subsection{Vanishing of the cyclotomic $\mu$-invariant}

\begin{corollary}\label{analyticmuvanishing}Recall the notation of Definition \ref{characteristicidealsdefinition}. Then for any $\chi \in \hat{\Delta}_+$, we have
$$\mu_{\mathrm{inv}}(f_{\chi}^+) = 0.$$
As a consequence, the total $\mu$-invariant 
$$\mu_{\mathrm{inv}}(\mathcal{X}^+) := \sum_{\chi \in \hat{\Delta}_+}\mu_{\mathrm{inv}}(f_{\chi}^+) = 0.$$
\end{corollary}

\begin{proof}This follows from (\ref{algebraicasymptotic}), (\ref{analyticorder}) and Proposition \ref{KLproposition} and the vanishing of the $\mu$-invariant of all nonzero branches of half the Kubota-Leopoldt $p$-adic $L$-function $\frac{1}{2}L_p$ first established by Ferrero-Washington \cite{FerreroWashington} (see also \cite{Sinnott} for another proof).

%This follows since by construction,
%$$(\sigma_{\frak{b}}  - \mathbb{N}(\frak{b}))g_{\chi}^+ = \chi((\sigma_{\frak{b}}  - \mathbb{N}(\frak{b}))\mu_{\mathrm{glob}}(\frak{g};\chi)^+) = \chi((\sigma_{\frak{b}} - \mathbb{N}(\frak{b}))\mu_{\mathrm{glob}}^+(\xi_{\frak{b}})),$$
%and the right-hand side has $\mu_{\chi}(\mu_{\mathrm{glob}}^+(\xi_{\frak{b}})) = 0$ by (\ref{muzero}).

\end{proof}

\subsection{Proof of the rational Elliptic Units Main Conjecture}\label{MCproofsection}In this section, we prove Theorem \ref{EMC} (when $p$ is ramified in $K/\mathbb{Q}$ and $K$ has class number 1) by invoking \cite[Theorem 5.7]{JohnsonLeungKings}, using as input the vanishing of the cyclotomic algebraic $\mu$-invariant. Recall in our setting (see Choice \ref{f0choice}) that $E/\mathbb{Q}$ has CM by $\mathcal{O}_K$, with $\frak{f}_E = \frak{f}(\lambda_E)$ and that $\frak{f}_E^{(p)} = f_0\mathcal{O}_K$ for some integer $f_0 \in \mathbb{Z}_{\ge 4}$. Recall our notation $\mathcal{K}_n = K(E[\frak{p}^n])$ (Definition \ref{mathcalKndefinition}). 

First, a quick calculation using the fact that $K$ is one of the imaginary quadratic fields listed in Assumption \ref{pramifiedassumption} shows that 
\begin{equation}\label{deltadecomposition}\mathrm{Gal}(\mathcal{K}_{\infty}/K) \cong \mathrm{Gal}(\mathcal{K}_{\infty}/K(\mu_{p^{\infty}})) \times \mathrm{Gal}(K(\mu_{p^{\infty}})/K), \hspace{1cm} \mathrm{Gal}(\mathcal{K}_{\infty}/K(\mu_{p^{\infty}})) \cong \Delta \times \mathcal{H}
\end{equation}
for some finite abelian group $\Delta$ and subgroup 
$$\mathbb{Z}_p \cong \mathcal{H} \subset \mathrm{Gal}(\mathcal{K}_{\infty}/K(\mu_{p^{\infty}})).$$
Henceforth fix such isomorphisms as in (\ref{deltadecomposition}) that are compatible with each other under the inclusion 
$$\mathrm{Gal}(\mathcal{K}_{\infty}/K(\mu_{p^{\infty}})) \subset \mathrm{Gal}(\mathcal{K}_{\infty}/K).$$
Thus,
$$\mathbb{Z}_p\llbracket \mathrm{Gal}(\mathcal{K}_{\infty}/K(\mu_{p^{\infty}}))\rrbracket \cong \mathbb{Z}_p[\Delta]\llbracket X\rrbracket.$$
Let $\Delta[p^{\infty}] \subset \Delta$ denote the $p$-power torsion subgroup and consider the prime-to-$p$ quotient
$$\Delta' := \Delta/\Delta[p^{\infty}].$$
Then (\ref{deltadecomposition}) shows that $\Delta$ is abstractly a subgroup of $\mathcal{O}_K^{\times}$, so since the latter is isomorphic to either $\mathbb{Z}/2, \mathbb{Z}/4$ or $\mathbb{Z}/6$ we have a decomposition
\begin{equation}\label{Deltadecomposition}\Delta \cong \Delta[p^{\infty}] \times \Delta',
\end{equation}
which we henceforth fix. Moreover, $\Delta'$ is either trivial or isomorphic to $\mathbb{Z}/2$ and in the latter case we have $p > 2$.

Before beginning the proof of Theorem \ref{EMC}, we collect a few general lemmas on $\mathcal{O}[\Delta]\llbracket X\rrbracket$-modules, where $\mathcal{O}/\mathbb{Z}_p$ is a finite extension. Give $\mathcal{O}[\Delta]\llbracket X\rrbracket$ the inverse limit topology as in \cite[Chapter 5]{Neukirch}. That is, the topology is generated by 
$$\{\varpi^k\mathcal{O}[\Delta]\llbracket X\rrbracket + (X^n)\}_{n,k},$$
where $\varpi$ is a uniformizer of $\mathcal{O}$. 

\begin{lemma}\label{maximalideallemma}Suppose $\frak{M}$ is an open maximal ideal of $\mathcal{O}[\Delta]\llbracket X\rrbracket$. Then $\frak{M}$ contains $X,(\varpi)$ and $\delta-1$ for any $\delta \in \Delta[p^{\infty}]$.
\end{lemma}

\begin{proof}By standard results on commutative algebra on maximal ideals of power series ring, we have $X \in \frak{M}$ and $\mathbf{m} = \frak{M} \cap \mathcal{O}[\Delta]$ is a maximal ideal of $\mathcal{O}[\Delta]$. Since $\frak{M}$ is open and $X \in \frak{M}$, we have $\varpi^k \in \frak{M}$ for some $k \ge 0$. Thus, $\varpi \in \frak{M}$, and so $\varpi\mathcal{O}[\Delta] \subset \mathbf{m}$. Thus $\mathbf{m}$ maps to a maximal ideal $\mathbf{m}'$ under 
$$\mathcal{O}[\Delta] \rightarrow \mathcal{O}/(\varpi)[\Delta] = \mathbb{F}_{p^n}[\Delta]$$
for some $n \ge 1$. Now we show that $\mathbf{m}'$ contains $\delta -1$ for any $\delta \in \Delta[p^{\infty}]$. The quotient $\mathbb{F}_{p^n}[\Delta]/\mathbf{m}'$ must be a finite extension $\mathbf{k}$ of $\mathbb{F}_{p^n}$, and the quotient $\mathbb{F}_{p^n}[\Delta] \rightarrow \mathbf{k}$ must map $\Delta \rightarrow \mathbf{k}^{\times}$. Since $\mathbf{k}^{\times}$ has order prime to $p$, this last map maps $\Delta[p^{\infty}]$ to 1. Thus $\delta-1$ is in its kernel, $\mathbf{m}'$, for any $\delta \in \Delta[p^{\infty}]$. This finishes the proof. 

\end{proof}

\begin{lemma}\label{reduceM'lemma}Suppose $M$ is a topological $\mathcal{O}[\Delta]\llbracket \mathbb{Z}_p\rrbracket$-module. Then $M$ is finitely generated if and only if $M \otimes_{\mathcal{O}[\Delta]\llbracket \mathbb{Z}_p\rrbracket}\mathcal{O}[\Delta']\llbracket X\rrbracket$ is a finitely generated $\mathcal{O}[\Delta']\llbracket \mathbb{Z}_p\rrbracket$-module. 
\end{lemma}

\begin{proof}$\impliedby$ : Let 
$$\mathrm{rad} \subset \mathcal{O}[\Delta]\llbracket \mathbb{Z}_p\rrbracket$$
be the radical, i.e. the intersection of all open maximal ideals of $\mathcal{O}[\Delta]\llbracket \mathbb{Z}_p\rrbracket$ (which is itself a closed ideal). By the Nakayama lemma for completed group rings (\cite[5.2.18]{Neukirch}), $M$ is finitely generated if any only if $M/(\mathrm{rad}\cdot M)$ is finitely generated.  By Lemma \ref{maximalideallemma}, 
$$\mathcal{I} := (X,\delta)_{\delta \in \Delta[p^{\infty}]} \subset \mathrm{rad},$$
and so 
$$M \otimes_{\mathcal{O}[\Delta]\llbracket X\rrbracket}\mathcal{O}[\Delta']\llbracket X\rrbracket = M/\mathcal{I}M \twoheadrightarrow M/(\mathrm{rad}\cdot M).$$
Now the assertion follows from the assumption.\\

$\implies$: This follows immediately from the surjection $M \twoheadrightarrow M \otimes_{\mathcal{O}[\Delta]\llbracket X\rrbracket}\mathcal{O}[\Delta']\llbracket X\rrbracket$ and the finite generatedness of $M$.

\end{proof}

%Let $\mathcal{O}'$ be the finite extension of $\mathcal{O}$ generated by the values of all characters $\chi \in \hat{\Delta}'$. Note that since $p \nmid \#\Delta'$, the $\Delta'$-isotypic decomposition gives
%$$\mathcal{O}'[\Delta']\llbracket X\rrbracket = \bigoplus_{\chi \in \hat{\Delta}'}\mathcal{O}'[\Delta']\llbracket X\rrbracket \otimes_{\mathcal{O}'[\Delta'],\chi}\mathcal{O}'$$
%where the tensor product is defined through the map $\chi : \Delta' \rightarrow \mathcal{O}'^{\times}$. Thus an $\mathcal{O}'[\Delta']\llbracket X\rrbracket$-module $M'$ is finitely generated if and only if each $M_{\chi}'$ is finitely generated as an $\mathcal{O}'\llbracket X\rrbracket$-module.

Write
\begin{equation}\label{M}M := H^2(\mathcal{O}_K[1/fp],\mathcal{O}_{K_p}\llbracket \mathrm{Gal}(\mathcal{K}_{\infty}/K)\rrbracket(1)).
\end{equation}
Let 
\begin{equation}\label{Linfty}L_{\infty} = \mathcal{K}_{\infty}^{\Delta[p^{\infty}]}
\end{equation}
be the fixed field of $\Delta[p^{\infty}]$ under (\ref{deltadecomposition}). Then by \cite[1.6.5 (3)]{FukayaKato}, we have 
\begin{equation}\label{M'}M' := M \otimes_{\mathcal{O}_{K_p}[\Delta]\llbracket X\rrbracket}\mathcal{O}_{K_p}[\Delta']\llbracket X\rrbracket \cong H^2(\mathcal{O}_K[1/fp],\mathcal{O}_{K_p}\llbracket \mathrm{Gal}(L_{\infty}/K)\rrbracket (1)).
\end{equation}
%Finally, write
%$$M' = \bigoplus_{\chi \in \hat{\Delta}'}M_{\chi}', \hspace{1cm} M_{\chi}' = M' \otimes_{\mathcal{O}_{K_p}[\Delta'],\chi}\mathcal{O}_{K_p}.$$

Recall 
$$\Delta' = \Delta/\Delta[p^{\infty}]$$
is either trivial or $\mathbb{Z}/2$ and in the latter case we have $p \neq 2$. Recall $\mathcal{H} \cong \mathbb{Z}_p$ from (\ref{deltadecomposition}). %Let $\tilde{\Gamma} = \mathrm{Gal}(K(\mu_{p^{\infty}})/K(\mu_{p^s})) \subset \mathrm{Gal}(K(\mu_{p^{\infty}})/K)$, so that $\tilde{\Gamma}$ is a maximal subgroup of $\mathrm{Gal}(K(\mu_{p^{\infty}})/K)$ isomorphic to $\mathbb{Z}_p$. 
Then using (\ref{deltadecomposition}) and (\ref{Deltadecomposition}), the subgroup 
$$\mathcal{H} \times \Delta[p^{\infty}] \times \mathrm{Gal}(K(\mu_{p^{\infty}})/K) \subset \mathrm{Gal}(\mathcal{K}_{\infty}/K)$$
has fixed field
$$L := \mathcal{K}_{\infty}^{\mathcal{H} \times \Delta[p^{\infty}] \times \mathrm{Gal}(K(\mu_{p^{\infty}})/K)},$$
is an extension of $K$ of degree at most 2 and $\mathrm{Gal}(L/K) \cong \Delta'$. Note that $\mathcal{K}_{\infty}/\mathbb{Q}$ is evidently Galois as the elliptic curve $E$ is defined over $\mathbb{Q}$, and so $\mathrm{Gal}(L/\mathbb{Q})$ is Galois of degree dividing 4, and is thus abelian. Since $K \subset L$, complex conjugation induces an element of order 2 in $\mathrm{Gal}(L/\mathbb{Q})$ that does not fix $K$, and which is thus not contained in the order 1 or 2 subgroup $\mathrm{Gal}(L/K) \subset \mathrm{Gal}(L/\mathbb{Q})$. We thus have that $\mathrm{Gal}(L/\mathbb{Q}) = \mathbb{Z}/2$ or $(\mathbb{Z}/2)^{\oplus 2}$ and in the latter case $p \neq 2$. Therefore, $L = \mathbb{Q}(\sqrt{-p},\sqrt{m})$ or $\mathbb{Q}(i,\sqrt{m})$ for some $m \in \mathbb{Z}$ such that $\mathbb{Q}(\sqrt{m})$ is ramified only possibly at primes dividing $fp\infty$. Moreover, if $\mathbb{Q}(\sqrt{m}) \neq \mathbb{Q}$ then $p \neq 2$. Suppose $p$ is ramified in $\mathbb{Q}(\sqrt{m})$, so that in particular $p \neq 2$. Then the field $\mathbb{Q}(\sqrt{-pm}) \subset L$ is unramified at $p$ since $p \neq 2$, and ramified only possibly at primes dividing $f\infty$. Thus in any case, after possibly replacing $m$ by $-pm$, we have found a subfield $K' \subset L$ of degree at most 2 over $\mathbb{Q}$ with $K' \cap K = \mathbb{Q}$, $L = KK'$ and such that $K'$ is unramified at $p$ and ramified possibly at primes dividing $f\infty$. In particular, $K' \subset \mathbb{Q}(\mu_f)$.

Recall $f_0$ from the beginning of this section (i.e. $f_0 \in\mathbb{Z}_{>0}$ with $f_0\mathcal{O}_K = \frak{f}_E^{(p)}$). In the setting of Definition \ref{cyclotomicsetting}, let 
$$f = f_0$$
and let
$$L^+ = L = KK' \subset K(\mu_{f_0}).$$
From (\ref{deltadecomposition}) and (\ref{Deltadecomposition}), we have 
$$\mathrm{Gal}(L_{\infty}/L_{\infty}^+) \cong \mathcal{H} \cong \mathbb{Z}_p.$$
By \cite[1.6.5 (3)]{FukayaKato}, we have 
$$M_+' := M' \otimes_{\mathcal{O}_{K_p}\llbracket \mathrm{Gal}(L_{\infty}/K)\rrbracket} \mathcal{O}_{K_p}\llbracket \mathrm{Gal}(L_{\infty}^+/K)\rrbracket \cong H^2(\mathcal{O}_K[1/f_0p],\mathcal{O}_{K_p}\llbracket \mathrm{Gal}(L_{\infty}^+/K)\rrbracket (1)).
$$
 Let ``f.g.'' denote ``finitely generated''. By the Nakayama lemma for completed group rings (\cite[5.2.18]{Neukirch}), we see that 
$$\text{$M_+'$ is a f.g. $\mathcal{O}_{K_p}$-module} \iff \text{$M'$ is a f.g. $\mathcal{O}_{K_p}\llbracket \mathrm{Gal}(L_{\infty}/L_{\infty}^+)\rrbracket \cong \mathcal{O}_{K_p}\llbracket \mathbb{Z}_p\rrbracket$-module}.$$
By Lemma \ref{reduceM'lemma}, 
 $$\text{$M$ from (\ref{M}) is a f.g. $\mathcal{O}_{K_p}\llbracket \mathrm{Gal}(\mathcal{K}_{\infty}/L_{\infty}^+)\rrbracket$-module} \iff \text{$M'$ is a f.g. $\mathcal{O}_{K_p}\llbracket \mathrm{Gal}(L_{\infty}/L_{\infty}^+) \rrbracket$-module}.$$
 Thus, 
\begin{equation}\label{finalequivalence0}\text{$M$ is a f.g. $\mathcal{O}_{K_p}\llbracket \mathrm{Gal}(\mathcal{K}_{\infty}/L_{\infty}^+)\rrbracket$-module} \iff \text{$M_+'$ is a f.g. $\mathcal{O}_{K_p}$-module}.
\end{equation} 

We establish one final equivalence. Note that 
$$\mathrm{Gal}(\mathcal{K}_{\infty}/L_{\infty}^+) \cong \Delta[p^{\infty}] \times \mathcal{H}$$
under (\ref{deltadecomposition}). Using this isomorphism, consider the fixed field of $\mathcal{H}$
$$F_{\infty} := \mathcal{K}_{\infty}^{\mathcal{H}}.$$
Since $\Delta[p^{\infty}] \cong \mathrm{Gal}(F_{\infty}/L_{\infty}^+)$ is finite, we have that 
$$\text{$M$ is a f.g. $\mathcal{O}_{K_p}\llbracket \mathrm{Gal}(\mathcal{K}_{\infty}/L_{\infty}^+)\rrbracket$-module} \iff \text{$M$ is a f.g. $\mathcal{O}_{K_p}\llbracket  \mathrm{Gal}(\mathcal{K}_{\infty}/F_{\infty})\rrbracket$-module}.$$
Thus by (\ref{finalequivalence0}), 
\begin{equation}\label{finalequivalence}\text{$M$ is a f.g. $\mathcal{O}_{K_p}\llbracket \mathrm{Gal}(\mathcal{K}_{\infty}/F_{\infty})\rrbracket$-module} \iff \text{$M_+'$ is a f.g. $\mathcal{O}_{K_p}$-module}.
\end{equation}

Now we can finally prove Theorem \ref{EMC}. Recall $\frak{f}_E = \frak{f}(\lambda_E)$, and $\frak{f}_E^{(p)} = (f_0)$ from Choice \ref{f0choice}.

\begin{proof}[Proof of Theorem \ref{EMC}]We follow the notation of \cite[Section 5.5]{JohnsonLeungKings} throughout this proof unless otherwise noted. Let $$\Omega = \mathcal{O}_{K_p}\llbracket \mathrm{Gal}(\mathcal{K}_{\infty}/K)\rrbracket, \hspace{1cm} \Omega(\frak{f}_E^{(p)}) = \mathcal{O}_{K_p}\llbracket \mathrm{Gal}(K(\frak{f}_E^{(p)}p^{\infty})/K)\rrbracket.$$
Recall $\mathrm{Gal}(\mathcal{K}_{\infty}/K) \cong \Delta_K \times \Gamma_K$ by (\ref{fixK}). Let 
$$\mathcal{J}_{\Omega} \subset \Omega$$ 
denote the annihilator of $\mathbb{Z}_p[\Delta_K](1)$, where ``(1)'' here and below denotes Tate twist. Let 
$${}_{\frak{a}}\zeta(\frak{f}_E^{(p)}) := \varprojlim_n {}_{\frak{a}}\zeta_{\frak{f}_E^{(p)}\frak{p}^n} \in H^1(\mathcal{O}_K[1/f_0p],\Omega(\frak{f}_E^{(p)})(1)) \otimes_{\Omega(\frak{f}_E^{(p)})}\mathrm{Frac}(\Omega(\frak{f}_E^{(p)}))$$
where ${}_{\frak{a}}\zeta_{\frak{f}_E^{(p)}\frak{p}^n}$ is as in Section 3.2 of op. cit. Define, using the map 
$$\mathrm{Nm}_{\frak{f}_E^{(p)}} : \mathbb{U}^1(\frak{f}_E^{(p)}) \rightarrow \mathbb{U}^1$$
from (\ref{UEnorm}),
$${}_{\frak{a}}\zeta_E := \mathrm{Nm}_{\frak{f}_E^{(p)}}({}_{\frak{a}}\zeta(\frak{f}_E^{(p)})) \in \mathbb{U}^1$$
and
$$\zeta_E := (\mathbb{N}\frak{a} - \sigma_{\frak{a}})^{-1} \cdot {}_{\frak{a}}\zeta_E \in H^1(\mathcal{O}_K[1/f_0p],\Omega(1)) \otimes_{\Omega}\mathrm{Frac}(\Omega).$$
Thus
$$\mathcal{J}_{\Omega} \cdot (\zeta_E) \subset H^1(\mathcal{O}_K[1/f_0p],\Omega(1)).$$ The main result on the equivariant main conjecture we will need is:

\begin{theorem}[Theorem 5.7 of \cite{JohnsonLeungKings}]\label{equivarianttheorem}In the above setting, assume that 
\begin{equation}\label{5.5}H^2(\mathcal{O}_K[1/f_0p],\Omega(1))_{\frak{q}} = 0
\end{equation}
for every height 1 prime ideal of $\Omega$ containing $p$. Then
\begin{equation}\label{5.7}\mathrm{det}_{\Omega}\left(H^1(\mathcal{O}_K[1/f_0p],\Omega(1))/(\mathcal{J}_{\Omega}\cdot(\zeta(\mathcal{K}_{\infty})))\right) = \mathrm{det}_{\Omega}\left(H^2(\mathcal{O}_K[1/f_0p],\Omega(1))\right).
\end{equation}
\end{theorem}

Note that strictly speaking, Theorem 5.7 of op. cit. is only explicitly stated for full ray class field towers $K(\frak{f}p^{\infty})/K$ for any ideal $\frak{f} \subset \mathcal{O}_K$. However, the formulation and proof of the Theorem applies \emph{mutatis mutandis} for the tower $\mathcal{K}_{\infty}/K$. The statement and the proof of the corresponding $\Lambda$-main conjectures remain the same. The only nontrivial point to check is that the argument of the proof of Lemma 4.6 of op. cit. goes through for $\mathcal{K}_{\infty}/K$. However, this follows upon noting that the ramification group of the prime $\frak{p}$ of $\mathcal{O}_K$ above $p$ has finite index in $\mathrm{Gal}(\mathcal{K}_{\infty}/K)$; this follows, for example, from Proposition \ref{GammaKtotallyramified} with $M = K$ and the fact that $\mathcal{K}_{\infty}/K_{\infty}$ is a finite extension.

Theorem \ref{EMC} immediately follows from Theorem \ref{equivarianttheorem} and 
\begin{enumerate}
\item taking the $\chi_E$-isotypic component (see Definition \ref{isotypicdefinition} with respect to (\ref{chiE})) of (\ref{5.7}) (note that this inverts $p$),
\item twisting by $\lambda_E^{-1}$ (see Definitions \ref{globalcharactertwistdefinition} and \ref{Mtwistdefinition}), and 
\item applying Lemma 5.8 of op. cit. and the discussion immediately following it.
\end{enumerate}

The verification of (\ref{5.5}) will follow from Lemma 5.11 of op. cit. after we show that $M$ from (\ref{M}) is a finitely generated $\mathcal{O}_{K_p}\llbracket \mathrm{Gal}(\mathcal{K}_{\infty}/F_{\infty})\rrbracket$-module (i.e. taking $\mathcal{H} = \mathrm{Gal}(\mathcal{K}_{\infty}/F_{\infty}) \cong \mathbb{Z}_p$ in loc. cit.). By (\ref{finalequivalence}), this in turn follows once we show that $M_+'$ is a finitely generated $\mathcal{O}_{K_p}$-module. This latter finite-generatedness will ultimately follow from our results on the vanishing of the total $\mu$-invariant of $\mathcal{X}^+$ (Corollary \ref{analyticmuvanishing}), using the argument of Corollary 5.10 of op. cit. \emph{mutatis mutandis}. We explain this in detail below. 

By Corollary \ref{analyticmuvanishing}, we have that $\mathcal{X}^+$ is a finitely generated $\mathbb{Z}_p$-module. Recall that $N_n^+$ is the maximal unramified pro-$p$ abelian extension of $L_n^+$, and 
$$\mathcal{Y}^+ = \varprojlim_n \mathrm{Gal}(N_n^+/L_n^+).$$
The natural surjection $\mathcal{X}^+ \twoheadrightarrow \mathcal{Y}^+$ implies that $\mathcal{Y}^+$ is also a finitely generated $\mathbb{Z}_p$-module. Moreover, it is clear that as 
$$L_{\infty}^+ = L^+(\mu_{p^{\infty}}) = K(\mu_{f_0p^{\infty}}) = \mathbb{Q}(\mu_{f_0p^{\infty}})$$
and all primes have finite splitting index in cyclotomic fields, then all primes of $K \subset \mathbb{Q}(\mu_{p^{\infty}})$ have finite splitting index in $L_{\infty}^+/K$. 

Hence by the exact sequence from the proof of Corollary 5.10 of op. cit.,
$$\mathcal{Y}^+ \otimes_{\mathbb{Z}_p}\mathcal{O}_{K_p} \rightarrow M_+'  = H^2(\mathcal{O}_K[1/f_0p],\mathcal{O}_{K_p}\llbracket \mathrm{Gal}(L_{\infty}^+/K)\rrbracket(1)) \rightarrow \bigoplus_{v|f_0p}\mathcal{O}_{K_p}\llbracket \mathrm{Gal}(L_{\infty}^+/K)/D_v\rrbracket,$$
and using the above noted fact that every prime $v|f_0p$ has finite splitting index in $L_{\infty}^+/K$ which implies that each $\mathrm{Gal}(L_{\infty}^+/K)/D_v$ is finite, we see that $M_+'$ is a finitely generated $\mathcal{O}_{K_p}$-module. 

We are done with the proof of Theorem \ref{EMC}.
\end{proof}

%\begin{corollary}\label{proveRubinMC}We have that (\ref{RubinMC1}) is true for any $\chi \in \hat{\tilde{\Delta}}$. 
%If $w_{\frak{f}} \neq 1$, we have that these statements are true over $\Lambda(G_{\infty},\mathcal{O}_{\hat{L}_{p,\infty}})[1/p]$ after tensoring all modules and rings with $\otimes_{\mathbb{Z}_p}\mathbb{Q}_p$. 
%\end{corollary}

%\begin{proof}This follows immediately from Theorem \ref{MCtruethm} and Proposition \ref{MCimplyprop}. 
%\end{proof}

%\begin{remark}As remarked before, the assumption $p > 2$ can very likely be removed with more care in Section \ref{valuationsection}. Recall, however, that when $p = 2$, $\mu_{\mathrm{glob}}$ is only defined for $K = \mathbb{Q}(i)$ or $D \equiv 0 \pmod{8}$. (Recall that $D$ is the fundamental discriminant of $K$.) 
%\end{remark}

\end{document}